\documentclass{amsart}

\usepackage{tikz-cd}

\usepackage{hyperref}

\usepackage{makecell}

\usepackage{url}

\usepackage{hhline}

\usepackage{booktabs}

\usepackage[english]{babel}
\usepackage{amsfonts, amsmath, amsthm, amssymb,amscd,indentfirst}
\usepackage{latexsym}

\usepackage{verbatim}
\usepackage{arydshln}
\usepackage[makeroom]{cancel}
\usepackage{mathrsfs}  
\usepackage{tikz,tikz-3dplot}
\usetikzlibrary{angles, quotes, calc}

\usepackage{scalerel}
\usepackage{stackengine,wasysym}
\usepackage{bm}
\usepackage{xcolor}
\usepackage{esint}

\usepackage{graphicx}
\usepackage{epstopdf}

\usepackage{palatino}

\usepackage{stackengine}
\usepackage{scalerel}

\numberwithin{equation}{section}

\theoremstyle{plain}
\newtheorem{theorem}{Theorem}[section]

\newtheorem{proposition}[theorem]{Proposition}

\newtheorem{corollary}[theorem]{Corollary}
\theoremstyle{definition}
\newtheorem{definition}[theorem]{Definition}

\newtheorem{example}[theorem]{Example}

\newtheorem{remark}[theorem]{Remark}
\newtheorem{convention}[theorem]{Convention}

\usepackage[
breakable,
]{tcolorbox} 
\usepackage{float}

\tcolorboxenvironment{theorem}{
	colframe=black,
	breakable,before skip=10pt,after skip=10pt }

\tcolorboxenvironment{proposition}{
	colframe=black,
	breakable,before skip=10pt,after skip=10pt } 

\tcolorboxenvironment{corollary}{
	colframe=black,
	breakable,before skip=10pt,after skip=10pt } 

\tcolorboxenvironment{convention}{
	colframe=black,
	breakable,before skip=10pt,after skip=10pt }

\tcolorboxenvironment{definition}{
	colframe=black,
	breakable,before skip=10pt,after skip=10pt }

\makeatletter
\DeclareMathAlphabet{\mathfrak}{U}{euf}{m}{n}
\SetMathAlphabet{\mathfrak}{bold}{U}{euf}{b}{n}
\makeatother

\begin{document}
	
	\title{Topics in Probability, Parametric Estimation and Stochastic Calculus}%
	
	\author{Levi Lopes de Lima}
	\address{Universidade Federal do Cear\'a,
		Departamento de Matem\'atica, Campus do Pici, R. Humberto Monte, s/n, 60455-760,
		Fortaleza/CE, Brazil ({\tt levi@mat.ufc.br}).}

\begin{abstract}
We begin our journey by recalling the fundamentals of Probability Theory that underlie one of its most significant applications to real-world problems: Parametric Estimation. Throughout the text, we systematically develop this theme by presenting and discussing the main tools it encompasses—concentration inequalities, limit theorems, confidence intervals, maximum likelihood, least squares, and hypothesis testing—always with an eye toward both their theoretical underpinnings and practical relevance. While our approach follows the broad contours of conventional expositions, we depart from tradition by consistently exploring the geometric aspects of probability, particularly the invariance properties of normally distributed random vectors. This geometric perspective is taken further in an extended appendix, where we introduce the rudiments of Brownian motion and the corresponding stochastic calculus, culminating in It\^o’s celebrated change-of-variables formula. To highlight its scope and elegance, we present some of its most striking applications: the sharp Gaussian concentration inequality (a central example of the “concentration of measure phenomenon”), the Feynman–Kac formula (used to derive a path integral representation for the Laplacian heat kernel), and, as a concluding delicacy, the Black–Scholes strategy in Finance.
	\end{abstract}

	\maketitle
	
	\tableofcontents

	\section{Introduction}\label{int:sec}
	
	Probability Theory is a multifaceted intellectual enterprise, at once rigorous in its mathematical formulation and remarkably flexible in its applications. It has become indispensable across pure mathematics, applied sciences, and the ever-expanding universe of data-driven disciplines. Among its many usages in real-world problems, there is a particular body of knowledge that stands out: Parametric Estimation. Rooted in the classical contributions of K.F. Gauss and P.-S. Laplace—especially in their formulation of the method of least squares—and later transformed into a systematic research program by R.A. Fisher, J. Neyman, E.S. Pearson and others, this framework of ideas has grown into one of the central pillars of modern statistics. Its methods permeate fields as diverse as Biology, Medicine, and Evolutionary theory, Psychology, Sociology, and Economics, not to mention their longstanding role in Physics, Chemistry, and Engineering. More recently, suitable refinements of this classical theory have proven crucial in assessing the efficiency of statistical procedures in modern Data Science and Machine Learning, as exemplified in the study of regression, classification, and sparsity-driven approaches within Statistical Learning \cite{james2013introduction}. 
	
The purpose of these notes is to provide a modest yet comprehensive introduction to this circle of ideas, written for those with adequate preparation in the necessary prerequisites—essentially Linear Algebra, Multivariate Calculus, and Measure Theory, with a touch of Fourier Analysis. Such background ensures the technical fluency required to follow the arguments and the mathematical maturity, at roughly the graduate level, to connect conceptual and computational aspects of the theory. The guiding principle is to move from the foundational elements of Probability Theory to a systematic development of estimation methods, always showing how abstract reasoning translates into applications to statistical problems.

With this in mind, the exposition opens in Sections \ref{sec:bas:prob} and \ref{cond:poss} with the fundamental notions of Probability Theory, including the key concepts of independence and conditioning, thereby establishing the groundwork for the developments that follow. 
Section \ref{normaldist} builds on this foundation by introducing the essentials of the theory of normally distributed (Gaussian) random vectors, whose geometric properties reappear throughout the text as a unifying theme. 
From this point, Section \ref{conc:ineq:appl} ventures into the domain of concentration inequalities, which provide
non-asymptotic bounds for tail probabilities, with the discussion centering on inequalities derived from the elementary yet powerful Cramér–Chernoff method. 
In addition to showing how this approach leads to the Johnson–Lindenstrauss lemma and to phase-transition behavior in the Erdős–Rényi random graph model, we also discuss the Gaussian concentration inequality, emphasizing its connection to Poincaré’s limit theorem and to the broader ``concentration of measure phenomenon''.

The exposition then transitions naturally to the asymptotic framework in Section \ref{fund:lim}, where the law of large numbers and the central limit theorem are derived within the Fourier-analytic setting of characteristic functions. These classical results form the theoretical foundation for constructing large-sample confidence intervals for the expectation of virtually any distribution of practical relevance and serve as a cornerstone for many developments in the subsequent sections. Section \ref{est:theo} turns to estimation proper, focusing on the role of statistical models (as formalized by Fisher) and the performance of estimators as measured by the mean squared error, alongside related notions such as consistency, bias-variance trade-off, and asymptotic normality. Building on this, Section \ref{sec:MLE} develops the method of maximum likelihood, analyzing its asymptotic behavior in the light of Fisher information and the celebrated Cramér–Rao lower bound, before establishing the optimal asymptotic normality of maximum
likelihood estimators. 
Section~\ref{mls:sub} is devoted to the method of least squares. It begins with a careful presentation of the statistical model underlying ordinary least squares and proceeds to its inferential ramifications, including confidence intervals, hypothesis tests, and measures of goodness-of-fit. The discussion emphasizes both interpretability and predictive accuracy in the classical low-dimensional regime ($p \ll n$). The section then moves to the challenges posed by higher-dimensional settings, introducing regularization techniques and sparsity via the LASSO as a natural gateway to the methods that dominate modern Data Science; see, for instance, \cite[Chapter~6]{james2013introduction}.
Section~\ref{exp:glms} offers an introduction to the exponential family and the theory of generalized linear models, thereby unifying under a single framework many of the models encountered in practice (logistic, Poisson, gamma, among others), and extending some of the regression tools of Section \ref{mls:sub} beyond the normal case.
Section \ref{suff:sub} discusses sufficiency, while Section \ref{sec:hyp:test} completes the standard estimation package with hypothesis testing, emphasizing likelihood ratio tests and illustrating them with canonical examples. 
Further perspectives appear in Section \ref{br:over}, which traces the transition from Fisher’s parametric paradigm to modern Statistical Learning Theory, and in Section \ref{bay:way}, which presents the Bayesian pathway as an alternative and increasingly influential approach.

A substantial Appendix, requiring only familiarity with the material up to Subsection \ref{normaldist:sub}, is devoted to Brownian motion and Itô’s calculus. It begins with the construction of Brownian paths and culminates in Itô’s celebrated change-of-variables formula. This framework allows for a complete proof of the sharp Gaussian concentration inequality, which provides the exact constant for the associated standard deviation, and demonstrates the strength of stochastic calculus through further applications such as the Feynman–Kac formula, yielding a path-integral representation of the Laplacian heat kernel, and the Black–Scholes model in Finance, a landmark contribution that transformed financial mathematics and earned a Nobel Prize in 1997.

The presentation draws on both classical references and recent monographs, many of which are cited throughout\footnote{I am also grateful for the many insightful conversations with colleagues, which have helped shape and refine not only these notes but also the companion computational labs. Special thanks go to C.~Barroso and J.~F.~Montenegro (UFC), T.~Alencar and J.~Silva (UFCA), T.~Alencar and J.~Silva (UFCA), Mykael Cardoso, J.~X.~da~Cruz~Neto, Rondinelle Marcolino, Ray Serra and Sandoel Vieira  (UFPI) and J.~Stoyanov.}.
At certain points, however, the exposition deliberately departs from the conventional treatment by emphasizing geometric perspectives in the theory. These include Fisher’s elegant method for deriving the distributions of ubiquitous statistics such as the chi-square, Student’s $\mathfrak t$, and correlation coefficients, the profound link between Scheffé-type simultaneous confidence bands for the mean response in the linear model and Weyl’s tube formula, and the illuminating interplay between concentration inequalities and the framework of high-dimensional probability, a theme that has become pivotal in contemporary Data Science. A further example is provided by the well-known dichotomy between model interpretability and prediction accuracy in linear models, where the geometry migrates from projections in sample space (as in the realm of ordinary least squares) to constraints in parameter space through regularization methods (such as ridge regression and the LASSO), illustrating how classical approaches adapt to remain effective in modern contexts. Taken together, these excursions highlight the central geometric role of normally distributed random vectors across key aspects of the theory. 

In addition to the formal development, we have incorporated throughout the text a number of remarks, examples, and short subsections designed to offer a complementary perspective. In most cases, these elements aim to illuminate how probabilistic ideas interact with contemporary themes in Data Science and Statistical Learning, often going beyond what is strictly required for the mathematical exposition. For ease of reference, such passages are consistently marked with the symbol $\bigstar$.

Further complementing the theoretical development, the text incorporates numerous illustrative applications and contextual remarks that accentuate real-world connections, helping place abstract concepts within the broader landscape of probability, statistics, and stochastic analysis.
In the same spirit of connecting abstract theory with concrete practice, the notes are accompanied by a series of computational labs available at

\begin{center}
	\href{https://github.com/levilopesdelima/stat-inference-labs}{\texttt{https://github.com/levilopesdelima/stat-inference-labs}}
\end{center}

\noindent
together with updates, corrections, and a continuously revised book-style version of the present manuscript\footnote{The topics currently covered in the labs include linear regression, the central limit theorem, the Johnson--Lindenstrauss lemma, maximum likelihood estimation, and James--Stein estimators, among others.}. These labs are not intended as a mere repository of R code snippets. Rather, they combine concise theoretical recaps with step-by-step simulations, numerical experiments, and visualizations designed to reinforce the underlying mathematical ideas. By systematically aligning formal arguments with computational exploration, they provide a structured environment in which probabilistic concepts may be both analyzed theoretically and explored empirically.

\noindent
{\bf Declaration of Generative AI in the writing process.}
As has become increasingly common in contemporary academic writing, generative AI tools (GPT-5.5, OpenAI) were used during the final stages of preparation of this manuscript, primarily for language refinement, structural editing, and bibliographic assistance.
All generated content was carefully reviewed and edited by the author, who assumes full responsibility for the mathematical correctness and final form of the work.

\section{The fundamentals of Probability Theory}\label{sec:bas:prob}
	
In this section, we present a concise overview of Probability Theory (or, more precisely, of those aspects most relevant to the applications that follow).
Since this material is covered in detail in many standard references, and our primary goal here is merely to establish notation and recall essential facts, proofs are only outlined or omitted altogether.

\subsection{The probabilistic setup: random variables and their distributions}\label{prob:basic}	
	Let $\Omega\neq\varnothing$ be a set and consider $\mathcal F$ a collection of subsets of $\Omega$. 
	
	\begin{definition}\label{sigma:alg}
	We say that $\mathcal F$ is a $\sigma$-{\em algebra} if
	\begin{itemize}
		\item $\Omega \in \mathcal F$;
		\item $A\in \mathcal F \Rightarrow A^c\in\mathcal F$;
		\item $\{A_i\}_{i=1}^{+\infty}\subset\mathcal F\Rightarrow \cup_{i=1}^{+\infty} A_i\in\mathcal F$.
	\end{itemize}
	\end{definition}

Trivial examples of $\sigma$-algebras are $\mathcal F=\{\varnothing,\Omega\}$ and $\mathcal F=2^\Omega$, the set of all subsets of $\Omega$. More generally, if $\mathcal U=\{U_\lambda\}_{\lambda\in\Lambda}$ is any collection of subsets of $\Omega$, we denote by $\mathcal F_{\mathcal U}=\mathcal F(U_\lambda)$ the $\sigma$-algebra generated by $\mathcal U$. By definition, this is the smallest $\sigma$-algebra contained all elements of $\mathcal U$. For example, if $\mathcal O^n$ is the set of open subsets in $\mathbb R^n$ then $\mathcal B^n:=\mathcal F_{\mathcal O_n}$ is the $\sigma$-algebra of Borel subsets. 
	
	\begin{definition}\label{meas:def}
	A {\em measure} on $(\Omega,\mathcal F)$ is a real valued function $P$ on $\mathcal F$ so that:
	\begin{itemize}
		\item $P(\varnothing)=0$;
		\item $P(A)\geq 0$, for any $A\in \mathcal F$;
		\item if $\{A_i\}_{i=1}^{+\infty}\subset{\mathcal F}$ satisfies $A_i\cap A_j=\varnothing$, $i\neq j$, then
		\[
		P(\cup_i A_i)=\sum_iP(A_i).
		\]
	\end{itemize}
	\end{definition}
	
	We say that a triple $(\Omega,\mathcal F,P)$ is a {\em measure space}. A classical example is $(\mathbb R^n, \mathcal L^n,\lambda^n)$, where $\mathcal L^n$ is the standard completion of $\mathcal B^n$ and $\lambda^n$ is Lebesgue measure.  
	If $P(\Omega)=1$ then we say that $(\Omega,\mathcal F,P)$ is a {\em probability space}, 
	a basic notion in Probability Theory. 
	In this setting, and when no confuson arises, we will represent the corresponding Lebesgue spaces simply by $L^p(\Omega)$, $1\leq p<+\infty$, with no further reference to the additional data defining the associated probability space. Also,
	each set $A\in\mathcal F$ is called an {\em event}. 
	
	Another key notion is that of {\em random vector}. If $(\Omega,\mathcal F,P)$ is a probability space, this is  a function $X:(\Omega,\mathcal F)\to (\mathbb R^n,\mathcal B^n)$ which is measurable in the sense that $X^{-1}(B)\in \mathcal F$ for any $B\in \mathcal B^n$. When $n=1$ we say that $X$ is a {\em random variable}.
	If $X$ is a random vector then we denote by $\mathcal F_X$ the $\sigma$-algebra generated by $X$, i.e. the $\sigma$-algebra generated by $\{X^{-1}(B);B\in\mathcal B^n\}$. Alternatively, $\mathcal F_X$ is the smallest $\sigma$-algebra of $\mathcal F$ with respect to which $X$ remains measurable.
	
	A central notion in Probability is that of {\em independence}\footnote{As already noted in \cite[Secton I.5]{kolmogorov2018foundations}: ``Historically, the independence of experiments and random variables represents the very mathematical concept that has given the theory of probabilities its peculiar stamp''.}. Here we define it at several levels, shifting the emphasis from events to random variables.
	
	\begin{definition}\label{def:ind:levels}
		In the setting above, we adopt the following notions of independence.
		\begin{enumerate}
			\item A finite collection of events $A_1,\ldots,A_k \in \mathcal F$, $k\geq 2$, is said to be \emph{independent} if
			\begin{equation}\label{def:ind:ev}
			P(A_1 \cap \cdots \cap A_k) = P(A_1) \cdots P(A_k).
			\end{equation}
			
			\item Let $\{\mathcal F_\lambda\}_{\lambda \in \Lambda}$ be an arbitrary family of $\sigma$-subalgebras of $\mathcal F$.  
			We say that $\{\mathcal F_\lambda\}_{\lambda\in\Lambda}$ is \emph{independent} if for every finite subcollection 
			$\{\mathcal F_{\lambda_\ell}\}_{\ell=1}^k$ and events $A_{\lambda_\ell} \in \mathcal F_{\lambda_\ell}$, 
			the events $\{A_{\lambda_\ell}\}_{\ell=1}^k$ are independent.
			
			\item A family $\{X_\lambda\}_{\lambda \in \Lambda}$ of random variables defined on the same probability space 
			is said to be \emph{independent} if $\{\mathcal F_{X_\lambda}\}_{\lambda \in \Lambda}$ is independent. 
			(Notation: for a pair of independent random variables, we write $X \perp\!\!\perp Y$.)
		\end{enumerate}
	\end{definition}

\begin{remark}\label{indep:various}
A finite collection of events $A_1,\ldots,A_k \in \mathcal F$, as in~(1) above, 
is said to be \emph{mutually independent} if, for any $2 \le l \le k$ and any 
subcollection of indices $1 \le i_1 < \cdots < i_l \le k$, one has
\[
P(A_{i_1} \cap \cdots \cap A_{i_l})
= P(A_{i_1}) \cdots P(A_{i_l}).
\]
This requirement is, of course, much stronger than~\eqref{def:ind:ev}, since it 
entails the verification of $2^k - k - 1$ conditions rather than a single one; 
see~\cite[Section~3]{stoyanov2013counterexamples} for a thorough discussion of this well-known distinction.
Instead of including this stronger notion in item~(1) of 
Definition~\ref{def:ind:levels}, we deliberately depart from tradition and adopt 
the weaker notion for at least two reasons. First, it fits naturally into the 
hierarchy of levels presented there, which progresses from events to random 
variables (cf.\ the comments immediately following Corollary~\ref{induncorr}). 
Second, the notion of independence of events is rarely (if ever!) used directly 
in the text, since the emphasis here is always on exploring 
independence at the level of random variables, which is properly defined in 
item~(3) and later reinterpreted via the product rule for the corresponding joint 
distribution (Proposition~\ref{inddens}).
\qed
	\end{remark}
	
We now consider the {\em expectaction} (or {\em expected value}) of a random vector $X:\Omega\to \mathbb R^n$, which is given by 
	\[
	\mathbb E(X):=\int_{\Omega}X\,dP\in\mathbb R^n.
	\]
	Usually we assume that this is finite (that is, $X$ is integrable: $X\in L^1(\Omega)$). 	Another key notion is that of {\em covariance} of two random variables: if $X,Y:\Omega\to\mathbb R$ then this  is given by
	\begin{eqnarray*}
		{\mathbb C}(X,Y) & = &\mathbb E\left((X-\mathbb E(X))(Y-\mathbb E(Y))\right)\\
		& = & \mathbb  E(XY)-\mathbb E(X)\mathbb E(Y).
	\end{eqnarray*}
	Here, we require that $X,Y\in L^2(\Omega)$ as this implies that $XY\in L^1(\Omega)$ by Cauchy-Schwarz.
	
	\begin{definition}
		\label{uncor}
		We say that $X$ and $Y$ are {\em uncorrelated} if ${\mathbb C}(X,Y)=0$. 
	\end{definition}
	
	That uncorrelatedness pertains to independence is a consequence of the next fundamental result.
	
	\begin{proposition}\label{indexp}
		If $X,Y:\Omega\to \mathbb R$ are independent random variables  then $\mathbb E(XY)=\mathbb E(X)\mathbb E(Y)$.
	\end{proposition}
	
	\begin{proof} (sketch)
		By a simple approximation, it suffices to assume that $X$ and $Y$ are simple functions with  $|X|,|Y|\leq M< +\infty$. Hence\footnote{Here and in the following, if $A\in\mathcal F$ we shall denote by ${\bf 1}_A$ the corresponding {\em indicator function}
			\[
			{\bf 1}_A(x)=
			\left\{
			\begin{array}{rl}
				1 & {\rm if} \, x\in A;\\
				0 & {\rm otherwise}
			\end{array}
			\right.
			\]},
		\[
		X=\sum_i a_i{\bf 1}_{F_i}, \quad Y=\sum_jb_j{\bf 1}_{G_j}, 
		\]
with $F_i=X^{-1}(a_i)\in\mathcal F_X$ and similarly for $G_j$.		
Hence,
		\[
		XY=\sum_{ij}a_ib_j{\bf 1}_{F_i\cap G_j},
		\]
so that
\[
\mathbb E(XY)  =  \sum_{ij}a_ib_jP(F_i\cap G_j).
\]
Using that $\{\mathcal F_X,\mathcal F_Y\}$ is independent together with (\ref{def:ind:ev}) we thus see that 
		\begin{eqnarray*}
			\mathbb E(XY)
			& = & 
			\sum_{ij}a_ib_jP(F_i)P(G_j)\\
			& = & \sum_ia_iP(F_i)\cdot \sum_jb_jP(G_j) \\
			& = &
			\mathbb E(X)\mathbb E(Y),  
		\end{eqnarray*}
		as claimed.
	\end{proof}

	\begin{corollary}
		\label{induncorr}
		If $X$ and $Y$ are independent then they are uncorrelated.
	\end{corollary}
	
	In a sense, the argument sketched in the proof of Proposition~\ref{indexp} respects the shift in levels adopted in Definition~\ref{def:ind:levels}: we first work at level~(1), that of events, which amounts to establishing the result for simple functions, and then use a suitable approximation (essentially the definition of integrals in Lebesgue theory) to move to level~(3), where the concept of independence is formulated in terms of random variables. 
	A straightforward proof of Proposition~\ref{indexp}, which operates directly at the highest level, is provided in Remark~\ref{one:line} below. 
	Note also that the converse of Corollary~\ref{induncorr} does not hold in general. 
	It does, however, hold in the important case where $X$ and $Y$ are the components of a normally distributed random vector; 
	see Proposition~\ref{unc:ind:n}.

	\begin{definition}
		\label{distr}
		If $X:\Omega\to\mathbb R^n$ is a random vector then its {\em distribution} (or {\em law}) is the probability measure $X_\sharp P$ on $\mathbb R^n$ given by
		\[
		X_\sharp P(B)=P(X^{-1}(B)),\quad B\in\mathcal B^n.
		\]
		We also represent $X_\sharp P$ by $P_X$ and set $P(X\geq a):=P_X([a,+\infty))$, etc.  Also, an element ${\bf x}\in\mathbb R^n$ in the image of a random vector $X:\Omega\to \mathbb R^n$ (or equivalently, in ${\rm supp}(X)$) is called a {\em realization}
		(or {\em observed value}) of $X$. 
	\end{definition} 
	
\begin{remark}\label{dogma} (The central role of distributions and the extension dogma) The moral here is that \emph{any} random variable $X$ is doomed to mediate between the (rather abstract) probability measure $P$ and its distribution $P_X$, a more tangible probability measure on $\mathbb{R}$. In this way, $X$ links two complementary levels of description of randomness. While $P$ is often difficult to visualize or describe, since it is defined on the sample space $\Omega$, a purely mathematical construct whose internal structure is rarely made explicit, its distribution $P_X$ lives on $\mathbb{R}$ and can therefore be analyzed through familiar descriptive tools such as densities, cumulative distribution functions, tail probabilities, and quantiles. Furthermore, because realizations of independent copies of $X$ (the random samples in Remark \ref{many:ind}) provide direct access to the features of $P_X$, it is this distribution that becomes the natural object of study in statistics, physics, and the applied sciences. This perspective is closely related to what probabilists sometimes call the {\em extension dogma}: 
only those notions that remain invariant under pullback by probability-preserving transformations of the sample space should be regarded as intrinsically probabilistic.
In particular, if $T:\Omega_1\to\Omega_2$ is such a transformation and $X$ is a random variable on $\Omega_2$ then $X$ and its pullback $X\circ T$ have the same distribution and must therefore be indistinguishable from the probabilistic point of view. Consequently, the essential content of probability theory resides not in the particular representation of randomness provided by the sample space, but in the distributions of the random variables defined upon it\footnote{Examples of legitimate probabilistic notions appearing in the sequel include expectation, variance, moments, independence and conditional expectation, among others that can be explicitly described in terms of the underlying distributions. By contrast, any structural property of $\Omega$ that is not reflected in these laws (such as the cardinality of an event), fails to qualify as intrinsically probabilistic.}; see \cite[Subsection 1.1.1]{tao2012topics}\qed
\end{remark}

	\begin{definition}
		If $X:\Omega\to\mathbb R$ is a random variable then its {\em cumulative distribution function (cdf)} is the function $F_X:\mathbb R\to [0,1]$ given by 
		$F_X(x)=X_\sharp P((-\infty,x])$. 
	\end{definition}
	
	Notice that $F_X$ completely determines $X_\sharp P=P_X$. Moreover, 
	\begin{equation}\label{dist:F}
		F_X(x)=\int_{-\infty}^x dP_X, \quad x\in\mathbb R.
	\end{equation} 
	
	\begin{definition}
		We say that randon variables $X:\Omega\to \mathbb R^n$ and $Y:\Omega'\to\mathbb R^n$ are {\em identically distributed (i.d.)} if $P_X=P_Y$. 
	\end{definition}

	\begin{proposition}
		\label{expformdist}We have
		\[
		\mathbb E(X)=\int_{\mathbb  R^n}{\bf x}\,dP_X,
		\]
		where ${\bf x}=(x_1,\cdots,x_n)$ is the position vector. More generally, if $f:\mathbb R^n\to\mathbb R^p$ is measurable, so that $f(X)=f\circ X:\Omega\to\mathbb R^p$ is a random vector, then 
		\begin{equation}\label{exp:f:comp}
			\mathbb E(f(X))=\int_{\mathbb  R^n}f({\bf x})\,dP_X,
		\end{equation}
		where $f({\bf x})=f\circ {\bf x}$.
	\end{proposition}
	
		The notion of distribution may be used to single out two important classes of 
	random vectors. 
	
		\begin{definition}
		\label{desfunc} Let $X:\Omega\to\mathbb R^n$ be a random vector. 
		We say that
		\begin{itemize}
			\item $X$ is {\em discrete} 
			if its range ${\rm Ran}(X):=X(\Omega)\subset \mathbb R^n$ is countable:
			\[
			{\rm Ran}(X)=\left\{{\bf x}_1,{\bf x}_2,\cdots,{\bf x}_k,\cdots\right\}. 
			\]
			In this case, the map
			\begin{equation}\label{mdf:dic:dist}
				{\bf x}_j\in 	{\rm Ran}(X)\mapsto p_j:=P_X(\{{\bf x}_j\})\in\mathbb R, \quad j=1,2,\cdots,
			\end{equation}
			is called the {\em mass distribution function (mdf)} and satisfies 
			\[
			\sum_jp_j=1,
			\]
			with (\ref{exp:f:comp}) meaning that 
			\begin{equation}\label{exp:form:dis}
				\mathbb E(f(X))=\sum_j p_j f({\bf x}_j).
			\end{equation}
			If needed, in (\ref{mdf:dic:dist}) we may replace ${\rm Ran}(X)$ by ${\rm supp}(P_X)$, in which case each $p_j>0$. 
			\item  
			$X$ is {\em continuous} if $P_X$ is absolutely continuous with respect to the Lebesgue measure $d{\bf x}$. In this case, the Radon-Nykodim derivative
			\begin{equation}\label{radon:nyk}
				\psi_X:=\frac{d P_X}{d{\bf x}}:\mathbb R^n\to \mathbb R,
			\end{equation}
			is called the {\em probability density function (pdf)} of $X$, 
			with (\ref{exp:f:comp}) meaning that 
			\begin{equation}\label{exp:form}
				\mathbb E(f(X))=\int_{\mathbb R^n}f({\bf x})\psi_X({\bf x})d{\bf x}.
			\end{equation}
		\end{itemize}
	\end{definition}

	In case $X$ is real, we recall that the {\em support} of $P_X$ is given by
	\[
	{\rm supp}(P_X)=\{x\in \mathbb R; F_{X}(x+\varepsilon)-F_{X}(x-\varepsilon)>0\, {\rm for}\,{\rm all}\,\varepsilon>0\}.
	\]
	Also
	notice that, at least for a distribution whose support ${\rm supp}(P_X)$ is contained in some closed, bounded interval, the absolute continuity in the second item above means that the corresponding cdf $F_X$ is absolutely continuous, or equivalently, the following assertions hold:
	\begin{enumerate}
		\item $F_X'$ exists a.s. and is integrable (both with respect to Lebesgue measure);
		\item there holds
		\[
		F_X(b)-F_X(a)=\int_a^b F'_X(t)dt, 
		\]
		where $[a,b]\subset {\rm supp}(P_X)$\footnote{For proofs of these claims we refer to \cite[Sections 6.4, 6.5, 18.4 and 20.3]{royden2010real}.}.
	\end{enumerate} 
	In particular, 
	\begin{equation}\label{minus:inf}
	\int_{-\infty}^xdP_X\stackrel{(\ref{dist:F})}{=}F_X(x)=\int_{-\infty}^xF_X'(t)dt,\quad x\in {\rm supp}(P_X),
	\end{equation}
	so that from (\ref{radon:nyk}) we see that 
	\begin{equation}\label{dens:radon}
		\psi_X=F_X'\quad  {\rm a.s.}
	\end{equation} 
	Finally, note that from (\ref{exp:form}) with $f\equiv 1$ we get
	\begin{equation}\label{exp:form:=1}
		\int_{\mathbb R^n} \psi_X({\bf x})d{\bf x}=1,
		\end{equation}
as expected.		
Clearly, for $n=1$ this also follows from (\ref{minus:inf}) and (\ref{dens:radon}) with $x=+\infty$.
	
	\begin{convention}\label{conv:cont}
		Given a random vector $X:\Omega\to\mathbb R^n$, we will, unless otherwise specified, assume that it is continuous in the sense introduced above, so that it possesses a probability density function $\psi_X$. To ensure that the standard tools of calculus (including the fundamental theorem) may be applied, we further assume that $\psi_X$ is piecewise smooth with at most finitely many singularities.  
			It should be stressed, however, that virtually every statement involving integration in the continuous case can, when properly interpreted, be reformulated for the discrete case, and conversely. For example, the right-hand side of (\ref{exp:form:dis}) may be written as
		\[
		\int_{\mathbb R^n} f({\bf x})\, dP_X,
		\]
		the abstract Lebesgue integral of $f$ with respect to the discrete measure $dP_X$. Taking into account (\ref{radon:nyk}), this expression becomes formally indistinguishable from the ``continuous'' integral on the right-hand side of (\ref{exp:form}). This highlights the advantage of adopting Lebesgue integration from the outset in modern probability theory. 
	\end{convention}
	
	We now consider random vectors $X_j:(\Omega,\mathcal F)\to (\mathbb R^{p_j},\mathcal B^{p_j})$, $j=1,\cdots,n$ with distributions $P_{X_j}$. We may form the random vector 
	\[
	(X_1,\cdots,X_n):(\Omega,\mathcal F)\to (\mathbb R^{p_1}\times\cdots\times\mathbb R^{p_n}, \mathcal B^{p_1}\otimes\cdots\otimes\mathcal B^{p_n})
	\] given by $(X_1,\cdots,X_n)(\omega)=(X_1(\omega),\cdots,X_n(\omega))$, $\omega\in\Omega$, so that the {\em joint distribution} $P_{(X_1,\cdots,X_n)}$ on $\mathbb R^{p_1}\times\cdots\times\mathbb R^{p_n}$ is well defined. Moreover, each choice of $k$ distinct indexes, say  $I=\{i_1,\cdots,i_k\}\subset\{1,.\cdots,n\}$, determines a {\em marginal distribution} $P_{(X_{i_1},\cdots,X_{i_k})}$ induced by
	\[
	X_{(I)}=(X_{i_1},\cdots,X_{i_k}):(\Omega,\mathcal F)\to (\mathbb R^{p_{i_1}}\times\cdots\times\mathbb R^{p_{i_k}}, \mathcal B^{p_{i_1}}\otimes\cdots\otimes\mathcal B^{p_{i_k}}).
	\]
	In other words, we have the commutative diagram 
\[
\begin{tikzcd}[column sep=1.5cm, row sep=1.5cm]
	\Omega \arrow[r, "X"] \arrow[dr, "X_{(I)}"] & \mathbb R^{p_1}\times\cdots\times\mathbb R^{p_n} \arrow[d, "\pi_{(I)}"] \\
	& \mathbb R^{p_{i_1}}\times\cdots\times\mathbb R^{p_{i_k}}
\end{tikzcd}
\]
where $\pi_{(I)}$ is the associated projection. From this, the next result follows immediately.

	\begin{proposition}\label{pdf:marg}
		(Joint distribution and pdf of a marginal) With the notation above,
		\[
		P_{(X_{i_1},\cdots,X_{i_k})}(B)=P_{(X_1,\cdots,X_n)}(B\times\mathbb R^{p_{j_1}}\times\cdots\times\mathbb R^{p_{j_k}}), 
		\]
		where $B\in \mathcal B^{p_{i_1}}\otimes\cdots\otimes\mathcal B^{p_{i_k}}$ and $\{1,\cdots, n\}=\{i_1,\cdots,i_k\}\cup \{j_1,\cdots,j_{n-k}\}$, a disjoint union.
		In particular, in the continuos case, 
		\[
		\psi_{(X_{i_1},\cdots,X_{i_k})}({\bf x}_{i_1},\cdots,{\bf x}_{i_k})=\int_{\mathbb R^{p_{j_1}}\times\cdots\times\mathbb R^{p_{j_k}}}\psi_{(X_1,\cdots,X_n)}({\bf x}_1,\cdots,{\bf x}_n)d{\bf x}_{j_1}\cdots d{\bf x}_{j_{n-k}}.
		\]
	\end{proposition}

	The next result provides a way of handling independence which is quite satisfactory from an operational viewpoint. 
	
	\begin{proposition}
		\label{inddens}
		$\{X_j\}_{j=1}^n$ is independent if and only if
		\[
		P_{(X_1,\cdots,X_n)}=P_{X_1}\otimes\cdots\otimes P_{X_n},
		\]
		the product measure. Equivalently, in terms of the corresponding pdfs,  
		\[
		\psi_{(X_1,\cdots,X_n)}({\bf x}_1,\cdots,{\bf x}_n)=\psi_{X_1}({\bf x}_1)\cdots\psi_{X_n}({\bf x}_n), \quad ({\bf x}_1,\cdots,{\bf x}_n)\in \mathbb R^{p_1}\times\cdots\times\mathbb R^{p_n}.
		\]
	\end{proposition}

	If $\{X_j\}_{j=1}^n$ is independent, then, with the notation above, a straightforward
	combination of Propositions~\ref{inddens} and~\ref{pdf:marg} shows that for any
	$2\le k\le n$ and any collection of indices
	$\{i_1,\ldots,i_k\}\subset\{1,\ldots,n\}$, the subcollection
	$\{X_{i_l}\}_{l=1}^k$ is independent as well.
	Thus, independence at the level of the full collection $\{X_j\}_{j=1}^n$
	automatically propagates to all finite subcollections.
	This observation shows that the notion of independence for random variables
	in Definition~\ref{def:ind:levels}, item~(3), is fully consistent with (and may be
	recovered from) the distribution-based formulation given in
	Proposition~\ref{inddens}.

	\begin{remark}\label{one:line}
		We may now give a direct proof of Proposition  \ref{indexp}: since
		\[
		\mathbb E(XY)=\iint_{\mathbb R^2}xy\,dP_{(X,Y)}(x,y),
		\]
		it follows from Proposition \ref{inddens} and Fubini that 
		\begin{eqnarray*}
			\mathbb E(XY) 
			& = & \iint_{\mathbb R^2}xy\,dP_X(x)\otimes dP_Y(y)\\
			& = & \int_{\mathbb R}x\,dP_X(x)\int_{\mathbb R}y\,dP_Y(y)\\
			& = & \mathbb E(X)\mathbb E(Y),
		\end{eqnarray*}
		as claimed. \qed			
	\end{remark}

	\begin{remark}\label{comp:dist}
		If $X:\Omega\to\mathbb R$ is a random variable and $Y=X^2$ we may compute $\psi_Y$ in terms of $\psi_X$ as follows. Since $F_Y\leq y$ if and only if $-\sqrt{y}\leq F_X\leq\sqrt{y}$, it follows that 
		\[
		F_{X^2}(y)=\left(F_X(\sqrt{y})-F_X(-\sqrt{y})\right){\bf 1}_{[0,+\infty)]}(y),
		\]
		so that from (\ref{dens:radon}) we get
		\[
		\psi_{X^2}(y)=F_{X^2}'(y)=\frac{1}{2\sqrt{y}}\left(\psi_X(\sqrt{y})+\psi_X(\sqrt{-y})\right){\bf 1}_{(0,+\infty)]}(y).
		\]
		A similar computation shows that 
		\begin{equation}\label{regra:01}
			\psi_{\sqrt{X}}(x)=2x\psi_X(x^2){\bf 1}_{[0,+\infty)]}(x)
		\end{equation}
		if $X\geq 0$. Also, if $Z$ and $V$ are given with $V>0$ and $Z\perp\!\!\perp V$ then, in terms of the joint distribution $P_{(V,Z)}$,  
		\[
		F_{Z/ V}(x)=
		P(Z\leq x V)=\iint_{\{z\leq xv\}} dP_{(V,Z)}(v,z).
		\]
		By the independence and Proposition \ref{inddens} we may write this as an iterated integral, 
		\begin{eqnarray*}
			F_{Z/ V}(x)
			& = & 
			\int_0^{+\infty}\left(\int_{\{z\leq xv\}}dP_Z(z)\right) dP_V(v)\\
			& = &
			\int_0^{+\infty}F_Z(xv)\psi_V(v)dv,
		\end{eqnarray*}	
		and upon derivation we find that
		\begin{equation}\label{regra:02}
			\psi_{Z/V}(x)=\int_0^{+\infty}\psi_Z(xv)v\psi_V(v)dv.
		\end{equation}
		Under the same conditions we also have that
		\begin{equation}\label{regra:03}
			F_{Z V}(x)=\int_0^{+\infty}F_Z(xv^{-1})\psi_V(v)dv,
		\end{equation}
		and hence we obtain
		\[
		\psi_{ZV}(x)=\int_0^{+\infty}\psi_Z(xv^{-1})v^{-1}\psi_V(v)dv,
		\]
		again upon derivation \qed
	\end{remark}
	
	\begin{remark}\label{many:ind}(Drawing a random sample from a population)
		If $X_j:(\Omega_j,\mathcal F_j,P^{(j)})\to (\mathbb R,\mathcal B)$, $j=1,\cdots,n$,  are random variables, form the product probability space
		\[
		(\Omega^{\sharp},\mathcal F^{\sharp},P^{\sharp})=\otimes_j(\Omega_j,\mathcal F_j,P^{(j)}),
		\] 
		and define the random variables $Y_j:(\Omega^{\sharp},\mathcal F^{\sharp},P^{\sharp})\to (\mathbb R,\mathcal B)$, $Y_j=X_j\circ \pi_j$, where  
		$\pi_j:\mathbb R^n\to\mathbb R$ is the canonical projection onto the $j^{\rm th}$ factor. Now, if $B_j\in\mathcal B$ we have
		\begin{eqnarray*}
			P^{(j)}_{X_j}(B_j) 
			& = & P^{(j)}(X_j^{-1}(B_j))\\
			& = & P^\sharp(\Omega_1\times\cdots \times\Omega_{j-1}\times X_j^{-1}(B_j)\times\Omega_{j+1}\times\cdots\times \Omega_n)\\
			& = & P^\sharp (\pi_j^{-1}(X_j^{-1}(B_j)))\\
			& = & P^\sharp(Y_j^{-1}(B_j)),
		\end{eqnarray*}
		so that $P^{(j)}_{X_j}=P^\sharp_{Y_j}$ for any $j$. On the other hand,
		if $P^\sharp_Y$ is the joint distribution of the random vector $Y=(Y_1,\cdots,Y_n):(\Omega^{\sharp},\mathcal F^{\sharp},P^{\sharp})\to (\mathbb R^n,\mathcal B^n)$,
		\begin{eqnarray*}
			P_Y^\sharp(B_1\times \cdots\times B_n)
			& = & P^\sharp(Y^{-1}(B_1\times\cdots\times B_n))\\
			& = & P^\sharp(X_1^{-1}(B_1)\times\cdots\times X_n^{-1}(B_n))\\
			& = & \Pi_j P^{(j)}(X_j^{-1}(B_j))\\
			& = & \Pi_j P^{(j)}_{X_j}(B_j)\\
			& = & \Pi_j P^{\sharp}_{Y_j}(B_j)\\
			& = & (\otimes_jP^\sharp_{Y_j})(B_1\times \cdots\times B_n),
		\end{eqnarray*}
		so that  $P^\sharp_Y= \otimes_jP^\sharp_{Y_j}$. Thus, 	by Proposition \ref{inddens},  $\{Y_j\}_{j=1}^n$ is independent.   
		Note that if $\{X_j\}$ is identically distributed then $\{Y_j\}$ is identically distributed as well, so we conclude: {\em given any random variable $X$ there exist $\{Y_j\}_{j=1}^n$ which is independent and identically distributed to $X$}\footnote{The passage from $\Omega_j$ to $\Omega^\sharp$, a key ingredient in the argument above, is a simple instance of the extension dogma mentioned in Remark \ref{dogma}.}. 
		It turns out to be a bit more involved to extend this construction to {\em countably} many random variables, as this involves the consideration of a suitable {\em infinite} product of sample spaces \cite[Section~9.6]{fristedt2013modern}. In any case, this stronger assertion, which is a rather special case of Kolmogorov's extension (Theorem~\ref{extkolm}), may indeed be regarded as the {\em Fundamental Theorem of Statistical Inference}, since it provides the very mathematical foundation for the notion of an {\em infinite random sample} drawn from a given population, a terminology we shall use below without further notice. The existence of such a sequence $\{Y_j\}_{j\ge1}$ of i.i.d.\!\footnote{Here and in the sequel, i.i.d.\! is a shorthand for independent and identically distributed.} random variables is an indispensable assumption underlying virtually all limit results in Statistics---most notably, the {\em Law of Large Numbers} (Theorem \ref{lln}) and the {\em Central Limit Theorem} (Theorem \ref{clt})---and, more broadly, the very formulation of a {\em statistical model} in the sense introduced by Fisher (Definition \ref{stat:mod:def}), whereby one postulates a parametric family of probability laws governing the distribution of these random samples. In this way, the abstract (possibly infinite) product construction above bridges the axiomatic framework of Probability with the empirical foundations of Statistical Inference.
		\qed
	\end{remark}

\begin{definition}
	\label{covmat}
	If $X_1,\cdots,X_n:\Omega\to\mathbb R$ are random variables (equivalently, $X=(X_1,\cdots,X_n):\Omega\to\mathbb R^n$ is a random vector) we define its {\em covariance matrix} by
	\[
	{\mathbb C}(X)_{ij}={\mathbb C}(X_i,X_j),
	\]
	where 
	\[
	{\mathbb C}(X_i,X_j)=\mathbb E(X_iX_j)-\mathbb E(X_i)\mathbb E(X_j), 
	\]
	with $\mathbb C(X)$ thus being a symmetric matrix. 
\end{definition}

\begin{convention}\label{conv:tensor}
	Unless otherwise stated, we will always identify an $n$-vector with the corresponding $n\times 1$ matrix; in particular, any such vector will be viewed as a {\em column} vector. Thus, if ${\bf a}, {\bf b}\in\mathbb R^n$ and the superscript $\top$ represents transposition then ${\bf a}{\bf b}^\top$ is the $n\times n$ matrix defining the Kronecker (or outer) product of ${\bf a}$ and ${\bf b}$ (also denoted by ${\bf a}\otimes{\bf b}$), whereas ${\bf a}^\top{\bf b}=\langle {\bf a}, {\bf b}\rangle$ is their standard inner product. It then follows that a covariance matrix may be written as
	\[
	{\mathbb C}(X)
	=\mathbb E(XX^\top)-\mathbb E(X)\mathbb E(X^\top)=
	\mathbb E(X\otimes X)-\mathbb E(X)\otimes \mathbb E(X),
	\]
	where the expectation is extended componentwise to (possibly random) vectors and matrices.
	\end{convention}

We now discuss two simple variants of these definitions:
\begin{itemize}
	\item 
In case $n=1$, $X=X_1$, Definition \ref{covmat} gives rise to the {\em variance} of of the random variable $X$:
\begin{equation}\label{var:deff}
	{\mathbb V}(X)={\mathbb C}(X)=\mathbb E(X^2)-(\mathbb E(X))^2.
\end{equation}
Since ${\mathbb V}(X)\geq 0$ we usually set $\sigma^2:={\mathbb V}(X)$ and $\sigma:=\sqrt{{\mathbb V}(X)}$, the {\em standard deviation}, which we also denote by ${\rm sd}(X)$.
Note also  that 
\begin{equation}\label{pol:var:f}
{\mathbb V}\left(\sum_iX_i\right)=\sum_i{\mathbb V}(X_i)+\sum_{i\neq j}{\mathbb C}(X_i,X_j). 
\end{equation}
Thus, if the $X_i$'s
are pairwise uncorrelated (in particular, if they are independent) then
\begin{equation}\label{uncorr:var}
	{\mathbb V}\left(\sum_iX_i\right)=\sum_i{\mathbb V}(X_i).
\end{equation} 
To simplify the notation we sometimes denote
 ${\mathbb C}(X)$ by ${\mathbb V}(X)$ in case $X$ is vector valued. 
\item If $X,Y:\Omega\to \mathbb R^n$are random vectors then we define their {\em covariance matrix} by
\[
{\mathbb C}(X,Y)
=\mathbb E(XY^\top)-\mathbb E(X)\mathbb E(Y^\top)=
\mathbb E(X\otimes Y)-\mathbb E(X)\otimes \mathbb E(Y).
\]
Note that ${\mathbb C}(X,X)={\mathbb C}(X)$. Moreover, the exact analogue of (\ref{pol:var:f}) holds:
\[
{\mathbb C}\left(\sum_iX_i\right)=\sum_i{\mathbb C}(X_i)+\sum_{i\neq j}{\mathbb C}(X_i,X_j),
\]
where now each $X_i:\Omega\to\mathbb R^n$ is a random vector. 
\end{itemize}

In order to properly compare distinct random variables it is sometimes convenient to pass to a suitable normalization. In most cases, this is accomplished as follows.

\begin{definition}\label{stand:rv}
	If $X:\Omega\to\mathbb R$ is a random variable with $\mathbb E(X)=\mu$ and ${\mathbb V}(X)=\sigma^2$ then its {\em standardization} is
	\[
	Z=\frac{X-\mu}{\sigma}.
	\]
\end{definition}

Note that $\mathbb E(Z)=0$ and ${\mathbb V}(Z)=1$, hence the terminology.

Im many applications it is useful to estimate from above the tail probabilities of a random variable whose expectation/variance is known. We now present a couple of elementary results in this direction, which can be regarded as examples of (quite conservative) concentration inequalities.

\begin{proposition}\label{markov:ineq}(Markov's inequality)
	If $X:\Omega\to \mathbb R$ is a non-negative random variable and $a>0$ then 
	\begin{equation}\label{markov:ineq:2}
		P(X\geq a)\leq \frac{\mathbb E(X)}{a}.
	\end{equation}
\end{proposition}

\begin{proof}
	Using (\ref{exp:form}) with $f(x)=x$ we compute 
	\begin{eqnarray*}
		\mathbb E(X) & = & \int_0^{+\infty}x\psi_X(x)dx\\
		& \geq &  \int_a^{+\infty}x\psi_X(x)dx\\
		& = & a\int_a^{+\infty}\psi_X(x)dx\\
		& = & aP(X\geq a),
	\end{eqnarray*}
	as claimed.
\end{proof}

\begin{corollary}\label{cheby:ineq}(Chebyshev's inequality)
	Let $X:\Omega\to\mathbb R$ be a random variable with $0<\sigma^2:={\mathbb V}(X)<+\infty$. Then
	\begin{equation}\label{cheby:ineq:00}
		P(|X-\mathbb E(X|\geq a)\leq \frac{\sigma^2}{a^2}.
	\end{equation}
	Equivalently,  
	\begin{equation}\label{cheby:ineq:2}
		P(|X-\mathbb E(X|\geq c\sigma)\leq c^{-2}, \quad c>0. 
	\end{equation}
\end{corollary}

\begin{proof}
	Note that 
	\[
	P(|X-\mathbb E(X)|\geq a)=P(|X-\mathbb E(X)|^2\geq a^2)
	\]
	and use (\ref{markov:ineq:2}) with $X$ replaced by $|X-\mathbb E(X)|^2$ and $a$ replaced by $a^2$.
\end{proof}

We now discuss the various modes of convergence of random variables. 

\begin{definition}\label{conv:dist:def}
	Let $\{X_j\}_{j=1}^{+\infty}$ a sequence of random variables and let $X$ be another random variable (all defined on the same sample space $(\Omega,\mathcal F,P)$).
	We say that 
	\begin{itemize}
		\item $X_j$ converges to $X$ {\em almost surely} (notation: $X_j\stackrel{a.s.}{\to}X$) if
		\[
		P\left(\lim_{j\to+\infty}X_j=X\right)=1.
		\]
		\item $X_j$ converges to $X$ {\em in probability} (notation: $X_j\stackrel{p}{\to}X$) if, for any $\varepsilon>0$,
		\[
		\lim_{j\to+\infty}P\left(|X_j-X|<\epsilon\right)=1.
		\]
		\item $X_j$ converges to $X$ {\em in distribution} (notation: $X_j\stackrel{d}{\to}X$) if 
		\[
		\lim_{j\to+\infty}F_{X_j}(x)= F_X(x),
		\]
		for any $x\in\mathbb R$ where $F_X$ is continuous. Equivalently, $\mathbb E(\xi(X_j))\to\mathbb E(\xi(X))$ for all $\xi:\mathbb R\to\mathbb R$ uniformly bounded and continuous. 
		\item  $X_j$ converges to $X$ {\em in the mean} (notation: $X_j\stackrel{m}{\to}X$) if
		\[
		\lim_{j\to+\infty}\mathbb E(|X_j-X|^2)=0.
		\] 
	\end{itemize}
\end{definition}

Since, as memtioned in Definition \ref{distr}, $F_X$ as also known as the law of $X$, convergence in distribution is also referred to as convergence in law (notation: $X_j\stackrel{l}{\to} X$).
Also, the equivalence between the two ways above of defining convergence in distribution is part of the Portmanteau theorem \cite[Lemma 2.2]{van2000asymptotic}.

\begin{proposition}\label{modes}
	One has $(\stackrel{a.s.}{\to})\Rightarrow(\stackrel{p}{\to})\Rightarrow(\stackrel{d}{\to})$. Also, $(\stackrel{m}{\to})\Rightarrow (\stackrel{p}{\to})$ and $(\stackrel{d}{\to})\Rightarrow (\stackrel{p}{\to})$ if the limiting variable is constant.
\end{proposition}

The following quite useful result is worth mentioning here\footnote{We refer \cite[Chapter 5]{gut2006probability} for much more on the convergence properties of random variables.}.

\begin{theorem}\label{slutsky}(Slutsky)
	If $X_j\stackrel{d}{\to}X$ and $Y_j\stackrel{p}{\to}c$, $c\in \mathbb R$, then 
	$X_j+Y_j\stackrel{d}{\to}X+c$ and $X_jY_j\stackrel{d}{\to}cX$. Also, if $Y_j\neq 0$ and $c\neq 0$ then $X_j/Y_j\stackrel{d}{\to}X_n/c$. Finally, these assertions hold true if $(\stackrel{d}{\to})$ gets replaced by $(\stackrel{p}{\to})$ everywhere.
\end{theorem}

\subsection{The analytical setup: characteristic functions}\label{char:fct:b}
We now introduce an important notion which will allow us to make use of analytical techniques in the theory\footnote{Recall from Convention \ref{conv:tensor} that we denote the inner product of vectors ${\bf a},{\bf b}\in\mathbb R^n$ either by $\langle{\bf a},\bf{b}\rangle$ or by ${\bf a}^\top{\bf b}={\bf b}^\top{\bf a}$, where the superscript $\top$ indicates transpose. Also, we set $\|{\bf a}\|^2={\bf a}^\top{\bf a}$ for the corresponding squared norm.}.

\begin{definition}
	\label{funcchar}If $X:\Omega\to\mathbb R^n$ is a random vector
	then its {\em characteristic function} 
	$\phi_X:\mathbb R^n\to\mathbb C$ is given by
	\[
	\phi_X({\bf u})=\mathbb E(e^{{\bf i}\langle X,{\bf u}\rangle})=\int_{\mathbb R^n}e^{{\bf i} \langle{\bf x},{\bf u}\rangle}dP_X({\bf x}).
	\]
	If $X$ carries a pdf $\psi_X$ then 
	\[
	\phi_X({\bf u})=\int_{\mathbb R^n}e^{{\bf i}\langle {\bf x},{\bf u}\rangle}\psi_X({\bf x})d{\bf x}.
	\]
\end{definition}

\begin{remark}\label{dist:point:dist}
	It is immediate from Definition \ref{conv:dist:def} that $X_j\stackrel{d}{\to} X$ implies $\phi_{X_j}\to\phi_X$ pointwise. \qed
\end{remark}

Since $\phi_X$ is the (inverse) Fourier transform of $P_X$, we expect that it completely determines the corresponding cdf $F_X$. A proof of this general statement, at least in case $X$ is real, may be found in \cite[Section 39]{gnedenko2018theory}, where an explicit formula for $F_X$ in terms of $\phi_X$ is indicated; see also the discussion in \cite[Section 3.2]{lukacs1970charac}. We present here two instances where this expectation is confirmed (with explicit formulas).

\begin{proposition}\label{four:inv}
	The following hold:
	\begin{enumerate}
		\item If $X$ is $\mathbb Z$-valued and $p_k:=P(X=k)$, $k\in\mathbb Z$, then  
		\begin{equation}\label{four:inv:1}
			p_k=\frac{1}{2\pi}\int_{-\pi}^\pi e^{-{\bf i}ku}\phi_X(u)du,\quad k\in\mathbb Z.
		\end{equation}
		\item If $X$ is real and has a characteristic function $\phi_X:\mathbb R\to\mathbb C$ such that $|\phi_X|$ is integrable 
		then its distribution is absolutely continuous with respect to Lebesgue measure with the corresponding pdf being continuous and given by   
		\begin{equation}\label{four:inv:2}
			\psi_X(x)=\frac{1}{2\pi}\int_{-\infty}^{+\infty} e^{-{\bf i}xu}\phi_X(u)du, \quad x\in\mathbb R. 
		\end{equation}
	\end{enumerate}
\end{proposition}

\begin{proof}
	We only prove (\ref{four:inv:1}) here\footnote{A direct proof of the inversion formula (\ref{four:inv:2}) may be found in \cite[Chapter 13]{fristedt2013modern}; see also \cite[Theorem 3.2.2]{lukacs1970charac}.}. If ${\rm supp}\,P_X\subset\mathbb Z$ then it is immediate to check that $\phi_X$ is $2\pi$-periodic. Also,  
	\begin{equation}\label{four:inv:3}
		\phi_X(u)=
		\sum_{l\in\mathbb Z} e^{{\bf i}lu}p_l,
	\end{equation}
	where the convergence is uniform. In particular, $\phi_X$ is continuous. Now integrate over $[-\pi,\pi]$ the product of this series by $e^{-{\bf i}ku}$ and use the well-known orthogonality relations for the basis $\{e^{{\bf i}mx}\}_{m=-\infty}^{+\infty}$ in order to obtain (\ref{four:inv:1}). 
\end{proof} 

\begin{remark}\!\!$\bigstar$\label{four:t:dens}
	The inversion formula (\ref{four:inv:2}) means that $\psi_X=\widehat{\phi_X}$, where the hat means Fourier transform. On the other hand, (\ref{four:inv:3}) provides the Fourier series expansion of $\phi_X$ with Fourier coefficients given by (\ref{four:inv:1}). 
	\qed
\end{remark}

We now describe a simple condition on a random variable ensuring that its characteristic function is sufficiently regular.

\begin{proposition}\label{reg:char}
	If a random variable $X$ satisfies $\mathbb E(|X|^r)<+\infty$ for some $r\geq 1$ then $\phi_X\in C^r(\mathbb R)$ and 
	\[
	\phi_X^{(j)}(u)={\bf i}^j\mathbb E(X^je^{{\bf i}Xu}), \quad u\in\mathbb R,\quad j=1,\dots,r.
	\] 
	In particular, as $u\to 0$,
	\[
	\phi_X(u)=\sum_{j=0}^r\frac{{\bf i}^j}{j!}\mathbb E(X^j)u^j+o(|u|^r).
	\]
\end{proposition}

\begin{proof}
	If $r=1$ we have $\mathbb E(|X|)<+\infty$ and since
	\[
	\frac{e^{{\bf i}X(u+h)}-e^{{\bf i}Xu}}{h}={\bf i}Xe^{{\bf i}Xu}+o(h)
	\]
	we may use dominated convergence to see that
	\[
	\phi_X'(u)=\lim_{h\to 0}\mathbb E\left(\frac{e^{{\bf i}X(u+h)}-e^{{\bf i}Xu}}{h}\right)=
	{\bf i}\,\mathbb E\left(Xe^{{\bf i}Xu}\right),
	\]
	which proves this case. The general assertion for $r\geq 2$ follows by induction taking into account that 
	\[
	\frac{({\bf i}X)^je^{{\bf i}X(u+h)}-({\bf i}X)^je^{{\bf i}Xu}}{h}=({\bf i}X)^je^{{\bf i}Xu}+o(h)
	\]
	and that $\mathbb E(|X|^j)\leq \mathbb E(|X|^r)^{j/r}$ by H\"older inequality. 
\end{proof}

\begin{proposition}\label{charac:p:fol}
	A real random variable satisfies:
	\begin{enumerate}
		\item $\phi_{\alpha X}(u)=\phi_X(\alpha u)$, $\alpha\in \mathbb R$. In particular, $\phi_{-X}=\overline{\phi_X}$.
		\item If $\mu=\mathbb E(X)$ is finite then $\phi_X(u)=1+u\mu{\bf i} +o(|u|)$ as $u\to 0$. 
		Moreover, if $\mu=0$ and $\sigma^2=\mathbb E(X^2)$ is finite then 
		\[
		\phi_X(u)=1-\frac{1}{2}\sigma^2u^2+o(|u|^2)
		\]
	\end{enumerate}
\end{proposition}

\begin{proof}
	(1) is obvious and (2) is an immediate consequence of Proposition \ref{reg:char} (after Taylor expanding $\phi_X$ around $u=0$).
\end{proof}

We now examine how the characteristic functions of independent random variables contribute to the characteristic and density functions of their sum or difference.  

\begin{proposition}\label{charac:p}
	The following properties hold for independent real random variables $X$ and $Y$:
	\begin{enumerate}
		\item $\phi_{X+Y}=\phi_X\phi_Y$.
		\item $\psi_{X+Y}=\psi_X\star \psi_Y$, where $\star$ means convolution.
		\item moreover, if $X$ and $Y$ are identically distributed then $\phi_{X-Y}=|\phi_X|^2$. 
	\end{enumerate}
\end{proposition}

\begin{proof}
	For (1) note that, in terms of the joint distribution $P_{(X,Y)}$,
	\begin{eqnarray*}
		\phi_{X+Y}(u) 
		& = & \iint_{\mathbb R^2}e^{{\bf i}(x+y)u} dP_{(X,Y)}(x,y)\\
		& \stackrel{(*)}{=} & \iint_{\mathbb R^2}e^{{\bf i}xu}e^{{\bf i}yu}dP_X(x)\otimes dP_Y(y)\\
		& = & \phi_X(u)\phi_Y(u),
	\end{eqnarray*} 
	where we used Proposition \ref{inddens} in $(*)$ and Fubini in the last step. Also, by Remark \ref{four:t:dens},
	\begin{eqnarray*}
		\psi_{X+Y} & = & \widehat{\phi_{X+Y}}\\
		& \stackrel{(3)}{=} & \widehat{\phi_X\phi_Y} \\
		& \stackrel{(**)}{=} & \widehat{\phi_X}\star \widehat{\phi_X}\\
		& = & \psi_X\star\psi_Y,
	\end{eqnarray*}
	where we used a well-known property of the Fourier transform in $(**)$. Finally, (3) follows from (1) and Proposition \ref{charac:p:fol} (1).
\end{proof}

\begin{remark}\label{cont:prods}
The clear contrast between the two types of products that appear on the right-hand sides of items (1) and (2) above already indicates why it is often preferable to work with characteristic functions rather than with the pdfs themselves; see Remark \ref{rem:dir:ind:n} for an illustrative example.	
	\end{remark}

\begin{definition}\label{symm:def:new}
	A random variable $X$ is {\em symmetric} (about $0$) if $X$ and $-X$ are identically distributed.
\end{definition}

\begin{proposition}\label{symm:prop:real}
	$X$ is symmetric if and only if $\phi_X$ is $\mathbb R$-valued, in which case
	there holds 
	\begin{equation}\label{symm:prop:real:2}
		\phi_X(u)=\mathbb E(\cos(Xu)).
	\end{equation}
\end{proposition}

\begin{proof}
	Immediate from the previous results.
\end{proof}

\begin{definition}\label{radem:def}
	We say that ${\bm\epsilon}$ is a {\em Rademacher variable} if ${\rm supp}\,P_{\bm\epsilon}=\{-1,1\}$ with $P({\bm\epsilon}=-1)=P({\bm\epsilon}=1)=1/2$.  
\end{definition}

\begin{proposition}\label{radem:sym:eq}
	If $\{\bm\epsilon, X\}$ is independent with $X$ symmetric then $X$ and $\bm\epsilon X$ are identically distributed. 
\end{proposition}

\begin{proof}
	The cdf of ${\bm\epsilon X}$ is 
	\begin{eqnarray*}
		F_{\bm\epsilon X}(x) 
		& = &
		P\left({\bm\epsilon} X\leq x\right)\\
		& = & P\left(\{X\leq x\}\cap\{\bm\epsilon=1\}\right) +
		P\left(\{-X\leq x\}\cap\{\bm\epsilon=-1\}\right),
	\end{eqnarray*}
	so independence gives 
	\begin{eqnarray*}
		F_{\bm\epsilon X}(x)
		& = & P\left(X\leq x\right)P\left({\bm\epsilon}=1\right))+
		P\left(-X\leq x\right)P\left({\bm\epsilon}=-1\right))	\\ 
		& = &\frac{1}{2}\left(F_X(x)+F_{-X}(x)\right)\\
		&  = & F_X(x),
	\end{eqnarray*}
	where in the last step we used that $F_X=F_{-X}$. 
\end{proof}

We now introduce another important notion which is closely related to characteristic functions. 

\begin{definition}\label{mgf:def}
	The {\em moment generating function} (mgf) of a random vector $X:\Omega\to\mathbb R^n$ is given by \begin{equation}\label{moment:2}
		\varphi_X({\bf u})=\mathbb E(e^{\langle X,{\bf u}\rangle}), \quad {\bf u}\in\mathbb R^n,
	\end{equation}
whenever the right-hand side is finite.	
\end{definition}

\begin{remark}\label{mgf:e:v}
	We will always assume that $\varphi_X$ is defined at least in a neighborhood $V\subset \mathbb R^n$ of the origin,  which happens if $\phi_X$ is analytic there \cite[Section 7.2]{lukacs1970charac}. In this case we have $\varphi_X({\bf u})=\phi_X(-{\bf i}{\bf u})$, ${\bf u}\in V$, a replacement we shall use in the sequel without further notice. Under these conditions, if $X\in\mathbb R$ then all the {\em moments} of $X$,
	\[
	\alpha_k(X):=\int_{-\infty}^{+\infty}x^kdP_X(s), \quad k=0,1,2,\dots,
	\]  
	are finite 
	with 
	\[
	\beta^{-1}:=\limsup_k\left(\frac{\alpha_k(X)}{k!}\right)^{1/k}<+\infty,
	\]
	so
	there holds
	\begin{equation}\label{mom:gen:exp}
		\varphi_X(u)=\sum_k\frac{\alpha_k(X)}{k!}u^k, \quad u\in (-\beta,\beta).
	\end{equation}
	Thus, $\alpha_k(X)=\varphi_X^{(k)}(0)$, which justifies the mgf terminology. 
	Note that the expectation and variance of $X$ are given by 
	\begin{equation}\label{mgf:e:v:2}
		\mathbb E(X)=\varphi_X'(0), \quad {\mathbb V}(X)=\varphi_X''(0)-(\varphi_X'(0))^2,
	\end{equation}
	with similar formulae holding for higher order centered moments.\qed
\end{remark}

\begin{example}\label{bern:trial}
	(Binomial trials as the sum of independent Bernoulli trials) 
	Set  $\mathbb N^{(n)}:=\{0,1,\cdots,n\}$, $n\geq 1$, and consider a discrete random variable $X$ whose probability distribution is supported on $\mathbb N^{(n)}$ and satisfies, for some $0<p<1$,
	\begin{equation}\label{bef:sum}
		P(X=k)=
		\binom{n}{k}
		p^k(1-p)^{n-k}, \quad k\in \mathbb N^{(n)}.
	\end{equation}
	We then say that $X\sim \mathsf{Bin}(p;n)$, the {\em binomial distribution}  determined by the pair $(p,n)$. 
	Using that the characteristic function of $X$ is 
	\begin{eqnarray*}
		\phi_{X}(u)
		& = & \sum_{k=0}^ne^{{\bf i}ku} 
		\binom{n}{k}
		p^k(1-p)^{n-k}\\
		& = & \sum_{k=0}^n 
		\binom{n}{k}
		(p e^{{\bf i}u})^k(1-p)^{n-k},  
	\end{eqnarray*}
	we obtain
	\begin{equation}\label{form:ch:sumb}
		\phi_{X}(u)=(1-p+pe^{{\bf i}u})^n. 
	\end{equation}
	Together with Proposition \ref{inddens}, Proposition \ref{charac:p} (1) and Proposition \ref{four:inv} (1), this shows  that, by possibly changing the underlying sample space, we may assume that $X= X_1+\cdots+X_n$, where $\{X_j\}_{j=1}^n$ is independent and each $X_j\sim \mathsf{Bin}(p;1)=:\mathsf{ Ber}(p)$, the {\em Bernoulli distribution}\footnote{This kind of distribution models {any} binary random experiment (such as  coin toss, for instance) with $k=1$ corresponding to success and $k=0$ corresponding to failure.}, so that $\mathbb E(X_j)=p$ and ${\mathbb V}(X_j)=p(1-p)$. Finally, 
	\begin{equation}\label{form:ch:sumb:2}
		\varphi_{X}(u)=(1-p+pe^{u})^n
	\end{equation}
	follows immediately from (\ref{form:ch:sumb}).
	\qed
\end{example}

\begin{example}\label{poisson:trials}
	(Poisson trials). 
	For each $n\geq 1$ consider a discrete random variable $Y$ whose probability distribution is
	supported on $\mathbb N_0=\{0\}\cup\mathbb N$ with 
	\[
	P(Y=k)=\frac{\lambda^ke^{-\lambda}}{k!}. 
	\]
	We represent this as $Y\sim \mathsf{Pois}(\lambda)$, the {\em Poisson distribution} with parameter $\lambda$.
	We compute: 
	\begin{eqnarray*}
		\phi_{Y}(u)
		& = & \sum_{k\geq 0} e^{{\bf i}ku}\frac{\lambda^ke^{-\lambda}}{k!} \\
		& = & e^{-\lambda}\sum_{k\geq 0}\frac{(\lambda e^{{\bf i}u})^k}{k!}\\
		& = & e^{-\lambda}e^{\lambda e^{{\bf i}u}},
	\end{eqnarray*}
	which gives 
	\begin{equation}\label{char:pois}
		\phi_{Y}(u)=e^{\lambda(e^{{\bf i}u}-1)},
	\end{equation}
	and hence
	\begin{equation}\label{char:pois:2}
		\varphi_{Y}(u)=		e^{\lambda(e^{u}-1)}.
	\end{equation}
	In particular, $\mathbb E(Y)={\mathbb V}(Y)=\lambda$. Also, if $Y\sim \mathsf{Pois}(n)$, $n\in\mathbb N$, then it follows from (\ref{char:pois}) with $\lambda=n$ that 
	we may decompose $Y=Y_1+\cdots+Y_n$ with $\{Y_j\}_{j=1}^n$ independent and each $Y_j\sim \mathsf{Pois}(1)$, so that $\mathbb E(Y_j)={\mathbb V}(Y_j)=1$. \qed
\end{example}

\begin{example}\label{geom:dist}(The geometric distribution)
Let $Z$ be a discrete random variable whose probability distribution is supported on $\mathbb N$ with
\[
	P(Z=k)=(1-p)^{k-1}p, \quad k\geq 1,
	\]
	which gives the probability that the first occurrence of success in a sequence of independent Bernoulli trials as in Example \ref{bern:trial} requires exactly $k$ steps. This is called the {\em geometric distribution} for an obvious reason:
	\[
	\sum_{k\geq 1}	P(Z=k)=\frac{p}{1-p}\sum_{k\geq 1}(1-p)^k=\frac{p}{1-p}\frac{1-p}{1-(1-p)}=1.
	\]
	Similarly,
	\begin{eqnarray*}
		\phi_Z(u)
		& = & \sum_{k\geq 1} (1-p)^{k-1}pe^{{\bf i}u}\\
		& = & \frac{p}{1-p}\sum_{k\geq 1}\left((1-p)e^{{\bf i}u}\right)^k\\
		& = & \frac{pe^{{\bf i}u}}{1-(1-p)e^{{\bf i}u}},
	\end{eqnarray*}
	so that 
	\[
	\varphi_Z(u)=\frac{p}{p-1+e^{-u}}, \quad u<-\ln (1-p).
	\]
	From this and (\ref{mgf:e:v:2}) we easily see that
	$\mathbb E(Z)={1}/{p}$ and ${\mathbb V}(Z)=({1-p})/{p^2}$.
	\qed
	\end{example}

\section{Conditioning}\label{cond:poss}

We now discuss the various ways of conditioning a given random variable. 

\subsection{Conditional probability}\label{cond:prob}
Let $X:\Omega\to\mathbb R^m$ and $Y:\Omega\to\mathbb R^p$ be random vectors with distributions $P_X$ and $P_Y$, respectively, defined on a probability space $(\Omega,\mathcal F,P)$. 
We denote by $P_{(X,Y)}$ the joint distribution of $(X,Y):\Omega\to\mathbb R^m\times\mathbb R^p$. 
As usual, we assume that all these distributions are absolutely continuous with respect to the Lebesgue measure and therefore admit probability density functions (pdfs). 
Given $B\in\mathcal B^p$, our aim is to define the {\em conditional probability} that $Y\in B$ given a realization ${\bf x}\in\mathbb R^m$ of $X$. 

\begin{definition}\label{trans:fc}
	A {\em Markov kernel} is a map $\kappa:\mathbb R^m\times\mathcal B^p\to [0,1]$ such that:
	\begin{itemize}
		\item for each $B\in\mathcal B^p$, the map ${\bf x}\mapsto \kappa({\bf x},B)$ is $\mathcal B^m$-measurable;
		\item for each ${\bf x}\in\mathbb R^m$, the map $B\mapsto \kappa({\bf x},B)$ is a probability measure on $(\mathbb R^p,\mathcal B^p)$.
	\end{itemize}
\end{definition}

Given a Markov kernel $\kappa$ and a probability measure $\mu$ on $(\mathbb R^m,\mathcal B^m)$, the rule  
\[
(A,B)\longmapsto (\mu\star\kappa)(A,B):=\int_A\kappa({\bf x},B)\,d\mu({\bf x}), 
\qquad (A,B)\in \mathcal B^m\times\mathcal B^p,
\]
defines a probability measure on $(\mathbb R^m\times\mathbb R^p,\mathcal B^m\times\mathcal B^p)$. The following result, known as the {\em disintegration theorem}, asserts the existence of a unique Markov kernel that plays the role of a ``conditional quotient'' of $P_{(X,Y)}$ by $P_X$ under the convolution operation $\star$.

\begin{proposition}\label{exist:trans}
	There exists a unique Markov kernel $P_{Y|X}$ such that 
	\[
	P_{(X,Y)}=P_X\star P_{Y|X}.
	\]
	Equivalently,
	\begin{equation}\label{exist:trans:02}
		P_{(X,Y)}(A,B)=\int_A P_{Y|X}({\bf x},B)\,dP_X({\bf x}), 
		\qquad (A,B)\in \mathcal B^m\times\mathcal B^p. 
	\end{equation}
\end{proposition}

\begin{proof}
	See \cite[Chapter~8]{klenke2013probability}.
\end{proof}

The kernel $P_{Y|X}$ thus provides the precise object that realizes, in measure-theoretic terms, the intuitive idea of conditioning on a given value of $X$.

\begin{definition}\label{prob:trans:x}
	If ${\bf x}\in \mathbb R^m$ we define the {\em conditional probability} by
	\[
	P_{Y|_{X={\bf x}}}=P_{Y|X}({\bf x},\cdot),
	\]
	which is a probability measure in $(\mathbb R^p,\mathcal B^p)$. 
\end{definition}

Thus, 
\begin{equation}\label{cond:prob:22}
	P(Y\in B|_{X={\bf x}}):=	P_{Y|_{X={\bf x}}}(B)=P_{Y|X}({\bf x},B), \quad B\in \mathcal B^p,
\end{equation}
should be interpreted as the {\em conditional probability} that $Y\in B$ given that $X={\bf x}$.
It is immediate from (\ref{exist:trans:02}) that the corresponding pdf's satisfy 
\begin{equation}\label{p:y:x=x}
	\psi_{Y|_{X={\bf x}}}({\bf y})=\frac{\psi_{(X,Y)}({\bf x},{\bf y})}{\psi_X({\bf x})}, \quad {\bf y}\in \mathbb R^p,
\end{equation}
whenever $\psi_X({\bf x})>0$, 
so that the corresponding {\em conditional expectation function} and {\em conditional covariance function} are 
\begin{equation}\label{exp:x}
	\mathbb E(Y|_{X={\bf x}})=\int_{\mathbb R^p}{\bf y}\psi_{Y|_{X={\bf x}}}({\bf y})d{\bf y} 
\end{equation}
and 
\begin{equation}\label{exp:v}
	{\mathbb C}(Y|_{X={\bf x}})=\int_{\mathbb R^p}({\bf y}-\mathbb E(Y|_{X={\bf x}})^2\psi_{Y|_{X={\bf x}}}({\bf y})d{\bf y}, 
\end{equation}
respectively.

\begin{remark}\label{dispens}
	Whenever possible, we may simply dispense with the existence theory sketched above and adopt (\ref{p:y:x=x}) as the definition of the {\em conditional pdf} of $Y$ given $X={\bf x}$.
	\qed
\end{remark}

We now turn to a few elementary yet useful consequences of the preceding theory. 
For simplicity, throughout the remainder of this subsection we assume that all random variables are real-valued, continuous, and possess probability density functions that are strictly positive on their domains. 
In particular, with a slight abuse of notation, (\ref{p:y:x=x}) may be interpreted as the pdf of the \emph{conditioned random variable} $Y|_{X={\bf x}}$, thus 
allowing the natural extension of the concepts introduced so far for random variables to this broader setting.

\begin{proposition}\label{mean:cond}
	The following hold:
	\begin{enumerate}
		\item If $\{X|_{Z={\bf z}},Y|_{Z={\bf z}}\}$ is independent for a given ${\bf z}$ then 
		\[
		\mathbb E(XY|_{Z={\bf z}})=\mathbb E(X|_{Z={\bf z}})\mathbb E(Y|_{Z={\bf z}});
		\]
		\item If $\{X,Y\}$ is independent then  $\mathbb E(Y)=\mathbb E(Y|_{X={\bf x}})$ for any ${\bf x}$; 
		\item If $\{X,Y,Z\}$ is independent then $\{X|_{Z={\bf z}},Y|_{Z={\bf z}}\}$ is independent for any ${\bf z}$.
	\end{enumerate}
\end{proposition}

\begin{proof}
	(1) is the ``conditional'' version of Proposition \ref{indexp}, so its proof follows as in Remark \ref{one:line} by noticing that, under the conditional independence assumption, there holds 
	\[
	\psi_{(X|_{Z={\bf z}},Y|_{Z={\bf z}})}({\bf x},{\bf y})=\psi_{X|_{Z={\bf z}}}({\bf x})
	\psi_{Y|_{Z={\bf z}}}({\bf y}),
	\]
	the ``conditional'' analogue of Proposition \ref{inddens}.
	For (2) note that by (\ref{p:y:x=x}) and Proposition \ref{inddens},
	\begin{equation}\label{ind:con:st}
		\psi_{Y|_{X={\bf x}}}({\bf y})=\frac{\psi_{(X,Y)}({\bf x},{\bf y})}{\psi_X({\bf x})}=\frac{\psi_X({\bf x})\psi_Y({\bf y})}{\psi_X({\bf x})}=
		\psi_Y({\bf y}).
	\end{equation}	
	As for (3), again by Proposition \ref{inddens},
	\begin{eqnarray*}
		\psi_{\left({X|_{Z={\bf z}}},{Y|_{Z={\bf z}}}\right)}({\bf x},{\bf y})
		& = & \frac{\psi_{(X,Y,Z)}({\bf x},{\bf y},{\bf z})}{\psi_Z({\bf z})}\\
		& = & 
		\frac{\psi_X({\bf x})\psi_Y({\bf y})\psi_Z({\bf z})}{\psi_Z({\bf z})}\\
		& \stackrel{(\ref{ind:con:st})}{=} &
		\psi_{X|_{Z={\bf z}}}({\bf x}) \psi_{Y|_{Z={\bf z}}}({\bf y}),
	\end{eqnarray*}
	and the result follows from the ``conditional'' version of Proposition \ref{inddens}.
\end{proof}		

\begin{remark}\label{rem:id:as}
	In general the converses to both items (2) and (3) in Proposition \ref{mean:cond} fail to hold true if the independence assumptions are removed. \qed
\end{remark}

Finally, we present another consequence of (\ref{p:y:x=x}) with notable applications to the so-called Bayesian approach to Statistical Inference; see Subsection \ref{bay:way}.	

\begin{theorem}\label{bayes}(Bayes rule)
	If both $\psi_X$ and $\psi_Y$ are everywhere positive then 
	\[
	\psi_{X|_{Y={\bf y}}}({\bf x})=\frac{\psi_{Y|_{X={\bf x}}}({\bf y})\psi_X({\bf x})}{\psi_Y({\bf y})}, \quad ({\bf x},{\bf y})\in\mathbb R^m\times\mathbb R^q,
	\]
	with 
	\[
	\psi_Y({\bf y})=\int_{\mathbb R^m}\psi_{Y|_{X={\bf x}}}({\bf y})\psi_X({\bf x})d{\bf x}.
	\]
\end{theorem}

\begin{proof}
	It follows from (\ref{p:y:x=x}) that 
	\[
	{\psi_{(X,Y)}({\bf x},{\bf y})}=
	\psi_{Y|_{X={\bf x}}}({\bf y})\psi_X({\bf x})
	=
	\psi_{X|_{Y={\bf y}}}({\bf x})\psi_Y({\bf y}),
	\]
	and the result follows.
\end{proof}

\subsection{Conditional expectation}\label{cond:expc}

Here we discuss how to condition a random variable with respect to a $\sigma$-subalgebra and then relate this to the discussion in the previous subsection (via Proposition \ref{int:prob}).

\begin{proposition}
	\label{condexp} Let $(\Omega,\mathcal F,P)$ be a probability space and let $\mathcal G\subset\mathcal F$ a $\sigma$-subalgebra. Given a random vector $X:\Omega\to\mathbb R^n$ there exists a unique random vector $Y:\Omega\to\mathbb R^n$ which is $\mathcal G$-measurable and satisfies 
	\[
	\int_{G}YdP=\int_G XdP,\quad G\in\mathcal G.
	\]
\end{proposition}

\begin{proof}
	Define a measure $Q$ on $\mathcal P$ by
	\[
	Q(G)=\int_GXdP,\quad G\in \mathcal G. 
	\] 
	Clearly, $Q$ is absolutely continuous with respect to $P|_{\mathcal G}$. Now take 
	$Y=dQ/dP|_{\mathcal G}$. 	
\end{proof}

\begin{remark}\label{l2}
	Given $X\in L^2(\Omega,\mathcal F,P)$ consider the closed subspace $L^2(\Omega,\mathcal G,P|_{\mathcal G})\subset L^2(\Omega,\mathcal F,P)$ and let $\pi:L^2(\Omega,\mathcal F,P)\to L^2(\Omega,\mathcal G,P|_{\mathcal G})$ be the standard orthogonal projection. It then follows that $Y=\pi X$. \qed
\end{remark}

\begin{definition}\label{cond:exp:def}
	We call $Y=\mathbb E(X|\mathcal G)$ the {\em conditional expectation} of $X$ given $\mathcal G$. If $\mathcal G=\mathcal F_Z$ for some other $Z$ then we set
	$\mathbb E(X|Z):=\mathbb E(X|\mathcal F_Z)$.
\end{definition}

Note that $\mathbb E(X|\mathcal G)$ is characterized by
\begin{equation}\label{charac:cd}
	\int_G\langle Z,\mathbb E(X|\mathcal G)\rangle dP=\int_G\langle Z,X\rangle dP,\quad G\in\mathcal G,
\end{equation}
for any $Z:\Omega\to\mathbb R^n$ $\mathcal G$-measurable. 

\begin{proposition}\label{ceprop}
	Conditional expectation satisfies the following properties:
	\begin{enumerate}
		\item $\mathbb E(aX+bX'|\mathcal G)=a\mathbb E(X|\mathcal G)+b\mathbb E(X'|\mathcal G)$;
		\item $\mathbb E(\mathbb E(X|\mathcal G))=\mathbb E(X)$;
		\item if $X$ is $\mathcal G$-measurable then $\mathbb E(X|\mathcal G)=X$;
		\item if $X\perp\!\!\perp\mathcal G$ then $\mathbb E(X|\mathcal G)=\mathbb E(X)$;
		\item if $\mathcal G\subset\mathcal H$ then $\mathbb E(X|\mathcal G)=\mathbb E(\mathbb E(X|\mathcal H)|\mathcal G)$; 
		\item if $X\perp\!\!\perp\mathcal G$ 
		then $\mathbb E(XY|\mathcal G)=\mathbb E(X)\mathbb E(Y|\mathcal G)$.
		In particular, if $Y$ is $\mathcal G$-measurable then 
		$\mathbb E(XY|\mathcal G)=Y\mathbb E(X)$;
		\item if $X$ is $\mathcal G$-measurable then $\mathbb E(XY|\mathcal G)=XE(Y|\mathcal G)$.
	\end{enumerate}
\end{proposition}
\begin{proof}
	(1) and (2) are obvious. For (3), just think of $X:(\Omega,\mathcal G)\to\mathbb R^n$ as a random vector. For (4),
	\begin{eqnarray*}
		\int_GXdP & = & \int_{\Omega} X{\bf 1}_G dP\\
		& \stackrel{(*)}{=} & \int_{\Omega} XdP\int_{\Omega}\chi_GdP\\
		& = & \mathbb E(X)P(G)\\
		& = & \int_G\mathbb E(X)dP,
	\end{eqnarray*}
	where the assumption was used in $(*)$. The result then follows by uniqueness. For (5), note that $G\in\mathcal G$ implies $G\in\mathcal H$ and hence
	\[
	\int_G\mathbb E(X|\mathcal H)dP=\int_GXdP=\int_G\mathbb E(X|\mathcal G)dP.
	\]
	Also, (6) is the obvious generalization of (4): using that 
	$X\perp\!\!\perp Y{\bf 1}_G$, $G\in\mathcal G$, we see that 
	\[\mathbb E(XY|\mathcal G)=
	\mathbb E(X|\mathcal G)\mathbb E(Y|\mathcal G)=
	\mathbb E(X)\mathbb E(Y|\mathcal G).
	\]
	 Finally, if $G\in\mathcal G$,
	\begin{eqnarray*}
		\int_G X\mathbb E(Y|\mathcal G) dP
		& \stackrel{(3)}{=} & 
		\int_G 
		\mathbb E(X|\mathcal G)
		\mathbb E(Y|\mathcal G) dP \\
		& \stackrel{(\ref{charac:cd})}{=} & 
		\int_G 
		\mathbb E(X|\mathcal G)
		Y dP\\
		& \stackrel{(3)}{=} &
		\int_G 
		X
		Y dP,
	\end{eqnarray*}
	which proves (7).
\end{proof}

\begin{example}\!\!$\bigstar$\label{birk} (Birkhoff ergodic theorem)
	If $(\Omega,\mathcal F,P)$ is a probability space then $T:\Omega\to\Omega$ is {\em measure preserving} if $T$ is $\mathcal F$-measurable and satisfies  $P(T^{-1}(A))=P(A)$ for any event $A\in\mathcal F$. Given a random variable $X:\Omega\to\mathbb R$ we then define, for $n\in\mathbb N$, $X^{(T)}_n:\Omega\to\mathbb R$ by
	\[
	X^{(T)}_n(\omega)=\frac{1}{n}\sum_{j=0}^{n-1}X(T^j\omega). 
	\]
	A version of {\em Birkhoff's ergodic theorem} \cite[Section 1.2]{krengel2011ergodic} says that there exists a random variable $X^{(T)}$ such that:
	\begin{equation}\label{birk:1}
		P\left(\omega\in\Omega;X^{(T)}_n(\omega)\to_{n\to+\infty} X^{(T)} (\omega)\right)=1,
	\end{equation}
	from which it follows that  $X^{(T)}\circ T=X^{(T)}$ and
	$\mathbb E(X^{(T)})=\mathbb E(X)$.
	In order to identify $X^{(T)}$ let us consider 
	\[
	\mathcal G_T=\left\{A\in\mathcal F;T^{-1}(A)=A\right\},
	\]
	the $\sigma$-subalgebra of $T$-{\em invariant} events. If $G\in \mathcal G_T$ define $X_G:={\bf 1}_GX$. Thus, 
	\[
	X_G(T^j\omega)={\bf 1}_G(T^j\omega)X(T^j\omega)={\bf 1}_G(\omega)X(T^j\omega),
	\] 
	so if we use (\ref{birk:1}) with $X_G$ replacing $X$ we see that $X_G^{(T)}={\bf 1}_GX^{(T)}$ and hence $\mathbb E({\bf 1}_GX^{(T)})=\mathbb E(X_G)=\mathbb E({\bf 1}_GX)$, which means that $X^{(T)}=\mathbb E(X|\mathcal G_T)$.  
	Also, if $T$ is {\em ergodic} in the sense that 
	\[
	\mathcal G_T=\{A\in\mathcal F; P(A)=0\,{\rm or}\,P(A)=1\}
	\] 
	then it is immediate to check that $X\perp\!\!\perp\mathcal G_T$ and it follows from Proposition \ref{ceprop} (4) that 
	\[
	P\left(\omega\in\Omega;X^{(T)}_n(\omega)\to_{n\to+\infty}\mathbb E(X)\right)=1,
	\]
	which is an improvement of (\ref{birk:1}).
	\qed
\end{example}

\begin{example}\!\!$\bigstar$\label{von} (von Neumann ergodic theorem)
	The $L^2$ version of Example \ref{birk} goes as follows. For any measure preserving $T$ as above let us consider
	\[
	L^2\left(\Omega,\mathcal G_T,P|_{\mathcal G_T}\right)=\left\{\widetilde X\in L^2(\Omega,\mathcal F,P); \widetilde X\circ T=\widetilde X\right\}. 
	\]
	Then {\em von Neumann's mean ergodic theorem} \cite[Section 1.1]{krengel2011ergodic} assures that for any random variable $X\in L^2(\Omega,\mathcal F,P)$ there exists a unique $X^{[T]}\in L^2\left(\Omega,\mathcal G_T,P|_{\mathcal G_T}\right)$ such that
	\begin{equation}\label{von:1}
		\lim_{n\to +\infty}\mathbb E(|X_n^{(T)}-X^{[T]}|^2)=0. 
	\end{equation}
	Using that 
	\[
	\iota_T:L^2(\Omega,\mathcal F,P)\to L^2(\Omega,\mathcal F,P), \quad \iota_T(X)=X\circ T,
	\]
	is an isometry 
	we compute, for any $\widetilde X\in L^2(X,\mathcal G_T,P|_{\mathcal G_T})$, 
	\begin{eqnarray*}
		\mathbb E(X^{[T]}\widetilde X) 
		& = & \lim_{n\to+\infty}\frac{1}{n}\sum_{j=0}^{n-1}\mathbb E\left((X\circ T^j)\widetilde X\right)\\
		& = & \lim_{n\to+\infty}\frac{1}{n}\sum_{j=0}^{n-1}\mathbb E\left((X\circ T^j)(\widetilde X\circ T^j)\right)\\
		& = & \lim_{n\to+\infty}\frac{1}{n}\sum_{j=0}^{n-1}\mathbb E\left(X\widetilde X\right)\\
		& = & \mathbb E(X\widetilde X),
	\end{eqnarray*}
	which means that $X^{[T]}$ is the $L^2$ projection of $X$ over $L^2\left(\Omega,\mathcal G_T,P|_{\mathcal G_T}\right)$. It follows from Remark \ref{l2} that $\overline X^{[T]}=\mathbb E(X|\mathcal G_T)$, so that (\ref{von:1}) is the ``mean squared'' version of (\ref{birk:1}). \qed
\end{example}

We now give an useful rewording of the random variable induced by the conditional expectation function (\ref{exp:x}) in terms of the notion of conditional expectation appearing in Definition \ref{cond:exp:def}.

\begin{proposition}\label{int:prob}
	If, as in (\ref{exp:x}), we set 
	\begin{equation}\label{int:prob:0}
		g({\bf x})=\mathbb E(Y|_{X={\bf x}}), \quad {\bf x}\in\mathbb R^m,
	\end{equation}
	then
	\begin{equation}\label{int:prob:1}
		g(X)=\mathbb E(Y|X).
	\end{equation}
	In particular, 
	\begin{equation}\label{int:prob:2}
		\mathbb E(g(X))=\mathbb E(Y).
	\end{equation}
\end{proposition} 

\begin{proof}
	We need to check that 
	\[
	\int_Cg(X)dP=\int_CYdP, \quad C\in \mathcal F_X,
	\]
	so we write	
	$C=X^{-1}(A)$, $A\in\mathcal B^m$,  
	in order to have
	\begin{equation}\label{rel:I}
		{\bf 1}_C(\omega)={\bf 1}_{A}(X(\omega)), \quad \omega\in\Omega.
	\end{equation}
	We first note that 
	\begin{eqnarray*}
		\int_Cg(X)dP & = & \int_\Omega{\bf 1}_Cg(X)dP\\
		& \stackrel{(\ref{rel:I})}{=} & \int_\Omega	{\bf 1}_A(X)g(X)dP\\
		& = & \int_{\mathbb R^m}{\bf 1}_A({\bf x})g({\bf x})\psi_X({\bf x})d{\bf x},
	\end{eqnarray*}
	and using both (\ref{int:prob:0}) and (\ref{exp:x}),
	\[
	\int_Cg(X)dP 
	= \int_{\mathbb R^m}{\bf 1}_A({\bf x})\left(\int_{\mathbb R^p}{\bf y}\psi_{Y|X={\bf x}}({\bf y})d{\bf y}\right)\psi_X({\bf x})d{\bf x}.
	\]
	From Fubini and (\ref{p:y:x=x}) we thus get	
	\begin{eqnarray*}
		\int_Cg(X)dP
		& {=} & \int\int_{\mathbb R^m\times\mathbb R^p}{\bf 1}_A({\bf x}){\bf y}\psi_{(X,Y)}({\bf x},{\bf y})d{\bf x}d{\bf y}\\
		& = & \int\int_{\mathbb R^m\times\mathbb R^p}{\bf 1}_A({\bf x}){\bf y}dP_{(X,Y)}\\
		& = & \mathbb E({\bf 1}_A(X)Y)\\
		& \stackrel{(\ref{rel:I})}{=} & \mathbb E({\bf 1}_CY)\\
		& = & \int_CYdP,
	\end{eqnarray*}	
	which proves (\ref{int:prob:1}). Finally, (\ref{int:prob:2}) follows from Proposition \ref{ceprop}, (2). 
\end{proof}

\begin{corollary}\label{total:law}(Laws of total expectation and variance)
	There hold
	\begin{equation}\label{total:exp}
		\mathbb E(X)=\mathbb E(\mathbb E(X|Y))
	\end{equation}
	and 
	\begin{equation}\label{total:var}
		{\mathbb V}(X)=\mathbb E({\mathbb V}(X|Y))+{\mathbb V}(\mathbb E(X|Y)). 
	\end{equation}
\end{corollary}

Note that (\ref{total:exp}) corresponds to  item (2) in Proposition \ref{ceprop} with $\mathcal G=\mathcal F_Y$.

\begin{example}\label{de:finetti}\!\!$\bigstar$
	(Exchangeability and de Finetti's theorem)
	A sequence of random variables $\{X_n\}_{n=1}^{+\infty}$ is {\em exchangeable}
	if for every $n\in\mathbb N$ and every permutation $\tau$ of $\{1,\dots,n\}$ the random vectors
	\[
	(X_1,\dots,X_n)
	\quad\text{and}\quad
	(X_{\tau(1)},\dots,X_{\tau(n)})
	\]
	are identically distributed.
	If, in addition, the sequence is {\em binary}, i.e.\ $X_n\in\{0,1\}$ for all $n$, then 
	a celebrated result due to de Finetti ensures 
	that there exists a random variable 
	$Y$ taking values in $[0,1]$ such that 
	the joint distribution of $(Y,X_1,\cdots,X_n)$ is given by
	\[
	P(Y\in B,\,X_1=x_1,\dots,X_n=x_n)
	=
	\int_B
	y^{\sum_{j=1}^n x_j}
	(1-y)^{\,n-\sum_{j=1}^n x_j}
	\,dP_Y(y),
	\]
	for every Borel set $B\subset[0,1]$ and every $(x_1,\dots,x_n)\in\{0,1\}^n$; see  \cite[Theorem 1.47]{schervish2012theory} and \cite{kirsch2019elementary}. In particular, 
		\begin{equation}\label{de:fine}
		P(X_1=x_1,\dots,X_n=x_n)
		=
		\int_0^1
		y^{\sum_{j=1}^n x_j}
		(1-y)^{\,n-\sum_{j=1}^n x_j}
		\,dP_Y(y).
	\end{equation}	
	On the other hand, by the general disintegration formula in (\ref{exist:trans:02}),
	\[
	P(Y\in B,\,X_1=x_1,\dots,X_n=x_n)
	=
	\int_B P(X_1=x_1,\dots,X_n=x_n\,|_{Y=y})\,dP_Y(y),
	\]
	and since $B$ is arbitrary we conclude that
	\[
	P(X_1=x_1,\dots,X_n=x_n\,|_{Y=y})
	=
	\prod_{j=1}^n y^{x_j}(1-y)^{1-x_j},
	\quad P_Y\text{-a.s.},
	\]
	which means that, conditionally on $Y=y$, the variables $\{X_n\}$ are independent and each 
	has distribution $\mathsf{Ber}(y)$: 
	\begin{equation}\label{cond:bern}
		X_n|_{Y=y}\sim_{\rm i.i.d.}\mathsf{Ber}(y). 
	\end{equation}	
Thus, an exchangeable binary sequence may be viewed as a {\em mixture of i.i.d.\ Bernoulli sequences}, 
	with mixing measure $P_Y$. Conversely, Theorem \ref{extkolm} (Kolmogorov's extension) ensures that any family of 
	finite-dimensional distributions of the form \eqref{de:fine} defines a unique probability measure 
	on $\{0,1\}^{\mathbb N}$, hence an exchangeable stochastic process; cf.\ Remark \ref{stocproc:d}\footnote{For instance, if $Y\sim \mathsf{Beta}(k,l)$ as in Definition \ref{beta:def}, with $k,l\in\mathbb N$, then $\{X_n\}$ corresponds to P\'olya's urn model; cf. \cite[Example 12.29]{klenke2013probability}.}. 
	To see that this model genuinely departs from (unconditional) independence, first note that (\ref{cond:bern}) gives
	\[
	\mathbb E({\bf 1}_{\{X_n=x_n\}}|_{Y=y})=P(X_n=x_n|_{Y=y})=y^{x_n}(1-y)^{1-x_n},
	\]
	so that Proposition \ref{int:prob} applies to yield 
	\[
	\mathbb E({\bf 1}_{\{X_n=x_n\}}|Y)=Y^{x_n}(1-Y)^{1-x_n}.
	\]
	Taking $x_n=1$ this says that, for each $E$ in the $\sigma$-algebra generated by $Y$,
	\[
	\int_E YdP=\int_E {\bf 1}_{\{X_n=1\}}dP=\int_EX_ndP,
	\]
	which means that
	\[
		\mathbb E(X_n\,|\,Y)= Y.
	\]
From this we obtain
		\[
		\mathbb E(X_n)\stackrel{(\ref{total:exp})}=\mathbb E(\mathbb E(X_n|Y))=\mathbb E(Y),
		\]
		and, since for $m\neq n$, $X_m$ and $X_n$ are conditionally independent given $Y$, we may also combine it with Proposition \ref{mean:cond} (1) to get 
		\[
		\mathbb E(X_mX_n|Y)=\mathbb E(X_m|Y)\mathbb E(X_n|Y)=Y^2,
		\]
		so that Proposition \ref{ceprop} (2) finally gives $\mathbb E(X_mX_n)=\mathbb E(Y^2)$. It follows that
		\[
		\mathbb C(X_m,X_n)=\mathbb V(Y),
		\]
		which shows that the variability of $Y$ creates correlation, and hence dependence, between observations. 
		At a more fundamental level, this departure from independence may be detected as follows. 
		For
		any $(x_1,\dots,x_n)\in\{0,1\}^n$ with
		\[
		P(X_1=x_1,\dots,X_n=x_n)>0,
		\]
		we have
		\begin{eqnarray*}
			P(X_{n+1}=1\,|_{X_1=x_1,\dots,X_n=x_n})
			& = & 
			\frac{
				P(X_1=x_1,\dots,X_n=x_n,X_{n+1}=1)
			}{
				P(X_1=x_1,\dots,X_n=x_n)
			}\\
			& = & 
			\frac{
				\int_0^1
				y\,y^{\sum_{j=1}^n x_j}
				(1-y)^{\,n-\sum_{j=1}^n x_j}
				\,dP_Y(y)
			}{
				\int_0^1
				y^{\sum_{j=1}^n x_j}
				(1-y)^{\,n-\sum_{j=1}^n x_j}
				\,dP_Y(y)
			}, 
		\end{eqnarray*}
		and using (\ref{exp:x}),
		\[
		P(X_{n+1}=1\,|_{X_1=x_1,\dots,X_n=x_n})
		=
		\mathbb E(Y\,|_{X_1=x_1,\dots,X_n=x_n}).
		\]
		Thus, after observing $X_1,\dots,X_n$, the predictive probability of success in the next trial is not fixed a priori as in the independent case, 
		where we would find that
		\[
			P(X_{n+1}=1\,|_{X_1=x_1,\dots,X_n=x_n})=	P(X_{n+1}=1)
		\]
		by Proposition \ref{mean:cond} (2), 
		but is updated through the ``posterior'' distribution of the latent parameter $Y$.
		In this sense, de Finetti's theorem gives a probabilistic foundation for the Bayesian viewpoint explored in Section \ref{bay:way}: exchangeability replaces independence, with the latter actually emerging from the former upon conditioning, and the hidden random variable $Y$ playing the role of the unknown Bernoulli parameter. This perspective stands in marked contrast to 
		the frequentist interpretation, in which empirical frequencies, encoded in the sample means $n^{-1}\sum_{j=1}^n X_j$ of an i.i.d.\ Bernoulli sample, converge to (and hence define) the fixed probability $P(X_1=1)$ by the classical Law of Large Numbers (Theorem~\ref{lln}). Instead, de Finetti’s theorem implies that, under the more primitive assumption of exchangeability, a simple symmetry condition on the sampling mechanism, these frequencies converge to a random variable $Y$. The emergence of this random limit reflects latent uncertainty about the underlying data-generating process, so that frequencies are no longer the foundation of probability, but rather an observable manifestation of this uncertainty.
		For extensions of de Finetti's representation formula in (\ref{de:fine}) to more general exchangeable processes we refer to \cite[Theorem 1.49]{schervish2012theory} and \cite[Theorem 12.26]{klenke2013probability}.\qed
\end{example}

\section{Normally distributed random variables and their friends}\label{normaldist}

Here we single out and study some important families of random variables that are closely related to the normal distribution. A comprehensive treatment of this subject can be found in \cite{tong1990multivariate}.

\subsection{Normally distributed random variables}\label{normaldist:sub}

We begin with the family of random variables that is, without doubt, the most pervasive in Probability Theory and its applications.

\begin{definition}
	\label{normdistrv}
	We say that  a random vector $X:\Omega\to\mathbb R^n$ is {\em normally distributed} (or a {\em Gaussian}, or simply that $X$ is a {\em normal}) if its pdf $\psi_X:\mathbb R^n\to\mathbb R$ is given by 
	\begin{equation}\label{exp:dens:jn}
		\psi_X({\bf x})=\frac{\sqrt{\det A}}{(2\pi)^{n/2}}e^{-\frac{1}{2}\langle A({\bf x}-{\bm \mu}),{\bf x}-{\bm \mu}\rangle},
	\end{equation}
	where $A$ is a positive definite, symmetric matrix and ${\bm \mu}\in\mathbb R^n$. We then write $X\sim\mathcal N({\bf x};{\bm\mu},{\bm \Sigma})$, or simply $X\sim\mathcal N({\bm \mu},{\bm \Sigma})$, where ${\bm \Sigma}=A^{-1}$. If $n=1$, in which case
	\begin{equation}\label{exp:dens:jn:n1}
		\psi_X(x)=\frac{1}{\sqrt{2\pi\sigma^2}}e^{-\frac{(x-\mu)^2}{2\sigma^2}},\quad x\in\mathbb R,
		\end{equation}
where $\mu\in\mathbb R$ and $\sigma^2>0$,		
	we represent this as $X\sim\mathcal N(\mu,\sigma^2)$.	
\end{definition}

The next proposition shows that this is well defined.

\begin{proposition}
	\label{welldef} One has  $\int_{\mathbb R^n}\psi_X({\bf x}) d{\bf x}=1$.
\end{proposition}

\begin{proof} Take $O$ an orthogonal matrix so that 
	\[
	OAO^{-1}=\Lambda={\rm diag}(\lambda_1,\cdots,\lambda_n),
	\]
	and define ${\bf y}=O({\bf x}-{\bm \mu})$. It follows that 
	\[
	\psi_X({\bf x})=\frac{\sqrt{\det A}}{(2\pi)^{n/2}}e^{-\frac{1}{2}\langle \Lambda {\bf y},{\bf y}\rangle},
	\]
	so that
	\[
	\int_{\mathbb R^n}\psi_X({\bf x})d{\bf x}=\frac{\sqrt{\det A}}{(2\pi)^{n/2}}\Pi_i\int_{\mathbb R}e^{-\frac{1}{2}\lambda_iy_i^2}dy_i.
	\]
	Thus, if $z_i=\sqrt{\lambda_i/2}y_i$ then 
	\[
	\int_{\mathbb R}e^{-\frac{1}{2}\lambda_ix_i^2}dy_i  =  
	\sqrt{\frac{2}{\lambda_i}}\int_{\mathbb R}e^{-z_i^2}dz_i =\sqrt{\frac{2\pi}{\lambda_i}}, 
	\]
	so that 
	\[
	\int_{\mathbb R^n}\psi_X({\bf x})d{\bf x}=\frac{\sqrt{\det A}}{(2\pi)^{n/2}}\frac{(2\pi)^{n/2}}{\Pi_i\sqrt{\lambda_i}}=1,
	\]
	where we used that  $\det A=\Pi_i\lambda_i$ in the last step.  
\end{proof} 

In general, if $X:\Omega\to\mathbb R^n$ is a random vector we have defined its {\em expectation vector}
\[
{\bm \mu}(X)=\mathbb E(X),
\]
and 
its 
{\em covariance matrix} 
\[
{\mathbb C}(X)_{ij}={\mathbb C}(X_i,X_j).
\]
We now compute these invariants assuming that $X$ is Gaussian random vector.

\begin{proposition}
	\label{invnor}If $X\sim\mathcal N({\bm \mu},{\bm \Sigma})$, ${\bm\Sigma}=A^{-1}$, then ${\bm \mu}(X)={\bm \mu}$ and ${\mathbb C}(X)={\bf\Sigma}$.
\end{proposition}

\begin{proof}
	In terms of the substitution above we have ${\bf x}={\bm \mu}+Q{\bf y}$, $Q=O^{-1}$. Thus, 
	\begin{eqnarray*}
		\mathbb E(X_i) & = & \int_{\mathbb R^n}x_i\psi_X({\bf x})d{\bf x} \\
		& = & \frac{\sqrt{\det A}}{(2\pi)^{n/2}}\int_{\mathbb R^n}\left({\bm\mu}_i+\sum_jQ_{ij}y_j\right)\Pi_ke^{-\frac{1}{2}\lambda_ky_k^2}dy_1\cdots dy_n. 
	\end{eqnarray*}
	But 
	\begin{equation}\label{eqesp}
		\int_{\mathbb R}y_je^{-\frac{1}{2}\lambda_jy_j^2}dy_j=0,
	\end{equation} 
	so we get ${\bm \mu}(X)_i={\bm \mu}_i$.
	Also, 
	\begin{eqnarray*}
		\mathbb E(X_iX_j) & = & \int_{\mathbb R^n}x_ix_j\psi_X({\bf x})d{\bf x} \\
		& = & \int_{\mathbb R^n}\left({\bm \mu}_i+\sum_kQ_{ik}y_k\right)
		\left({\bm \mu}_j+\sum_lQ_{jl}y_l\right)\psi_X({\bf x})d{\bf x},
	\end{eqnarray*} 
	and using (\ref{eqesp}) again we get 
	\begin{eqnarray*}
		\mathbb E(X_iX_j) 
		& = & \int_{\mathbb R^n}{\bm \mu}_i{\bm \mu}_j\psi_X({\bf x})d{\bf x} +\frac{\sqrt{\det A}}{(2\pi)^{n/2}}
		\sum_{kl}Q_{ik}Q_{jl}\int_{\mathbb R^n}y_ky_l\Pi_p(e^{-\frac{1}{2}\lambda_py_p^2})dy_p\\
		& = & {\bm \mu}_i{\bm \mu}_j+\frac{\sqrt{\det A}}{(2\pi)^{n/2}}\sum_kQ_{ik}Q_{jk}\int_{\mathbb R^n}y_k^2\Pi_pe^{-\frac{1}{2}\lambda_py_p^2}dy_p\\
		& = & {\bm \mu}_i{\bm \mu}_j+\frac{\sqrt{\det A}}{(2\pi)^{n/2}}\sum_kQ_{ik}Q_{jk}\int_{\mathbb R}y_k^2e^{-\frac{1}{2}y_k^2}dy_k\times\Pi_{p\neq k}\int_{\mathbb R}e^{-\frac{1}{2}y_p^2}dy_p\\
		& = &  {\bm \mu}_i{\bm \mu}_j+\frac{\sqrt{\det A}}{(2\pi)^{n/2}}(2\pi)^{\frac{n-1}{2}}\sum_kQ_{ik}Q_{jk}\int_{\mathbb R}y_k^2e^{-\frac{1}{2}y_k^2}dy_k\times\Pi_{p\neq k}\frac{1}{\lambda_p^{1/2}}.
	\end{eqnarray*} 
	But
	\begin{eqnarray*}
		\int_{\mathbb R}y_k^2e^{-\frac{1}{2}\lambda_ky_k^2}dy_k & = & -\frac{y_k}{\lambda_k}e^{-\frac{1}{2}\lambda_ky_k^2}|_{-\infty}^{+\infty}+\frac{1}{\lambda_k}\int_{\mathbb R^n}e^{-\frac{1}{2}\lambda_ky_k^2}dy_k\\
		& = & \frac{1}{\lambda_k}\frac{(2\pi)^{1/2}}{\lambda_k^{1/2}},
	\end{eqnarray*}
	so that 
	\[
	\mathbb E(X_iX_j)={\bm \mu}_i{\bm \mu}_j+\sum_k\frac{Q_{ik}Q_{jk}}{\lambda_k}={\bm \mu}_i{\bm \mu}_j+
	{\bm  \Sigma}_{ij},
	\]
	where we used  that ${\bm  \Sigma}=A^{-1}=Q\Lambda^{-1}Q^{-1}$. 
	This completes the proof.  
\end{proof}

We now compute the characteristic  function of a normally distributed random vector.

\begin{proposition}
	\label{funchnor} If $X\sim\mathcal N({\bm \mu},{\bm\Sigma})$ then 
	\begin{equation}\label{nordef}
		\phi_X({\bf u})=e^{\langle {\bm \mu},{\bf u}\rangle{\bf i}-\frac{1}{2}\langle {\bm\Sigma}{\bf u},{\bf u}\rangle}.
	\end{equation}
\end{proposition}

\begin{proof}
	Recalling that $Q^{-1}AQ=\Lambda$ and ${\bf x}={\bm \mu}+Q{\bf y}$, we have
	\begin{eqnarray*}
		\phi_X({\bf u}) & = & \int_{\mathbb R^n} e^{\langle {\bf x},{\bf u}\rangle{\bf i}}\frac{\sqrt{\det A}}{(2\pi)^{n/2}}e^{-\frac{1}{2}\langle A({\bf x}-{\bm \mu}),{\bf x}-{\bm \mu}\rangle}dx\\
		& = & \frac{\sqrt{\det A}}{(2\pi)^{n/2}}\int_{\mathbb R^n}
		e^{\langle {\bm \mu},{\bf u}\rangle{\bf i}+\langle {\bf y},{\bf v}\rangle{\bf i}-\frac{1}{2}\langle \Lambda {\bf y},{\bf y}\rangle}d{\bf y},
	\end{eqnarray*}	
	where ${\bf v}=Q^\top{\bf u}$. Now observe that if $(\,,)$ is the sesquilinear product in $\mathbb C^n$ then
	\begin{eqnarray*}
		-\frac{1}{2}\left(\Lambda^{1/2}{\bf y}-{\bf i}\Lambda^{-1/2}{\bf v},\Lambda^{1/2}{\bf y}-{\bf i}\Lambda^{-1/2}{\bf v}\right) & = & -\frac{1}{2}\langle \Lambda {\bf y},{\bf y}\rangle\\
		& & \quad +\frac{1}{2}\langle\Lambda^{-1}{\bf v},{\bf v}\rangle +{\bf i}\langle {\bf v},{\bf y}\rangle,
	\end{eqnarray*}
	which gives
	\begin{eqnarray*}
		\phi_X({\bf u}) & = & \frac{\sqrt{\det A}}{(2\pi)^{n/2}}e^{\langle {\bm \mu},{\bf u}\rangle{\bf i}-\frac{1}{2}\langle\Lambda^{-1}{\bf v},{\bf v}\rangle}\int_{\mathbb R^n}e^{-\frac{1}{2}\left(\Lambda^{1/2}{\bf y}-{\bf i}\Lambda^{-1/2}{\bf v},\Lambda^{1/2}{\bf y}-{\bf i}\Lambda^{-1/2}{\bf v}\right)}d{\bf y}\\
		& \stackrel{(*)}{=} & \frac{\sqrt{\det A}}{(2\pi)^{n/2}}e^{\langle {\bm \mu},{\bf u}\rangle{\bf i}-\frac{1}{2}\langle Q\Lambda^{-1}Q^{-1}{\bf u},{\bf u}\rangle}\int_{\mathbb R^n}e^{-\frac{1}{2}\left(\Lambda^{1/2}{\bf y},\Lambda^{1/2}{\bf y}\right)}d{\bf y}\\
		& = & e^{\langle {\bm \mu},{\bf u}\rangle{\bf i}-\frac{1}{2}\langle {\bm\Sigma}{\bf u},{\bf u}\rangle},
	\end{eqnarray*}
	where in $(*)$ we changed the contour of integration (and used the appropriate multi-dimensional version of Cauchy's theorem in Complex Variables). 
\end{proof}

\begin{corollary}\label{prov}
	If $	X\sim \mathcal N({\bm \mu},{\bm\Sigma})$ then its mgf is
	\begin{equation}\label{nordef:2}
		\varphi_X({\bf u})=e^{\langle {\bm \mu},{\bf u}\rangle+\frac{1}{2}\langle {\bm\Sigma}{\bf u},{\bf u}\rangle}.
	\end{equation} 
	In particular, if  $n=1$ and $X\sim \mathcal N(\mu,\sigma^2)$ then 
	\begin{equation}\label{mgf:normal:n}
		\varphi_X(u)=e^{u\mu+\frac{1}{2}\sigma^2u^2}.
	\end{equation}
\end{corollary}

\begin{corollary}\label{det:norm:dist}
	If $X$ is a normally distributed random vector, $X\sim \mathcal N({\bm\mu},{\bm\Sigma})$, then its distribution $P_X$ is completely determined by its characteristic function.
\end{corollary}

\begin{proof}
	Note that 
	\[
	\int_{\mathbb R^n} |\phi_X({\bf u})|d{\bf u}=	\int_{\mathbb R^n}
	e^{ -\frac{1}{2}\langle {\bm\Sigma}{\bf u},{\bf u}\rangle}
	d{\bf u}<+\infty
	\]
	and apply (the appropriate multi-variate version of) Proposition \ref{four:inv} (2). 
\end{proof}

\begin{corollary}\label{ortho:norm}
	For a random vector  $X\in\mathbb R^n$ the following hold:
	\begin{itemize}
			\item 
		$X\sim\mathcal N(\bm\mu,{\bm\Sigma})$ if and only if  $\langle {\bf u}, X\rangle\sim\mathcal N(\langle{\bf u},{\bm\mu}\rangle, \langle{\bm\Sigma}{\bf u},{\bf u}\rangle)$
		for any ${\bf u}\in\mathbb R^n$.
		\item
	If 	$X\sim\mathcal N(\bm\mu,{\bm\Sigma})$ then 
	\[
	CX\sim\mathcal N(C{\bm\mu},C{\bm\Sigma}C^\top\}),
	\]
	where $C$ is an invertible $n\times n$ matrix.
	\end{itemize}
\end{corollary}

\begin{proof}	
	The first assertion is an immediate consequence of the identity
	\[
	\phi_{X}({u\bf u})=
	\phi_{\langle{\bf u},  X\rangle}(u), \quad u\in\mathbb R, \quad {\bf u}\in\mathbb R^n.
	\]
	As for the second, use that $\langle{\bf u},CX\rangle=\langle C^\top {\bf u},X\rangle$. 
\end{proof}	

Next we explore further consequences of the theory above.

\begin{proposition}\label{norm:spce}
	If $X:\Omega\to\mathbb R$ is normally distributed, $X\sim \mathcal N(\mu,\sigma^2)$, then 
	\begin{enumerate}
		\item $\phi_X(u)=e^{\mu u{\bf i}-\frac{1}{2}\sigma^2u^2}$;
		\item if $(r,s)\in\mathbb R^2$, $r\neq 0$, then $rX+s\sim\mathcal N(r\mu+s,r^2\sigma^2)$. 
		\item If $Y\sim\mathcal N(\overline\mu,\overline\sigma^2)$ and $Y\perp\!\!\perp X$ then 
		$X+Y\sim\mathcal N(\mu+\overline\mu,\sigma^2+\overline\sigma^2)$. As a consequence, if $\{X_j\}_{j=1}^n$ is independent with $X_j\sim \mathcal N(\mu_j,\sigma_j^2)$ then
		\begin{equation}\label{norm:space:3}
			\sum_ja_jX_j\sim\mathcal N\left(\sum_ja_j\mu_j,\sum_ja_j^2\sigma_j^2\right),\quad a_j\in\mathbb R. 
		\end{equation}
		In particular, if $\mu_j=0$ and  $\sigma_j=\sigma$ then  $X=(X_1,\cdots,X_n)$ satisfies
		\begin{equation}\label{norm:space:4}
			\langle X,\vec{a}\rangle\sim \mathcal N(0,\|\vec{a}\|^2\sigma^2), 
		\end{equation}
		where $\vec{a}=(a_1,\cdots,a_n)$. 
	\end{enumerate}
\end{proposition}

\begin{proof}
	(1) is a special case of (\ref{nordef}). As for (2), note that $\phi_s(u)=e^{su{\bf i}}$, $s\perp\!\!\perp rX$ and use Proposition \ref{charac:p} (1) and Proposition \ref{charac:p:fol} (1)
	to check that 
	\[
	\phi_{rX+s}(u)=e^{(r\mu+s)u{\bf i}-\frac{1}{2}r^2\sigma^2u^2},
	\]
	and finally use Corollary \ref{det:norm:dist}. 
	Clearly, (3) is proved with the same kind of argument. For instance,
	\begin{eqnarray*}
	\phi_{X+Y}(u) 
	& = &
	\phi_X(u)\phi_Y(u)\\
	& = & e^{{\bf i}u\mu-\frac{1}{2}u^2\sigma^2}
	e^{{\bf i}u\overline\mu-\frac{1}{2}u^2\overline\sigma^2} \\
	& = &
	e^{{\bf i}u(\mu+\overline\mu)-\frac{1}{2}u^2(\sigma^2+\overline\sigma^2)},
	\end{eqnarray*}
as claimed. 	
\end{proof}

\begin{example}\label{mom:normal}
	(Moments of a normal) If $X\sim\mathcal N(0,\sigma^2)$ then (\ref{mgf:normal:n}) gives
	\[
	\varphi_X(u)=e^{\frac{1}{2}\sigma^2u^2}=\sum_{k\geq 0}\frac{\sigma^{2k}}{2^{k}k!}u^{2k},
	\]
	so if we compare with (\ref{mom:gen:exp}) we conclude that 
	\begin{equation}
		\alpha_l(X)=
		\left\{
		\begin{array}{ll}
			\frac{l!}{2^{l/2}(l/2)!}\sigma^{2k}, & l\,{\rm even} \\
			0,&  l\,{\rm odd}
		\end{array}
		\right.
	\end{equation}
	which provides explicit expressions for all the moments of $X$. 
	\qed
\end{example}

\begin{example}\label{moment} (The log-normal distribution )
	If $n=1$ and $X\sim \mathcal N(\mu,\sigma^2)$ then (\ref{mgf:normal:n}) implies that $Y=e^X$ satisfies 
	\begin{equation}\label{log:exp}
		\mathbb E(Y)=\mathbb E(e^X)=e^{\mu+\frac{1}{2}\sigma^2}
	\end{equation}
	and 
	\[
	\mathbb E(Y^2)=\mathbb E(e^{2X})=e^{2\mu+2\sigma^2},
	\]
	so that 
	\begin{equation}\label{log:var}
		{\mathbb V}(Y)=\mathbb E(Y^2)-\mathbb E(Y)^2=(e^{\sigma^2}-1)e^{2\mu+\sigma^2}. 
	\end{equation}
	Hence, we may summarize (\ref{log:exp}) and (\ref{log:var}) by writing   
	\begin{equation}\label{log:Y}
		Y=e^X\sim \mathcal L\mathcal N(e^{\mu+\frac{1}{2}\sigma^2},(e^{\sigma^2}-1)e^{2\mu+\sigma^2}),
	\end{equation}
	where $\mathcal L\mathcal N$ stands for ``log-normal'' (which means that $X=\ln Y$ follows a normal). 
	Alternatively, we may write 
	\begin{equation}\label{log:Y2}
		Y\sim \Lambda(\mu,\sigma^2),
	\end{equation}
	if emphasis on the parameters of the underlying normal distribution is needed \cite{aitchison1969lognormal}.   In this notation, it is immediate from Proposition \ref{norm:spce}-(2) that (\ref{log:Y2}) implies
	\begin{equation}\label{log:transl}
		e^aY\sim \Lambda(\mu+a,\sigma^2), \quad a\in\mathbb R.
	\end{equation}
	Now, an (obvious) generalization of (\ref{log:Y}) is 
	\[
	Y^u\sim \mathcal L\mathcal N(e^{\mu u+\frac{1}{2}\sigma^2u^2},(e^{\sigma^2u^2}-1)e^{2\mu u+\sigma^2u^2}), \quad u\in\mathbb R,
	\]
	so that  the corresponding {\em coefficient of variation},   
	\begin{equation}\label{cv:logn:pop}
		{\rm cv}(Y^u):=\frac{{\rm sd}(Y^u)}{\mathbb E(Y^u)},
	\end{equation}
	is given by
	\[
	{\rm cv}(Y^u)=\sqrt{e^{\sigma^2u^2}-1}.
	\]
	In particular, it does not depend on $\mu=\mathbb E(\ln Y)$ and satisfies the ``scaling-plus-inversion invariance property'' 
	\begin{equation}\label{log:n:inv}
		{\rm cv}(\alpha Y^u)={\rm cv}(Y^u)={\rm cv}(Y^{-u}), \quad \alpha>0.
	\end{equation}
	For later reference, we also note that the pdf $\psi_{Y}$  of $Y=e^X$ is 
	\begin{equation}\label{dist:ln}
		\psi_{Y}(x)=\frac{1}{\sigma x\sqrt{2\pi}}e^{-\frac{1}{2}\left(\frac{\ln x-\mu}{\sigma}\right)^2},
	\end{equation}
	so that the corresponding cdf
	is 
	\begin{equation}\label{cum:fct}
		F_{Y}(x)=\Phi\left(\frac{\ln x-\mu}{\sigma}\right),
	\end{equation}
	where 
	\begin{equation}\label{cdf:norm}
		\Phi(x):=\frac{1}{\sqrt{2\pi}}\int_{-\infty}^xe^{-\frac{1}{2}y^2}dy
	\end{equation}
	is the cdf of $\mathcal N(0,1)$.\qed
\end{example}

We now turn to another elegant application of the formalism of characteristic functions to Gaussian random variables. Recall that if $X$ and $Y$ are independent random variables, then they are necessarily uncorrelated (Corollary \ref{induncorr}). We shall now verify that the converse also holds in the case of the components of a normally distributed random vector.

\begin{proposition}\label{unc:ind:n}
	\label{nice} If $X=(X_1,\cdots,X_k):\Omega\to\mathbb R^k$ is a normally distributed random vector, say $X\sim\mathcal N({\bm\mu},{\bm\Sigma})$, with ${\bm\Sigma}$ diagonal  then:
	\begin{itemize}
		\item
		$\{X_j\}_{j=1}^k$ is independent;
		\item $X_j\sim \mathcal N({\bm\mu}_j,\sigma_{X_j}^2)$, \quad $\sigma_{X_j}^2={\mathbb V}(X_j)$.
	\end{itemize}
\end{proposition}

\begin{proof}
	By assumption, ${\bm\Sigma}={\rm diag}(\sigma_{X_1}^2,\cdots,\sigma_{X_j}^2)$. If ${\bm \mu}=\mathbb E(X)$ we have from (\ref{nordef}) that  
	\begin{eqnarray*}
		\phi_X({\bf u}) & = & e^{\langle {\bm \mu},{\bf u}\rangle{\bf i}-\frac{1}{2}\langle {\bm\Sigma}{\bf u},{\bf u}\rangle}\\
		& = & e^{\left(\sum_j{\bm\mu}_ju_j\right){\bf i}-\frac{1}{2}\sum_j{\bm\Sigma}_{jj}u_j^2}\\
		& = & \Pi_j e^{{\bm\mu}_ju_j{\bf i}-\frac{1}{2}{\bm\Sigma}_{jj}u_j^2}\\
		& = & \Pi_j\phi_{Y_j}(u_j),		
	\end{eqnarray*}
	where $Y_j\sim\mathcal N({\bm\mu}_j,\sigma_{X_j}^2)$ by (\ref{mgf:normal:n}) and Proposition \ref{four:inv}, (2).  
	Using (the multi-variate version of) (\ref{four:inv:2}) we then have
	\begin{eqnarray*}
		\psi_{X}({\bf x}) 
		& = & 
		\frac{1}{(2\pi)^n}\int_{\mathbb R^n}e^{-{\bf i}\langle {\bf x},{\bf u}\rangle}\phi_X({\bf u})d{\bf u}\\
		& = & \Pi_j\frac{1}{2\pi}\int_{\mathbb R}e^{-{\bf i}x_ju_j}\phi_{Y_j}(u_j)du_j\\
		& = & \Pi_j\psi_{Y_j}(x_j),
	\end{eqnarray*} 
	which not only proves that $\{X_j\}_{j=1}^k$ is independent (by Proposition \ref{inddens}) but also that $\psi_{X_j}=\psi_{Y_j}$ (by Proposition \ref{pdf:marg}), which concludes the proof.
\end{proof}

\begin{corollary}\label{assert:eq:ind}
	The following assertions are equivalent:
	\begin{itemize}
		\item  $X=(X_1,\cdots,X_k)\sim \mathcal N(\vec{0},\sigma^2{\rm Id}_k)$.
		\item $\{X_j\}_{j=1}^k$ is independent and $X_j\sim\mathcal N(0,\sigma^2)$.
	\end{itemize}
\end{corollary}

\begin{corollary}\label{unc:ind:n:c}(Rotational invariance)
	Let $\{X_j\}_{j=1}^k$ be independent with $X_j\sim\mathcal N(0,\sigma^2)$ and consider $Y_{l}=\sum_{j=1}^kC_{lj}X_{j}$, where $C=\{C_{lj}\}$ is orthogonal. Then $\{Y_l\}_{l=1}^k$ is independent with $Y_l\sim\mathcal N(0,\sigma^2)$.
\end{corollary}

\begin{proof}
	Write $Y=CX$ with $X\sim\mathcal N(\vec{0},\sigma^2 I)$. It follows that 
	\[
		\phi_Y({\bf u})
	 =  \phi_{CX}({\bf u})
	 {=}  \phi_X(C{\bf u}),
	\]
where we used Proposition \ref{charac:p:fol} (1) and Proposition \ref{charac:p} (1) in the last step.  It follows from Proposition \ref{funchnor} that 
	\begin{eqnarray*}
		\phi_Y({\bf u})
		& = & e^{-\frac{1}{2}\langle\sigma^2 C{\bf u},C{\bf u}\rangle}\\
		& = & e^{-\frac{1}{2}\sigma^2\|{\bf u}\|^2},
	\end{eqnarray*}
	so that
	$Y\sim\mathcal N(\vec{0},\sigma^2I)$ as well (by Corollary \ref{det:norm:dist}) and  the independence of $\{Y_j\}$ now follows from the proposition. 
\end{proof}

\begin{definition}\label{assert:eq:def}
	If $X=(X_1,\cdots,X_k)$ satisfies any of the conditions in Corollary \ref{assert:eq:ind} with $\sigma=1$ (so that $X\sim\mathcal N(\vec{0},{\rm Id}_k)$) then we say that $X$ is a {\em standard} normal random vector.
\end{definition}

\begin{remark}\label{proj:prop}
	The projection property in (\ref{norm:space:4}) is an easy consequence of rotational invariance. Indeed, let $Y=OX$, where $X\sim\mathcal N(\vec{0},\sigma^2{\rm Id})$ and $O$ is an orthogonal matrix whose first line is $\|\vec{a}\|^{-1}\vec{a}$. Then
	\[
	|\vec{a}\|^{-1}\langle X,\vec{a}\rangle=(OX)_1=Y_1\sim\mathcal N(0,\sigma^2), 
	\]
	where Corollary \ref{unc:ind:n:c} has been used in the last step. It follows that
	$\langle X,\vec{a}\rangle\sim N(0,\|\vec{a}\|^2\sigma^2)$, as claimed. 
	\qed
\end{remark}

\begin{remark}\label{corr:v:ind:qq}
	In Proposition \ref{unc:ind:n} it is essential to assume that the normal random variables $X_j$, $j=1,\cdots,n$, are {\em jointly} normally distributed in the sense that $X=(X_1,\cdots,X_n)$ is normally distributed. In fact, there exist  random variables $X_1$ and $X_2$ with ${\mathbb C}(X_1,X_2)=0$, $X_1,X_2\sim \mathcal N(0,1)$ but with
	$X=(X_1,X_2)$ not being normally distributed and hence with 
	$\{X_1,X_2\}$ {not} being independent. The classical example is obtained by taking $X_1\sim\mathcal N(0,1)$, ${\bm \epsilon}$ a Rademacher random variable as in Definition \ref{radem:def} which is independent from $X_1$ and $X_2={\bm \epsilon} X_1$. To check the claims above,
	we first compute
	\begin{eqnarray*}
		{\mathbb C}(X_1,X_2)
		& = & \mathbb E(X_1X_2)-\mathbb E(X_1)\mathbb E(X_2)\\
		& = & \mathbb E(X_1X_2)\\
		& = & \mathbb E(X_1^2{\bm \epsilon})\\
		& = & \mathbb E(X_1^2)\mathbb E({\bm \epsilon}) \\
		& = & 0, 
	\end{eqnarray*}
	where we used that $\mathbb E({\bm \epsilon})=0$ in the last step. Also, 
	the fact that $X_2\sim\mathcal N(0,1)$ follows from Proposition \ref{radem:sym:eq}. To check that $X$ is not normally distributed just note that $X_1+X_2$ vanishes on $\epsilon^{-1}(-1)$ and hence fails to follow a normal, so the claim follows from Corollary \ref{ortho:norm}.
	Finally, if $\{X_1,X_2\}$ were independent then  $\{|X_1|,|X_2|\}$ would be independent as well, which is a contradiction because $|X_1|=|X_2|$.  
	\qed
\end{remark}

\begin{remark}\label{rem:dir:ind:n} (The effectiveness of the characteristic function)
	The simplest case $n=2$ already illustrates the difficulty in trying to prove Proposition \ref{unc:ind:n}  by means of pdfs (thus directly relying on Proposition \ref{inddens}). Let us assume that 
	\begin{equation}\label{normal:biv}
		(X_1,X_2)\sim\mathcal N({\bm\mu},{\bm\Sigma}),
	\end{equation}
	where, with self-explanatory notation, 
	\[
	{\bm\mu}=
	\left(
	\begin{array}{c}
		\mu_{X_1}\\
		\mu_{X_2}
	\end{array}
	\right)
	\]
	and 
	\[
	{\bm\Sigma}=
	\left(
	\begin{array}{cc}
		\sigma_{X_1}^2 & \sigma_{X_1X_2}\\
		\sigma_{X_1X_2} & 	\sigma_{X_1}^2
	\end{array}
	\right)=
	\left(
	\begin{array}{cc}
		\sigma_{X_1}^2 & \rho\sigma_{X_1}\sigma_{X_2}\\
		\rho\sigma_{X_1}\sigma_{X_2} & 	\sigma_{X_1}^2
	\end{array}
	\right),
	\]
	where 
	\[
	\rho=\frac{\sigma_{X_1X_2}}{\sigma_{X_1}\sigma_{X_2}}
	\]
	is the {\em correlation coefficient}. Hence, one must check that $\rho=0$ implies that $\{X_1,X_2\}$ is independent, with each marginal following the appropriate normal distribution. To proceed, note that 
	\[
	\det {\bm\Sigma}=(1-\rho^2)\sigma_{X_1}^2\sigma_{X_2}^2,
	\]
	so that 
	\begin{eqnarray*}
		{\bm\Sigma}^{-1}
		& = &
		\frac{1}{(1-\rho^2)\sigma_{X_1}^2\sigma_{X_2}^2}
		\left(
		\begin{array}{cc}
			\sigma_{X_2}^2 & -\rho\sigma_{X_1}\sigma_{X_2}\\
			-\rho\sigma_{X_1}\sigma_{X_2} & 	\sigma_{X_1}^2
		\end{array}
		\right)\\
		& = & 
		\frac{1}{1-\rho^2}
		\left(
		\begin{array}{cc}
			1/\sigma_{X_1}^2 & -\rho/\sigma_{X_1}\sigma_{X_2}\\
			-\rho/\sigma_{X_1}\sigma_{X_2} & 	1/\sigma_{X_2}^2
		\end{array}
		\right),
	\end{eqnarray*}
	and leading this to (\ref{exp:dens:jn}), with $A={\bm\Sigma}^{-1}$, we see that the joint density of $(X_1,X_2)$ is
	\begin{eqnarray}
		\psi_{(X_1,X_2)}(x_1,x_2)
		& = & 
		\frac{1}{2\pi\sigma_{X_1}\sigma_{X_2}\sqrt{1-\rho^2}} \times \nonumber\\
		& & \quad \times\,
		e^{-\frac{1}{2(1-\rho^2)}
			\left(\frac{(x_1-\mu_{X_1})^2}{\sigma_{X_1}^2}-
			\frac{2\rho(x_1-\mu_{X_1})(x_2-\mu_{X_2})}{\sigma_{X_1}\sigma_{X_2}}
			+\frac{(x_2-\mu_{X_2})^2}{\sigma_{X_2}^2}\right)}. \label{rem:dir:ind:n:1}
	\end{eqnarray}
	Thus, if $\rho=0$ this decomposes
	as
	\[
	\psi_{(X_1,X_2)}(x_1,x_2)=
	\frac{1}{\sqrt{2\pi}\sigma_{X_1}}
	e^{-\frac{(x_1-\mu_{X_1})^2}{2\sigma_{X_1}^2}}\times
	\frac{1}{\sqrt{2\pi}\sigma_{X_2}}
	e^{-\frac{(x_2-\mu_{X_2})^2}{2\sigma_{X_2}^2}},	
	\]
	from which the claim follows immediately. However, it is not clear how this  argument, which involves explicitly inverting the covariance matrix ${\bm\Sigma}$, carries over as $n$ gets indefinitely large. This should be compared with the general proof displayed above, which relies on the inversion formula (\ref{four:inv:2}) combined with the fact that ${\bm\Sigma}$ appears {\em linearly} in the exponent of (\ref{nordef}). As yet another nice application of characteristic functions, let us note that, in general, if $X=(X_1,X_2)$ is given  then the characteristic function of the marginal $X_1$ is 
	\[
	\phi_{X_1}(u_1)
	= \mathbb E(e^{{\bf i}u_1X_1})
	=
	\mathbb E(e^{{\bf i}(u_1X_1+0X_2}),
	\]      
	that is, 
	\[
	\phi_{X_1}(u_1)=\phi_{(X_1,X_1)}(u_1,0),
	\]
	which tells us how to calculate the characteristic function of a marginal in terms of the characteristic function of the joint distribution. In particular, when applied to a bi-variate normal as in (\ref{normal:biv}), and {\em not} necessarily assuming that $\{X_1,X_2\}$ is independent, this clearly implies that the marginals are normally distributed in the expected way: $X_j\sim\mathcal N(\mu_{X_j},\sigma_{X_j}^2)$. Needless to say, a similar result holds for the marginals of a multivariate, normally distributed random vector, with essentially the same proof. \qed
\end{remark}	

\begin{remark}\!\!$\bigstar$\label{reg:to:m:1}
	(Regression to the mean) If $X=(X_1,X_2)$ is a bi-variate normal as in (\ref{normal:biv}) then we know from Remark \ref{rem:dir:ind:n} that $X_j\sim\mathcal N(\mu_{X_j},\sigma_{X_j}^2)$, $j=1,2$. Using this, (\ref{rem:dir:ind:n:1}), (\ref{p:y:x=x}) and a little algebra we get
	\begin{eqnarray*}
		\psi_{X_2|_{X_1=x_1})}(x_2)
		& = & 
		\frac{1}{\sqrt{2\pi}\sqrt{1-\rho^2}\sigma_{X_2}} \times \nonumber\\
		& & \quad \times\,
		e^{-\frac{1}{2(1-\rho^2)\sigma_{X_2}^2}
			\left(x_2-\mu_{X_2}-\rho\frac{\sigma_{X_2}}{\sigma_{X_1}}
			(x_1-\mu_{X_1})
			\right)^2}, 
	\end{eqnarray*}
	so that 
	\[
	{X_2|_{X_1=x_1}}\sim \mathcal N\left(\mu_{X_2}+\rho\frac{\sigma_{X_2}}{\sigma_{X_1}}
	(x_1-\mu_{X_1}),(1-\rho^2)\sigma_{X_2}^2\right)
	\]
	or equivalently, 
	\[
	\frac{{X_2|_{X_1=x_1}}-\mu_{X_2}}{\sigma_{X_2}}
	\sim\mathcal N\left(\rho\frac{x_1-\mu_{X_1}}{\sigma_{X_1}},1-\rho^2\right).
	\]
	In particular, 
	\begin{equation}\label{aver:reg}
		\frac{\mathbb E\left({X_2|_{X_1=x_1}}\right)-\mu_{X_2}}{\sigma_{X_2}}
		=\rho\frac{x_1-\mu_{X_1}}{\sigma_{X_1}},
	\end{equation}
	which says that, on average, the proper standardization of $X_2|_{X_1=x_1}$ is proportional to the observed standardization of $X_1$ by a factor which is strictly less than $1$ (in absolute value) unless $X_1$ and $X_2$ are perfectly correlated ($|\rho|=1$). More specifically, let us suppose that the variables model random measurements of some hereditary trait  (stature, for instance) that passes from parents ($X_1$) to offspring ($X_2$) and happens to be ``stable'' in the sense that both variables follow the same normal $\mathcal N(\mu,\sigma^2)$ (and of course are jointly normally distributed as well). We then obtain from  (\ref{aver:reg}) that 
	\[
	\mathbb E\left({X_2|_{X_1=x_1}}\right)-x_1=-(1-\rho)(x_1-\mu),
	\]
	which means that, on average, $X_2|_{X_1=x_1}$ lies somewhere between $x_1$ and $\mu$.
	This ``regression to the mean'', first (empirically) discovered by F. Galton, has played a fundamental role in the conceptual development of Multivariate Analysis \cite{stigler1990history,stigler1997regression,gorroochurn2016classic}.
	In order to relate this to the simple linear regression model as discussed in  Remark \ref{reg:to:m:2}, 
	note from (\ref{rem:dir:ind:n:1}) that the ellipses of ``equal frequency'' for the joint distribution are given by 
	\[
	(x_1-\mu)^2-2\rho(x_1-\mu)(x_2-\mu)+(x_2-\mu)^2={\rm const.}, 
	\]
	so the contact points of the corresponding vertical tangent lines satisfy
	\[
	x_2-\mu=\rho(x_1-\mu), 
	\]
	which, upon comparison with (\ref{reg:to:m:3}) and (\ref{reg:to:m:4}), identifies $\rho$ to the slope of the associated regression line\footnote{For assessments of the social and intellectual contexts of his time and the nasty ideology behind Galton's pursuit of this statistical result, we refer to \cite{cowan1972francis,hilts1973statistics,mackenzie1981statistics,bulmer2003francis}.}. \qed
\end{remark}

\subsection{Random variables related to the normal}\label{normaldist:rel}

We now present a few distributions closely related to the normal.

\begin{definition}\label{gamma:dist}
	A randon variable $Y:\Omega\to\mathbb R$ is ${\mathsf {Gamma}}({\alpha,\lambda})$-{\em distributed}, where $\alpha,\lambda>0$, if its pdf is
	\begin{equation}\label{gamma:dist:1}
	\Gamma_{\alpha,\lambda}(x)=\frac{\alpha^\lambda}{\Gamma(\lambda)}x^{\lambda-1}e^{-\alpha x}{\bf 1}_{(0,+\infty)}(x),
	\end{equation}
	where 
	\[
	\Gamma(\lambda)=\int_0^{+\infty}y^{\lambda-1}e^{-y}dy,
	\]
	is the Gamma function. We then say that $\alpha$ and $\lambda$ are the {\em inverse scale} and {\em shape} parameters of $X$, respectively. In particular, $Y$ is {\em chi-square distributed} with $k\geq 1$ degrees of freedom if its pdf is $\chi^2_k:=\Gamma_{1/2,k/2}$. Explicitly,
	\begin{equation}\label{chi:exp}
		\chi^2_k(x)=\frac{1}{2^{k/2}
			\Gamma(k/2)}x^{\frac{k}{2}-1}e^{-x/2}{\bf 1}_{(0,+\infty)}(x).
	\end{equation}
\end{definition}

\begin{proposition}\label{gamma:mgf}
	If $Y\sim {\mathsf{Gamma}}({\alpha,\lambda})$ then its mgf is $\varphi_Y(u)=(1-\alpha^{-1}u)^{-\lambda}$, $|u|<\alpha$. In particular, $\mathbb E(Y)=\lambda/\alpha$ and ${\mathbb V}(Y)=\lambda/\alpha^2$. 
\end{proposition}

\begin{proof}
	We have 
	\begin{eqnarray*}
		\varphi_Y(u) & = & 
		\frac{\alpha^\gamma}{\Gamma(\lambda)}\int_0^{+\infty}x^{\lambda-1}e^{-(\alpha-u)x}dx\\
		& \stackrel{y=(\alpha-u)x}{=} &
		\frac{\alpha^\gamma}{\Gamma(\lambda)}(\alpha-u)^{-\lambda}\int_0^{+\infty}y^{\lambda-1}e^{-y}dy\\
		&= & {\alpha^\gamma}(\alpha-u)^{-\lambda}.
	\end{eqnarray*}
	The last assertion  follows from Remark \ref{mgf:e:v}.  
\end{proof}

\begin{corollary}\label{chi:sq:ms}
	If $Y\sim\chi^2_k$ then 
	$\varphi_Y(u)=(1-2u)^{-k/2}$, $|u|<1/2$. In particular,
	$\mathbb E(Y)=k$ and ${\mathbb V}(Y)=2k$.
\end{corollary}

\begin{corollary}\label{gamma:cf}
	If $Y\sim {\mathsf{Gamma}}({\alpha,\lambda})$ then its characteristic function  is given by  $\phi_Y(u)=(1-\alpha^{-1}u{\bf i})^{-\lambda}$. In particular, if $Y\sim\chi^2_k$ then $\phi_Y(u)=(1-2u{\bf i})^{-k/2}$. 
\end{corollary}

\begin{corollary}\label{a:times:chi}
	If $a>0$ and $Y\sim{\mathsf{Gamma}}({\alpha,\lambda})$ then $aY\sim {\mathsf{Gamma}}({\alpha/a,\lambda})$. In particular,
	if  $Y\sim\chi^2_k$ then $aY\sim\mathsf{Gamma}({1/2a,k/2})$.
\end{corollary}

\begin{proof}
	Recall from Proposition \ref{charac:p:fol} (1) that 
	\[
	\phi_{aY}(u)=\phi_Y(au)=(1-\alpha^{-1}au{\bf i})^{-k/2}.
	\] 
\end{proof}

Note that this justifies the adopted terminology for $\alpha$. 

\begin{corollary}\label{gamma:sum}
	If $\{Y_j\}_{j=1}^k$ is independent and $Y_j\sim {\mathsf{ Gamma}}(\alpha,\lambda_j)$ then 
	\[
	\sum_j Y_j\sim \mathsf{Gamma}\left({\alpha,\sum_j\lambda_j}\right). 
	\]
	In particular, if $Y_j\sim \chi^2_{k_j}$
	then 
	\[
	\sum_j Y_j\sim \chi^2_{\sum_jk_j}. 
	\]
\end{corollary}

\begin{corollary}\label{sum:norm:sq}
	If $\{Z_j\}_{j=1}^k$ is independent with $Z_j\sim\mathcal N(0,1)$ then 
	\[
	Z:=\sum_j Z_j^2\sim \chi^2_k.
	\]
	In particular, $\mathbb E(Z)=k$ and ${\mathbb V}(Z)=2k$.
\end{corollary}

\begin{proof}
	By Remark \ref{comp:dist} and Proposition \ref{norm:spce} (1),  
	\[
	\psi_{Z_j^2}(x)=\frac{1}{\sqrt{2\pi}}x^{-1/2}e^{-x/2}{\bf 1}_{[0,+\infty)]}(x), 
	\] 
	so that $Z_j^2\sim\chi^2_{1}$ (recall that $\Gamma(1/2)=\sqrt{\pi}$). The result now follows from Corollary \ref{gamma:sum}.
\end{proof}

\begin{remark}\!\!$\bigstar$\label{fisher:comp:1}(The geometric way to $\chi^2_k$)
	Corollary \ref{sum:norm:sq} can be elegantly retrieved as an application of the ``$n$-space computations'' introduced by R. Fisher \cite{fisher1915frequency,fisher1921probable,fisher1925applic}. Indeed, from Proposition \ref{inddens} we know that the amount of probability density spanned by a standard  normal vector $Z=(Z_1,\cdots,Z_k)\in\mathbb R^k$ in an infinitesimal region of volume $d{z}=dz_1\cdots dz_k$ is 
	\begin{eqnarray*}
		\frac{1}{(2\pi)^{k/2}}e^{-\|z\|^2/2}dz
		& = & 	\frac{1}{(2\pi)^{k/2}}e^{-\|z\|^2/2}\|z\|^{k-1}d\|z\|d\theta\\
		& = &  \frac{1}{2}\frac{1}{(2\pi)^{k/2}}e^{-\|z\|^2/2}(\|z\|^2)^{\frac{k}{2}-1}
		d\|z\|^2d\theta,
	\end{eqnarray*}
	where $z=(\|z\|,\theta)\in (0,+\infty)\times\mathbb S^{k-1}$ is the polar decomposition of $Z$\footnote{In this and similar computations, as in Remark \ref{fis:comp:new} and Example \ref{corr:dist}, we represent a realization of a random variable, say $Z_j$ or $\Theta$, by the corresponding lower-case symbol (in this case, $z_j$ or $\theta$).}.
	Again by Proposition \ref{inddens}, if we view this latter expression as the joint distribution of $(\|Z\|^2,\Theta)$ then $\{\|Z\|^2,\Theta\}$ is independent with $\Theta=X/\|X\|$ being {\em uniformly} distributed over $\mathbb S^{n-1}$. Hence, by Proposition \ref{pdf:marg} the infinitesimal density of $\|Z\|^2$ is 
	\[
	\psi_{\|Z\|^2}(\|z\|^2)d\|z\|^2=\frac{\omega_{k-1}}{2}\frac{1}{(2\pi)^{k/2}}e^{-\|z\|^2/2}(\|z\|^2)^{\frac{k}{2}-1}d\|z\|^2,
	\]
	where $\omega_{k-1}$ is the volume of $\mathbb S^{k-1}$. Since
	\begin{equation}\label{vol:form:sph}
		\omega_{k-1}=\frac{2\pi^{k/2}}{\Gamma(k/2)}
	\end{equation}
	it suffices to set 
	$x=\|z\|^2$ in order to recover (\ref{chi:exp}). The same computation gives that  $Y=(Y_1,\cdots,Y_k)\sim \mathcal N(\vec{0},\sigma^{2}{\rm Id}_k)$ implies $\|Y\|^2\sim \mathsf{Gamma}({1/2\sigma^2,k/2})$; cf. Corollary \ref{gamma:sum}. For $k=3$ and $\sigma^2=\kappa T/2$, where $\kappa$ is the Boltzmann constant and $T$ is the temperature, this gives
	\[
	\psi_E(\epsilon)d\epsilon=\frac{2\sqrt{\epsilon}}{\sqrt{\pi}(\kappa T)^{3/2}}e^{-\frac{\epsilon}{\kappa T}}d\epsilon,
	\] 
	the {\em energy distribution} of a Maxwellian gas \cite[Section 13]{kittel2004elementary}. In particular, by Proposition \ref{gamma:mgf}, $\mathbb E(E)=3\kappa T/2$, which confirms the {\em principle of equipartition of energy}. \qed   
\end{remark}

By Corollary \ref{assert:eq:ind} we may rephrase Corollary \ref{sum:norm:sq} as saying that $Z\sim\mathcal N(\vec{0},{\rm Id}_k)$ implies $\|Z\|^2\sim\chi^2_k$. It turns out that this is just a special case of a more general result which makes it clear the geometric meaning of the notion of degree of freedom for a chi-square distribution. 

\begin{proposition}\label{u:v:gen}
	If $Z\sim\mathcal N(\vec{0},{\rm Id}_k)$ and $W=\langle Y,QY\rangle$, where $Q$ is a $n\times n$ symmetric and idempotent matrix with ${\rm rank}\, Q=r\leq k$, then $W\sim \chi^2_r$. 
\end{proposition}

\begin{proof}
	Since $Q:\mathbb R^n\to\mathbb R^n$ defines a projection onto its range ${\rm Im}\, Q$, a linear subspace of dimension $r$, we may use
	the projection property in (\ref{norm:space:4}), with $\sigma=1$ and $\vec{a}$ running over an orthonormal basis of ${\rm Im}\, Q$,  to conclude that $QZ\sim\mathcal N(\vec{0},{\rm Id}_r)$. Thus, $W=\langle Z, QZ\rangle=\|QZ\|^2\sim\chi^2_r$ by Corollary \ref{sum:norm:sq}. 
\end{proof}

This kind of geometric argument has  many useful applications, including the next one, whose proof we omit. 

\begin{proposition}\label{lin:quad:n}
	Let $Z\sim\mathcal N(\vec 0,{\rm Id}_{n})$, $c\in\mathbb R^n$ and $A$ a symmetric $n\times n$ matrix. Then $\langle c,Z\rangle$ and $\langle Z,AZ\rangle$ are independent if and only if $Ac=\vec 0$. 
\end{proposition}

We now discuss some more  random variables related to the normal distribution.

\begin{definition}\label{tstu:def}
	A random variable $X$ is $\mathfrak t${\em -Student distributed} with $k\geq 1$ degrees of freedom if
	\begin{equation}\label{tsts:def:2}
		\psi_X(x)=\mathfrak t_{k}(x):=\frac{\Gamma(\frac{k+1}{2})}{\sqrt{k\pi}\Gamma(k/2)}\left(1+k^{-1}x^2\right)^{-(k+1)/2}.  
	\end{equation}
\end{definition}

\begin{proposition}\label{norm:chi:stu}
	If $Z\sim\mathcal N(0,1)$ and $W\sim \chi^2_k$ with $Z\perp\!\!\perp W$ then $Z/\sqrt{W/k}\sim \mathfrak t_k$. 
\end{proposition}

\begin{proof}
	Note that $Z/\sqrt{W/k}=\sqrt{k}Z/V$, where $\sqrt{k}Z\sim\mathcal N(0,k)$ and 
	$V:=\sqrt{W}$ so that $\psi_{V}(v)= 2v \chi^2_k(v^2)$ by (\ref{regra:01}). It follows from (\ref{regra:02}) that  
	\[
	\psi_{Z/\sqrt{W/k}}(x)=\frac{1}{\sqrt{k\pi}2^{(k-1)/2}\Gamma(k/2)}\int_0^{+\infty}e^{-\frac{1}{2}(1+k^{-1}x^2)v^2}v^kdv.
	\]
	The substitution $w=\frac{1}{2}(1+k^{-1}x^2)v^2$ then finishes the job.
\end{proof}

\begin{remark}\label{fisher:comp:2}
 Proposition \ref{norm:chi:stu} can also be derived using Fisher’s geometric method, as illustrated in Remark \ref{fisher:comp:1}. This line of reasoning was first presented in \cite{fisher1925applic} and is reproduced here in Remark \ref{fis:comp:new}, where the method is applied to obtain the pdf of Student’s sampling distribution defined in (\ref{stu:est}) below.
	\qed
\end{remark}

\begin{definition}\label{F:def}
	Given $k_1, k_2\in\mathbb N$ we say that a random variable  $X$ is ${\bm{\textsf F}}_{k_1,k_2}$-{\em distributed} if 
	\[
	\psi_X(x)={\bm{\textsf F}}_{k_1,k_2}(x):=c_{k_1,k_2}x^{k_1/2-1}
	\left(1+\frac{k_1}{k_2}x\right)^{-\frac{k_1+k_2}{2}}{\bf 1}_{(0,+\infty)}(x),
	\]
	where 
	\[
	c_{k_1,k_2}=\frac{\Gamma\left(\frac{k_1+k_2}{2}\right)}{\Gamma(k_1/2)\Gamma(k_2/2)}
	\left(\frac{k_1}{k_2}\right)^{k_1/2}.
	\]
\end{definition}

\begin{proposition}\label{chi:to:F}
	If $W_1\sim\chi^2_{k_1}$ and $W_2\sim\chi^2_{k_2}$ with $W_1\perp\!\!\perp W_2$ then 
	\[
	\frac{W_1/k_1}{W_2/k_2}\sim {{\mathsf F}}_{k_1,k_2}.
	\] 
\end{proposition}

\begin{proof}
	From (\ref{regra:02}) and (\ref{chi:exp}) we find that
	\[
	\psi_{W_1/W_2}(x)=\frac{x^{\frac{k_1}{2}-1}}{2^{\frac{k_1+k_2}{2}}\Gamma(k_1/2)\Gamma(k_2/2)}
	\int_0^{+\infty}v^{\frac{k_1+k_2}{2}-1}e^{-(1+x)v/2}dv,
	\]
	so that the substitution $w=(1+x)v/2$ transforms this into
	\begin{equation}\label{chi:radio:pdf}
		\psi_{W_1/W_2}(x)=\frac{\Gamma\left(\frac{k_1+k_2}{2}\right)}{\Gamma(k_1/2)\Gamma(k_2/2)}x^{\frac{k_1}{2}-1}(1+x)^{-\frac{k_1+k_2}{2}}. 
	\end{equation}
	The result now follows because 
	
	\[
	\psi_Y(x)=\frac{k_1}{k_2}\psi_{W_1/W_2}\left(\frac{k_1}{k_2}x\right).
	\]
\end{proof}

\begin{corollary}\label{f-dis:cons}
	If $Y\sim {{\mathsf F}}_{k_1,k_2}$ then $Y^{-1}\sim \mathsf F_{k_2,k_1}$. Also, if $T\sim \mathfrak t_k$ then  $T^2\sim {{\mathsf F}}_{1,k}$.
\end{corollary}

\begin{definition}\label{beta:def}
	A random variable $X$ is 
	${\rm Beta}$-distributed with shape parameters $\alpha,\beta>0$ if
	\begin{equation}\label{beta:def:2}
		\psi_X(x)=\mathsf{Beta}(\alpha,\beta)(x):= \frac{\Gamma(\alpha+\beta)}{\Gamma(\alpha)\Gamma(\beta)}x^{\alpha-1}(1-x)^{\beta-1}{\bf 1}_{(0,1)}(x), \quad x\in\mathbb R.
	\end{equation}
\end{definition}

\begin{proposition}\label{beta:chi}
	If $W_1\sim\chi^2_{k_1}$ and $W_2\sim\chi^2_{k_2}$ with $W_1\perp\!\!\perp W_2$ then 
	\[
	\widehat W:=\frac{W_1}{W_1+W_2}\sim \mathsf{ Beta}\left(\frac{k_1}{2},\frac{k_2}{2}\right).
	\] 
\end{proposition}

\begin{proof}
	We have
	\[
	\widehat W=\frac{W_1/W_2}{1+W_1/W_2}, 
	\]
	so that, for $x\in (0,1)$, 
	\begin{eqnarray*}
		F_{\widehat W}(x)
		& = & 
		P\left(\widehat W\leq x\right)\\
		& = & 
		P\left(\frac{W_1}{W_2}\leq \frac{x}{1-x}\right)\\
		& = & 
		F_{W_1/W_2}\left(\frac{x}{1-x}\right).
	\end{eqnarray*}
	By taking derivative with respect to $x$,
	\[
	\psi_{\widehat W}(x)=(1-x)^{-2}\psi_{W_1/W_2}\left(\frac{x}{1-x}\right),
	\]
	and the result follows from  (\ref{chi:radio:pdf}). 
\end{proof}

\section{Concentration inequalities}\label{conc:ineq:appl}

We now elaborate on the idea that an exponential bound on the moment generating function (mgf) of a random variable naturally yields {\em non-asymptotic} estimates for its tail probabilities, usually referred to as \emph{concentration inequalities}. As will be seen through concrete examples, the \emph{Cram\'er--Chernoff method} offers a systematic approach to deriving such inequalities and has numerous applications in both pure and applied mathematics. For comprehensive treatments of concentration phenomena and their significance in modern Data Science, including their role in the analysis of high-dimensional regression models; see \cite{vershynin2018high,wainwright2019high}. As introductory illustrations, we present below their connections with the Johnson--Lindenstrauss Lemma and with the Erd\"{o}s--R\'enyi random graph model.

\subsection{Sub-exponential random variables and the Johnson-Lindenstrauss Lemma}\label{sub:exp:jl}
If  $X\sim \mathcal N(\mu,\sigma^2)$ and $t>0$ then the tail probability 
\[
P(|X-\mu|\geq t)=\frac{2}{\sqrt{2\pi}\sigma}\int_t^{+\infty}e^{-x^2/2\sigma^2}dx
\]
may be easily estimated by observing that $x/t\geq 1$ implies 
\[
P(|X-\mu|\geq t)\leq \frac{2}{\sqrt{2\pi}\sigma}\int_t^{+\infty}\frac{x}{t}e^{-x^2/2\sigma^2}dx,
\]
thus yielding the exponential tail bound
\begin{equation}\label{exp:bound:1}
	P(|X-\mu|\geq t)\leq 	\sqrt{\frac{2}{\pi}}\frac{\sigma}{t}e^{-t^2/2\sigma^2},
\end{equation}
which happens to blow up as $t\to 0$. We may remedy this by considering the function 
\begin{eqnarray*}
	\xi(t)
	&= & 
	P(X-\mu\geq t)-\frac{1}{2}e^{-t^2/2\sigma^2}\\
	& = & 
	1-F_{X-\mu}(t)-\frac{1}{2}e^{-t^2/2\sigma^2}.
\end{eqnarray*}
Since 
\[
\xi'(t)=\left(\frac{t}{2\sigma^2}-\frac{1}{\sqrt{2\pi}\sigma}\right)e^{-t^2/2\sigma^2},
\]
we see that $x$ decreases in the interval $(0,\sqrt{2/\pi}\sigma)$ and increases in the interval $(\sqrt{2/\pi}\sigma),+\infty)$. Since $\xi(0)=0$ and $\xi(t)\to 0$ as $t\to +\infty$ we conclude that 
\begin{equation}\label{exp:bound:1:2}
	P(|X-\mu|\geq t)\leq e^{-t^2/2\sigma^2}, 	
\end{equation}
which provides a sharp bound as $t\to 0$. 
As yet another way to obtain an exponential tail probability bound, observe that 
\begin{eqnarray*}
	P(|X-\mu|\geq t)
	& = & 2 	P(X-\mu\geq t)\\
	& \stackrel{u\geq 0}{=} & 2 P(e^{u(X-\mu)}\geq e^{ut})\\
	& \stackrel{(\ref{markov:ineq:2})}{\leq} & 2\frac{\mathbb E(e^{u(X-\mu)})}{e^{ut}}\quad {\rm (Markov's\,\,inequality)}\\	
	& \stackrel{(\ref{mgf:normal:n})}{=} & 2e^{\frac{1}{2}\sigma^2u^2-ut},
\end{eqnarray*}
and since the right-hand side reaches its minimal value at $u=t/\sigma^2$ we conclude that
\begin{equation}\label{exp:bound:2}
	P(|X-\mu|\geq t)\leq 2e^{-t^2/2\sigma^2}. 
\end{equation}
Although in general (\ref{exp:bound:1}) and (\ref{exp:bound:1:2}) give sharper bounds than (\ref{exp:bound:2}), this latter argument seems to be more promising as it suggests that a suitable exponential control on the mgf of a random variable might yield corresponding bounds for its tail probabilities, a trick usually referred to as the {\em Cram\'er-Chernoff method}. Its flexibility is illustrated by considering the following class of random variables.

\begin{definition}\label{sub:g:def:0}
	We say that $X$ is {\em Sub-Gaussian} if 
	\begin{equation}\label{sub:g:def}
		\mathbb E\left(e^{(X-\mu)u}\right)\leq e^{\frac{\sigma^2u^2}{2}},\quad u\in\mathbb R,
	\end{equation}
	which we represent as $X\in {\bf{\mathsf{SubG}}}(\sigma)$. 
\end{definition}

Clearly, any $X\in {\bf{\mathsf{SubG}}}(\sigma)$ satisfies (\ref{exp:bound:2}).
As a simple example distinct from a normal to which this applies, note that the Rademacher variable ${\bm\epsilon}$ in Definition \ref{radem:def}
satisfies 
\begin{equation}\label{radem:ccin}
	\mathbb E_{\bm\epsilon}(e^{{\bm\epsilon} u})=\cosh u\leq e^{\frac{u^2}{2}}, \quad u\in\mathbb R,
\end{equation}
so that ${\bm\epsilon}\in {\bf{\mathsf{SubG}}}(1)$. 
Also, if $\{X_j\}_{j=1}^N$ is independent with $X_j\in {\bf{\mathsf{SubG}}}(\sigma_j)$ then 
\begin{equation}\label{rade:00}
	\sum_jX_j\in {\bf{\mathsf{SubG}}}\left(\sqrt{\sum_j\sigma_j^2}{}\right).
\end{equation}
Thus, if we apply this to $X^{(N)}=\epsilon_1+\cdots+\epsilon_N$, where 
$\{\epsilon_j\}_{j=1}^N$ is a collection of independent Rademacher variables, we see that $X^{(N)}\in {\bf{\mathsf{SubG}}}(\sqrt{N})$ and hence
\begin{equation}\label{rade:1}
	P(|X^{(N)}|\geq t)\leq 2e^{-t^2/2 N},
\end{equation}
or equivalently,	
\begin{equation}\label{rade:2}
	P\left(\frac{|X^{(N)}|}{N}\geq t\right)\leq 2e^{-Nt^2/2}.
\end{equation}
More generally, this same argument shows that if $X_j\in {\bf{\mathsf{SubG}}}(\sigma)$ is i.i.d.\! and $\overline X_N=X^{(N)}/N$ is the associated sample mean then 
\begin{equation}\label{rade:g}
P\left(\left|\overline X_N-\mathbb E(X_j)\right|\geq t\right)\leq 2e^{-\frac{nt^2}{2\sigma^2}}.
\end{equation}

The next result substantially enriches the class of sub-Gaussian  random variables.

\begin{proposition}\label{subg:bounded}
	If $X$ is bounded, say $a\leq X\leq b$, then $X\in {\bf{\mathsf{SubG}}}(b-a)$. 
\end{proposition}	

\begin{proof}
	We may assume that $\mathbb E(X)=0$. 
	Let $Y$ be an independent copy of $X$, so that $X-Y$ is symmetric by Proposition \ref{charac:p} (3). Hence, if ${\bm\epsilon}$ is a Rademacher variable independent from both $X$ and $Y$ then ${\bm\epsilon}(X-Y)$ and $X-Y$ are identically distributed by Proposition \ref{radem:sym:eq}. It follows that
	\begin{eqnarray*}
		\mathbb E_{(X,Y)}(e^{(X-Y)u}) 
		& = & 
		\mathbb E_{(X,Y)}\left(\mathbb E_{\bm\epsilon}\left(e^{(X-Y)u}\right)\right) \\
		& = & 
		\mathbb E_{(X,Y)}\left(\mathbb E_{\bm\epsilon}\left(e^{{\bm\epsilon}(X-Y)u}\right)\right)\\
		& \stackrel{(\ref{radem:ccin})}{\leq} &
		\mathbb E_{(X,Y)}\left(e^{\frac{(X-Y)^2u^2}{2}}\right),
	\end{eqnarray*}	
	and since $|X-Y|\leq b-a$, we see  that
	\[
	\mathbb E_{(X,Y)}(e^{(X-Y)u}) \leq e^{\frac{(b-a)^2u^2}{2}}. 
	\]
	On the other hand,
	\begin{eqnarray*}
		\mathbb E_{(X,Y)}(e^{(X-Y)u})
		& = & 
		\mathbb E_X\left(e^{Xu}\mathbb E_Y\left(e^{-Yu}\right)\right)\\
		& \geq  & 
		\mathbb E_X\left(e^{Xu}e^{-{\mathbb E_Y(Y)}u}\right)\\
		& = & \mathbb E_X\left(e^{Xu}\right),
	\end{eqnarray*}
	where we used Jensen inequality and $\mathbb E_Y(Y)=0$. 
	Putting all the pieces of this computation together we conclude that
	\[
	\mathbb E_X\left(e^{Xu}\right)\leq e^{\frac{(b-a)^2u^2}{2}},
	\]
	as claimed.
\end{proof}

\begin{remark}\label{hoeff:lemma}
	Under the conditions of Proposition \ref{subg:bounded},
	it is possible to show that $X\in {\bf{\mathsf{SubG}}}((b-a)/2)$, which is known as the {\em Hoeffding lemma}. \qed
\end{remark}

\begin{remark}\!\!$\bigstar$\label{rem:symmetrization}
	(Empirical processes, symmetrization and Rademacher complexity)
	Perhaps more relevant than the result itself in Proposition \ref{subg:bounded}, which is classical and can be obtained by other means, is the method of
	proof, since it provides a first glimpse of the symmetrization methods that pervade much of modern Statistical Learning Theory.
 Roughly speaking, one replaces the original (non-symmetric) quantity by a difference of independent copies, thereby producing a symmetric object that can be analyzed via random sign (Rademacher) arguments. 
	We illustrate the underlying argument as follows. Let $Z$ be a random variable taking values in $\mathcal Z\subset\mathbb R$ and $\mathcal F$ be a class of measurable functions on $\mathcal Z$. For each $f\in\mathcal F$ we consider
	\[
		Pf:=\mathbb E\left(f(Z)\right),
	\]
	and its empirical counterpart
	\[
		P_n f:=\frac1n\sum_{i=1}^n f(Z_i),
	\]
where $\{Z_1,\cdots,Z_n\}$ is an i.i.d. sample drawn from $P_Z$. As explained in \cite[Chapter 4]{wainwright2019high}, a central problem in the theory of empirical processes is to establish a high probability bound for the uniform deviation
	\[
\sup_{f\in\mathcal F}|P_n f-Pf|.
\]
Ideally, one would like to find conditions on $\mathcal F$ ensuring that this quantity is $o(1)$ in probability as $n\to +\infty$, which amounts to a {\em uniform} law of large numbers; compare with the "pointwise"  version of this foundational result in Proposition \ref{lln}.
To this effect, 	
	let $\{Z'_1,\cdots,Z'_n\}$ be an independent copy of the sample.
If 
\[
	P_n'f:=\frac1n\sum_{i=1}^n f(Z_i')
\]
and $\mathcal G$ is the $\sigma$-algebra generated by $\{Z_1,\cdots,Z_n\}$,
we have
\[
\mathbb E\left(P_n'f|\mathcal G\right)
= \frac{1}{n}\sum_{i=1}^n \mathbb E\left(f(Z_i')|\mathcal G\right)
{=}
 \frac{1}{n}\sum_{i=1}^n \mathbb E\left(f(Z_i')\right)=Pf,
\]
where we used Proposition \ref{ceprop} (4) 
in the next to last step. Also, $\mathbb E(P_nf|\mathcal G)=P_nf$ by Proposition \ref{ceprop} (3). 
Hence, by conditional Jensen,
	\begin{eqnarray*}
	\sup_{f\in\mathcal F}|P_n f-Pf|
	 & = &
	\sup_{f\in\mathcal F}\left|\mathbb E\left(P_n f-P_n'f|\mathcal G\right)\right|\\
	& \leq &
	\mathbb E\left(\sup_{f\in\mathcal F}|P_n f-P_n'f|\,|\mathcal G\right),
	\end{eqnarray*}
so if we take  expectations and use Proposition \ref{ceprop} (2), 
	\[
	\mathbb E\left(\sup_{f\in\mathcal F}|P_n f-Pf|\right)
	\leq
	\mathbb E\left(\sup_{f\in\mathcal F}|P_n f-P_n'f|\right).
	\]
	Now, if $\epsilon_1,\dots,\epsilon_n$  are i.i.d.\ Rademacher random variables, independent of everything else, it follows from Propositions \ref{inddens} and \ref{radem:sym:eq}
	that the random vectors 
	\[
	\left(f(Z_1)-f(Z_1'),\cdots,f(Z_n)-f(Z_n')\right)
	\]
	and 
		\[
	\left(\epsilon_1(f(Z_1)-f(Z_1')),\cdots,\epsilon_n(f(Z_n)-f(Z_n'))\right)
	\]
	are identically distributed. 
Therefore,
	\[
	\mathbb E\left(\sup_{f\in\mathcal F}|P_n f-P_n'f|\right)
	=
	\mathbb E\left(\sup_{f\in\mathcal F}\left|
	\frac1n\sum_{i=1}^n \epsilon_i\big(f(Z_i)-f(Z_i')\big)
	\right|\right),
	\]
	and we end up with
	\[
	\mathbb E\left(\sup_{f\in\mathcal F}|P_n f-P_n'f|\right)
	\le
	2\,	\mathfrak R_n(\mathcal F),
	\]
	where 
	\[
	\mathfrak R_n(\mathcal F)
	:=
\mathbb E\left(\sup_{f\in\mathcal F}\left|\frac1n\sum_{i=1}^n \epsilon_i f(Z_i)\right|\right)
	\]
	denotes the \emph{Rademacher complexity} of $\mathcal F$. If we further assume that the elements of $\mathcal F$ are uniformly bounded, this may be combined
	with a 
	bounded difference concentration inequality (such as McDiarmid’s) to prove that
	\begin{equation}\label{non:as:rad}
	P\left(\sup_{f\in\mathcal F}|P_n f-Pf|\lesssim \mathfrak R_n(\mathcal F)+\sqrt{\frac{\ln (1/\delta)}{n}}\right)\geq 1-\delta, \quad 0<\delta<1.
	\end{equation}
In particular, a uniform law of large numbers holds for $\mathcal F$ as long as $\mathfrak R_n(\mathcal F)=o(1)$.
For a proof of the celebrated Vapnik–Chervonenkis (VC) uniform deviation bound (\ref{vc:bound}) in supervised binary classification along these lines, see 
	\cite[Section 4.3]{wainwright2019high} and \cite[Chapter  3]{mohri2018foundations}.
	\qed
\end{remark}

Although, as already noted, the Hoeffding-type concentration inequalities obtained from 
Proposition~\ref{subg:bounded} (or Remark~\ref{hoeff:lemma}) have many 
interesting applications, they ignore the dispersion of a 
random variable as measured by its standard deviation.
For example, if $X\sim\mathsf{Bin}(p;n)$, then by Example~\ref{bern:trial} we may 
represent it as a sum of independent Bernoulli variables, so that 
(\ref{rade:00}) applies:
\begin{equation}\label{hoeff:bin}
	P\left(X-np>t\right)\le e^{-\tfrac{t^2}{2n}}, 
	\qquad t>0,
\end{equation}
an estimate that does not involve the population probability $p\in(0,1)$; 
see also Subsection~\ref{chern:erdos}.
This observation motivates us to move beyond the sub-Gaussian framework and 
consider other classes of random variables for which variance-sensitive 
concentration inequalities can be established.

\begin{definition}\label{sub:exp}
	A random variable $Y$ with $\mathbb E(Y)=\mu$ is {\em sub-exponential} if there exist positive parameters $(\nu,\beta)$ such that
	\begin{equation}\label{sub:exp:2}
		\mathbb E\left(e^{u(Y-\mu)}\right)\leq e^{\frac{\nu^2u^2}{2}}, \quad |u|<\frac{1}{\beta}. 
	\end{equation}
	We represent this as $Y\in{\bf{\mathsf{SubE}}}(\nu,\beta)$.
\end{definition}

\begin{remark}\label{symm:sub:exp}
	If $\mu=0$ then $Y\in{\bf{\mathsf{SubE}}}(\nu,\beta)$ implies $-Y\in{\bf{\mathsf{SubE}}}(\nu,\beta)$. Also, if $\{Y_j\}_{j=1}^n$ is independent and $Y_j\in {\bf{\mathsf{SubE}}}(\nu_j,\beta_j)$ with 
	\[
	\sum_j Y_j\in {\bf{\mathsf{SubE}}}\left(\sqrt{\sum_j\nu_j^2},\max_j\,\{\beta_j\}\right),
	\]
	where we assume that $\mu_j=\mathbb E(Y_j)=0$.
	\qed
\end{remark}

\begin{example}\label{sub:exp:ex}
	From Corollary \ref{chi:sq:ms} we have that $Y\sim\chi^2_k$ implies 
	\[
	\mathbb E\left(e^{u(Y-k)}\right)=e^{-ku}\left(1-2u\right)^{-k/2}, \quad |u|<\frac{1}{2}.
	\]
	Note that $\mathbb E(e^{\frac{1}{2}(Y-k)})=+\infty$ and hence $Y$ is not sub-Gaussian.
	However,  
	by Taylor expanding around $u=0$ we find that
	\[
	f(u):=\left(1-2u\right)^{-k/2}=1+ku+k\left(\frac{k}{2}+1\right)u^2+\cdots
	\]
	and
	\[
	g(u):=e^{k(2u^2+u)}=1+ku+{k}\left(\frac{k}{2}+2\right)u^2+\cdots,
	\]
	from which we easily see that $g(u)\geq f(u)$ for $|u|<1/4$ since both functions remain convex in this interval. In other words,   
	\[
	\mathbb E\left(e^{u(Y-k)}\right)\leq e^{2ku^2}, \quad |u|<\frac{1}{4},
	\] 
	and we conclude that $Y\in {\bf{\mathsf{SubE}}}(2\sqrt{k},4)$. \qed
\end{example}

It turns out that sub-exponential random variables obey a concentration inequality that reveals a sharp threshold separating the sub-Gaussian regime from the purely sub-exponential one.

\begin{proposition}\label{tail:sub:exp}
	If $Y\in {\bf{\mathsf{SubE}}}(\nu,\beta)$ with $\mathbb E(Y)=\mu$  then 
	\[
	P(|Y-\mu|\geq t)\leq 
	\left\{
	\begin{array}{ll}
		2e^{-\frac{t^2}{2\nu^2}}, &  0\leq t< \frac{\nu^2}{\beta}\\
		2e^{-\frac{t}{2\beta}}, &   t\geq \frac{\nu^2}{\beta}
	\end{array}
	\right.
	\]
\end{proposition}

\begin{proof}
	Clearly, we may assume that $\mu=0$, so Markov inequality (\ref{markov:ineq:2}) gives 
	\[
	P(Y\geq t)\leq e^{h_t(u)}, \quad 0\leq u<\frac{1}{\beta},
	\]
	with the graph of $h_t(u)=-ut+{\nu^2u^2}/{2}$ being an upward pointing parabola passing through $(0,0)$ and with vertex located at $(t/\nu^2,-t^2/2\nu^2)$.
	Thus, in the first case, when $t/\nu^2 <1/\beta$, the tail probability is bounded by 
	\[
	e^{h_t(t/\nu^2)}=e^{-\frac{t^2}{2\nu^2}},
	\]
	whereas in the second case, when $t/\nu^2\geq 1/\beta$, it is bounded by
	\[
	e^{h_t(1/\beta)}\leq e^{-\frac{t}{\beta}+\frac{\nu^2}{2\beta^2}}\leq e^{-\frac{t}{2\beta}}.
	\]
	By Remark \ref{symm:sub:exp}, an identical estimate holds for  $P(Y\leq -t)$, which concludes the proof.
\end{proof}

\begin{corollary}\label{tail:sub:exp:2}
	If $Y\sim \chi^2_k$ then 
	\begin{equation}\label{tail:sub:exp:3}
		P(|k^{-1}Y-1|\geq t)\leq 
		\left\{
		\begin{array}{ll}
			2e^{-\frac{kt^2}{8}}, &  0<t< 1\\
			2e^{-\frac{kt}{8}}, & t\geq 1
		\end{array}
		\right.
	\end{equation}
\end{corollary}

\begin{proof}
	From Example \ref{sub:exp:ex} and the proposition (with $\mu=k$ and $(\nu,\beta)=(2\sqrt{k},4)$) we know that
	\begin{equation}\label{tail:sub:exp:4}
		P(|Y-k|\geq \tau)\leq  
		\left\{
		\begin{array}{ll}
			2e^{-\frac{\tau^2}{8k}}, &  0<\tau<k \\
			2e^{-\frac{\tau}{8}}, & \tau\geq k
		\end{array}
		\right.
	\end{equation}
	Since 
	\[
	P(|k^{-1}Y-1|\geq k^{-1}\tau)=P(|Y-k|\geq \tau),
	\] 
	the substitution $\tau=k t$ finishes the proof.
\end{proof}

We now use the concentration inequality (\ref{tail:sub:exp:3}) to establish a celebrated result with a number of applications both in pure and applied mathematics.

\begin{theorem}\label{jl:lemma}(Johnson-Lindenstrauss). If  $\mathcal C=\{x_1,\cdots,x_n\}\subset\mathbb R^p$ is a collection of $n$ points then, given $\epsilon, \delta\in(0,1)$, there exist $m= O(\epsilon^{-2}\ln (n/\delta))$ and a map $F:\mathbb R^p\to\mathbb R^m$ such that
	\begin{equation}\label{jl:lemma:2}
		1-\epsilon\leq \frac{\|F(x_i)-F(x_j)\|^2}{\|x_i-x_j\|^2}\leq 1+\epsilon,  		
	\end{equation}
	for all $x_i\neq x_j$ in $\mathcal C$.
\end{theorem}

\begin{proof}
	Use Remark \ref{many:ind} to construct a $m\times p$ random matrix $\bf A$ whose entries are independent and $\mathcal N(0,1)$-distributed random variables and define the linear map
	\begin{equation}\label{jl:lemma:3}
		F:\mathbb R^p\to\mathbb R^m, \quad F(x)=\frac{{\bf A}x}{\sqrt{m}},
	\end{equation}
	with $m$ to be chosen later on.
	If ${\bf a}_i$ is the $i^{\rm th}$ line of $\bf A$ 	and $x\neq \vec{0}$ then 
	$\langle {\bf a}_i,x/\|x\|\rangle\sim\mathcal N(0,1)$ by Proposition \ref{norm:spce} (3), so
	Corollary \ref{sum:norm:sq} applies to ensure that  
	\[
	\frac{\|{\bf A}x\|^2}{\|x\|^2}=\sum_{i=1}^m\left\langle {\bf a}_i,\frac{x}{\|x\|}\right\rangle^2\sim\chi^2_m.
	\]
	Thus, by Corollary \ref{tail:sub:exp:2}, 
	\[
	P\left(\left|m^{-1}\frac{\|{\bf A}x\|^2}{\|x\|^2}-1\right|\geq \epsilon\right)\leq 2e^{-m\epsilon^2/8}, \quad 0<\epsilon<1,
	\] 
	or equivalently, 
	\[
	P\left(\frac{\|F(x)\|^2}{\|x\|^2}\notin(1-\epsilon,1+\epsilon)\right)\leq 2e^{-m\epsilon^2/8}, \quad 0<\epsilon<1. 
	\]
	From this we 
	easily deduce that 
	\[
	P\left( \frac{\|F(x_i)-F(x_j)\|^2}{\|x_i-x_j\|^2}\notin (1-\epsilon,1+\epsilon)\,{\rm for}\,{\rm some}\,x_i\neq x_j\right)\leq 2\left(\begin{array}{c}
		n\\
		2\end{array}\right)e^{-m\epsilon^2/8},  
	\]
	and the result follows if we impose that the right-hand side equals $\delta$ (which determines $m$ as in the statement of the theorem), since this means that the probability that (\ref{jl:lemma:2}) holds true is $\geq 1-\delta>0$. 
\end{proof}

\begin{remark}\!\!$\bigstar$\label{prob:meth}
	The argument above provides a clear illustration of the celebrated \emph{probabilistic method}. In broad terms, this technique consists of reformulating the deterministic assertion to be proved as the occurrence of a random event. In the present case, the deterministic statement in (\ref{jl:lemma:2}) is recast via the introduction of the random matrix $\mathbf{A}$ in (\ref{jl:lemma:3}). Once it is shown that this event takes place with \emph{positive} probability, the validity of the original statement follows immediately. For many further examples of this powerful method, particularly within Combinatorics, see \cite{alon2016probabilistic}.
	\qed
\end{remark}

\begin{remark}\!\!$\bigstar$\label{data:sci:ap}
	In applications to Data Science \cite{vershynin2018high,blum2020foundations}, it is common for the number $n$ of \emph{samples} to be much smaller than the number $p$ of \emph{features}. The striking aspect of Theorem \ref{jl:lemma} is that, once a controlled distortion $\epsilon > 0$ is allowed in the ``approximate projection'' $F$, the set $\mathcal{C}$ acquires a sort of intrinsic dimension $m$\footnote{Concretely, there exists a subset $\mathcal{I}\subset \{1,\dots,p\}$ of indices with $\sharp \mathcal{I}=p-m$ such that $y_j=0$ for all $y\in F(\mathcal{C})$ and $j\in \mathcal{I}$. In other words, up to a global distortion measured by $\epsilon$, $m$ corresponds to the number of relevant (nonzero) features of the elements of $\mathcal{C}$, which explains the significance of the result for data reduction.}. Remarkably, this intrinsic dimension scales like $\ln n$ and is entirely independent of the ambient dimension $p$, which may even be infinite.
	\qed
\end{remark}

\begin{remark}\label{sub:related}
	If $Y\in{\bf{\mathsf{SubE}}}(\nu,\beta)$ then Proposition \ref{tail:sub:exp} says that 
	\[
	P(|Y-\mathbb E(Y)|\geq t)\leq 2e^{-\min\left\{\frac{t^2}{2\nu^2},\frac{t}{2\beta}\right\}},
	\]
	thus confirming that as $t\to+\infty$ this tail bound is much heavier than the one for $X\in{\bf{\mathsf{SubG}}}(\sigma)$ in (\ref{exp:bound:2})  since in this regime the upper bound here is $2e^{-t/2\beta}$. In a sense this reflects the fact that $X\in{\bf{\mathsf{SubG}}}(\sigma)$ with $\mathbb E(X)=0$ implies $Y=X^2\in{\bf{\mathsf{SubE}}}(\nu,\beta)$ for some $(\nu,\beta)$ to be determined below. To check this claim, take $v\in (0,1)$, multiply both sides of the defining condition for $X$ in (\ref{sub:g:def}) (with $\mu=0$) by $e^{-\sigma^2 u^2/2v}$ and integrate to obtain
	\begin{eqnarray*}
		\int_{-\infty}^{+\infty}e^{\frac{\sigma^2u^2(v-1)}{2v}}du
		& \geq & \int_{-\infty}^{+\infty}\mathbb E\left(e^{uX-\frac{\sigma^2u^2}{2v}}\right)du\\
		& = & \int_{-\infty}^{+\infty}\left(
		\int_{-\infty}^{+\infty}
		e^{ux-\frac{\sigma^2u^2}{2v}}dP_X(x)
		\right)du\\
		& = & \int_{-\infty}^{+\infty}\left(
		\int_{-\infty}^{+\infty}
		e^{ux-\frac{\sigma^2u^2}{2v}}du
		\right)dP_X(x).
	\end{eqnarray*} 
	Using that 
	\[
	\int_{-\infty}^{+\infty}e^{-au^2+bu+c}du=e^{\frac{b^2+4ac}{4a}}\sqrt{\frac{\pi}{a}}, \quad a>0,\quad b,c\in\mathbb R,
	\]
	we may compute the Gaussian integrals above to conclude that
	\[
	\mathbb E\left(e^{\frac{X^2v}{2\sigma^2}}\right)\leq \frac{1}{\sqrt{1-v}},\quad 0\leq v<1,
	\]
	or equivalently, 
	\[
	\mathbb E\left(e^{X^2u}\right)\leq \frac{1}{\sqrt{1-2\sigma^2u}}, \quad 0\leq u <\frac{1}{2\sigma^2}.
	\]
	In particular, there holds 
	\[
	\mathbb E\left(e^{(X^2-\mathbb E(X^2))u}\right)<+\infty
	\]
	for $u\in[0,\varepsilon)$, $\varepsilon>0$. By Remark \ref{mgf:e:v}, $\mathbb E(|X|^4)<+\infty$ and we may Taylor expand up to second order:
	\[
	\mathbb E\left(e^{(X^2-\mathbb E(X^2))u}\right)=1+
	\frac{\mathbb E((X^2-\mathbb E(X^2))^2)}{2}u^2+o(u^2).
	\] 
	Comparing this with 
	\[
	e^{\frac{\nu^2u^2}{2}}=1+\frac{\nu^2u^2}{2}+o(u^2),
	\]
	we easily see that $Y=X^2\in{\mathsf{SubE}(\nu,\beta)}$ if we take 
	\[
	\nu^2>\mathbb E((X^2-\mathbb E(X^2))^2)
	\] 
	and $\beta>0$ large enough. Finally, note that if $\sigma=1$ then we can take $\beta=1$. \qed
\end{remark}

\begin{remark}\label{sum:sub:exp}
	If $\{Y_j\}_{j=1}^k$ is independent with $Y_j\in {\bf{\mathsf{SubE}}}(\nu_j,\beta_j)$ and $\mathbb E(Y_j)=0$ then we easily see that $Y=\sum_jY_j\in {\bf{\mathsf{SubE}}}(\nu,\beta)$, where
	\[
	\nu^2={\sum_j\nu_j^2}, \quad \beta=\max_j\beta_j. 
	\]
	You may apply this to $Y_j=X_j^2$, $X_j\sim\mathcal N(0,1)$. Using the well-known fact that, for a standard normal, $\mathbb E((X_j^2-\mathbb E(X_j^2))^2)=3$, and the computation in Example \ref{sub:related} above, we conclude that
	$\sum_jX_j^2\in {\bf{\mathsf{SubE}}}(\sqrt{2}k,\beta)$, for some $\beta>0$,
	which provides essentially the same result as in Example \ref{sub:exp:ex}. Clearly, this suffices for the applications culminating in Theorem \ref{jl:lemma} above.  Notice
	that in this derivation no appeal to the explicit expression for the pdf of a chi-square distribution has been needed. \qed  
\end{remark} 

\subsection{The Gaussian concentration inequality}\label{vers:g:ineq}
A suitable rewording of the inequalities in Corollary \ref{tail:sub:exp:2} provides valuable insights on the ``high dimensional'' behavior of standard normal random vectors, including the precise formulation of a remarkable dimension-free concentration inequality for Lipschitz functions of such vectors; see (\ref{gauss:conc:ineq}) below. 
Indeed, 	let $X=(X_1,\cdots,X_k)$ be such a vector, which means by definition that  $\{X_j\}_{j=1}^k$ is independent with $X_j\sim\mathcal N(0,1)$. Recalling that 
\[
|a-1|\geq \delta \Longrightarrow |a^2-1|\geq\max\left\{\delta,\delta^2\right\}, \quad a\geq 0, \quad \delta>0,
\]
we have
\begin{eqnarray*}
	P\left(\left|\frac{\|X\|}{\sqrt{k}}-1\right|\geq \frac{\tau}{k}\right)
	& \leq & 
	P\left(\left|\frac{\|X\|^2}{{k}}-1\right|\geq \max\left\{
	\frac{\tau}{k},\frac{\tau^2}{k^2}
	\right\}\right)\\
	& = & 
	P\left(\left|{\|X\|^2}-k\right|\geq \max\left\{
	{\tau},\frac{\tau^2}{k}
	\right\}\right)	\\
	& = &
	P\left(\left|{\|X\|^2}-k\right|\geq 
	\left\{
	\begin{array}{ll}
		\tau, & \tau<k\\
		\frac{\tau^2}{k} & \tau\geq k
	\end{array}
	\right.
	\right)\\
	& \leq & 2e^{-\frac{\tau^2}{8k}}, \quad \tau>0,		
\end{eqnarray*}
where we used (\ref{tail:sub:exp:4}) with $Y=\|X\|^2$ in the last step. By setting $t=\tau/k$ we then obtain
\begin{equation}\label{bd:conc:1}
	P\left(\left|\frac{\|X\|}{\sqrt{k}}-1\right|\geq t\right)\leq 2e^{-\frac{kt^2}{8}}, \quad t>0,
\end{equation}
or equivalently, 
\begin{equation}\label{bd:conc:2}
	P\left(\left|{\|X\|}-{\sqrt{k}}\right|\geq t\right)\leq 2e^{-\frac{t^2}{8}}, \quad t>0,
\end{equation}
These concentration inequalities say that, with a {very high}\footnote{That is, as close to $1$ as we wish!} probability, $X/\sqrt{k}$ remains at an {\em arbitrarily small} distance from the unit sphere $\mathbb S^{k-1}\subset\mathbb R^k$ or, equivalently, $X$ remains at a {\em bounded} distance from the round sphere $\mathbb S^{k-1}_{\sqrt{k}}\subset\mathbb R^k$ of radius $\sqrt{k}$ as $k\to +\infty$. We note the striking similarity between (\ref{bd:conc:1}) and (\ref{rade:2}): in both cases the relevant random variable, which turns out to be a function of a large  collection of independent variables, becomes {\em almost constant} when properly re-scaled; for more on this perspective, which in a sense underlies the modern applications of the ``concentration of measure phenomenon'' to (high dimensional) probability, see \cite{talagrand1996new}.   

\begin{remark}\!\!$\bigstar$\label{poin:conv}
	(Poincar\'e's limit theorem) From Remark \ref{fisher:comp:1} we know that a random vector $Z^{[k]}$ uniformly distributed over $\mathbb S^{k-1}_{\sqrt{k}}$ may be expressed as 
	\begin{equation}\label{rand:unif}
		Z^{[k]}=\sqrt{k}\Theta^{[k]},
	\end{equation}
	where $\Theta^{[k]}= X^{[k]}/\|X^{[k]}\|$, $X^{[k]}\sim\mathcal N(\vec{0},{\rm Id}_k)$. 
	Now, it follows from
	(\ref{bd:conc:1}) that $\|X^{[k]}\|/\sqrt{k}\stackrel{p}{\to}1$ as $k\to+\infty$\footnote{This assertion also follows from the Law of Large Numbers (Theorem \ref{lln} below). 
		Indeed, if $X \sim\mathcal N(\vec{0},{\rm Id}_k)$ then
		\[
		\|X\|^2=\sum_{j=1}^k X_j^2, 
		\]
		where $X_j\sim\mathcal N({0},1)$ and hence $X_j^2\sim \chi^2_1$ with $\mathbb E(X_j^2)=1$ by Corollary \ref{sum:norm:sq}. Thus, by making $X=X^{[k]}$, Theorem \ref{lln} applies to ensure that $\|X^{[k]}\|^2/k\stackrel{p}{\to}1$, as claimed.  
	}. On the other hand, if $x\in\mathbb R=\mathbb R^1\subset\mathbb R^2\subset\cdots\subset\mathbb R^k\subset\cdots$, 
	$\|x\|=1$, 
	we know that $\langle X^{[k]},x\rangle\sim\mathcal N(0,1)$ by Proposition \ref{norm:spce} (3). Thus,  
	from the identity
	\[
	\langle Z^{[k]},x\rangle=\frac{\sqrt{k}}{\|X^{[k]}\|}\langle X^{[k]},x\rangle
	\] 
	and Theorem \ref{slutsky} we conclude that $\langle Z^{[k]},x\rangle\stackrel{d}{\to}\mathcal N(0,1)$.
	Put in another way, as $k\to+\infty$ the marginals of $Z^{[k]}$ corresponding to a given set of $l\geq 1$ coordinates converge in distribution to a standard normal vector $Z^{[\infty]}\sim\mathcal N(\vec{0},{\rm Id}_l)$, a statement usually referred to as ``Poincar\'e's limit theorem'' \cite{hida1964gaussian,mckean1973geometry,diaconis1987dozen}. \qed
\end{remark}

As yet another instance of an insight coming from the concentration inequalities above,  
if we compare the bounds in (\ref{bd:conc:1}) and (\ref{bd:conc:2}) with the normal bound in (\ref{exp:bound:2}), we see that the corresponding fluctuations, as measured by the standard deviation, are $O(1/\sqrt{k})$ and $O(1)$, respectively. Noticing that 
\[
\frac{1}{\sqrt{k}}={\rm Lip}\left(x\mapsto \frac{\|x\|}{\sqrt{k}}\right)
\]  
and 
\[
{1}={\rm Lip}\left(x\mapsto {\|x\|}\right), 
\] 
where ${\rm Lip}$ denotes the Lispschitz constant of a function on $\mathbb R^k$ (with respect to the euclidean norm),
we are thus led to suspect that the dimension-free inequality
\begin{equation}\label{gauss:conc:ineq}
	P\left(\left|F(X)-\mathbb E(F(X))>t\right|\right)\leq 2e^{-\frac{Ct^2}{{\rm Lip}(F)^2}}, \quad t>0,
\end{equation}
should hold true for some universal constant $C>0$ not depending on $k$,
where $F:\mathbb R^k\to\mathbb R$ is assumed to be Lipschitz. Of course, the optimal possibility is $C=1/2$, in which case 
(\ref{gauss:conc:ineq}) says that
if ${\rm Lip}(F)=1$ then  $F(X)$ is at least as much concentrated around its mean as each $X_j$, regardless of the size $k$ of the sample. 
For discussions on this {\em Gaussian Concentration Inequality} which explore the connection with several other mathematical topics, including  the geometric notion of {\em isoperimetry} and the analytical concept of {\em hypercontractivity}, we refer to \cite{ledoux2001concentration,ledoux2006isoperimetry,ledoux2013probability,boucheron2013conc}; see also Remark \ref{poin:impr} below. An elegant approach to the sharpest version of (\ref{gauss:conc:ineq}), due to 
Maurey and Pisier \cite[Chapter 2]{pisier2006probabilistic} and 
relying heavily on It\^o calculus, is presented in Subsection \ref{gauss:conc:sec} below. We provide here a more pedestrian (but no less elegant!) argument, also available in \cite[Chapter 2]{pisier2006probabilistic}, which delivers a constant slightly smaller than the  optimal one ($1/2$ gets replaced by $2/\pi^2$). 

Without loss of generality, we may assume that $F$ is smooth (so that $\|\nabla F\|\leq {\rm Lip}(f)$ a.s.) and $\mathbb E(F(X))=0$. This latter assumption implies, via Jensen's inequality, that $\mathbb E(e^{vF(X)})\geq e^{v\mathbb E(F(X))}=1$ for any $v\in\mathbb R$, which gives
\begin{equation}\label{doubling}
	\mathbb E\left(e^{uF(X)}\right)\leq\mathbb E_{(X,X ')}\left(e^{u\left(F(X)-F(X')\right)}\right),\quad u\geq 0,
\end{equation} 
where $X'$ is an independent copy of $X$ (so that $(X,X')\sim\mathcal N(\vec{0},{\rm Id}_{2k})$). For each $\theta\in[0,\pi/2]$ define 
\[
X_\theta=\cos\theta\, X+\sin\theta\,X'
\]
and 
\[
X'_\theta=\frac{dX_\theta}{d\theta}=-\sin\theta\, X+\cos\theta\,X'.   
\]
Since 
\[
\left(
\begin{array}{c}
	X_\theta\\
	X'_\theta
\end{array}
\right)
=
\left(
\begin{array}{cc}
	\cos\theta \,{\rm Id}_{k} & \sin\theta \,{\rm Id}_{k}\\
	-\sin\theta\,{\rm Id}_{k} & \cos\theta\,{\rm Id}_{k}
\end{array}
\right)
\left(
\begin{array}{c}
	X\\
	X'
\end{array}
\right),	
\]
it follows from the rotational invariance in Corollary \ref{unc:ind:n:c} that 
$(X_\theta,X'_\theta)$ is identically distributed to $(X,X')$ with $X'_\theta$ being an independent  copy of $X_\theta$ as well. 
Now, 
\[
F(X)-F(X')=\int_{0}^{\pi/2}\langle(\nabla F)(X_\theta),X'_\theta\rangle d\theta,
\]
so that Jensen's inequality gives
\[
e^{u\left(F(X)-F(X')\right)}\leq \frac{2}{\pi}\int_0^{\pi/2}e^{\frac{\pi}{2}u\langle(\nabla F)(X_\theta),X'_\theta\rangle }d\theta, 
\]
and hence, 
\begin{eqnarray*}
	\mathbb E_{(X,X')}\left(e^{u\left(F(X)-F(X')\right)}\right)
	& \leq & 
	\frac{2}{\pi}\int_0^{\pi/2}
	\mathbb E_{(X,X')}\left(e^{\frac{\pi}{2}u\langle(\nabla F)(X_\theta),X_\theta'\rangle}\right)d\theta\\
	& = & 
	\mathbb E_{(X,X')}\left(e^{\frac{\pi}{2}u\langle(\nabla F)(X),X'\rangle}\right)\\
	& = & 
	\int_{\mathbb R^k}\left(
	\mathbb E_{X'}\left(e^{\frac{\pi}{2}u\langle(\nabla F)(x),X'\rangle}\right)
	\right)dP_X(x),
\end{eqnarray*}
where in the second step we used the consequences of the rotational invariance mentioned above to ensure that the expectation integrand does not depend on $\theta$
and in the last step we used the independence of $\{X,X'\}$. 
Now, for each $x$ such that $\nabla F(x)\neq\vec{0}$ we known from  Proposition \ref{norm:spce} (3) that 
\[
\frac{\pi}{2}\left\langle (\nabla F)(x),X'\right\rangle\sim\mathcal N\left(0,\frac{\pi^2}{4}\|\nabla F(x)\|^2\right),
\] 
so that (\ref{mgf:normal:n}) 
gives
\[
\mathbb E_{X'}\left(e^{\frac{\pi}{2}u\langle(\nabla F)(x),X'\rangle}\right)\leq e^{{\frac{\pi^2}{8}{\rm L}(f)^2}{u^2}}.
\]
Note that the right-hand side does not depend on $x$ and the estimate remains true if $(\nabla F)(x)=\vec{0}$.
Thus, if we put all the pieces of our calculation together we obtain 
\[
\mathbb E\left(e^{uF(X)}\right)\leq e^{{\frac{\pi^2}{8}{\rm L}(f)^2}{u^2}},
\] 
which amounts to saying that $F(X)\in {\bf{\mathsf{SubG}}}(\pi{\rm Lip}(f)/2)$, a quite good but not entirely satisfactory estimate due to the $\pi/2$ factor appearing in the sub-Gaussian parameter\footnote{This should be compared to the sharp estimate in (\ref{sharp:gauss}) obtained by means of the full machinery of the Stochastic Calculus.}. In any case, 
we may now appeal to the Cram\'er-Chernoff method introduced above: for each $t>0$,
\[
P\left(\left|F(X)\right|>t\right)\leq 2 e^{{\frac{\pi^2}{8}{\rm L}(f)^2}{u^2}-ut}, \quad u\geq 0,
\]
and minimizing the right-hand side we finally get
\[
P\left(\left|F(X)\right|>t\right)\leq 2e^{-\frac{2t^2}{\pi^2{\rm L}(f)^2}},
\]
as claimed.

\begin{remark}\!\!$\bigstar$\label{poin:impr} (Poncar\'e's limit revisited)
	Let $\Pi_{k,l}:\mathbb R^k\to\mathbb R^l$ be the orthogonal projection associated to the natural embedding $\mathbb R^l\hookrightarrow\mathbb R^k$ and let $P_{k}$ be the uniform probability measure on
	$\mathbb S^{k-1}_{\sqrt{k}}$. With this terminology, the ``Poncar\'e's  limit theorem'' in Remark \ref{poin:conv} 
	says that the random vectors $\widetilde{\Pi}_{k,l}=\Pi_{k,l}|_{\mathbb S^{k-1}_{\sqrt{k}}}:(\mathbb S^{k-1}_{\sqrt{k}},P_k)\to\mathbb R^l$ converge in distribution to $Z^{[\infty]}\sim\mathcal N(\vec{0},{\rm Id}_l)$,
	which means that $\mathbb E(\xi(\widetilde{\Pi}_{k,l}))\to \mathbb E(\xi(Z^{[\infty]})))$ for any $\xi:\mathbb R^l\to \mathbb R$ uniformly bounded and continuous; cf. Definition \ref{conv:dist:def}. 
	It turns out that with a bit more of effort it may be checked that 
	this statement actually holds true with $\xi={\bf 1}_A$, the indicator function 
	of an arbitrary Borel set $A\in\mathcal B^l$. Precisely, 
	\begin{equation}\label{poin:impr:2}
		\lim_{k\to +\infty}
		P_{\widetilde\Pi_{k,l}}(A)
		=\frac{1}{(2\pi)^{l/2}}\int_Ae^{-\|y^2\|/2}dy. 
	\end{equation}
	To prove this claim, let us first observe that, with the notation of Remark \ref{poin:conv},  
	\begin{eqnarray*}
		P_{\widetilde\Pi_{k,l}}(A)
		& = & 
		P_k(\widetilde\Pi^{-1}_{k,l}(A))\\
		& = & 
		P_{k}\left(\Pi^{-1}_{k,l}(A)\cap \mathbb S^{k-1}_{\sqrt{k}}\right)\\
		& = & 
		P\left(\Pi_{k,l}(Z^{[k]})\in A\right),
	\end{eqnarray*}
	so if 
	$R_m^2:=X_1^2+\cdots+ X_m^2$, $1\leq m\leq k$, (\ref{rand:unif}) gives 
	\begin{eqnarray*}
		P_{\widetilde\Pi_{k,l}}(A)
		& = & 
		P\left(\frac{\sqrt{k}}{R_k}\left(X_1,\cdots,X_l\right)\in A\right)\\
		& = &
		P\left(\left(k\frac{R_l^2}{R_k^2}\right)^{1/2}
		\frac{1}{R_l}\left(X_1,\cdots,X_l\right)\in A\right). 	
	\end{eqnarray*}
	Now, by Remark \ref{fisher:comp:1}, 
	\[
	\left\{R^2_l,R^2_k-R^2_l,\frac{1}{R_l}\left(X_1,\cdots,X_l\right)\right\}
	\]	
	is independent and therefore
	\[
	\left\{\frac{R_l^2}{R_k^2},\frac{1}{R_l}\left(X_1,\cdots,X_l\right)\right\}
	\]
	is independent as well. On the other hand, 
	by Corollary \ref{sum:norm:sq} and Proposition \ref{beta:chi},
	\[
	\frac{R_l^2}{R_k^2}=\frac{R_l^2}{R_l^2+(R_k^2-R_l^2)}\sim \mathsf{ Beta}\left(\frac{l}{2},\frac{k-l}{2}\right),
	\]
	so if we put together these facts we get,  by Proposition \ref{inddens}, 
	\begin{eqnarray*}
		P_{\widetilde\Pi_{k,l}}(A)
		& = & 
		\frac{\Gamma(k/2)}{\Gamma(l/2)\Gamma((k-l)/2)}\times \\
		&  & \quad \omega_{l-1}^{-1}
		\int_{\mathbb S^{l-1}_1}\int_0^1
		{\bf 1}_A(\sqrt{kx}\theta)x^{\frac{l}{2}-1}(1-x)^{\frac{k-l}{2}-1}d\mathbb S^{l-1}_1(\theta)dx,
	\end{eqnarray*}
	where $d\mathbb S^{l-1}_1(\theta)$ is the (unormalized) volume element of the unit sphere $\mathbb S^{l-1}_1$ (induced by the embedding $\mathbb S^{l-1}_1\hookrightarrow \mathbb R^l$) and $\omega_{l-1}={\rm vol}_{l-1}(\mathbb S_1^{l-1})$. Thus, if we set  
	$u=\sqrt{kx}$ and use (\ref{vol:form:sph}) we find that 
	\begin{eqnarray*}
		P_{\widetilde\Pi_{k,l}}(A)
		& = & 
		\frac{\Gamma(k/2)}{\Gamma((k-l)/2)}\frac{1}{\pi^{l/2}k^{l/2}}\times \\
		&  & \quad 
		\int_{\mathbb S^{l-1}_1}\int_0^{\sqrt{k}}{\bf 1}_A(u\theta)u^{l-1}\left(1-\frac{u^2}{k}\right)^{\frac{k-l}{2}-1}
		d\mathbb S^{l-1}_1(\theta)du. 
	\end{eqnarray*}
	Since $l$ is held fixed, the Stirling approximation in (\ref{stir:gamma}) below gives
	\[
	\frac{\Gamma(k/2)}{\Gamma((k-l)/2)}\approx_{k\to+\infty}
	2^{-l/2}k^{l/2}\left(\frac{k-l}{k}\right)^{-l/2},
	\]	 		
	so we end up with
	\[
	\lim_{k\to+\infty}P_{\widetilde\Pi_{k,l}}(A)
	=
	\frac{1}{(2\pi)^{l/2}}
	\int_{\mathbb S^{l-1}_1}\int_0^{+\infty}{\bf 1}_A(u\theta)u^{l-1}e^{-u^2/2}d\mathbb S^{l-1}_1(\theta)du,
	\]
	which proves the claim because this double integral clearly equals the right-hand side of (\ref{poin:impr:2}) under the substitution $y=u\theta$. The limit theorem in (\ref{poin:impr:2}) is the key ingredient in explicitly solving the isoperimetric problem for the {\em Gaussian space} $(\mathbb R^l,\delta,(2\pi)^{-l/2}e^{-|y|^2/2}dy)$ by essentially viewing it as the limit of the corresponding problem for large, high-dimensional spheres $\mathbb S^{k-1}_{\sqrt{k}}$ as $k\to +\infty$
	\cite{borell1975brunn,sudakov1978extremal}, a celebrated result which by its turn may be used to establish a version of (\ref{gauss:conc:ineq})  with the optimal constant $C=1/2$, but this time with the mean replaced by the median \cite[Chapter 2]{ledoux2006isoperimetry}.  Needless to say, this is a prominent instance of the ``concentration of measure phenomenon'' extensively studied elsewhere \cite{gromov1999metric,ledoux2001concentration,boucheron2013conc,shioya2016metric}. \qed   
\end{remark}

\subsection{Chernoff-type bounds for binomial trials and the Erd\"os-R\'enyi model}\label{chern:erdos}
As illustrated in (\ref{hoeff:bin}), the sub-Gaussian version of the Cram\'er-Chernoff method yields an estimate which fails to account for the  dispersion of a binomial trial (as measured by its standard deviation).
We may remedy this by directly applying the method to $X\sim \mathsf{Bin}(p;n)$ as in Example \ref{bern:trial} in order to get, for $t>0$, 
\begin{eqnarray*}
	P\left(X\geq t\right)
	& = & 
	P\left(e^{Xu}\geq e^{tu}\right)\\
	& = & e^{-tu}\mathbb E\left(e^{Xu}\right) \quad ({\rm Markov})\\
	&\stackrel{(\ref{form:ch:sumb:2})}{=} & 
	e^{-tu}\left(1-p+pe^{u}\right)^n.
\end{eqnarray*} 
Using
that $1+x\leq e^x$, $x\geq 0$, we obtain
\begin{equation}\label{cr:chern:bin:1}
	P\left(X\geq t\right)\leq 	e^{-tu+np(e^u-1)},
\end{equation}
and since  	
the function on the exponent is minimized at $u=\ln(t/\lambda)$, where $\lambda=np=\mathbb E(X)$ is the expectation, we end up with the {\em Chernoff-type} inequality
\begin{equation}\label{cr:chern:bin:2}
	P\left(X\geq t\right)\leq e^{-\lambda}\left(\frac{e\lambda}{t}\right)^t, \quad t>\lambda.
\end{equation}
Regarding this analysis, the following comments are worth mentioning:
\begin{itemize}
	\item From (\ref{cr:chern:bin:2}) we have 
	\begin{equation}\label{cr:chern:bin:27}
		P\left(X\geq t\right)\leq C_1e^{C_2t-t\ln t},
	\end{equation}
	where $C_1=e^{-\lambda}$ and $C_2=1+\ln\lambda$, which for $t$ large gives a tail behavior somehow interpolating between the sub-Gaussian and sub-exponent\-ial regimes. 
	\item Past experience with the sub-exponent\-ial case in Proposition \ref{tail:sub:exp}
	suggests that we should be able to recover a sub-Gaussian tail for {\em small} deviations around the mean $\lambda$ (which is the only critical point of the exponential function in the right-hand side of (\ref{cr:chern:bin:27})). This is the case indeed: if we insert $t=(1+\varepsilon)\lambda$, $|\varepsilon|<1$,  in (\ref{cr:chern:bin:2}) we see that 
	\begin{eqnarray*}
		P\left(X\geq (1+\varepsilon)\lambda\right)	
		& \leq & \left(e^{\varepsilon-(1+\varepsilon)\ln(1+\varepsilon)}\right)^\lambda\\
		& = & \left(e^{-\frac{\varepsilon^2}{2}+\frac{\varepsilon^3}{6}+o(|\varepsilon|^4)}\right)^\lambda,
	\end{eqnarray*}
	which easily leads to the estimate  
	\begin{equation}\label{two:t:chern}
		P\left(|X-\lambda|\geq \varepsilon\lambda\right)		\leq  2e^{-\frac{\varepsilon^2}{3}\lambda},
		\quad 0<\varepsilon<1.
	\end{equation}
	\item 	From (\ref{cr:chern:bin:1}) and (\ref{char:pois:2}) with $\lambda=np$ we see that the Cram\'er-Chernoff method delivers the same estimate as in (\ref{cr:chern:bin:2}) had we started with $X\sim\mathsf {Pois}(\lambda)$, 
	the Poisson variable with the same expectation as our original binomial variable; see Example \ref{poisson:trials}. This suggests that, at least in the asymptotic regime $n\to+\infty$, the classes $\mathsf{Bin}(\lambda/n;n)$ and $\mathsf{Pois}(\lambda)$ are closely related, a claim substantiated by the Law of Rare Events (Theorem \ref{law:rare} below). This deep relationship between binomial and Poisson distributions finds many applications in the theory of random graphs, notably in connection with the Erd\"os-R\'enyi model studied in the sequel; see  the proof of Proposition \ref{tot:edge:lim} below for a simple manifestation of this connection and \cite{van2024random} for the general theory. 
\end{itemize}

We now illustrate the use of the Chernoff-type bounds for binomial trials developed so far in the art of precisely determining the exact threshold for the emergence of certain "phase transitions" in the most commonly studied class of random graphs.

For each $N\geq 2$ let us define $[N]=\{1,\cdots,N\}$, which we  call a {\em set of vertices}. We represent by $[ij]$ the {\em unordered} pair derived from $\{i,j\}\in [N]\times [N]$ with $i\neq j$ (so that $[ij]=[ji]$). The union of all such objects is the {\em set of potential edges}, denoted $E$. Note that $\sharp(E)=N(N-1)/2$.
We now inject a probabilistic ingredient in the construction of graphs starting with $E$. The most obvious possibility, which we adopt here, is simply to flip a (possibly biased) coin for each potential edge in order to decide whether it effectively occurs as a link between two vertices, with the provision that the flips should comprise independent events. In formal terms,   
for each ${e}\in E$ we consider a random variable $X_e\sim\mathsf{Ber}(p)$ so that ${\rm supp}\,P_{X_e}=\{0,1\}$ with $P(X_e=1)=1-P(X_e=0)=p$. Now define
$
\Omega=\{0,1\}^{\binom{N}{2}}
$,
the cartesian product of $N(N-1)/2$ copies of $\{0,1\}$, one for each $e\in E$, and $P=\otimes_{e\in E} P_{X_e}$, the product probability on $\Omega$. 
Note that each element $\omega\in\Omega$ may be viewed as a function $\omega:\Omega\to\{0,1\}$ and hence defines a graph whose {\em edge set} is 
\[
E_\omega=\{e\in E;\omega(e)=1\}. 
\] 
For this reason, the sample space $(\Omega,2^{\Omega},P)$ is called the {\em Erd\"os-R\'enyi model} for a random graph, usually denoted by $\mathbb G(N;p)$. 

\begin{proposition}\label{ind:proj}
	Let $\pi_e:\Omega\to\{0,1\}$ be the canonical projection onto the factor corresponding to $e$. Then each $\pi_e$ is identically distributed to $X_e$ (in particular, $\pi_e\sim\mathsf{Ber}(p)$) with $\{\pi_e\}_{e\in E}$ being independent. 
\end{proposition}

\begin{proof}
	This is a special case of the general procedure in Remark \ref{many:ind}.
\end{proof}

By construction of $\mathbb G(N;p)$, any event (a subset of $\Omega$) defines  a specific collection of graphs. For instance, for each $e\in E$ we may consider
\[
\Omega_e=\{\omega\in\Omega;\omega(e)=1\},
\]
the set of all graphs having $e$ as a vertex.
The next result confirms that a random graph in the Erd\"os-R\'enyi model is obtained by 
flipping
a coin for each potential vertex with the flippings being independent moves.  

\begin{proposition}\label{ind:graph}
	$\{\Omega_e\}_{e\in E}$ is a set of independent events. 
\end{proposition}

\begin{proof}
	Note that $\omega(e)=\pi_e(\omega)$ so that $\Omega_e=\pi_e^{-1}(1)$ and then apply Proposition \ref{ind:proj}. 
\end{proof}

Since $\pi_e={\bf 1}_{\Omega_e}$ for each $e$, the total number of edges in a random graph $\omega$ is given by $\mathscr E_N(\omega)$, where 
\[
\mathscr E_N=\sum_{e\in E}\pi_e,
\] 
so that, from Proposition \ref{ind:proj} and Example \ref{bern:trial},
\[
\mathscr E_N\sim 
\mathsf{Bin}
\left(p;
\binom{N}{2}
\right)
\Longrightarrow
\mathbb E(\mathscr E_N)= \binom{N}{2}p.
\]
It follows from (\ref{two:t:chern})
that
\[
P\left(\left|\mathscr E_N-\mathbb E(\mathscr E_N)\right|< \varepsilon\mathbb E(\mathscr E_N)\right)		\geq 1- 2e^{-\frac{\varepsilon^2}{3}
	\mathbb E(\mathscr E_N)},
\quad 0<\varepsilon<1.
\]
Thus, as $N\to +\infty$,
\[
\frac{\mathscr E_N}{\mathbb E(\mathscr E_N)}
{\to} 1\quad \textrm{in probability}
\]
so that 
the total number of edges {\em asymptotically} approaches its expected value. 
More generally, if we assume that $p=p_N$ (that is, the biased coin possibly changes with $N$) then the same conclusion holds as long as 
\begin{equation}\label{reg:exp:inf}
	\mathbb E(\mathscr E_N)=
	\binom{N}{2}
	p_N\to+\infty,
\end{equation}
with the expectation now being computed with respect to 
$\mathsf{Bin}
\left(p_N;
\binom{N}{2}
\right)$.
At this point, a slightly more ambitious task would be to make sure that, with very high probability, a minimal amount of edges emerges in the regime determined by (\ref{reg:exp:inf}).  
That this is the case indeed follows from the next result, which actually shows that the {\em asymptotic} emergence of a fixed number of edges in the Erd\"os-R\'enyi
model is explicitly determined by the limiting value of $\mathbb E(\mathscr E_N)$.

\begin{proposition}\label{tot:edge:lim}
	Under the conditions above, if $m\in\mathbb N$,
	\[	
	\lim_{N\to+\infty}P\left(\mathscr E_N>m\right)=
	\left\{
	\begin{array}{ll}
		0 & \mathbb E(\mathscr E_N) \to 0\\
		1-	e^{-\lambda}\sum_{k=0}^m\frac{\lambda^k}{k!} &  \mathbb E(\mathscr E_N) \to \lambda\in\mathbb R_+\\
		1 & \mathbb E(\mathscr E_N) \to +\infty
	\end{array}
	\right.
	\]	
\end{proposition}

\begin{proof}
	We only prove the convergence in the middle since the remaining items, at least formally, follow from this case. By the Law of Rare Events (Theorem \ref{law:rare} below) there exists a Poisson variable $Z\sim\mathsf {Pois}(\lambda)$ such that $\mathscr E_N\stackrel{d}{\to} Z$ as $n\to+\infty$. Hence, 
	\begin{eqnarray*}
		\lim_{N\to+\infty}P\left(\mathscr E_N>m\right)
		& = & P\left(Z>m\right)\\
		& = & 1- P\left(Z\leq m\right)\\
		& = & 1-e^{-\lambda}
		\sum_{k=0}^m\frac{\lambda^k}{k!},
	\end{eqnarray*} 
	as claimed. 
\end{proof}

We now turn to the incidence properties of $\mathbb G(N;p)$. For each vertex $i\in [N]$ consider the random variable 
\[
d_{i}=\sum_{j;j\neq i}\pi_{[ij]}. 
\]
Clearly, for each graph $\omega\in\Omega$, $d_{i}(\omega)$ measures the number of edges of $\omega$ having $i$ as a vertex. We call $d_{i}$ the {\em degree}.  

\begin{proposition}\label{degree:erdos}
	For each $i$, $d_{i}\sim \mathsf{Bin}(p;N-1)$. In particular, $d:=\mathbb E(d_{i})=(N-1)p$. 
\end{proposition}
\begin{proof}
	Immediate from Proposition \ref{ind:proj} and Example \ref{bern:trial}.
\end{proof}

Recall that a random graph is {\em almost regular} if the degree of each vertex equals its expected value with very high probability. The next result identifies the threshold on the degree function beyond which almost regularity holds in the Erd\"os-R\'enyi model.   

\begin{proposition}\label{erdos:regul}
	For any $\varepsilon,\delta\in (0,1)$ there exists $C=C_{\varepsilon,\delta}>0$ such that $d\geq C\ln N$ 
	implies
	\[
	P\left( |d_i-d|\leq \varepsilon d \,{\rm for}\,{\rm all}\, i\right)		\geq  1-\delta.
	\]
\end{proposition}

\begin{proof}
	For each $i\in[N]$ we have from Proposition \ref{degree:erdos} and (\ref{two:t:chern}) that
	\[
	P\left(|d_i-d|>\varepsilon d\right)		\leq  2e^{-\frac{\varepsilon^2}{3}d},
	\]
	so that
	\[
	P\left(|d_i-d|>\varepsilon d \,{\rm for}\,{\rm some}\, i\right)		\leq  2Ne^{-\frac{\varepsilon^2}{3}d},
	\]
	and hence
	\[
	P\left(|d_i-d|\leq \varepsilon d \,{\rm for}\,{\rm all}\, i\right)		\geq  1-2Ne^{-\frac{\varepsilon^2}{3}d}.
	\]
	Thus, we must find $C$ such that
	\[
	2Ne^{-\frac{\varepsilon^2}{3}C\ln N}\leq\delta,
	\]
	or equivalently, 
	\[
	C\geq \frac{3}{\varepsilon^2}h_\delta(N), \quad h_\delta(N)=\frac{\ln(2/\delta)+\ln N}{\ln N}. 
	\]
	Now, as $N$ varies $h_\delta(N)$ is uniformly bounded by  $M_\delta=\ln(4/\delta)/\ln 2$, so it suffices to take 
	$C\geq 3M_\delta/\varepsilon^2$.
\end{proof}

\begin{remark}\!\!$\bigstar$\label{er:persp}
The almost regularity in Proposition \ref{erdos:regul} implies a sort of homogeneous behavior of the random graph around each of its vertices\footnote{Incidentally, this homogeneity confirms that the Erd\"os-R\'enyi random graph fails to reliably model  real-world complex networks, where a sizable amount of variability of the incidence pattern of the vertices is observed.}. In particular, the event that no vertex is isolated occurs with high probability. Now, it turns out that in the regime where $p_N\approx \ln N/N$ with $N\to +\infty$, this event is essentially equiprobable to the event defining connectedness of a random graph, which suggests that  
$\ln N/N$ should be a sharp threshold for the {\em asymptotic} occurrence of this topological property. Indeed, arguing along these lines it may be shown that if $p_N=c_N\ln N/N$ 
and 
\[
K:=\lim_{N\to+\infty}(c_{N}-1)\ln N=\lim_{N\to+\infty}(Np_N-\ln N)
\]
exists as an extended real number then
\[
\lim_{N\to+\infty}P\left(\{\omega\in \mathbb G(N;p_N):\omega\,{\rm is}\,{\rm connected}\}\right)=e^{-e^{-K}}.
\]
In particular,  
\[
\lim_{N\to+\infty}P\left(\{\omega\in \mathbb G(N;p_N):\omega\,{\rm is}\,{\rm connected}\}\right)=
\left\{
\begin{array}{ll}
	0& c_N\to c<1\\
	1 & c_N\to c>1
\end{array}
\right.
\] 
For full discussions on this and similar ``phase transition'' phenomena exhibiting a sharp threshold in the Erd\"os-R\'enyi model, see \cite{janson2011random,frieze2016introduction,van2024random}.\qed
\end{remark}

\section{The fundamental limit theorems}\label{fund:lim}

We now present two asymptotic results that play a central role in the theory. It should be emphasized, however, that in contrast with the concentration estimates of Section \ref{conc:ineq:appl}, whose strength lies precisely in their \emph{non-asymptotic} nature, the applicability of limit theorems becomes reliable only in the asymptotic regime, that is, when the number of random variables in the random sample grows without bound; see Remark \ref{miscon}. The proofs outlined below rely on a particular instance of a profound convergence theorem due to L\'evy \cite[Theorem 18.1]{williams1991probability}. The version adopted here provides the appropriate converse to Remark \ref{dist:point:dist} and may be treated by means of elementary Fourier Analysis.

\begin{theorem}\label{levy:conv} (L\'evy's convergence)
	Let $\{Z_j\}_{j=1}^\infty$  be a sequence of random variables such that $\phi_{Z_j}$ converges pointwise to $\phi_Z$, where $Z$ is another random variable. Then $Z_j\to Z$ in distribution. 
\end{theorem} 

\begin{proof}
	By a simple approximation we may assume that $\xi$ in Definition \ref{conv:dist:def} is an arbitrary Schwartz function. 
	Hence, 
	\begin{eqnarray*}
		\mathbb E(\xi(Z_j))
		& = & \int_{-\infty}^{+\infty}\xi(z_j)dP_{Z_j}(z_j)\\
		& = & \int_{-\infty}^{+\infty}\left(\int_{-\infty}^{+\infty}\widehat\xi(u)e^{{\bf i}z_ju}du\right)dP_{Z_j}(z_j),
	\end{eqnarray*}
	where we used Fourier inversion in order to recover $\xi$ from its Fourier transform 
	$\widehat\xi$, which is Schwarz as well and hence uniformly bounded.
	Using Fubini and dominated convergence we get   
	\begin{eqnarray*}
		\mathbb E(\xi(Z_j))
		& = & \int_{-\infty}^{+\infty}\widehat\xi(u)\left(\int_{-\infty}^{+\infty} e^{{\bf i}z_ju}dP_{Z_j}(z_j)\right) du\\
		& = &  \int_{-\infty}^{+\infty}\widehat\xi(u)\phi_{Z_j}(u)du\\
		& \stackrel{j\to+\infty}{\longrightarrow}&  \int_{-\infty}^{+\infty}\widehat\xi(u)\phi_{Z}(u)du\\
		& \vdots & (\textrm{the same computation as above in reverse order})\\
		& = &  \mathbb E(\xi(Z)),
	\end{eqnarray*}
	and the result follows. 
\end{proof}

We may now present the first fundamental limit theorem. 

\begin{theorem}\label{lln}
	(Law of large numbers, LLN) If $\{X_j\}_{j\geq 1}$ is a sequence of i.i.d. (that is, independent and identically distributed) real random variables with $\mathbb E(X_j)=\mu$ then the sequence of random variables 
	\[
	\overline X_n:=\frac{1}{n}(X_1+\cdots+X_n)
	\]
	converges in probability to $\mu$ as $n\to +\infty$.
\end{theorem}

\begin{proof}
	By Proposition \ref{modes}, it suffices to prove that $\overline X_n\to \mu$ in distribution. 
	By Propositions \ref{charac:p:fol} and \ref{charac:p}, if $|u|/n$ is small,   
	\[
	\phi_{\overline X_n}(u)  =  \Pi_{j=1}^n\phi_{X_j}(u/n)
	=  \left[1+\mu\frac{u}{n}{\bf i}+o\left(\frac{|u|}{n}\right)\right]^n, \quad n\to +\infty,
	\]
	so that, for {\em any} $u\in\mathbb R$,
	\[
	\lim_{n\to+\infty}\phi_{\overline X_n}(u)=  e^{u\mu{\bf i}}
	=  \phi_{\mu}(u),
	\]
	where $\phi_\mu$ is the characteristic function of the random variable identically constant to $\mu$.
	The result now follows from Theorem \ref{levy:conv}.
\end{proof}

\begin{remark}\label{lln:append}
	We append three complements to this result:
	\begin{enumerate}
		\item 
		If we further assume that $\mathbb E(|X_j|^2)<+\infty$ then 
		it also follows from the
		argument based on (\ref{lln:new}) below, which relies on Chebyshev's inequality and hence provides a quite effective (i.e. {\em non}-asymptotic) estimate;
		\item 
		For obvious reasons, Theorem \ref{lln} is usually referred as the {\em weak} LLN. With some more effort we may show that the convergence holds in a rather {strong} sense: $\overline X_n\stackrel{a.s}{\to}\mu$.  This latter result is usually known as {\em Kolmogorov's LLN}; see  \cite[Section 1.4]{krengel2011ergodic}, where it is shown how it follows from Birkhoff's ergodic theorem discussed in Example \ref{birk}); 
	\item 
	(Almost sure convergence in LNN via concentration and Borel-Cantelli)
	Let $\{X_j\}_{j\ge1}$ be as in Theorem \ref{lln}, and assume that they are {\em uniformly bounded}, say $a\le X_j\le b$ almost surely. 
	It follows from Proposition \ref{subg:bounded} and the corresponding Hoeffding-type concentration inequality in (\ref{rade:g}) that there exists $C=C_{a,b}>0$ such that, 
	for every $\varepsilon>0$,
	\[
	P\big(|\overline X_n-\mu|\ge \varepsilon\big)
	\le 2
	e^{-Cn\varepsilon^2}.
	\]
	In particular, the right-hand side is summable in $n$, so that
	\[
	\sum_{n=1}^\infty P\big(|\overline X_n-\mu|\ge \varepsilon\big)<\infty.
	\]
	By the Borel--Cantelli lemma \cite[Theorem 2.3.1]{durrett2019probability}, we see that,
	with full probability, only finitely many of the events $\{|\overline X_n-\mu|\ge \varepsilon\}$ occur. Equivalently, for each $\varepsilon>0$ we eventually have $|\overline X_n-\mu|<\varepsilon$.
	Applying this argument to the sequence $\varepsilon_m=1/m$, $m\in\mathbb N$, and taking a countable intersection, we conclude that
	$
	\overline X_n\stackrel{a.s.}{\longrightarrow} \mu
	$,
	thus establishing this version of LLN
in the bounded setting.
 \qed
\end{enumerate}
\end{remark}

\begin{remark}\label{tk:to:norm}
	The limiting behavior of the Student's $\mathfrak t$-distribution $\mathfrak t_k$ in Definition \ref{tstu:def} as the number of degrees of freedom $k$ grows indefinitely may be determined if one makes use of the Stirling asymptotics for the gamma function: 
	\begin{equation}\label{stir:gamma}
		\Gamma(k)\approx_{k\to+\infty} \sqrt{2\pi}k^{k-\frac{1}{2}}{e}^{-k};
	\end{equation}
	see Remark \ref{dem:lap} below for a probabilistic proof of this result.
	Using this, a little computation starting with (\ref{tsts:def:2}) then shows that
	\[
	\lim_{k\to +\infty}\mathfrak t_k(x)=\frac{1}{\sqrt{2\pi}} e^{-x^2/2}, \quad x\in\mathbb R,
	\]
	so that 
	\begin{equation}\label{tk:norm}
		\mathfrak t_k\stackrel{d}{\to}\mathcal N(0,1)
	\end{equation}
	by Scheff\'e's lemma \cite{scheffe1947useful}.
	We point out that this may also be justified with a simple application of Theorem \ref{lln}. Indeed, 
	by Remark \ref{many:ind} we may pick an independent sample $Z, X_1,\cdots,X_k\sim\mathcal N(0,1)$, so that $X_j^2\sim \chi^2_1$ by Proposition \ref{sum:norm:sq}. By LLN, as $k\to +\infty$ we have that $W_k:=\sum_{j=1}^k X_j^2$ satisfies 
	\[
	\sqrt{\frac{W_k}{k}}\to \sqrt{\mathbb E(\chi^2_1)}=1
	\]    
	in probability. Hence, $Z/\sqrt{W_k/k}\to \mathcal N(0,1)$ in distribution so that (\ref{tk:norm}) may be verified using that $Z/\sqrt{W_k/k}$ is $\mathfrak t_k$-distributed by Proposition \ref{norm:chi:stu}. As another instance of this kind of argument, let us check that if $X\sim{\bm{\textsf F}}_{k_1,k_2}$ then 
	\[
	k_1X\stackrel{d}{\to} \chi^2_{k_1} \quad {\rm as}\,\,k_2\to +\infty.
	\]  
	Indeed, from Proposition \ref{chi:to:F} we may write
	\[
	X=\frac{W_1/k_1}{W_2/k_2},
	\]
	where $W_1\perp\!\!\perp W_2$ and $W_j\sim\chi^2_{k_j}$, $j=1,2$. Thus,
	\[
	k_1X=\frac{W_1}{W_2/k_2}
	\]
	and since $W_2/k_2\stackrel{p}{\to}1$ the claim follows. 
	\qed
\end{remark}

The next result provides an accurate asymptotic description of the distribution of a rescaled version of $\overline X_n$ and highlights the pervasive role of the normal distribution in Probability Theory. The underlying rationale can be sketched as follows. From Theorem \ref{lln}, one suspects the existence of a (possibly monotone) function $\nu:\mathbb N\to\mathbb R$ with $\nu(n)\to+\infty$ as $n\to\infty$, such that, informally,
\[
\overline X_n \approx \mu + O(\nu(n)^{-1}). 
\]
In this case,
\[
\nu(n)\left(\overline X_n-\mu\right) \approx O(1),
\]
which suggests that the rescaled deviation might converge to a finite distribution. To identify $\nu$, assume temporarily that the sample is normally distributed, $X_j\sim \mathcal N(\mu,\sigma^2)$. Proposition \ref{norm:spce} then shows that $\sqrt{n}(\overline X_n-\mu)\sim \mathcal N(0,\sigma^2)$, indicating that $\nu(n)=\sqrt{n}$.
The remarkable feature of the forthcoming result is that this relation, which holds exactly for normal samples, in fact extends asymptotically to arbitrary distributions (with finite variance). In line with Theorem \ref{lln}, one obtains convergence in distribution to a normal law, regardless of the original distribution of the sample.

\begin{theorem}\label{clt}
	(Central Limit Theorem, CLT) Let $\{X_j\}_{j\geq 1}$ be a sequence of i.i.d. real random variables with $\mathbb E(X_j)=\mu$ and  ${\mathbb V}(X_j)=\sigma^2>0$. Then 
	the sequence formed by the standardization of the sample mean,
	\begin{equation}\label{stand:samp}
		Z_n:=\frac{\sum_{j=1}^nX_j-n\mu}{\sqrt{n}\sigma}=\frac{\overline X_n-\mu}{\sigma/\sqrt{n}}, 
	\end{equation}
	converges in distribution to a random variable whose pdf is the standard normal distribution $\mathcal N(0,1)$. Equivalently, $\sqrt{n}(\overline X_n-\mu)\stackrel{d}{\to} \mathcal N(0,\sigma^2)$. 
\end{theorem}

\begin{proof}
	Define $Y_j=(X_j-\mu)/\sigma$ so that $\mathbb E(Y_j)=0$ and ${\mathbb V}(Y_j)=1$. Since $Z_n=\sum_jY_j/\sqrt{n}$, by Propositions \ref{charac:p:fol} and \ref{charac:p} we get, for $|u|/\sqrt{n}$ small,  
	\[
	\phi_{Z_n}(u)  =  \Pi_{j=1}^n\phi_{Y_j}\left(\frac{u}{\sqrt{n}}\right)
	=  \left[1-\frac{u^2}{2n}+o\left(\frac{|u|^2}{n}\right)\right]^n, 
	\]
	so that
	\begin{equation}\label{rep:eq:clt}
		\lim_{n\to+\infty}\phi_{Z_n}(u)=  e^{-\frac{1}{2}u^2}, \quad u\in\mathbb R.
	\end{equation}
	By Corollary \ref{det:norm:dist} and Proposition \ref{norm:spce} (1),  the right-hand side 
	is the characteristic function of a random variable $Z\sim\mathcal N(0,1)$,
	so we may apply Theorem \ref{levy:conv} to conclude the proof.
\end{proof}	

\begin{remark}\label{tao}
	An enlightening discussion of several proofs of Theorem \ref{clt}, including the one above, may be found in \cite[Chapter 2]{tao2012topics}. \qed
\end{remark}

\begin{remark}\label{miscon}
A common misconception in the practical use of Theorem \ref{clt} is to assume that the sample mean $\overline X_n$ itself converges in distribution to a normal law. An even more misleading belief, often repeated in applications, is that there exists a fixed threshold for the sample size (frequently taken as $n=30$) beyond which $\overline X_n$ becomes \emph{exactly} normally distributed. Neither of these statements is supported by the theorem. On the contrary, both contradict Theorem \ref{lln}, which establishes that $\overline X_n$ converges (almost surely, by Remark \ref{lln:append} (2)) to the constant population mean $\mu$.  
What Theorem \ref{clt} does guarantee is the \emph{approximation} of $\overline X_n$ in distribution by $\mathcal N(\mu,\sigma^2/n)$ as $n\to\infty$. Thus, for any $-\infty\leq a<b\leq+\infty$ we may, for practical purposes, write
\begin{equation}\label{miscon:p}
	P(a\leq \overline X_n\leq b)\approx_{n\to +\infty} 
	\frac{\sqrt{n}}{\sqrt{2\pi}\sigma}
	\int_a^be^{-\frac{n(x-\mu)^2}{2\sigma^2}}dx,
\end{equation}
where $\approx_{n\to +\infty}$ indicates that the equality holds only asymptotically, in the regime of very large samples. In this notation, the CLT may be expressed as
\begin{equation}\label{clt:v1}
	\overline X_n\approx_{n\to+\infty}\mathcal N(\mu,\sigma^2/n),
\end{equation}
or, equivalently,
\begin{equation}\label{clt:v2}
	X^{(n)}:=X_1+\cdots+X_n\approx_{n\to+\infty} \mathcal N(n\mu,n\sigma^2).
\end{equation}
It is important, especially in applications, to note that the rate of convergence in these approximations can be made explicit. For instance, under the assumptions of Theorem \ref{clt}, with $\mathbb E(X_j)=0$, ${\mathbb V}(X_j)=\sigma^2$ and with the additional requirement that $\rho:=\mathbb E(|X-\mu|^3)<+\infty$, the classical Berry--Esseen theorem guarantees that
\[
\sup_{x\in\mathbb R}\left|F_{\sqrt{n}\overline X_n/\sigma}(x)-\Phi(x)\right|\leq \frac{C\rho}{\sqrt{n}\sigma^3}, \quad C>0.
\]
Hence, apart from the universal constant $C$ and the dependence on sample size through $1/\sqrt{n}$, the convergence rate is governed by the \emph{shape factor} $\varepsilon:=\rho/\sigma^3$, which measures the skewness of the parent distribution.  
As for the convergence behavior under sample size, although in particular cases the convergence may be faster than $O(n^{-1/2})$, there are distributions for which this worst-case bound is sharp, even when higher-order moments are finite. A simple example is provided by the Rademacher variable discussed in Remark \ref{corr:v:ind:qq} \cite[Chapter 1]{sazonov1981normal}. \qed
\end{remark}

\begin{remark}\label{dem:lap}
	We insist that the proof presented above {\em does} cover the case in which the initial i.i.d. sequence $\{X_j\}$ is {\em discrete}, as it operates at the level of characteristic functions. In fact, this is how the CLT first appeared, incarnated in the famous {De Moivre-Laplace} formulas (\ref{dem:lap:form})-(\ref{dem:lap:form:2}) below \cite{fischer2011history}. 
	Let $\{X_j\}_{j=1}^n$ be independent with $X_j\sim \mathsf{Ber}(p)$, the {Bernoulli distribution}. 
	From Example \ref{bern:trial}, 
	we know that $X^{(n)}= X_1+\cdots+X_n\sim \mathsf{Bin}(p;n)$, the {binomial distribution}. 
	Since  $\mathbb E( X_j)=p$ and ${\mathbb V}(X_j)=p(1-p)$, CLT applies\footnote{See Remark \ref{just:bern:clt} below for a direct justification of this step along the lines of the proof of Theorem \ref{clt}.} to give
	\begin{equation}\label{bin:clt}
		Z_n=	\sqrt{n}\frac{n^{-1}X^{(n)}-p}{\sqrt{p(1-p)}}\stackrel{d}{\to}
		\mathcal N(0,1), \quad {n\to+\infty},
	\end{equation}  
	or equivalently, if we combine (\ref{clt:v2}) and (\ref{bef:sum}), 
	\begin{equation}\label{dem:lap:form}
		\sum_{a\leq k \leq b}	\left(
		\begin{array}{c}
			n\\
			k
		\end{array}
		\right) p^k(1-p)^{n-k}\approx_{n\to+\infty}\frac{1}{\sqrt{2\pi np(1-p)}}\int_a^b e^{-\frac{(x-np)^2}{2np(1-p)}}dx,\quad a< b. 
	\end{equation}
	It is not hard to check that this is the same as having  
	\begin{equation}\label{dem:lap:form:2}
		\left(
		\begin{array}{c}
			n\\
			k
		\end{array}
		\right) p^k(1-p)^{n-k}\approx_{n\to+\infty}\frac{1}{\sqrt{2\pi np(1-p)}}e^{-\frac{(k-np)^2}{2np(1-p)}}
	\end{equation}
	uniformly in $k$ satisfying 
	\begin{equation}\label{unif:k:n}
		k=np+\sqrt{np(1-p)}O(1),
	\end{equation}
	which may be proved by using Stirling's formula in (\ref{stirlings}) below and the fact that (\ref{unif:k:n}) implies that $k/n\to p$ as $n\to+\infty$; see \cite[Section 7.3]{chung2006elementary}. 
	As yet another application of CLT in the discrete setting, let us assume that $\{Y_j\}_{j=1}^n$ is independent with $Y_j\sim \mathsf{Pois}(1)$, the Poison distribution as in  Example \ref{poisson:trials}. Thus, $Y^{(n)}=Y_1+\cdots+Y_n\sim \mathsf{Pois}(n)$, 
	and  CLT applies to yield
	\begin{equation}\label{pois:clt}
		\sqrt{n}(n^{-1}{Y^{(n)}-1})\stackrel{d}{\to}\mathcal N(0,1), \quad {n\to+\infty},
	\end{equation}
	that is, 
	\[
	\sum_{a\leq k\leq b}\frac{n^ke^{-n}}{k!}\approx_{n\to +\infty}\frac{1}{\sqrt{2\pi n}}\int_a^be^{-\frac{(x-n)^2}{2n}}dx,
	\]
	which is the same as having
	\[
	\frac{n^ke^{-n}}{k!}\approx_{n\to +\infty}\frac{1}{\sqrt{2\pi n}}e^{-\frac{(k-n)^2}{2n}}
	\]
	uniformly in $k$ such that 
	\[
	k=n+\sqrt{n}O(1).
	\]
	Taking $k=n$ gives
	\begin{equation}\label{stirlings}
		n!\approx_{n\to +\infty}\sqrt{2\pi}n^{n+\frac{1}{2}}e^{-n},
	\end{equation}
	which is {\em Stirling's asymptotic formula}. If we take into account that $\Gamma(n)=(n-1)!$, this clearly implies (\ref{stir:gamma}). \qed
\end{remark}

\begin{remark}\label{just:bern:clt}
	We may directly justify (\ref{bin:clt}), the CLT for a Bernoulli population, as follows. Propositions \ref{charac:p:fol} and \ref{charac:p} applied to (\ref{bin:clt}) give 
	\[
	\phi_{Z_n}(u)=\phi_{X^{(n)}}(u)e^{-{\bf i}\frac{np}{\sqrt{npq}}u}, 
	\]
	so that (\ref{form:ch:sumb}) leads to
	\begin{eqnarray*}
		\phi_{Z_n}(u)
		& = & \left(q+pe^{{{\bf i}\frac{u}{\sqrt{npq}}}}\right)^n e^{-{\bf i}\frac{np}{\sqrt{npq}}u}\\
		& = & \left(qe^{-{\bf i}\sqrt{\frac{p}{nq}}u}+pe^{{\bf i}\sqrt{\frac{q}{np}}u}\right)^n.
	\end{eqnarray*}
	Expanding the exponential terms in parentheses  and performing some cancellations we find that 
	\begin{equation}\label{exp:clt}
		\phi_{Z_n}(u)=\left(1-\frac{u^2}{2n}+o\left(\frac{u^2}{n}\right)\right)^n \stackrel{n\to+\infty}{\longrightarrow} e^{-u^2/2},
	\end{equation}
	which reproduces (\ref{rep:eq:clt}) in this case. We may also obtain a proof of (\ref{pois:clt}), the CLT for a Poisson population, along the same lines. Indeed, this time the left-hand side of (\ref{pois:clt}) is 
	\[
	Z_n=\frac{Y^{(n)}}{\sqrt{n}}-\sqrt{n},
	\]
	so that 
	\begin{eqnarray*}
		\phi_{Z_n}(u)
		& = & \phi_{Y^{(n)}}\left(\frac{u}{\sqrt{n}}\right)e^{-{\bf i}\sqrt{n}u}\\
		& \stackrel{(\ref{char:pois})}{=} & e^{n\left(e^{{\bf i}\frac{u}{\sqrt{n}}}-1\right)} e^{-{\bf i}\sqrt{n}u}\\
		& = & \left(e^{e^{{\bf i}\frac{u}{\sqrt{n}}}-1-{\bf i}\frac{u}{\sqrt{n}}}\right)^n\\
		& = & \left(e^{-u^2/2n+o(u^2/n)}\right)^n\\
		& = & \left(1-\frac{u^2}{2n}+o\left(\frac{u^2}{n}\right)\right)^n,
	\end{eqnarray*}
	so we may proceed as in (\ref{exp:clt}), 
	as claimed. \qed
\end{remark}

The following immediate consequence of CLT, which uses the notation of Example \ref{moment}, is also worth mentioning here.

\begin{theorem}\label{mult:CLT}(Multiplicative CLT)
	If $\{Y_j\}_{j\geq 1}$ is a i.i.d.\!\! sequence of {\em positive} random variables satisfying $\mathbb E(\ln Y_j)=\mu$ and ${\mathbb V}(\ln Y_j)=\sigma^2$ then 
	\[
	\sqrt[n]{\Pi_{j=1}^nY_j}\approx_{n\to +\infty}\Lambda(\mu,\sigma^2/n)=\mathcal L\mathcal N(e^{\mu+\frac{\sigma^2}{2n}},(e^{\frac{\sigma^2}{n}}-1)e^{2\mu+\frac{\sigma^2}{n}}). 
	\]
\end{theorem}

\begin{example}\!\!$\bigstar$\label{gibrat}
	We say that a sequence of positive random variables $\{X_j\}_{j=0}^{+\infty}$ satisfy {\em Gibrat's law of proportionate effect} if there exist random variables $\{Y_j\}_{j=1}^{+\infty}$ such that $Y_j\perp\!\!\perp X_{j-1}$ and the corresponding cdfs satisfy
	\[
	F_{X_j}(z)=\int_{0}^{+\infty}F_{Y_j}(xu^{-1})dF_{X_{j-1}}(u), \quad j\geq 1.
	\] 
	It then follows from (\ref{regra:03}) that $X_j=Y_jX_{j-1}$ and hence
	\[
	X_0^{-1}X_n=\Pi_{j=1}^nY_j, \quad n\geq 1.
	\]
	Thus, if $\{Y_j\}$ is as in Theorem \ref{mult:CLT} we see that 
	\[
	X_0^{-1}X_n\approx_{n\to +\infty}\Lambda(n\mu,n\sigma^2)=\mathcal L\mathcal N(e^{n\mu+n\frac{\sigma^2}{2}},(e^{n{\sigma^2}}-1)e^{2n\mu+n{\sigma^2}}). 
	\]
	Variations of this simple argument go a long way toward explaining the occurrence of lognormal distributions in a large class of natural and social phenomena \cite{aitchison1969lognormal}. \qed  
\end{example}

As a final illustration of the usefulness of Theorem \ref{levy:conv}, 
we now present a result describing the limiting distribution of a sequence of binominal distributions $\mathsf{Bin}(p_n;n)$, with $np_n$ approaching a positive constant, as a Poisson distribution.

\begin{theorem}\label{law:rare}(Law of Rare Events)
	If $X_n\sim \mathsf{Bin}(p_n;n)$ and  
	$np_n\to \lambda>0$  as $n\to+\infty$ then there exists $Z\sim {\mathsf{Pois}}(\lambda)$ such that
	\[
	X_n\stackrel{d}{\to}Z. 
	\]
\end{theorem}

\begin{proof}
	We compute for any $u\in\mathbb R$:
	\begin{eqnarray*}
		\phi_{X_n}(u)
		& \stackrel{(\ref{form:ch:sumb})}{=} & 
		\left(1-p_n+p_ne^{{\bf i}u}\right)^n\\
		& = & 
		\left(1-\frac{\lambda}{n}+\frac{\lambda}{n}e^{{\bf i}u}+o(n^{-1})\right)^n\\
		& = & \left(1+\frac{\lambda}{n}\left(e^{{\bf i}u}-1\right)+o(n^{-1})\right)^n\\
		& \to & e^{\lambda\left(e^{{\bf i}u}-1\right)} \\
		& \stackrel{(\ref{char:pois})}{=} & \phi_{Z}(u).
	\end{eqnarray*}
	Now apply Theorem \ref{levy:conv}.
\end{proof}

\section{Estimation}\label{est:theo}

Here we shall use the theory developed so far to provide an introduction to Estimation Theory, an important topic in Statistics with countless applications. 

\subsection{Parametric estimation and the mean squared error}\label{param:est}
With the preliminary ``large sample'' results of Section \ref{fund:lim} at hand, we now turn our attention to a {\em non-asymptotic} problem that appears very often in
real world applications, where we only have access to {\em finitely} many measurements.

\begin{definition}\label{rand:sample}
	A {\em random sample} is a {finite} family $\{X_j\}_{j=1}^n$ of 
	i.i.d. random variables. 
\end{definition}

We usually represent a random sample by 
\[
X_1,\cdots,X_n\sim\psi,
\]
or simply by $X_j\sim\psi$,
where $\psi$ is the common pdf.
Also, in the following we set $\mathbb E(X_j)=\mu$ and ${\mathbb V}(X_j)=\sigma^2$, $j=1,\cdots,n$.

\begin{definition}\label{stat:mod:def}
	A ({\em parametric}) {\em statistical model} is a random sample  
	\[
	X_1,\cdots,X_n\sim\psi_\theta,
	\]
	where the associated pdf is allowed to depend on the unknown parameter $\theta$ running in a given subset $\Theta\subset \mathbb R^q$. 
\end{definition}

\begin{remark}\label{underlying}
	Implicit in this definition is the existence of an underlying family of probability spaces, say $(\Omega,\mathcal F,\{\mathcal P_\theta\}_{\theta\in\Theta})$, so that $\{X_j\}$ is i.i.d. with respect to each element in this family. 
	Also, by Proposition \ref{inddens} the joint pdf of 
	$(X_1,\cdots,X_n):\Omega\to\mathbb R^n$ is 
	\begin{equation}\label{underlying:2}
		{\bf x} =(x_1,\cdots,x_n)\mapsto \psi_\theta({\bf x}):=\Pi_{j=1}^n\psi(x_j;\theta), 
	\end{equation}
	where $\psi(x_j;\theta)=\psi_\theta(x_j).$.
	\qed
\end{remark} 

\begin{remark}\label{stat:mod:def:2}
	Sometimes it is convenient to enlarge the scope of Definition \ref{stat:mod:def} above in order to include samples $X_j\sim\psi_\theta$ for which the ``identically distributed'' assumption no long holds, so that only independence is retained. In the following, whenever we make use of this extended version of a statistical model, we will make explicit reference to this remark. \qed
\end{remark}

Given the {statistical model} $X_j\sim \psi_\theta$ as above, we will always assume that it is {\em identifiable} in the sense that the map $\theta\mapsto\psi_\theta$ is injective. In any case, the corresponding {\em point estimator problem} consists of finding an {\em estimator} \begin{equation}\label{estim:def}
	\widehat\theta=h(X_1,\cdots,X_n)
\end{equation}
for some {\em statistic}\footnote{To be precise, a statistic is any measurable function $h=h(X_1,\ldots,X_n)$. An estimator for $\theta$, denoted by $\widehat\theta$, is a statistic that does not depend on the unknown parameter $\theta$.
} $h:\mathbb R^n\to\Theta\subset\mathbb R^q$, which is supposed to yield an ``efficient'' guess of the true (and unknown) parameter $\theta\in\Theta$. 
The evaluation $\widehat\theta({\bf x})$ of an estimator at a realization ${\bf x}\in\mathbb R^n$ of a given random sample $X=(X_1,\cdots,X_n)$ is called an {\em estimate}.

The statistical analysis of point estimators naturally divides into two complementary tasks:
\begin{itemize}
	\item One must first evaluate the performance of a given estimator against competing alternatives in order to identify the most suitable choice for the problem under consideration. Typically, the candidates are restricted to a predetermined family of estimators, though there is no guarantee that the ``best'' option within this family will remain optimal in a broader sense. An illustration of this point appears in the performance analysis, under mean squared error, of the variance estimators $\widehat\sigma^2_c$, $c>0$, in Subsection \ref{meas:comp}.
	\item Once an estimator has been selected, it is important to recognize that the true parameter value will almost never coincide exactly with its point estimate. Accordingly, one needs methods to quantify the variability of the random estimate around the true value, rather than reporting the point estimate alone. This leads naturally to the notion of confidence intervals, discussed in Subsection \ref{conf:int:sub}.
\end{itemize}

We thus start here by considering the first issue. 

\begin{definition}\label{unbias}
	The {\em bias} of an estimator $\widehat\theta$ is given by
	\[
	{\rm bias}(\widehat\theta)=\mathbb E(\widehat\theta-\theta).
	\]
	An estimator $\widehat\theta$ is said to be {\em unbiased} if ${\rm bias}(\widehat\theta)=0$ (equivalently, $\mathbb E(\widehat\theta)=\theta$ for any $\theta$). Also, the {\em mean squared error} (mse) of $\widehat\theta$ is  
	\begin{equation}\label{mse:def}
		{\rm mse}\,(\widehat\theta)=\mathbb E(\|\widehat\theta-\theta\|^2).
	\end{equation}
\end{definition}

\begin{remark}\label{stric:sp}
	Strictly speaking, the dependence of the invariants above on $\theta$ should be emphasized. For instance, the unbiasedness condition actually means that $\mathbb E_{\mathcal P_\theta}(\widehat\theta)=\theta$, where 
	\[
	\mathbb E_{\mathcal P_\theta}(X)=
	\int_{\mathbb R^n}{\bf x}\,dP_\theta({\bf x})
	\stackrel{(\ref{underlying:2})}{=}
	\int_{\mathbb R^n}{\bf x}\psi_\theta({\bf x})d{\bf x}, 
	\]
	where $P_\theta=X_\sharp \mathcal P_\theta$ is the distribution of $X$ coming from $\mathcal P_\theta$ (see Remark \ref{underlying}).
	However, in order to keep the notation light, we usually refrain from doing so. Notice also that our notation ignores the dependence of $\widehat\theta$ on the size $n$ of the random sample. Whenever emphasizing this is needed, we write $\widehat\theta=\widehat\theta_n$. \qed
\end{remark}

\begin{proposition}\label{trade:off}(bias-variance trade-off)
	There holds 
	\begin{equation}\label{mse:2}
		{\rm mse}\,(\widehat\theta)={\rm tr}\,{\mathbb C}(\widehat\theta)+\|{\rm bias}(\widehat\theta)\|^2.
	\end{equation}
\end{proposition}

\begin{proof}
	If $\theta\in\mathbb R$ expand
	\[
	{\rm mse}\,(\widehat\theta)=\mathbb E((\widehat\theta-\mathbb E(\widehat\theta)+\mathbb E(\widehat\theta)-\theta)^2)
	\]		
	and check that the crossed terms cancel, thus yielding (\ref{mse}) below. The vector case then follows because 
	\begin{eqnarray*}
		{\rm mse}\,(\widehat\theta)
		& = & \sum_j\mathbb E((\widehat\theta_j-\theta_j)^2)\\
		& \stackrel{(\ref{mse})}{=} & \sum_j {\mathbb V}(\widehat\theta_j)+\sum_j|{\rm bias}(\widehat\theta_j)|^2\\
		& = & {\rm tr}\,{\mathbb C}(\widehat\theta)+\|{\rm bias}(\widehat\theta)\|^2.
	\end{eqnarray*}
\end{proof}

\begin{convention}\label{conv:uni}
	Unless otherwise explicitly stated, we always assume in the sequel that $\theta\in\Theta\subset\mathbb R$, the {\em uni-dimensional} case, so that (\ref{mse:2}) reduces to
	\begin{equation}\label{mse}
		{\rm mse}\,(\widehat\theta)={\mathbb V}(\widehat\theta)+|{\rm bias}(\widehat\theta)|^2.
	\end{equation}
\end{convention}

Here we adopt the viewpoint that the measure of the ``performance'' of an estimator is encoded in the ``smallness'' of the corresponding mse. 
In particular, a bound of the type ${\rm mse}(\widehat\theta)\leq Cn^{-\alpha}$, $\alpha>0$, immediately provides an $O(n^{-\alpha/2})$ convergence rate estimate (in the mean) on how $\widehat\theta$ approaches $\theta$ as $n\to+\infty$. Another kind of convergence of estimators appears in the next definition. 

\begin{definition}\label{cons:deff}
	We say that an estimator $\widehat\theta_n=\widehat\theta$
	as above is {\em consistent} if $\widehat\theta_n\to\theta$ in probability (with respect to $\theta$). 
\end{definition}

\begin{proposition}\label{cons:deff:con}
	If $\widehat\theta_n$ is consistent with a uniformly bounded variance then ${\rm bias}(\widehat\theta_n)\to 0$ as $n\to +\infty$ (thus, $\widehat\theta_n$ is {\em asymptotically unbiased}).  
\end{proposition}

\begin{proof}
	By Proposition \ref{modes},  $\widehat\theta_n\to\theta$ in distribution so that $\mathbb E(\widehat\theta_n)\to \mathbb E(\theta)=\theta$ and hence $\mathbb E(\widehat\theta_n)$ is uniformly bounded (for each $\theta$). Combining this with the bound on the variance we see that $\mathbb E(|\widehat\theta_n|^2)\leq M_\theta$ for some $M_\theta>0$. 
	Now, for any $\varepsilon>0$ we have  
	\begin{eqnarray*}
		|\mathbb E(\widehat\theta_n-\theta)|
		&\leq & 
		|\mathbb E((\widehat\theta_n-\theta){\bf 1}_{|\widehat\theta_n-\theta|<\varepsilon})|+
		|\mathbb E((\widehat\theta_n-\theta){\bf 1}_{|\widehat\theta_n-\theta|\geq \varepsilon})|\\
		& <& \varepsilon + \mathbb E(|\widehat\theta_n|{\bf 1}_{|\widehat\theta_n-\theta|\geq\varepsilon}) + \mathbb E(|\theta|{\bf 1}_{|\widehat\theta_n-\theta|\geq \varepsilon})\\
		& \leq & \varepsilon +
		\sqrt{\mathbb E(|\widehat\theta_n|^2)}\sqrt{\mathbb E({\bf 1}_{|\widehat\theta_n-\theta|\geq \varepsilon})}+|\theta|\mathbb E({\bf 1}_{|\widehat\theta_n-\theta|\geq \varepsilon}),
	\end{eqnarray*}
	where we used Cauchy-Schwarz in the last step.  
	Thus, 
	\[
	|{\rm bias}(\widehat\theta_n)|
	<  \varepsilon +\sqrt{M_\theta}\sqrt{P(|\widehat\theta_n-\theta|\geq \varepsilon)}+|\theta| P(|\widehat\theta_n-\theta|\geq \varepsilon)
	\]
	and since $ P(|\widehat\theta_n-\theta|\geq \varepsilon)\to 0$ 
	the result follows.
\end{proof}

\begin{proposition}\label{mse:consist}
	If ${\rm mse}(\widehat\theta_n)\to 0$ as $n\to +\infty$ then $\widehat\theta_n$ is consistent.
\end{proposition}

\begin{proof}
	By Chebychev's inequality (\ref{cheby:ineq:2}), for any $\varepsilon>0$,
	\[
	P_\theta\left(|\widehat\theta_n-\theta-{\rm bias}(\widehat\theta_n)|\geq\varepsilon\right)\leq\frac{{\mathbb V}(\widehat\theta_n-\theta)}{\varepsilon^2}\to 0,
	\]	
	which means that $\widehat\theta_n-{\rm bias}(\widehat\theta_n){\to}\theta$ in probability (with respect to $\theta$). Since ${\rm bias}(\widehat\theta_n)\to 0$ as well, Theorem \ref{slutsky} applies to ensure that $\widehat\theta_n{\to}\theta$ in probability.
\end{proof}

\begin{definition}\label{an:defin}
	An estimator $\widehat\theta_n$
	as above is {\em asymptotically normal} with {\em asymptotic variance} $\sigma_\theta^2>0$, $\theta\in\Theta$, if there exists $Z_\theta\sim\mathcal N(0,\sigma_\theta^2)$ such that $\sqrt{n}(\widehat\theta_n-\theta)\to Z_\theta$ in distribution (with respect to $\theta$). 
\end{definition}

\begin{proposition}\label{an:imp:cons}
	If $\widehat\theta_n$ is asymptotically normal then it is consistent.
\end{proposition}

\begin{proof}
	If $Z_{\theta,n}=Z_\theta/\sqrt{n}\sim\mathcal N(0,\sigma^2_\theta/n)$ then 
	
	\[
	\widehat\theta_n-\theta-Z_{\theta,n}=
	\frac{1}{\sqrt{n}}\left(\sqrt{n}(\widehat\theta_n-
	\theta-Z_{\theta,n})\right) \stackrel{p}{\to}0.
	\]	
	But  Chebychev's inequality gives, for any $\varepsilon>0$, 
	\[
	P(|Z_{\theta,n}|\geq{\varepsilon})\leq\frac{\sigma_\theta^2}{n\varepsilon^2}\to 0,
	\]
	that is, $Z_{\theta,n}\stackrel{p}{\to}0$
	and the result follows by Theorem \ref{slutsky}. 
\end{proof}

\begin{remark}\label{not:use:var}
	The true nature of the asymptotic variance $\sigma^2_\theta$ has not been explored in the previous argument, which makes sense because consistency concerns position rather than dispersion. 
	It is worth noting, however, that asymptotic normality immediately entails that ${\mathbb V}(\widehat\theta_n)$ remains uniformly bounded (for each $\theta$). Consequently, Proposition \ref{cons:deff:con} ensures that $\widehat\theta_n$ is asymptotically unbiased. This conclusion may also be drawn directly from the definition, since
 $\sqrt{n}\,\mathbb E(\widehat\theta_n-\theta)\to \mathbb E(Z_\theta)=0$. \qed		
\end{remark}

We include here an useful consequence of asymptotic normality.

\begin{proposition}\label{delta:m} (the delta method) If $\widehat\theta_n$ is asymptotically normal (as in Definition \ref{an:defin}) and $g:\Theta\subset\mathbb R\to\mathbb R$ is a $C^1$ function whose derivative vanishes nowhere then $\sqrt{n}(g(\widehat\theta_n)-g(\theta))\to Z_{\theta,g}$ in distribution, where $Z_{\theta,g}\sim\mathcal N(0,|g'(\theta)|^2\sigma_\theta^2)$. 
\end{proposition}

\begin{proof}
	By Taylor,
	\[
	\sqrt{n}\left(g(\widehat\theta_n)-g(\theta)\right)=g'(\widetilde\theta_n)\sqrt{n}\left(\widehat\theta_n-\theta\right),  
	\]
	for some $\widetilde\theta_n$ between $\widehat\theta_n$ and $\theta$. Since $\widehat\theta_n\stackrel{p}{\to}\theta$ by Proposition \ref{an:imp:cons}, it is not hard to check that $g'(\widetilde\theta_n)\stackrel{p}{\to}g'(\theta)$, so the result follows from Theorem \ref{slutsky}. 
\end{proof}

\begin{example}\label{al:mean}
	For {\em any} random sample $\{X_j\}$ as above it is immediate to check that the {\em sample mean}
	\begin{equation}\label{weigh:st}
		\overline X_n:=\frac{1}{n}\left(X_1+\cdots+X_n\right)
	\end{equation}
	is an unbiased estimator for the expected value of the underlying distribution. In other words, $\mathbb E(\overline X_n)=\mu$, where $\mu=\mathbb E\mathbb (X_j)$ is the common expectation. Also, if $\sigma^2={\mathbb V}(X_j)$ is the common variance of the sample (the population variance) then it follows from (\ref{uncorr:var}) that
	\[
	{\mathbb V}(\overline X_n)=\frac{1}{n^2}\sum_{j=1}^n{\mathbb V}(X_j)=\frac{1}{n^2}\sum_{j=1}^n\sigma^2,
	\]
	that is,
	\begin{equation}\label{form:var:mean}
		{\mathbb V}(\overline X_n)=\frac{\sigma^2}{n}
	\end{equation}
	and hence
	$	{\rm mse}(\overline X_n)=\sigma^2/n$. This is the reason why we call
	\[
	Z_n=\frac{\overline X_n-\mu}{\sigma/\sqrt{n}}
	\]
	the {\em standardization} of the sample mean;
	compare with (\ref{stand:samp}). Note that $\mathbb E(Z_n)=0$ and ${\mathbb V}(Z_n)=1$. Finally, note that $\overline X_n$ is consistent (as an estimator for $\mu$) either by LLN or by Proposition \ref{mse:consist} and that for any $g$ as in Proposition \ref{delta:m}, CLT applies to ensure that $g(\overline X_n)$, as an estimator of $g(\mu)$, satisfies  
	\[
	\sqrt{n}\left(g(\overline X_n)-g(\mu)\right)\stackrel{d}{\to}\mathcal N(0,|g'(\mu)|^2\sigma^2), 
	\] 
	hence being asymptotically normal as well. 
	\qed
\end{example}

\begin{example}\label{weigh:est:mu}(Weighted estimators for the population mean) If $w=(w_1,\cdots,w_n)$ is a {\em weight vector} (which means that $\sum_jw_j=1$), then we may consider the corresponding {\em weighted estimator} for $\mu$ given by 
	\[
	\overline{X}_n^w=\sum_{j=1}^nw_jX_j,
	\]  
	which includes (\ref{weigh:st}) as a rather special case. One easily verifies that ${\rm bias}(\overline{X}_n^w)=0$ and  ${\rm mse}(\overline{X}_n^w)={\mathbb V}(\overline{X}_n^w)=\sigma^2\sum_j{w_j^2}$, so the  estimator with the least mse in this class is obtained by minimizing $w\mapsto |w|^2$ under the constraint $\sum_jw_j=1$, which gives $w=(1/n,\cdots,1/n)$, corresponding to the sample mean $\overline X_n$. \qed
\end{example}

\begin{example}\label{monte:carlo}(Monte Carlo estimator)
	Let $X:\Omega\to\mathbb R^m$ be a random vector with a pdf $\psi$ whose support is contained in the unit cube $[0,1]^m$ and let $f:[0,1]^m\to\mathbb R$ be such that $f\psi$ is Lebesgue integrable. If $X_j\sim \psi$ is a i.i.d. sample it is immediate from Example \ref{al:mean} that the {\em Monte Carlo estimator}
	\[
	\widehat\mu^f_{(n)}:=\frac{1}{n}\sum_{j=1}^nf(X_j),
	\]
	is an unbiased estimator for the unknown parameter
	\begin{equation}\label{int:mont}
		\mu^f:=\mathbb E(f(X_j))=\int_{\mathbb [0,1]^m}f(x)\psi(x)dx
	\end{equation}
	which is {consistent} because 
	\begin{equation}\label{lim:monte}
		\lim_{n\to+\infty}\widehat\mu^f_{(n)}= \mu^f
	\end{equation}
	in probability by LLN. As usual, this also follows from Proposition \ref{mse:consist}, given that 
	\begin{equation}\label{est:l2}
		{\rm mse}(\widehat\mu^f_{(n)})={\mathbb V}(\widehat\mu^f_{(n)})=\frac{\sigma_f^2}{{n}},
	\end{equation}
	where $\sigma_f^2$ is the (common) variance of $f(X_j)$. Notice that this conveys a $O(n^{-1/2})$ convergence rate for (\ref{lim:monte}) which can be made explicit if   
	we apply Chebyshev's inequality (\ref{cheby:ineq:2}) with $X=\widehat\mu^f_{(n)}$, $\sigma=\sigma_f/\sqrt{n}$ and $c=1/\sqrt{\delta}$, $\delta>0$, so that 
	\[
	P\left(\left|\widehat\mu^f_{(n)}-\mu^f\right|\leq
	\frac{1}{\sqrt{\delta}}\frac{\sigma_f}{\sqrt{n}}\right)\geq 1-\delta.
	\]
	Thus, one needs at least
	\begin{equation}\label{n:effec}
		n\approx \frac{1}{\delta}\frac{\sigma_f^2}{\varepsilon^2}
	\end{equation}
	samples 
	in order to obtain a dispersion of at most $\varepsilon$ of the estimator around the expected value $\mu^f$ with probability at least $1-\delta$. 						
	This may be substantially improved if we appeal to CLT (Theorem \ref{clt}) to obtain
	\[
	\lim_{n\to+\infty}P\left(\left|\widehat\mu^f_{(n)}-\mu^f\right|\leq \eta\frac{\sigma_f}{\sqrt{n}}\right)= \frac{1}{\sqrt{2\pi}}
	\int_{-\eta}^{\eta}e^{-x^2/2}dx
	\stackrel{(\ref{exp:bound:1:2})}{\approx} 1-e^{-\eta^2/2},\quad \eta\to 0,
	\]
	which allows us to replace the previous estimate by 
	\begin{equation}\label{n:asym}
		n\approx 2\ln(1/\delta)\frac{\sigma_f^2}{\varepsilon^2}. 
	\end{equation}	
	One should be aware, however, that whereas the estimate (\ref{n:effec}) holds non-asympt\-ot\-ically, as its comes from Chebyshev's inequality, the estimate (\ref{n:asym}) becomes reliable only in the asymptotic regime ($n\to+\infty$); see Remarks \ref{lln:append} and \ref{miscon}. 						
	Regardless of the shape of their dependence on $\delta$, the estimates above share 
	the nice property of not depending on $m$, so that the ``dimensionality curse'' is not present here. Of course this is one of the reasons why Monte Carlo methods, based on (\ref{lim:monte}), are quite versatile  in approximating multiple integrals as those in the right-hand side of (\ref{int:mont})\footnote{This method should be compared with the usual numerical approach which requires evaluation of $f\psi$ on a $\varepsilon$-net and hence has a complexity that grows like $\varepsilon^{-m}$.}. Besides its slow $O(n^{-1/2})$ convergence rate, an obvious drawback of this method is its explicit dependence on the standard deviation $\sigma_f$, which is at least as hard to compute as $\mu^f$ itself. The simplest choices avoiding this latter problem (by explicitly bounding the variance) corresponds to taking $\{X_j\}$ {\em uniformly} distributed in $[0,1]^m$ (so that $\psi={\bf 1}_{[0,1]^m}$), and $f={\bf 1}_{B}$, the indicator of a Borel set $B\subset [0,1]^m$, so that $\mu^f={\rm vol}_m(B)$, the $m$-volume of $M$, and $\sigma_f^2={\rm vol}_m(B)(1-{\rm vol}_m(B))\leq 1/4$. Thus, at least the volumes of (well-behaved) Borel subsets can be efficiently calculated if we are able to provide low cost simulations of independent, uniformly distributed random variables on $[0,1]^m$ \cite{robert2010introducing}.  \qed
\end{example}

\subsection{Computing the mean squared error of  $\widehat\sigma^2_c$}\label{meas:comp}
We further illustrate  the concepts introduced in Definition \ref{unbias} by    
determining the ``best'' estimator for the variance $\sigma^2$ in the family 
\begin{equation}\label{fam:est}
	\widehat\sigma^2_c(X)=h_c(X_1,\cdots,X_n),\quad c>0,
\end{equation}
where 
\[
h_c(X_1,\cdots,X_n)=c\sum_{j=1}^n(X_j-\overline X_n)^2
\]
and
\[
\overline X_n=\frac{1}{n}\left(X_1+\cdots+X_n\right) 
\]
is the {sample mean}; see Example \ref{al:mean}. Thus, $\theta=\sigma^2>0$ and $\Theta=\mathbb R_+$. 
This involves minimizing the corresponding {mean squared error} ${\rm mse}\,(\widehat\sigma^2_c)$ viewed as a function of $c$\footnote{Whenever no confusion arises, we will omit the dependence of $\widehat\sigma^2_c$ on $X=(X_1,\cdots,X_n)$.}.

\begin{proposition}\label{est:theor}
	One has 
	\begin{equation}\label{est:theor:2}
		{\rm bias}(\widehat\sigma^2_c)=\left(c(n-1)-1\right)\sigma^2.
	\end{equation}
\end{proposition}

\begin{proof}
	Using that
	\begin{equation}\label{expans:xj}
		X_j-\overline X_n=\frac{n-1}{n}X_j-\frac{1}{n}\sum_{k\neq j}X_k
	\end{equation}
	we first note that $\mathbb E(X_j-\overline X_n)=0$ 
	and hence
	\[
	\mathbb E(\widehat\sigma^2_c)=c\sum_{i=1}^n\mathbb E((X_j-\overline X_n)^2)=c\sum_{j=1}^n{\mathbb V}(X_j-\overline X_n). 
	\] 
	From 
	the independence assumption and (\ref{uncorr:var}) we get, again using (\ref{expans:xj}), 
	\[
	{\mathbb V}(X_j-\overline X_n)=\frac{(n-1)^2}{n^2}{\mathbb V}(X_j)+\frac{1}{n^2}\sum_{k\neq j}{\mathbb V}(X_k)=\frac{n-1}{n}\sigma^2, 
	\]
	so that 
	\begin{equation}\label{psi:theta}
		\mathbb E(\widehat\sigma^2_c)=cn\frac{n-1}{n}\sigma^2=c(n-1)\sigma^2,
	\end{equation}
	and the result follows.
\end{proof}

\begin{corollary}\label{unb:sigmac}
	$\widehat\sigma^2_c$ is {unbiased} only if $c=(n-1)^{-1}$. 
\end{corollary}

\begin{definition}\label{samp:cov}
	The {\em sample variance} of the sample $\{X_j\}$ as above is	
	\[
	S_n^2:=\widehat\sigma^2_{(n-1)^{-1}}=\frac{1}{n-1}\sum_{j=1}^n(X_j-\overline X_n)^2.
	\]
	Also, $S_n=\sqrt{S_n^2}$ is the {\em sample standard deviation}. 
\end{definition}

Recall that in general an unbiased estimator $\widehat\theta$ satisfies $\mathbb E(\widehat\theta)=\theta$, so that the target parameter $\theta$ is the expected value of the corresponding sample distribution. Intuitively, the unbiasedness property says that on average the estimator hits the right target. This is the main reason why unbiased estimators are often used in applications and we provide below two classical results (Theorems \ref{gauss:markov} and \ref{cr:rao:th}) ensuring that unbiased estimators minimize their variance (and hence their mse) within certain classes of competing unbiased estimators. However, we point out that the family $\widehat\sigma^2_c$ above may be used 
to illustrate that in general the best variance estimator might not be unbiased (that is, the function $c\mapsto {\rm mse}(\widehat\sigma^2_c)$ is minimized for some $c\neq (n-1)^{-1}$), which turns out to be a manifestation of the variance-bias trade-off in (\ref{mse}). For this we need to compute ${\mathbb V}(\widehat\sigma^2_c)$, which we do by assuming in the rest of the calculation  that each $X_j$ is normally distributed: $X_j\sim\mathcal N(\mu,\sigma^2)$.  
Let us set 
\[
U^2=\sum_{j=1}^n\left(\frac{X_j-\mu}{\sigma}\right)^2, \quad 	V^2=n\left(\frac{\overline X_n-\mu}{\sigma}\right)^2.
\]

\begin{proposition}\label{sum:sq:stnorm}
	If $\{X_j\}$ is independent with $X_j\sim\mathcal N(\mu,\sigma^2)$ then $U^2\sim \chi^2_n$ and $V^2\sim \chi^2_1$.
\end{proposition}

\begin{proof}
	Note that $\sigma^{-1}(X_j-\mu)\sim \mathcal N(0,1)$ by Proposition \ref{norm:spce}, which can also be used to check that $V^2=W^2$, where  
	\[
	W=\frac{\overline X_n-\mu}{\sigma/\sqrt{n}}\sim \mathcal N(0,1).
	\]
	The results then follow from Proposition \ref{sum:norm:sq}.
\end{proof}

\begin{proposition} \label{indep:ms}
	If $\{X_j\}$ is independent and  $X_j\sim \mathcal N(\mu,\sigma^2)$ then
	$\overline X_n$ and 
	\[
	\mathcal S^2:=\sum_j(X_j-\overline X_n)^2
	\]
	are independent. In particular, $\overline X_n\perp\!\!\perp \widehat\sigma^2_c$, $c>0$.
\end{proposition}

\begin{proof}
	We may assume that $\mu=0$. Note that 
	\[
	\mathcal S^2=\sum_jX^2_j-Y_1^2,
	\]
	where 
	\begin{equation}\label{express:y1}
		Y_1=\sqrt{n}\overline X_n=\sum_j\frac{X_j}{\sqrt{n}}
	\end{equation}
	is normal.
	By Gramm-Schmidt there exists an orthogonal $n\times n$ matrix, say  $O$,  whose first line is the vector $(1/\sqrt{n},\dots,1/\sqrt{n})$ i.e. $O_{1j}=1/\sqrt{n}$. If $Y=OX$ then  $Y_1$ is indeed given by (\ref{express:y1}) so that
	\[
	\mathcal S^2=\|X\|^2-Y_1^2=\|Y\|^2-Y_1^2=\sum_{l=2}^{n}Y_l^2. 
	\]
	From Corollary \ref{unc:ind:n:c}, $Y'=(Y_2,\cdots,Y_{n})$ is normally distributed and $\{Y_l\}_{l=2}^n$ is independent as well so its covariance matrix is diagonal. Moreover, using again the independence of $\{X_j\}$, we compute for $l\geq 2$ that
	\begin{eqnarray*}
		{\mathbb C}(Y_1,Y_l)
		& = & {\mathbb C}\left(\sum_jO_{1j}X_j,\sum_k O_{lk}X_k\right)\\
		& = & \sigma^2\sum_j O_{1j}O_{lj}\\
		& = & 0,
	\end{eqnarray*}
	so that $\{Y_j\}_{j=1}^n$  is independent by Proposition \ref{unc:ind:n}. In particular, $\mathcal S^2=\|Y'\|^2\perp\!\!\perp Y_1/\sqrt{n}= \overline X_n$, as claimed.     
\end{proof}

\begin{proposition}\label{u:v}
	If 
	\begin{equation}\label{sigma:def}
		\Sigma^2:=\sigma^{-2}\sum_j(X_j-\overline X_n)^2
	\end{equation}
	then
	\begin{equation}\label{chi2:nm1}
		\Sigma^2\sim\chi^2_{n-1}.
	\end{equation}
	In particular,  
	\begin{equation}\label{var:c:cp}
		{\mathbb V}(\Sigma^2)=2(n-1).
	\end{equation}
\end{proposition}

\begin{proof} Upon multiplication by $\sigma^{-2}$,
	the elementary algebraic identity
	\begin{equation}\label{U:sigma:V:0}
		\sum_{j=1}^n\left(X_j-\mu\right)^2=\sum_{j=1}^n\left(X_j-\overline X_n\right)^2+ n\left(\overline X_n-\mu\right)^2
	\end{equation}
	becomes 
	\begin{equation}\label{U:sigma:V}
		U^2=\Sigma^2+V^2.
	\end{equation}
	Since $\{\Sigma^2,V^2\}$ is independent (by 
	Proposition \ref{indep:ms})  we have $\phi_{U^2}=\phi_{\Sigma^2}\phi_{V^2}$ so that Corollary \ref{gamma:cf} applies to give
	\[
	\phi_{\Sigma^2}(u)=(1-2u{\bf i})^{-(n-1)/2},
	\]
	which yields (\ref{chi2:nm1}). 
	Finally, (\ref{var:c:cp}) follows from Corollary \ref{chi:sq:ms}.
\end{proof} 

\begin{remark}\label{U:sigma:V:r}
	The parameter $\mu$ plays no essential role in the validity of the identity (\ref{U:sigma:V:0}), 
	the only relevant point being that $\overline X_n$ is the arithmetic mean of $\{X_j\}_{j=1}^n$. Thus, (\ref{U:sigma:V:0}) 
	remains true if $\mu$ gets replaced by any real number:
	\begin{equation}\label{U:sigma:V:c}
		\sum_{j=1}^n\left(\mathfrak y_j-c\right)^2=\sum_{j=1}^n\left(\mathfrak y_j-\overline {\mathfrak y}_n\right)^2+ n\left(\overline{\mathfrak y}_n-c\right)^2, \quad c\in\mathbb R,
	\end{equation}
	where 
	\[
	\overline{\mathfrak y}_n=\frac{\mathfrak y_1+\cdots+\mathfrak y_n}{n}.
	\]
	In this more general form, this important identity resurfaces at many points below (see Remark \ref{fis:comp:new}, and Examples \ref{corr:dist} and \ref{ow:anova}).  \qed
\end{remark}

We may record the result of our computation as follows. 

\begin{proposition}\label{stat:normal}
	If $\{X_j\}$ is independent with $X_j\sim \mathcal N(\mu,\sigma^2)$ then 
	\begin{equation}\label{stat:norm:eq:0}
		{\mathbb V}(\widehat\sigma^2_c)=2(n-1)c^2\sigma^4.
	\end{equation}
	As a consequence, 
	\begin{equation}\label{stat:normal:eq}
		{\rm mse}\,(\widehat\sigma^2_c)=\left((n-1)(n+1)c^2-2(n-1)c+1\right)\sigma^4.
	\end{equation}
\end{proposition}

\begin{proof}
	Combine (\ref{mse}), (\ref{est:theor:2}) and (\ref{var:c:cp}).
\end{proof}

\begin{corollary}\label{stat:normal:c}
	Under the conditions above,  ${\rm mse}\,(\widehat\sigma^2_c)$
	is minimized for $c=(n+1)^{-1}$.
\end{corollary}

\begin{remark}\label{comp:mse}
	Since 
	\[
	{\rm bias}\widehat\sigma^2_{(n+1)^{-1}}=-2(n+1)^{-1}\sigma^2,
	\]
	which only vanishes in the asymptotic limit $n\to+\infty$, as already advertised Corollary \ref{stat:normal:c} illustrates that an unbiased estimator may fail to be the most efficient one (if the ``performance'' is measured by ${\rm mse}$); a quite similar phenomenon, involving the so-called James-Stein estimator for the mean of certain normal random vectors, appears in Example \ref{james:stein}.  \qed
\end{remark}

\begin{remark}\label{student} (Studentized mean)
	If $Z\sim \mathcal N(0,1)$ and $W\sim\chi^2_k$  then Proposition \ref{norm:chi:stu} says that 
	\[
	\frac{Z}{\sqrt{W/k}}\linebreak\sim \mathfrak t_k,
	\]  
	the {Student's $\mathfrak t$-distribution} with $k\geq 1$ degrees of freedom \cite{student1908probable,fisher1925applic}. In the setting of Proposition \ref{stat:normal} (that is, under sample normality) we may apply this to 
	$Z=\sqrt{n}(\overline X_n-\mu)/\sigma$ (which is $\mathcal N(0,1)$ by  Proposition \ref{norm:spce}) and $W=\widehat\sigma^2_{\sigma^{-2}}$ (just use (\ref{chi2:nm1})) to conclude that 
	\begin{equation}\label{stu:est}
		T_{n-1}:=\frac{\overline X_n-\mu}{S_n/\sqrt{n}}\sim \mathfrak t_{n-1}, 
	\end{equation}
	where $S_n$ is the sample standard deviation (Definition \ref{samp:cov}). Notice that $Z\perp\!\!\perp W$ here by Proposition \ref{indep:ms}\footnote{This independence between the sample mean $\overline X_n$ and the sample standard deviation $S_n$, which is crucial in precisely determining the shape of the sampling distribution of $T_{n-1}$, turns out to be a characteristic feature of normal samples; see \cite{lukacs1942characterization} for a proof which is a clever application of Proposition \ref{charac:p} (3).}.
	We say that $T_{n-1}$ is the {\em studentized mean} of the normal sample $\{X_j\}$. It plays a key role in finding ``small sample'' estimates for the population mean of a normally distributed sample with no prior knowledge of the population variance; see Subsection \ref{conf:int:sub} below. \qed
\end{remark}

\begin{remark}\label{non:norm:s:var}
	If the random sample $\{X_j\}_{j=1}^n$ is not necessarily normal then a somewhat tedious computation gives
	\begin{equation}\label{stat:norm:eq:2}
		{\mathbb V}(\widehat\sigma^2_c)=\frac{(n-1)^2}{n}c^2\sigma^4
		\left(\kappa(X_j)-\frac{n-3}{n-1}\right),			
	\end{equation}
	where 
	\[
	\kappa(X)=\frac{\mathbb E((X-\mathbb E(X))^4)}{{\mathbb V}(X)^2}
	\]
	is the {\em kurtosis}  of $X$, which is finite if we require that 
	$\mathbb E(|X|^4)<+\infty$; see \cite{oneill2014some} for this and many other moment computations.
	In the normal case we have  $\kappa(X_j)=3$ and (\ref{stat:norm:eq:2}) reduces to (\ref{stat:norm:eq:0}). Of course, this general computation suffices if we are merely interested in the conclusions of Proposition \ref{stat:normal}, but we stress that the elegant argument above based on sample normality has the added bonus of yielding an explicit expression for the sampling distribution of the studentized mean $T_{n-1}$ considered in Remark \ref{student}. \qed
\end{remark}

\begin{remark}\!\!$\bigstar$\label{fis:comp:new}
	(The geometric way to Student)
	The calculation leading to Proposition \ref{u:v}, in particular the independence between the sample mean $\overline X_n$ and the sample variance $S_n^2$ in Proposition \ref{indep:ms} which plays a central role in accessing Student's distribution in (\ref{stu:est}) above, may be retrieved by means of the ``$n$-space computations'' due to R. Fisher already mentioned in Remarks \ref{fisher:comp:1} and \ref{fisher:comp:2}\footnote{Recall that we always represent a realization of a random variable, say $\overline X_n$, by the corresponding lower-case symbol, in this case $\overline x_n$.}. Indeed, the probability density spanned by the independent normal random vector $X=(X_1,\cdots,X_n)$, $X_j\sim\mathcal N(\mu,\sigma^2)$, in an infinitesimal region of volume $dx=dx_1\cdots dx_n$ is 
	\begin{equation}\label{joint:x:s2:th}
		\frac{1}{(2\pi)^{n/2}\sigma^n}e^{-\frac{\sum_j(x_j-\mu)^2}{2\sigma^2}}dx 
		= 
		\frac{1}{(2\pi)^{n/2}\sigma^n} e^{-\frac{n(\overline x_n-\mu)^2}{2\sigma^2}}
		e^{-\frac{(n-1)s_n^2}{2\sigma^2}}dx,
	\end{equation}
	where Remark \ref{U:sigma:V:r} has been used.
	Following \cite{fisher1925applic} we now observe that $\overline x_n$ is proportional to the height of $x$ with respect to the hyperplane $H^{n-1}$ defined by $\Sigma_j x_j=0$, whereas $s_n$ is proportional to the distance of $x$ to the line $l^1$ given by $x_1=\cdots =x_n$, with ${\mathscr X_n}:=(\overline x_n,\cdots,\overline x_n)\in l^1$ realizing this distance. Since $H^{n-1}$ and $l^1$ are perpendicular to each other, we may use the corresponding ``cylindrical'' coordinate system to check that $dx$ is proportional to $s_n^{n-2} d\overline x_n ds_nd\theta$, where $d\theta$ is the volume element of the unit sphere $\mathbb S^{n-2}\subset H^{n-1}$ with center located at $H^{n-1}\cap l^1$, the origin of $H^{n-1}$. Leading this to (\ref{joint:x:s2:th}), integrating with respect to $\theta$ and using Proposition \ref{pdf:marg} shows that the infinitesimal joint probability density of the random vector $(\overline X_n,S_n^2)$ is    
	\begin{equation}\label{joint:x:s2}
		\psi_{(\overline X_n,S_n^2)}(\overline x_n,s_n^2)d\overline x_n ds_n^2\approx	e^{-\frac{n(\overline x_n-\mu)^2}{2\sigma^2}}d\overline x_n \times 	e^{-\frac{(n-1)s_n^2}{2\sigma^2}}(s_n^2)^{\frac{n-3}{2}}ds_n^2,
	\end{equation}
	where $\approx$ here means that we are neglecting certain normalizing constants which will take care of themselves. Incidentally, this geometric argument makes it obvious the connection to the previous computational proof of Proposition \ref{indep:ms}: the orthogonal map $Y=OX$ used there carries $H^{n-1}$ onto the coordinate hyperplane $Y_1=0$, which has the net effect of reducing the size of the sample data by one, thus allowing for an induction argument based on Corollary \ref{unc:ind:n:c}. Moreover, it 
	has a number of consequences which we now describe.
	\begin{itemize}
		\item Clearly, (\ref{joint:x:s2}) implies  
		that $\{\overline X_n,S_n^2\}$ is independent.  
		\item Also, it follows from (\ref{joint:x:s2}) that 
		\[
		\psi_{S^2_n}(s_n^2)ds_n^2\approx e^{-\frac{(n-1)s_n^2}{2\sigma^2}}(s_n^2)^{\frac{n-3}{2}}ds_n^2,
		\]
		so if we combine this with 
		(\ref{sigma:def}) we see that 
		\begin{equation}\label{chi:geom}
			\psi_{\Sigma^2}(s_\sigma^2)ds^2_{\sigma}\approx e^{-s_\sigma^2/2}(s_\sigma^2)^\frac{n-3}{2}ds_\sigma^2, \quad s^2_\sigma=\frac{(n-1)s_n^2}{\sigma^2},  
		\end{equation}
		from which we easily deduce (\ref{chi2:nm1}); compare with the computation in Remark \ref{fisher:comp:1}. 
		\item 
		Finally, the geometric argument also provides another way of explicitly computing the probability density of the studentized mean in (\ref{stu:est}). Indeed, (\ref{joint:x:s2}) implies  that  
		\begin{equation}\label{lead:back}
		\psi_{\overline X_n}(\overline x_n)d\overline x_n\approx e^{-\frac{n(\overline x_n-\mu)^2}{2\sigma^2}}d\overline x_n\approx (s_n^2)^{1/2}e^{-\frac{s_n^2t_{n-1}^2}{2\sigma^2}}dt_{n-1},
		\end{equation}
		where for $s_n^2$ fixed we set  
		\[
		t_{n-1}=\frac{\overline x_n-\mu}{s_n/\sqrt{n}},
		\]
		so that 
		\[
		d\overline x_n=\frac{s_n}{\sqrt{n}}	dt_{n-1}
		\]
		by the previously established independence of $\{\overline X_n,S_n\}$.
		Replacing (\ref{lead:back}) back into (\ref{joint:x:s2}) yields an explicit expression for the joint density $\psi_{(T_{n-1},S_n^2)}(t_{n-1},s_n^2)dt_{n-1}ds_n^2$, so that integration with respect to $s_n^2$ gives 
		\begin{eqnarray*}
			\psi_{T_{n-1}}(t_{n-1})dt_{n-1}
			& \approx &  \left(\int_0^{+\infty}(s_n^2)^{\frac{n-2}{2}}e^{-\frac{s_n^2(n-1+t_{n-1}^2)}{2\sigma^2}}ds_n^2\right)dt_{n-1} \\
			& \approx & \left(n-1+t_{n-1}^2\right)^{-n/2}dt_{n-1}\\
			& \approx & \mathfrak t_{n-1}(t_{n-1})dt_{n-1},
		\end{eqnarray*}
		as claimed. \qed 
	\end{itemize}
\end{remark}

\subsection{Confidence intervals}\label{conf:int:sub}
If $\widehat\theta$ is an {\em unbiased} estimator for the parameter $\theta$  whose standard deviation $\sigma_{\widehat\theta}$ is known then Chebyshev's inequality (\ref{cheby:ineq:2}) gives
\[
P(|\widehat\theta-\theta|\leq c\sigma_{\widehat\theta})\geq 1-c^{-2}, \quad c>1,
\] 
which translates into a ``confidence interval'' estimate for the unknown parameter: 
\begin{equation}\label{conf:int:th}
	\theta\in [\widehat\theta\mp c\sigma_{\widehat\theta}]\,{\rm with}\,{\rm prob.}\,{\rm at}\,{\rm least}\,1-c^{-2},
\end{equation}
where here and in the following we denote the interval $[a-b,a+b]$ simply by $[a\mp b]$ whenever convenient.
In particular, this applies to $\widehat\theta=\overline X_n$, the sample mean estimator, which is an unbiased estimator for the population mean $\mu$ (by Example \ref{al:mean}). 
Since $\sigma_{\overline X_n}=\sigma/\sqrt{n}$, if we take $c^{-2}\approx\delta$ and $c\sigma/\sqrt{n}\approx \epsilon$, where $\delta$ and $\epsilon$ are arbitrarily small positive real numbers, then we see that 
\begin{equation}\label{lln:new}
	n>\frac{\sigma^2}{\epsilon^2\delta} \Longrightarrow 
	\mu\in [\overline X_n\mp\epsilon]\,{\rm with}\,{\rm prob.}\,
	\approx
	\,1-\delta.
\end{equation} 
Notice that this retrieves the ``convergence in probability'' version of LLN (under the additional assumption that $\sigma$ is finite); see Remarks \ref{lln:append} and \ref{miscon}. 
We may also turn this into a confidence interval estimate as in (\ref{conf:int:th}):
\begin{equation}\label{conf:int:ls:0}
	\mu\in \left[\overline X_n\mp \frac{1}{\sqrt{\delta}}\frac{\sigma}{\sqrt{n}}\right]\,{\rm with}\,{\rm prob.}\,\approx\,1-\delta,
\end{equation}
If we further require that
the sample is normally distributed ($X_j\sim \mathcal N(\mu,\sigma^2)$) then we can employ
\begin{equation}\label{relax}
	Z_n:= \frac{\overline X_n-\mu}{\sigma/\sqrt{n}}\sim\mathcal N(0,1),
\end{equation} 
where we used Proposition \ref{norm:spce}, 
to find a ``small sample'' confidence interval for the unknown expected value:
\begin{equation}\label{conf:int:ls}
	\mu\in \left[\overline X_n\mp z_{1-\delta/2}\frac{\sigma}{\sqrt{n}}\right]\,{\rm with}\,{\rm prob.}\,\approx\,1-\delta,
\end{equation}
where, for a given $0<\beta<1$, the {\em normal quantile} $z_{\beta}>0$ is determined by  $P(Z\leq z_{\beta})=\beta$, where $Z\sim \mathcal N(0,1)$\footnote{In order to unify the notation for quantiles, we will always impose that their subscripts denote cumulative probabilities. For instance, if $\delta=0.05$ then $z_{1-\delta/2}\approx 1.96$, which shows that roughly two standard deviations around the normal mean suffice to ensure the customary $95\%$ confidence statement.}. 
Notice that if we (more realistically!) relax the normality assumption then (\ref{conf:int:ls}) becomes a ``large sample'' estimate since (\ref{relax}) holds asymptotically as $n\to+\infty$ due to CLT. Upon comparison with (\ref{conf:int:ls}) we see that this amounts to replacing $1/\sqrt{\delta}$ by $z_{1-\delta/2}$ in the estimate for the dispersion around the sample mean\footnote{Thus, if $\delta=0.05$ we are replacing $1/\sqrt{0.05}\approx 4.47$ by $1.96$, which shrinks the dispersion by a factor of $4.47/1.96\approx 2.28$ while still retaining the same confidence level. But recall that this reduction only becomes reliable in the asymptotic regime (Remarks \ref{lln:append} and \ref{miscon}).}.
In any case, the estimates (\ref{conf:int:ls:0}) and (\ref{conf:int:ls}) remain ineffective as long as $\sigma$ is unknown, 
in which case 
(and coming back to a normal sample $X_j\sim\mathcal N(\mu,\sigma^2)$), Remark \ref{student} suggests replacing (\ref{relax}) by (\ref{stu:est}) so as 
to obtain the ``small sample'' estimate
\begin{equation}\label{int:conf:tt}
	\mu\in \left[\overline X_n\mp \mathfrak t_{n-1,1-\delta/2}\frac{S_n}{\sqrt{n}}\right]
	\,{\rm with}\,{\rm prob.}\,\approx\,1-\delta,
\end{equation}
where $P(T_{n-1}\geq \mathfrak t_{n-1,\beta})=1-\beta$ is the ``tail'' probability associated to the $\mathfrak t$-distribution $\mathfrak t_{n-1}$ which defines the corresponding quantile $\mathfrak t_{n-1,\beta}$. The obvious advantage of (\ref{int:conf:tt}) over (\ref{conf:int:ls}) is that no previous knowledge of $\sigma$ is used. Finally, note that in Remark \ref{normal:ww} below it is shown that $\widehat\sigma^2_{n^{-1}}\stackrel{p}{\to}\sigma^2$ under this normality assumption. Since 
\[
\widehat\sigma^2_{(n-1)^{-1}}=\frac{n}{n-1}\widehat\sigma^2_{n^{-1}}, 
\]
we see from Theorem \ref{slutsky} that 
\begin{equation}\label{inf:pop}
	S_n=\sqrt{\widehat\sigma^2_{(n-1)^{-1}}}\stackrel{p}{\to}\sigma.
\end{equation} 
It then follows from  
\[
T_{n-1}=\frac{\sigma}{S_n} Z_n
\] 
and Theorem \ref{slutsky} that (\ref{conf:int:ls}) and (\ref{int:conf:tt}) provide essentially the same information in this asymptotic regime (in the sense that $T_{n-1}-Z_n\stackrel{p}{\to}0$). We stress, however, the usefulness of (\ref{int:conf:tt}) when dealing with small samples, which attests in favor of Student's fundamental contribution coming from Remark \ref{student}.

\begin{remark}\label{kurt:sigma}
	The convergence in (\ref{inf:pop}) holds more generally (that is, with no normality assumption) if we assume that the random sample satisfies $\mathbb E(|X_j|^4)<+\infty$. Indeed, we already know from Corollary \ref{unb:sigmac} that ${\rm bias}(S^2_n)=0$. Also, from (\ref{stat:norm:eq:2}) with $c=(n-1)^{-1}$ we see that ${\mathbb V}(S^2_n)\to 0$. Thus, ${\rm mse}(S^2_n)\to 0$ and Proposition \ref{mse:consist} applies to ensure that $S_n^2\stackrel{p}{\to}\sigma^2$, from which (\ref{inf:pop}) follows. \qed
\end{remark}

\begin{remark}\label{conf:inter}
	It is important to have in mind that if we evaluate the sample mean $\overline X_n$ through a measurement so as to obtain a numerical value, say $\mu_n$, then the corresponding realization of (\ref{conf:int:ls}), namely, 
	\begin{equation}\label{real:intconf}
		\mu\in \left[\mu_n\mp z_{1-\delta/2}\frac{\sigma}{\sqrt{n}}\right]
		\,{\rm with}\,{\rm prob.}\,\approx\,1-\delta,
	\end{equation}
	is completely devoid of sense. Indeed, since any trace of randomness has been removed from  the interval in (\ref{real:intconf}) (it has now become deterministic!) then either $\mu$ definitely belongs to it or not, with probability $0$ or $1$. Thus, the proper way to interpret (\ref{conf:int:ls}) is to regard the corresponding interval as stochastic in nature and to adopt the ``frequentist'' perspective according to which the ``relative frequency'' that 
	\[
	\mu\in \left[\mu_n-z_{1-\delta/2}\frac{\sigma}{\sqrt{n}},
	\mu_n+z_{1-\delta/2}\frac{\sigma}{\sqrt{n}}
	\right]
	\]
	approaches $1-\delta$ as the number of successive realizations $\mu_n$ of $\overline X_n$ becomes larger and larger. Needless to say, similar remarks hold for (\ref{int:conf:tt}) under realizations of $\overline X_n$ and $S_n$. 
	\qed 
\end{remark}

\begin{remark}\label{pivotal}
	Strictly speaking, $Z_n$ and $T_{n-1}$ do {not} qualify as estimators as they are statistics which depend on the underlying parameters. Instead, they are referred to as  {\em pivotal quantities}, 
	a terminology incorporating the appreciated property
	that their distributions do not depend on these parameters.  \qed
\end{remark}

\begin{example}\label{samp:ber}(Sampling from a Bernoulli population)
	If $X_j\sim \mathsf{Ber}(p)$ as in Remark \ref{dem:lap} 
	then the computation leading to (\ref{bin:clt}) also gives
	\begin{equation}\label{ber:pop:clt}
		\sqrt{n}\frac{\overline X_n-p}{\sqrt{p(1-p)}}\stackrel{d}{\to}\mathcal N(0,1), 
	\end{equation}
	where $\overline X_n=X^{(n)}/n$ is the corresponding sample mean. This translates into the ``large sample'' estimate
	\[ 
	p\in \left[\overline X_n\mp z_{1-\delta/2}\frac{\sqrt{p(1-p)}}{\sqrt{n}}\right]
	\,{\rm with}\,{\rm prob.}\,\approx\,1-\delta,
	\]
	which displays the usual drawback, namely, the size of the confidence interval depends on $p$, the parameter we want to estimate. The conservative way of remedying  this is to implement the rather crude estimate $p(1-p)\leq 1/4$ to eliminate the dependence on $p$. Another, certainly more effective, route consists of combining  Theorem \ref{slutsky} and LLN to replace (\ref{ber:pop:clt}) by 
	\[
	\sqrt{n}\frac{\overline X_n-p}{\sqrt{\overline X_n(1-\overline X_n)}}\stackrel{d}{\to}\mathcal N(0,1),
	\] 
	which gives 
	\[ 
	p\in \left[\overline X_n\mp z_{1-\delta/2}\frac{\sqrt{\overline X_n(1-\overline X_n)}}{\sqrt{n}}\right]
	\,{\rm with}\,{\rm prob.}\,\approx\,1-\delta.
	\]
Since $\sqrt{\overline X_n(1-\overline X_n)}\leq \tfrac{1}{2}$, an increase in sample size by a factor of $\gamma>0$ yields only a $\gamma^{-1/2}$-order reduction in the dispersion around $\overline X_n$, the center of the confidence interval. This simple observation has many applications, as Bernoulli trials provide a convenient model for a wide range of binary random experiments such as coin toss, two-candidate election polls, and male--female birth ratios, among others.
\qed
\end{example}

\begin{example}\label{two:samples} (The difference of means of two normal samples)
	Let $\{X_j\}_{j=1}^m$ and $\{Y_k\}_{k=1}^n$ be normally distributed random samples, say with $X_j\sim\mathcal N(\mu_X,\sigma_X^2)$ and $Y_k\sim\mathcal N(\mu_Y,\sigma_Y^2)$, which we assume to be independent to one another. In general, we also assume that the true parameter $\theta=(\mu_X,\mu_Y,\sigma_X^2,\sigma_Y^2)$ is unknown. In order to estimate the difference of means $\mu:=\mu_X-\mu_Y$ we first note that 
	\[
	\mathbb E\left(\overline X_m-\overline Y_n\right)=\mu, \quad 
	{\mathbb V}\left(\overline X_m-\overline Y_n\right)=\frac{\sigma_X^2}{m}+\frac{\sigma_Y^2}{n},
	\]
	and hence 
	\[
	Z_{X,Y}:=\frac{\overline D_{mn}-\mu}{\sqrt{\frac{\sigma_X^2}{m}+\frac{\sigma_Y^2}{n}}}\sim\mathcal N(0,1), 
	\]
	where $\overline D_{mn}=\overline X_m-\overline Y_n$ is an unbiased estimator for $\mu$.
	Thus, if both $\sigma_X$ and $\sigma_Y$ are known we get the ``small sample'' confidence interval 
	\[
	\mu\in \left[\overline D_{mn}\mp z_{1-\delta/2}{\sqrt{\frac{\sigma^2_X}{m}+\frac{\sigma^2_Y}{n}}}\right]
	\,{\rm with}\,{\rm prob.}\,\approx\,1-\delta.
	\]
	In the general case we may proceed as follows.
	If
	\[
	S_{X}^2:=\frac{1}{m-1}\sum_j\left(X_j-\overline X_m\right)^2, \quad 
	S_{Y}^2:=\frac{1}{n-1}\sum_k\left(Y_k-\overline Y_n\right)^2, 
	\]
	are the unbiased estimators for $\sigma_X^2$ and $\sigma_Y^2$ respectively (by Corollary \ref{unb:sigmac}) 
	then Corollary \ref{u:v} implies that 
	\begin{equation}\label{lin:comb:chi}
		\Sigma_X^2:=(m-1)\frac{S_{X}^2}{\sigma_X^2}\sim \chi^2_{m-1} \,\,{\rm and}\,\,
		\Sigma_Y^2:=(n-1)\frac{S_{Y}^2}{\sigma_Y^2}\sim \chi^2_{n-1}
	\end{equation}
	are independent 
	and hence, by Corollary \ref{gamma:sum}, 
	\[
	W_{X,Y}:=	\Sigma_X^2+	\Sigma_Y^2\sim \chi^2_{m+n-2}. 
	\]
	If we set 
	\[
	S_{X,Y}^2(\eta)=\frac{(m-1)S_X^2+(n-1)\eta S_Y^2}{m+n-2}, \quad \eta:=\frac{\sigma_X^2}{\sigma^2_Y}, 
	\]
	then it follows from Proposition \ref{norm:chi:stu} that 
	\begin{eqnarray*}
		T_{X,Y} 
		& := & 
		\frac{Z_{X,Y}}{\sqrt{W_{X,Y}/(m+n-2)}}\\
		& =  &	\frac{\overline D_{mn}-\mu}{\sqrt{c_{mn}(\eta)S_{X,Y}^2(\eta)}}, \quad c_{mn}(\eta)={\frac{1}{m}+\frac{1}{n\eta }},
	\end{eqnarray*}
	satisfies
	\begin{equation}\label{two:samp:stat}
		T_{X,Y}\sim \mathfrak t_{m+n-2}, 
	\end{equation}
	thus being a pivotal quantity with respect to the parameter 
	$\eta$; see Remark \ref{pivotal}.
	Hence, if the population variances, though unknown, are such that their ratio $\eta$ is {known}, 
	then there holds  
	\begin{equation}\label{int:conf:dif:m}
		\mu\in \left[\overline D_{mn}\mp t_{m+n-2,1-\delta/2}
		\sqrt{c_{mn}(\eta)S^2_{XY}(\eta)}
		\right]
		\,{\rm with}\,{\rm prob.}\,\approx\,1-\delta,
	\end{equation}
	a ``small sample'' confidence interval for $\mu$ quite similar in spirit to (\ref{int:conf:tt}), which handles the case of a single normal sample. In particular, if the population variances are assumed to be equal then (\ref{int:conf:dif:m}) holds with $\eta=1$, in which case one has 
	\begin{equation}\label{stud:var}
		T_{X,Y}=\frac{\overline D_{mn}-\mu}{\sqrt{c_{mn}(1)S^{2}_{XY}}},
	\end{equation}
	where
	\[
	S^{2}_{XY}=\frac{(m-1)S_X^2+(m-1)S_Y^2}{m+n-2},
	\]
	is the {\em pooled variance}, a weighted sum of the sample variances. Otherwise, (\ref{int:conf:dif:m}) has no practical usefulness as it provides a dispersion around $\overline D_{mn}$ depending on the unknown parameter $\eta$. Proceeding as before, we are thus led to consider the {\em Behrens-Fisher-Welch statistics}
	\[
	\mathcal Z_{X,Y}=
	\frac{\overline D_{mn}-\mu}{\sqrt{\frac{ S_X^2}{m}+\frac{S_Y^2}{n}}},
	\]
	the obvious counterpart of $Z_{X,Y}$. 
	Unfortunately, no simple, closed expression for the pdf of $\mathcal Z_{X,Y}$ does seem to exist. 
	To appreciate the difficulties involved,
	note that $\mathcal Z_{X,Y}=Z_{X,Y}/\sqrt{\mathcal W_{X,Y}}$, where if 
	\begin{equation}\label{nuis:1}
		\beta_X=
		\frac{1}{m-1}
		\left(1+\frac{m}{n}\eta^{-1}\right)^{-1}, \quad \beta_Y=
		\frac{1}{n-1}
		\left(1+\frac{n}{m}\eta\right)^{-1}
	\end{equation}
	then
	\[
	\mathcal W_{X,Y} 
	=  \beta_X\Sigma_X^2+\beta_Y\Sigma_Y^2	,
	\]
	so that 
	Corollary \ref{a:times:chi} and (\ref{lin:comb:chi}) may be used to ensure that 
	\begin{equation}\label{nuis:2} 
		\mathcal W_{X,Y} 	\sim 
		\mathsf{Gamma}\left(\frac{1}{2\beta_X},\frac{m-1}{2}\right)+
		\mathsf{Gamma}\left(\frac{1}{2\beta_Y},\frac{n-1}{2}\right),
	\end{equation}
	where $(m-1)\beta_X+(n-1)\beta_Y=1$ by (\ref{nuis:1}). 		
	Since $Z_{X,Y}\perp\!\!\perp \mathcal W_{X,Y}$ and $Z_{X,Y}\sim\mathcal N(0,1)$, it follows from Remark \ref{comp:dist} that computing $\psi_{\mathcal Z_{X,Y}}$ essentially reduces  to figuring out $\psi_{\mathcal W_{X,Y}}$, the pdf of a sum of independent $\mathsf{Gamma}$-distributed random variables
	whose inverse scale parameters are distinct except when 
	\[
	\eta=\frac{m(m-1)}{n(n-1)}.
	\]
Thus, in most cases this sum is not $\mathsf{Gamma}$-distributed; cf.\ Corollary~\ref{gamma:sum}. In fact, it is known that the exact expression for $\psi_{\mathcal W_{X,Y}}$ and, more generally, for the pdf of a linear combination of chi-square distributions, requires an infinite series expansion involving certain transcendental functions~\cite{ray1961exact,moschopoulos1985distribution,hong2022exact}. This, in turn, makes it necessary to rely on suitable approximations for $\psi_{\mathcal Z_{X,Y}}$, a situation that has prompted extensive research into the computational accuracy and efficiency of such methods\footnote{See~\cite{bausch2013efficient} for a critical review of recent developments on this topic.}. 
To complicate matters further, as is apparent from (\ref{nuis:1}) and (\ref{nuis:2}), $\psi_{\mathcal Z_{X,Y}}$ explicitly depends on the ``nuisance'' parameter $\eta$, implying that $\mathcal Z_{X,Y}$ cannot serve as a pivotal quantity; see Remark~\ref{pivotal}. 
Thus, despite significant progress in its practical treatment~\cite{kim1998behrens}, the broader challenge of obtaining the most efficient estimate of $\mu$ when $\eta$ is unknown, commonly known as the \emph{Behrens--Fisher problem}~\cite[Section~3.8]{welsh1996aspects}, remains largely unresolved.
 \qed
\end{example}

\begin{remark}\label{F:test:q:var}(${\bm{\textsf F}}$-test for the equality of variances) 
	The reliability of the assumption on the equality of the population variances which led to (\ref{stud:var})  may be statistically justified (or not!) by running an ${\bm{\textsf F}}$-test; for more on this see Section \ref{sec:hyp:test} below. We start by noticing that under the corresponding 
	{\em null hypothesis}
	\[
	H_0: \quad \sigma_X^2=\sigma_Y^2,
	\]
	(\ref{lin:comb:chi}) and Proposition \ref{chi:to:F} ensure that the appropriate test statistics
	\begin{equation}\label{F:stats:var}
		U:=\frac{S_X^2}{S_Y^2}=
		\frac{\Sigma^2_X/(m-1)}{\Sigma^2_Y/(n-1)}\sim {\bm{\textsf F}}_{m-1,n-1}.
	\end{equation}
	Now, a very rough analysis of the departure from $H_0$ goes as follows. If this hypothesis is not satisfied  (so that $\sigma_X^2=\eta\sigma_Y^2$, $\eta\neq 1$) then (\ref{F:stats:var}) becomes 
	\[
	U
	=\frac{\frac{\eta}{m-1}\Sigma_X^2}{\frac{1}{n-1}\Sigma_Y^2},
	\]
	and we see from Corollary \ref{a:times:chi} and Proposition \ref{gamma:mgf} that:
	\begin{itemize}
		\item The denominator  satisfies
		\[
		\frac{1}{n-1}\Sigma_Y^2\sim\mathsf{ Gamma}\left(\frac{n-1}{2},\frac{n-1}{2}\right),
		\]
		so its distribution remains the same regardless of the validity of $H_0$;
		\item On the other hand, the numerator satisfies 
		\[
		\frac{\eta}{m-1}\Sigma_X^2\sim\mathsf{ Gamma}\left(\frac{m-1}{2\eta},\frac{m-1}{2}\right),
		\]
		a Gamma distribution whose scale factor is proportional to 
		\[
		\eta=	\mathbb E\left(\frac{\eta}{m-1}\Sigma_X^2\right), 
		\]
		the parameter quantifying the departure from $H_0$. 
	\end{itemize}	
	Thus, at least on average, very small or very large realizations for $U$ (substantially departing from $u=1$) provide statistical evidence  
	for {\em rejecting} $H_0$.
	Precisely,  if we fix $0<\beta<1$ and consider the corresponding $\mathsf F$-{\em quantile}
	${\bm{\textsf f}}_{k_1,k_2,\beta}$ determined by  
	\begin{equation}\label{f:quantiles:0}
		F_{\bm{\textsf F}_{k_1,k_2}}
		({\bm{\textsf f}}_{k_1,k_2,\beta})=\beta,
	\end{equation}
	where $F_{\bm{\textsf F}_{k_1,k_2}}$ is the cdf of ${\bm{\textsf F}}_{k_1,k_2}$, then $H_0$ should be rejected ``at significance level $\alpha$'' if the realization $u$ of the statistics in (\ref{F:stats:var}) satisfies
	\begin{equation}\label{two:sided}
		u\in \left(0,{\bm{\textsf f}}_{m-1,n-1,\alpha/2}\right]\bigcup \left[{\bm{\textsf f}}_{m-1,n-1,1-\alpha/2},+\infty\right).
	\end{equation}
	A more convincing justification for this rather informal argument 
	may be found in Section \ref{sec:hyp:test} below. \qed
\end{remark}

\begin{remark}\label{rec:f:quant}(Reciprocity of the ${\bm{\textsf F}}$-quantiles)
	Regarding the ${\bm{\textsf F}}$-quantiles defined in (\ref{f:quantiles:0}), let us take $X\sim {\bm{\textsf F}}_{m,n}$ so that $X^{-1}\sim
	{\bm{\textsf F}}_{n,m}$ by Corollary \ref{f-dis:cons}. For any $\alpha>0$ we then have
	\[
	\frac{\alpha}{2} 
	= 
	P\left(X\leq {\bm{\textsf f}}_{m,n,\alpha/2}\right)\\
	= 	P\left(X^{-1}\geq \frac{1}{\bm{\textsf f}_{m,n,\alpha/2}}\right),
	\]
	so that
	\[
	P\left(X^{-1}\leq \frac{1}{\bm{\textsf f}_{m,n,\alpha/2}}\right)=1-\frac{\alpha}{2}=
	P\left(X^{-1}\leq {\bm{\textsf f}}_{n,m,1-\alpha/2}\right),
	\]			
	from which the identity
	\[
	{\bm{\textsf f}}_{m,n,\alpha/2}{\bm{\textsf f}}_{n,m,1-\alpha/2}=1
	\]
	follows.
	\qed
\end{remark}

\begin{example}\label{corr:dist}(The sampling distribution of the  correlation coefficient) Let us retain the notation of Example \ref{two:samples}, but this times with $m=n$ and 
	assuming that the {\em independent} random sample 
	\[
	(X,Y):=\{(X_1,Y_1),\cdots,(X_m,Y_m)\}
	\]
	has been drawn from a bi-variate normal population whose marginals are not necessarily independent. 
	Thus, 
	\begin{equation}\label{bi:var:norm}
		(X_j,Y_j)\sim\mathcal N\left(
		\left(
		\begin{array}{c}
			\mu_X\\
			\mu_Y
		\end{array}
		\right),
		\left(
		\begin{array}{cc}
			\sigma_X^ 2 & \sigma_{XY}\\
			\sigma_{XY} & \sigma^2_Y
		\end{array}
		\right)
		\right),\quad j=1,\cdots,n,
	\end{equation}
	where $\sigma_{XY}={\mathbb C}(X_j,Y_j)$ is the {\em population covariance} \footnote{By Proposition \ref{unc:ind:n}, $\{X_j,Y_j\}$ is independent if and only if $\sigma_{XY}=0$.}, so 
	by (\ref{rem:dir:ind:n:1}) the joint distribution of $(X,Y)$ is
	\begin{eqnarray}\label{bi:var:norm:like}
		\psi_{(X,Y)}(x,y)dxdy
		& = & 
		\frac{1}{(2\pi\sigma_X\sigma_Y\sqrt{1-\rho^2})^m} \times \nonumber\\
		& & \quad \times\,
		e^{-\frac{1}{2(1-\rho^2)}\sum_j
			\left(\frac{(x_j-\mu_X)^2}{\sigma_X^2}-
			\frac{2\rho(x_j-\mu_X)(y_j-\mu_Y)}{\sigma_X\sigma_Y}
			+\frac{(y_j-\mu_Y)^2}{\sigma_Y^2}\right)}dxdy, 
	\end{eqnarray}
	where 
	\[
	\rho=\frac{\sigma_{XY}}{\sigma_X\sigma_Y}
	\]
	is the {\em correlation coefficient}, 
	a population parameter whose estimation is a central theme in Multivariate Statistical Analysis \cite{anderson2003introduction}\footnote{As usual, we assume that $|\rho|<1$, thus avoiding the degenerate cases $\rho=\pm 1$.}. 
	It turns out that Fisher's geometric approach in Remark \ref{fis:comp:new} can be successfully employed to this end \cite{fisher1915frequency}. Indeed, using Remark \ref{U:sigma:V:r} we have the identities
	\[
	\sum_j\left(x_j-\mu_X\right)^2=m\left(
	\left(\overline x_m-\mu_X\right)^2+\widehat\sigma^2_{m^{-1}}(x)
	\right)
	\]
	and 
	\[
	\sum_j\left(y_j-\mu_y\right)^2=m\left(
	\left(\overline y_m-\mu_X\right)^2+\widehat\sigma^2_{m^{-1}}(y)
	\right), 
	\]
	where $\widehat\sigma^2_{m^{-1}}(x)$ and   $\widehat\sigma^2_{m^{-1}}(y)$ are the realizations of the variance estimators appearing in (\ref{fam:est}) with $c=m^{-1}$. Also, we will need their
	polarized version 
	\[
	\sum_j\left(x_j-\mu_X\right)\left(y_j-\mu_Y\right)=m\left(
	\left(\overline x_m-\mu_X\right)\left(\overline y_m-\mu_Y\right)+\widehat\sigma^2_{m^{-1}}(x,y)
	\right),
	\] 			
	where 
	\begin{equation}\label{emp:inner}
	\widehat\sigma^2_{m^{-1}}(X,Y):=\frac{1}{m}\sum_j\left(X_j-\overline X_m\right)\left(Y_j-\overline Y_m\right).
	\end{equation}
	Leading these identities to (\ref{bi:var:norm:like})  we get
	\begin{eqnarray*}
		\psi_{(X,Y)}dxdy
		& = & 
		\frac{1}{(2\pi\sigma_X\sigma_Y\sqrt{1-\rho^2})^m}\times \\
		& & \quad \times\,e^{-\frac{m}{2(1-\rho^2)} 
			\left(
			\frac{(\overline x_m-\mu_X)^2}{\sigma_X^2}
			-\frac{2\rho(\overline x_m-\mu_X)(\overline y_m-\mu_Y)}{\sigma_X\sigma_Y}
			+\frac{(\overline y_m-\mu_Y)^2}{\sigma_Y^2}
			\right)}\times\\
		& & \quad e^{-\frac{m}{2(1-\rho^2)}\left(\frac{\widehat\sigma^2_{m^{-1}}(x)}{\sigma_X^2}
			-\frac{2\rho\widehat\rho \widehat\sigma_{m^{-1}}(x)\widehat\sigma_{m^{-1}}(y)}{\sigma_X\sigma_Y}
			+\frac{\widehat\sigma^2_{m^{-1}}(y)}{\sigma_Y^2}
			\right)}dxdy,		
	\end{eqnarray*}
	where 
	\[
	\widehat\rho=\frac{\widehat\sigma^2_{m^{-1}}(X,Y)}{\widehat\sigma_{m^{-1}}(X)\widehat\sigma_{m^{-1}}(Y)}
	\] 
	is the {\em sample correlation coefficient}, the natural estimator for $\rho$; see Example \ref{corr:mle} for a justification of this latter claim. 
	If we could view realizations of the samples $X$ and $Y$ as  {\em independent} elements of $\mathbb R^m_X$ and $\mathbb R^m_Y$, respectively, then the geometric reasoning in Remark  \ref{fis:comp:new} would ensure that 	
	\begin{equation}\label{corr:corr}
		dxdy\approx \widehat\sigma^{m-2}_{m^{-1}}(x)d\widehat\sigma_{m^{-1}}(x)d\overline x_md\theta_X
		\times
		\widehat\sigma^{m-2}_{m^{-1}}(y)d\widehat\sigma_{m^{-1}}(y)d\overline y_md\theta_Y,
	\end{equation}
	from which we would compute $\psi_{(\widehat\sigma_{m^{-1}}(X),\widehat\sigma_{m^{-1}}(Y),
		\widehat\rho)}$ after integrating $\psi_{(X,Y)}dxdy$ above  against $d\overline x_md\theta_Xd\overline y_md\theta_Y$. 
	However, and this is the key point here, $x$ and $y$ are {\em not} allowed to vary freely 
	as 
	they are
	constrained to move in such a way that $x\in \mathbb S^{m-2}_{\sqrt{m}\widehat\sigma_{m^{-1}}(x)}({{\mathscr X_m}})$ and $y\in \mathbb S^{m-2}_{\sqrt{m}\widehat\sigma_{m^{-1}}(y)}({{\mathscr Y_m}})$ with 
	\[
	\widehat\rho =\cos\theta, \quad \theta=\measuredangle(\overline{{\mathscr X_m}x},\overline{{\mathscr Y_m}y}).
	\] 				
	Thus, 	
	\begin{eqnarray*}
		\psi_{(\widehat\sigma_{m^{-1}}(X),\widehat\sigma_{m^{-1}}(Y),
			\widehat\rho)}dv
		& \approx & 
		e^{-\frac{m}{2(1-\rho^2)}\left(\frac{\widehat\sigma^2_{m^{-1}}(s)}{\sigma_X^2}
			-\frac{2\rho\widehat\rho \widehat\sigma_{m^{-1}}(x)\widehat\sigma_{m^{-1}}(y)}{\sigma_X\sigma_Y}
			+\frac{\widehat\sigma^2_{m^{-1}}(y)}{\sigma_Y^2}
			\right)} \times\\
		& & \quad \times\,
		\widehat\sigma^{m-2}_{m^{-1}}(x)
		\widehat\sigma^{m-2}_{m^{-1}}(y)f(\widehat\rho)dv,
	\end{eqnarray*} 
	where $dv=d\widehat\sigma_{m^{-1}}(x)d\widehat\sigma_{m^{-1}}(y)d\widehat\rho$ and the extra factor $ f(\widehat\rho)$ comes from the 
	constraint referred to above. 
	Incidentally, this already shows that $\{\overline X_m,\overline Y_m\}$ is independent from $\{\widehat\sigma^2_{m^{-1}}(X),\widehat\sigma^2_{m^{-1}}(Y),\widehat\sigma_{m^{-1}}(XY)\}$, as in the uni-variate case; cf. Proposition \ref{indep:ms}.
	Now, in order to determine $f(\widehat\rho)$ note that 
	if $x$ is fixed then, corresponding to an infinitesimal displacement $d\theta$, the segment $\overline{{\mathscr Y_m}y}$ describes an infinitesimal spherical slab in $\mathbb S^{m-2}_{\sqrt{m}\widehat\sigma_{m^{-1}}(y)}({{\mathscr Y_m}})$ with radius 
	\[
	\sqrt{m}\widehat\sigma_{m^{-1}}(y)\sin\theta=\sqrt{m}\widehat\sigma_{m^{-1}}(y)\sqrt{1-\widehat\rho^2}
	\]
	and height 
	\[
	\sqrt{m}\widehat\sigma_{m^{-1}}(y)|d\theta|=\sqrt{m}\widehat\sigma_{m^{-1}}(y)\frac{d\widehat\rho}{\sqrt{1-\widehat\rho^2}},
	\]
	thus tracing a volume proportionate to 
	\[
	(\sqrt{m}\widehat\sigma_{m^{-1}}(y)\sin\theta)^{m-3}	\sqrt{m}\widehat\sigma_{m^{-1}}(y)|d\theta|
	=
	\widehat\sigma^{m-2}_{m^{-1}}(Y)\underbrace{(1-\widehat\rho^2)^{\frac{m-4}{2}}}_{=f(\widehat\rho)}d\widehat\rho,
	\]
	which finally gives
	\begin{eqnarray*}
		\psi_{(\widehat\sigma_{m^{-1}}(X),\widehat\sigma_{m^{-1}}(Y),\widehat\rho)}dv
		& \approx & 
		e^{-\frac{m}{2(1-\rho^2)}\left(\frac{\widehat\sigma^2_{m^{-1}}(x)}{\sigma_X^2}
			-\frac{2\rho\widehat\rho \widehat\sigma_{m^{-1}}(x)\widehat\sigma_{m^{-1}}(y)}{\sigma_X\sigma_Y}
			+\frac{\widehat\sigma^2_{m^{-1}}(y)}{\sigma_Y^2}
			\right)} \times\\
		& & \quad \times\,
		\widehat\sigma^{m-2}_{m^{-1}}(x)
		\widehat\sigma^{m-2}_{m^{-1}}(y)d\widehat\sigma_{m^{-1}}(x)d\widehat\sigma_{m^{-1}}(y) \times \\
		& & \quad\quad \times\, (1-\widehat\rho^2)^{\frac{m-4}{2}} d\widehat\rho.
	\end{eqnarray*}  
	As usual, explicit, albeit quite complicated, expressions for the desired pdf $\psi_{\widehat\rho}$, which may even be chosen so as to only involve elementary functions, are obtained by integrating this against the area element $d\widehat\sigma_{m^{-1}}(x)d\widehat\sigma_{m^{-1}}(y)$, with a further integration against $d\widehat\rho$ being needed to restore the normalizing constant. 
	Of course, the computational difficulty here comes from the mixed term in the exponential which prevents $\{\widehat\sigma^2_{m^{-1}}(X),\widehat\sigma^2_{m^{-1}}(Y),\widehat\rho\}$ from being independent (except when $\rho=0$). 
	In any case, the resulting expressions are found to depend  on the parameters of 
	the underlying normal bi-variate population only through $\rho$ (and not on any other combination of the entries of the variance matrix in (\ref{bi:var:norm})), and in fact they all reduce to 
	\[
	\psi_{\widehat\rho}(r)\approx (1-r^2)^{\frac{m-4}{2}}{\bf 1}_{(-1,1)}(r), \quad r\in\mathbb R,
	\]
	when $\rho=0$,
	which suffices to efficiently testing the mutual independence of $\{X_j,Y_j\}$ for any value of $m$ along the lines of the general theory developed in Section \ref{sec:hyp:test}. Otherwise, one has to appeal to asymptotic methods in order to construct ``large sample'' confidence intervals for $\rho$; see Example \ref{binormal:mle} below.   
	We refer to the original sources \cite{student1908probablecorr,fisher1915frequency}, as well as to \cite[Chapter 14]{kendall1946advanced} and \cite[Chapter 4]{anderson2003introduction}, for discussions of the basic properties of $\psi_{\widehat\rho}$ and their applications.  
	\qed
\end{example}

\begin{example}(One way ANOVA)\label{ow:anova}
	Fix $p\in\mathbb N$, $p\geq 3$, a finite sequence $\{n_j\}_{j=1}^{p}\subset\mathbb N$ and   for each $j$ consider a random sample $\{X_{jk}\}_{k=1}^{n_j}$ with $X_{jk}\sim\mathcal N(\mu_j,\sigma^2)$  such that all these $n:=\sum_jn_j$ samples $X_{ij}$ form an  independent set. In other words, we are dealing here with $p$ independent normal random samples with varied sizes and expectations but sharing the {\em same} variance, with the parameters $\{\mu_1,\cdots,\mu_p,\sigma^2\}$ being regarded as unknown. 
	Within each sample we have the decomposition coming from Remark \ref{U:sigma:V:r}, 
	\begin{equation}\label{ow:anova:2}
		\sum_{k=1}^{n_j}\left(X_{jk}-\overline X_{\bullet\bullet}\right)^2=
		\sum_{k=1}^{n_j}\left(X_{jk}-\overline X_{j\bullet}\right)^2+n_j\left(\overline X_{j\bullet}-\overline X_{\bullet\bullet}\right)^2, 
	\end{equation}
	where 
	\[
	\overline X_{j\bullet}=\frac{1}{n_j}\sum_{k=1}^{n_j}X_{jk}
	\]
	and 
	\[
	\overline X_{\bullet\bullet}=\frac{1}{n}\sum_{j=1}^p\sum_{k=1}^{n_j}X_{ij}=
	\frac{1}{n}\sum_{j=1}^pn_j	\overline X_{j\bullet}.
	\]
	Note that 
	\begin{equation}\label{var:parc}
		\mathbb E(\overline X_{j\bullet}^2)=\mu_j^2+\frac{\sigma^2}{n_j},	
	\end{equation}
	and 
	\begin{equation}\label{var:total}
		\mathbb E(\overline X_{\bullet\bullet}^2)=\frac{1}{n^2}\left(\sum_{j=1}^pn_j\mu_j\right)^2+
		\frac{\sigma^2}{n}. 	
	\end{equation}	
	Now, the same argument leading to the proof of Proposition \ref{u:v} implies 
	that 
	\begin{equation}\label{dist:within}
		\sigma^{-2}\sum_{k=1}^{n_j}\left(X_{jk}-\overline X_{j\bullet}\right)^2\sim \chi^2_{n_j-1}
	\end{equation}
	is independent of $\overline X_{j\bullet}$ and hence of $\left(\overline X_{j\bullet}-\overline X_{\bullet\bullet}\right)^2$. Thus, 
	if we set
	\[
	S^2_{\rm Total}=\sum_{j=1}^p\sum_{k=1}^{n_j}\left(X_{jk}-\overline X_{\bullet\bullet}\right)^2,
	\]
	the {\em total} sum of squares,  
	\[
	S^2_{\rm Within}=\sum_{j=1}^p\sum_{k=1}^{n_j}\left(X_{jk}-\overline X_{j\bullet}\right)^2,
	\]
	the sum of squares {\em within} the samples, 
	and 
	\[
	S^2_{\rm Between}=\sum_{j=1}^pn_j\left(\overline X_{j\bullet}-\overline X_{\bullet\bullet}\right)^2, 
	\]
	the sum of squares {\em between} the samples, 
	then 
	\begin{equation}\label{dec:wit:bet}
		S^2_{\rm Total}=	S^2_{\rm Within}+	S^2_{\rm Between},
	\end{equation}
	with 
	\[
	\sigma^{-2}S^2_{\rm Within}\sim\chi^2_{n-p}
	\]
	being independent of $S^2_{\rm Between}$. On the other hand, again by the argument leading to (\ref{dist:within}), but this time under the {\em null hypothesis}
	\begin{equation}\label{null:hyp}
		H_0:\quad \mu_1=\cdots=\mu_p,
	\end{equation}
	we have 
	\[
	\sigma^{-2}S^2_{\rm Total}\sim\chi^2_{n-1}
	\]
	and hence
	\[
	\sigma^{-2}S^2_{\rm Between}\sim\chi^2_{p-1},
	\]
	which gives
	\begin{equation}\label{exp:between}
		\mathbb E(S^2_{\rm Between})= (p-1)\sigma^2
	\end{equation}
	by 	Corollary \ref{chi:sq:ms}.
	It then follows from Proposition \ref{chi:to:F} that 
	\begin{equation}\label{dist:an}
		V:=\frac{S^2_{\rm Between}/(p-1)}{S^2_{\rm Within}/(n-p)}\sim 
		{\bm{\textsf F}}_{p-1,n-p}\,\, {\rm under} \,\, H_0.
	\end{equation}
	In order to proceed we now observe that:
	\begin{itemize}
		\item 
		The decomposition (\ref{dec:wit:bet}), which is an easy consequence of the fundamental algebraic identity in Remark \ref{U:sigma:V:r},  plays a central role in our analysis as it displays $S^2_{\rm Total}$, the total sum of squares, as resulting from the contribution of two terms of rather distinct types:
		$S^2_{\rm Within}$ collects together the variations {\em within} the various samples 
		whereas
		$S^2_{\rm Between}$ measures 
		the variation {\em between} the samples; 
		\item 
		In consonance with the previous item, the computation leading to 
		the statistics in (\ref{dist:an}) shows that 
		the distribution of its numerator is conditioned to the validity of $H_0$
		whereas the distribution of its denominator remains the same regardless of the validity of this hypothesis; 
		\item
		If $H_0$ is not necessarily satisfied then starting from the fact that
		\[
		S^2_{\rm Between}=\sum_{j=1}^pn_j\overline X_{j\bullet}^2-n\overline X_{\bullet\bullet}^2,
		\]
		we easily deduce by means of 
		(\ref{var:parc}) and (\ref{var:total}) that 
		\[
		\mathbb E(S^2_{\rm Between})=(p-1)\sigma^2+\sum_{j=1}^pn_j(\mu_j-\overline\mu)^2, \quad \overline\mu=\frac{1}{n}\sum_{j=1}^pn_j\mu_j,
		\]
		which assumes its minimal value, given by (\ref{exp:between}), exactly when $H_0$ holds true. 
	\end{itemize}
	Thus, at least on average, a sufficiently large value of $V$ provides statistical evidence for {\em rejecting} $H_0$.
	Precisely, if we fix $0<\alpha<1$
	then $H_0$ should be rejected ``at significance level $\alpha$'' if the realization $v$ of $V$ in (\ref{dist:an}) satisfies 
	\begin{equation}\label{one:sided}
		v\in \left[{\bm{\textsf f}}_{p-1,n-p,1-\alpha},+\infty\right).	 
	\end{equation}
	Here, we use the notation for $\mathsf F$-quantiles introduced in (\ref{f:quantiles:0}).
	Again, we refer to Section \ref{sec:hyp:test} for a more theoretically inclined  		
	justification of this procedure, in particular for the proper understanding of why the ``rejection subsets'' in the right-hand side of the ${\bm{\textsf F}}$-tests in (\ref{two:sided}) and (\ref{one:sided}) differ in their ``connectedness''. \qed
\end{example}

\begin{example}\label{estim:var:nomr} (Estimating the variance of a normal population) If $X_j\sim\mathcal N(\mu,\sigma^2)$ we now estimate $\sigma^2$, the population variance. We first assume that $\mu$ is known and consider the estimator 
	\[
	\mathscr S^2_n=\sum_{j=1}^n(X_j-\mu)^2,
	\]
so that 
\[
\frac{\mathscr S^2_n}{\sigma^2}=\sum_{j=1}^n\left(\frac{X_j-\mu}{\sigma}\right)^2\sim \chi^2_n
\]	
by Corollary \ref{sum:norm:sq}. If, as usual, we define the  $\chi^2_k$-{\em quantile} associated to $0<\beta<1$ by 
\begin{equation}\label{quantile:chi}
F_{\chi^2_k}(\chi^2_{k,\beta})=\beta, 
\end{equation}
then  
\begin{eqnarray*}
P\left(\frac{\mathscr S^2_n}{\chi^2_{n,1-\delta/2}} \leq \sigma^2\leq \frac{\mathscr S^2_n}{\chi^2_{n,\delta/2}}\right) 
& = & 
P\left( \chi^2_{n,\delta/2}\leq \frac{\mathscr S^2_n}{\sigma^2}\leq \chi^2_{n,1-\delta/2}\right)
\\
& = & 1-\delta,	
	\end{eqnarray*}
so we obtain the confidence interval 
\[
\sigma^2\in\left[
\frac{\mathscr S^2_n}{\chi^2_{n,1-\delta/2}},\frac{\mathscr S^2_n}{\chi^2_{n,\delta/2}}
\right]\,\,\,\, \textrm{with prob.}\,\,\,\, 1-\delta.
\]	
It is instructive to determine how the center $c_n$	and the length $l_n$ of this (clearly asymmetric) interval behaves as $n\to+\infty$. For this we note that:
\begin{itemize}
	\item From LLN (Theorem \ref{lln}), $\mathscr S^2_n/n\stackrel{p}{\to}\sigma^2$ so that $\mathscr S^2_n\approx n\sigma^2$ as $n\to+\infty$; see Remark \ref{normal:ww} for the details.
	\item Since $Y_j\sim \chi^2_1$ implies $Y^{(n)}=Y_1+\cdots+Y_n\sim \chi^2_n$ then CLT (Theorem \ref{clt}) and Corollary \ref{chi:sq:ms} give
	\[
	\frac{Y^{(n)}-n}{\sqrt{2n}}\stackrel{d}{\to}\mathcal N(0,1),
	\]
	so that
	\[
	\chi^2_{n,\beta}\approx n+\sqrt{2n}z_{\beta}+O(1),\quad 0<\beta <1,
	\]
	which yields
	\[
	\chi^2_{n,1-\delta/2}\approx n+\sqrt{2n}z+O(1)\,\,\,\,\textrm{and}\,\,\,\,
	\chi^2_{n,\delta/2}\approx n-\sqrt{2n}z+O(1), \quad z=z_{1-\delta/2}. 
	\]
\end{itemize} 
From these facts we easily deduce that the endpoints of the interval satisfy
\[
\frac{\mathscr S^2_n}{\chi^2_{n,1-\delta/2}}=\sigma^2\left(1-\sqrt{2}zn^{-1/2}+O(n^{-1})\right)
\]
and 
\[
\frac{\mathscr S^2_n}{\chi^2_{n,\delta/2}}=\sigma^2\left(1+\sqrt{2}zn^{-1/2}+O(n^{-1})
\right),
\]
so that
\[
c_n\approx \sigma^2 + O(n^{-1})\,\,\,\,
\textrm{and}
\,\,\,\,
l_n\approx \sigma^2\left(2\sqrt{2}zn^{-1/2}+O(n^{-1})\right). 
\]
As expected, the interval is asymptotically centered at $\sigma^2$ and has length decaying at the rate $O(n^{-1/2})$. If $\mu$ is unknown it is natural to use instead the estimator 
\[
\widetilde{\mathscr S}^2_n=\sum_{j=1}^n(X_j-\overline X_n)^2=(n-1)S_n^2,
\]
which corresponds to replacing $\mu$ by $\overline X_n$ in the definition of $\mathscr S^2_n$. Since $\widetilde{\mathscr S}^2_n/\sigma^2\sim\chi^2_{n-1}$ by Proposition \ref{u:v},
we end up with the confidence interval 
\[
\sigma^2\in\left[
\frac{\widetilde{\mathscr S}^2_n}{\chi^2_{n-1,1-\delta/2}},
\frac{\widetilde{\mathscr S}^2_n}{\chi^2_{n-1,\delta/2}}
\right]\,\,\,\, \textrm{with prob.}\,\,\,\, 1-\delta,
\]
whose center and length satisfy 
\[
\widetilde c_n\approx \sigma^2 +O((n-1)^{-1})\,\,\,\,\textrm{and}\,\,\,\,
\widetilde l_n\approx \sigma^2\left(2\sqrt{2}z(n-1)^{-1/2}+O((n-1)^{-1})\right)=O((n-1)^{-1/2}).
\]
Although $\widetilde l_n>l_n$ for each $n$, there holds $\widetilde l_n/l_n\to 1$ as $n\to+\infty$, so the intervals are essentially identical for large samples. We note, however, that both intervals are quite conservative for $n$ small.\qed
\end{example}

\section{Maximum likelihood}\label{sec:MLE}
We now present a remarkable class of estimators, introduced by R. Fisher, which displays, under suitable regularity assumptions, many desirable asymptotic properties, including asymptotic normality (Theorem \ref{asym:cr}). 

\subsection{Maximum likehood estimators}\label{subsec:MLE}
We start with an {\em independent} family $\{X_j\}_{j=1}^n$ of random variables with $X_j\sim \psi_j(x_j;\theta)>0$, $\theta\in\Theta$. 

\begin{definition}\label{likeli:def}
	The 
	{\em likelihood function} of the random vector $X=(X_1,\cdots,X_n):\Omega\to\mathbb R^n$ 
	is 
	\begin{equation}\label{like:exp}
		L({\bf x};\theta)=\Pi_{j=1}^n\psi_j(x_j;\theta),
	\end{equation}
	where  ${\bf x}=(x_1,\cdots,x_n)\in\mathbb R^n$ is viewed as a realization of $X$.
	In particular, if $\{X_j\}$ is i.i.d.  ($X_j\sim \psi(x_j;\theta$) then 
	\begin{equation}\label{like:exp:ind}
		L({\bf x};\theta)=\Pi_{j=1}^n\psi_\theta(x_j), \quad \psi_\theta(x_j)=\psi(x_j;\theta).
	\end{equation}
\end{definition}

We begin with a motivation showing that, in the i.i.d.\ setting (which covers most cases considered here), maximum likelihood estimators arise as approximate solutions to a natural variational problem.

\begin{definition}\label{kull:leib:div:d}
	If $\theta_0\in\Theta$ is the unknown parameter to estimate, we define the {\em Kullback-Leibler divergence} (centered at $\theta_0$) by 
	\begin{equation}\label{kl:dist}
		\theta\in\Theta\mapsto D^{KL}_{\theta_0}(\theta):=\int_{\mathbb R}\psi_{\theta_0}({x})
		\ln\left(\frac{\psi_{\theta_0}({x})}{\psi_{\theta}({x})}\right)d{x}.
	\end{equation}
\end{definition}	

Using Jensen's inequality and assuming, as throughout, that our statistic model $X_j\sim\psi_\theta$ is identifiable\footnote{Recall that identifiability means that the map $\theta\in\Theta\mapsto \psi_\theta$ is injective.}, we readily see that $D^{KL}_{\theta_0}(\theta)\geq 0$ for any $\theta$, with the equality holding only if $\theta=\theta_0$.
This observation naturally leads us to seek
an estimator $\theta^*$ satisfying
\[
\theta^*={\rm argmin}_{\theta\in\Theta}D^{KL}_{\theta_0}(\theta). 
\]  
Since
\begin{equation}\label{kl:arg:1}
D^{KL}_{\theta_0}(\theta)={\rm const}_{\theta_0}-\mathbb E_{\theta_0}(\ln(\psi_\theta)),
\end{equation}
this minimization problem is equivalent to
\[
\theta^* = {\rm argmax}_{\theta\in\Theta}\mathbb E_{\theta_0}(\ln(\psi_\theta)).
\]
By LLN we may approximate, for $n$ sufficiently large,
\begin{equation}\label{kl:arg:2}
	\mathbb E_{\theta_0}(\ln (\psi_\theta))\approx \frac{1}{n}\sum_{j=1}^n\ln\psi_\theta(X_j)=\mathbb E_{\widetilde\theta_X}\left(\psi_\theta(X_j)\right),
	\end{equation}
where 
\[
\widetilde\theta_X=\frac{1}{n}\sum_{j=1}^n\delta_{X_j}
\]	
is the {\em empirical distribution} associated with a sample $X$ drawn from the unknown distribution $\psi_{\theta_0}$; cf. Example \ref{gliv}.
Combining  (\ref{kl:arg:1}) and (\ref{kl:arg:2}), we obtain
\begin{equation}\label{kl:arg:3}
\frac{1}{n}l(X;\theta) \approx	 {\rm const}_{\theta_0}-D^{KL}_{\widetilde\theta_X}(\theta),
\end{equation}
	where 
\begin{equation}\label{like:exp:ind0}
	l({\bf x};\theta)=\ln L({\bf x};\theta)
\end{equation}
is the {\em log-likelihood function},
so that the choice
\begin{equation}\label{MLE:heur}
	\theta^*(X) \approx  \frac{1}{n}{\rm argmax}_{\theta\in\Theta}l(X;\theta) =
	{\rm argmax}_{\theta\in\Theta}L(X;\theta)
\end{equation}
emerges as a natural candidate for estimating the unknown parameter
 $\theta_0$ based on the observed value ${\bf x}$.
In this sense, the resulting estimator approximately minimizes the Kullback--Leibler divergence from the model to the empirical distribution. Taking into account that the empirical distribution $\widetilde\theta_X$ is consistent for $\theta_0$, we see that, at least in the i.i.d.\ case,  the ML estimator asymptotically minimizes the Kullback-Leibler ``distance'' to $\theta_0$ in (\ref{kl:dist}). 

This heuristic argument motivates the following fundamental construction due to R.~A.~Fisher, which remains the most widely used method for deriving estimators.

\begin{definition}
	\label{mle:def:post}
	(Maximum Likelihood Estimation, MLE)
	Under the conditions above, a maximum likelihood (ML) estimator $\widehat\theta=\widehat\theta(X)$ is defined by 
	\begin{equation}\label{est:mle:det}
		\widehat\theta({\bf x})={\rm argmax}_{\theta\in\Theta} L({\bf x};\theta)={\rm argmax}_{\theta\in\Theta} l({\bf x};\theta).,
	\end{equation}
	where ${\bf x}$ is an observed value of $X$.
\end{definition}

As usual, we assume that the log-likelihood $l$ is strictly concave in $\theta$ so the solution  to (\ref{est:mle:det}) is unique (whenever it exists). 

\begin{remark}\label{kl:int:lrt:0}
Regarding the heuristic discussion above, we note that the interpretation of maximum likelihood estimation as the minimization of Kullback–Leibler divergence goes back to S.~Kullback and R.~Leibler~\cite{kullback1951information} and was later made explicit by H.~Akaike~\cite{akaike1973information}.
Systematic treatments of this circle of ideas from a decision-theoretic perspective, leading to the consideration of a host of information criteria in model selection, may be found in~\cite{sakamoto1986akaike,konishi2008information,burnham2013model}; see also Remark~\ref{rem:aic}, where we discuss the Akaike Information Criterion (AIC).
\qed	
\end{remark}

\begin{example}\label{normal:w} (MLE from a normal population)
	If $X_j\sim\mathcal N(\mu,\sigma^2)$ then 
	\begin{equation}\label{normal:w:1}
		L({\bf x};\theta)=(2\pi\theta_2)^{-n/2}e^{-\frac{1}{2\theta_2}\sum_{j=1}^n(x_j-\theta_1)^2}, 
	\end{equation}
	where $\theta=(\theta_1,\theta_2)=(\mu,\sigma^2)\in \Theta=\mathbb R\times\mathbb R_+$,
	so that 
	\begin{equation}\label{log:like}
		l({\bf x};\theta)=\ln L({\bf x};\theta)=-\frac{n}{2}\ln (2\pi\theta_2)-\frac{1}{2\theta_2}\sum_j(x_j-\theta_1)^2.
	\end{equation}
	The usual first derivative test shows that the 
	ML estimator $\widehat\theta_1$  and $\widehat\theta_1$ corresponding to (\ref{log:like}) should satisfy 
	\[
	0=\frac{\partial l}{\partial \theta_1}(\widehat\theta)=\frac{1}{\theta_2}\sum_j(X_j-\widehat\theta_1), \quad 
	0=\frac{\partial l}{\partial \theta_2}(\widehat\theta)=-\frac{n}{2\theta_2}+\frac{1}{2\theta_2^2}\sum_j(X_j-\widehat\theta_1)^2,
	\]
	which gives
	\[
	\widehat\theta_1=\overline  X_n=\frac{1}{n}\sum_jX_j, \quad \widehat\theta_2=\widehat\sigma^2_{n^{-1}}=\frac{1}{n}\Sigma_j(X_j-\overline X_n)^2. 
	\]
	We thus see that the ML estimator $\widehat\theta_2=\widehat\sigma^2_{n^{-1}}$ for the variance coming from (\ref{log:like}) not only fails to be unbiased but also satisfies 
	\[
	{\rm mse}(\widehat\sigma^2_{(n-1)^{-1}})>{\rm mse}(\widehat\sigma^2_{n^{-1}})>{\rm mse}(\widehat\sigma^2_{(n+1)^{-1}}),
	\] 
	so its performance, as measured by mse, lies somewhere between those of the variance estimators considered so far. 
	We point out that the asymptotic performance of a (sufficiently regular and consistent)  ML estimator is examined in Theorem \ref{asym:cr} below. In particular, asymptotically normality is established there, which confirms that $\widehat\sigma^2_{n^{-1}}$ stands out as the most efficient estimator from this viewpoint. \qed 
\end{example}

\begin{example}\label{corr:mle} (MLE from a jointly normal population)
	Using the notation of Example \ref{corr:dist}, we see that the right-hand side of (\ref{bi:var:norm:like}) allows us to write down the likelihood function  of the jointly normal random sample $\{X,Y\}$ as 
	\begin{eqnarray*}
		L({\bf x},{\bf y};\theta)
		& = & 
		\frac{1}{(2\pi\sigma_X\sigma_Y\sqrt{1-\rho^2})^m} \times \\
		& & \quad \times
		e^{-\frac{1}{2(1-\rho^2)}\sum_j
			\left(\frac{A(x_j)}{\sigma_X^2}-
			\frac{2\rho B(x_j,y_j)}{\sigma_X\sigma_Y}
			+\frac{C(y_j)}{\sigma_Y^2}\right)},
	\end{eqnarray*}
	where $\theta=(\mu_X,\mu_Y,\sigma_X^2,\sigma_Y^2,\rho)$ and 
	\begin{equation}\label{coeff:exp}
		A(x_j)=(x_j-\mu_X)^2, \quad B(x_j,y_j)=(x_j-\mu_X)(y_j-\mu_Y), \quad C(x_j)=(y_j-\mu_Y)^2, 
	\end{equation} 
	so that 
	\begin{eqnarray}\label{exp:loglike:joint}
		l({\bf x},{\bf y};\theta)
		& = & 
		-m\ln (2\pi\sigma_X\sigma_Y) -\frac{m}{2}\ln  (1-\rho^2) - \nonumber\\
		& & \quad 
		{-\frac{1}{2(1-\rho^2)}\sum_j
			\left(\frac{A(x_j)}{\sigma_X^2}-
			\frac{2\rho B(x_j,y_j)}{\sigma_X\sigma_Y}
			+\frac{C(y_j)}{\sigma_Y^2}\right)},
	\end{eqnarray}
	Starting from this,
	a straightforward analysis involving the first derivative test $\nabla_\theta l=0$ confirms  that 
	$\widehat\theta:=(\overline X_m,\overline Y_m,\widehat\sigma^2_{m^{-1}}(X),\widehat\sigma^2_{m^{-1}}(Y),\widehat\rho)$
	is the ML estimator of $\theta$; see \cite[Section 14.11]{kendall1946advanced} or \cite[Corollary 3.2.2]{anderson2003introduction} for the details. \qed
\end{example}

\begin{example}\label{mle:exp} (MLE from an exponential population)
	If $X_j\sim{\rm Exp}(\lambda)$, the exponential distribution with parameter $\lambda>0$, where
	\[
	{\rm Exp}(\lambda)(x)=\lambda e^{-\lambda x}{\bf 1}_{[0,+\infty)},
	\]
	then 
	\[
	L({\bf x};\lambda)=\lambda^{n}e^{-\lambda\sum_jx_j}\Longrightarrow l({\bf x};\lambda)=n\ln \lambda-\lambda\sum_jx_j,
	\]
	so that 
	\[
	\widehat\lambda = \frac{n}{\sum_jX_j},
	\]
	which matches the fact that $\mathbb E(X_j)=1/\lambda$. \qed
\end{example}

\begin{example}\label{mle:discrete}
	If $X_j\sim \mathsf{Ber}(p)$, the Bernoulli distribution in Remark \ref{dem:lap}, then for $x\in\{0,1\}$ we have 
	$P(X_j=x)=p^x(1-p)^{1-x}$, where $p\in (0,1)$ is the unknown parameter. If we assume as always that $\{X_j\}$ is independent the associated likelihood function is
	\[
	L({\bf x};p)= p^{\sum_jx_j}(1-p)^{n-\sum_jx_j}, 
	\] 
	and hence,
	\begin{equation}\label{logl:ber}
		l({\bf x};p)=\left(\sum_jx_j\right)\ln p+\left(n-\sum_jx_j\right)\ln (1-p). 
	\end{equation}
	The usual first derivative test for a minimum at $p=\widehat p$ is 
	\[
	0=\frac{\partial l}{\partial p}(\widehat p)=\frac{\sum_jx_j}{\widehat p}-\frac{n-\sum_jx_j}{1-\widehat p},
	\]
	so that 
	\[
	\widehat p=\frac{1}{n}\sum_jX_j,
	\]
	the sample mean. Also, if $X_j\sim{\mathsf{Pois}}(\rho)$, the Poisson distribution with parameter $\rho>0$, so that $P(X_j=x)=\rho^xe^{-\rho}/x!$, $x\in\{0,1,2,\cdots,\}$, then a simple computation shows that
	\[
	\widehat\rho=\frac{1}{n}\sum_jX_j, 
	\]
	which confirms that  the corresponding ML estimator is also the sample mean.
	\qed
\end{example}

\begin{remark}\label{com:analy}
	The analysis in the examples above should be complemented with the usual second derivative test to check in each case that the ML estimator attains the (unique) global maximum of the corresponding likelihood function. \qed
\end{remark}

\subsection{Fisher information and Cram\'er-Rao lower bound}\label{cr:rao:fisher}
We now present a universal lower bound for the covariance matrix in each class of estimators with a prescribed expectation (in particular, for unbiased estimators) in terms of an invariant (Fisher information) depending on the likelihood function of the given statistical model. Instead of restricting ourselves to statistical models, let us assume for the moment only that $X_j:\Omega\to\mathbb R$ are independent random variables, $i=1,\cdots,n$, giving rise to a random vector $X=(X_1,\cdots,X_n)$\footnote{In other words, we consider here a statistical model in the  extended sense of Remark \ref{stat:mod:def:2}.}. Let 
$L=L({\bf x};\theta)>0$ be the corresponding likelihood function and 
$l({\bf x};\theta)=\ln L({\bf x};\theta)$ the log-likelihood function, where ${\bf x}=(x_1,\cdots,x_n)$ and $\theta\in\Theta\subset\mathbb R^q$, the space of parameters. Clearly, the notions of statistics and estimators can be easily adapted to this broader setting. In the sequel we assume that $l$ is regular enough so that all the differential/integral manipulations hold true.

\begin{definition}\label{rscr:fische}
	Under the conditions above, we define the {\em score vector} (of the given sample $X$) as 
	\[
	s(X;\theta)=\nabla_\theta l(X;\theta).
	\]
	Also, the corresponding {\em Fisher information matrix} is 
	\begin{equation}\label{fisher:mat:def}
		\mathscr F^X(\theta)={\mathbb C}(s(X;\theta)). 
	\end{equation}
\end{definition}

We henceforth assume that the symmetric matrix $\mathscr F=\mathscr F^X$ in (\ref{rscr:fische}) is positive definite, so the inverse matrix $\mathscr F^{-1}$ exists.
Since the random effects present in the sample have been averaged out after taking covariance of the score, $\mathscr F=\mathscr F(\theta)$ is an invariant of the given model, in particular not being attached to any potential estimator. 

\begin{remark}\label{ill:ind:fis}
	In order to illustrate the importance of requiring that $\{X_j\}$ is independent, let us assume that we are in the unidimensional case, $\Theta\subset\mathbb R$, so  we call the scalar $\mathscr F$ simply the {\em Fisher information}.
	One has
	\[
	s(X;\theta)=\sum_j\frac{\frac{d}{d\theta}\psi_j(X_j;\theta)}{\psi_j(X_j;\theta)},
	\]
	a sum of {\em independent} random variables, so that by (\ref{uncorr:var}),
	\[
	{\mathbb V}(s(X;\theta))=\sum_j {\mathbb V}\left(\frac{\frac{d}{d\theta}\psi_j(X_j;\theta)}{\psi_j(X_j;\theta)}\right),
	\]
	which means that 
	\[
	\mathscr F^{X}=\sum_j\mathscr F^{X_j}.
	\]
	Thus, independence leads to a simple additive formula describing how the Fisher information of the whole sample decomposes as a sum of contributions coming from its parts.  
	If we additionally require that $X_j\sim\psi_\theta$ is i.d.d., which is the only case treated in all examples below, then this becomes
	\[
	\mathscr F_{(n)}=n\mathscr F_{(1)},
	\]
	with $\mathscr F_{(n)}=\mathscr F^{X}$ refering as before to the whole sample whereas $\mathscr F_{(1)}=\mathscr F^{X_j}$ refers to {\em any} single observation. \qed
\end{remark}

\begin{example}\label{ber:fis:c} (Fisher information for a Bernoulli sample)
	If $X_j\sim \mathsf{Ber}(p)$ then (\ref{logl:ber}) gives
	\[
	s(X;p)=\frac{\partial}{\partial p} l(X;p)=\frac{\sum_jX_j}{p}-\frac{n-\sum_jX_j}{1-p}=\frac{n}{p(1-p)}\overline X-\frac{n}{1-p},
	\]
	and since ${\mathbb C}(\overline X)={\mathbb C}(X_j)/n=p(1-p)/n$, we conclude that
	\begin{equation}\label{fish:inf:ber}
		\mathscr F_{(n)}(p)=\frac{n}{p(1-p)}.
	\end{equation}
	Thus, 
	the Fisher information increases with the sample size according to a rate which is inversely proportional to the ``fluctuation'' (as measured by the population variance). In a sense, this simple example justifies the qualification of ``information'' for this concept; for more on this point see Remark \ref{decay:fluc} below. \qed
\end{example}

\begin{example}\label{fisher:normal} (Fisher information for a normal sample)
	If $X_j\sim\mathcal N(\theta_1,\theta_2)$ is drawn from a normal population, where $\theta_1=\mu$ and $\theta_2=\sigma^2$,  then (\ref{log:like}) leads to 
	\begin{equation}\label{score:norm}
		s(X;\theta)=\nabla_\theta l({ X};\theta)=
		\left(
		\begin{array}{c}
			\frac{1}{\theta_2}\sum_j(X_j-\theta_1)\\
			-\frac{n}{2\theta_2}+\frac{1}{2\theta_2^2}\sum_j(X_j-\theta_1)^2
		\end{array}
		\right),\quad \theta=
		\left(
		\begin{array}{c}
			\theta_1\\
			\theta_2
		\end{array}
		\right).
	\end{equation}
	Since $X_j-\theta_1\sim\mathcal N(0,\theta_2)$, independence implies that 
	\begin{equation}\label{fis:d1}
		{\mathbb V}\left(\frac{1}{\theta_2}\sum_j(X_j-\theta_1)\right)=\frac{n}{\theta_2}.
	\end{equation}
	On the other hand, 
	\[
	-\frac{n}{2\theta_2}+\frac{1}{2\theta_2^2}\sum_j(X_j-\theta_1)^2=	-\frac{n}{2\theta_2}+\frac{1}{2\theta_2}\sum_j\left(\frac{X_j-\theta_1}{\sqrt{\theta_2}}\right)^2,
	\]
	and since $(X_j-\theta_1)/\sqrt{\theta_2}\sim\mathcal N(0,1)$, we see that
	\[
	\sum_j\left(\frac{X_j-\theta_1}{\sqrt{\theta_2}}\right)^2\sim\chi^2_n
	\]
	by Corollary \ref{sum:norm:sq}, so that Corollary \ref{chi:sq:ms} applies to give  
	\begin{equation}\label{chi:chi}
		{\mathbb V}\left(\sum_j\left(\frac{X_j-\theta_1}{\sqrt{\theta_2}}\right)^2\right)=2n,
	\end{equation}
	and hence,
	\begin{equation}\label{fis:d2}
		{\mathbb V}\left(	-\frac{n}{2\theta_2}+\frac{1}{2\theta_2^2}\sum_j(X_j-\theta_1)^2\right)=\frac{n}{2\theta_2^2}.
	\end{equation}
	Note that (\ref{fis:d1}) and (\ref{fis:d2}) provide the diagonal terms of the corresponding Fisher information matrix. 
	In order to compute the off-diagonal terms we observe that 
	\[
	\mathbb E\left(\frac{1}{\theta_2}\sum_j(X_j-\theta_1)\right)=0
	\]
	and 
	\begin{eqnarray*}
		\mathbb E\left(-\frac{n}{2\theta_2}+\frac{1}{2\theta_2^2}\sum_j(X_j-\theta_1)^2\right)
		& = & -\frac{n}{2\theta_2}+\frac{1}{2\theta_2}\mathbb E\left(\sum_j\left(\frac{X_j-\theta_1}{\sqrt{\theta_2}}\right)^2\right)\\
		& = &  -\frac{n}{2\theta_2}+ \frac{n}{2\theta_2}\\
		& = & 0,
	\end{eqnarray*}
	where we used (\ref{chi:chi}) and Corollary \ref{chi:sq:ms} in the next to the last step; this should be compared with the general result in Corollary \ref{any:fct:2} below.
	It follows that    
	\[
	{\mathbb C}\left( \frac{1}{\theta_2}\sum_j(X_j-\theta_1),-\frac{n}{2\theta_2}+\frac{1}{2\theta_2^2}
	\sum_j(X_j-\theta_1)^2\right)=\frac{1}{2\theta_2^3}
	\sum_j\mathbb E\left((X_j-\theta_1)^3\right),
	\]
	which clearly vanishes. We thus conclude that
	\begin{equation}\label{fish:inf:nor:0}
		\mathscr F_{(n)}(\theta)=
		\left(
		\begin{array}{cc}
			{n}/{\theta_2} & 0\\
			0 & {n}/{2\theta_2^2}
		\end{array}
		\right).
	\end{equation}
	In particular, if we consider $\theta_2$ as known,
	\begin{equation}\label{fish:inf:nor}
		\mathscr F_{(n)}(\theta_1)=\frac{n}{\theta_2}.
	\end{equation}
	We will see in Remark \ref{fisher:normal:2} below a much simpler route to retrieve (\ref{fish:inf:nor:0}). \qed
\end{example}	

In both (\ref{fish:inf:nor}) and (\ref{fish:inf:ber}), we observe that the Fisher information $\mathscr F=\mathscr F_{(n)}$ equals the reciprocal of the variance of the corresponding estimator, in both cases the (unbiased) sample mean. 
This reflects the remarkable general fact that the information matrix provides a universal lower bound for the covariance of any class of estimators with a prescribed bias, including in particular the unbiased ones; see Theorem~\ref{cr:rao:th} below. 
The next result is the first step toward establishing this fundamental connection.

\begin{proposition}\label{any:fct}
	For any (sufficiently regular) vector $t=t({\bf x};\theta)$ there holds
	\[
	\mathbb E(s\otimes t)=\nabla_\theta \mathbb E(t)-\mathbb E(\nabla_\theta t). 
	\]
\end{proposition}

\begin{proof}
	We compute:
	\begin{eqnarray*}
		\mathbb E(s \otimes t) & = & \int_{\mathbb R^n} L({\bf x};\theta)^{-1}\nabla_\theta L({\bf x};\theta)\otimes t({\bf x};\theta)L({\bf x};\theta)d{\bf x}\\
		& = & \int_{\mathbb R^n} \nabla_{\theta}(L({\bf x};\theta)t({\bf x};\theta))d{\bf x}
		-\int_{\mathbb R^n} L({\bf x};\theta)\nabla_\theta t({\bf x};\theta)d{\bf x}\\
		& = & \nabla_{\theta}\int_{\mathbb R^n} L({\bf x};\theta)t({\bf x};\theta)d{\bf x}
		-\int_{\mathbb R^n} L({\bf x};\theta)\nabla_\theta t({\bf x};\theta)d{\bf x},	
	\end{eqnarray*}
	as claimed.
\end{proof}

\begin{corollary}\label{any:fct:2}
	There holds 
	\begin{equation}\label{fischer:0}
			\mathbb E(s)=\vec 0.
	\end{equation}		
			 In particular, 
	\begin{equation}\label{fischer}
		\mathscr F=\mathbb E(s\otimes s)=-\mathbb E(\nabla^2_{\theta\theta}l).
	\end{equation}
\end{corollary}

\begin{proof}
	Take $t=(1,\cdots,1)$.
\end{proof}

\begin{corollary}\label{any:fct:3}
	If $t=t({\bf x})$ then $\mathbb E(s\otimes t)=\nabla_\theta \mathbb E(t)$. In particular, if  $t=\widehat\theta$ is an estimator with $g(\theta):=\mathbb E(\widehat\theta)$ then $\mathbb E(s\otimes\widehat\theta)=\nabla_\theta g$.
\end{corollary}

\begin{proof}
	The first assertion is immediate and the second one follows from 
	the fact that $\widehat\theta$, as an estimator,  does {\em not} depend on $\theta$.
\end{proof}

\begin{remark}\label{fisher:normal:2}
	As a checking we may use (\ref{fischer}) to recalculate the Fisher information matrix of a normal sample $X_j\sim \mathcal N(\theta_1,\theta_2)$ as in Example \ref{fisher:normal}. From (\ref{score:norm}) we have
	\[
	\nabla_{\theta\theta}l(X;\theta)=
	\left(
	\begin{array}{cc}
		-\frac{n}{\theta_1} & -\frac{1}{\theta_2^2}\sum_j(X_j-\theta_1)\\
		-\frac{1}{\theta_2^2}\sum_j(X_j-\theta_1) & \frac{n}{2\theta_2^2}-\frac{1}{\theta_2^3}\sum_j(X_j-\theta_1)^2
	\end{array}
	\right),
	\] 
	so that 
	\[
	\mathscr F_{(n)}(\theta)=
	\left(
	\begin{array}{cc}
		\frac{n}{\theta_1} & 0\\
		0 & -\frac{n}{2\theta_2^2}+\frac{1}{\theta_2^3}\sum_j\mathbb E\left((X_j-\theta_1)^2\right)
	\end{array}
	\right),
	\]
	and since $\mathbb E((X_j-\theta_1)^2)={\mathbb V}(X_j)=\theta_2$, we recover (\ref{fish:inf:nor:0}). Note that this computation is much simpler because it is based on computing expectations, at the cost of taking one more derivative of log-likelihood function but with no need to compute covariances, and hence bypasses any appeal to the connection between sums of squares of normals and chi-squares. \qed
\end{remark}

Let $\mathcal E$ be the set of all estimators (for $\theta$). Given $g:\Theta\to\mathbb R^q$ define
\[
\mathcal E_g=\left\{\widehat\theta\in\mathcal E;\mathbb E(\widehat\theta)=g(\theta)\right\}. 
\]
Equivalently, $\mathcal E_g$ is the set of all {\em unbiased} estimators for $g(\theta)$.
Note that each $\widehat\theta\in\mathcal E_g$ satisfies 
\begin{equation}\label{exp:g:bias}
{\rm bias}(\widehat\theta)=g(\theta)-\theta
\end{equation}
and hence
\[
{\rm mse}(\widehat\theta)=\|g(\theta)-\theta)\|^2+{\rm tr}\,{\mathbb C}(\widehat\theta). 
\]
The next result provides a uniform lower bound for the covariance (and hence for the ${\rm mse}$) of estimators in each class $\mathcal E_g$ (provided it is not empty). 

\begin{theorem}\label{cr:rao:th}(Cram\'er-Rao)
	There  holds 
	\begin{equation}\label{cr:rao:th:3:1}
		{\mathbb C}(\widehat\theta)\geq \nabla_\theta g\mathscr F(\theta)^{-1} 
		{\nabla_\theta g}^\top
	\end{equation}
	for any $\widehat\theta\in\mathcal E_g$. In particular, 
	\begin{equation}\label{cr:rao:th:3}
		{\mathbb C}(\widehat\theta)\geq \mathscr F(\theta)^{-1}
	\end{equation}
	if $\widehat\theta$ is unbiased ($g(\theta)=\theta$).
\end{theorem}

\begin{proof}
	We first consider the uni-dimensional case $\Theta\subset\mathbb R$. From Corollaries \ref{any:fct:3} and \ref{any:fct:2} we have
	\begin{eqnarray*}
		g'(\theta)
		& = & \int_{\mathbb R^n}s({\bf x};\theta)\widehat\theta({\bf x})L({\bf x};\theta)dx\\
		& = & \int_{\mathbb R^n}s({\bf x};\theta)\left(\widehat\theta({\bf x})-g(\theta)\right)L({\bf x};\theta)d{\bf x},
	\end{eqnarray*}
	so that Cauchy-Schwartz inequality gives 
	\[
	|g'(\theta)|^2\leq {\mathbb V}(s(X;\theta)){\mathbb V}(\widehat\theta(X)),
	\]
	as claimed. 
	The proof of the multi-dimensional  case is quite similar and makes 
	use of a well-known algebraic inequality: for any random vectors $Z,W\in\mathbb R^q$ with ${\mathbb C}(W)>0$ there holds
	\[
	{\mathbb C}(Z)\geq {\mathbb C}(Z,W){\mathbb C}(W)^{-1}{\mathbb C}(W,Z). 
	\] 
	Taking $Z=\widehat\theta$ and $W=s$ we get 
	\begin{eqnarray*}
		{\mathbb C}(\widehat\theta) & \geq & {\mathbb C}(\widehat\theta,s){\mathbb C}(s)^{-1}{\mathbb C}(s,\widehat\theta)\\
		& = & \mathbb E(\widehat\theta\otimes s)\mathbb E(s\otimes s)^{-1}\mathbb E(s\otimes\widehat\theta)\\
		& = & \nabla_\theta g\mathscr F(\theta)^{-1} {\nabla_\theta g}^\top,
	\end{eqnarray*}
	as claimed.
\end{proof}

\begin{corollary}\label{cr:rao:th:cor}
	The best estimator in $\mathcal E_g$ (if it exists) is the one whose covariance matrix attains the lower bound in (\ref{cr:rao:th:3:1}). In particular,
	an unbiased estimator whose covariance matrix attains the lower bound in (\ref{cr:rao:th:3}) has the best performance (as measured by the mse).
\end{corollary}

\begin{example}\label{ex:cr:discr} (The sample mean as the best unbiased estimator of the expected value of a Bernoulli or Poisson population)
	It follows from (\ref{fish:inf:ber}) that, for a Bernoulli population,
	\[
	\mathscr F_{(n)}(p)=\frac{n}{p(1-p)}=\frac{1}{{\mathbb V}(\widehat p)},
	\]
	so Corollary \ref{cr:rao:th:cor} applies and the sample mean $\widehat p$ is the best unbiased estimator for the expected value.
	A similar reasoning, based on the explicit computation of the corresponding Fisher information, confirms that the sample mean is the best unbiased estimator for the expected value of a Poisson random sample as in Example \ref{mle:discrete}.
	\end{example}

\begin{example}\label{mean:best}(The sample mean as the best estimator of the expected value of a normal population)
	As observed in~\cite{hodges1951some}, the Cramér–Rao inequality in Theorem~\ref{cr:rao:th} may be used to prove that the sample mean $\widehat\theta_1=\overline X$ is the {\em best} estimator for the mean $\theta_1=\mu$ of a normal population, in the sense that it attains the smallest possible ${\rm mse}$ among {\em all} such estimators.\footnote{In this case we say that $\widehat\theta_1$ is {\em admissible}, meaning that there exists no other estimator $\widehat\theta_\bullet$ such that ${\rm mse}(\widehat\theta_\bullet)<{\rm mse}(\widehat\theta_1)$.}; see also~\cite[Example~5.2.8]{lehmann2006theory} for an alternative proof of this fact. 
	Here we adopt the notation of Example~\ref{fisher:normal}, so that our sample satisfies $X_j\sim\mathcal N(\theta_1,\theta_2)$, $j=1,\dots,n$. Let $\widehat\theta_\bullet$ be an estimator of $\theta_1$ such that ${\rm mse}(\widehat\theta_\bullet)\leq {\rm mse}(\widehat\theta_1)=\theta_2/n$. Setting $b(\theta_1)={\rm bias}_{\theta_1}(\widehat\theta_\bullet)$ and noting that $\mathcal F(\theta_1)=n/\theta_2$ by~(\ref{fish:inf:nor:0}), we obtain from~(\ref{cr:rao:th:3:1}) and~(\ref{exp:g:bias}) that
	\begin{equation}\label{eq:theta:1}
		b(\theta_1)^2+\frac{\theta_2}{n}\,(1+b'(\theta_1))^2\leq \frac{\theta_2}{n}.
	\end{equation}
	It follows that $b$ is uniformly bounded, $|b|\leq \sqrt{\theta_2/n}$, and that $b'\leq 0$ everywhere. If there existed $\epsilon>0$ such that $b'\leq -\epsilon$ for $\theta_1\to\pm\infty$, then $b$ would not remain bounded, which is impossible. Hence $b'(\pm\infty)=0$, and from (\ref{eq:theta:1}) we deduce that $b(\pm\infty)\to0$. Since $b$ is nonincreasing, this forces $b\equiv0$. 
	Therefore $\widehat\theta_\bullet$ is unbiased and satisfies ${\rm mse}(\widehat\theta_\bullet)={\rm mse}(\widehat\theta_1)$, as claimed. As also noted in~\cite{hodges1951some}, this argument can be adapted to the setting of Bernoulli and Poisson populations in Example~\ref{ex:cr:discr}, thereby eliminating the need for an unbiasedness assumption on the competing estimators.
	 \qed
	\end{example}

\begin{example}\label{james:stein}
	(The James-Stein estimator \cite{james1961estimation})
Starting with a single sample $X\sim \mathcal N({\bm\mu},\sigma^2{\rm Id}_p)$, where $\sigma^2$ is known, the log-likehood function 
\begin{equation}\label{log:like:js}
l({\bf x};\bm\mu)=-\frac{p}{2}\ln 2\pi\sigma^2-\frac{1}{2\sigma^2}\|{\bf x}-{\bm\mu}\|^2
\end{equation}
tells us 
that the MLE for ${\bm\mu}$ is $\widehat{\bm\mu}=X$. Since $\sigma^{-2}\|X-{\bm\mu}\|^2\sim\chi^2_p$ we know that
\begin{equation}\label{mse:vec:mu}
{\rm mse}(\widehat{\bm\mu})=\mathbb E(\|X-{\bm\mu}\|^2)=p\sigma^2,
\end{equation}
with only the variance contributing (since $\widehat{\bm\mu}$ is unbiased).
Now let us compare $\widehat{\bm \mu}$ with the {\em James-Stein estimator}
\begin{equation}\label{j:s:est}
{\widehat{\bm\mu}}_{JS}=\left(1-\frac{(p-2)\sigma^2}{\|\widehat{\bm\mu}\|^2}\right)\widehat{\bm\mu}=
X-\left(p-2\right)\sigma^2\frac{X}{\|X\|^2},
\end{equation}
whose mean squared error is 
\begin{eqnarray*}
{\rm mse}({\widehat{\bm\mu}}_{JS})
	& = & \mathbb E\left(\|{\widehat{\bm\mu}}_{JS}-\bm\mu\|^2\right)\\
	& = & 
\mathbb E\left(
\left\|\widehat{\bm\mu}-{\bm\mu}-(p-2)\sigma^2\frac{\widehat{\bm\mu}}{\|\widehat{\bm\mu}\|^2}\right\|^2
\right)\\
	& = & 
	{\rm mse}(\widehat{\bm\mu}) -2(p-2)\sigma^2\mathbb E\left(\frac{\langle X,X-\bm\mu\rangle}{\|X\|^2}\right)+(p-2)^2\sigma^4\mathbb E\left(\|X\|^{-2}\right)\\
& = & 
	p\sigma^2 -2(p-2)\sigma^2\mathbb E\left(\frac{\langle X,X-\bm\mu\rangle}{\|X\|^2}\right)+(p-2)^2\sigma^4\mathbb E\left(\|X\|^{-2}\right),	
\end{eqnarray*}
where we used (\ref{mse:vec:mu}) in the last step.
In order to handle the mixed term in the right-hand side we first note from (\ref{log:like:js}) that the score vector is 
\[
s({\bf x};\bm\mu)=\nabla_{{\bm\mu}}{l}({\bf x};\bm\mu)=\sigma^{-2}\left({\bf x}-\bm\mu\right),
\]
so if we make $t=t(X)\in\mathbb R$ in Corollary \ref{any:fct:3} we obtain
\begin{equation}\label{eq:stein}
\mathbb E\left(\frac{\partial}{\partial x_j}t(X)\right)=\frac{\partial}{\partial{\mu}_j}\mathbb E\left(t(X)\right)=\sigma^{-2}\mathbb E\left(t(X)(X_j-{\bm\mu}_j)\right), \quad j=1,\cdots,p,
\end{equation}
a result usually known as {\em Stein's equation}\footnote{Remarkably enough, the validity of (\ref{eq:stein}) for all $t$ varying in a suitable class of test functions completely characterizes $\sigma^{-1}(X-\bm\mu)$ as a standard normal random vector, which turns out to be the starting point of Stein's approach to the Berry-Esseen theorem discussed in Remark \ref{miscon} \cite{chen2021stein}.}.
By taking $t(X)=X_j/\|X\|^2$ and summing over $j$ we realize that 
\begin{equation}\label{counter:stein}
\mathbb E\left(\frac{\langle X,X-\bm\mu\rangle}{\|X\|^2}\right)=(p-2)\sigma^2\mathbb E\left(\|X\|^{-2}\right),
\end{equation}
which gives 
\[
{\rm mse}({\widehat{\bm\mu}}_{JS})= p\sigma^2
-(p-2)^2\sigma^4 \mathbb E\left(\|X\|^{-2}\right).
\]
On the other hand, it follows from (\ref{exp:form}) that 
\[
\mathbb E\left(\|X\|^{-2}\right)=\frac{1}{(2\pi)^{p/2}}\int_{\mathbb R^p}\|{\bf x}\|^{-2}e^{-\frac{1}{2}\|{\bf x}-\bm\mu\|^2}d{\bf x},
\]
an integral which becomes finite if $\int r^{p-3}dr$ converges near $r=0$. Thus, we conclude that  
${\rm mse}_{\bm\mu}({\widehat{\bm\mu}}_{JS})< p\sigma^2$ if $p\geq 3$ for {\em any} $\bm\mu$, which confirms that in those cases the unbiased MLE estimator $\widehat{\bm\mu}$ fails to be the most efficient one (if the ``performance'' is measured by the mean squared error); cf. Remark \ref{comp:mse}. 
Regarding this remarkable estimator, we add the following comments.
\begin{itemize}
	\item Since $\widehat{\bm\mu}=X$ is simply the sample mean as we have just a single observation at our disposal, this is in sharp contrast with the result in Example \ref{mean:best}, which says that the sample mean is the best estimator for ${\bm \mu}$ is $p=1$. To reinforce this analogy, let us take a random sample  $X_j\sim \mathcal N({\bm\mu},\sigma^2{\rm Id}_p)$, $j=1,\cdots,n$, so that $n$ observations of the underlying multivariate normal population are available. Now, the log-likelihood function is 
	\[
		l({\bf x};\bm\mu)=-\frac{np}{2}\ln 2\pi\sigma^2-\frac{1}{2\sigma^2}\sum_j\|{\bf x}_j-{\bm\mu}\|^2,\quad {\bf x}=({\bf x}_1,\cdots,{\bf x}_n)\in\mathbb R^{np},
	\]
	so the MLE for ${\bm\mu}$ is $\widehat{\bm\mu}^{(n)}=\overline X$, where 
	\[
	\overline X=\frac{1}{n}\sum_jX_j\in\mathbb R^p.
	\]
	Since $\sigma^{-2}\|X_j-{\bm\mu}\|^2\sim\chi^2_p$ for each $j$ and $\{X_j-{\bm \mu}\}_{j=1}^n$ is independent, we  have
	\begin{eqnarray*}
		{\rm mse}(\widehat{\bm\mu}^{(n)})
		& = & 
		\mathbb E(\|\overline X-{\bm\mu}\|^2)\\
		& = & 
		\frac{1}{n^2}\mathbb E\left(\left\|\sum_j(X_j-\bm\mu)\right\|^2\right)\\
			& = & 
		\frac{1}{n^2}\sum_j\mathbb E\left(\left\|X_j-\bm\mu\right\|^2\right)\\
		& = & \frac{p}{n}\sigma^2,
	\end{eqnarray*}
	where again  only the covariance contributes (since $\widehat{\bm\mu}^{(n)}$ is unbiased). It turns out that essentially the same argument as above confirms that $\widehat{\bm\mu}^{(n)}$  fails to be admissible if $p\geq 3$, as the corresponding James-Stein estimator
	\[
	{\widehat{\bm\mu}}_{JS}^{(n)}=\left(1-\frac{(p-2)\sigma^2/n}{\|\widehat{\bm\mu}^{(n)}\|^2}\right)\widehat{\bm\mu}^{(n)}=
	\overline X-\left(p-2\right)\frac{\sigma^2}{n}\frac{\overline X}{\|\overline X\|^2}
	\]
	satisfies ${\rm mse}({\widehat{\bm\mu}}_{JS}^{(n)})<p\sigma^2/n$.
	\item In case $\sigma^2$ is unknown, and  restricting ourselves to the case $n=1$ for simplicity, let us replace (\ref{j:s:est}) by 
	\[
		{\widehat{\bm\mu}}_{JS_l}=\left(1-\frac{(p-2)c_l{\mathfrak s}}{\|\widehat{\bm\mu}\|^2}\right)\widehat{\bm\mu}=
		X-\left(p-2\right)c_l{\mathfrak s}\frac{X}{\|X\|^2},
	\]
	where $p\geq 3$, $c_l$ is a positive constant (depending on a positive integer $l$ given in advance) to be determined below and $\mathfrak s$ is the appropriate estimator of $\sigma^2$ in the sense that $\sigma^{-2}\mathfrak s\sim\chi^2_l$ and $\{\mathfrak s,X\}$ is independent. 
Setting ${\bm\mu}_\bullet=\sigma^{-1}{\bm\mu}$, $\widehat{\bm\mu}_\bullet=\sigma^{-1}\widehat{\bm\mu}$, $X_\bullet=\sigma^{-1}X$ and $\mathfrak s_\bullet=\sigma^{-2}\mathfrak s$, we compute  
\begin{eqnarray*}
	{\rm mse}({\widehat{\bm\mu}}_{JS_l})
	& = & \mathbb E\left(\|{\widehat{\bm\mu}}_{JS_l}-\bm\mu\|^2\right)\\
	& = & 
	\sigma^2
	\mathbb E\left(
	\left\|\widehat{\bm\mu}_\bullet-{\bm\mu}_\bullet-(p-2)c_l\mathfrak s_\bullet\frac{\widehat{\bm\mu}_\bullet}{\|\widehat{\bm\mu}_\bullet\|^2}\right\|^2
	\right)\\
	& = & 
	\sigma^2
	\left(
	{\rm mse}(\widehat{\bm\mu}_\bullet) -2(p-2)c_l l\mathbb E\left(\frac{\langle X_\bullet,X_\bullet-\bm\mu_\bullet\rangle}{\|X_\bullet\|^2}\right)+(p-2)^2c_l^2l(l+2)\mathbb E\left(\|X_\bullet\|^{-2}\right)
	\right),
\end{eqnarray*}
where we used the independence and that  $\mathbb E(\mathfrak s_\bullet)=l$ and $\mathbb E(\mathfrak s_\bullet^2)=l(l+2)$ in the last step. Combining this with the obvious counterpart of (\ref{counter:stein}) we end up with 
\[
{\rm mse}({\widehat{\bm\mu}}_{JS_l})=\sigma^2
\left(
p-
(p-2)^2l\left[
2c_l-c_l^2(l+2)
\right]\mathbb E(\|X_\bullet\|^{-2})
\right),
\]
from which we see that the best choice is $c_l=1/(l+2)$, in which case
\[
	{\widehat{\bm\mu}}_{JS_l}=\sigma^2\left(k-\frac{p-2}{l+2}\frac{{\mathfrak s}}{\|\widehat{\bm\mu}\|^2}\right)\widehat{\bm\mu}
\]
certainly satisfies ${\rm mse}({\widehat{\bm\mu}}_{JS_l})<p\sigma^2$. 
	\item 
As the formulas above make clear (see, for instance, (\ref{j:s:est})), the James--Stein estimator shrinks the sample mean toward the origin. This adjustment introduces a small amount of bias by pulling the estimate away from its observed value, yet when the number $p$ of components in the underlying normal mean vector is sufficiently large, the reduction in variance more than compensates for this bias. The outcome is an estimator with a smaller total error, as measured by the mean squared error\footnote{A similar phenomenon appears in Corollary \ref{stat:normal:c}, where $\widehat\sigma^2_{(n+1)^{-1}}$ may be interpreted as a shrinkage of both the unbiased estimator $\widehat\sigma^2_{(n-1)^{-1}}$ and the maximum likelihood estimator $\widehat\sigma^2_{n^{-1}}$.}. 
This principle of shrinkage represented more than a technical refinement; it marked a genuine paradigm shift in Statistics. It laid the foundation for regularization techniques such as Ridge regression and the Lasso, which have since become essential tools in Data Science and Machine Learning, particularly in high-dimensional contexts; see Subsection~\ref{ridge} for further discussion in the setting of linear regression.
	\item Instead of shrinking toward the origin, it is often convenient to choose some $\bm\nu\in\mathbb{R}^k$ and replace (\ref{j:s:est}) by  
	\[
	{\widehat{\bm\mu}}_{JS_{\bm\nu}}
	=\left(1-\frac{(p-2)\sigma^2}{\|\widehat{\bm\mu}-\bm\nu\|^2}\right)(\widehat{\bm\mu}-\bm\nu)+\bm\nu,
	\]
	thereby performing shrinkage toward $\bm\nu$. The resulting estimator always satisfies ${\rm mse}({\widehat{\bm\mu}}_{JS_{\bm\nu}})<\sigma^2k$.  
	Although the optimal choice of $\bm\nu$ is not generally known, this formulation introduces a useful degree of flexibility. A particularly natural, data-driven option is the \emph{grand mean vector} $\overline X_{\mathrm{gm}}{\bf 1}$, where $\overline X_{\mathrm{gm}}$ denotes the arithmetic mean of the components of the observed sample mean $X$. This choice was used in the classical analysis of the baseball batting averages data set by \cite{efron1977stein}.
	\qed
	\end{itemize}
	\end{example}

\begin{remark}\!\!$\bigstar$\label{uncertain:p}
	By rewriting (\ref{cr:rao:th:3}) as 
	\[
	{\mathbb C}(\widehat\theta)\mathscr F(\theta)\geq {\rm Id}_n,
	\]
	it is patent the resemblance of the Cram\'er-Rao lower bound to the uncertainty principle in Quantum Mechanics. \qed	
\end{remark}

\subsection{Asymptotic normality of ML estimators}\label{asym:mle:est}
We now check that under suitable regularity assumptions (which are too complicated to  reproduce here) the ML estimator achieves the Cram\'er-Rao lower bound as the sample size $n$ grows indefinitely, which follows from the fact that consistent ML estimators are asymptotically normal (in the sense of Definition \ref{an:defin}), with their asymptotic covariance $\sigma^2_\theta$ determined by the (inverse of the) Fisher information matrix. As usual we consider an infinite family $X_j\sim\psi_\theta$ of i.i.d. random variables, so that for each $n$ the log-likelihood of $X^{[n]}=(X_1,\cdots,X_n)$ is given by 
\begin{equation}\label{mle:ch}
	l^{(n)}({\bf x};\theta)=\sum_{j=1}^n \ln \psi_\theta(x_j),  
\end{equation}
where $\theta\in\Theta$ is the true (but unknown) parameter. For simplicity, let us assume that $\Theta\subset\mathbb R$ (the uni-dimensional case) so that $\mathscr F(\theta)>0$ is the {Fisher information}. For each $n$ let $\widehat\theta_n$ be the corresponding (and unique!) ML estimator so that 
\begin{equation}\label{mle:cons}
	\frac{d}{d\theta}l^{(n)}({\bf x};\widehat\theta_n)=0.
\end{equation}

\begin{theorem}\label{asym:cr}(Asymptotic normality)
	Under the conditions above, 
	if $\widehat\theta_n$ is {consistent} (in the sense of Definition \ref{cons:deff}) then 
	\begin{equation}\label{asym:stand}
		\sqrt{n}(\widehat\theta_n- \theta)\to\mathcal N(0,\mathscr F_{(1)}(\theta)^{-1})
	\end{equation}
	in distribution (with respect to $\theta$),
	where $\mathscr F_{(1)}$ is the Fisher information of a single observation (say, $X_1$). As a consequence,	
	\begin{equation}\label{asym:mle}
		\widehat\theta_n\approx_{n\to+\infty}\mathcal N(\theta,\mathscr F_{(n)}(\theta)^{-1}),
	\end{equation} 
	where $\mathscr F_{(n)}=n\mathscr F_{(1)}$ is the Fisher information of the whole sample $X^{[n]}$. 
\end{theorem}

\begin{remark}\label{decay:fluc}(Asymptotic efficiency)
	It follows from the asymptotic normality established in
	Theorem~\ref{asym:cr} that, as $n\to+\infty$,
	\[
	n\,\mathrm{var}(\widehat\theta_n)\;\longrightarrow\;
	\mathscr F_{(1)}(\theta)^{-1}.
	\]
	Since the limiting value coincides with the Cram\'er--Rao lower bound
	\eqref{cr:rao:th:3} for the variance of any sufficiently regular (though
	not necessarily asymptotically normal) unbiased estimator of $\theta$
	based on a single observation, this property is usually referred to as
	\emph{asymptotic efficiency}.
	Equivalently, the fluctuations of $\widehat\theta_n$ around $\theta$, as
	measured by its standard deviation, decay at the expected rate $n^{-1/2}$, with
	proportionality constant given by the reciprocal of the square root of
	the Fisher information for a single observation.
	More generally, if $g:\Theta\subset\mathbb R\to\mathbb R$ is a $C^1$
	function with nowhere-vanishing derivative, then (\ref{asym:stand}), together with the delta method (Proposition~\ref{delta:m}), yields
	\[
	\sqrt{n}\bigl(g(\widehat\theta_n)-g(\theta)\bigr)\stackrel{d}{\to}
	\mathcal N\!\left(0,\;|g'(\theta)|^2\,
	\mathscr F_{(1)}(\theta)^{-1}\right).
	\]
	This shows that the asymptotic variance
	of the transformed estimator $g(\widehat\theta_n)$ attains the
	Cram\'er--Rao lower bound~\eqref{cr:rao:th:3} within the class
	$\mathcal E_g$. Hence, $g(\widehat\theta_n)$ is asymptotically efficient
	for estimating $g(\theta)$.
	\qed
\end{remark}

\begin{remark}\label{conf:int:f}
	(Large sample confidence intervals for $\theta$ via the ``consistency trick'')
From (\ref{asym:mle}) we know that 
	\[
	P(a\leq \widehat\theta_n\leq b)\approx_{n\to+\infty}\sqrt{\frac{\mathcal F_{(n)}(\theta)}{2\pi}}\int_a^be^{-\frac{\mathcal F_{(n)}(\theta)(x-\theta)^2}{2}}dx,
	\]
	but we can rely on Theorem \ref{slutsky} to replace $\theta$ by $\widehat\theta_n$ in the right-hand side because $\widehat\theta_n\to\theta$ in probability (consistency), 
	so as to obtain 
	\[
	P(a\leq \widehat\theta_n\leq b)\approx_{n\to+\infty}\sqrt{\frac{\mathcal F_{(n)}(\widehat\theta_n)}{2\pi}}\int_a^be^{-\frac{\mathcal F_{(n)}(\widehat\theta_n)(x-\widehat\theta_n)^2}{2}}dx.
	\]
	The key point here is that the right-hand side depends solely on sample data, and only through the estimator $\widehat\theta$. In 
	the language of confidence intervals of Subsection \ref{conf:int:sub}, this translates into 
	\begin{equation}\label{conf:int:th:up}
		\theta\in \left[\widehat\theta_n\mp\frac{z_{1-\delta/2}}{\sqrt{\mathcal F_{(n)}(\widehat\theta_n)}}\right]\,{\rm with}\,{\rm prob.}\,\approx\,1-\delta,
	\end{equation}
	the ``large sample'' estimate for $\theta$.
	\qed 
\end{remark}

\begin{remark}\label{normal:ww}
	The consistency requirement in Theorem \ref{asym:cr} may be often justified under suitable regularity assumptions on the underlying pdf's, which in particular apply to the ML estimator $\widehat\sigma^2_{n^{-1}}$ in Example \ref{normal:w} \cite[Theorem 2.5 ]{newey1994large}. We may also directly retrieve the consistency of $\widehat\sigma^2_{n^{-1}}$ as follows. First note from (\ref{U:sigma:V:c}) that 
	\begin{equation}\label{cons:sigma}
		\widehat\sigma^2_{n^{-1}}=
		\frac{1}{n}\sum_{j=1}^n\sigma^2\left(\frac{X_j-\mu}{\sigma}\right)^2-
		(\overline X_n-\mu)^2.
	\end{equation}
	Also, recalling that $X_j$ is drawn from a normal population,  $\sigma^{-1}(X_j-\mu)\sim\mathcal N(0,1)$ implies that $\sigma^{-2}(X_j-\mu)^2\sim \chi^2_1$ by Corollary \ref{sum:norm:sq} and hence 
	\[
	\mathbb E\left(\sigma^2\left(\frac{X_j-\mu}{\sigma}\right)^2\right)=\sigma^2
	\]
	by Corollary \ref{chi:sq:ms}. Thus, LLN (Theorem \ref{lln})  applies to ensure that 
	\[
	\frac{1}{n}\sum_{j=1}^n\sigma^2\left(\frac{X_j-\overline X_n}{\sigma}\right)^2\stackrel{p}{\to} \sigma^2.
	\]
	On the other hand, it also follows from LLN
	that $(\overline X_n-\mu)^2\stackrel{p}{\to}0$. Thus, 		
	$\widehat\sigma^2_{n^{-1}}\stackrel{p}{\to} \sigma^2$ by (\ref{cons:sigma}). 
	Another approach to this same conclusion follows by taking $c=n^{-1}$ in (\ref{stat:normal:eq}) to check that ${\rm mse}(\widehat\sigma^2_{n^{-1}})\to 0$ as $n\to +\infty$, so that consistency follows by Proposition \ref{mse:consist}. 
	In fact, any of these methods may be adapted to check that the ML estimators in Examples \ref{mle:exp} and \ref{mle:discrete} above are consistent as well. We also note that, in general, Proposition \ref{an:imp:cons} applies to ensure that consistency is a necessary condition for asymptotic normality. \qed
\end{remark}

\begin{proof} (of Theorem \ref{asym:cr}) Set $\widetilde l^{(n)}=n^{-1}l^{(n)}$ and note that by (\ref{mle:cons}) and the Mean Value Theorem,
	\begin{equation}\label{mean:val}
		0=\frac{d}{d\theta}\widetilde l^{(n)}(\widehat\theta_n)=	\frac{d}{d\theta}\widetilde l^{(n)}(\theta)+	\frac{d^2}{d\theta^2}\widetilde l^{(n)}(\theta^\bullet_n)(\widehat\theta_n-\theta), 
	\end{equation}
	for some $\theta^\bullet_n$ lying between $\widehat\theta_n$ and $\theta$.  
	A computation shows that for any $\theta$ we have
	\[
	\frac{d^2}{d\theta^2}\widetilde l^{(n)}(\theta)=\frac{1}{n}\sum_{j=1}^n\left(\frac{d^2}{d\theta^2}\ln \psi_\theta(X_j)\right)\to \mathbb E\left(\frac{d^2}{d\theta^2}\ln \psi_\theta(X_1)\right),
	\]
	where the convergence is in probability by LLN. Since $\widehat\theta_n\to\theta$ in probability, we conclude that 
	\begin{equation}\label{conv:den}
		\frac{d^2}{d\theta^2}\widetilde l^{(n)}(\theta^\bullet_n)\to \mathbb E\left(\frac{d^2}{d\theta^2}\ln \psi_\theta(X_1)\right)=-\mathscr F_{(1)}(\theta),
	\end{equation}
	where the convergence is in probability and we used (\ref{fischer}) in the last step.
	On the other hand,
	\[
	\sqrt{n}\frac{d}{d\theta}\widetilde l^{(n)}(\theta)=\sqrt{n}\left(\frac{1}{n}\sum_{j=1}^n\frac{d}{d\theta}\ln \psi_\theta(X_j)\right),
	\]
	which may be rewritten 
	as
	\[
	\sqrt{n}\frac{d}{d\theta}\widetilde l^{(n)}(\theta)=\sqrt{n}\left(\frac{1}{n}\sum_{j=1}^n\frac{d}{d\theta}\ln \psi_\theta(X_j)-\mathbb E\left(\frac{d}{d\theta}\ln \psi_\theta(X_1)\right)\right),
	\]
	as the term within the expectation is a score (Corollary \ref{any:fct:2}). Thus we may apply CLT to see that
	\begin{equation}\label{conv:num}
		\sqrt{n}\frac{d}{d\theta}\widetilde l^{(n)}(\theta)\to \mathcal N\left(0,{\mathbb C}\left(\frac{d}{d\theta}\ln \psi_\theta(X_1)\right)\right)=\mathcal N(0,\mathscr F_{(1)}(\theta)), 
	\end{equation}
	where the convergence is in distribution and the last step follows from the definition of $\mathscr F_{(1)}$. 
	Since 
	(\ref{mean:val}) leads to  
	\[
	\sqrt{n}(\widehat\theta_n-\theta)=-\frac{\sqrt{n}\frac{d}{d\theta}\widetilde l^{(n)}(\theta)}{\frac{d^2}{d\theta^2}\widetilde l^{(n)}(\theta^\bullet_n)}, 
	\]
	we may use 
	Theorem \ref{slutsky},  
	(\ref{conv:den}) and (\ref{conv:num}) to complete the proof. 
\end{proof}

\begin{remark}\label{asym:cr:gen}
	Although Theorem \ref{asym:cr} is proved under i.i.d.\!\! sampling for simplicity, the same asymptotic normality holds 
	when the sample $\{X_1,\dots,X_n\}$ consists of independent, though not identically distributed, random variables satisfying the usual Lindeberg--Feller and regularity conditions \cite[Chapter 4]{amemiya1985advanced}. In this case, (\ref{asym:stand}) gets replaced by 
	\[
	\sqrt{n}\,(\widehat\theta_n-\theta)
	\xrightarrow{d}
	\mathcal N\left(0,\overline{\mathscr F}(\theta)^{-1}\right),
	\]
where 
\[
\overline{\mathscr F}(\theta)=\lim_{n\to+\infty}\frac{1}{n}\sum_{j=1}^n\mathscr F^{X_j}(\theta),
\]
so that (\ref{asym:mle}) becomes 
\[
\overline\theta_n\approx \mathcal N\left(\theta,\overline{\mathscr F}_{(n)}(\theta)^{-1}\right), \quad \overline{\mathscr F}_{(n)}(\theta)=n\overline{\mathscr F}(\theta).
\] 
	This more general setting covers, in particular, the regression models considered in Chapter \ref{mls:sub} (see Example \ref{mls:a:n:p}, for instance) and in the discussion of generalized linear models  in Section \ref{exp:glms}, where the responses are independent but their conditional distributions depend on covariates.  \qed
\end{remark}

We now discuss the implications of this theory for some of the statistical models discussed earlier.

\begin{example}\label{norm:pop:an}
	We start with 
	\[
	\widehat\theta=(\widehat\theta_1,\widehat\theta_2)=(\overline X_n,\widehat\sigma^2_{n^{-1}}),
	\]
	the ML estimator  for the bi-dimensional parameter $(\theta_1,\theta_2)=(\mu,\sigma^2)$
	coming from a normal population 
	as in Example \ref{normal:w} above. We first look only at $\widehat\theta_2$, which amounts to declaring that  $\mu$ is known. As observed in Remark \ref{normal:ww}, this estimator  is consistent and hence asymptotically normal by Theorem \ref{asym:cr}. 
	In order to determine the associated limiting normal distribution by means of Theorem \ref{asym:cr} we need to recall the corresponding Fisher information. From (\ref{fish:inf:nor:0}) with $n=1$,
	\begin{equation}\label{fisc:inf}
		\mathscr F_{(1)}(\theta_2)=
		\mathscr F_{(1)}(\theta)_{22}
		=\frac{1}{2\theta_2^2},  \Longrightarrow \mathscr F_{(n)}(\theta_2)=\frac{n}{2\theta_2^2},
	\end{equation}
	so (\ref{asym:stand}) and (\ref{asym:mle}) apply to give 
	\begin{equation}\label{asym:norm:sig}
		\sqrt{n}(\widehat\sigma^2_{n^{-1}}-\sigma^2)\stackrel{d}{\to}\mathcal N(0,2\sigma^4)
	\end{equation}
	and
	\[
	\widehat\sigma^2_{n^{-1}}\approx_{n\to+\infty}\mathcal N(\sigma^2,{2\sigma^4}/{n}).
	\]
	Thus, in view of (\ref{fisc:inf}), (\ref{conf:int:th:up}) translates into
	\begin{equation}\label{est:sig:2}
		\sigma^2\in \left[\left(1-\sqrt{\frac{2}{n}}z_{1-\delta/2}\right)\widehat\sigma^2_{n^{-1}},\left(1+\sqrt{\frac{2}{n}}z_{1-\delta/2}\right)\widehat\sigma^2_{n^{-1}}\right]\,{\rm with}\,{\rm prob.}\,{\rm at}\,{\rm least}\,1-\delta.
	\end{equation}
	We next consider the bi-dimensional case $\theta=(\theta_1,\theta_2)$. Again by (\ref{fish:inf:nor:0}),   
	\[
	{\mathscr F_{(1)}}(\theta_1,\theta_2)=	\left(
	\begin{array}{cc}
		\sigma^2 & 0\\
		0 & 2\sigma^4
	\end{array}
	\right),
	\] 
	which gives 
	\begin{equation}\label{comp:as:cov}
		\sqrt{n}\left(\left(
		\begin{array}{c}
			\widehat\theta_1\\
			\widehat\theta_2	
		\end{array}
		\right)-
		\left(
		\begin{array}{c}
			\mu\\
			\sigma^2	
		\end{array}
		\right)\right) \stackrel{d}{\to}\mathcal N\left(
		\left(
		\begin{array}{c}
			0\\
			0
		\end{array}
		\right),
		\left(
		\begin{array}{cc}
			\sigma^2 & 0\\
			0 & 2\sigma^4
		\end{array}
		\right)
		\right), 
	\end{equation}
	or equivalently, 
	\begin{equation}\label{clt:sigma}
		\left(
		\begin{array}{c}
			\widehat\theta_1\\
			\widehat\theta_2	
		\end{array}
		\right) 
		\approx_{n\to+\infty}
		\mathcal N\left(
		\left(
		\begin{array}{c}
			\mu\\
			\sigma^2
		\end{array}
		\right),
		\left(
		\begin{array}{cc}
			\sigma^2/n & 0\\
			0 & 2\sigma^4/n
		\end{array}
		\right)
		\right). 
	\end{equation}
	A key point here is that the asymptotic covariance matrix in (\ref{comp:as:cov}) only depends on $\theta_2=\sigma^2$, which allows us to proceed as in Remark \ref{conf:int:f}: consistency permits to replace $\sigma^2$ by $\widehat\theta_2$ in the covariance matrix of (\ref{clt:sigma}) to obtain 
	\begin{equation}\label{clt:sigma:cons}
	\left(
	\begin{array}{c}
		\widehat\theta_1\\
		\widehat\theta_2	
	\end{array}
	\right) \approx_{n\to+\infty}
	\mathcal N\left(
	\left(
	\begin{array}{c}
		\mu\\
		\sigma^2
	\end{array}
	\right),
	\left(
	\begin{array}{cc}
		\widehat\theta_2/n & 0\\
		0 & 2\widehat\theta_2^2/n
	\end{array}
	\right)
	\right),
	\end{equation}
	an asymptotic estimate in which the ``fluctuation'' around the center $(\widehat\theta_1,\widehat\theta_2)$ of the ``confidence region'' where the unknown parameter $(\mu,\sigma^2)$ is supposed to be (with probability at least $1-\delta$) depends on sample data, and only through the estimator $\widehat\theta_2=\widehat\sigma^2_{n^{-1}}$. \qed 
\end{example}

\begin{example}\label{binormal:mle}
	Let us refer to the notation and terminology of Examples \ref{corr:dist} and \ref{corr:mle}, with the (simplifying and justifiable) assumption that $\mu_X=\mu_Y=0$, so that $\theta=(\sigma_X^2,\sigma_Y^2,\rho)\in \mathbb R^+\times\mathbb R^+\times (-1,1)$. In order to determine the asymptotic behaviour of the ML estimator $\widehat\theta_m=(\widehat\sigma^2_{m^{-1}}(X),\widehat\sigma^2_{m^{-1}}(Y),\widehat\rho_m)$ for the unknown population parameter $\theta$ as $m\to +\infty$ (for a given jointly normal sample $\{X_j,Y_j\}$) we start with (\ref{exp:loglike:joint}) and, after a somewhat tedious computation, we end up with a complicated expression for the $3\times 3$ matrix $\nabla_{\theta\theta}l(X,Y;\theta)$ whose entries depend {\em linearly} on the symbols in (\ref{coeff:exp}) evaluated on the sample, with the corresponding coefficients being algebraic on the components of $\theta$.  Using that
	\[
	\mathbb E(A(X_j))=\sigma_X^2,\quad \mathbb E(B(X_j,Y_j))=\sigma_{XY}=\rho\sigma_X\sigma_Y, \quad \mathbb E(C(Y_j))=\sigma_Y^2,
	\]
	and 
	(\ref{fischer})
 we conclude that the corresponding Fisher information matrix is 
	\[
	\mathscr F=
	\frac{1}{1-\rho^2}
	\left(
	\begin{array}{ccc}
		\frac{2-\rho^2}{4\sigma_X^4} & -\frac{\rho^2}{4\sigma_X^2\sigma_Y^2}  &  -\frac{\rho}{2\sigma_X^2} \\
		-\frac{\rho^2}{4\sigma_X^2\sigma_Y^2}	&  	\frac{2-\rho^2}{4\sigma_Y^4}  &  -\frac{\rho}{2\sigma_Y^2} \\
		-\frac{\rho}{2\sigma_X^2}	& -\frac{\rho}{2\sigma_Y^2}   &  	\frac{1+\rho^2}{1-\rho^2} 
	\end{array}
	\right),
	\]
	so that 
	\[
	\mathscr F^{-1}=
	\left(
	\begin{array}{ccc}
		2\sigma_X^4	& 2\rho^2\sigma_X^2\sigma_Y^2 & \rho(1-\rho^2)\sigma_X^2 \\
		2\rho^2\sigma_X^2\sigma_Y^2	&  2\sigma_X^4 &  \rho(1-\rho^2)\sigma_Y^2\\
		\rho(1-\rho^2)\sigma_X^2	& \rho(1-\rho^2)\sigma_Y^2 & (1-\rho^2)^2
	\end{array} 
	\right).
	\]
	From this and Theorem \ref{asym:cr} we thus derive not only that
	\[
	\sqrt{m}\left(\widehat\sigma^2_{m^{-1}}(X)-\sigma_X^2\right)\stackrel{d}{\to}\mathcal N(0,2\sigma_X^2), \quad  \sqrt{m}\left(\widehat\sigma^2_{m^{-1}}(Y)-\sigma_Y^2\right)\stackrel{d}{\to}\mathcal N(0,2\sigma_Y^2), 
	\]
	which are fully compatible with (\ref{asym:norm:sig}), but also that 
	\begin{equation}\label{z-rho-fisher}
		\sqrt{m}\left(\widehat\rho_m-\rho\right)\stackrel{d}{\to}\mathcal N(0,(1-\rho^2)^2),
	\end{equation}
	which identifies the asymptotic variance of the sample correlation coefficient $\widehat\rho_m$ as being $(1-\rho^2)^2$; cf. Definition \ref{an:defin}. Moreover, Remark \ref{decay:fluc} guarantees that this asymptotic invariance equals the Cr\'amer-Rao lower bound (\ref{cr:rao:th:3}) for the variances of all {\em unbiased} estimators for $\rho$.
	We remark, however, that $\widehat\rho_m$ itself is {\em not} unbiased even though there holds  
	\[
	F(\widehat\rho_m)-\widehat\rho_m\stackrel{p}{\to} 0
	\]
	for any {\em unbiased} estimator of the form $F(\widehat\rho_m)$, where $F$ is assumed to be odd \cite{olkin1958unbiased}. In particular, 
	\[
	\sqrt{m}\left(F(\widehat\rho_m)-\rho\right)\stackrel{d}{\to}\mathcal N(0,(1-\rho^2)^2). 
	\]
	As in Example \ref{norm:pop:an} above, we may combine (\ref{z-rho-fisher}) with the consistency of $\widehat\rho_m$ to construct a large sample confidence interval for the unknown correlation coefficient $\rho$, 
	namely, 
	\[
	\rho\in\left[\widehat\rho_m\mp z_{1-\delta/2} \frac{1-\widehat\rho_m^2}{\sqrt{m}}\right]\,{\rm with}\,{\rm prob.}\,\approx\,1-\delta
	\]
	whose  
	``fluctuation'' around its center $\widehat\rho_m$ depends only on sample data, and through the estimator $\widehat\rho_m$. 
	An alternate route is to apply 
	the delta method 
	(Proposition 
	\ref{delta:m}) 
	with the {\em Fisher} $z$-{\em transformation}
	\[
	z=g(\rho)=\frac{1}{2}\ln\left(\frac{1+\rho}{1-\rho}\right)
	\]
	to (\ref{z-rho-fisher}) so as to get 
	\[
	\sqrt{m}\left(\widehat z_m-z\right)\stackrel{d}{\to}\mathcal N(0,1),
	\]
	which allows us to obtain 
	large sample estimates for $z$ in terms of the familiar normal quantiles $z_{1-\delta/2}$ and then transform them back to corresponding estimates for $\rho=\tanh z$; see \cite[Section 14.18]{kendall1946advanced} and \cite[Subsection 4.2.3]{anderson2003introduction} \qed
\end{example}

\begin{example}\label{est:coef:var:n} (The coefficient of variation of a normal population)
	Let $g:\Theta\to \mathbb R$ be any smooth function
	satisfying $\nabla g\neq \vec{0}$ everywhere, where $\Theta=\mathbb R\times\mathbb R^+$ is the parameter space of a normal population as above; cf. Example \ref{normal:w}. Recall that $\theta=(\theta_1,\theta_2)=(\mu,\sigma^2)$ in this case. It then follows from (\ref{comp:as:cov}) and the multi-dimensional 
	version of the delta method in Proposition \ref{delta:m} that 
	\begin{equation}\label{comp:as:cov:2}
		\sqrt{n}\left(g\left(
		\begin{array}{c}
			\widehat\theta_1\\
			\widehat\theta_2	
		\end{array}
		\right)-
		g\left(
		\begin{array}{c}
			\mu\\
			\sigma^2	
		\end{array}
		\right)\right) \stackrel{d}{\to}\mathcal N\left(
		0,
		\nabla g(\theta)^\top \left(
		\begin{array}{cc}
			\sigma^2 & 0\\
			0 & 2\sigma^4
		\end{array}
		\right)\nabla g(\theta)
		\right).
	\end{equation}
	Assuming that $\mu\neq 0$, we may apply this to $g(\theta)=\sqrt{\theta_2}/\theta_1=\sigma/\mu$, the {\em coefficient of variation}; cf. (\ref{cv:logn:pop}). Since
	\[
	\nabla g\left(\mu,\sigma^2\right)^\top=\left(-\frac{\sigma}{\mu^2},\frac{1}{2\mu\sigma}\right)
	\]
	we obtain that the corresponding estimator, $\widehat\sigma_{n^{-1}}/\overline X_n$, is asymptotically normal,
	\[
	\sqrt{n}\left(\frac{\widehat\sigma_{n^{-1}}}{\overline X_n}-\frac{\sigma}{\mu}\right)\stackrel{d}{\to}\mathcal N\left(0,\frac{\sigma^2}{\mu^2}\left(\frac{1}{2}+\frac{\sigma^2}{\mu^2}\right)\right),
	\]
	with the asymptotic variance depending on $\sigma/\mu$ itself. 
	As usual, we may use consistency to get
	\[
	\frac{\sigma}{\mu}\in \left[\frac{\widehat\sigma_{n^{-1}}}{\overline X_n}\mp
	z_{1-\delta/2} \frac{\widehat\sigma_{n^{-1}}}{\sqrt{n}\overline X_n}\left(\frac{1}{2}+\frac{\widehat\sigma^2_{n^{-1}}}{\overline X_n^2}\right)^{1/2}\right]\,{\rm with}\,{\rm prob.}\,\approx\,1-\delta,
	\]
	a large sample confidence interval estimate for the coefficient of variation.
	\qed
\end{example}

\begin{example}\label{ex:mle:gamma} (MLE for a Gamma population)
	It follows from (\ref{gamma:dist:1}) that the log-likelihood function of a Gamma distribution $\mathsf{Gamma}({\alpha,\lambda})$ is 
\begin{equation}\label{eq:loglike:gam}
l({\bf x};\theta)=n\left(\lambda\ln\alpha-\ln\Gamma(\lambda)+(\lambda-1)\overline{\ln x}-\alpha\overline{x}\right), \quad \theta=(\alpha,\lambda),
\end{equation}	
where $\overline{\ln x}$ is the arithmetic mean of $\{\ln x_1,\cdots,\ln x_n\}$; recall that $x_j>0$ for each $j$. Hence, the score vector is
\begin{equation}\label{eq:scor:gamma}
s({\bf x};\theta)=n
\left(
\begin{array}{c}
\frac{\lambda}{\alpha}-\overline x\\
\ln\alpha - \psi(\lambda)+\overline{\ln x}, 	
	\end{array}
\right),\quad \psi(\lambda)=\frac{d}{d\lambda}\ln\Gamma(\lambda), 
\end{equation}
so the ML estimator $\widehat\theta=(\widehat\alpha,\widehat\gamma)$ satisfies
\begin{equation}\label{eq:mle:gamma}
\left\{
\begin{array}{rcl}
\frac{\widehat\lambda}{\widehat\alpha}& = & \overline x\\
 \psi(\widehat\lambda)-\ln\widehat\alpha & = & \overline{\ln x}	
	\end{array}
\right.
\end{equation}
Note that trying to find a solution for this system in closed form is out of question so possible strategies here are:
\begin{itemize}
	\item to use our favorite optimization package to find
	\[
	\widehat\theta={\rm argmax}_\theta l({\bf x};\theta)
	\]
starting from (\ref{eq:loglike:gam});
\item  to use the first equation  in (\ref{eq:mle:gamma}) to eliminate $\widehat\alpha$ in the second equation,  {\em numerically} solve for $\widehat{\lambda}$ in the resulting equation, namely,
\begin{equation}\label{uniq:sol:eq}
\psi(\widehat\lambda)-\ln\widehat\lambda =\overline{\ln x}- \ln\overline x,
\end{equation}
and then replacing the result back  in the first equation in order to get $\widehat\alpha$\footnote{That a unique solution $\widehat\lambda$ to (\ref{uniq:sol:eq}) exists for any given ${x}$ is a consequence of the facts that i) $\overline{\ln x} < \ln \overline x$ if each $x_j>0$; ii)  the function $\lambda\mapsto \Psi(\lambda)=\psi(\lambda)-\ln\lambda $ is monotone continuous and satisfies 
	\[
	\lim_{\lambda\to 0}\Psi(\lambda)=-\infty \,\,{\rm and}\,\,
	\lim_{\lambda\to-\infty}\Psi(\lambda)=0.
	\]}. 	
\end{itemize}
In any case, 
with the ML estimator so determined, we may proceed to compute the associated Fisher information matrix by means of (\ref{eq:scor:gamma}) and (\ref{fischer}): 
\[
\mathscr F_{(n)}(\theta)=n
\left(
\begin{array}{cc}
\lambda/\alpha^2	& -1/\alpha\\
-1/\alpha & \psi_1(\lambda)
	\end{array}
\right)=n\mathscr F_{(1)}(\theta), \quad \psi_1=d\psi/d\lambda.
\]
It is not hard to check that
 $\det \mathscr F_{(1)}(\theta)=(\lambda\psi_1(\lambda)-1)/\alpha^2>0$, so  Theorem \ref{asym:cr} gives asymptotic normality for $\widehat\theta_n$:
\[
\sqrt{n}
\left(
\left(
\begin{array}{c}
	\widehat\alpha_n\\
	\widehat\lambda_n
	\end{array}
\right)
-
\left(
\begin{array}{c}
	\alpha\\
	\lambda
\end{array}	
\right)
\right)
\stackrel{d}{\to}
\mathcal N\left(\vec{0},\mathscr F_{(1)}^{-1}\right)
=
\mathcal N
\left(
\left(
\begin{array}{c}
	0\\
	0
\end{array}
\right),
\frac{1}{\lambda \psi_1(\lambda)-1}
\left(
\begin{array}{cc}
	\alpha^2\psi_1(\lambda) & \alpha\\
	\alpha & \lambda
	\end{array}
\right)
\right).
\]
As usual, we may combine this with consistency in order to obtain
\begin{equation}\label{as:norm:gamma}
\left(
\begin{array}{c}
	\widehat\alpha_n\\
	\widehat\lambda_n
\end{array}
\right)\approx_{n\to +\infty}
\mathcal N\left(
\left(
\begin{array}{c}
	\alpha\\
	\lambda
	\end{array}
\right)
,\mathscr F_{(n)}^{-1}\right)
=
\mathcal N
\left(
\left(
\begin{array}{c}
	\alpha\\
	\lambda
\end{array}
\right),
\frac{1}{n(\widehat\lambda_n \psi_1(\widehat\lambda_n)-1)}
\left(
\begin{array}{cc}
	\widehat\alpha_n^2\psi_1(\widehat\lambda_n) & \widehat\alpha_n\\
	\widehat\alpha_n & \widehat\lambda_n
\end{array}
\right)
\right),
\end{equation}
which may be used to construct not only large sample confidence intervals for $\alpha$ and $\lambda$ (separately) but also large sample confidence regions for the whole vector parameter $\theta$; see Remark \ref{rem:conf:r:ml} below for this latter kind of construction. 
	\end{example}
	
\begin{remark}\label{rem:conf:r:ml} (Confidence region for the unknown parameter $\theta$ via asymptotic normality of the ML estimator)
Starting with the (possibly multivariate) version of (\ref{asym:mle}), where $\theta\in\mathbb R^p$, $p\geq 1$, and $\mathscr F_{(n)}(\theta)$ is a $p\times p$ symmetric, positive definite matrix, consistency of the ML estimator $\widehat\theta_n$ leads to the asymptotic normality relation
\[
\widehat\theta_n\approx_{n\to+\infty}\mathcal N(\theta,\mathscr F_{(n)}(\widehat\theta_n)^{-1}), 
\]	
from which (\ref{as:norm:gamma}) is a rather special case (with $p=2$). 
In order to extract from this a confidence region for the unknown vector parameter $\theta$, let us write $\mathscr F_{(n)}(\widehat\theta_n)=A^\top A$ so that Corollary \ref{ortho:norm} gives
\[
A(\widehat\theta_n-\theta)\sim\mathcal N(0,{\rm Id}_p)
\]
and hence 
\[
(\widehat\theta_n-\theta)^\top\mathscr F_{(n)}(\widehat\theta_n)(\widehat\theta_n-\theta)
=\|A(\widehat\theta_n-\theta)\|^2\sim\chi^2_p.
\]
In other words, the quadratic form in the left-hand side is a pivotal quantity to which the standard method may be applied: if we recall the definition of the $\chi^2$-quantile in (\ref{quantile:chi})  then 
\[
P\left(
(\widehat\theta_n-\theta)^\top\mathscr F_{(n)}(\widehat\theta_n)(\widehat\theta_n-\theta)\leq \chi^2_{p,1-\alpha}\right)\approx 1-\alpha. 
\]
Since $\mathscr F_{(n)}(\widehat\theta_n)$ is positive definite, the random confidence region where $\theta$ is supposed to lie (within the given confidence level) is ellipsoidal in nature, with its size, shape, and orientation being completely determined by the totality of the elements of $\mathscr F_{(n)}(\widehat\theta_n)$. Moreover, since its construction takes into account the possible correlations among the various components of $\widehat\theta_n$, as encoded in the off-diagonal elements of the asymptotic covariance matrix $\mathscr F_{(n)}(\widehat\theta_n)^{-1}$, in such cases it certainly encloses a much tighter volume than the $p$-cube which is the product of the separate confidence intervals for the entries of $\theta$. 
	\end{remark}

\begin{remark}\label{rem:aic}
(Akaike Information Criterion)
The information-theoretic motivation behind maximum likelihood estimation discussed in 
Remark~\ref{kl:int:lrt:0} leads to an efficient way of measuring the goodness of fit of the 
model $\psi_{\widehat\theta}$, where $\widehat\theta=\widehat\theta(X)$ denotes the ML estimator 
associated with the statistical model
\[
X_1,\ldots,X_n \sim \psi_\theta, \qquad \theta\in\Theta\subset\mathbb R^p .
\]
More precisely, we wish to assess how effective $\psi(Z;\widehat\theta)$ is at predicting 
$\psi(Z;\theta_0)$, where $\theta_0$ is the unknown true parameter and 
$Z\sim\psi_{\theta_0}$ is independent of the sample $X$. The heuristic discussion leading to Definition~\ref{likeli:def}, which emphasizes the minimization of the Kullback-Leibler divergence,  suggests that the  {\em expected log-likelihood},
\[
m(\theta)
:= n\,\mathbb E_Z\!\left(\ln \psi(Z;\theta)\right)=c_n-nD^{KL}_{\theta_0}(\theta),
\]
is the  quantity to be maximized. This naturally leads to considering  $m(\widehat\theta)$ as a measure of goodness of fit. Since this quantity still depends on the sample $X$, we pass to
the {\em mean expected log-likelihood} $\mathbb E_X(m(\widehat\theta))$, 
which should be viewed as an unknown quantity to be estimated by 
the purely data-driven {\em maximized log-likelihood} $l(\widehat\theta)$, where $l(\theta):=l(X;\theta)$
denotes the log-likelihood evaluated at the sample.
It turns out, however, that this estimation carries a bias. 
A neat geometric interpretation for this bias is available (see Remark \ref{kl:fisher:aic}), but for present purposes it is convenient to make it explicit by writing
\begin{eqnarray*}
	\mathrm{bias}
	&:= & 
	\mathbb E_X\!\left(l(\widehat\theta)\right)-\mathbb E_X\left(m(\widehat\theta)\right)\\
	&= & \left(\mathbb E_X\!\left(l(\widehat\theta)\right)-m(\theta_0)\right)
	- \left(\mathbb E_X\!\left(m(\widehat\theta)\right)-m(\theta_0)\right),
\end{eqnarray*}
which shows that it arises from the fact that
both  $m(\widehat\theta)$ and $l(\widehat\theta)$  are biased as 
 estimators of $m(\theta_0)$. Although the exact biases are generally out of reach for 
fixed $n$, they may be determined in the asymptotic regime with relatively little effort. 
To this end, we first expand $m(\theta)$ around the true value $\theta=\theta_0$, 
obtaining
\[
\begin{aligned}
	m(\theta)
	&\approx m(\theta_0)
	+ n\,\mathbb E_Z\!\left(\nabla_\theta \ln\psi(Z;\theta_0)\right)^{\!\top}
	(\theta-\theta_0) \\
	&\quad + \frac{1}{2}n(\theta-\theta_0)^\top
	\mathbb E_Z\!\left(\nabla_{\theta\theta}\ln \psi(Z;\theta_0)\right)
	(\theta-\theta_0) \\
	&= m(\theta_0)
	-\frac{1}{2}(\theta-\theta_0)^\top
	\mathscr F_{(n)}(\theta_0)
	(\theta-\theta_0),
\end{aligned}
\]
where we used Corollary~\ref{any:fct:2} in the last step.
Since $\widehat\theta\stackrel{p}{\to}\theta_0$ by consistency, we may take 
$\theta=\widehat\theta$ for $n$ sufficiently large. Arguing as in 
Remark~\ref{rem:conf:r:ml}, we then have
\[
(\widehat\theta-\theta_0)^\top
\mathscr F_{(n)}(\theta_0)
(\widehat\theta-\theta_0)
\;\xrightarrow{d}\; \chi^2_p,
\]
so that, upon taking expectations,
\begin{equation}\label{eq:aic:1}
	\mathbb E_X\!\left(m(\widehat\theta)\right)
	\approx m(\theta_0)-\frac{p}{2},
	\end{equation}
which indicates a bias loss of $p/2$ when using $m(\widehat\theta)$	to estimate $m(\theta_0)$. 
On the other hand, expanding $l(\theta)=l(X;\theta)$ around $\theta=\widehat\theta$ yields
\[
\begin{aligned}
	l(\theta)
	&\approx l(\widehat\theta)
	+ \nabla_\theta l(\widehat\theta)^\top(\theta-\widehat\theta) \\
	&\quad + \frac{1}{2}(\theta-\widehat\theta)^\top
	\nabla_{\theta\theta}l(\widehat\theta)
	(\theta-\widehat\theta) \\
	&= l(\widehat\theta)
	+ \frac{1}{2}(\theta-\widehat\theta)^\top
	\nabla_{\theta\theta}l(\widehat\theta)
	(\theta-\widehat\theta),
\end{aligned}
\]
since the first-order term vanishes by the defining property of $\widehat\theta$ as a 
maximizer of $l$.
By LLN,
\[
\frac{1}{n}\nabla_{\theta\theta}l(\theta_0)
= \frac{1}{n}\sum_{j=1}^n \nabla_{\theta\theta}\ln \psi(X_j;\theta_0)
\;\stackrel{p}{\longrightarrow}\;
-\mathscr F_{(1)}(\theta_0),
\]
and consistency therefore implies
\[
\nabla_{\theta\theta}l(\widehat\theta)
\;\stackrel{p}{\longrightarrow}\;
-\mathscr F_{(n)}(\theta_0).
\]
Consequently,
\[
l(\theta)
\approx l(\widehat\theta)
-\frac{1}{2}(\theta-\widehat\theta)^\top
\mathscr F_{(n)}(\theta_0)
(\theta-\widehat\theta).
\]
Taking $\theta=\theta_0$ and arguing as above, we obtain
\begin{equation}\label{eq:aic:2}
\mathbb E_X\!\left(l(\widehat\theta)\right)	
	=
\mathbb E_X\!\left(l(\theta_0)\right) +\frac{p}{2}.
\end{equation}
However,
\[
\mathbb E_X\!\left(l(\theta_0)\right)
= \mathbb E_X\!\left(\sum_{j=1}^n\ln\psi(X_j;\theta_0)\right)
= n\,\mathbb E_Z\!\left(\ln\psi(Z;\theta_0)\right)
= m(\theta_0),
\]
and substituting this into~\eqref{eq:aic:2} yields
\begin{equation}\label{eq:aic:3}
	\mathbb E_X\!\left(l(\widehat\theta)\right)
	\approx m(\theta_0)+\frac{p}{2},
\end{equation}
which indicates a bias gain of $p/2$ when using $l(\widehat\theta)$ to estimate $m(\theta_0)$. 
Combining~\eqref{eq:aic:1} and~\eqref{eq:aic:3}, we conclude that
\[
\mathbb E_X\!\left(l(\widehat\theta)\right)
\approx \mathbb E_X\!\left(m(\widehat\theta)\right) + p,
\]
so that $l(\widehat\theta)$, as an estimator of $\mathbb E_X\!(m(\widehat\theta))$, has asymptotic bias equal to $p$, the number of free parameters in the model. 
After an appropriate normalization, this leads to the definition of the 
\emph{Akaike information criterion}:
\begin{equation}\label{aic:form:f}
{\rm AIC}(X) = -2\,l(X;\widehat\theta) + 2p.
\end{equation}
As an illustration, consider the normal model of Example~\ref{normal:w}. In this case,
\[
\begin{aligned}
	{\rm AIC}(X)
	&= -2\left(
	-\frac{n}{2}\ln(2\pi\widehat\sigma^2_{n^{-1}})
	-\frac{1}{2\widehat\sigma^2_{n^{-1}}}
	\sum_{j=1}^n (X_j-\overline X_n)^2
	\right) + 2\cdot 2 \\
	&= -2\left(
	-\frac{n}{2}\ln(2\pi\widehat\sigma^2_{n^{-1}})
	-\frac{n}{2}
	\right) + 4,
\end{aligned}
\]
so that
\[
{\rm AIC}(X)
= n\ln(2\pi\widehat\sigma^2_{n^{-1}}) + n + 4.
\]
Thus, according to this criterion, the goodness of fit of the model is entirely determined by $\widehat\sigma^2_{n^{-1}}$, the ML estimator of the variance, which is consistent with the computation of the asymptotic covariance matrix of the ML estimator of the model parameter $(\theta_1,\theta_2)=(\mu,\sigma^2)$ in  (\ref{clt:sigma:cons}). 
\qed
	\end{remark}

\begin{remark}\label{kl:fisher:aic}
	(Kullback--Leibler divergence, Fisher information, and AIC)
For any sufficiently regular parametric model with true parameter
	$\theta_0$, the Kullback--Leibler 
	divergence from Definition \ref{kull:leib:div:d} admits the local expansion, as $\theta\to\theta_0$,
	\[
	D^{KL}_{\theta_0}(\theta)
	=
	\frac12(\theta-\theta_0)^\top \mathscr F_{(1)}(\theta_0)(\theta-\theta_0)
	+
	o\big(\|\theta-\theta_0\|^2\big),
	\]
	where $\mathscr F_{(1)}(\theta_0)$ denotes the Fisher information matrix associated
	with a single observation $Y\sim\psi_\theta$. Indeed,  we have
	\[
	\left(\nabla_\theta D^{KL}_{\theta_0}\right)(\theta)=-\mathbb E_{\theta_0}\left(\nabla_\theta\ln\psi(Y;\theta)\right),\quad 
	\left(\nabla_{\theta\theta} D^{KL}_{\theta_0}\right)(\theta)=-\mathbb E_{\theta_0}\left(\nabla_{\theta\theta}\ln\psi(Y;\theta)\right),
	\]
	so that Corollary \ref{any:fct:2} gives
	\[
	\left(\nabla_\theta D^{KL}_{\theta_0}\right)(\theta_0)=0,\quad 
	\left(\nabla_{\theta\theta} D^{KL}_{\theta_0}\right)(\theta_0)=\mathscr F_{(1)}(\theta_0), 
	\]
	as claimed; for a full treatment of this local identification of the Fisher information matrix with the Hessian of the Kullback–Leibler divergence we refer to \cite[Chapter 4]{calin2014geometric}.
	On the other hand, if 
	$\widehat\theta$ is the corresponding  MLE based on an i.i.d.\
	sample of size $n$, we may argue as in Remark \ref{rem:conf:r:ml} to see that 
	\[
	n(\widehat\theta-\theta_0)^\top \mathscr F_{(1)}(\theta_0)(\widehat\theta-\theta_0) \stackrel{d}{\to} \chi^2_p,
	\]
	which implies
	\[
	\mathbb E\!\left(
	(\widehat\theta-\theta_0)^\top
	\mathscr F_{(1)}(\theta_0)
	(\widehat\theta-\theta_0)
	\right)
	=
	\frac{p}{n}
	+
	o\!\left(\frac1n\right)
	\]
	and
	therefore
	\[
	\mathbb E_{{\theta_0}}\!\left(2n
	D_{\theta_0}^{KL}(\widehat\theta)
	\right)
	=
	{p}
	+
	o\!\left(1\right).
	\]
	Now recall from Remark \ref{rem:aic} that Akaike's Information Criterion (AIC) is motivated by the problem of selecting,
	among competing parametric models, the one whose fitted distribution is closest
	to the true data--generating distribution in the sense that the  Kullback--Leibler divergence $D_{\theta_0}^{KL}(\widehat\theta)$ should be minimized.
	In this way, its final expression (\ref{aic:form:f})
	may be
	interpreted as an approximately unbiased estimator of twice the expected
	Kullback--Leibler divergence between the fitted model and the true distribution,
	with the computation above indicating that the penalty term reflects the local curvature of the KL divergence as
	governed by Fisher information.
	\qed
\end{remark}		

\begin{remark}\label{taken:tho}
	Comparing Theorem \ref{asym:cr} with Theorem \ref{clt}, one observes that the ``sample universality'' so valued in the latter classical result is irretrievably lost. In essence, Theorem \ref{asym:cr} states that for each choice of log-likelihood function as in (\ref{mle:ch}), itself fully determined by the underlying density $\psi(\cdot,\theta)$ through (\ref{like:exp:ind0}) and (\ref{like:exp:ind}), the maximum likelihood method produces an estimator of the form (\ref{est:mle:det}) to which a corresponding ``limit theorem'' applies, as in (\ref{asym:stand}). Unlike the Central Limit Theorem, therefore, Theorem \ref{asym:cr} is inherently model-dependent. \qed
\end{remark}

\section{The method of least squares}\label{mls:sub}
If $\theta_2=\sigma^2$ is known, maximizing $l$ in (\ref{log:like}) is equivalent to minimizing 
\[
\theta_1\mapsto\frac{1}{2}\sum_j(x_j-\theta_1)^2,
\]
which 
furnishes a variational characterization of the arithmetic mean $n^{-1}\sum_jx_j$. This is of course a manifestation of 
the Method of Least Squares, a celebrated procedure which provides a solution to  the following kind of problem. Let us arrange the outcome of $n$  measurements of $p$ features (regressors, independent/explanatory variables, predictors, covariates, etc.) of a population by means of the  $n\times (p+1)$-matrix
\[
{\mathfrak x}=
\left(	
\begin{array}{c:c}
	& {\bf x}_1 \\
	{\bf 1}\,\, & \vdots\\
	& {\bf x}_n 
\end{array}
\right)	
\]
where each 
\[
{\bf x}_j=(x_{j1},\cdots, x_{jp}), \quad j=1,\cdots,n,
\]
is a row $p$-vector (representing the outcome of the $j^{\rm th}$ measurement) and ${\bf 1}$ is the column $n$-vector whose entries all equal $1$. If we suspect that these features relate to a response (regressand, dependent/explained variable, etc.) which has also been measured, thus yielding an $n$-vector ${\bf y}$, we may try to ``predict'' the response at some unknown feature by ``best fitting'' 
a (possibly non-linear) functional dependence, say
${\bf y}=F({\mathfrak x})$,
to the {available data} $({\mathfrak x},{\bf y})$\footnote{In the modern language of Supervised Learning, the pair $({\mathfrak x},{\bf y})$ is referred to as the {\em training data set}, since it is used to instruct the learning algorithm on how to fit the model $F$ \cite{hastie2009elements}.}. The simplest choice is to postulate that $F$ is {\em linear}, so that 
${\bf y}={\mathfrak x}\widehat \beta$, where
$\widehat{\beta}=(\widehat\beta_0,\widehat\beta_1,\cdots,\widehat\beta_p)$ is determined by minimizing the corresponding quadratic objective function:  
\begin{equation}\label{min:sq}
	\widehat\beta={\rm argmin}_\beta f(\beta), \quad 	f(\beta)=\frac{1}{2}\|{\bf y}-{{\mathfrak x}}\beta\|^2, 
\end{equation}
hence the ``least squares'' terminology.
 
This is a purely geometric problem (best fitting a hyperplane to a cloud of $n$ points in $\mathbb R^p\times\mathbb R=\mathbb R^{p+1}$, where we usually assume that $p+1\ll n$), which can be solved by the methods of Calculus. 
Indeed, since
\[
(\nabla f)(\beta)=-{{\mathfrak x}}^\top\left({\bf y}-{{\mathfrak x}}\beta\right),
\]
where $\top$ means transpose, 			
we obtain
\begin{equation}\label{mls:form:mul}
	\widehat\beta = ({{\mathfrak x}}^\top{{\mathfrak x}})^{-1}{{\mathfrak x}}^\top{\bf y}, 
\end{equation}		
where we assume that ${\mathfrak x}$ has full column-rank (this not only implies that $p+1\leq n$ but also that the {\em Gram matrix} ${{\mathfrak x}}^\top{{\mathfrak x}}$ is symmetric and positive definite, hence invertible). Moreover, since
\[
\nabla^2f={{\mathfrak x}}^\top{{\mathfrak x}},
\] 
we conclude that $\widehat\beta$ is the unique global minimum. Under these conditions, we then say that 
\begin{equation}\label{fit:vec}
	\widehat{\bf y}=\mathfrak x\widehat\beta
\end{equation}	
is the {\em fitted vector} (that is, the vector of fitted values). 
Finally, we observe that $\widehat\beta$ is {\em linear} in ${\bf y}$ with coefficients depending on $\mathfrak x$.

\subsection{The statistical model behind the method of least squares}\label{subsec:stat:mls}
In the examples below, we examine the statistical rationale underlying the purely data-driven minimization problem in (\ref{min:sq}). In line with the general estimation framework presented in Subsection \ref{param:est}, this requires introducing suitable assumptions on how the data array $({\bf y},\mathfrak x)$ is drawn from an underlying population. These assumptions allow us to set up a statistical model in which $\beta$ is treated as an unknown population parameter and $\widehat\beta$ is validated as an efficient estimator of $\beta$\footnote{For a critical discussion of model building and interpretation in Regression Analysis, with emphasis on the distinction between the information inherent in the data and the inferential consequences of assumptions about the sampling process that generated them, a perspective that extends to virtually any statistical analysis, see \cite{berk2004regression}.}.

\begin{example}\label{mle:imp:lsm:0}(The general regression model)
	We start by viewing  $(\mathfrak x,\bf y)$ as the realization of a $\mathbb R^{n\times(p+1)+n}$-valued random vector $(\mathfrak X,{\bf Y})$ in $L^2(\Omega)$ and satisfying
	\begin{itemize}
		\item 
		$\mathfrak X=({\bf 1}\,\,{\bf X})$, where ${\bf X}$ is a random $\mathbb R^{n\times p}$-valued vector whose realization is ${\bf x}=({\bf x}_1,\cdots,{\bf x}_n)^\top$. In other words, $\mathfrak X$ is a random matrix whose first column is deterministic (non-random) and equals ${\bf 1}$. Also, we assume that $\mathfrak X$ has full column-rank a.s.
		\item 	$\{({\bf X}_{j\bullet},{\bf Y}_{j})\}_{j=1}^n$, where $\bullet\in\{1,\cdots,p\}$, are i.i.d.\!\! copies of the same $\mathbb R^{p}\times\mathbb R$-valued random vector, say $(\mathscr X,\mathscr Y)$. 
	\end{itemize}
	Under these conditions,
	we  may initially impose the (not necessarily linear) {\em regression model}
	\begin{equation}\label{reg:model:1}
		{\bf Y}=F(\mathfrak X)+{\bf e}, 
	\end{equation}
	with the {\em regression function} $F:\mathbb R^{n(p+1)}\to\mathbb R^n$ being  defined by 
	\begin{equation}\label{reg:model:2}
		F(\mathfrak x)=\mathbb E\left({\bf Y}|_{\mathfrak X=\mathfrak x}\right),
	\end{equation}
	where we use here the notation of Subsection \ref{cond:prob}; see Remark \ref{reg:func:choice} below 
	for the 
	justification of this choice of $F$, where it is shown that it minimizes the corresponding mean squared error:
	\[
	\mathbb E(\|{\bf Y}-F(\mathfrak X)\|^2)
	\leq \mathbb E(\|{\bf Y}-G(\mathfrak X)\|^2),
	\]
	for any $G:\mathbb R^{n(p+1)}\to\mathbb R^n$. 
	Thus, 
	we may view (\ref{reg:model:1}) as the definition of the random {\em error} ${\bf e}$, which  by Proposition \ref{int:prob}
	may also be expressed as 
	\[
	{\bf e}={\bf Y}-\mathbb E({\bf Y}|\mathfrak X),
	\]
	so that
	(\ref{total:exp}) easily implies {\em exogeneity},
	\begin{equation}\label{exo:p}
		\mathbb E({\bf e}|{\mathfrak X})={0},
	\end{equation}
	and hence	
	\begin{equation}\label{zero:mean:e}
		\mathbb E({\bf e})={0}.
	\end{equation}
	More generally, again by (\ref{total:exp}),
	\[
	\mathbb E(\mathfrak X^\top{\bf e})=\mathbb E(\mathbb E(\mathfrak X^\top{\bf e}|\mathfrak X))=\mathbb E(\mathfrak X^\top\mathbb E({\bf e}|\mathfrak X)),
	\]
	where we used 
	Proposition \ref{ceprop} (7) in the last step, so that  (\ref{exo:p}) applies to give
	\begin{equation}\label{uncorr:e}
		\mathbb E(\mathfrak X^\top{\bf e})=0,	
	\end{equation}
	so (\ref{reg:model:1}) decomposes ${\bf Y}$ as a sum of a term $F(\mathfrak X)$ which is ``explained'' by $\mathfrak X$ and an error which has zero mean (conditioned to $\mathfrak X$) and is uncorrelated to (any function of) $\mathfrak X$\footnote{The implications of this remarkable decomposition to Regression Theory (and to Econometrics, in particular) are discussed at length in \cite[Chapter 3]{angrist2009mostly}.}. Also, 
	if we define the {\em error covariance function} by
	\[
	{\bm\sigma}^2(\mathfrak x):={\mathbb C}({\bf e}|_{\mathfrak X=\mathfrak x})\stackrel{(\ref{exo:p})}{=}\mathbb E({\bf e}\otimes{\bf e}|_{\mathfrak X=\mathfrak x}),\quad \mathfrak r\in\mathbb R^{n(p+1)}, 
	\] 
	then
	\[
	{\mathbb C}({\bf e})\stackrel{(\ref{zero:mean:e})}{=}\mathbb E({\bf e}\otimes{\bf e})\stackrel{(\ref{total:exp})}{=}\mathbb E(\mathbb E({\bf e}\otimes{\bf e}|\mathfrak X)),
	\]
	and using Proposition \ref{int:prob},
	\[
	{\mathbb C}({\bf e})=\mathbb E({\bm\sigma}^2(\mathfrak X)).
	\] 
	In other words, the unconditioned error covariance equals the expected value of the conditioned error covariance. \qed
\end{example}

We now specialize the general setup above to the cases which appear more frequently in applications. 

\begin{example}\label{mle:imp:lsm}
	(The linear regression model)
	The simplest of all choices for the regression function above is $F({\mathfrak x})={\mathfrak x}\beta$, 
	which gives rise to 
	the {\em linear regression model}
	\begin{equation}\label{lin:reg:mod}
		{\bf Y}={\mathfrak X}\beta+{\bf e},
	\end{equation}
	where 		
	\begin{eqnarray}
		\beta
		& = & 
		{\rm argmin}_{\beta'\in\mathbb R^{p+1}}\frac{1}{n}\mathbb E(\|{\bf Y}-{\mathfrak X}\beta'\|^2)\nonumber\\
		& = & 
		{\rm argmin}_{\beta'\in\mathbb R^{p+1}}
		\mathbb E\left(({\mathscr Y}-\widetilde{\mathscr X}^\top\beta')^2\right),\quad  \widetilde{\mathscr X}=(1,{\mathscr X}),  \label{best:fit:beta}
	\end{eqnarray}
	provides the best linear fitting for ${\bf Y}$ (or $\mathscr Y$) in the $L^2$ sense. 		
	In this setting, (\ref{uncorr:e})
	should be interpreted as the ``projection condition'' that ${\bf e}={\bf Y}-\mathfrak X\beta$ should be ``orthogonal'' (again in the $L^2$ sense) to any potential linear fitting; 
	see  Remark \ref{empiric} for an elaboration of this viewpoint. Also, (\ref{exo:p}) is automatically satisfied due to Remark \ref{l2}.
	Now, if we apply the usual first
	order test from Calculus to (\ref{best:fit:beta}) we find that
	\[
	\beta=\mathbb E(|\widetilde{\mathscr X}|^2)^{-1}\mathbb E({\mathscr Y}\widetilde{\mathscr X}^\top),
	\] 
	but this does not say much about the true nature of $\beta$ because the joint distribution of $({\mathscr X},{\mathscr Y})$ remains unknown, which makes the  expectations intractable. Thus, this population  parameter  should somehow be estimated from data (a realization $({\bf x},{\bf y})$ of $({\bf X},{\bf Y})$) with $\widehat\beta$ in (\ref{mls:form:mul}) being the most obvious candidate for an estimator. 
	As it is always the case with any estimator, its efficiency only gets validated by the establishment of good inferential properties (say, by confirming that its ${\rm mse}$ is minimized within a given class of estimators and/or that it is consistent and asymptotically normal, etc.; see the general discussion in Subsection \ref{param:est}), so with this purpose in mind 
	it 
	is convenient to add to (\ref{exo:p}) the assumption of {\em spherical error},
	which means that there exists $\sigma>0$ such that
	\begin{equation}\label{homo:p}
		{\mathbb C}({{\bf e}}|_{\mathfrak X=\mathfrak x})=\sigma^2{\rm Id}_{n}, \quad {\rm independently}\,{\rm of}\,\mathfrak x.
	\end{equation}
	Thus,
	\begin{equation}\label{homo:p:2}
	\mathbb E({\bf Y}|_{\mathfrak X=\mathfrak x})=\mathfrak x\beta \quad 
	{\rm and} \quad
	{\mathbb C}({\bf Y}|_{\mathfrak X=\mathfrak x})=\sigma^2{\rm Id}_{n}
	\end{equation}
	summarize the assumptions of the linear regression model	
	\footnote{
		Although the stronger assumption of the 
		independence of $\{{\bf e}|_{\mathfrak X_j=\mathfrak x_j}\}_{j=1}^n$
		may eventually be useful (as in Examples \ref{bound:prob} and \ref{bound:prob:2}, for instance), we 
		stress that only uncorrelatedness, as expressed by (\ref{homo:p}), is imposed at this point, 
		as this already allows us to derive some nice inferential properties for $\widehat\beta$; cf. Propositions \ref{g:m:prep} and \ref{gauss:markov} and Remark \ref{large:lsm}.
		In any case, if ${\bf e}|_{\mathfrak X=\mathfrak x}$ is normally distributed, as in Example \ref{mle:imp:lsm:n}, then these assumptions (uncorrelatedness and independence) are equivalent indeed (by Corollary \ref{assert:eq:ind}).}. In the language of Example \ref{mle:imp:lsm:0}, (\ref{homo:p}) means that the random matrix ${\bm\sigma}^2(\mathfrak X)$ is actually {\em constant} and equals $\sigma^2{\rm Id}_{n}$, an artifact also known as {\em homoscedasticity}.
		\qed
\end{example}

\begin{example}\label{mle:imp:lsm:n} (The linear regression model with a normal error)
	In the setting of the linear regression model (\ref{lin:reg:mod}),
	the ``empirical'' quadratic minimization in (\ref{min:sq}) may be justified   
	via MLE under a normality assumption on the error\footnote{As it is well-known, this  connection between OLS and the normal distribution has been first observed by Gauss and Laplace \cite[Chapter 4]{stigler1990history}.}. Precisely, and in alignment with (\ref{exo:p}) and (\ref{homo:p}), let us further assume that the error ${\bf e}$ is such that $\{{\bf e}_j|_{\mathfrak X_j=\mathfrak x_j}\}_{j=1}^n$ is independent and distributed according to
	\begin{equation}\label{norm:as:mls:0}
		{\bf e}_j|_{\mathfrak X_j=\mathfrak x_j}\sim\mathcal N(0,\sigma^2),
	\end{equation} 
	or equivalently, 
	\begin{equation}\label{norm:as:mls}
		{\bf e}|_{\mathfrak X=\mathfrak x}\sim
		\mathcal N(\vec{0},\sigma^2{\rm Id}_{n}),
	\end{equation}
	by Proposition \ref{unc:ind:n}.
	It then follows from (\ref{lin:reg:mod}) that $\{{\bf Y}_j|_{\mathfrak X_j=\mathfrak x_j}\}_{j=1}^n$ is independent with
	\begin{equation}\label{model:mls}
		{\bf Y}_j|_{\mathfrak X_j=\mathfrak x_j}\sim\mathcal 
		N\left(\sum_{k=0}^p\mathfrak x_{jk}\beta_k,\sigma^2\right).
	\end{equation}
	In this way, we obtain an identifiable statistical model, conditionally on the observed value $\mathfrak x$ of $\mathfrak X$, in which $\beta$ appears as the unknown parameter. Actually, this construction fits  the extended notion of a statistical model in
	Remark
	\ref{stat:mod:def:2}, since the conditional distributions in (\ref{model:mls}) vary across observations as their expectations depend on covariates (the same observation affects the general linear model specified by (\ref{homo:p:2})). 
	In particular, we may apply MLE to (\ref{model:mls}), as in Definition \ref{mle:def:post}, to find the corresponding estimator.
	Indeed, (\ref{model:mls}) can be succinctly written as 
	\[
	{\bf Y}|_{\mathfrak X=\mathfrak x}
	\sim \mathcal N({\mathfrak x}\beta,\sigma^2{\rm Id}_{n}),
	\]	
	so the corresponding likelihood function is 
	\begin{equation}\label{like:lrmod}
		L({\bf y};\beta)=\left({2\pi\sigma^2}\right)^{-n/2}e^{-\frac{\|{\bf y}-{\mathfrak x}\beta\|^2}{2\sigma^2}}.		
	\end{equation}
	Since the corresponding log-likelihood function to be maximized is
	\begin{equation}\label{log:beta}
		l({\bf y};\beta)=-\frac{n}{2}\ln (2\pi\sigma^2)-\frac{\|{\bf y}-{\mathfrak x}\beta\|^2}{2\sigma^2},
	\end{equation}
	we 
	see  that, up to irrelevant constants (depending on $\sigma^2$, here assumed known), solving this maximization problem is equivalent to finding  $\widehat\beta$ as in (\ref{min:sq}),
	thus confirming that MLE implies OLS under the stated assumptions. 	\qed
	\end{example}

Although the terminology {\em ordinary least squares (OLS)} is sometimes reserved for the specialized linear model with normal errors in Example \ref{mle:imp:lsm:n}, we will use it more broadly to denote the framework defined by the empirical estimator $\widehat\beta$ in (\ref{min:sq}) and the assumptions in (\ref{homo:p:2}). In particular, we refer to $\widehat\beta$ as the OLS estimator.

\begin{remark}\label{reg:func:choice}
	The appearance of the regression function $F$ in (\ref{reg:model:2}) may be justified by the fact that 
	\[
	\mathbb E(\|{\bf Y}-F(\mathfrak X)\|^2)=\inf_G \mathbb E(\|{\bf Y}-G(\mathfrak X)\|^2),
	\]
	for any $G:\mathbb R^{n(p+1)}\to\mathbb R^n$ measurable. To check this, first note that
	\begin{eqnarray*}
		\mathbb E(|{\bf Y}-G(\mathfrak X)|^2)
		& = & \mathbb E(\|{\bf Y}-F(\mathfrak X)+F(\mathfrak X)-G(\mathfrak X)\|^2)\\
		& = & 	\mathbb E(\|{\bf Y}-F(\mathfrak X)\|^2)+ 	\mathbb E(\|F(\mathfrak X)-G(\mathfrak X)\|^2) \\
		&   & \quad + 2\mathbb E(\langle{\bf Y}-F(\mathfrak X),F(\mathfrak X)-G(\mathfrak X)\rangle).
	\end{eqnarray*}
	Also, 
	by Proposition \ref{ceprop} (2) and Proposition \ref{int:prob},
	\begin{eqnarray*}
		\mathbb E(\langle{\bf Y}-F(\mathfrak X),F(\mathfrak X)-G(\mathfrak X)\rangle)
		& = & \mathbb E(\mathbb E(\langle{\bf Y}-F(\mathfrak X),F(\mathfrak X)-G(\mathfrak X)\rangle|\mathfrak X))\\
		& = & 
		\mathbb E(\langle{\bf Y}-F(\mathfrak X),F(\mathfrak X)-G(\mathfrak X)\rangle|_{\mathfrak X=\mathfrak x})\\
		& = & \langle F(\mathfrak x)-G(\mathfrak x),\mathbb E({\bf Y}|_{\mathfrak X=\mathfrak x})-F(\mathfrak x)\rangle\\
		& = & 0.	
	\end{eqnarray*}
	Thus, 
	\[
	\mathbb E(\|{\bf Y}-G(\mathfrak X)\|^2)\geq \mathbb E(\|{\bf Y}-F(\mathfrak X)\|^2),
	\]
	with the equality holding if and only if $G(\mathfrak X)=F(\mathfrak X)$ a.s.\qed
\end{remark}

\begin{remark}\label{rem:r:vs:nr} (Fixed vs.\ random design)
	The preceding discussion shows that the design matrix $\mathfrak X$ may be treated in two distinct ways: one may either regard it as random, in which case it directly contributes to the randomness of ${\bf Y}$, or one may condition on a realization $\mathfrak X=\mathfrak x$, thereby fixing the design and allowing this randomness to arise solely from the error term.
	Conceptually, the distinction concerns the scope of the randomness.
	In the fixed-design approach, common in controlled experiments, $\mathfrak X$ is treated as non-random and inference is made conditionally on its observed values.
	In the random-design approach, typical of observational studies, $\mathfrak X$ is itself random, and the regression coefficients are interpreted as population features rather than merely as quantities attached to the observed design.
	Although the algebraic form of estimators such as $\widehat\beta=(\mathfrak x^\top\mathfrak x)^{-1}\mathfrak x^\top{\bf y}$ remains unchanged, the interpretation of their variance and the probabilistic framework underlying inference differ substantially.
	In experimental fields such as agriculture, where the investigator typically controls the values of ${\bf x}$ and subsequently records ${\bf y}$, the former perspective is more natural. By contrast, in disciplines such as Econometrics, where such control is usually not available, $\mathfrak X$ is more appropriately regarded as random. As already pointed out, the algebraic form of the relevant inferential statistics is, at the level of formulas, largely unaffected by the choice of viewpoint, particularly under a suitable normality assumption \cite[Chapter~10]{rencher2008linear}.
	Since normality will be assumed in most of what follows, it is convenient to omit explicit conditioning on $\mathfrak X=\mathfrak x$, with the understanding that $\mathfrak X=({\bf 1},{\bf X})$ should still be regarded as random whenever normal-theory distributional calculations are not an issue (for instance, as in Remark~\ref{coeff:det}, where the coefficient of determination is examined, or in Example~\ref{super:learn}, where regression theory is embedded in the broader framework of Statistical Learning). Hence, except in those cases, all probabilistic statements are to be interpreted conditionally on $\mathfrak X=\mathfrak x$, corresponding to the fixed-design viewpoint.
	\qed
\end{remark}

\subsection{Inference, goodness of fit and prediction for OLS}\label{subsec:inf:mle}
With a statistical model for OLS at hand, we now proceed to the pertinent inferential analysis. We start by observing that 
although the normality assumption for the error in (\ref{norm:as:mls}) is essential for interpreting OLS via MLE, good statistical properties of the associated estimator $\widehat\beta$
may be derived under the much less stringent assumptions of  the linear regression model in Example \ref{mle:imp:lsm}. 

\begin{proposition}\label{g:m:prep}
	Under the conditions of Example \ref{mle:imp:lsm} there hold
	\begin{equation}\label{g:m:prep:3}
		\mathbb E(\widehat\beta)=\beta, \quad {\mathbb C}(\widehat\beta)=\sigma^2 ({{\mathfrak x}}^\top{{\mathfrak x}})^{-1}.
	\end{equation}
	As a 
	consequence, $\widehat\beta$ is unbiased
	and ${\rm mse}(\widehat\beta)=\sigma^2{\rm tr} ({{\mathfrak x}}^\top{{\mathfrak x}})^{-1}$. 
\end{proposition}

\begin{proof}
	With the simplifying notation suggested by Remark \ref{rem:r:vs:nr},
	we are assuming that 
	\begin{equation}\label{g:m:prep:2}
		\mathbb E({\bf e}_j)=0\quad  {\rm and} \quad {\mathbb C}({\bf e}_j,{\bf e}_k)=\sigma^2\delta_{jk}, \quad j,k=1,\cdots,n.
	\end{equation}
	Hence, from (\ref{mls:form:mul}),
	\[
	\widehat\beta=({{\mathfrak x}}^\top{{\mathfrak x}})^{-1}{\mathfrak x}^\top(\mathfrak x\beta+{\bf e}),	
	\]
	so that
	\begin{equation}\label{exp:hat:beta}
		\widehat\beta=\beta+({{\mathfrak x}}^\top{{\mathfrak x}})^{-1}{{\mathfrak x}}^\top{\bf e},
	\end{equation}
	which gives
	\[
	\mathbb E(\widehat\beta)=\mathbb E(\beta)+({{\mathfrak x}}^\top{{\mathfrak x}})^{-1}{{\mathfrak x}}^\top\mathbb E({\bf e})=\beta+({{\mathfrak x}}^\top{{\mathfrak x}})^{-1}{{\mathfrak x}}^\top(\vec{0} )=\beta.
	\]
	Also,
	\begin{eqnarray*}
		{\mathbb C}(\widehat\beta)
		& = & \mathbb E((\widehat\beta-\beta)\otimes(\widehat\beta-\beta))\\
		& = & \mathbb E((({{\mathfrak x}}^\top{{\mathfrak x}})^{-1}{{\mathfrak x}}^\top{\bf e})\otimes (({{\mathfrak x}}^\top{{\mathfrak x}})^{-1}{{\mathfrak x}}^\top{\bf e}))\\
		& = & ({{\mathfrak x}}^\top{{\mathfrak x}})^{-1}{\mathfrak x}^\top\mathbb E({\bf e}\otimes{\bf e}){\mathfrak x}({{\mathfrak x}}^\top{{\mathfrak x}})^{-1}\\
		& = & ({{\mathfrak x}}^\top{{\mathfrak x}})^{-1}{\mathfrak x}^\top{\mathbb C}({\bf e}){\mathfrak x}({{\mathfrak x}}^\top{{\mathfrak x}})^{-1},
	\end{eqnarray*}
	and using that ${\mathbb C}({\bf e})=\sigma^2{\rm Id}$ by (\ref{g:m:prep:2}), the result follows.
\end{proof}

As a first check on the efficiency of OLS estimator $\widehat\beta$ in (\ref{mls:form:mul}), let us see how it competes with a general {\em linear} estimator
\[
\overline\beta=C{\bf Y},
\]
where $C$ is a $(p+1)\times n$ matrix which is allowed to depend on $\mathfrak x$ but not on ${\bf Y}$. This leads to a remarkable result confirming that $\widehat\beta$  attains the best performance (as measured by the mse) within a natural class of estimators. 

\begin{theorem}\label{gauss:markov}(Gauss-Markov) Let $\overline\beta$ as above be unbiased with ${\bf e}$ satisfying (\ref{g:m:prep:2}). Then ${\mathbb C}(\overline\beta)\geq {\mathbb C}(\widehat\beta)$. 
\end{theorem}

\begin{proof}
	We write 
	\[
	\overline\beta=\left(({{\mathfrak x}}^\top{{\mathfrak x}})^{-1}{{\mathfrak x}}^\top+ D\right){\bf Y},
	\]
	where $D$ has the same properties as $C$. It follows that
	\[
	\overline\beta=
	(({\mathfrak x}^\top{\mathfrak x})^{-1}{\mathfrak x}^\top+D)({\mathfrak x}\beta+{\bf e})=
	\widehat\beta+D{\mathfrak x}\beta+D{\bf e}
	\]
	with
	$\mathbb E(D{\bf e})=D\mathbb E({\bf e})=0$ (by Proposition 4.9 (3))
	so that
	\[
	\mathbb E(\overline\beta)=\beta+D{\mathfrak x}\beta, 
	\]
	and letting $\beta$ vary we see that 
	a vanishing bias for $\overline\beta$ implies $D{\mathfrak x}=0$. 
	Hence,
	\[
	{\mathbb C}(\overline\beta)={\mathbb C}(\widehat\beta)+{\mathbb C}(D{\bf e})+2{\mathbb C}(\widehat\beta,D{\bf e}).
	\]
	Now note that ${\mathbb C}(D{\bf e})={D}{\mathbb C}({\bf e})D^\top=\sigma^2{D}D^\top$. Moreover,  using (\ref{exp:hat:beta}) and the fact that $\beta$ is non-random, 
	\begin{eqnarray*}
		{\mathbb C}(\widehat\beta, D{\bf e})
		&= &{\mathbb C}(\beta, D{\bf e})+{\mathbb C}(({\mathfrak x}^\top{\mathfrak x})^{-1}{\mathfrak x}^\top{\bf e}, D{\bf e})\\
		& = & 
		{\mathbb C}(({\mathfrak x}^\top{\mathfrak x})^{-1}{\mathfrak x}^\top{\bf e}, D{\bf e})\\
		& = &
		({\mathfrak x}^\top{\mathfrak x})^{-1}{\mathfrak x}^\top{\mathbb C}({\bf e}) D^\top\\
		& = & \sigma^2({\mathfrak x}^\top{\mathfrak x})^{-1}( D{\mathfrak x})^\top\\
		& = & 0.	
	\end{eqnarray*}
	Thus,
	\[
	{\mathbb C}(\overline\beta)={\mathbb C}(\widehat\beta)+\sigma^2{D}{D}^\top,
	\]
	and the result follows because ${ D}{ D}^\top\geq 0$.
\end{proof}

\begin{example}\label{weigh:comb}
	If we take it for granted that ${\bf X}$ has no influence whatsoever on ${\bf Y}$ then we are actually dealing with
	the ``intercept-only'' case $\beta=(\beta_0,0,\cdots,0)$, so we must impose  $\widehat\beta=(\widehat\beta_0,0,\cdots,0)$ and, as expected, 
	(\ref{mls:form:mul}) gives 
	\begin{equation}\label{weigh:comb:2}
		\widehat\beta_0=\overline{\bf Y}=\frac{1}{n}\sum_j{\bf Y}_j,
	\end{equation}
	the sample mean of  ${\bf Y}$; in the simple linear regression case of Example \ref{simp:lin:reg} below, this is immediate from (\ref{simp:lin:reg:2}). If $w=(w_1,\cdots,w_n)\in\mathbb R^n$ we may consider the more general linear combination in the entries of ${\bf Y}$ given by 
	\[
	\widehat\beta_0^w=\sum_jw_j{\bf  Y}_j=\widehat\beta_0+\sum_j\left(w_j-\frac{1}{n}\right){\bf Y}_j,
	\]
	so that
	\begin{eqnarray*}
		\mathbb E(\widehat\beta_0^w)
		& = & 
		\beta_0+\sum_j\left(w_j-\frac{1}{n}\right)\mathbb E({\bf Y}_j)\\
		& = &
		\beta_0+\sum_j\left(w_j-\frac{1}{n}\right)(\mathfrak x\beta)_j\\ 
		& = & 	\beta_0+\beta_0\sum_j\left(w_j-\frac{1}{n}\right),
	\end{eqnarray*}
	and $\widehat\beta^w_0$ is unbiased if and only $w$ is a weight vector, $\sum_jw_j=1$. Thus, Gauss-Markov applies to ensure that $\widehat\beta_0$ attains the best performance among all these {\em weighted} estimators of $\beta_0$. In particular, Gauss-Markov may be regarded as a generalization of Example \ref{weigh:est:mu}. \qed
\end{example}

\begin{example}\label{simp:lin:reg} 
	(Simple linear regression) If $p=1$ in Example \ref{mle:imp:lsm} then $\mathfrak X=({\bf 1},{\bf X})$, where ${\bf X}=({\bf X}_{11},\cdots,{\bf X}_{n1})^\top$, so that
	\[ 
	{\mathfrak X}^\top{\mathfrak X}=
	\left(
	\begin{array}{cc}
		n & n\overline{\bf X}\\
		n\overline{\bf X} & \|{\bf X}\|^2
	\end{array}
	\right),\quad \overline{\bf X}=\frac{1}{n}\sum_j{\bf X}_{j1}.
	\]
	Since ${\bf X}$ is supposed not to be a multiple of ${\bf 1}$ a.s., Cauchy-Schwartz
	implies that
	\[
	\det 	{\mathfrak X}^\top{\mathfrak X} =n\|{\bf X}\|^2-n^2\overline{\bf X}^2>0,
	\]
	so that ${\mathfrak X}^\top{\mathfrak X}$ is invertible and 
	\[
	({\mathfrak X}^\top{\mathfrak X})^{-1}=\frac{1}{n\|{\bf X}\|^2-n^2\overline{\bf X}^2}
	\left(
	\begin{array}{cc}
		\|{\bf X}\|^2  & -n\overline{\bf X}\\
		-n\overline{\bf X} & n
	\end{array}
	\right).
	\]
	A little computation using (\ref{mls:form:mul}) then gives 
	\begin{equation}\label{simp:lin:reg:2}
		\widehat\beta:=
		\left(
		\begin{array}{c}
			\widehat\beta_0\\
			\widehat\beta_1
		\end{array}
		\right)
		=\left(
		\begin{array}{c}
			\overline{\bf Y}-	\widehat\beta_1\overline{\bf X}\\
			S_{\bf X\bf Y}/S_{\bf X\bf X}	
		\end{array}
		\right), 
	\end{equation}
	where 
	\[
	\overline{\bf Y}=	\frac{1}{n}\sum_j{\bf Y}_j
	\]
	is the sample mean of ${\bf Y}$ and 
	\begin{equation}\label{simp:lin:reg:3}
		S_{\bf X\bf Y}=\sum_j({\bf X}_{j1}-\overline{\bf X})({\bf Y}_j-\overline{\bf Y}), \quad
		S_{\bf X\bf X}=\sum_j ({\bf X}_{j1}-\overline{\bf X})^2.
	\end{equation}
	Thus, the second line in (\ref{simp:lin:reg:2}) is used to compute the {\em slope} $\widehat\beta_1$ from sample data whereas the first line determines the {\em intercept} $\widehat\beta_0$. 
	In this case, the fitted value is realized as
	\begin{equation}\label{fit:vec:r}
		\widehat{\bf y}=\widehat\beta_0+\widehat\beta_1{\bf x}.
	\end{equation}
	Also, $\widehat\beta$ is unbiased with 
	\[
	{\mathbb C}(\widehat\beta)=
	\frac{\sigma^2}{n\|{\bf x}\|^2-n^2\overline{\bf x}^2}
	\left(
	\begin{array}{cc}
		\|{\bf x}\|^2  & -n\overline{\bf x}\\
		-n\overline{\bf x} & n
	\end{array}
	\right),
	\] 
	under the assumptions of Proposition \ref{g:m:prep}.
	\qed
\end{example}

We now discuss the construction of confidence intervals for the entries of the unknown parameter $\beta$.  Here we remain in the setting of Example \ref{mle:imp:lsm:n}, so we assume  that, conditionally on $\mathfrak X=\mathfrak x$, $\{{\bf e}_j\}$ is independent with ${\bf e}_j\sim\mathcal N(0,\sigma^2)$ as in (\ref{norm:as:mls}). 
It follows from (\ref{g:m:prep:3}), (\ref{exp:hat:beta})  and Corollary \ref{ortho:norm} that
\begin{equation}\label{pivot:q:new}
\widehat\beta-\beta\sim\mathcal N(\vec{0},\sigma^2\mathfrak s), \quad  \mathfrak s:=({\mathfrak x}^\top\mathfrak x)^{-1},
\end{equation}
so we  may use the pivotal quantity  
\begin{equation}\label{pivot:q}
	\frac{\widehat\beta_j-\beta_j}{{\bf s}_j}\sim\mathcal N(0,1), \quad {{\bf s}}_j:={\sigma}\sqrt{\mathfrak s_{jj}},
\end{equation}
to exhibit confidence intervals for the unknown parameter $\beta_j$ in case $\sigma$ is known. 
Precisely, in the notation of Subsection \ref{conf:int:sub},
\begin{equation}\label{conf:int:lsm:1}
	\beta_j\in \left[\widehat \beta_j\mp z_{1-\delta/2}{\bf s}_j\right]
	\,{\rm with}\,{\rm prob.}=1-\delta.
\end{equation}
Otherwise, we proceed as follows. 
We define the {\em residual}
\begin{equation}\label{def:res:vec}
	\widehat{\bf e}:={\bf Y}-\widehat{{\bf Y}},
\end{equation}
where $\widehat{{\bf y}}={\mathfrak x}\widehat\beta$ is the fitted vector as in (\ref{fit:vec}).
As we shall see, $\|\widehat{\bf e}\|^2/(n-p-1)$ qualifies as an appropriate estimator for the error variance $\sigma^2$.

\begin{proposition}\label{rss:dist}
	$\widehat{\bf e}\sim \mathcal N(\vec{0},\sigma^2{\rm Id}_{n-p-1})$ and $\|\widehat{\bf e}\|^2/\sigma^2\sim \chi^2_{n-p-1}$, $n\geq p+2$. 
\end{proposition}

\begin{proof}
	We compute 
	\[
	\widehat{\bf e}=({\rm Id}_n-{\mathfrak x}({\mathfrak x}^\top{\mathfrak x})^{-1}{\mathfrak x}^\top){\bf y}=Q({\mathfrak x}\beta+{\bf e}),
	\]
	where $Q={\rm Id}_{n}-{\mathfrak x}({\mathfrak x}^\top{\mathfrak x})^{-1}{\mathfrak x}^\top$ is an idempotent, symmetric matrix satisfying 
	\begin{equation}\label{rank:Q}
		Q{\mathfrak x}=0,
	\end{equation}
	which means that ${\rm rank}\,Q=n-p-1\geq 1$. Therefore, 
	$\widehat{\bf e}=Q{\bf e}$ is normally distributed as in the statement (either by Proposition 4.9 (3) or by rotational invariance (Corollary \ref{unc:ind:n:c})).
	Moreover, since
	\begin{equation}\label{e:hat:e}
		\frac{\|\widehat{\bf e}\|^2}{\sigma^2}=\left\langle \frac{{\bf e}}{\sigma}, Q\left(\frac{{\bf e}}{\sigma}\right) \right\rangle,
	\end{equation}
	the last assertion follows from Proposition \ref{u:v:gen}.
\end{proof}

\begin{remark}\label{geom:mls}
	(The geometry of the linear model and regression diagnostics)
	The projection matrix\footnote{Recall that this means that $H$ is symmetric and idempotent, hence defining an orthogonal projection onto its range.} $H={\mathfrak x}({\mathfrak x}^\top{\mathfrak x})^{-1}{\mathfrak x}^\top$ appearing in the argument above is usually called the ``hat matrix'', as it projects ${\bf y}$ onto $\widehat{\bf y}=H{\bf y}\in C({\mathfrak x})\equiv \mathbb R^{p+1}$, the {\em design space}, which is the $(p+1)$-subspace of $\mathbb R^n$ spanned by the columns of the design matrix $\mathfrak x$. 
	On the other hand, its complementary projection matrix $Q={\rm Id}_{n}-H$ projects ${\bf y}$ (and also ${\bf e}$, because ${\bf y}-{\bf e}=\mathfrak x\beta\in C(\mathfrak x)$) onto the residual $\widehat{\bf e}\in C({\mathfrak x})^\perp\equiv \mathbb R^{n-p-1}$ lying in the orthogonal complement of  $C(\mathfrak x)$. 
	This nice ``orthogonal'' geometry, which hinges on the  general setting of Example \ref{mle:imp:lsm}, is depicted in Figure \ref{figg}, where $\widetilde{\bf e}=H{\bf e}$.  In particular,  the orthogonal decomposition ${\bf y}=\widehat{\bf y}+\widehat{\bf e}$ clearly implies that  the
	{\em sample correlation} between $\widehat{\bf y}$ and $\widehat{\bf e}$, defined by 
	\begin{equation}\label{samp:corr:def}
		{\rm corr}(\widehat{\bf y},\widehat{\bf e})=\frac{\widehat{\bf e}^\top\widehat{\bf y}}{\|\widehat{\bf e}\|\|\widehat{\bf y}\|},
	\end{equation}
	vanishes,
	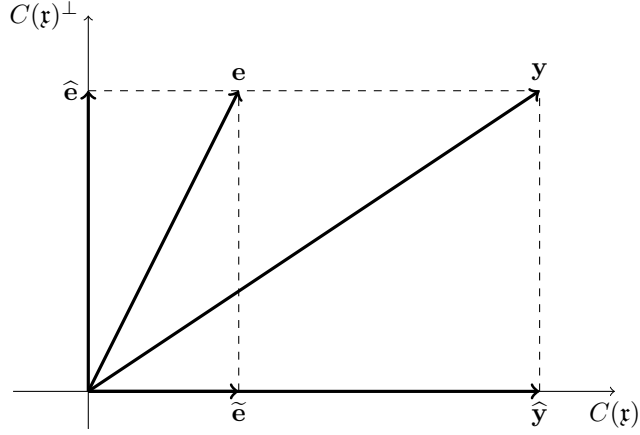
\begin{figure}
		\begin{center}
			\begin{tikzpicture}[scale=2]
				\begin{scope}
					\draw[->](-0.5,0) -- (3.5,0) node[anchor=north]  {$C(\mathfrak x)$};
					\draw[->] (0,-0.25) -- (0,2.5) node[anchor=east]  {$C(\mathfrak x)^\perp$};
					\draw[->][black, very thick] (0,0) -- (3,2) node[anchor=south] {${\bf y}$};
					\draw[->][black, very thick] (0,0) --  (1,2) node[anchor=south] {${\bf e}$};
					\draw[->][black, very thick] (0,0) -- (3,0) node[anchor=north] {$\widehat{\bf y}$};
					\draw[->][black, very thick] (0,0) -- (0,2) node[anchor=east] {$\widehat{\bf e}$} ;
					\draw[->][black, very thick] (0,0) -- (1,0) node[anchor=north] {$\widetilde{\bf e}$};
					\draw[dashed] (3,0) -- (3,2);
					\draw[dashed] (0,2) -- (3,2);
					\draw[dashed] (1,0) -- (1,2);
				\end{scope}
			\end{tikzpicture}
		\end{center}
		\caption{The geometry of the linear model}
		\label{figg}
	\end{figure}
which justifies the common practice of using a scatterplot of the residuals against the fitted values in order to identify patterns of goodness of fit (or lack thereof) of a given linear model, as far as linearity and homoscedasticity go \cite[Section 6.1]{faraway2006linear}.
Now, if  we further specialize to the setting of Example \ref{mle:imp:lsm:n}, which assumes ${\bf e}\sim\mathcal N(\vec{0},\sigma^2{\rm Id}_{n})$, then, by the projection property in (\ref{norm:space:4}),
	\begin{equation}\label{proj:res:dist}
		\widehat{\bf e}\sim\mathcal N(\vec{0},\sigma^{2}{\rm Id}_{n-p-1}),
	\end{equation}
	so in particular 
	$\widehat{\bf e}/|\widehat{\bf e}|$ is uniformly distributed in $\mathbb S^{n-p-2}\subset C(\mathfrak x)^\perp$ by Remark \ref{fisher:comp:1}. 
	Since $\widehat{\bf e}$ is accessible from data and adjusted values, graphical methods (say, a Q-Q plot) may be used to confirm the empirical validity of (\ref{proj:res:dist}), which somehow works as an indirect checking of the theoretical assumption on the normality of errors underlying Example \ref{mle:imp:lsm:n}; again, see \cite[Section 6.1]{faraway2006linear}. A further gauging of the goodness of fit of the model may be implemented after properly combining the residual and the fitted vector. The simplest way of doing this, which leads to a sharpening of (\ref{samp:corr:def}), is to look at the joint distribution of $(\widehat{\bf y},\widehat{\bf e})$ under error normality.  
	Since 
	\[
	\left(
	\begin{array}{c}
		\widehat{\bf y}\\
		\widehat{\bf e}
	\end{array}
	\right)
	=
	\left(
	\begin{array}{c}
		H{\bf y}\\
		Q{\bf y}
	\end{array}
	\right)
	=
	\left(
	\begin{array}{cc}
		H & 0\\
		0 & Q
	\end{array}
	\right)
	\left(
	\begin{array}{c}
		{\bf y}\\
		{\bf y}
	\end{array}
	\right),
	\]
	we see that $(\widehat{\bf y},\widehat{\bf e})$ is jointly normally distributed. To find the specific normal distribution we note that 
	\[
	\mathbb E
	\left(\left(
	\begin{array}{c}
		\widehat{\bf y}\\
		\widehat{\bf e}
	\end{array}
	\right)\right)=
	\left(
	\begin{array}{cc}
		H & 0\\
		0 & Q
	\end{array}
	\right)
	\mathbb E\left(
	\begin{array}{c}
		{\bf y}\\
		{\bf y}
	\end{array}
	\right)
	=
	\left(
	\begin{array}{cc}
		H & 0\\
		0 & Q
	\end{array}
	\right)
	\left(
	\begin{array}{c}
		\mathfrak x\beta\\
		\mathfrak x\beta
	\end{array}
	\right)
	=
	\left(
	\begin{array}{c}
		\mathfrak x\beta\\
		0
	\end{array}
	\right)
	\]
	and 
	\[
	{\mathbb C}
	\left(\left(
	\begin{array}{c}
		\widehat{\bf y}\\
		\widehat{\bf e}
	\end{array}
	\right)\right)=
	\left(
	\begin{array}{cc}
		H & 0\\
		0 & Q
	\end{array}
	\right)
	{\mathbb C}\left(
	\begin{array}{c}
		{\bf y}\\
		{\bf y}
	\end{array}
	\right)
	\left(
	\begin{array}{cc}
		H & 0\\
		0 & Q
	\end{array}
	\right)^\top
	=\sigma^2
	\left(
	\begin{array}{cc}
		H & 0\\
		0 & Q
	\end{array}
	\right),
	\]
	so that 
	\[
	{\mathbb C}
	\left(\Lambda\left(
	\begin{array}{c}
		\widehat{\bf y}\\
		\widehat{\bf e}
	\end{array}
	\right)\right)=
	\sigma^2
	\Lambda\left(
	\begin{array}{cc}
		H & 0\\
		0 & Q
	\end{array}
	\right)\Lambda^\top,
	\]
	with $\Lambda$ being any $2n\times 2n$ matrix. By choosing $\Lambda$ orthogonal with the corresponding conjugation performing the appropriate diagonalization  and viewing $(\widehat{\bf y},\widehat{\bf e})$ as an element of $\mathbb R^{p+1}\times \mathbb R^{n-p-1}=\mathbb R^n$, we find that there exists an orthogonal $n\times n$ matrix $\Lambda'$ such that
	\[
	{\mathbb C}
	\left(\Lambda'\left(
	\begin{array}{c}
		\widehat{\bf y}\\
		\widehat{\bf e}
	\end{array}
	\right)\right)=
	\sigma^2{\rm Id}_{n}.
	\]
	By Corollary \ref{ortho:norm}, and viewing  $\mathfrak r\beta$ as an element of $\mathbb R^{p+1}$, 
	\[
	\Lambda' \left(
	\begin{array}{c}
		\widehat{\bf y}\\
		\widehat{\bf e}
	\end{array}
	\right)\sim\mathcal N\left(\Lambda'
	\left(
	\begin{array}{c}
		\mathfrak x\beta\\
		0
	\end{array}
	\right),\sigma^2{\rm Id}_{n}
	\right),
	\]
	and 
	therefore $\{\widehat{\bf y},
	\widehat{\bf e}\}$ is independent
	by Corollary \ref{unc:ind:n:c}. 
	Further uses of the residual in the art of quantifying the goodness of fit of the linear model may be found in Remark \ref{coeff:det} below. \qed
\end{remark}

We now come back to the business of finding confidence intervals for the entries of $\beta$, this time with $\sigma^2$ regarded as unknown.
In this case, the next result justifies the replacement of ${\bf s}_j$ in (\ref{pivot:q}) by 
\begin{equation}\label{pivot:eq:2}
	\widehat{\bf s}_j:=
	\widehat\sigma \sqrt{\mathfrak s_{jj}}, \quad 
	\widehat\sigma:=\frac{\|\widehat{\bf e}\|}{\sqrt{n-p-1}}.
\end{equation}
Note that with this notation, Proposition \ref{rss:dist} says that 
\begin{equation}\label{pivot:eq:3}
	(n-p-1)\frac{\widehat\sigma^2}{\sigma^2}=\frac{\|\widehat{\bf e}\|^2}{\sigma^2}\sim \chi^2_{n-p-1}. 
\end{equation}

\begin{proposition}\label{proc:conf:int:mls}
	$\widehat\sigma^2$ is an unbiased and consistent estimator for $\sigma^2$ (as $n\to+\infty$ and $p$ is held fixed). In particular, $\widehat\sigma$ is consistent for the error standard deviation $\sigma$.  Moreover, $\{\widehat\beta,\widehat\sigma^2\}$ is independent with
	\[
	\frac{\widehat\beta_j-\beta_j}{\widehat{\bf s}_j}\sim \mathfrak t_{n-p-1}. 
	\]
\end{proposition}

\begin{proof}
	From (\ref{pivot:eq:3}) and Corollary \ref{chi:sq:ms} we have $\mathbb E(\widehat\sigma^2)=\sigma^2$, that is, ${\rm bias}(\widehat\sigma^2)=0$. Also, \begin{equation}\label{exp:var:res}
		{\mathbb V}(\widehat\sigma^2)=
		\frac{2\sigma^4}{n-p-1},
	\end{equation}
	so that ${\rm mse}(\widehat\sigma^2)\to 0$ as $n\to+\infty$ and consistency follows from Proposition \ref{mse:consist}.  
	Moreover, since
	\[
	\frac{\widehat\beta_j-\beta_j}{\widehat{\bf s}_j}=\frac{\frac{\widehat\beta_j-\beta_j}{{\bf s}_j}}{\sqrt{\frac{\|\widehat{e}\|^2/\sigma^2}{{n-p-1}}}},
	\]
	the last assertion follows from (\ref{pivot:q}), (\ref{pivot:eq:3}) and Proposition \ref{norm:chi:stu} once one verifies that $\{\widehat\beta,\|\widehat{\bf e}\|^2\}$ is independent. To check this,  note that 
	(\ref{exp:hat:beta}) gives 
	\[
	{\mathfrak x}^\top\mathfrak x
	\left(
	\frac{\widehat\beta-\beta}{\sigma}\right)=\mathfrak x^\top\left(\frac{{\bf e}}{\sigma}
	\right),
	\]
	which together with  (\ref{e:hat:e}),  (\ref{rank:Q}) and 
	Proposition \ref{lin:quad:n} implies that 
	$\{\mathfrak x^\top\mathfrak x\widehat\beta,\|\widehat{\bf e}\|^2\}$ is independent, from which the result follows (because $\mathfrak x^\top\mathfrak x$ is invertible). 
\end{proof}

Thus, again using the notation of Subsection \ref{conf:int:sub},
\begin{equation}\label{conf:int:lsm:2}
	\beta_j\in \left[\widehat \beta_j\mp \mathfrak t_{n-p-1,1-\delta/2}\widehat{\bf s}_j\right]
	\,{\rm with}\,{\rm prob.}\,\approx\,1-\delta,
\end{equation}
a confidence interval estimate for $\beta_j$ in case $\sigma$ is unknown. Notice that if $n-p-1\gg 0$ then we can replace $\mathfrak t_{n-p-1,1-\delta/2}$ by $z_{1-\delta/2}$ with a negligible error; this uses Remark \ref{tk:to:norm}. 

\begin{example}\label{conf:ref:beta} 
	(Confidence region for the whole vector parameter $\beta$, with $\sigma^2$ unknown) Pick $\mathfrak p$ so that $\mathfrak p^\top\mathfrak p=\mathfrak s$ as in (\ref{pivot:q:new}) and set $\mathfrak n=\sigma^{-1}(\mathfrak p^\top)^{-1}(\widehat\beta-\beta)$. It is immediate that $\mathfrak n\sim \mathcal N(\vec{0},{\rm Id}_{p+1})$ so that 
	\[
	\frac{(\widehat\beta-\beta)^\top\mathfrak s(\widehat\beta-\beta)}{\sigma^2}=\mathfrak n^\top\mathfrak n\sim \chi^2_{p+1}
	\]
	and is independent of $\widehat\sigma^2$. Therefore, by Propositions \ref{proc:conf:int:mls} and \ref{chi:to:F},
	\[
	{\frac{(\widehat\beta-\beta)^\top
			{{\mathfrak s}}({\widehat\beta}-\beta)}{(p+1)\widehat\sigma^2}}
	=
	\frac{{\mathfrak n}^\top{\mathfrak n}/(p+1)}{\widehat\sigma^2/\sigma^2}
	\sim
	{\bm{\textsf{F}}_{p+1,n-p-1}},
	\]
	which gives the ``confidence region'' estimate
	\begin{equation}\label{est:w:beta}
		\beta\in \mathcal U_{n,p,\delta}(\widehat\beta;\widehat\sigma^2)\,\,{\rm with}\,\,{\rm prob.}\,\,1-\delta,
	\end{equation}
	where 
	\begin{equation}\label{est:w:beta:2}
		\mathcal U_{n,p,\delta}(\widehat\beta;\widehat\sigma^2)=\left\{
		\beta'\in \mathbb R^{p+1};(\widehat\beta-\beta')^\top{{\mathfrak s}}(\widehat\beta-\beta')\leq (p+1)\widehat\sigma^2{{\bm{\textsf f}}}_{p+1,n-p-1,1-\delta} 
		\right\}.
	\end{equation}
	In other words, the random 
	ellipsoidal region $\mathcal U_{n,p,\delta}(\widehat\beta;\widehat\sigma^2)$, which is fully specified by $n$, $p$, $\delta$ and the sample data $\widehat\beta$ and $\widehat\sigma^2$, covers the true vector parameter $\beta$ with probability $1-\delta$. \qed
\end{example}

\begin{example}\label{simult:band}(Simultaneous confidence band for the mean response)
	According to \cite[page 27]{cook2009applied}, ``the primary goal in a regression analysis is to understand, as far as possible
	with the available data, how the conditional distribution of the response
	varies across sub-populations determined by the possible values of the predictors''. Precisely, and using the notation of Example \ref{mle:imp:lsm:0}, if 
	\[
	{\bm{\textsf R}}^p=\{\bm{\textsf x}'=(1,{x}'_1,\cdots,{x}'_p); { x}'_j\in\mathbb R, j=1,\cdots, p\}
	\] 
	and ${\bm{\textsf x}}\in\bm{\textsf R}^p$ is given, this amounts to checking how much information on 
	the conditioned random variable $\mathscr Y|_{\widetilde {\mathscr X}={\bm{\textsf  x}}}$ may be extracted from the data array $({\bf x},{\bf y})$. 
	We insist that ${\bm{\textsf x}}$ should be viewed as a future observation which has been consolidated from the same population {\em after} the data set has been drawn.
	From this perspective, the method of least squares in Example \ref{mle:imp:lsm}, according to which 
	\begin{equation}\label{resp:var}
		\mathscr Y_{\bm{\textsf x}}:=\mathscr Y|_{\widetilde {\mathscr X}={\bm{\textsf  x}}}=\bm{\textsf  x}^\top\beta+{\bm{\textsf e}}\,\, {\rm with}\,\, \mathbb E({\bm{\textsf e}})=0,
	\end{equation}
	represents a first step toward this goal as it  allows us to make use of the estimate $\widehat\beta$, which has been computed from $({\bf x},{\bf y})$, in order to retrieve information on the 
	{\em mean response} 
	\begin{equation}\label{mean:response}
		{\bm{\textsf x}}^\top\beta=\mathbb E(\mathscr Y_{\bm{\textsf x}}),
	\end{equation}
	a collection of summaries which, as ${\bm{\textsf x}}$ varies, exhausts the realizations of the conditional expectation $\mathbb E(\mathscr Y|{\widetilde{\mathscr X}})$ (by Proposition \ref{int:prob}). 
	Indeed, if we further specialize to the setting of Example \ref{mle:imp:lsm:n}, where 
	\begin{equation}\label{norm:pred:0}
		{\bm{\textsf e}}\sim\mathcal N(0,\sigma^2),
	\end{equation}
	then  ${\bm{\textsf x}}^\top\widehat\beta$ may be used
	 to properly estimate the 
	population parameter in (\ref{mean:response}): from (\ref{pivot:q:new})
	we have
	\begin{equation}\label{norm:pred}
		{\bm{\textsf x}}^\top(\widehat\beta-\beta)\sim \mathcal N({0},\sigma^2{\bm{\textsf x}}^\top\mathfrak s{\bm{\textsf x}}), \quad \mathfrak s=(\mathfrak x^\top\mathfrak x)^{-1}, 
	\end{equation}
	that is, 
	\[
	\frac{{\bm{\textsf x}}^\top(\widehat\beta-\beta)}{\sigma\sqrt{{\bm{\textsf x}}^\top\mathfrak s{\bm{\textsf x}}}}\sim\mathcal N(0,1),
	\]
	so that, by Propositions \ref{proc:conf:int:mls} and \ref{norm:chi:stu},
	\[
	\frac{{\bm{\textsf x}}^\top(\widehat\beta-\beta)}{\widehat\sigma
		\sqrt{{\bm{\textsf x}}^\top\mathfrak s{\bm{\textsf x}}}}
	\sim{{\mathfrak t_{n-p-1}}},
	\]
	and hence,
	\begin{equation}\label{int:conf:mresp}
		{\bm{\textsf x}}^\top\beta\in\left[
		{\bm{\textsf x}}^\top{\widehat\beta}\mp
		{{\mathfrak{t}}}_{n-p-1,1-\delta/2}{\widehat\sigma}
		\sqrt{{\bm{\textsf x}}^\top{{\mathfrak s}}{\bm{\textsf x}}}
		\right]
		\,\,\text{with prob.}\,\,1-\delta.
	\end{equation}
	If we allow for a bit more of spread, 
	a ``simultaneous'' version of this pointwise bound is also available, with the corresponding estimate holding for {\em any} $\bm{\textsf x}\in\bm{\textsf R}^n$,  as follows. 
	Starting with (\ref{est:w:beta}) and (\ref{est:w:beta:2}) and using the notation of Example \ref{conf:ref:beta},
	\begin{eqnarray*}
		1-\delta 
		& = & 
		P\left(\frac{(\widehat\beta-\beta)^\top
			{{\mathfrak s}}(\widehat\beta-\beta)}{(p+1)\widehat\sigma^2}\leq
		{{\bm{\textsf f}}}_{p+1,n-p-1,1-\delta}\right)\\
		& = &
		P\left(\frac{\|{\mathfrak n}\|}{\widehat\sigma/\sigma}\leq\sqrt{(p+1) {\bm{\textsf f}}_{p+1,n-p-1,1-\delta}}\right),
	\end{eqnarray*}
	and recalling that,  by Cauchy-Schwarz, 
	\[
	\sup_{{\bm{\textsf x}}\in\bm{\textsf R}^p}\frac{|(\mathfrak p{\bm{\textsf x}})^\top{\mathfrak n}|}{\|\mathfrak p{\bm{\textsf x}}\|\|\mathfrak n\|}=1,
	\]
	we get
	\begin{eqnarray*}
		1-\delta
		& = & 
		P\left(
		\sup_{{\bm{\textsf x}}\in\bm{\textsf R}^p}
		\frac{|(\mathfrak p{\bm{\textsf x}})^\top\mathfrak n|}{(\widehat\sigma/\sigma)\sqrt{(\mathfrak p{\bm{\textsf x}})^\top\mathfrak p{\bm{\textsf x}}}}
		\leq 
		\sqrt{(p+1) \bm{\textsf f}_{p+1,n-p-1,1-\delta}}
		\right)\\
		& = & 
		P\left(
		\sup_{{\bm{\textsf x}}\in\bm{\textsf R}^p}\frac
		{|{\bm{\textsf x}}^\top(\widehat\beta-\beta)|}{\widehat\sigma\sqrt{{\bm{\textsf x}}^\top\mathfrak s{\bm{\textsf x}}}}
		\leq 
		\sqrt{(p+1) \bm{\textsf f}_{p+1,n-p-1,1-\delta}}
		\right),
	\end{eqnarray*}
	so that 
	\begin{equation}\label{scheffe}
		{\bm{\textsf x}}^\top\beta\in\left[
		{\bm{\textsf x}}^\top\widehat\beta\mp
		\sqrt{(p+1) \bm{\textsf f}_{p+1,n-p-1,1-\delta}}\,\widehat\sigma\sqrt{{\bm{\textsf x}}^\top\mathfrak s{\bm{\textsf x}}}
		\right] \forall {\bm{\textsf x}}\in\bm{\textsf R}^p
		\,\,{\rm with}\,\,{\rm prob.}\,\,1-\delta, 
	\end{equation}
	As fully explained in \cite{liu2010simultaneous}, this Scheff\'e-type simultaneous confidence band plays a fundamental role in the inference theory of OLS models; see also Example \ref{scheffe:band} for a generalization thereof. \qed 
\end{example}

\begin{example}\label{pred:resp}
	(Simultaneous prediction band for the response)
	With the estimates for the mean response provided by (\ref{int:conf:mresp}) and (\ref{scheffe}) at hand, we may now pose ourselves the problem of ``predicting'' where the response itself,	
	\begin{equation}\label{resp:new:eq}
		\mathscr Y_{\bm{\textsf x}}=\bm{\textsf x}^\top\beta+\bm{\textsf e}\sim\mathcal N(\bm{\textsf x}^\top\beta,\sigma^2),
	\end{equation}	
	is likely to fall, where, as in Example \ref{simult:band}, we should think of $\bm{\textsf x}$ as a new observation for the regressor which has taken place after the data $({\bf x},{\bf y})$ has been gathered, hence the ``prediction'' terminology\footnote{This emphasis on response prediction foreshadows a defining tendency in Data Science: the prioritization of predictive accuracy, often at the cost of model interpretability; see Remark~\ref{int:acc} and the surrounding discussion.}. In particular, $\bm{\textsf e}$ is independent of $\widehat\beta$, which has been constructed out of $({\bf x},{\bf y})$, so that $\mathscr Y_{\bm{\textsf x}}$ is independent of $\widehat\beta$ as well. 
	Combining this with 
	(\ref{resp:new:eq}), (\ref{norm:pred}) and Proposition \ref{norm:spce} (3) we see that 
	\[
	{\mathscr Y}_{\bm{\textsf x}}-\bm{\textsf x}^\top\widehat\beta\sim\mathcal N\left(0,\sigma^2\left(1+{\bm{\textsf x}}^\top\mathfrak s{\bm{\textsf x}}\right)\right),
	\]
	which by the standard argument gives the  sought-after pointwise ``prediction interval'' for the response,
	\begin{equation}\label{pred:int:resp}
		\mathscr Y_{\bm{\textsf x}}\in\left[
		{\bm{\textsf x}}^\top\widehat\beta\mp\mathfrak t_{n-p-1,1-\delta/2}\widehat\sigma\sqrt{1+{\bm{\textsf x}}^\top\mathfrak s{\bm{\textsf x}}}
		\right]
		\,\,{\rm with}\,\,{\rm prob.}\,\,1-\delta,
	\end{equation}
	with this new terminology being adopted because
	the random variable $\mathscr Y_{\bm{\textsf x}}$ is {\em not} a pararameter, so this fails to be a confidence interval in the ordinary sense. In any case,  
	upon comparison with (\ref{int:conf:mresp}) we see that when passing from the mean response 
	to the response 
	itself, the point estimate ${\bm{\textsf x}}^\top\widehat\beta$ remains the same but the dispersion gets expanded by a factor that makes it at least as large as $\mathfrak t_{n-p-1,1-\delta/2}\widehat\sigma$, a lower bound  which depends on the already observed data $({\bf x},{\bf y})$ but {\em not} on the future observation $\bm{\textsf x}$ for the regressor. 
	Following \cite[Theorem 1]{carlstein1986simultaneous} and \cite[Theorem 1]{sadooghi1990simultaneous}, we may also contemplate a ``simultaneous'' version of (\ref{pred:int:resp}), which is obtained by means of an easy generalization of Scheff\'e's argument leading to (\ref{scheffe}). Indeed,  
	\[
	b=
	\left(
	\begin{array}{c}
		\widehat\beta-\beta\\
		{\bm{\textsf e}}	
	\end{array}
	\right)
	\sim
	\mathcal N
	\left(\vec{0},
	\left(
	\begin{array}{cc}
		\sigma^2\mathfrak s & 0\\
		0 & \sigma^2
	\end{array}
	\right)
	\right)
	\]		
	and 
	\[
	\overline{\mathfrak s}=\left(
	\begin{array}{cc}
		\mathfrak s & 0\\
		0 & 1
	\end{array}
	\right)
	\]
	are such that
	\[
	\sigma^{-2}b^\top{\overline{\mathfrak s}}b
	=  \sigma^{-2}(\widehat\beta-\beta)^\top{\mathfrak s}(\widehat\beta-\beta)+\sigma^{-2}\|\bm{\textsf e}\|^2
	\sim  \chi^2_{p+2},
	\]
	so that
	\[
	\frac{b^\top{\overline{\mathfrak s}}b}{(p+2)\widehat\sigma^2}\sim{\bm{\textsf F}}_{p+2,n-p-1},
	\]
	and hence
	\begin{eqnarray*}
		1-\delta
		&=&
		P\left(\frac{b^\top{\overline{\mathfrak s}}b}{(p+2)\widehat\sigma^2}\leq
		\bm{\textsf f}_{p+2,n-p-1,1-\delta}
		\right)
		\\
		&=&
		P\left(
		\frac{\|\overline{\mathfrak n}\|}{\widehat\sigma/\sigma}
		\leq \sqrt{(p+2)\bm{\textsf f}_{p+2,n-p-1,1-\delta}}
		\right),
	\end{eqnarray*}
	where $\overline{\mathfrak n}=\sigma^{-1}(\overline{\mathfrak p}^\top)^{-1}b$ with $\overline{\mathfrak p}$ satisfying $\overline{\mathfrak p}^\top\overline{\mathfrak p}=\overline{\mathfrak s}$; cf.\! the corresponding unbarred objects in Example \ref{conf:ref:beta}. 
	Thus, again using Cauchy-Schwarz,
	\begin{eqnarray*}
		1-\delta
		& = & 
		P\left(
		\sup_{{\bm{\textsf z}}\in\bm{\textsf R}^{p+1}}
		\frac{|(\overline{\mathfrak p}{\bm{\textsf z}})^\top\overline{\mathfrak n}|}{(\widehat\sigma/\sigma)\sqrt{(\overline{\mathfrak p}{\bm{\textsf z}})^\top\overline{\mathfrak p}{\bm{\textsf z}}}}
		\leq 
		\sqrt{(p+2) \bm{\textsf f}_{p+2,n-p-1,1-\delta}}
		\right)\\
		& = & 
		P\left(
		\sup_{{\bm{\textsf z}}\in\bm{\textsf R}^{p+1}}\frac
		{|{\bm{\textsf z}}^\top b|}{\widehat\sigma\sqrt{{\bm{\textsf z}}^\top\overline{\mathfrak s}{\bm{\textsf z}}}}
		\leq 
		\sqrt{(p+2) \bm{\textsf f}_{p+2,n-p-1,1-\delta}},
		\right),
	\end{eqnarray*}
	where $\bm{\textsf R}^{p+1}=\bm{\textsf R}^{p}\times\mathbb R$, so if we choose ${\bm{\textsf z}}=({\bm{\textsf x}},-1)$, where ${\bm{\textsf x}}\in \bm{\textsf R}^{p}$ is arbitrary, we get 
	$
	{\bm{\textsf z}}^\top b={\bm{\textsf x}}^\top\widehat\beta-\mathscr Y_{\bm{\textsf x}}
	$
	and ${\bm{\textsf z}}^\top\overline{\mathfrak s}{\bm{\textsf z}}=1+{\bm{\textsf x}}^\top{\mathfrak s}{\bm{\textsf x}}$,
	which finally gives
	\[
	\mathscr Y_{\bm{\textsf x}}
	\in\left[
	{\bm{\textsf x}}^\top\widehat\beta\mp
	\sqrt{(p+2) \bm{\textsf f}_{p+2,n-p-1,1-\delta}}\,\widehat\sigma\sqrt{1+{\bm{\textsf x}}^\top\mathfrak s{\bm{\textsf x}}}
	\right] \forall {\bm{\textsf x}}\in\bm{\textsf R}^p
	\,\,{\rm with}\,\,{\rm prob.}\,\,1-\delta
	\]
	as a simultaneous prediction band for the response. 
	\qed
\end{example}

\begin{remark}\label{coeff:det}(Coefficient of determination and goodness of fit)
	As illustrated in Figure~\ref{figg}, where the residual vector $\widehat{\mathbf e}$ lies orthogonally to $C(\mathfrak x)$, 
	the residual sum of squares
	\[
	SS_{\mathrm{Res}} := \|\widehat{\mathbf e}\|^2 = \sum_j (\mathbf Y_j - \widehat{\mathbf Y}_j)^2
	\]
	measures the amount of variation in $\mathbf Y$ that remains unexplained by the linear regression model 
	(whose fitted values form the vector $\widehat{\mathbf y} \in C(\mathfrak x)$). 
	Since, by Remark \ref{geom:mls} and Proposition \ref{rss:dist}, 
	\[
	({\mathbf Y}-\widehat{\mathbf Y})^\top(\widehat{\mathbf Y}-\overline{\mathbf Y}{\mathbf 1})=\widehat{\mathbf e}^\top{\widehat{\mathbf Y}}-\widehat{\mathbf e}^\top{\overline{\mathbf Y}}{\mathbf 1}=-\widehat{\mathbf Y}\sum_j\widehat{\mathbf e}_j=0, 
	\]
	$SS_{\mathrm{Res}}$ also appears in the {\em orthogonal} decomposition
	\begin{equation}\label{dec:ss}
		S_{\mathbf Y \mathbf Y} = SS_{\mathrm{Reg}} + SS_{\mathrm{Res}},
	\end{equation}
	where
	\[
	SS_{\mathrm{Reg}} = \sum_j (\widehat{\mathbf Y}_j - \overline{\mathbf Y})^2,
	\qquad
	S_{\mathbf Y \mathbf Y} = \sum_j (\mathbf Y_j - \overline{\mathbf Y})^2
	\]
	represent, respectively, the variation explained by the model and the total variation of $\mathbf Y$ about its mean 
	$\overline{\mathbf Y}$ (notation consistent with~\eqref{simp:lin:reg:3}).  
	Since $SS_{\mathrm{Res}} \le S_{\mathbf Y \mathbf Y}$, it is natural to define the 
	\emph{coefficient of determination}
	\begin{equation}\label{def:coeff:det}
		R^2 = \frac{SS_{\mathrm{Reg}}}{S_{\mathbf Y \mathbf Y}}
		= 1 - \frac{SS_{\mathrm{Res}}}{S_{\mathbf Y \mathbf Y}},
	\end{equation}
	as a measure of the model’s \emph{goodness of fit}: the closer $R^2$ is to its maximal value $1$, 
	the better the model accounts for the observed variation.  
		A simple justification of this interpretation may be given in the case of simple linear regression 
	(Example~\ref{simp:lin:reg}). Indeed,
	\begin{align*}
		\widehat{\mathbf e}
		&= \mathbf Y - (\widehat\beta_0 \mathbf 1 + \widehat\beta_1 \mathbf X) \\
		&= \mathbf Y - \big((\overline{\mathbf Y} - \widehat\beta_1 \overline{\mathbf X}) \mathbf 1 
		+ \widehat\beta_1 \mathbf X \big) \\
		&= \mathbf Y - \overline{\mathbf Y} \mathbf 1 
		- \widehat\beta_1 (\mathbf X - \overline{\mathbf X} \mathbf 1) \\
		&= \mathbf Y - \overline{\mathbf Y} \mathbf 1 
		- \frac{S_{\mathbf X \mathbf Y}}{S_{\mathbf X \mathbf X}} 
		(\mathbf X - \overline{\mathbf X} \mathbf 1),
	\end{align*}
	which yields
	\[
	SS_{\mathrm{Res}} 
	= \frac{S_{\mathbf X \mathbf X} S_{\mathbf Y \mathbf Y} - S_{\mathbf X \mathbf Y}^2}
	{S_{\mathbf X \mathbf X}}.
	\]
	Eliminating $SS_{\mathrm{Res}}$ from this expression and~\eqref{def:coeff:det} gives
	\[
	R^2 = \widehat\rho^2,
	\]
	where
	\[
	\widehat\rho = 
	\frac{S_{\mathbf X \mathbf Y}}
	{\sqrt{S_{\mathbf X \mathbf X} S_{\mathbf Y \mathbf Y}}}
	\]
	is the \emph{sample correlation coefficient} of the pair 
	$(\mathbf X, \mathbf Y)$ underlying the model (cf.~\eqref{samp:corr:def} and Example~\ref{mle:imp:lsm}).  
	It should also be mentioned that
	when comparing models with different numbers of regressors, it is often preferable to adjust for 
	the associated degrees of freedom, leading to the \emph{adjusted coefficient of determination}
	\[
	R^2_{\mathrm{adj}}
	= 1 - \frac{SS_{\mathrm{Res}}/(n-p-1)}{S_{\mathbf Y \mathbf Y}/(n-1)}
	= 1 - \frac{n-1}{n-p-1}(1 - R^2),
	\]
	a statistic that penalizes model complexity and provides a fairer basis for comparison.  
	It should be emphasized, however, that relying solely on $R^2$ or $R^2_{\mathrm{adj}}$ as measures of fit 
	can be misleading, since no confidence levels are inherently attached to them (though this can be remedied; see \cite[Section~10.5]{rencher2008linear}). Consequently, they are treated here as mere point statistics which are best used
	in conjunction with confidence intervals 
	for the parameters (as in~\eqref{conf:int:lsm:2}) and with the formal hypothesis-testing framework 
	developed in Section~\ref{sec:hyp:test}. As an illustration of this latter procedure, Example~\ref{ex:f:reg} below presents the classical $\mathsf F$-test 
	which 
	provides a much more rigorous assessment 
	of the overall statistical significance of the fitted model by
	checking whether the observed ratio
	\[
	\frac{SS_{\mathrm{Reg}}/p}{SS_{\mathrm{Res}}/(n-p-1)}
	= \frac{n-p-1}{p} \,
	\frac{SS_{\mathrm{Reg}}/S_{\mathbf Y \mathbf Y}}
	{SS_{\mathrm{Res}}/S_{\mathbf Y \mathbf Y}}
	\]
	exceeds the corresponding critical value.\qed
\end{remark}

\begin{remark}\!\!$\bigstar$\label{reg:to:m:2}
	(Regression to the mean, again)
	Using the notation of Example \ref{simp:lin:reg} and Remark \ref{coeff:det}, we see from (\ref{simp:lin:reg:2}) that the slope of the regression line
	is
	\begin{equation}\label{reg:to:m:3}
		\widehat\beta_1=\frac{\sqrt{S_{{\bf y}{\bf y}}}}{\sqrt{S_{{\bf x}{\bf x}}}}\widehat\rho,
	\end{equation}
	so that
	\[
	\widehat\beta_0=\overline{\bf y}-\widehat\rho\frac{\sqrt{S_{{\bf y}{\bf y}}}}{\sqrt{S_{{\bf x}{\bf x}}}}\overline{\bf x},
	\]
	and from (\ref{fit:vec:r}) we deduce that the fitted value is
	\[
	\widehat{\bf y}=\overline{\bf y}-
	\widehat\rho\frac{\sqrt{S_{{\bf y}{\bf y}}}}{\sqrt{S_{{\bf x}{\bf x}}}}({\bf x}-\overline{\bf x}).
	\]
	Equivalently,
	\begin{equation}\label{reg:to:m:4}
		\frac{\widehat{\bf y}-\overline{\bf y}}{{\sqrt{S_{{\bf y}{\bf y}}}}}=\widehat\rho\frac{{\bf x}-\overline{\bf x}}{{\sqrt{S_{{\bf x}{\bf x}}}}},
	\end{equation}
	and we conclude that, unless the sampling informs us that ${\bf x}$ and ${\bf y}$ are perfectly correlated ($|\widehat\rho|=1$), we should regard the appropriated standardization of  $\widehat{{\bf y}}$ as being strictly smaller (in absolute value) than the standardization of ${\bf x}$, a circumstance which certainly indicates a ``regression to the mean''; compare with Remark \ref{reg:to:m:1}. \qed
\end{remark}

\begin{remark}\label{cr:rao:app:lsm} ($\widehat\beta$ as a best unbiased estimator)
	It follows from (\ref{log:beta}) that the MLE for the parameter $\theta=(\beta,\sigma^2) $ of the regression linear model under error normality is $\widehat\theta=(\widehat\beta,\widehat{\sigma}^{2}_\bullet)$, where $\widehat\beta$ is the usual OLS estimator for $\beta$ in (\ref{mls:form:mul}), and 
	$\widehat{\sigma}^{2}_\bullet=\|\widehat{\bf e}\|^2/n$ with $\widehat{\bf e}={\bf Y}-\widehat{{\bf Y}}$
	being the residual as in (\ref{def:res:vec}), so that (\ref{g:m:prep:3}), (\ref{exp:var:res}) and the independence of $\{\widehat\beta,\widehat\sigma^2_\bullet\}$ lead to
	\[
	{\mathbb C}(\widehat\theta)=
	\left(
	\begin{array}{cc}
		{\mathbb C}(\widehat\beta) & 0\\
		0 & {\mathbb V}(\widehat\sigma^2_\bullet)
	\end{array}
	\right)
	=
	\left(
	\begin{array}{cc}
		\sigma^2 ({{\mathfrak x}}^\top{{\mathfrak x}})^{-1} & 0 \\
		0 & 2(n-p-1)\sigma^4/n^2
	\end{array}
	\right),
	\] 
	where we used that $\widehat\sigma^2_\bullet=(n-p-1)\widehat\sigma^2/n$ with $\widehat\sigma$ as in (\ref{pivot:eq:2}). 
	Also, again from (\ref{log:beta}) we easily compute
	\[
	l_{\beta\beta}=-\frac{\mathfrak x^\perp\mathfrak x}{\sigma^2}, \quad 
	l_{\sigma^2\sigma^2}=\frac{n}{2\sigma^4}-\frac{\|{\bf y}-\mathfrak x\beta\|^2}{\sigma^6}, \quad l_{\beta\sigma^2}=\frac{\mathfrak x^\perp({\bf y}-\mathfrak x\beta)}{\sigma^2},
	\]
so it follows from	$\sigma^{-2}({\bf y}-\mathfrak x\beta)\sim\mathcal N(\vec{0}, {\rm Id}_n)$, (\ref{fischer}) and Corollary \ref{sum:norm:sq}
that	
	\begin{equation}\label{fisher:mls}
		\mathscr F(\theta)=
		\left(
		\begin{array}{cc}
			\mathscr F(\beta) & 0 \\
			0 & \mathscr F(\sigma^2)
		\end{array}
		\right)
		=
		\left(
		\begin{array}{cc}
			\sigma^{-2} {{\mathfrak x}}^\top{{\mathfrak x}} & 0 \\
			0 & n/2\sigma^4
		\end{array}
		\right).
	\end{equation}
	Regarding this analysis, we observe that:
	\begin{itemize}
		\item Since the Cram\'er-Rao lower bound is attained for $\widehat\beta$,  Corollary \ref{cr:rao:th:cor} implies that it is the best {\em unbiased} estimator for $\beta$ (under error normality), which should be compared with Theorem \ref{gauss:markov} (Gauss-Markov), where normality is relaxed to (\ref{g:m:prep:2}) but the competing unbiased estimators are required to be linear.
\item We have 
\[
{\mathbb V}(\widehat\sigma^2_\bullet)=2\frac{n-p-1}{n^2}\sigma^4<\frac{2\sigma^4}{n}=\mathscr F(\sigma^2)^{-1},
\]
 a clear violation of the Cram\'er-Rao lower bound (\ref{cr:rao:th:3}), but of course this poses no contradiction to Theorem \ref{cr:rao:th} because $\widehat\sigma^2_\bullet$ is {\em not} unbiased.
 \item 	The unbiased estimator $\widehat\sigma^2$ for $\sigma^2$ satisfies 
 ${\mathbb V}(\widehat\sigma^2)=2\sigma^4/(n-p-1)>\mathscr F(\sigma^2)^{-1}$, corresponding to a strict inequality in (\ref{cr:rao:th:3}). 
 \item  If we view $\widehat\sigma^2_\bullet\in\mathcal E_g$, where 
 \[
 g(\sigma^2)=\mathbb E(\widehat\sigma^2_\bullet)=\frac{n-p-1}{n}\sigma^2,
 \]	
 then the strict inequality in (\ref{cr:rao:th:3:1}) holds.
		\end{itemize}
Therefore, both $\widehat\sigma^2$ and $\widehat\sigma^2_\bullet$ fail to realize the best estimator for $\sigma^2$ in their natural classes, which suggests the existence of another estimator with a better performance in each case. \qed		
\end{remark}

\begin{example}\label{aic:ols}
	(AIC for the normal linear model)
Starting from (\ref{aic:form:f}) we may now compute the AIC for the linear model discussed in Remark \ref{cr:rao:app:lsm}. Recalling that the ML estimator is  $\widehat\theta=(\widehat\beta,\widehat\sigma^2_{\bullet})$, we obtain 
\[
\begin{aligned}
	{\rm AIC}(\mathfrak X)
	&= -2\left(
	-\frac{n}{2}\ln(2\pi\widehat\sigma^2_{\bullet})
	-\frac{1}{2\widehat\sigma^2_{\bullet}}
	\|{\bf y}-\mathfrak x\widehat\beta\|^2
	\right) + 2(p+2) \\
	&= -2\left(
	-\frac{n}{2}\ln(\widehat\sigma^2_{\bullet})
	-\frac{n}{2}\ln(2\pi)
	-\frac{n}{2}
	\right) + 2(p+2),
\end{aligned}
\]	
and therefore
\[
	{\rm AIC}(\mathfrak X)=n\ln(\widehat\sigma^2_{\bullet})+ 
	2(p+2)+n\ln(2\pi)+n.
\]
Note that the last two terms do not depend on the model (only on the observed data) and  may be dropped for comparison purposes, as it is customary in model selection \cite{konishi2008information,hastie2009elements,burnham2013model}. Regarding this computation, we add two further remarks:
\begin{itemize}
\item In this normal linear model, the maximized log-likelihood term in the AIC represents, up to an additive constant, the Kullback–Leibler divergence $D^{KL}_{\widetilde\theta_{\mathfrak X}}(\widehat\theta)$ between the empirical distribution $\widetilde\theta_{\mathfrak X}$ of the sample and the fitted model $\widehat\theta$. By contrast, the penalty term $2(p+2)$ is independent of the observed sample and arises as an asymptotic correction accounting for the gap between empirical and population Kullback–Leibler risks induced by the estimation of  $p+2$ parameters; see Remark \ref{kl:fisher:aic} for a geometric interpretation of this phenomenon. 
\item Once the linear model, and hence the number of parameters, is fixed, the normal likelihood induces no further notion of goodness of fit beyond the residual sum of squares: up to additive constants, the maximized log-likelihood, and hence the first term of the AIC, depends on the data solely through $\widehat\sigma^2_{\bullet}$. This connects the information-theoretic interpretation of AIC with classical goodness-of-fit diagnostics: in the normal linear model, all likelihood-based criteria reduce, up to additive and penalization terms, to functions of the residual sum of squares, reflecting its role as a sufficient statistic for model adequacy; cf. Remark \ref{coeff:det}.
\end{itemize}
For a non-asymptotic alternative to information-criterion–based model selection, grounded in concentration inequalities as introduced in Section~\ref{conc:ineq:appl}, see~\cite{massart2007concentration}. \qed
	\end{example}

\begin{remark}\label{large:lsm}
	(Asymptotic normality of the OLS estimator)
	The ``small sample'' computations leading to the confidence interval estimates (\ref{conf:int:lsm:1}) and (\ref{conf:int:lsm:2})  rely heavily on the normality of the error and should be compared to the corresponding ``small sample'' estimates for the population mean $\mu$ in (\ref{conf:int:ls}) and (\ref{int:conf:tt}), respectively. Similarly to what occurred there, under the more general assumptions in (\ref{g:m:prep:2}) we must resort to the fundamental limit theorems in Section \ref{fund:lim} in order to establish the asymptotic normality of the LSM estimator $\widehat\beta$ from which ``large sample'' estimates should be retrieved. We use the  assumptions underlying the linear regression model in Example \ref{mle:imp:lsm} and conveniently decompose the Gram matrix as 
	\[
	\mathfrak X^\top\mathfrak X=\sum_{j=1}^n\mathfrak X_j^\top\mathfrak X_j,
	\] 
	with a similar expression holding for $\mathfrak X^\top{\bf e}$.
	Thus, 
	\[
	\sqrt{n}\left(\widehat\beta-\beta\right)=
	\left(\frac{1}{n}\sum_{j=1}^n\mathfrak X_j^\top\mathfrak X_j\right)^{-1}\sqrt{n}\left(\frac{1}{n}\sum_{j=1}^n\mathfrak X_j^\top{\bf e}_j\right).
	\]
	From LLN (Theorem \ref{lln}) we know that 
	\[
	\frac{1}{n}\sum_{j=1}^n\mathfrak X_j^\top\mathfrak X_j\stackrel{p}{\to} {\mathfrak C}:=\mathbb E(\mathfrak X_j^\top\mathfrak X_j),
	\]
	a symmetric and positive definite $(p+1)\times (p+1)$ random matrix. 
	Also, since $\mathbb E(\mathfrak X_j^\top{\bf e})=\vec{0}$ by (\ref{uncorr:e})
	and 
	\begin{eqnarray*}
		{\mathbb C}(\mathfrak X_j^\top{\bf e}_j) 
		& \stackrel{(\ref{total:var})}{=} & \mathbb E({\mathbb C}(\mathfrak X_j^\top{\bf e}_j|\mathfrak X))+{\mathbb C}(\mathbb E(\mathfrak X^\top_j{\bf e}_j|\mathfrak X))\\
		& = & \mathbb E(\mathfrak X^\top_j{\mathbb C}({\bf e}_j|\mathfrak X)\mathfrak X_j)\\
		& \stackrel{(\ref{homo:p})}{=} & \sigma^2\mathfrak C, 
	\end{eqnarray*}
	we may use CLT (Theorem \ref{clt}) to check that 
	\[
	\sqrt{n}\left(\frac{1}{n}\sum_{j=1}^n\mathfrak X_j^\top{\bf e}_j\right)\stackrel{d}{\to}\mathcal N\left(\vec{0},\sigma^2\mathfrak C\right).
	\]
	Combining these calculations with 
	Theorem \ref{slutsky} we conclude that, as $n\to+\infty$ and $p$ is held fixed, 
	\begin{equation}\label{tcl:lr}
		\sqrt{n}\left(\widehat\beta-\beta\right)\stackrel{d}{\to}\mathcal N\left(\vec{0},
		\sigma^2{\mathfrak C}^{-1}\right),
	\end{equation}
	so that $\widehat\beta$ is asymptotically normal (and hence consistent) with asymptotic covariance
	$\sigma^2{\mathfrak C}^{-1}$, which should be reliably estimated in order to obtain the desired confidence regions. We refer to \cite{amemiya1985advanced,hayashi2011econometrics} for full accounts of the estimation theory of the  linear regression model, including the justification of our somewhat sloppy use of LLN and CLT above (recall that the conditional distributions across observations in a linear model are independent but not identically distributed as they depend on the observed covariates). \qed
\end{remark}

\begin{remark}\label{mls:a:n:p}(Asymptotic normality for the linear regression model under error normality)
	Using the results of Remark \ref{cr:rao:app:lsm}, notably the computation of the corresponding Fisher information matrix in (\ref{fisher:mls}), we find that the
	ML estimator $(\widehat\beta,\widehat\sigma^2_\bullet)$ for the parameter $(\beta,\sigma^2)$ in  
	the linear regression model (under error normality) satisfies, as $n\to+\infty$ and $p$ is held fixed,  
	\[
	\sqrt{n}
	\left(
	\left(
	\begin{array}{c}
		\widehat\beta\\
		\widehat\sigma^2_\bullet	
	\end{array}
	\right)-
	\left(
	\begin{array}{c}
		\beta\\
		\sigma^2	
	\end{array}
	\right)
	\right)
	\stackrel{d}{\to}
	\mathcal N
	\left(
	\left(
	\begin{array}{c}
		{0}\\
		0	
	\end{array}
	\right),
	\left(
	\begin{array}{cc}
		\sigma^{2}	\mathfrak s_\bullet & 0 \\ 
		0 & 	2\sigma^4	
	\end{array}
	\right)
	\right),
	\]
	where
	\begin{equation}\label{mle:lim}
		\mathfrak s_\bullet=\left(\lim_{n\to+\infty}\frac{\mathfrak x^\top\mathfrak x}{n}\right)^{-1}.
		\end{equation}
	which establishes its asymptotic normality; here we use the generalization of Theorem \ref{asym:cr} described in Remark \ref{asym:cr:gen}.
	In particular, we have the asymptotic relations 
		\[
	\widehat\beta\approx \mathcal N\left(\beta,\frac{\sigma^2\mathfrak s_\bullet}{n}\right)\,\,{\rm and}\,\,\,
	\widehat\sigma^2_\bullet\approx \mathcal N\left(
	\sigma^2, \frac{2\sigma^4}{n}
	\right).
	\]
If, as usual, we appeal to consistency, then the first one yields
\begin{equation}\label{tcl:lr:n}
	\beta_j\in\left[\widehat\beta_j\mp z_{1-\delta/2}
	\frac{\widehat{\bf s}_{\bullet j}}{\sqrt{n}}\right]
	\,\,{\rm with}\,\,{\rm prob.}\,\,\approx 1-\delta, \quad \widehat{\bf s}_{\bullet j}=\widehat\sigma_\bullet\sqrt{\mathfrak s_{\bullet jj}},\quad 
	\widehat\sigma_\bullet=\frac{\|\widehat{\bf e}\|}{\sqrt{n}},
	\end{equation}
	which aligns with (\ref{conf:int:lsm:1}) because $\widehat{\bf s}_{\bullet j}/\sqrt{n}\approx {\bf s}_j$ by (\ref{mle:lim}) and Proposition \ref{proc:conf:int:mls},
	while the second one gives
	\[
	\sigma^2\in\left[\widehat\sigma^2_\bullet\mp z_{1-\delta/2}
	\sqrt{\frac{2}{n}}
	{\widehat\sigma_\bullet^2}\right]\,\,{\rm with}\,\,{\rm prob.}\,\,\approx 1-\delta,
	\]
	a large sample confidence interval for $\sigma^2$.
	  \qed
\end{remark}

\begin{example}\!\!$\bigstar$\label{super:learn}(Supervised Statistical Learning as a generalization of regression)
	Let $X\in\mathbb R^p$ be a random vector of covariates and $Y\in\mathcal Y\subset\mathbb R$ a response variable and  
	assume that $(X,Y)$ follows an unknown distribution $P_{(X,Y)}$ on $\mathbb R^p\times\mathcal Y$, and that we observe i.i.d.\ samples $(X_1,Y_1),\dots,(X_n,Y_n)$ drawn from $P_{(X,Y)}$. 
	Given a loss function $\ell:\mathcal A\times\mathcal Y\to\mathbb R$ and a model (or hypothesis) class $\mathcal F$ of measurable predictors $f:\mathbb R^p\to\mathcal A$, where $\mathcal A\subset \mathbb R$ denotes the prediction (or action) space, one seeks to minimize the {\em population risk}
	\[
	R(f)=\mathbb E\left( \ell(f(X),Y)\right),
	\]
	giving rise to the ``within $\mathcal F$'' optimal predictor
	\[
	f^*_{\mathcal F}={\rm argmin}_{f\in\mathcal F}R(f).
	\]
	Since $P_{(X,Y)}$ is unknown, a direct assessment of $R(f^*_{\mathcal F})$ is out of the question. We therefore proceed in the usual way: we replace $R(f)$ by the {\em empirical risk} 
	\[
	\widehat R_n(f)=\frac1n\sum_{i=1}^n \ell(f(X_i),Y_i),
	\]
	and consider empirical risk minimization over $\mathcal F$:
	\begin{equation}\label{emp:risk:0}
	\widehat f_n:={\rm argmin}_{f\in\mathcal F}\widehat R_n(f).
	\end{equation}
	The fundamental question is to determine conditions ensuring that the empirical minimizer $\widehat f_n$ is {\em consistent} in the sense that, as $n\to+\infty$,
	\begin{equation}\label{emp:risk}
	R(\widehat f_n)\to R(f^*_{\mathcal F}),
	\end{equation}
	say in probability.
	This abstract formulation encompasses the two main paradigms of Supervised Learning:
	\begin{itemize}
		\item \textbf{Binary classification.} 
		If $\mathcal Y=\{-1,1\}$ and $\ell(f(x),y)=\mathbf 1_{\{f(x)\neq y\}}$, 
		the risk reduces to the misclassification probability 
		\[
		R(f)=P(f(X)\neq Y),
		\]
		with the optimal predictor being the {\em Bayes classifier} 
		\[
		f^*(x)=\operatorname{sign}(\mathbb E(Y|_{X=x}));
		\]
		see \cite{devroye1996,vapnik1998statistical}.
		
		\item \textbf{Regression.} 
		If $\mathcal Y\subset\mathbb R$ and $\ell(f(x),y)=(y-f(x))^2$, 
		the risk becomes the mean squared error, which, as we have seen in 
		Remark \ref{reg:func:choice},
		is
		 minimized by the regression function 
		\[
		f^*(x)=\mathbb E(Y|_{X=x}).
		\]
	\end{itemize}
	Here, $f^*=f^*_{\mathcal F_{\rm all}}$, where ${\mathcal F_{\rm all}}$ is the set of {\em all} measurable $f:\mathbb R^p\to \mathcal Y$\footnote{Thus, it might well happen that $R(f^*)< R(f^*_{\mathcal F})$ for a given $\mathcal F\subsetneq \mathcal F_{\rm all}$; for more on this ``approximation error'', see Remark \ref{bias:var:sl}.}.
	In order to investigate the ``within $\mathcal F$'' consistency formulated in (\ref{emp:risk}), we consider the decomposition
	\[
	0\leq R(\widehat f_n)-R(f^*_{\mathcal F})
	=
	\bigl(R(\widehat f_n)-\widehat R_n(\widehat f_n)\bigr)
	+
	\bigl(\widehat R_n(\widehat f_n)-\widehat R_n(f^*_{\mathcal F})\bigr)
	+
	\bigl(\widehat R_n(f^*_{\mathcal F})-R(f^*_{\mathcal F})\bigr),
	\]
	and note that the second term is non-positive by~(\ref{emp:risk:0}), while the last term is $o(1)$ with high probability by LLN (Theorem~\ref{lln}) and the fact that
	\[
	\mathbb E\bigl(\ell(f_{\mathcal F}^*(X_i),Y_i)\bigr)=R(f^*_{\mathcal F}),
	\]
	which also shows that $\widehat R_n(f^*_{\mathcal F})$ is an unbiased estimator of the population quantity $R(f^*_{\mathcal F})$, since $f^*_{\mathcal F}$ is non-random\footnote{If $f\in\mathcal F$ is such that $\ell(f(X),Y)\in {\bf{\mathsf{SubG}}}(\sigma)$ as in Definition \ref{sub:g:def:0} then the Hoeffding-type concentration inequality in (\ref{rade:g}) yields the high probability bound
		\begin{equation}\label{vc:bound:0}
		P\left(\left|\widehat R_n(f)-R(f)\right|\leq 
		\sigma\sqrt{\frac{2\ln(2/\delta)}{n}}\right)\geq 1-\delta,
		\end{equation}	
	which clearly applies to the classification setting above (with $f=f^*_{\mathcal F}$), where we can take $\sigma=1$ by Proposition \ref{subg:bounded}.
	}.
	On the other hand, since $\widehat f_n$ is data-dependent, LLN cannot be applied directly to the first term, and we are  thus left with the rather crude bound 
	\[
	R(\widehat f_n)-\widehat R_n(\widehat f_n)
	\le
	\sup_{f\in\mathcal F}\bigl|R(f)-\widehat R_n(f)\bigr|.
	\]
	This analysis makes it clear that
	the central difficulty in establishing (\ref{emp:risk}) stems from the fact that, although $R_n(f)\stackrel{p}{\to} R(f)$ for each fixed $f\in\mathcal F$ by LLN, consistency of empirical risk minimization requires a \emph{uniform} law of large numbers over the model class $\mathcal F$. 
	In the classification setting, this requirement is met by Vapnik--Chervonenkis (VC) theory: if $\mathcal F$ has finite VC dimension $h_{\rm VC}$ then
	\begin{equation}\label{vc:bound}
	P\left(\sup_{f\in\mathcal F}|\widehat R_n(f)-R(f)|
	\lesssim 
	\sqrt{\frac{h_{\rm VC}\ln n+\ln(1/\delta)}{n}}\right)\geq 1-\delta,
	\end{equation}
	thus ensuring consistency; see \cite{vapnik1998statistical, bousquet2003introduction,mohri2018foundations,wainwright2019high}. 
Thus, when passing from the pointwise bound in (\ref{vc:bound:0}) to the uniform bound in (\ref{vc:bound}), one gains an extra term involving $h_{\rm VC}$, which reflects the complexity of the model class $\mathcal F$.
	However, VC theory is intrinsically combinatorial and adapted to indicator-valued losses: in many other instances of Supervised Learning 
	(such as regression, classification with convex surrogate losses, and contemporary high-dimensional models such as neural networks) the model classes are real-valued and often heavily overparameterized, so that combinatorial shattering dimensions no longer capture the geometric structure governing generalization. For example, for linear regression with model function $f_\beta(x)=x^\top\beta$ in dimension $p+1$ as in Example \ref{mle:imp:lsm}, the combinatorial measure of complexity corresponding to $h_{\rm VC}$ is known as {\em pseudo-dimension}, which equals $p+1$; this  yields uniform deviation bounds of order $\sqrt{p/n}$ under suitable boundedness assumptions, thus recovering the already explored classical requirement that the sample size should substantially dominate the number of parameters ($p\ll n$). In particular, this bound becomes uninformative in the high dimensional setting of Subsection \ref{ridge} below, where $p\gg n$.
	Consequently, more refined complexity measures (covering numbers, Rademacher complexities, among others) have been developed to extend uniform convergence beyond the classical VC framework (see \cite{anthony2009neural,bartlett2021deep}).
	More recently, the need to understand the behavior of highly overparameterized models, particularly deep neural networks, has led to the emergence of a broader mathematical perspective combining ideas from approximation theory, high-dimensional probability, optimization, and statistical learning theory; see \cite{berner2023modern,kutyniok2022ai} for surveys on the emerging mathematical foundations of deep learning and artificial intelligence.
	\qed
\end{example}

\begin{remark}\!\!$\bigstar$\label{bias:var:sl}(Bias-variance trade-off in Statistical Learning)
	Using the notation of Example \ref{super:learn}, we may decompose the excess risk over $\mathcal F_{\rm all}$ as
	\[
	R(\widehat f_n)- R(f^*)=
	\left(	R(f^*_{\mathcal F})-R(f^*)\right)
	+	\left(R(\widehat f_n)- R(f^*_{\mathcal F})\right),
	\]
	where the last term, the {\em estimation error}, was briefly analyzed there (at least in the classification case) and reflects the variability due to the random sampling process. 
	On the other hand, the first term
	is the {\em approximation error}, as it quantifies the loss incurred by restricting the search to $\mathcal F$, and vanishes only if $f^*\in\mathcal F$.
	This decomposition provides an integrated (risk-based) formulation of the classical {\em bias-variance trade-off} from regression theory. 
	Indeed, in the general regression setting with squared loss of Example \ref{mle:imp:lsm:0} one has, for any $f\in{\mathcal F_{\rm all}}$,
	\[
	R(f)-R(f^*)= 
	\mathbb E\left((Y-f(X))^2\right)-
		\mathbb E\left((Y-f^*(X))^2\right)
	= \mathbb{E}_{P_X}\left((f(X)-f^*(X))^2\right),
	\]
where the last identity follows from the standard orthogonality argument in the population counterpart of Remark \ref{reg:func:choice}, reflecting the fact that $f^*=f^*_{\mathcal F_{\rm all}}$.
	Consequently,
	\begin{equation}\label{app:term}
	R(f^*_{\mathcal F})-R(f^*)=\mathbb{E}_{P_X}\left((f^*_{\mathcal F}(X)-f^*(X))^2\right)=\|f^*_{\mathcal F}-f^*\|_{L^2(P_X)}^2,
	\end{equation}
	which corresponds to the squared bias induced by restricting the regression model to $\mathcal F$ and may be symbolically represented as $\mathbb{E}_{P_X}(\mathrm{bias}(X)^2)$ in order to emphasize the comparison with the classical pointwise formulation.
	On the other hand,
	applying the same orthogonality argument within
	 $\mathcal F$ we get
	\[
	R(\widehat f_n)-R(f^*_{\mathcal F})=\mathbb{E}_{P_X}\left((\widehat f_n(X)-f^*_{\mathcal F}(X))^2\right),
	\]
	but notice that, differently from (\ref{app:term}), which is purely deterministic, this term is data-dependent, so if we 
	take a further expectation with respect to the sample, we end up with
	\[
	\mathbb{E}\left(R(\widehat f_n)-R(f^*_{\mathcal F})\right)
	=\mathbb{E}_{P_X}\left(\mathbb V(\widehat f_n(X))\right),
	\]
	which corresponds to the variance term. 
	Equivalently, by conditioning these relations on $X=x$, thereby isolating the randomness induced by the training sample, we recover the classical pointwise bias-variance decomposition, whereas the present formulation corresponds to its integrated version with respect to the marginal distribution $P_X$. These computations show that, in this more general framework of Statistical Learning, the two effects are captured respectively by the approximation error (bias)  and the estimation error (variance).
	In particular, enlarging the class $\mathcal F$ typically reduces the approximation error, since a richer class can better approximate the optimal predictor $f^*$, but increases the estimation error, as more complex models are harder to learn from finite samples (overfitting). Conversely, restricting $\mathcal F$ stabilizes estimation at the cost of introducing a larger bias (underfitting). The central problem of Statistical Learning is therefore to balance these two competing sources of error by judiciously choosing $\mathcal F$; see \cite[Subsection 2.2.2]{james2013introduction} and \cite[Subsection 2.4]{von2011statistical} for further discussions, including illustrative diagrams  of this remarkable trade-off.
\end{remark}

\begin{example}\!\!$\bigstar$\label{rem:pca}
	(Principal component analysis and random matrices) Given a random vector $X\in\mathbb R^p$, can we determine a direction 
	$u\in\mathbb S^{p-1}$ that captures the maximum possible variance of the projection $u^\top X$? Since
	\begin{equation}\label{pop:pca}
	{\mathbb V}(u^\top X)=u^\top \mathscr C u,
	\end{equation}
	and $\mathscr C={\mathbb C}(X)$ is symmetric, there exist $p\times p$ matrices $O$ and $D$, 
	with $O$ orthogonal and $D={\rm diag}(\lambda_1,\ldots,\lambda_p)$, such that
	\[
	\mathscr C=ODO^\top.
	\]
	Clearly, we may reorder the columns of $O=(O_1,\ldots,O_p)$ so that
	$\lambda_1\ge\cdots\ge\lambda_p$. Setting $v=O^\top u$, $\|v\|=1$, we obtain
	\[
	u^\top \mathscr C u
	= v^\top D v
	= \sum_{j=1}^p \lambda_j v_j^2
	\le \lambda_1,
	\]
	so that $\lambda_1$ is an upper bound for the variances. Moreover, if $\{{\rm e}_j\}_{j=1}^p$ is the canonical basis of $\mathbb R^p$,
	\[
	\mathscr C O_j
	= ODO^\top O_j
	= OD{\rm e}_j
	= \lambda_j O{\rm e}_j
	= \lambda_j O_j,
	\]
	which means that $\lambda_j$ is the $j^{\rm th}$ eigenvalue of $\mathscr C$, with $O_j$
	as a corresponding eigenvector. In particular,
	$O_j^\top\mathscr C O_j=\lambda_j$, showing that the maximal variance is achieved by projecting $X$ onto $O_1$.
	Unfortunately, this analysis relies on the spectral decomposition of $\mathscr C$, which is not directly observable since the distribution of $X$ is unknown.
	In practice, one therefore replaces $\mathscr C$ by its empirical counterpart. Thus, given i.i.d.\ observations
	$X_1,\ldots,X_n\in\mathbb R^p$ drawn from the distribution of $X$, we set
	\begin{equation}\label{non:centered}
	\overline X_n=\frac{1}{n}\sum_{i=1}^n X_i,
	\qquad
	\widehat{\mathscr C}_n({\bf X})=\frac{1}{n}\sum_{i=1}^n (X_i-\overline X_n)(X_i-\overline X_n)^\top,
	\end{equation}
	so that $\widehat{\mathscr C}_n({\bf X})$ is the sample covariance matrix (with the usual maximum likelihood normalization).
	As already anticipated by the notation, here we proceeded as in Example~\ref{mle:imp:lsm:0} by arranging these observations in an $n\times p$ matrix
	${\bf X}$, the \emph{design matrix}\footnote{Be aware, however, that there is no response vector ${\bf Y}$ here: in modern language, we are passing from Supervised to Unsupervised Learning; see \cite{hastie2009elements}, where this transition is discussed in detail.}, so that
	\begin{equation}\label{eq:def:sigman}
	\overline X_n=\frac{1}{n}{\bf X}^\top{\bf 1},
	\end{equation}
	where ${\bf 1}$ is the $n$-vector whose entries all equal $1$. Also, 
	\begin{eqnarray*}
		\widehat{\mathscr C}_n({\bf X})
		& =&
	\frac{1}{n}\left(
	\sum_iX_iX_i^\top-\left(\sum_iX_i\right)\overline  X_n^\top
	-\overline  X_n\sum_iX_i^\top+\overline  X_n\overline  X_n^\top
	\right)	\\
	& = & 
	\frac{1}{n}\left(
	{\bf X}^\top{\bf X}-\overline  X_n\overline  X_n^\top
	\right),
	\end{eqnarray*}
	and using (\ref{eq:def:sigman}),
	\begin{equation}\label{eq:sigma:pca}
		\widehat{\mathscr C}_n({\bf X})=\frac{1}{n}{\bf X}^\top {\mathscr H}{\bf X},
	\end{equation}
	where
	\[
	{\mathscr H}={\rm Id}_n-\frac{1}{n}{\bf 1}{\bf 1}^\top:\mathbb R^n\to\mathbb R^n
	\]
	defines a projection operator, as it satisfies
	\begin{equation}\label{eq:proj:pca}
		{\mathscr H}^\top={\mathscr H}
		\qquad\text{and}\qquad
		{\mathscr H}^2={\mathscr H}.
	\end{equation}
		If $w\in\mathbb R^n$ then
	\begin{equation}\label{eq:proj:pca:2}
		{\mathscr H}w = 
		w-\frac{1}{n}{\bf 1}{\bf 1}^\top w
		= w-\frac{1}{n}(w^\top{\bf 1}){\bf 1}
		=w-\overline w\,{\bf 1},
	\end{equation}
	where $\overline w=w^\top{\bf 1}/n$ denotes the average of the entries of $w$, which shows that the range of ${\mathscr H}$
	is the orthogonal complement of the line generated by ${\bf 1}$.
	Consequently, the matrix $\mathscr H{\bf X}$ contains the centered observations
	$X_i-\overline X_n$ as its rows. Combining~\eqref{eq:sigma:pca} and~\eqref{eq:proj:pca} we obtain that
	\[
	\widehat{\mathscr C}_n({\bf X})=\frac{1}{n}({\mathscr H}{\bf X})^\top{\mathscr H}{\bf X}
	\]
	is the covariance matrix of the centered sample, therefore confirming that principal component analysis  is performed on the centered data ${\mathscr H}{\bf X}$. Thus, we may assume without loss of generality that 
	$X$ is centered ($\mathbb E(X)=\vec{0}$).
It follows that
	\[
	u^\top \widehat{\mathscr C}_n({\bf X}) u
	=\frac{1}{n}({\mathscr H}{\bf X}u)^\top {\mathscr H}{\bf X}u,
	\]
	and using~\eqref{eq:proj:pca:2} with $w={\bf X}u$,
	\begin{equation}\label{var:emp:obs}
	u^\top \widehat{\mathscr C}_n({\bf X}) u
	=\widehat\sigma^2_{n^{-1}}(u^\top {\bf X^\top}),
	\end{equation}
	which identifies the quadratic form ${\mathscr Q}_n(u)$ associated to $\widehat{\mathscr C}_n({\bf X})$ with 
	the sample variance of the projected observations
	$(u^\top X_1,\ldots,u^\top X_n)$; here we use the notation in (\ref{fam:est}).
	This allows us to interpret the direction maximizing the empirical variance of the projections
	\[
	u\in\mathbb S^{p-1}\longmapsto {\mathscr Q}_n(u)
	\]
	as the one along which the cloud of points in $\mathbb R^p$ defined by the sample
	$(X_1,\ldots,X_n)$ spreads the most. As usual, this direction is given by an eigenvector
	$\widehat u_1$ associated with the largest eigenvalue of $\widehat{\mathscr C}_n({\bf X})$; this vector
	is called the \emph{first principal component direction}, and the corresponding maximal
	empirical variance ${\mathscr Q}_n(\widehat u_1)$ equals that largest eigenvalue $\widehat\lambda_1$.
	More generally, ordering the eigenvalues of $\widehat{\mathscr C}_n({\bf X})$ as
	$\widehat\lambda_1\ge\cdots\ge\widehat\lambda_p$, with orthonormal eigenvectors
	$\widehat u_1,\ldots,\widehat u_p$, yields the usual {\em principal component decomposition},
	with the random variable $\widehat u_j^\top X$ providing the $j^{\rm th}$ {\em principal component score}. Moreover, by (\ref{var:emp:obs}) 
we see that
\[
\widehat\lambda_j=\widehat u_j^\top \widehat{\mathscr C}_n({\bf X}) \widehat u_j
=\widehat\sigma^2_{n^{-1}}(\widehat u_j^\top {\bf X^\top})
\]
equals the sample variance of the observed scores 
	$(\widehat u_j^\top X_1,\ldots,\widehat u_j^\top X_n)$, 
	thus measuring the amount of sample variability explained along that principal direction. This analysis may be complemented as follows. Since  $\widehat{\mathscr C}_n({\bf X})=\widehat O\,\widehat D\,\widehat O^\top$, where $\widehat O=(\widehat u_1,\ldots,\widehat u_p)$ and  $\widehat D={\rm diag}(\widehat\lambda_1,\ldots,\widehat\lambda_p)$, we may use the centered version of (\ref{non:centered}) to find that the sample covariance of the empirical principal coordinates 
	$
	\widehat Y_i=\widehat O^\top X_i\in\mathbb R^p$,
	$ i=1,\ldots,n$,
	is
	\[
	\mathscr C_n(\widehat{\bf Y}):=
	\frac{1}{n}\sum_{i=1}^n \widehat Y_i \widehat Y_i^\top 
	=
	\widehat O^\top
	\left(
	\frac{1}{n}\sum_{i=1}^n X_i X_i^\top
	\right)
	\widehat O 
	= 	\widehat O^\top \widehat{\mathscr C}_n({\bf X}) \widehat O=\widehat D, 
	\]
	so we conclude that the empirical principal component scores are
	uncorrelated and their sample variances are precisely the eigenvalues
	$\widehat\lambda_j$\footnote{In the population setting the corresponding scores $Y_j=u_j^\top X$
	are random variables, and the obvious relation $\mathbb C(Y_j,Y_l)=\delta_{jl}\lambda_l$
	is defined through expectation with respect to the distribution of $X$.
	In the empirical setting, however, we only observe a finite sample, so the identity
	$\mathscr C_n(\widehat{\mathbf Y})=\widehat D$
	simply expresses the orthogonality of the score vectors under the empirical
	inner product $\widehat\sigma^2_{n^{-1}}(\cdot,\cdot)$ defined in (\ref{emp:inner}).
	Thus empirical principal component scores are uncorrelated in this sample sense by construction. This algebraic fact should not be confused with independence, which only possibly holds at the population level
	under additional distributional assumptions (e.g. Gaussianity).}.
This empirical assessment of the population projected variations in~\eqref{pop:pca},
	via the spectral decomposition of $\widehat{\mathscr C}_n({\bf X})$, allows us, for a given $k<p$,
	to project the sample $(X_1,\ldots,X_n)$ onto the subspace spanned by
	$\widehat u_1,\ldots,\widehat u_k$, which maximizes the explained empirical
	variation among all $k$-dimensional subspaces of $\mathbb R^p$\footnote{
		In applications one often takes $k\le 3$ for visualization purposes; 
		see \cite{jolliffe2002pca}.
		Also, in Supervised Learning problems the first $k$ principal component scores may be used as predictors, leading to {\em principal component regression},
		an approach that is often used to alleviate multicollinearity (since, as we have seen in the previous footnote, the principal component scores are orthogonal by construction) and to reduce the dimensionality of the design matrix before fitting a regression model, while still retaining most of the variability present in the original predictors; see \cite{hastie2009elements}.
	}. 
	Since 
		\[
	\frac{1}{n}\|\mathscr H{\bf X}\|_{\rm Frob}^2=\mathrm{tr}(\widehat{\mathscr C}_n({\bf X}))=\sum_{j=1}^p \widehat\lambda_j
	\]
	means that the total mean squared spread of the cloud of points, as measured by the normalized squared Frobenius norm of $\mathscr H{\bf X}$,  equals the total sample variance, and also taking into account that the above mentioned projection
	corresponds to replacing 
$\widehat{\mathscr C}_n({\bf X})=\sum_{j=1}^p\widehat\lambda_j\widehat u_j\widehat u_j^\top$ by its truncated version 
\[\widehat{\mathscr C}_{n,k}({\bf X}):=
{\rm argmin}_{{\rm rank}(A)\leq k}\|\widehat{\mathscr C}_n({\bf X})-A\|^2_{\rm Frob}
=\sum_{j=1}^k\widehat\lambda_j\widehat u_j\widehat u_j^\top,
\]
	we see that 
	the amount of variation retained by such a projection, or equivalently the
	proportion of the total sample
	variability captured by the first $k$ principal components, is given by
		\[
		\frac{\mathrm{tr}(\widehat{\mathscr C}_{n,k}({\bf X}))}{\mathrm{tr}(\widehat{\mathscr C}_n({\bf X}))}=
	\frac{\sum_{j=1}^k \widehat\lambda_j}{\sum_{j=1}^p \widehat\lambda_j}.
	\]
However, applying this simple empirical criterion for determining the proportion of explained variability requires that we
evaluate the efficiency of $\widehat{\mathscr C}_n({\bf X})$ as a spectral
estimator of the population covariance matrix $\mathscr C$.
More precisely, 
the eigenvalues and eigenvectors used in principal component
analysis are obtained from the random matrix
$\widehat{\mathscr C}_n({\bf X})$, so their statistical behavior depends
on the fluctuations of this matrix around its population counterpart
$\mathscr C$. Understanding how these spectral quantities behave as
the sample size grows therefore leads naturally to the study of
random matrices and their asymptotic properties.
Regarding this crucial issue, we mention the following well-known facts:
\begin{itemize}
	\item
	If $p$ is fixed, then LLN (Theorem \ref{lln}) implies that
	$\widehat{\mathscr C}_n({\bf X})$ is a consistent estimator of $\mathscr C$, in the sense that
	\[
	\widehat{\mathscr C}_n({\bf X}) \stackrel{p}{\longrightarrow} \mathscr C
	\qquad\text{as } n\to\infty.
	\]
	Moreover, if $\mathbb E\|X\|^4<\infty$, then CLT (Theorem \ref{clt})  implies that
	\[
	\sqrt{n}\!\left(\widehat{\mathscr C}_n({\bf X})-\mathscr C\right)
	\stackrel{d}{\longrightarrow} \mathcal N(\vec{0},\mathscr V),
	\]
	where the asymptotic covariance matrix $\mathscr V={\mathbb C}(X\otimes X)$ 
	depends on the fourth moments of $X$.
	In case $X$ is Gaussian this reduces to the classical Wishart asymptotics; see,
	for instance, \cite[Chapter~3]{anderson2003introduction} and
	\cite[Chapter~3]{muirhead2009aspects};
	\item
	If we move to the ``high dimensional'' setting and let $p=p_n$ grow with $n$ then, as expected, the situation changes dramatically.
	Indeed, as a nice application of the machinery of concentration inequalities, it is proved in
	\cite[Section~4.7]{vershynin2018high} that
	\[
	\|\widehat{\mathscr C}_n({\bf X})-\mathscr C\|_{\rm op}
	= O\!\left(
	\|\mathscr C\|_{\rm op}
	\left(\sqrt{\frac{p}{n}}+\frac{p}{n}\right)
	\right),
	\]
with high probability, 	
where $X$ is assumed
	to be sub-Gaussian. Thus, consistency in operator norm is retained only
	if $p/n\to 0$. Also, if $p$ is held fixed then we get
	\[
	\sqrt{n}	\|\widehat{\mathscr C}_n({\bf X})-\mathscr C\|_{\rm op}=O(1).
	\]
	which aligns with the asymptotic normality of $\widehat{\mathscr C}_n$ mentioned in the previous item;
	\item
	When $p/n$ does not vanish asymptotically, $\widehat{\mathscr C}_n({\bf X})$ is no longer a reliable
	spectral estimator of $\mathscr C$, and its eigenvalues and eigenvectors may
	exhibit substantial deviations from those of $\mathscr C$.
	In particular, in the regime $p/n \to c \in (0,\infty)$, the empirical spectral
	distribution of $\widehat{\mathscr C}_n({\bf X})$ converges to the Marchenko-Pastur law,
	illustrating the intrinsic high-dimensional bias of the sample covariance
	matrix~\cite{bai2010spectral}.
\end{itemize}
The results mentioned in the items above only scratch the surface of the fascinating field of Random Matrix Theory, the branch of Probability and Mathematical Physics that investigates the distributional and asymptotic behavior of eigenvalues and eigenvectors of matrices with random entries, particularly in high-dimensional regimes, and that now plays a fundamental role in statistics, high-dimensional inference, number theory, wireless communications, and machine learning. For an accessible and mathematically rigorous introduction to this subject, we refer to \cite{tao2012topics}.
\qed 
	\end{example}

\subsection{Regularization in high dimension, sparsity and the LASSO}\label{ridge} 

The presence of the root-squared diagonal terms $\sqrt{\mathfrak s_{jj}}$, 
where $\mathfrak s=({\mathfrak x}^\top\mathfrak x)^{-1}$, in the confidence 
interval estimates above tends to increase the spread when ${\mathfrak x}^\top\mathfrak x$ 
is ill-conditioned, for instance when the ratio between its extremal eigenvalues 
is excessively large. In addition, the analysis above relies crucially on the 
assumption that $p+1\leq n$, which makes it inapplicable in the high-dimensional 
regime where $p\gg n$ and ${\mathfrak x}^\top\mathfrak x$ is no longer invertible.  
A possible way to address this limitation is to adopt the regularized regression estimator
\[
\widehat\beta_\lambda={\rm argmin}_\beta \widehat{\mathscr L}_\lambda(\beta),
\]
where
\[
\widehat{\mathscr L}_\lambda(\beta)=\tfrac{1}{2}\|{\bf y}-{\mathfrak x}\beta\|^2+\lambda\|\beta\|^2, 
\quad \lambda>0.
\]
This leads to the explicit solution
\[
\widehat\beta_\lambda=({\mathfrak x}^\top{\mathfrak x}+2\lambda I)^{-1}{\mathfrak x}^\top{\bf y},
\]
which is well defined even if ${\mathfrak x}^\top{\mathfrak x}$ does not have full column rank. 
This justifies the designation \emph{ridge regularization} for this approach 
\cite{hastie2015statistical,hastie2009elements,wainwright2019high,lederer2022fundamentals}.  
Moreover, under the conditions of Proposition \ref{g:m:prep} one obtains
\[
\mathbb E(\widehat\beta_\lambda)=({\mathfrak x}^\top{\mathfrak x}+2\lambda I)^{-1}
{\mathfrak x}^\top{\mathfrak r}\beta
\]
and
\[
{\mathbb C}(\widehat\beta_\lambda)=\sigma^2({\mathfrak x}^\top{\mathfrak x}+2\lambda I)^{-1}
{\mathfrak x}^\top{\mathfrak x}
({\mathfrak x}^\top{\mathfrak x}+2\lambda I)^{-1}.
\]
Although $\widehat\beta_\lambda$ is not unbiased, these expressions show that there exists 
$\lambda_0>0$ such that ${\rm mse}(\widehat\beta_\lambda)<{\rm mse}(\widehat\beta)$ for 
$0<\lambda<\lambda_0$ \cite{theobald1974generalizations}; see also Remark \ref{comp:mse}, 
where a similar effect is described for the variance estimators $\widehat\sigma^2_c$, $c>0$. 
Since ridge regression and its many variants are widely employed in practice, this confirms that 
a small amount of bias is acceptable when it comes with a significant reduction in variance. 
As explained below, starting with Remark \ref{int:acc}, this principle connects naturally with the geometric viewpoint developed earlier in Remark \ref{geom:mls}; see also Figure \ref{figg}.

\begin{remark}\label{int:acc}(Dichotomy between model interpretability and prediction accuracy in linear models) In the classical regime $p<n$, 
the least squares solution decomposes the response vector ${\bf y}$ into two orthogonal pieces: 
the projection $\widehat{\bf y}=H{\bf y}$ onto the column space $C(\mathfrak x)$, and the residual 
$\widehat{\bf e}=({\rm Id}_n-H){\bf y}$ lying in its orthogonal complement $C(\mathfrak x)^\perp$. 
These two components correspond to two distinct, complementary features of the OLS. The projection 
onto $C(\mathfrak x)$ carries the {\em predictive} content, since it represents the systematic 
variation in the response explained by the regressors. The orthogonal complement, by contrast, 
provides the basis for {\em inference} and {\em interpretability}: it isolates the random fluctuation 
not captured by the model, and this separation underpins our ability to quantify uncertainty, 
construct confidence intervals, and perform tests of significance (as in Section \ref{sec:hyp:test} below). In particular, each estimated 
coefficient inherits a transparent meaning: the expected change in the response for a unit change in 
the corresponding predictor, holding others fixed. Thus, prediction is geometrically tied to the 
column space, while interpretability rests on the existence of its orthogonal counterpart.  
\end{remark}
In the classical setting where $p\ll n$, we have already observed an emphasis on predictive accuracy in Examples~\ref{simult:band} and \ref{pred:resp}, where confidence bands for the (mean) response were constructed. This tendency becomes even more pronounced as the number of predictors increases to the point where $p\geq n$, at which stage the clean dichotomy described in Remark~\ref{int:acc} breaks down. Indeed, the column space of $\mathfrak x$ expands to fill $\mathbb R^n$, so that every response vector lies in it and the residuals vanish identically. 
In this regime, prediction not only persists but may interpolate the training data (the rows of $\mathfrak x$) exactly, while the residual space disappears. In the absence of a nontrivial orthogonal complement, the classical geometric foundation for inference and interpretability collapses, and the usual tools based on unexplained variation cease to apply. 
Regularization then emerges as a new source of geometry. Rather than relying on a residual subspace, ridge regression constrains the parameter vector itself, yielding numerically stable and statistically robust estimates by shrinking the coefficients toward zero, thereby reducing variance at the cost of a controlled bias. In this way, high-dimensional regression—central to data science practice—may be viewed as a migration of geometry: from projections in sample space, where prediction and inference were cleanly separated, to constraints in parameter space, where stability and a different form of interpretability are achieved through shrinkage.

A modern expression of this principle is given in Example \ref{high:sp:lasso}, which discusses 
the prediction properties of the LASSO procedure, introduced in \cite{tibshirani1996regression} and now 
widely used in Data Science \cite{hastie2009elements,james2013introduction,wainwright2019high,lederer2022fundamentals}. 
In contrast to ridge, the LASSO not only shrinks coefficients but also drives many of them exactly to zero, 
effectively selecting a subset of variables. This sparsity reintroduces a strong element of interpretability: 
the model highlights which predictors truly matter, while ignoring the rest. To emphasize the underlying geometric migration, this material is preceded by prediction bounds for the classical low-dimensional regime 
($p \ll n$), presented in Examples \ref{limitation}, \ref{bound:prob}, and \ref{bound:prob:2}, where the error distribution is considered under increasingly relaxed assumptions.

\begin{example}\label{limitation} (High probability bounds for the prediction error under normality)
	In addition to the parameter recovery methods already discussed (based on the construction of confidence intervals for the unknown parameter $\beta$),
we may also look at 
	\begin{equation}\label{rbeta:e:0}
		{\mathfrak x}\widehat\beta-{\mathfrak x}\beta={\mathfrak x}({\mathfrak x}^t{\mathfrak x})^{-1}{\mathfrak x}^\top{\bf e},
	\end{equation}
	where we assume as always that $p\leq n$ and $\mathfrak x$ has full column rank   and hence ${\mathfrak x}^\top{\mathfrak x}$ is invertible\footnote{Here and in the rest of this subsection we will assume,   
without loss of generality, that the intercept vanishes, so that $\beta_0=0$, $\mathfrak x=({\mathfrak x}_1, \cdots, {\mathfrak x}_p)$, where $\mathfrak x_j$ is the $j^{\rm th}$ column of $\mathfrak x$, $j=1,\cdots,p$, etc.}. 
	In the notation of Remark \ref{geom:mls},  
	\begin{equation}\label{rbeta:e}
		{\mathfrak x}\widehat\beta-{\mathfrak x}\beta=H{\bf e}={\bf e}-\widehat{\bf e}\in C(\mathfrak x),
	\end{equation}
	the difference  between the true error and the residual.
	In other words, rather than  paying attention to the projection of ${\bf e}$ onto $C(\mathfrak x)^\perp$ under $Q={\rm Id}_{n}-H$, which defines the residual $\widehat{\bf e}$, we now focus on its  projection onto $C(\mathfrak x)$ under $H$; in Figure \ref{figg}, ${\mathfrak x}\widehat\beta-{\mathfrak x}\beta$ is represented by $\widetilde{\bf e}$, so that 
	\begin{equation}\label{pred:er:def}
		\|\widetilde{\bf e}\|^2=\|{\mathfrak x}\widehat\beta-{\mathfrak x}\beta\|^2
	\end{equation} 
	is usually termed the {\em prediction error}.
	Under the normality assumption ${\bf e}\sim \mathcal N(\vec{0},\sigma^2{\rm Id}_{n})$, it follows from (\ref{rbeta:e}) and rotational invariance that  
	\[
	\sigma^{-2}\|{\mathfrak x}\widehat\beta-{\mathfrak x}\beta\|^2\sim \chi^2_{p},	
	\]	
	so 
	if 
	\begin{equation}\label{av:risk:1}
		\widetilde{{\rm mse}}(\mathfrak x\widehat\beta)=\frac{{\rm mse}(\mathfrak x\widehat\beta)}{n}
	\end{equation}
	is the {\em average prediction risk} then 
	\begin{equation}\label{av:risk:2}
		\widetilde{{\rm mse}}(\mathfrak x\widehat\beta)
		=\frac{\sigma^2p}{n},
	\end{equation}
	where we used Corollary \ref{chi:sq:ms}\footnote{ 
		A justification for adopting (\ref{av:risk:2}) as a measure of accuracy for the {prediction error} in (\ref{pred:er:def}) appears in  Remark \ref{empiric} below.
	}.
	Of course,  Markov's inequality (\ref{markov:ineq:2}) allows us to pass from this ``expectation bound'' to the corresponding ``high probability bound'',
	\begin{equation}\label{av:risk:3}
		P\left(\frac{\|\mathfrak r\widehat\beta-\mathfrak r\beta\|^2}{n}\leq \frac{\sigma^2}{\delta}\frac{p}{n}\right)\geq 1-\delta, \quad \delta>0.
	\end{equation}
	As expected, 
	this analysis only provides satisfactory prediction results for the classical OLS  if either $\sigma^2$, which is assumed known, is very small or $p\ll n$.
	\qed
\end{example}

\begin{remark}\label{empiric}
	We have seen in Example \ref{mle:imp:lsm} that statistical reasoning demands that the OLS estimator should be specified by solving the minimization problem
	\begin{equation}\label{emp:risk:deff}
		\widehat\beta={\rm argmin}_\beta \widehat{\mathscr L}(\beta),\quad \widehat{\mathscr L}(\beta)=\frac{1}{n}\|{\bf y}-\mathfrak x\beta\|^2.
	\end{equation}
	At least if $n$ is large, we may argue\footnote{Say, by ``freezing'' $\beta$ and applying the LLN.} that this is the ``empirical'' version of the more fundamental minimization problem
	\[
	\beta_{\rm m}:={\rm argmin}_\beta \mathscr R(\beta)
	\]
	with
	\[
	\mathscr R(\beta)=\mathbb E(\|{\bf Y}-\mathfrak X\beta\|^2)
	\]
	being the associated {\em risk function}; cf.\!\! Example \ref{super:learn}. Since $\mathfrak X\beta_{\rm m}$ geometrically  corresponds to the orthogonal projection of ${\bf Y}$ onto the subspace generated by the columns of $\mathfrak X$, we easily see that ${\bf e}_{\rm m}:={\bf Y}-\mathfrak X\beta_{\rm m}$ satisfies 
	\begin{equation}\label{orth:cond}
		\mathbb E(\mathfrak X^\top{\bf e}_{\rm m})=0, \quad 
	\end{equation}
	so that
	\begin{eqnarray*}
		\mathscr R(\widehat\beta)
		& = & \mathscr R(\beta_{\rm m}+\widehat\beta-\beta_{\rm m})\\
		& = & \mathbb E(\|{\bf Y}-{\mathfrak X}(\beta_{\rm m}+\widehat\beta-\beta_{\rm m})\|^2)\\
		& = & \mathbb E(\|{\bf e}_{\rm m}-{\mathfrak X}(\widehat\beta-\beta_{\rm m})\|^2)\\
		& \stackrel{(\ref{orth:cond})}{=} & \mathbb E(\|{\bf e}_{\rm m}\|^2)+\mathbb E(\|{\mathfrak X}(\widehat\beta-\beta_{\rm m})\|^2),
	\end{eqnarray*}
	which gives
	\[
	\mathbb E(\|{\mathfrak X}(\widehat\beta-\beta_{\rm m})\|^2)=	\mathscr R(\widehat\beta)-\mathscr R(\beta_{\rm m}). 
	\]
	If we replace $\beta_{\rm m}$ by $\beta$ (to comply with the notation of Example \ref{limitation}) and condition on $\mathfrak X=\mathfrak x$ we see that 
	\[
	\widetilde{{\rm mse}}({\mathfrak x}\widehat\beta)=\frac{\mathscr R(\widehat\beta)-\mathscr R(\beta)}{n},
	\]
	which justifies the terminology employed in (\ref{av:risk:1}).\qed
\end{remark}

\begin{example}\label{bound:prob} 
	(High probability bounds for the prediction error without normality)
	The calculations leading to the expectation and high probability bounds  in (\ref{av:risk:2}) and (\ref{av:risk:3}) rely heavily on the usual normality assumption on the error. It turns out that we may still obtain a quite effective high probability bound for the prediction error $\|{\mathfrak x}\widehat\beta-{\mathfrak x}\beta\|^2$ in (\ref{pred:er:def}) by merely assuming that, besides (\ref{homo:p}), the errors $\{{\bf e}_j\}_{j=1}^n$ are assumed to be independent and 
	{\em sub-Gaussian} in the sense that
	\begin{equation}\label{sub:g:pred}
		\mathbb E\left(e^{{\bf e}_ju}\right)\leq e^{\sigma^2u^2/2}, \quad u\in\mathbb R,
	\end{equation}
	so that ${\bf e}_j\in{\mathsf{SubG}}(\sigma)$ as in
	Definition \ref{sub:g:def:0}.
	The key point is that (\ref{rbeta:e:0}) leads to 
	\begin{equation}\label{prediction}
		\|{\mathfrak x}\widehat\beta-{\mathfrak x}\beta\|^2=\|D(D^\top D)^{-1}D^\top U^\top{\bf e}\|^2,
	\end{equation}
	where $UDV^\top$ is a singular value decomposition for $\mathfrak x$ (in particular, $U$ and $V$ are both orthogonal). Using that $D$ is diagonal, it is not hard to check that, under these conditions,
	\[
	D(D^\top D)^{-1}D^\top=
	\left(
	\begin{array}{cc}
		I_{p\times p} & {} \\
		{} & 0_{(n-p)\times (n-p)}
	\end{array}
	\right),
	\]
	which gives
	\begin{equation}\label{exp:sub:g}
		\|{\mathfrak x}\widehat\beta-{\mathfrak x}\beta\|^2=\sum_{j=1}^{p}|(U^\top{\bf e})_j|^2.
	\end{equation}
	Using that
	\[
	(U^\top{\bf e})_j=\sum_k U_{kj}{\bf e}_k, \quad \sum_k U_{kj}^2=1,
	\]
	(\ref{sub:g:pred}) and the independence one easily verifies that
	\[
	\mathbb E\left(e^{(U^\top{\bf e})_ju}\right)\leq e^{\sigma^2u^2/2},
	\]
	that is, 
	each $\sigma^{-1}(U^\top{\bf e})_j\in{\mathsf{SubG}}(1)$,
	and from (\ref{exp:sub:g}) we find that
	$\sigma^{-2}\|{\mathfrak x}\widehat\beta-{\mathfrak x}\beta\|^2\in{\mathsf{SubE}}(\nu,1)$ is sub-exponential as in Definition \ref{sub:exp}; see Remarks \ref{sub:related} and \ref{sum:sub:exp}. Using the concentration inequalities in Proposition \ref{tail:sub:exp} we thus conclude that
	\begin{eqnarray}
		P\left(
		\frac{\|\mathfrak r\widehat\beta-\mathfrak r\beta\|^2}{n}
		\leq  
		t\sigma^2\frac{p}{n}\right)
		& = & 
		P\left({\sigma^{-2}\|{\mathfrak x}\widehat\beta-\mathfrak x\beta\|^2}\nonumber
		\leq 
		p t\right)\\
		& \geq & 
		1-2e^{-t/2},\label{bound:prob:est}
	\end{eqnarray}
	for $t\geq\nu^2$, which morally corresponds to (\ref{av:risk:3}) under the replacement $t\to\delta^{-1}$. \qed
\end{example}

\begin{example}\label{bound:prob:2}
	(High probability bounds for the prediction error without normality, again)
	An estimate similar to (\ref{bound:prob:est}) may be obtained under the  more general assumptions of Example \ref{mle:imp:lsm}), where no further knowledge of the error distribution is available besides (\ref{homo:p}).
	As we shall see, this ignorance will be counterbalanced by a precise control on the spectrum of the modified Gram matrix $\widehat{\bm \Sigma}:={\mathfrak x}^\top{\mathfrak x}/n$, which is known to be positive definite.  
	As in Remark \ref{empiric}, we identify $\beta_{\rm m}$ to $\beta$ and explore the variational characterization of $\widehat\beta$ in (\ref{emp:risk:deff}) to get $\widehat{\mathscr L}(\widehat\beta)\leq \widehat{\mathscr L}(\beta)$, which means that  
	\[
	\frac{\|{\bf y}-\mathfrak x\widehat\beta\|^2}{n}\leq \frac{\|{\bf e}\|^2}{n}.
	\]  
	If we set ${\bf y}=\mathfrak x\beta+{\bf e}$ in the right-hand side, expand the square and cancel out the terms which are quadratic in the errors we get 
	\begin{equation}\label{pred:ef:error}
		\frac{\|\mathfrak r\widehat\beta-\mathfrak r\beta\|^2}{n}\leq 2\frac{({\mathfrak r}^\top{\bf e})^\top(\widehat\beta-\beta)}{n}\leq 2\frac{\|{\mathfrak r}^\top{\bf e}\|}{n}
		\|\widehat\beta-\beta\|,
	\end{equation}
	where Cauchy-Schwarz has been used in the last step.
	If $\lambda_{{\rm min}}(\widehat{\bm \Sigma})\leq \lambda_{{\rm max}}(\widehat{\bm \Sigma})$ stand for the (positive) extremal eigenvalues of $\widehat{\bm \Sigma}$
	then we have
	\begin{equation}\label{st:conv}
		\frac{\|\mathfrak x\widehat\beta-\mathfrak r\beta\|^2}{{n}}
		=
		\langle \widehat{\bm\Sigma}(\widehat\beta-\beta),\widehat\beta-\beta\rangle
		\geq{{\lambda_{\rm min}(\widehat{\bm \Sigma})}}\|\widehat\beta-\beta\|^2, 
	\end{equation}
	which may be viewed as a control on the sample correlation between the columns of $\mathfrak x$ (because $n\widehat{\bm\Sigma}_{jk}=\mathfrak x_j^\top\mathfrak x_k=\|\mathfrak x_j\|\|\mathfrak x_k\|{\rm corr}(\mathfrak x_j,\mathfrak x_k)$; cf (\ref{samp:corr:def})),
	so if we
	combine these estimates we get 
	\[
	\frac{\|\mathfrak r\widehat\beta-\mathfrak r\beta\|^2}{n}\leq 4 
	\frac{\|{\mathfrak r}^\top{\bf e}\|^2}{n^2\lambda_{\rm min}(\widehat{\bm \Sigma})}.
	\]
	On the other hand, again using our standing assumptions (including (\ref{homo:p})) we compute
	\begin{eqnarray*}
		\mathbb E\left(\|{\mathfrak r}^\top{\bf e}\|^2\right)
		& = & 
		\mathbb E\left({\rm tr}\left((\mathfrak r^\top{\bf e})
		(\mathfrak r^\top{\bf e})^\top\right)
		\right)\\
		& = & 
		{\rm tr}\,{\mathbb C}\,(\mathfrak r^\top{\bf e})\\
		& = &
		\sigma^2{\rm tr}(\mathfrak r^\top\mathfrak x) \\
		& \leq & 
		\sigma^2pn\lambda_{\rm max}(\widehat{\bm \Sigma}), 
	\end{eqnarray*}
	which gives the expectation bound
	\begin{equation}\label{variab:est}
		{\widetilde{{\rm mse}}({\mathfrak x}\widehat\beta)}\leq 
		4\sigma^2\lambda(\widehat{\bm \Sigma})\frac{p}{n}, \quad \lambda(\widehat{\bm \Sigma}):=\frac{\lambda_{\rm max}(\widehat{\bm \Sigma})}{\lambda_{\rm min}(\widehat{\bm \Sigma})},
	\end{equation}
	from which we obtain the high probability bound
	\begin{equation}\label{bound:prob:est:2}
		P\left(\frac{\|\mathfrak r\widehat\beta-\mathfrak r\beta\|^2}{n}\leq \frac{4\sigma^2}{\delta}\lambda(\widehat{\bm \Sigma})\frac{p}{n}\right)\geq 1-\delta, \quad \delta>0,
	\end{equation}
	again via Markov.
	\qed 
\end{example}

Although its derivation requires only mild assumptions on the error distribution, 
the high-probability bound in (\ref{bound:prob:est:2}) remains essentially similar 
to (\ref{bound:prob:est}) and (\ref{av:risk:3}). In particular, its explicit 
dependence on the dimensional ratio $p/n$ shows that, without further control of 
the error variance $\sigma^2$ and of the condition number $\lambda(\widehat{\bm \Sigma})$, 
the linear model can be trusted only when $p\ll n$.  
Outside this regime, for instance when $p<n$ but  $p\approx n$ with $n$ large, OLS 
faces at least two well-known deficiencies: {\em high variability} (while 
$\mathfrak x\widehat\beta$ is unbiased, variance estimates such as (\ref{variab:est}) 
fail to provide reliable control), and {\em low interpretability} (the sheer number 
of predictors obscures the identification of variables truly relevant for explaining 
the response).  
A natural remedy is to introduce a penalization term into the classical model, 
as in Example \ref{ridge} on ridge regression; see \cite[Introduction]{lederer2022fundamentals} 
for a useful overview of this approach\footnote{This kind of regularization has become a cornerstone of 
Supervised Learning, where it is crucial to determine on which side of the threshold 
$p\approx n$ a given problem lies 
\cite{donoho2000high,hastie2009elements,belloni2011high,buhlmann2011statistics,
	hastie2015statistical,frigessi2016some,vershynin2018high,wainwright2019high,lederer2022fundamentals}.}.  
As the next example shows, the situation becomes even more delicate in the 
high-dimensional regime $p\gg n$, where in particular the key correlation assumption 
in (\ref{st:conv}) breaks down, since $\mathfrak x^\top\mathfrak x$ is no longer invertible.

\begin{example}\label{high:sp:lasso}
	(High dimensionality, sparsity and the LASSO)
The discussion in the previous paragraph suggests regularizing a suitable multiple of the least squares objective function in order to restore interpretability in case $p\gg n$. When employing the $L^1$ norm of the vector parameter $\beta$, this gives rise to the {\em LASSO estimator}
	\[
	\widehat\beta_L={\rm argmin}_{\beta'}f_L(\beta'),\quad \widetilde{\mathscr L}_L(\beta')=\frac{1}{2n}\|{\bf y}-\mathfrak x\beta'\|^2+\lambda\|\beta'\|_1,
	\] 
	where $\lambda> 0$ is a tuning parameter to be chosen later and 
	\[
	\|\beta'\|_1=\sum_{j=1}^{p}|\beta'_j|. 
	\]
	Since $\widetilde{\mathscr L}_L(\widehat\beta_L)\leq \widetilde{\mathscr L}_L(\beta)$, where $\beta$ is the true  parameter appearing in the model equation ${\bf y}=\mathfrak x\beta+{\bf e}$, we thus get with a help from H\"older inequality,  
	\begin{eqnarray*}
		\frac{1}{n}\|\mathfrak x\widehat\beta_L-\mathfrak x\beta)\|^2
		& \leq & 
		\frac{2}{n}
		(\mathfrak x^\top{\bf e})^\top(\widehat\beta_L-\beta)+2\lambda\left(\|\beta\|_1-\|\widehat\beta_L\|_1\right)\\
		& \leq & 
		\frac{2}{n}
		\|\mathfrak x^\top{\bf e}\|_{\infty}\|\widehat\beta_L-\beta\|_1
		+2\lambda\left(\|\beta\|_1-\|\widehat\beta_L\|_1\right),
	\end{eqnarray*}
	an estimate which should be compared 
	to (\ref{pred:ef:error}), with
	its right-hand side effectively disentangling the  
	contributions coming from the ``effective error'' $2\|\mathfrak x^\top{\bf e}\|_{\infty}/n$ and the penalization. 
	Now, sparsity enters the game precisely to handle this latter term, as it  contemplates the belief, substantiated by an ``omniscient oracle'', 
	that a considerable portion of regressors
	may be dispensed with, so the corresponding parameter entries may be set to vanish. Precisely, there exists $S\subsetneq\{1,\cdots,p\}$ with $s:=\sharp S\ll n$ such that $\beta_j=0$ exactly when $j\notin S$. 
	Thus, if $\beta_S$  is the ``restriction'' of $\beta$ to $S$, so that $\beta=\beta_S+\beta_{S^c}$, and setting $\widehat\delta=\widehat\beta_L-\beta$, we have 
	\begin{eqnarray*}
		\|\beta\|_1-\|\widehat\beta_L\|_1
		& = & \|\beta_S\|_1-\|\beta+\widehat\delta\|_1\\
		& = & \|\beta_S\|_1-\|\beta_S+\widehat\delta_S+\widehat\delta_{S^c}\|_1\\
		& = & \|\beta_S\|_1-\|\beta_S+\widehat\delta_S\|_1
		-\|\widehat\delta_{S^c}\|_1\\
		& \leq & \|\widehat\delta_S\|_1
		-\|\widehat\delta_{S^c}\|_1,
	\end{eqnarray*}
	which gives
	\[
	\frac{1}{n}\|\mathfrak x\widehat\delta\|^2\leq \frac{2}{n}
	\|\mathfrak x^\top{\bf e}\|_{\infty}\|\widehat\delta\|_1+2\lambda\left(\|\widehat\delta_S\|_1
	-\|\widehat\delta_{S^c}\|_1\right),
	\] 
	so if we further assume that  
	the tuning parameter dominates the
	``effective error''  according to  
	\begin{equation}\label{key:lambda}
		\frac{2}{n}
		\|\mathfrak x^\top{\bf e}\|_{\infty}
		\leq\lambda
	\end{equation}
	we end up with 
	\begin{equation}\label{bas:ineq}
		\frac{1}{n}\|\mathfrak x\widehat\delta\|^2\leq \lambda\left(3\|\widehat\delta_S\|_1
		-\|\widehat\delta_{S^c}\|_1\right).
	\end{equation}
	As a direct consequence of this basic inequality we see that
	\[
	\widehat\delta\in \mathcal C(S):=\left\{
	\beta'\in \mathbb R^{p+1}; \|\beta'_{S^c}\|_1\leq 3\|\beta'_{S}\|_1
	\right\},
	\]
	which suggests that the appropriate replacement for (\ref{st:conv}) is to assume, for some $\kappa>0$, that 
	\begin{equation}\label{re:cond}
		\frac{1}{n}\|\mathfrak x\beta'\|^2\geq \kappa\|\beta'\|^2, \quad \beta'\in \mathcal C(S). 
	\end{equation}
	Under this {\em restricted eigenvalue (RE)} condition,
	\begin{eqnarray*}
		\frac{1}{n}\|\mathfrak x\widehat\delta\|^2
		& \stackrel{(\ref{bas:ineq})}{\leq} & 3\lambda\|\widehat\delta_S\|_1 \\
		& \leq & 3\lambda\sqrt{s}\|\widehat\delta\|\\
		& \stackrel{(\ref{re:cond})}{\leq}&
		\frac{3\lambda\sqrt{s}}{\sqrt{\kappa n}} \|\mathfrak x\widehat\delta\|,
	\end{eqnarray*}
	which finally gives the bound
	\begin{equation}\label{lasso:bound}
		\frac{\|\mathfrak x\widehat\beta_L-\mathfrak x\beta\|^2}{n}\leq \frac{9\lambda^2s}{\kappa}. 
	\end{equation}
	In order to estimate in terms of $\lambda$ the probability of the event in (\ref{key:lambda}), to which the validity of (\ref{lasso:bound}) is conditioned,
	let us assume for simplicity that ${\bf e}\sim\mathcal N(\vec{0},\sigma^2{\rm Id}_{n\times n})$. By the projection property in (\ref{norm:space:4}),
	\[
	2\frac{\mathfrak x^\top_k{\bf e}}{n}\sim\mathcal N\left(0,4\frac{\sigma^2}{n}\left\|\frac{\mathfrak x_k}{\sqrt{n}}\right\|^2\right), \quad k=1,\cdots,p,
	\]  
	so if the columns of the design matrix are normalized so that
	\[
	\left\|\frac{\mathfrak x_k}{\sqrt{n}}\right\|\leq C,
	\] 
	the standard Gaussian concentration inequality in (\ref{exp:bound:2}) leads to
	\begin{eqnarray*}
		P\left(
		2\left\|
		\frac{\mathfrak r^\top{\bf e}}{n}
		\right\|_{\infty}\leq\lambda\right)
		& \geq & 
		1-2p
		e^{-\frac{n\lambda^2}{8C^2\sigma^2}}\\
		& = & 
		1-2
		e^{-\frac{n\lambda^2}{8C^2\sigma^2}+\ln p}.
	\end{eqnarray*}
	This gives
	\[
	P\left(2
	\left\|
	\frac{\mathfrak r^\top{\bf e}}{n}
	\right\|_{\infty}\leq\lambda\right)\geq 1-2e^{-\frac{t^2}{2}}
	\]
	if 
	\[
	\lambda^2=8C^2\sigma^2\left(\frac{\ln p}{n}+\frac{t^2}{2n}\right), 
	\]
	so with this choice of $\lambda$, (\ref{lasso:bound}) 
	immediately yields the bound 
	\begin{equation}\label{bound:lasso:fin}
		\frac{\|\mathfrak x\widehat\beta_L-\mathfrak x\beta\|^2}{n}\leq
		\frac{72C^2\sigma^2}{\kappa}\frac{s}{n}\left({\ln p}+\frac{t^2}{2}\right)
	\end{equation}
	with at least the same probability. 
	Upon comparison with (\ref{bound:prob:est:2}) and not taking into account certain structural constants, we have been able to replace the dimensional ratio $p/n$ by
	$s\ln p/n$, which is linear in the ``sparsity index'' $s=\|\beta\|_0$ and scales logarithmically with $p$, added to another term which is driven by the ``oracle rate'' $sn^{-1}=o(1)$. Thus, it suffices to take $n\gg s\ln p>s$ in order to have LASSO's prediction nearly as accurate as if $S={\rm supp}\,\beta$, whose elements classify the relevant regressors, was known a priori. We mention that similar estimates hold true under much weaker assumptions on the error\footnote{For instance, if the error is sub-Gaussian then the corresponding concentration inequalities in Section \ref{conc:ineq:appl} might be useful.} and even for other kinds of penalizations; we refer to \cite[Chapter 6]{buhlmann2011statistics}, \cite[Chapter 11]{hastie2015statistical}, \cite[Chapter 7]{wainwright2019high} and \cite[Chapter 6]{lederer2022fundamentals} 
	for such generalizations and, more importantly, for the heuristics behind the crucial RE condition in (\ref{re:cond}) above. Finally, the practical question remains of fine-tuning the parameter $\lambda$ so as to obtain the right balance between variability and interpretability. In this regard, the feasibility of the most adopted procedure, cross-validation, is theoretically confirmed in \cite{chetverikov2021cross}, where it is shown that, under suitable conditions, its use only adds
	to the right-hand side of (\ref{bound:lasso:fin})
	a multiplicative factor which is $O(\sqrt{\ln pn})$, hence negligible for most realistic purposes. \qed
\end{example}

\begin{remark}\!\!$\bigstar$(From simultaneous response bands to distribution-free conformal prediction)\label{rem:scheffe:cp}
	Example \ref{pred:resp} shows that, within the Gaussian linear model, one may construct simultaneous prediction bands for the response in the sense that
	\[
	P\Big(Y(x)\in C(x)\ \text{for all }x\in S\Big)\ge 1-\delta,
	\]
	which provides a strong, uniform control of the response process \(x\mapsto Y(x)\) by means of the prediction sets $C(x)$, where $x$ varies over (possibly a subdomain \(S\) of) the covariate space. A key feature underlying this construction is the \emph{fixed design assumption}, whereby the covariates are treated as deterministic and all randomness is carried by the noise; see Remark \ref{rem:r:vs:nr}. In particular, the high probability statement above is implicitly \emph{conditional on the design}, and the analysis reduces to controlling the fluctuations of the noise around a fixed regression surface, thus permitting a precise geometric treatment which exploits the linear structure and normality of the model.
	From the standpoint of modern Statistical Learning, however, this viewpoint is fundamentally limited. Indeed, in the general framework of Example \ref{super:learn}, as applied to (not necessarily linear) regression, the data are viewed as i.i.d.\ draws \(\{(X_i,Y_i)\}\) from an unknown distribution $P_{(X,Y)}$, so that the covariates themselves are random. In this setting, one is no longer conditioning on a fixed design, but rather averaging over $P_{(X,Y)}$. As a consequence, uniform guarantees of the above form
	are no longer attainable in finite samples without strong assumptions on the conditional law of \(Y\) given \(X\); see \cite[Lemma 1]{lei2014distribution}, which shows that any procedure achieving exact finite-sample conditional coverage must necessarily produce trivial (essentially unbounded) prediction sets.
	Conformal prediction (CP), introduced in \cite{vovk2005algorithmic} and further developed in the regression setting in \cite{lei2018distribution}, embraces this random-design perspective. Rather than conditioning on the covariates, CP constructs prediction sets satisfying the distribution-free guarantee
	\[
	P\big(Y\in C(X)\big)\ge 1-\delta,
	\]
	where the probability is taken over $P_{(X,Y)}$. In other words, CP replaces conditional control at fixed \(x\) by marginal control averaged over the design. This shift is what makes it possible to obtain nontrivial, finite-sample guarantees without substantive modeling assumptions beyond exchangeability of the sample data (cf.\ Example \ref{de:finetti}).
	Operationally, CP separates prediction from calibration through the following steps:
\begin{itemize}
	\item Choose a predictor \(\widehat f\) to fit the model to the training data;
	\item Define a conformity score \(s(x,y)\), which measures the agreement between the response \(y\) and the prediction \(\widehat f(x)\) on a complementary calibration data set, thereby inheriting the structure of the underlying learning algorithm;
	\item Calibrate a threshold \(\widehat q\) via empirical quantiles of the conformity scores;
	\item Form the prediction sets
	\[
	C(x)=\{y:\, s(x,y)\le \widehat q\}.
	\]
\end{itemize}
	The crucial point here is that the validity of these sets is distribution-free, while their size reflects the accuracy of \(\widehat f\). In particular, this decoupling makes CP naturally compatible with high-dimensional procedures, such as the Lasso in Example \ref{high:sp:lasso}, where classical inference is unavailable: even when \(p\gg n\), one may still obtain valid prediction intervals, whose width adapts to the predictive performance of the underlying estimator.
	Thus, the passage from Example \ref{pred:resp} to CP reflects a fundamental shift: from conditioning on a fixed design and exploiting model geometry to averaging over a random design and ultimately relying on minimal symmetry. From the viewpoint of Statistical Learning, as sketched in Example \ref{super:learn},  CP provides a general mechanism for endowing arbitrary prediction algorithms \(\widehat f\in\mathcal F\) with rigorous, distribution-free uncertainty quantification, thereby extending the classical notion of prediction bands to the high-dimensional, model-agnostic regime; see \cite{angelopoulos2023conformal} for a recent survey on conformal prediction and \cite{angelopoulos2024theoretical} for an account of the subject in textbook form.
	\qed
\end{remark}

\section{The exponential family and generalized linear models}\label{exp:glms}

The linear model, introduced in Example \ref{mle:imp:lsm:n}, has historically served as the canonical tool for regression analysis. At its core, it assumes that a response vector $\mathbf{y}\in\mathbb{R}^n$ can be represented as  
\begin{equation}
	\mathbf{y} = {\bf x}\beta + \mathbf{e}, \qquad \mathbf{e} \sim \mathcal{N}(0,\sigma^2 {\rm Id}_n),
\end{equation}  
where ${\bf x}\in\mathbb{R}^{n\times p}$ is the design matrix of predictors\footnote{For simplicity, here we assume that no intercept is present and that ${\bf X}$ is fixed (Remark \ref{rem:r:vs:nr}).}, $\beta\in\mathbb{R}^p$ is the parameter vector, and $\mathbf{e}$ is a homoscedastic normal error.  
As discussed in Section \ref{mls:sub}, these assumptions lead to tractable maximum likelihood estimation (which coincides with least squares), exact inference based on normal theory leading to a high degree of interpretability through parameter recovery, and elegant prediction properties for the mean response.  
In practice, however, empirical data rarely conforms to the normal-homoscedastic paradigm as outcomes may be binary (such as success or failure in a Bernoulli trial), counts (as in Poisson processes), or strictly positive, highly skewed data (for example, waiting times). In such cases, the linear model becomes conceptually inadequate, since it implicitly assumes additivity on the original scale of the response with a variance that is independent of the mean. These limitations motivate the consideration of a broader framework, where the key step lies in recognizing the role of a much wider family of distributions to which the response is supposed to follow.

\begin{definition}\label{exp:fam}
	Let $Y$ be a random variable whose pdf (or mdf), say $\psi(y;\theta)$, depends on an unknown parameter 
	$\theta \in \mathbb{R}$. 
	Then we say that $\psi$ belongs to the {\em exponential family} if it takes the form
	\[
	\psi(y;\theta) = \exp\left( \frac{\xi(\theta) y - b(\theta)}{\phi} + c(y,\phi) \right),
	\]
	where $\phi>0$ is the {\em dispersion parameter} and $\xi$, $b$ and $c$ are known functions. In this case, $\xi=\xi(\theta)$ is the {\em natural parameter} of $Y$.
\end{definition}

We represent this by $Y\sim {\sf ExpFam}(\xi,\phi)$, leaving implicit the dependence on $b$ and $c$. For simplicity, we assume that $\phi$ is known. To ensure that  $\theta$ is well defined as a function of $\xi$, we assume throughout that $\xi'\neq 0$, with the prime meaning derivative with respect to $\theta$. The next result shows that the expectation and variance of $Y$	may be expressed as rational functions of derivatives of $\xi$ and $b$ up to second order.

\begin{proposition}\label{exp:plus:var}
	Under the conditions above there hold
	\begin{equation}\label{exp:p:var:1}
		\mathbb E(Y)=\frac{b'}{\xi'}, \quad {\mathbb V}(Y)=\frac{\phi(\xi'b''-b'\xi'')}{(\xi')^3}.
	\end{equation}
\end{proposition}

\begin{proof}
	The log-likelihood (for a single observation of $Y$) is 
	\[
	l=\frac{\xi y-b}{\phi}+c,
	\] 
	from which we find that 
	\begin{equation}\label{l:der:xi}
		l'=\frac{\xi'y-b'}{\phi}.
	\end{equation}
	Now, with this notation (\ref{fischer:0}) says that $\mathbb E(l')=0$, which immediately yields the expression for $\mathbb E(Y)$. On the other hand, (\ref{fischer}) means that $\mathbb E((l')^2)=-\mathbb E(l''))$, which gives
	\[
	\frac{1}{\phi^2}\left((\xi')^2\mathbb E(y^2)-2\xi'b'\mathbb E(y)+(b')^2\right)=-\frac{1}{\phi}\left(\xi''\mathbb E(Y)-b''\right)=-\frac{1}{\phi}\left(\xi''\frac{b'}{\xi'}-b''\right).
	\]
	Since $\xi'b'\mathbb E(y)=(b')^2$, we may rearrange terms in order to get
	\[
	\frac{1}{\phi^2}(\xi')^2{\mathbb V}(Y)=\frac{1}{\phi^2}(\xi')^2\left(\mathbb E(Y^2)-\mathbb E(Y)^2\right)=\frac{1}{\phi}\frac{\xi'b''-b'\xi''}{\xi'},
	\]
	which completes the proof.
\end{proof}

\begin{definition}\label{def:mean:var}
	If $Y\sim {\sf ExpFam}(\xi,\phi)$ then its {\em mean} and {\em variance} functions are respectively given by
	\[
	\mu=\mathbb E(Y), \quad V(\mu)=\dot\mu,
	\]	
	where the dot means derivative with respect to $\xi$.
\end{definition}

\begin{proposition}\label{var:prop:exp}
	(Mean-variance relationship)
	If $Y\sim {\sf ExpFam}(\xi,\phi)$ then
	\begin{equation}\label{var:prop:exp:2}
		{\mathbb V}(Y)=\phi V(\mu).
	\end{equation}
\end{proposition}

\begin{proof}
	Immediate from (\ref{exp:p:var:1}) and the chain rule.
\end{proof}	

\begin{example}\label{rem:ident:expfam}(Naturality)
	A distribution in the exponential family is called \emph{natural} when $\xi(\theta)=\theta$, so that $\theta$ itself is the natural parameter to be estimated. In this case, (\ref{exp:p:var:1}) reduces to
	\begin{equation}\label{exp:p:var:3}
		\mu=\dot b, \quad {\mathbb V}(Y)=\phi \ddot b.
	\end{equation}
	From this it is straightforward to verify that the specific form of the mean--variance relationship essentially determines the distribution within this subclass; see \cite[Theorem 2.11]{jorgensen1997theory}.
	Now, if $Y_j\sim {\sf ExpFam}(\theta,\phi)$ is a random sample then the corresponding log-likelihood is 
	\[
	l({\bf y};\theta)=\frac{1}{\phi}\left(\theta\sum_j y_j-nb(\theta)\right)+\sum_jc(y_j,\phi), 
	\]
	so the score is
	\begin{equation}\label{like:exp:fm}
	l_\theta({\bf y};\theta)=\frac{n}{\phi}\left(\overline y_n-b'(\theta)\right)	
	\end{equation}	
	and the ML estimator  $\widehat\theta$ is determined by the equation $b'(\widehat\theta)=\overline Y_n$. Since (\ref{var:prop:exp:2}) clearly implies that $b''>0$, it follows that $b'$ is strictly increasing and we get $\widehat\theta=(b')^{-1}(\overline Y_n)$. Thus, the MLE of the natural parameter $\theta$ depends only on the sample mean. Also, by means of (\ref{var:prop:exp:2}), (\ref{exp:p:var:3}) and (\ref{like:exp:fm})  we may rewrite the score in terms of $\mu$ as  
	\begin{equation}\label{like:exp:mu}
		l_\mu({\bf y};\mu)=\frac{d\theta}{d\mu}	l_\theta({\bf y};\theta(\mu))=\frac{n}{\phi V(\mu)}\left(\overline y_n-b'(\theta)(\mu)\right)=\frac{n}{\phi V(\mu)}\left(\overline y_n-\mu\right),
	\end{equation}
	so the ML estimator $\widehat\mu$ of $\mu$ is the sample mean. Turning to asymptotics, from (\ref{like:exp:fm}) and (\ref{fischer}) we find that the Fisher information is 
	\begin{equation}\label{eq:fisher:exp} 
		\mathcal F_{(n)}(\theta)=\frac{nb''(\theta)}{\phi},
		\end{equation}
	so that Theorem \ref{asym:mle} and consistency give the large sample estimate
	\[
	\widehat\theta_n\approx\mathcal N\left(\theta,\frac{\phi}{nb''(\widehat\theta_n)}\right). 
	\] 
	For the mean parameter $\mu$ we may either start with (\ref{like:exp:mu}) and (\ref{fischer:0}), which directly gives the corresponding Fisher information
	\[
	\mathcal F_{(n)}(\mu)=n/\phi V(\mu),
	\]
	or use that $\mu'=V=b''$ together with the delta method, as explained in Remark \ref{decay:fluc}, thus obtaining the large sample estimate 
	\[
	\widehat\mu_n\approx\mathcal N\left(\mu,\frac{\phi V(\widehat\mu_n)}{n}\right).
	\]
	As usual, these asymptotic normality results immediately provide the basis for constructing large-sample confidence intervals for the parameters $\theta$ and $\mu$ (cf. Remark \ref{conf:int:f}).
		\qed
\end{example}

\begin{remark}\label{kl:with:exp}
	(Kullback--Leibler divergence within an exponential family)
	Let $\{\psi(\cdot;\theta);\theta\in\mathbb R\}$ be a one-parameter exponential family as in Definition \ref{exp:fam}.
		Since, for any $\theta_0,\theta\in\mathbb R$,
		\[
	\ln\frac{\psi(y;\theta_0)}{\psi(y;\theta)}
	=
	\frac{
		\big(\xi(\theta_0)-\xi(\theta)\big)y
		-
		\big(b(\theta_0)-b(\theta)\big)
	}{\phi},
	\] 
	the corresponding 
		Kullback--Leibler divergence from Definition \ref{kull:leib:div:d} is
	\begin{eqnarray*}
	D^{KL}_{\theta_0}(\theta)
& 	= & \mathbb E_{\theta_0}\!\left(
	\ln\frac{\psi(Y;\theta_0)}{\psi(Y;\theta)}
	\right)\\
	& = & 
	\frac{1}{\phi}
	\left(
	\big(\xi(\theta_0)-\xi(\theta)\big)
	\mathbb E_{\theta_0}(Y)
	-
	\big(b(\theta_0)-b(\theta)\big)
	\right),
	\end{eqnarray*}
so that Proposition \ref{exp:plus:var} yields 	
	\[
	D^{KL}_{\theta_0}(\theta)
	=
	\frac{1}{\phi}
	\left[
	\big(\xi(\theta_0)-\xi(\theta)\big)
	\frac{b'(\theta_0)}{\xi'(\theta_0)}
	-
	\big(b(\theta_0)-b(\theta)\big)
	\right].
	\]
	In the natural case of Example \ref{rem:ident:expfam}, where $\xi(\theta)=\theta$, this expression simplifies to
	\[
		D^{KL}_{\theta_0}(\theta)
	=
	\frac{1}{\phi}\Big(
	b(\theta)-b(\theta_0)-b'(\theta_0)(\theta-\theta_0)
	\Big).
	\]
	In particular, by convexity of $b$, the Kullback--Leibler divergence is nonnegative and vanishes if and only if $\theta_0=\theta$.
	Moreover, by expanding $b$ around  $\theta=\theta_0$ and using (\ref{eq:fisher:exp}) we find that 
	\begin{equation}\label{kl:with:exp:2}
		D^{KL}_{\theta_0}(\theta)=\frac{1}{2}\mathscr F_{(1)}(\theta_0)\left(\theta-\theta_0\right)^2+o\left(\left(\theta-\theta_0\right)^2\right).
	\end{equation}
	Thus, in the natural parametrization of an exponential family, the Fisher information $\mathscr F_{(1)}(\theta_0)$ determines the local second-order behavior of the Kullback–Leibler divergence centered at $\theta_0$, which is consistent with the general result in Remark \ref{kl:fisher:aic}.
	\qed
\end{remark}

\begin{table}[ht]
	\centering
	\begin{tabular}{|c|c|c|c|c|c|c|c|c|}
		\hhline{~|--------|} 
		\multicolumn{1}{c|}{} 
		& $\theta$ & \textrm{Likelihood $L(y;\theta)$}  & $\xi$ & $\theta=\theta(\xi)$ & $\phi$ & $b$ & $\mu=\mathbb E(Y)$ & $V(\mu)$\\ 
		\hhline{|=|========|} 
		\multicolumn{1}{|c|}{\makecell{Binomial \\ (Example \ref{bern:trial})}}
		& $p$ 
		& $\binom{n}{y}p^y(1-p)^{n-y}$
		& $\ln \tfrac{p}{1-p}$ 
		& $\tfrac{1}{1+e^{-\xi}}$ 
		& $1$
		& $-n\ln(1-p)$ 
		& $p$ 
		& $\mu(1-\mu)$
		\\ \hline
		\multicolumn{1}{|c|}{\makecell{Poisson \\ (Example \ref{poisson:trials})}}
		& $\lambda$ 
		& $\tfrac{e^{-\lambda}\lambda^y}{y!}$
		& $\ln \lambda$ 
		& $e^\xi$ 
		& $1$
		& $\lambda$
		& $\lambda$
		& $\mu$ 
		\\ \hline
		\multicolumn{1}{|c|}{\makecell{Normal with $\sigma$ known\\ (Definition \ref{normdistrv} )}}
		& $\mu$ 
		& $\tfrac{1}{\sqrt{2\pi}\sigma}\,e^{-\lvert y-\mu\rvert^2/2\sigma^2}$
		& ${\mu}$ 
		& $\xi$ 
		& $\sigma^2$
		& $\tfrac{\mu^2}{2}$ 
		& $\mu$ 
		& $1$
		\\ \hline
		\multicolumn{1}{|c|}{\makecell{Gamma with $\lambda$ known\\ (Definition \ref{gamma:dist} )}}
		& $\alpha$ 
		& $\tfrac{\alpha^\lambda}{\Gamma(\lambda)}y^{\lambda-1}e^{-\alpha y}$
		& $-\alpha$ 
		& $-\xi$ 
		& ${1}$
		& $-\lambda\ln\alpha$ 
		& $\frac{\lambda}{\alpha}$ 
		& $\frac{\mu}{\alpha}$
		\\ \hline
	\end{tabular}
	\caption{Examples of  distributions in the exponential family}\label{tab:exp-family}
\end{table}

As shown in Table \ref{tab:exp-family}, most of the distributions considered so far can be expressed as members of the exponential family\footnote{A simple, commonly used example outside the exponential family is Student’s 
	${\mathfrak t}$-distribution in Definition \ref{tstu:def}.}. A distinctive role is played by the normal distribution, which is the only one in the table whose variance is entirely independent of the mean.  
This observation paves the way for a substantial enrichment of the class of regression models, while still preserving the desirable inferential properties of the classical linear model, as will be seen below.

\begin{definition}\label{def:glm}
	A {\em generalized linear model (GLM)} for  {independent} responses $\{Y_i\}_{i=1}^n$ consists of the following ingredients:
	\begin{enumerate}
		\item \textbf{Random component:} each $Y_i$ follows a one-parameter exponential family: $Y_i\sim {\sf ExpFam}({\xi_i},\phi)$, 
		where $\xi_i$ is the canonical parameter and $\phi$ a common dispersion parameter.
		\item \textbf{Systematic component:} a linear predictor
		\begin{equation}\label{syst:comp}
			\eta_i = {\bf x}_i^\top \beta,
		\end{equation}
		linking covariates ${\bf x}_i$ to coefficients $\beta$.
		\item \textbf{Link function:} a monotone differentiable map $g$ connecting the mean $\mu_i:=\mathbb{E}(Y_i)$ to the predictor,
		\[
		g(\mu_i) = \eta_i.
		\]
		When $g(\mu_i)=\xi_i$, the link is called \emph{canonical}.
	\end{enumerate}
\end{definition}

Thus, the GLM extends the linear model by allowing non-normal response distributions and by permitting nonlinear, yet monotone, transformations in the relationship between the mean response and the linear predictor. 
In particular, since the systematic component is linear in $\beta$ and the non-linearity only affects the mean, GLMs remain interpretable in the sense of Remark \ref{int:acc}, while still retaining much of its predictive power.
 From (\ref{var:prop:exp:2}) it also follows that  
\[
\mathrm{Var}(Y_i) = \phi V(\mu_i),
\]  
which shows that {\em heteroscedasticity} (unequal variances across observations) is inherent to a GLM.  

\begin{example}\label{most:com:glms}
	Because of their flexibility, which balances mathematical rigor with empirical applicability, GLMs are widely used in both theory and applications \cite{agresti2015foundations,dobson2018introduction}. Here we restrict ourselves to three of the most prominent examples, corresponding to the first three rows of Table \ref{tab:exp-family}:  
	\begin{itemize}
		\item \emph{Logistic regression} arises when $Y_i\sim \mathsf{Ber}(p_i)$, a Bernoulli distribution, with the \emph{logit link}  
		\begin{equation}\label{logit:link}
		\eta_i={\rm logit}(p_i):=\ln\left(\frac{p_i}{1-p_i}\right).
		\end{equation} 
		Equivalently,  if we solve for $p_i=P(Y_{i}=1|_{X_i=x_i})$,
		\begin{equation}\label{logit:inv}
			P(Y_{i}=1|_{X_i=x_i})={\rm logit}^{-1}(\eta_i)=\frac{1}{1+e^{-{\bf x}_i^\top\beta}}.
		\end{equation}
		\item \emph{Poisson regression} corresponds to $Y_i\sim \mathsf{Pois}(\lambda_i)$, a Poisson distribution, with the \emph{log link}  
		\[
		\eta_i=\ln\mu_i.
		\]  
		\item The classical {\em linear model} (from Example \ref{mle:imp:lsm:n}, with $\sigma^2$ known) assumes $Y_i\sim \mathcal N(\mu_i,\sigma^2)$ with the \emph{identity link} $\eta_i=\mu_i$. 
	\end{itemize}
	Note that the link is canonical in all these cases.\qed
\end{example}

We now turn to the most basic aspects of the estimation framework for GLMs. Since these models can be regarded as natural extensions of the normal linear model, it is reasonable to adopt maximum likelihood as the method for estimating $\beta$  (cf.  Example \ref{mle:imp:lsm:n}). The corresponding log-likelihood for $n$ observations is
\begin{equation}\label{like:glm}
l({\bf y};\beta) = \sum_i l_i({\bf y};\beta), \quad l_i({\bf y};\beta)= \frac{y_i \xi_i - b(\xi_i)}{\phi} + c(y_i,\phi),
\end{equation}
so we should compute 
\[
\frac{\partial l_i}{\partial \beta_j}=
\frac{\partial l_i}{\partial \xi_i}
\frac{\partial \xi_i}{\partial \mu_i}
\frac{\partial \mu_i}{\partial \eta_i}
\frac{\partial \eta_i}{\partial \beta_j}, \quad j=1,\cdots,p, 
\]
with the likelihood equations being obtained by summing up these terms over $i$ and equating the result to zero (Definition \ref{mle:def:post}).
Now, (\ref{l:der:xi}) yields
\[
\frac{\partial l_i}{\partial \xi_i}=\frac{\frac{\partial \xi_i}{\partial \theta_i}y_i-\frac{\partial b}{\partial\theta_i}}{\phi},
\]
which together with (\ref{exp:p:var:1}) gives
\[
\frac{\partial l_i}{\partial \xi_i}=\frac{\frac{\partial\xi_i}{\partial\theta_i}\left(y_i-\mu_i\right)}{\phi}\,\, \textrm{and}\,\,
\frac{\partial \xi_i}{\partial \mu_i}=\frac{\phi}{\frac{\partial\xi_i}{\partial\theta_i}{\mathbb V}(Y_i)}.
\]
Also, (\ref{syst:comp}) implies 
\[
\frac{\partial \eta_i}{\partial \beta_j}=x_{ij},
\]
so if we put all the pieces of our computation together we obtain the following fundamental result.

\begin{proposition}\label{mle:glm}
	The maximum likelihood estimator $\widehat\beta_{GLM}$ of a GLM satisfies the system of equations  
	\begin{equation}\label{mle:glm:2}
		{\bf x}^\top DV^{-1}({\bf y}-{\bm \mu})=0,
	\end{equation}		
	where ${\bm\mu}=(\mu_1,\cdots,\mu_n)^\top$, $D={\rm diag}(\partial\mu_i/\partial\eta_i)$, a diagonal matrix whose entries depend on the specific shape of the link function of the model, and $V={\rm diag}({\mathbb V}(Y_i))$. As a consequence, if the link function is canonical then this reduces to 
	\begin{equation}\label{mle:glm:3}
		{\bf x}^\top({\bf y}-{\bm \mu})=0,
	\end{equation}	
\end{proposition}

\begin{proof}
	The calculation above shows that the score components are
	\begin{equation}\label{score:comp}
	\frac{\partial l_i}{\partial\beta_j}=\frac{y_i-\mu_i}{{\mathbb V}(Y_i)}x_{ij}\frac{\partial\mu_i}{\partial\eta_i},	
		\end{equation}
so the defining condition for $\widehat\beta_{GLM}$, ${\partial l_i}/{\partial\beta_j}=0$, is equivalent to (\ref{mle:glm:2}).		
	As for the last assertion,  from $\eta_i=\xi_i$ we find that 
	\[
	\frac{\partial \mu_i}{\partial\eta_i}=
	\frac{\partial \mu_i}{\partial\xi_i}=\frac{\partial^2b}{\partial\xi_i^2},
	\]
	where we used (\ref{exp:p:var:3}) in the last step. Also, again by (\ref{exp:p:var:3}), 
	\[
	{\mathbb V}(Y_i)=\phi \frac{\partial^2b}{\partial\xi_i^2}.
	\]
Together, these identities imply that  $DV^{-1}=\phi^{-1}{\rm Id}_n$. 
\end{proof}

Although the dependence on $\beta$ (and hence on $\widehat\beta_{GLM}$) is not explicit in either \eqref{mle:glm:2} or \eqref{mle:glm:3}, it is in fact present because $\mu_i=g^{-1}({\bf x}_i^\top{\beta})$. In general this dependence is non-linear, so the likelihood equations must be solved for $\beta$ by means of an iterative method (usually, Newton-Raphson).  
As an illustration, in the logistic model the equations reduce to ${\bf x}^\top({\bf y}-{\bf p})=0$, where ${\bf p}=(p_1,\ldots,p_n)^\top$. Here the non-linearity is entirely due to the inverse logit relation in \eqref{logit:inv}. By contrast, if the GLM specializes to the classical linear model, then ${\bm \mu}={\bf x}\beta$ (linearity) and \eqref{mle:glm:3} simplifies to ${\bf x}^\top({\bf y}-{\bf x}\beta)=0$, which directly yields the usual least squares estimator under the standard assumptions.

With the maximum likelihood framework established, we now briefly examine the asymptotic properties of $\widehat\beta_{GLM}$. 
From (\ref{score:comp}) and 
\eqref{fischer} we find that the Fisher information for the $i^{\text{th}}$ observation is  
\[
\mathcal F^{(i)}_{jk}
= 
\mathbb E\left(\frac{\partial l_i}{\partial\beta_j}\frac{\partial l_i}{\partial\beta_k}\right)
= \frac{x_{ij}x_{ik}}{{\mathbb V}(Y_i)}\left(\frac{\partial\mu_i}{\partial\eta_i}\right)^2,
\]  
so that,
by independence, the Fisher information for the entire sample $Y^{[n]}=(Y_1,\ldots,Y_n)$ is  
\[
\mathcal F=\sum_i\mathcal F^{(i)}={\bf x}^\top W{\bf x}, \qquad W={\rm diag}\left(\frac{(\partial \mu_i/\partial \eta_i)^2}{{\mathbb V}(Y_i)}\right).
\]  
Therefore, applying Theorem \ref{asym:cr} (see also its generalization in Remark \ref{asym:cr:gen}) we deduce that as $n\to\infty$,  
\begin{equation}\label{asym:min:s}
\widehat\beta_{GLM} \approx \mathcal N(\beta,({\bf x}^\top W{\bf x})^{-1}). 
\end{equation} 
As is customary, consistency permits the substitution of $W$ by $\widehat W=W(\widehat\beta_{GLM})$, leading to the practical approximation  
\begin{equation}\label{asym:mim}
\widehat\beta_{GLM} \approx \mathcal N(\beta,({\bf x}^\top \widehat W{\bf x})^{-1}),
\end{equation} 
which forms the basis for constructing large-sample confidence intervals for the components of $\beta$. In the linear model case (with $\sigma^2$ known) we have $\partial\mu_i/\partial\eta_i=1$ and ${\mathbb V}(Y_i)=\sigma^{-2}$, so that $W=\sigma^{-2}{\rm Id}_n$ and (\ref{asym:mim}) essentially reduces to (\ref{tcl:lr:n}).

With the appropriate care, most of the well-established estimation theory for the linear model can thus be carried over to this broader framework of GLMs. In particular, notions such as asymptotic efficiency, hypothesis testing, and likelihood-based inference retain essentially the same mathematical structure, even though the underlying distribution of the response is no longer normal \cite{agresti2015foundations,dobson2018introduction,gill2019generalized}. This transfer of results is precisely what makes GLMs so attractive: they extend the familiar tools of linear regression to a far wider range of data types, while preserving a rigorous probabilistic foundation. As a consequence, GLMs provide a unified language for both theoretical developments and applied work, bridging the gap between classical models and modern data analysis.

\begin{example}\label{aic:glm} (AIC for GLMs)
The AIC for a GLM with a canonical link may be computed explicitly. Indeed, since $\xi_i=\eta_i={\bf x}_i^\top\beta$, it follows from (\ref{aic:form:f}) and (\ref{like:glm}) that
\begin{equation}\label{eq:aic:glm}
\textrm{AIC}=-\frac{2}{\phi}\sum_i\left(
{y_i{\bf x}_i^\top\widehat\beta_{GLM}-b({\bf x}_i^\top\widehat\beta_{GLM})}
\right)+2(p+2),
\end{equation}
where $\widehat\beta_{GLM}$ denotes the corresponding MLE and 	
$p$ is the number of regressors (with the intercept excluded). In particular, this formula applies to the logistic model discussed in Example \ref{most:com:glms}. Since $\theta_i=p_i$, the first row in Table \ref{tab:exp-family} (with $n=1$)  yields 
\[
\theta_i(\xi_i)=\frac{1}{1+e^{-\xi_i}},
\]
and therefore
\[
b(\xi_i)=-\ln(1-\theta_i)=\ln(1+e^{\xi_i}).
\]
We thus conclude that
\[
\textrm{AIC}=2\sum_i\left(
{\ln\left(1+e^{{\bf x}_i^\top\widehat\beta_{GLM}}\right)-y_i{\bf x}_i^\top\widehat\beta_{GLM}}
\right)+2(p+1),
\]
Comparing this expression with (\ref{eq:aic:glm}), we see that the penalty term $p+2$ is replaced by $p+1$,
which is consistent with the fact that, in the logistic model, the variance function is
$V(\mu)=\mu(1-\mu)$.
\qed
	\end{example}

\begin{example}\!\!$\bigstar$\label{ex:irt}(GLMs and Item Response Theory)  
	A particularly fruitful domain where generalized linear models intersect with modern statistical methodology is \emph{Item Response Theory} (IRT), which plays a central role in psychometrics and educational assessment \cite{hambleton1991fundamentals,deayala2013theory,vanderlinden2016handbook}. Conceptually, IRT can be regarded as a GLM with a latent predictor, where the individual ability parameter $\gamma_j$ functions as an unobserved covariate, typically following a centered normal, say $\gamma_j\sim\mathcal N(0,1)$. For instance,  
	in the classical {\em Rasch model}, the probability that an individual $j$ with ability $\gamma_j$ answers item $i$ correctly is  
	\[
	P(Y_{ij}=1|_{\gamma_j, b_i})=\mathrm{logit}^{-1}(\gamma_j-b_i),
	\]  
	where $b_i$ is the item difficulty parameter. This is directly analogous to the inverse logit link (\ref{logit:inv}) in the GLM framework above, with linear predictor $\eta_{ij}=\gamma_j-b_i$.
	More generally, variants of the logistic regression model in (\ref{logit:link})
	underlie both GLMs and IRT, with the key distinction being that in IRT, part of the predictor vector corresponds to latent person parameters rather than observed covariates. In this way, the {\em 2PL model} extends the Rasch model by introducing item discrimination $a_i$,  
	\begin{equation}\label{irt:2pl}
	P(Y_{ij}=1|_{\gamma_j,a_i,b_i})=\mathrm{logit}^{-1}\!\big(a_i(\gamma_j-b_i)\big),
	\end{equation}
	while the {\em 3PL model}, widely used in practice, adds a pseudo-guessing parameter $c_i$.  We should also point out that from an asymptotic perspective, the connection between GLMs and IRT is especially revealing. Since IRT models are essentially Bernoulli GLMs with latent predictors, the same large-sample principles apply: maximum likelihood estimators of item parameters (difficulty $b_i$, discrimination $a_i$, and pseudo-guessing $c_i$) are consistent and asymptotically normal under standard regularity conditions,
	making sure that appropriate versions of Fisher's foundational conception in Theorem \ref{asym:cr} remain operational in this broader context.
	In particular, the Fisher information for the IRT likelihood plays the same role as in the GLM framework, forming the basis for variance formulas and for the construction of confidence intervals and hypothesis tests.  
		We illustrate these ideas by developing the corresponding asymptotic theory for the 2PL model in \eqref{irt:2pl}. 
		To simplify matters, we estimate the respondent’s ability $\gamma_j$ under the  assumption that the item parameters $(a_i,b_i)$ are known
		\footnote{We are making two simplifying assumptions here. First, the item parameters are assumed to have been calibrated prior to analysis, so that their estimation uncertainty is ignored; such pre-calibration is routinely performed in large-scale assessments and adaptive testing systems (e.g., PISA, ENEM, TOEFL). Second, although the latent traits $\gamma_j$ are modeled as random effects ($\gamma_j \sim \mathcal N(0,1)$), we condition on the observed response patterns and treat each $\gamma_j$ as an unknown constant when estimating individual abilities. Both assumptions are relaxed in more general formulations, where item and person parameters are estimated jointly and the latent distribution is integrated into the likelihood \cite{baker2004item,vanderlinden2016handbook}. The connection between such hierarchical treatments in IRT and generalized linear mixed models is discussed in Remark~\ref{glms:to:irt}.}.
		As usual, we assume \emph{local independence}, meaning that the $N$ responses $Y_i=Y_{ij}$ are conditionally independent given $\gamma_j$, i.e.\ $\{Y_i|_{\gamma_j}\}_{i=1}^N$ is independent. 
		Accordingly, and in alignment with \eqref{logl:ber}, the corresponding log-likelihood is 
		\begin{equation}
			l({\bf y};\gamma_j) 
			= \sum_{i} \left( y_{i} \ln P_{i}(\gamma_j) + (1 - y_{i}) \ln Q_i(\gamma_j) \right),
		\end{equation}
		where $P_i(\gamma_j)$ is a shorthand for the expressions  in (\ref{irt:2pl}) and 
		$ Q_i (\gamma_j)= 1 - P_i(\gamma_j)$.
		Hence, the associated score function is 
		\begin{align*}
			s({\bf y};\gamma_j)
			&= \sum_i \frac{\partial}{\partial \gamma_j} 
			\left(y_i \ln P_i + (1-y_i) \ln Q_i\right) \\
			&= \sum_{i} \left\{ y_{i} \left( \frac{1}{P_i} a_i P_i Q_i \right)
			+ (1 - y_{i}) \left( \frac{1}{Q_i} (-a_i P_i Q_i) \right) \right\},
		\end{align*}
		which simplifies to
		\begin{equation}\label{irt:score}
			s({\bf y};\gamma_j) = \sum_{i} a_i (y_{i} - P_i(\gamma_j)).
		\end{equation}
		The maximum likelihood estimator therefore satisfies
		\begin{equation}
			\sum_{i} a_i P_i(\widehat\gamma_j) = \sum_{i} a_i y_{i},
		\end{equation}
		a non-linear equation in $\widehat\gamma_j$ that must be solved numerically.  
		The Fisher information follows from \eqref{irt:score} and \eqref{fisher:mat:def}:
		\begin{align*}
			\mathscr F(\gamma_j)
			&= \mathbb E \!\left(\left( \sum_{i} a_i (Y_{i} - P_i) \right)^2 \right) \\
			&= \sum_{i} \mathbb E\!\left(a_i^2 (Y_{i} - P_i)^2 \right)
			+ \sum_{\substack{i\neq k}}
			\mathbb E\!\left(a_i a_k (Y_{i} - P_i)(Y_{k} - P_k)\right),
		\end{align*}
		with the mixed terms vanishing due to local independence, Proposition~\ref{indexp}, 
		and the fact that $\mathbb E(Y_i)=P_i$ (recall that $Y_i\sim \mathsf{ Ber}(P_i)$).  
		Since ${\mathbb V}(Y_i)=P_i Q_i$, we thus obtain
		\begin{equation}
			\mathscr F(\gamma_j)=\sum_{i} a_i^2 P_i(\gamma_j) Q_i(\gamma_j),
		\end{equation}
		and consequently the large-sample approximation
		\begin{equation}
			\widehat\gamma_j
			\approx \mathcal N\!\left(\gamma_j,\,
			\frac{1}{\sum_{i} a_i^2 P_i(\widehat\gamma_j) Q_i(\widehat\gamma_j)}\right),
		\end{equation}
		which parallels the asymptotic result previously obtained for the GLM estimator $\widehat\beta_{\mathrm{GLM}}$.
Thus, the GLM perspective not only clarifies the statistical structure of IRT but also provides a rigorous foundation for inference, ensuring that the asymptotic theory developed for GLMs can be effectively transplanted into psychometric applications. IRT, therefore, should not be seen as a distinct paradigm but as a specialized application of GLMs with latent predictors, offering a robust statistical framework for modeling educational and psychological measurement.
\qed  
\end{example}

\begin{remark}\!\!$\bigstar$\label{glms:to:irt}
	(GLMMs as the bridge between GLMs and IRT).
 	From a conceptual standpoint, the passage from generalized linear models (GLMs) to Item Response Theory (IRT) naturally goes through an intermediate class, namely, {\em generalized linear mixed models (GLMMs)} \cite{stroup2013,jiangnguyen2021}. 
 	In a GLMM, the linear predictor of a GLM is extended by the inclusion of random effects, allowing part of the variation in the response to be attributed to unobserved random components. Formally, while a GLM is written as
 	\[
 	g\left(\mathbb E(Y_i|_{{\bf X}={\bf x}})\right)={\bf x}_i^\top\beta,\quad i=1,\cdots,n,
 	\]
 	with $Y_i$ following a member of the exponential family, a GLMM generalizes this expression to
 	\[
 			g\left(\mathbb E(Y_i|_{{\bf X}={\bf x},\Gamma=\gamma})\right)={\bf x}_i^\top\beta +{\bf z}_i^\top\gamma,
 	\]
 	where $\gamma\in\mathbb R^q$ 
 	comprises the random effects associated with an individual labeled by $i$, typically supposed to follow a centered normal distribution, and $\bf z$ is the associated $n\times q$ matrix design, which we assume fixed here. The latent ability parameter in IRT fulfills exactly this role: it acts as a  random effect at the individual level, representing an unobserved source of variability across respondents. 
 	Accordingly, IRT models may be regarded as Bernoulli GLMMs in which the random component captures the heterogeneity among individuals that remains unobserved in the classical GLM framework. Thus, starting with linear models, at each step new layers of generality emerge—link functions, random components, latent traits—culminating in the IRT framework, where the random effect becomes not a nuisance term but the very object of substantive interpretation \cite{deboeckwilson2004}.
 \end{remark}

\section{Sufficiency}\label{suff:sub}

In a statistical model, consider moving from the random sample
\[
X=(X_1,\dots,X_n), \quad X_j \sim \psi_\theta,
\]
to an estimator $\widehat\theta$ defined through a statistic $h=h(X)$. A natural question then arises: how much information from the data has actually been retained in this transition? A complete answer would require a precise definition of the amount of information carried by the sample, which lies beyond the scope of these notes. A more modest but still important task is to verify whether the chosen statistic captures \emph{all} the relevant information about the parameter $\theta$, in the sense that no additional knowledge from the sample is required for its estimation. 
Put differently, the aim is to identify situations where the ``extra randomness'' in the sample $X$ that is not reflected in $h(X)$ is unrelated to $\theta$, and thus irrelevant for inference. This idea admits a neat probabilistic formulation in terms of conditional distributions, as introduced in Section \ref{cond:poss}. 

\begin{definition}\label{suff:stat}
	A statistic $h=h(X)$ is said to be \emph{sufficient} if, for any realization ${\bf x}$ of $X$, the conditional probability distribution $\psi_{\theta;X|h(X)=h({\bf x})}$ evaluated at ${\bf x}$ does \emph{not} depend on $\theta$.
\end{definition}

Using (\ref{p:y:x=x}), it follows that sufficiency of $h$ ensures the existence of a function $\xi=\xi({\bf x})$ such that
\begin{equation}\label{dist:xxx}
	\frac{\psi_{\theta;(h(X),X)}(h({\bf x}),{\bf x})}{\psi_{\theta;h(X)}(h({\bf x}))} = \xi({\bf x}).
\end{equation}
A key observation is that the inclusion of events $\{X={\bf x}\}\subset\{h(X)=h({\bf x})\}$ implies
\begin{equation}\label{eq:dist}
	\psi_{\theta;(h(X),X)}(h({\bf x}),{\bf x})=\psi_{\theta;X}({\bf x})=L({\bf x};\theta),
\end{equation}
the likelihood function. Substituting this into (\ref{dist:xxx}) yields a practical characterization of sufficiency: it occurs precisely when the dependence of $L({\bf x};\theta)$ on $\theta$ is confined to a factor that depends on ${\bf x}$ only through the statistic $h$. This captures the essential content of the notion: all the information needed to estimate $\theta$ is already contained in the sufficient statistic, making further reference to the raw data $X$ unnecessary.

\begin{theorem}\label{suff:char} (Fisher-Neyman factorization)
	$h$ is sufficient if and only if the likelihood function factorizes as
	\begin{equation}\label{suff:char:eq}
		L({\bf x};\theta)=\eta(h({\bf x}),\theta)\xi({\bf x}),
	\end{equation}
	for positive functions $\eta$ and $\xi$. 
\end{theorem}

\begin{proof}
	We have already seen that sufficiency implies (\ref{suff:char:eq}). For the converse 
	we first note that (\ref{eq:dist}) leads to
	\begin{eqnarray*}
		\psi_{\theta;h(X)}(h({\bf x}))
		& = & \int_{\{{\bf x}':h({\bf x}')=h({\bf x})\}}\psi_{\theta;(h(X),X)}(h({\bf x}'),{\bf x}')d{\bf x}'\\
		& = & \int_{\{{\bf x}':h({\bf x}')=h({\bf x})\}}\psi_{\theta;X}({\bf x}')d{\bf x}'
	\end{eqnarray*}
	so we may again use 
	(\ref{p:y:x=x}) to compute:
	\begin{eqnarray*}
		\psi_{\theta;X|h(X)=h({\bf x})}
		& = & 	\frac{\psi_{\theta;(h(X),X)}(h({\bf x}),{\bf x})}{\psi_{\theta;h(X)}(h({\bf x}))}\\
		& = & \frac{\eta(h({\bf x}),\theta)\xi({\bf x})}{\int_{\{{\bf x}':h({\bf x}')=h({\bf x})\}}\eta(h({\bf x}'),\theta)\xi({\bf x}')
			d{\bf x}'}\\
		& = & \frac{\eta(h({\bf x}),\theta)\xi({\bf x})}{\eta(h({\bf x}),\theta)\int_{\{{\bf x}':h({\bf x}')=h({\bf x})\}}\xi({\bf x}')
			d{\bf x}'}.
	\end{eqnarray*}
	Thus, 
	\[
	\psi_{\theta;X|h(X)=h({\bf x})}=\frac{\xi({\bf x})}{\int_{\{{\bf x}':h({\bf x}')=h({\bf x})\}}\xi({\bf x}')
		d{\bf x}'}
	\]
	only depends on ${\bf x}$. 
\end{proof}

\begin{corollary}\label{mle:suff}
	A unique ML estimator is a function of a sufficient statistic. More generally, if a ML estimator exists then an ML estimator may be chosen so as to be a function of a sufficient statistic. 
\end{corollary}

\begin{proof}
	Given that the ML estimator $\widehat\theta$ is obtained by maximizing the likelihood function 
	$L({\bf x};\theta)$
	in $\theta$ (for each ${\bf x}$), this is an obvious consequence of   (\ref{suff:char:eq}).
\end{proof}

\begin{example}\label{suff:norm:pop}(Sufficiency in  a normal population) If $X_j\sim\mathcal N(\mu,\sigma^2)$ we know from Example \ref{normal:w} 
	that
	\[
	L({\bf x};\theta)=(2\pi\theta_2)^{-n/2}e^{-\frac{1}{2\theta_2}\sum_{j=1}^n(x_j-\theta_1)^2}, \quad {\bf x}=(x_1,\cdots,x_n),
	\]  
	where $\theta=(\theta_1,\theta_2)=(\mu,\sigma^2)\in \Theta=\mathbb R\times\mathbb R_+$.
	We distinguish three cases:
	\begin{itemize}
		\item ($\theta_2$ is known and $\theta_1$ is the unknown parameter)  Set
		\[
		h_1({\bf x})=\frac{1}{n}\sum_jx_j
		\]
		so that $\sum_j(x_j-h_1(x))=0$ implies 
		\begin{eqnarray*}
			\sum_j(x_j-\theta_1)^2
			& = & \sum_j(x_j-h_1({\bf x})+h_1({\bf x})-\theta_1)^2\\
			& = & \sum_j(x_j-h_1({\bf x}))^2+n(h_1({\bf x})-\theta_1)^2, 
		\end{eqnarray*}
		which leads to the factorization
		\begin{equation}\label{fact:norm:tot}
			L({\bf x};\theta_1)=\underbrace{(2\pi\theta_2)^{-n/2}
				e^{-\frac{\sum_j(x_j-h_1({\bf x}))^2}{2\theta_2}}}_{\xi({\bf x})}
			\underbrace{e^{-\frac{n(h_1({\bf x})-\theta_1)^2}{2\theta_2}}}_{\eta(h_1({\bf x}),\theta_1)}.
		\end{equation}
		This shows that $h_1$ is a sufficient statistic for $\theta_1$ (given $\theta_2$). 
		\item ($\theta_1$ is known and $\theta_2$ is the unknown parameter) Here, 
		\[
		h_2({\bf x})=\sum_j(x_j-\theta_1)^2
		\]
		qualifies as a statistic and
		\[
		L({\bf x};\theta_2)=\underbrace{(2\pi\theta_2)^{-n/2}e^{-\frac{h_2({\bf x})}{2\theta_2}}}_{\eta(h_2({\bf x}),\theta_2)}\times \underbrace{1}_{\xi({\bf x})}
		\]
		shows that $h_2$ is a sufficient statistic for $\theta_2$ (given $\theta_1$). 
		\item ($\theta=(\theta_1,\theta_2)$ is the unknown bi-dimensional parameter). Here we set 
		\[
		\widetilde h_2({\bf x})=\sum_j(x_j-h_1({\bf x}))^2
		\]
		so (\ref{fact:norm:tot}) gives
		\[
		L({\bf x};\theta)=\underbrace{(2\pi\theta_2)^{-n/2}e^{-\frac{\widetilde h_2({\bf x})+n(h_1({\bf x})-\theta_1)^2}{2\theta_2}}}_{\eta((h_1({\bf x}),\widehat h_2({\bf x}),\theta))}\times \underbrace{1}_{\xi({\bf x})},
		\]
		which shows that $H({\bf x})=(h_1({\bf x}),\widetilde h_2({\bf x}))$ is a sufficient statistic for $\theta$. We thus see that the common practice, which has been extensively used in Subsection \ref{conf:int:sub}, of regarding $H$ as a sufficient statistic when sampling from a normal population, is fully justified.  \qed
	\end{itemize}
\end{example}

\begin{example}\label{suff:expon}(Sufficiency in an exponential population)
	If $X_j\sim{\rm Exp}(\lambda)$ then from Example \ref{mle:exp} we get 
	\[
	L({\bf x};\lambda)=\underbrace{\lambda^ne^{-\lambda h({\bf x})}}_{\eta(h({\bf x}),\lambda)}\times \underbrace{1}_{\xi({\bf x})}, 
	\]
	where $h({\bf x})=\sum_jx_j$ is a sufficient statistic for $\lambda$. \qed
\end{example}

\begin{example}\label{suff:bern}(Sufficiency in a Bernoulli or Poisson population)
	If $X_j\sim\mathsf{Ber}(p)$ then Example \ref{mle:discrete} gives
	\[
	L({\bf x};p)=\underbrace{p^{h({\bf x})}(1-p)^{n-h({\bf x})}}_{\eta(h({\bf x}),p)}\times 
	\underbrace{1}_{\xi({\bf x})},
	\]
	which shows that $h({\bf x})=\sum_jx_j$ is a sufficient statistic for estimating $p$. On the other hand,  if $X_j\sim{\mathsf{Pois}}(\rho)$ then, again by Example \ref{mle:discrete}, 
	\[
	L({\bf x};\rho)=\underbrace{\rho^{k({\bf x})}e^{-n\rho}}_{\eta(k({\bf x}),\rho)}\times 
	\underbrace{(\Pi_jx_j!)^{-1}}_{\xi({\bf x})},
	\]
	which confirms that $k({\bf x})=\sum_jx_j$ is a sufficient statistic for estimating $\rho$. \qed
\end{example}

\begin{example}\label{exp:fam:suff} (Sufficiency in the natural exponential family)
If $Y_j\sim {\sf ExpFam}(\theta,\phi)$ as in Example \ref{rem:ident:expfam} then 
	\[
L({\bf y};\theta)=\underbrace{e^{\frac{{\theta h({\bf y})}-nb(\theta)}{\phi}}}_{\eta(h({\bf y}),\theta)}\times 
\underbrace{e^{\sum_jc(y_j,\phi)}}_{\xi({\bf y})},
\]	
which shows that $h({\bf y})=\sum_jy_j$ is a sufficient statistic for estimating $\theta$. \qed	
	\end{example}

As illustrated by the computations in Remark \ref{conf:int:f}, at least in the regime of large samples it follows from (\ref{conf:int:th:up}) that the dependence on sample data of confidence intervals for ML estimators  occurs only through the estimator itself. This general observation clearly  aligns with Corollary \ref{mle:suff} and is definitely confirmed by all the examples examined above, where a simple relationship of the given sufficient statistic with the corresponding ML estimator is manifest.

\section{Hypothesis testing}\label{sec:hyp:test}
Our aim here is to discuss a bit more on the
heuristics behind the choices of the rejection regions appearing in the $\textsf F$-tests implemented in Remark \ref{F:test:q:var} and Example \ref{ow:anova} above.

\subsection{A glimpse at the Neyman-Pearson setup}\label{subsec:ht:setup}
 As usual, we are given a parametric statistic model 
\[
X_1,\cdots,X_n\sim\psi_\theta, \quad \theta\in\Theta\subset\mathbb R^p
\] 
as	in Definition \ref{stat:mod:def}  and Remark \ref{underlying}, so that 
$(\Omega,\mathcal F,\{\mathcal P_\theta\}_{\theta\in\Theta})$ is the underlying family of probability spaces and $P_\theta=\psi_\theta dx$ is the common distribution of the components of the associated random vector $X=(X_1,\cdots,X_n):\Omega\to \mathbb R^n$. 
Given {\em disjoint} subsets  $\Theta_0, \Theta_{\rm a}\subset \Theta$ with $\Theta=\Theta_{0}\cup\Theta_{\rm a}$, {\em hypothesis testing} concerns the prospect of using the available data in an observed value ${\bf x}$ of
$X$ to provide statistical evidence for deciding between the {\em null hypothesis}
\[
H_0: \quad \theta\in\Theta_0
\]
and the {\em alternative hypothesis}
\[
H_{\rm a}: \quad \theta\in\Theta_{\rm a}.
\]
One adheres to the usual asymmetry in regarding $H_0$ as the {\em status quo} and then chooses a statistics $h=h(X):\Omega\to\mathbb R$ and a {\em rejection region} $R\subset\mathbb R$ so that $H_0$ gets {\em rejected} if the realization $h({\bf x})$ of $h(X)$ takes value in $R$. A pair $T=(h,R)$ as above is called a {\em test} for the given statistical model and we denote by $\mathscr T$ the collection of all such tests\footnote{For instance, if $R=[r,+\infty)$ then we say that $r$ is a {\em critical value} for the test.}. As we will see in Remark \ref{size:test:3} below, the eventual implementation of a test $T\in\mathscr T$ necessarily involves the knowledge of the distribution of (a perhaps complicated function of) $h(X)$ under the null hypothesis. 

In order to quantify the possible types of errors in making such a  decision, we consider the {\em power function} $\pi:\Theta\to[0,1]$ of $(h,R)$, 
\[
\pi(\theta)=\mathcal P_\theta(h(X)\in R).
\] 
We then see that restriction to $\Theta_0$, namely, 
\[
\gamma(\theta):=\pi|_{\Theta_0}(\theta), \quad \theta\in\Theta_0,
\]   
quantifies the {\em type I error} of rejecting $H_0$ when it is true, whereas restriction to $\Theta_{\rm a}$,
\[
\delta(\theta):=\pi|_{\Theta_{\rm a}}(\theta), \quad \theta\in\Theta_{\rm a},
\]
is such that 
\[
1-\delta(\theta)=\mathcal P_\theta(h(X)\notin R)
\]
measures the  {\em type II error} of {\em not} rejecting $H_0$ when it is false. Ideally, one would seek for an strategy minimizing {\em both} errors at a time, but simple examples show that this is doomed to fail in general. The standard way to overcome this is to 
search for a test which minimizes type II error under the constraint that type I error remains uniformly bounded from above by a fixed amount given in advance. 

\begin{definition}\label{def:size}
	Given $\alpha\in(0,1)$ we say that a test $T\in\mathscr T$ has {\em confidence level} $\alpha$ if 
	\begin{equation}\label{def:size:2}
		\sup_{\theta\in\Theta_0}\gamma(\theta)= \alpha, 
	\end{equation}
	and we denote by $\mathscr T_\alpha$ the collections of all such tests. 
\end{definition}

\begin{definition}\label{def:best:test}
	A test $T\in\mathscr T_\alpha$ is {\em uniformly most powerful} (UMP) if it satisfies (with self-explanatory notation)
	\[
	\delta(\theta)\geq \delta^*(\theta), \quad \theta\in\Theta_{\rm a},
	\]
	for any $T^*\in\mathscr T_\alpha$.
\end{definition}	

\begin{remark}\label{size:test:3}
	Note that (\ref{def:size:2}), which may be rewritten as
	\begin{equation}\label{size:test:4}
		\sup_{\theta\in\Theta_0}\mathcal P_\theta(h(X)\in R)=\alpha,
	\end{equation}
	allows us to explicitly determine the rejection region $R$ from the confidence level $\alpha$ only in case a (perhaps approximate) knowledge of the distribution of $h(X)$ {\em under} $H_0$ is at hand, a procedure  illustrated in the examples considered below. In other words, the ubiquitous ``Problem of Distribution'' in Parametric Statistics resurfaces in this setting as well, although here the relevant statistics $h(X)$ gets restricted to the parametric region where the null hypothesis holds true. \qed   
\end{remark}

The celebrated Neyman-Pearson lemma \cite[Theorem 8.3.12]{casella2021statistical} exhibits a UMP test in the simple hypotheses case, where both $\Theta_0$ and $\Theta_{\rm a}$ contain a single element. Unfortunately, such a test may not exist even for one of the simplest {\em composite} hypotheses  cases, namely, a ``two-sided'' test of the form $\Theta\subset \mathbb R$ some open interval, $\Theta_0=\{\theta_0\}$ for some $\theta_0\in\Theta$ and $\Theta_1=\Theta\backslash\Theta_0$ 
(as in Remark \ref{F:test:q:var}, for instance); see \cite[Example 8.3.19]{casella2021statistical} and the surrounding discussion for more on this rather delicate point. Of course, we may always restrict further the class of contenders where the ideal test should be sought (consistent, unbiased, etc.) but it seems that none of these strategies produces a test with optimal performance in {\em all} cases. 

\subsection{Testing via likelihood ratios}\label{subsec:lratios}
The state of affairs indicated in the previous paragraph suggests that, instead of {\em systematically} trying to find the best test in a given context, one should proceed {\em heuristically} so  as to single out  a family of tests which are relatively easy to implement, reproduce most known composite tests for samples of any size and have nice asymptotic properties (see Remark \ref{kl:int:lrt} below).

Recall from Subsection \ref{asym:mle:est} that the ML estimators, 
computed in terms of the likelihood function as in Definition \ref{likeli:def}, have many remarkable properties, including asymptotic normality. Moreover,
given the available information contained in the realization ${\bf x}$ of the random sample $X$, the discussion surrounding (\ref{MLE:heur}) justifies  
regarding  $\sup_{\theta\in\Theta_0}L({\bf x};\theta)$ as the best evidence in favor of $H_0$ and $\sup_{\theta\in\Theta_{\rm a}}L({\bf x};\theta)$ as the best evidence in favor of $H_{\rm a}$, which suggests  formulating a hypothesis test based on the {\em likelihood ratio}
\begin{equation}\label{like:r}
	{\bf x}\in\mathbb R^n\mapsto\frac{\sup_{\theta\in\Theta_0}L({\bf x};\theta)}{\sup_{\theta\in\Theta_{\rm a}}L({\bf x};\theta)}\in [0,+\infty].
\end{equation} 
At the risk of (over)simplifying the exposition, but at the same time remaining in a generality that will suffice for the applications we have in mind, we assume from now on that $\Theta_0\subset \Theta$ has a negligible size (as a subset of $\Theta$) and $\Theta_{\rm a}=\Theta\backslash\Theta_0$. Typically, $\Theta\subset \mathbb R^p$ will be an open subset, $\Theta_0$  the portion of an affine $k$-plane lying in $\Theta$ with $k<p$ and $\Theta_{\rm a}$  its complement in $\Theta$.
In this setting it seems reasonable to
replace $\Theta_{\rm a}$ by $\Theta$ in the denominator of (\ref{like:r}), so that under suitable regularity assumptions the likelihood ratio becomes 
\begin{equation}\label{like:r:s}
	{\bf x}\in\mathbb R^n\mapsto
	\Lambda({\bf x}):=
	\frac{L({\bf x};\widehat\theta_0)}{L({\bf x};\widehat\theta)}\in [0,1],
\end{equation}	
where 
\[
\widehat\theta=\sup_{\theta\in\Theta}L({\bf x};\theta)
\]
is the MLE for $\theta$ (as in Definition \ref{mle:def:post}).
and 
\[
\widehat\theta_0=\sup_{\theta\in\Theta_0}L({\bf x};\theta)
\]
is the {\em null} MLE for $\theta$. 
This leads to a remarkable class of statistical tests.

\begin{definition}\label{def:lrt}
	Under the conditions above, the {\em likelihood ratio test} $T=(h,R)\in\mathscr T_\alpha$ is performed by choosing
	\[
	h({\bf x})=-2\ln\Lambda({\bf x})
	\]
	and $R=[r,+\infty)$, where $r>0$ is determined by 
	\begin{equation}\label{rej:int}
		\sup_{\theta\in \Theta_0}\mathcal P_\theta(h(X)\geq r)= \alpha. 
	\end{equation}
\end{definition}

\begin{remark}\label{size:test:5}
	Clearly, (\ref{rej:int}) is a special case of (\ref{size:test:4}), where the rejection region $R$ now takes the form $[r,+\infty)$ for some critical value $r>0$, and we are supposed to solve it for $r$ given the confidence level $\alpha$. But notice that, as already observed in Remark \ref{size:test:3}, this requires knowing the distribution of $h$ under $H_0$. In this regard, if $h(X)$  is found to be $\Theta_0$-{\em ancillary} in the sense that its distribution does {\em not} depend on $\theta\in \Theta_0$ then  
	\begin{equation}\label{rej:int:2}
		\alpha=\mathcal P_\theta(h(X)\geq r)\,\,{\rm for}\,\,{\rm any}\,\,\theta\in\Theta_0 
	\end{equation}
	determines 	$r$ as a function of $\alpha$. 
	\qed
\end{remark}

\begin{remark}\label{kl:int:lrt}
It follows from (\ref{kl:arg:3}) that the likelihood ratio statistics satisfies
\[
h(X)\approx 2n \left(D^{KL}_{\widetilde\theta_X}(\widehat\theta_0)-D^{KL}_{\widetilde\theta_X}(\widehat\theta)\right), 
\]	
so that likelihood ratio tests admit a natural information-theoretic interpretation: they measures how much KL divergence to the empirical distribution $\widetilde\theta_X$ is reduced when passing from the null model to the full model. Again, taking into account that $\widetilde\theta_X$ is consistent (for the true parameter), we see that the likelihood ratio statistics asymptotes the gap, as measured by the Kullback-Leibler divergence, between the best approximations of the truth under the null and alternative.
\qed
	\end{remark}

Although we will restrict ourselves to normal populations, the examples below will suffice to illustrate the remarkable flexibility of this construction. Moreover, in all these examples the likelihood ratio statistics is $\Theta_0$-ancillary in the sense of Remark \ref{size:test:5}, so that (\ref{rej:int:2}) applies in order to solve for $r$ in terms of $\alpha$.

\begin{example}\label{z:test}
	($z$-test for the mean of a normal population with known variance) 
	As in Example \ref{normal:w} we
	assume that $X_j\sim\mathcal N(\mu,\sigma^2)$, where $\theta_2=\sigma^2$ is known. Thus, $\Theta=\{\theta_1=\mu\}=\mathbb R$, $\Theta_0=\{\mu_0\}$ for some $\mu_0\in\mathbb R$, $\Theta_a=\mathbb R\backslash\{\mu_0\}$ and we want to test
	\[
	H_0: \mu=\mu_0\quad{\rm vs}\quad  H_{\rm a}:\mu\neq \mu_0.
	\]
	We recall that $\widehat\theta_1=\overline X_n$ is the MLE for $\mu$. Using (\ref{normal:w:1}) and (\ref{U:sigma:V:c}) we compute 
	\begin{eqnarray*}
		h({\bf x})
		& = & 
		-2\ln\frac{L({\bf x};\mu_0)}{L({\bf x};\widehat\theta_1)}\\
		& = & 
		-2\ln \frac{(2\pi\sigma^2)^{-n/2}e^{-\frac{1}{2\sigma^2}
				\sum_j(x_j-\mu_0)^2}}{(2\pi\sigma^2)^{-n/2}
			e^{-\frac{1}{2\sigma^2}
				\sum_j(x_j-\widehat\theta_1)^2}}\\
		& = & -2\ln e^{-\frac{n}{2\sigma^2}(\widehat\theta_1-\mu_0)^2}\\
		& = & \left(\frac{\overline x_n-\mu_0}{\sigma/\sqrt{n}}\right)^2. 
	\end{eqnarray*}
	Thus, under $H_0$ we see that
	\[
	h(X)=Z(X)^2\sim\chi^2_1,
	\] 		
	where
	\[
	Z(X)=\frac{\overline X_n-\mu_0}{\sigma/\sqrt{n}}\sim\mathcal N(0,1). 
	\]	
	Now, by (\ref{rej:int}) the rejection interval $R=[r,+\infty)$ is determined by
	\[
	\mathcal P_{\theta_0}(Z(X)^2\geq r)=\alpha,
	\]
	so we may take  $r=\chi^2_{1,1-\alpha}$, the $\chi^2$-quantile  as in (\ref{quantile:chi}). Since
	\[
	\alpha=\mathcal P_{\theta_0}(Z(X)^2\geq {r})=\mathcal P_{\theta_0}(-\sqrt{r}\leq Z(X)\leq \sqrt{r}), 
	\]
	we may also use the standard normal quantiles to make sure that if
	\[
	Z({\bf x})\in(-\infty,-z_{1-\alpha/2}]\cup [z_{1-\alpha/2},+\infty) 
	\]
	then $H_0$ gets rejected.
	\qed
\end{example}

\begin{example}\label{t:test}
	($t$-test for the mean of a normal population with unknown variance) 
	As in Example \ref{normal:w} we
	assume that $X_j\sim\mathcal N(\mu,\sigma^2)$, where $(\theta_1,\theta_2)=(\mu,\sigma^2)$ is unknown, so that 
	$\Theta=\mathbb R\times \mathbb R_+$.
	Also, we fix $\mu_0\in\mathbb R$ and set
	$\Theta_0=\{\mu_0\}\times\mathbb R_+$, a half-line contained in $\Theta$. 
	As before,  we want to test
	\[
	H_0: \mu=\mu_0\quad{\rm vs}\quad H_{\rm a}:\mu\neq \mu_0.
	\]
	We recall that $\widehat\theta=(\widehat\theta_1,\widehat\theta_2)$, where 
	$\widehat\theta_1=\overline X_n$ 
	and 
	\[
	\widehat\theta_2=\frac{1}{n}\sum_j(X_j-\overline\theta_1)^2
	\]
	is the MLE for $\theta_2$, 
	so that
	\[
	\sup_{\theta\in\Theta}L({\bf x};\theta)=L({\bf x};\widehat \theta_1,\widehat\theta_2). 
	\]
	On the other hand, one has 
	\[
	\sup_{\theta\in\Theta_{0}}L({\bf x};\theta)=L({\bf x};\mu_0,\widehat\theta_{20}),
	\]
	where the null MLE for $\theta_2$ is
	\[
	\widehat\theta_{20}=\frac{1}{n}\sum_j(x_j-\mu_0)^2. 
	\]
	Thus, the likelihood ratio is
	\begin{eqnarray*}
		\Lambda({\bf x})
		& = &
		\left(\frac{\widehat\theta_{20}}{\widehat\theta_{2})}\right)^{-n/2}
		\frac{e^{-\frac{1}{2\widehat\theta_{20}}\sum_j(x_j-\mu_0)^2}}{e^{-\frac{1}{2\widehat\theta_{2}}\sum_j(x_j-\widehat\theta_1)^2}}\\
		& = & 
		\left(\frac{\widehat\theta_{20}}{\widehat\theta_{2}}\right)^{-n/2}
		\frac{e^{-n/2}}{e^{-n/2}}\\
		& = & 
		\left(\frac{\widehat\theta_{20}}{\widehat\theta_{2}}\right)^{-n/2},
	\end{eqnarray*}
	so that, again using (\ref{U:sigma:V:c}),
	\[
	h({\bf x})=n\ln\left[1+\frac{1}{n-1}\left(\frac{\overline x_n-\mu_0}{s_n(x)/\sqrt{n}}\right)^2\right].
	\]
	Thus, under $H_0$ we see that 
	\begin{equation}\label{h:toexp}
		h(X)=n\ln\left[1+\frac{1}{n-1}T_{n-1}(X)^2\right],
	\end{equation}
	where
	\[
	T_{n-1}(X)=\frac{\overline X_n-\mu_0}{S_n(X)/\sqrt{n}}\sim\mathfrak t_{n-1},
	\]
	or also
	\[
	T_{n-1}(X)^2\sim \textsf F_{1,n-1}
	\]
	by Corollary \ref{f-dis:cons}.
	Quite informally, we may expand (\ref{h:toexp}) as $n\to+\infty$ to obtain
	\begin{eqnarray*}
		h(X)
		& = & \ln\left[1+\frac{1}{n-1}T_{n-1}(X)^2\right]^n\\
		& = & \ln\left[1+\frac{n}{n-1}T_{n-1}(X)^2+\cdots\right]\\
		& = & \frac{n}{n-1}T_{n-1}(X)^2+\cdots\\
		& \stackrel{p}{\to}&
		\chi^2_1,
	\end{eqnarray*}
	where the dots represent lower order terms (which vanish as $n\to+\infty$) and we used Remark \ref{tk:to:norm} in the last step. Thus, for large samples we may take $R=[\chi^2_{1,1-\alpha},+\infty)$ as the ``approximate'' rejection interval. Otherwise, we use 
	that
	\[
	\alpha=\mathcal P_{\theta_0}(T_{n-1}(X)^2\geq 	c_{n,r})=
	\mathcal P_{\theta_0}(-\sqrt{c_{n,r}}\leq T_{n-1}(X))\leq 	\sqrt{c_{n,r}}),
	\]
	where
	\[
	c_{n,r}=(n-1)(e^{r/n}-1),
	\]
	to reject $H_0$ if either
	\[
	T_{n-1}({\bf x})^2\in\left[\textsf f_{1,n-1,1-\alpha},+\infty\right)
	\]
	or equivalently
	\[
	T_{n-1}({\bf x})\in(-\infty,\mathfrak t_{n-1,\alpha/2}]\cup [\mathfrak t_{n-1,1-\alpha/2},+\infty),
	\]		
	which is a more familiar presentation of the test.
	\qed		 
\end{example}

\begin{example}\label{f:test:dif:v}($\textsf F$-test for the equality of variances of independent normal populations) We will use the notation of Example \ref{two:samples} and Remark \ref{F:test:q:var} 
	with the aim of testing 
	\[
	H_0: \sigma_X^2=\sigma_Y^2\quad{\rm vs}\quad H_{\rm a}:\sigma_X^2\neq\sigma_Y^2.
	\]
	Since we assume independence of the samples, Example \ref{normal:w} implies that the corresponding likelihood function is  
	\[
	L({\bf x},{\bf y};\theta)
	=  
	\frac{1}{(2\pi)^{(m+n)/2}\theta_{2X}^{m/2}\theta^{n/2}_{2X}} 
	e^{-\frac{1}{2}\left(\sum_{j=1}^m
		\frac{(x_j-\theta_{1X})^2}{\theta_{2X}}
		+
		\sum_{k=1}^n
		\frac{(y_k-\theta_{1Y})^2}{\theta_{2Y}}\right)},
	\]
	where
	\[
	\theta=(\theta_{1X},\theta_{1Y},\theta_{2X},\theta_{2Y})=(\mu_X,\mu_Y,\sigma^2_X,\sigma^2_Y)\in\mathbb R^2\times\mathbb R_+^2,
	\]
	so the MLE estimators are
	\[
	\widehat\theta_{1X}=\overline X_m,\quad \widehat\theta_{1Y}=\overline Y_m
	\] 	 
	and 
	\[
	\widehat\theta_{2X}=\frac{m-1}{m}S^2_X=
	\frac{1}{m}\sum_j(X_j-\widehat\theta_{1X})^2, \quad 
	\widehat\theta_{2Y}=\frac{n-1}{n}S^2_Y=
	\frac{1}{n}\sum_k(Y_k-\widehat\theta_{1Y})^2.
	\]
	Hence, 
	\begin{eqnarray*}
		\sup_{\theta\in\Theta} L({\bf x},{\bf y};\theta)
		& = & 
		L({\bf x},{\bf y};\widehat\theta_{1X},\widehat\theta_{1Y},\widehat\theta_{2X},
		\widehat\theta_{2Y})\\
		& = & 
		\frac{1}{(2\pi)^{(m+n)/2}\widehat\theta_{2X}^{m/2}\widehat\theta_{2Y}^{n/2}} 
		e^{-\frac{1}{2}\left(\sum_{j=1}^m
			\frac{(x_j-\widehat\theta_{1X})^2}{\widehat\theta_{2X}}
			+
			\sum_{k=1}^n
			\frac{(y_k-\widehat\theta_{1Y})^2}{\widehat\theta_{2Y}}\right)}\\
		& = &
		\frac{e^{-(m+n)/2}}{(2\pi)^{(m+n)/2}\widehat\theta_{2X}^{m/2}\widehat\theta^{n/2}_{2X}}. 
	\end{eqnarray*}	
	On the other hand, restriction to $\Theta_0$ gives
	\[
	L({\bf x},{\bf y};\theta)
	=  
	\frac{1}{(2\pi)^{(m+n)/2}\theta_2^{(m+n)/2}} 
	e^{-\frac{1}{2\theta_2}\left(\sum_j
		(x_j-\theta_{1X})^2
		+\sum_k{(y_k-\theta_{1Y})^2}\right)},
	\]
	where $\theta_2=\sigma_X^2=\sigma_Y^2$ is the common variance, so that maximization over all possible values of $(\theta_{1X},\theta_{1Y},\theta_{2})$ is achieved at the null MLE $(\widehat\theta_{1X},\widehat\theta_{1Y},\widehat\theta_{20})$,
	where
	\[
	\widehat\theta_{20}=\frac{1}{m+n}\left(\sum_j
	(x_j-\widehat\theta_{1X})^2
	+\sum_k{(y_k-\widehat\theta_{1Y})^2}\right).
	\]
	It follows that
	\begin{eqnarray*}
		\sup_{\theta\in\Theta_0} L({\bf x},{\bf y};\theta)
		& = & 
		L({\bf x},{\bf y};\widehat\theta_{1X},\widehat\theta_{1Y},\widehat\theta_{20})\\
		& = & 
		\frac{1}{(2\pi)^{(m+n)/2}\widehat\theta_{20}^{(m+n)/2}} 
		e^{-\frac{1}{2\widehat\theta_{20}}\left(\sum_j
			((x_j-\widehat\theta_{1X})^2 +\sum_k
			+{(y_k-\widehat\theta_{1Y})^2}\right)}\\
		& = &
		\frac{e^{-{(m+n)/2}}}{(2\pi)^{(m+n)/2}\widehat\theta_{20}^{(m+n)/2}}, 
	\end{eqnarray*}
	so the likelihood radio is 
	\[
	\Lambda({\bf x},{\bf y})=\left(\frac{\widehat\theta_{20}}{\widehat\theta_{2X}}\right)^{-m/2}
	\left(\frac{\widehat\theta_{20}}{\widehat\theta_{2Y}}\right)^{-n/2},
	\]
	and from this we easily deduce that
	\[
	h({\bf x},{\bf y})
	= 
	\ln\left[c
	\left(1+a\left(\frac{S_{X}^2}{S_{Y}^2}\right)^{-1}\right)^m
	\left(1+b\frac{S_{X}^2}{S_{Y}^2}\right)^n
	\right],
	\]
	where $a$, $b$ and $c$ are certain constants depending only on $m$ and $n$ with $ab=1$.  
	Thus, as in Remark \ref{F:test:q:var} we see that under  $H_0$,
	\[
	h(X,Y)=
	\ln\left(c
	\left(1+aU(X,Y)^{-1}\right)^m
	\left(1+bU(X,Y)\right)^n
	\right),
	\] 
	where 
	\[
	U(X,Y)=\frac{S^2_X}{S^2_Y}\sim\textsf F_{m-1,m-1}.
	\]
	Now note that
	$	h({\bf x},{\bf y})\geq e^r$ if and only if $u:=U({\bf x},{\bf y})$ satisfies $f(u)\geq e^r/c$, where 
	\[
	f(u):=(1+au^{-1})^m(1+bu)^n, \quad u>0.
	\]
	Also, a little Calculus shows that $f$ is strictly convex with its unique minimal value achieved at $u_0=ma/n=m/nb$ (this analysis uses that $ab=1$ in a crucial way). Hence, there exist a {\em maximal} $0<\underline u<u_0$ and a {\em minimal} $\overline u>u_0$ with the property that 
	for any $r$ such that
	\[
	f\left(\frac{ma}{n}\right)=f\left(\frac{m}{nb}\right)=\left(1+\frac{n}{m}\right)^m
	\left(1+\frac{m}{n}\right)^n<\frac{e^r}{c}
	\]
	there holds 
	$f(u)\geq e^r/a$ whenever either $u\leq\underline u$ or $u\geq \overline u$.  This means that we may reject $H_0$ if $u$ falls outside $(\underline u,\overline u)$. More concretely, given a confidence level $\alpha$ small enough, we may reject $H_0$ if  
	\[
	u\in \left(0,\textsf f_{m-1,n-1,\alpha/2}\right]\bigcup \left[\textsf f_{m-1,n-1,1-\alpha/2},+\infty\right),
	\]
	which is consistent with (\ref{two:sided}).
	\qed
\end{example}

\begin{example}\label{t:text:dif:m}($\textsf F$-test for the equality of means of $p\geq 2$ independent normal populations with a common but unknown variance) We will use the notation from Example \ref{ow:anova}, so we have {\em independent} random samples $X_{jk}\sim\mathcal N(\mu_j,\sigma^2)$ for $j=1,\cdots,p$. Thus, 
	$\Theta=\mathbb R^{p}\times \mathbb R_+$ with $\theta=(\theta_{11},\cdots,\theta_{1p},\theta_2)$, where  $\theta_{1j}=\mu_j$ and $\theta_2=\sigma^2$, the common variance. Also, 
	$\Theta_0=\{\theta\in\Theta;\theta_{11}=\cdots=\theta_{1p}\}$ and we want to test
	\[
	H_0:\mu_1=\cdots=\mu_p\quad{\rm vs}\quad H_{\rm a}: \mu_j\neq \mu_{j'}\,\,{\rm for}\,\,{\rm some}\,\,j\neq j'.
	\] 
	Notice that this is precisely the one way ANOVA test in Example \ref{ow:anova}. If ${\bf x}=({\bf x}_1,\cdots,{\bf x}_p)$ with ${\bf x}_j=({\bf x}_{j1},\cdots,{\bf x}_{jn_j})\in\mathbb R^{n_j}$ then the likelihood function is
	\begin{eqnarray*}
		L({\bf x};\theta)
		& = & 
		\Pi_{j=1}^p\Pi_{k=1}^{n_j}{(2\pi \theta_2)^{-n_j/2}}
		e^{-\frac{1}{2\theta_2}\sum_{k=1}^{n_j}(x_{jk}-\theta_{1j})^2}\\
		& = & 
		{(2\pi \theta_2)^{-n/2}}
		e^{-\frac{1}{2\theta_2}\sum_{j=1}^p\sum_{k=1}^{n_j}({\bf x}_{jk}-\theta_{1k})^2}.
	\end{eqnarray*}
	By passing to the log likelihood function and maximizing over $\Theta$ in the usual way we see that the MLE for $\theta$ is $\widehat\theta=(\widehat\theta_{{11}},\cdots,\widehat\theta_{1p},\widehat\theta_2)$, where
	\[
	\widehat\theta_{{1j}}({\bf x})=\overline{\bf x}_{j\bullet}=\frac{1}{n_j}\sum_{k=1}^{n_j}{\bf x}_{jk} \,\,{\rm and}\,\, 
	\widehat\theta_2({\bf x})=\frac{1}{n}\sum_{j=1}^p\sum_{k=1}^{n_j}({\bf x}_{jk}-\overline{\bf x}_{j\bullet})^2 
	\] 
	are realizations of $\overline X_{j\bullet}$ and $S^2_{\rm Within}(X)/n$, respectively. 
	On the other hand, restricted to $\Theta_0$ we have
	\[
	L({\bf x};\theta)=
	{(2\pi \theta_2)^{-n/2}}
	e^{-\frac{1}{2\theta_2}\sum_{j=1}^p\sum_{k=1}^{n_j}({\bf x}_{jk}-\theta_0)^2},
	\]
	where $\theta_0=\mu_1=\cdots=\mu_p$ is the common mean, so that maximization over $\Theta_0$ gives the corresponding null MLEs
	\[
	\widehat\theta_{00}({\bf x})={\bf x}_{\bullet\bullet}=\frac{1}{n}\sum_{j=1}^pn_j\overline{\bf x}_{j\bullet}=\frac{1}{n}\sum_{j=1}^p\sum_{k=1}^{n_j}{\bf x}_{jk}\,\,{\rm and}\,\, 
	\widehat\theta_{20}({\bf x})=\frac{1}{n}\sum_{j=1}^p\sum_{k=1}^{n_j}({\bf x}_{jk}-\overline{\bf x}_{\bullet\bullet})^2,
	\]
	which
	are realizations of $\overline X_{\bullet\bullet}$ and $S^2_{\rm Total}(X)/n$, respectively. It follows that
	\begin{eqnarray*}
		\Lambda({\bf x})
		& = & 
		\frac
		{{(2\pi \widehat\theta_{20}({\bf x}))^{-n/2}}
			e^{-\frac{1}{2\widehat\theta_{20}({\bf x})}\sum_{j=1}^p\sum_{k=1}^{n_j}({\bf x}_{jk}-{\bf x}_{\bullet\bullet})^2}}
		{{(2\pi \widehat\theta_{2}({\bf x}))^{-n/2}}
			e^{-\frac{1}{2\widehat\theta_{2}({\bf x})}\sum_{j=1}^p\sum_{k=1}^{n_j}({\bf x}_{jk}-{\bf x}_{j\bullet})^2}}\\
		& = &
		\left(\frac{\widehat\theta_{20}({\bf x})}{\widehat\theta_{2}({\bf x})}\right)^{-n/2}
		\frac{e^{-n/2}}{e^{-n/2}},		
	\end{eqnarray*}
	so that using (\ref{ow:anova:2}), 
	\[
	h({\bf x})=n\ln\left(1+\frac{p-1}{n-p}\frac{s^2_{\rm Between}({\bf x})/(p-1)}
	{s^2_{\rm Within}({\bf x})/(n-p)}\right).
	\]
	Hence, as in Example \ref{ow:anova} we see that under $H_0$,
	\[
	h(X)=n\ln \left(1+\frac{p-1}{n-p}V\right),
	\]
	where
	\[
	V=\frac{S^2_{\rm Between}/(p-1)}{S^2_{\rm Within}/(n-p)}\sim \textsf F_{p-1,n-p}.
	\]
	Again, we may expand this as $n\to +\infty$ to find that
	\begin{eqnarray*}
		h(X)
		& = & 
		\ln \left(1+\frac{p-1}{n-p}V\right)^n\\
		& = & 
		\ln\left(1+\frac{n(p-1)}{n-p}V+\cdots\right)\\
		& = & (p-1)V+\cdots\\
		& \stackrel{d}{\to} & \chi^2_{p-1}, 
	\end{eqnarray*}
	where we used Remark \ref{tk:to:norm} in the last step.  Thus, for large samples we may take $R=[\chi^2_{p-1,1-\alpha},+\infty)$ as the rejection interval. Otherwise, we use that
	\[
	\alpha=\sup_{\theta_0\in\Theta_0}\mathcal P_{\theta_0}\left(V({\bf x})\geq\frac{n-p}{p-1}\left(e^{r/n}-1\right) 	\right)
	\]
	to reject $H_0$ if
	\[
	V({\bf x})\in\left[\textsf f_{p-1,n-p,1-\alpha},+\infty\right)
	\]
	as in (\ref{one:sided}).\qed
\end{example}			

\begin{example}\label{ex:f:reg}($\textsf F$-test for statistical significance of the linear regression model)
	We consider here the linear regression model in (\ref{model:mls}), whose likelihood function $L({\bf y};\beta,\sigma^2)$ is given by (\ref{like:lrmod}), in order to 
	test the full ``intercept-only'' hypothesis appearing in Example \ref{weigh:comb}:
	\begin{equation}\label{null:hyp:lm}
		H_0: \beta_1=\cdots=\beta_p=0\quad {\rm vs}\quad H_{\rm a}: \beta_j\neq 0\,\,{\rm for}\,\,{\rm some}\,\,j.		
	\end{equation}
	In other words, the null hypothesis here says that ${\bf X}$ has no influence whatsoever on ${\bf Y}$ so its rejection provides statistical evidence for employing the model as it is posed in Example \ref{mle:imp:lsm:n} (that is, with the full ``slope'' $(\beta_1,\cdots,\beta_n)$ included). We have $\theta=(\beta,\theta_2)$, where $\theta_2=\sigma^2$, so the usual calculation implies that the corresponding MLE is $(\widehat\beta,\widehat\theta_2)$, where 
	\[
	\widehat\theta_2=\frac{1}{n}|{\bf Y}-\mathfrak x\widehat\beta|^2=
	\frac{1}{n}|\mathfrak x\beta+{\bf e}-\mathfrak x\widehat\beta|^2,
	\]
	so that (\ref{rbeta:e}) gives
	\[
	\widehat\theta_2=
	\frac{|\widehat{\bf e}|^2}{n}\\
	= 
	\frac{SS_{\rm Res}}{n},
	\]
	and we verify that
	\[
	\sup_{\theta\in\Theta} L({\bf y};\beta,\theta_2)
	=   L({\bf y};\widehat\beta,\widehat\theta_2)
	=  (2\pi {SS_{\rm Res}}/{n})^{-n/2}e^{-n/2}. 
	\]
	On the other hand, under the null hypothesis, 
	\[
	L({\bf y};\beta_0,\theta_2)=(2\pi\theta_2)^{-n/2}
	e^{-\frac{|{\bf y}-\beta_0{\bf 1}|^2}{2\theta_2}},
	\]
	so that the null MLE estimator for $\theta_2$ is
	\begin{equation}\label{est:t20}
		\widehat\theta_{20}=\frac{1}{n}\|{\bf Y}-\widehat\beta_0{\bf 1}\|^2\stackrel{(\ref{weigh:comb:2})}{=}\frac{1}{n}\|{\bf Y}-\overline{\bf Y}{\bf 1}\|^2=\frac{SS_{\bf Y\bf Y}}{n}, 
	\end{equation}
	which gives 
	\[
	\sup_{\theta\in\Theta_0} L({\bf y};\beta,\theta_2)
	=   L({\bf y};\widehat\beta_0,\widehat\theta_{20})
	=  (2\pi{SS_{\bf Y\bf Y}}/{n})^{-n/2}e^{-n/2}. 
	\]
	It then follows that the likelihood ratio statistics is 
	\begin{equation}\label{stat:test:lm}
		h({\bf y})= \ln\left(\frac{\widehat\theta_{20}}{\widehat\theta_2}\right)^n
		=
		\ln\left(\frac{SS_{\bf Y\bf Y}}{SS_{\rm Res}}\right)^n
		=
		\ln\left(1+\frac{SS_{\rm Reg}}{SS_{\rm Res}}\right)^n,
	\end{equation}
	where we used that, as in (\ref{dec:ss}), 
	\begin{equation}\label{dec:ss:2}
		SS_{\bf Y\bf Y}=SS_{\rm reg}+SS_{\rm Res}. 
	\end{equation}
	We now proceed to the appropriate counting of degrees of freedom as we did in Example \ref{t:text:dif:m}. 
	We see from Propositions \ref{rss:dist} and \ref{proc:conf:int:mls}
	that  $\sigma^{-2}SS_{\rm Res}\sim\chi^2_{n-p-1}$ is independent of $SS_{\rm Reg}$ and hence of $\widehat{\bf Y}=\mathfrak x\widehat\beta$, so $SS_{\rm Res}$ is independent of $SS_{\rm Reg}$ (recall that we are conditioning on $\mathfrak X=\mathfrak x$). On the other hand, 
	under $H_0$ we have ${\bf Y}\sim\mathcal N(\beta_0{\bf 1},\sigma^2I_{n\times n})$ and $\widehat\beta_0=\overline{\bf Y}$, which gives
	$\sigma^{-2}SS_{\bf Y\bf Y}\sim\chi^2_{n-1}$ by Proposition \ref{u:v}. Thus, again  under $H_0$, we get from (\ref{dec:ss:2}) that $\sigma^{-2}SS_{\rm Reg}\sim\chi^2_{p}$ and we conclude that
	\[
	h({\bf Y})=\ln\left(1+\frac{p}{n-p-1}W({\bf Y})\right)^n,
	\]
	where
	\[
	W({\bf Y})=\frac{SS_{\rm Reg}/p}{SS_{\rm Res}/(n-p-1)}\sim\textsf F_{p,n-p-1}.  
	\]
	As in Example \ref{t:text:dif:m}, $W({\bf Y})\stackrel{d}{\to}\chi_p^2$ as $n\to+\infty$, 
	so for large samples we may take $R=[\chi^2_{p,1-\alpha},+\infty)$ as the rejection interval. Otherwise, we must
	reject $H_0$ if
	\begin{equation}\label{much:info}
		W({\bf y})\in\left[\textsf f_{p,n-p-1,1-\alpha},+\infty\right).
	\end{equation}
	We mention that the theoretical procedure leading to the $\textsf F$-test above, based on a likelihood ratio test, is flexible enough to handle a general linear hypothesis test on the parameters, in which the null hypothesis may be expressed as $B\beta=c$, where $B$ is a suitable $q\times (p+1)$ matrix and $c$ is a $q$-vector; see \cite[Subsection 1.5]{amemiya1985advanced} and \cite[Chapter 4]{seber2003linear}. 
	For instance, if $c=\vec{0}$ and $B$ is suitably chosen then we can form the test
	\begin{equation}\label{null:hyp:lm:n}
		H_0^{\rm n}: \beta_{q+1}=\cdots=\beta_p=0\quad {\rm vs}\quad H_{\rm a}^{\rm n}: \beta_j\neq 0\,\,{\rm for}\,\,{\rm some}\,\,j\in\{q+1,\cdots,p\},		
	\end{equation}
	which compares the full model and a new null model in which only the first $q$ independent variables possibly appear as significant predictors; as indicated in Figure \ref{figg2}, the corresponding design spaces satisfy $C^{\rm n}(\mathfrak x)\subset C(\mathfrak x)$ with $\dim C(\mathfrak x)/ C^{\rm n}(\mathfrak x)=p-q$, this being the reason why the models are {\em nested}; see Remark \ref{geom:mls}. 
	\begin{figure}
		\begin{center}
			\begin{tikzpicture}[scale=2]
				\begin{scope}
					\draw[->][black, very thick](0,0,0) -- (2.2,1.8,0) node[midway,above]  {$\bf y$};
					\draw[-][black, very thick] (0,0,0) -- (1.992,  0.000, -0.249) node[midway,above]  {$\widehat{\bf y}$};
					\draw[->][black, very thick] (2.088,  0.000, -0.261) -- (2.4,0,-0.3);
					\draw[->][black, very thick] (0,0,0) -- (2.4,0,0.7) node[midway,below]  {$\widehat{\bf y}^{\rm n}$};
					\coordinate (O) at (2.4,0,-0.3);
					\coordinate (A) at (2.2,1.8,0);
					\draw[->][black, thick] (O) -- (A) node[midway, right]  {$\widehat{\bf e}$};
					\draw[-][black] (-1,0,2) -- (4,0,2) node[anchor=north]  {$C(\mathfrak x)$};
					\draw[-][black] (-1,0,2) -- (-1,0,-2)  ;
					\draw[-][black] (4,0,2) -- (4,0,-2)  ;
					\draw[-][black] (-0.72 , 0.00 ,-0.21) -- (3.60, 0.00, 1.05)  node[above] {$C^{\rm n}(\mathfrak x)$};
					\coordinate (B) at (2.4,0,0.7);
					\draw[->][dashed,thick] (O) -- (B)  ;
					\draw[->][black, thick] (B) -- (A) node[midway,left]  {$\widehat{\bf e}^{\rm n}$};
					\coordinate (OO) at (0,0,0);
					\draw[-][black] (B) -- (OO);
					\pic [draw, black, angle radius=3mm] {right angle = OO--B--A};	
					\pic [draw, black, angle radius=3mm] {right angle = A--O--B};
					\coordinate (BL) at (2.424, 0.000, 0.707);
					\pic [draw, black, angle radius=3mm] {right angle = O--B--BL};
					\pic [draw, black, angle radius=5mm] {angle = O--B--A};
				\end{scope}
			\end{tikzpicture}
		\end{center}
		\caption{The geometry of nested models}
		\label{figg2}
	\end{figure}
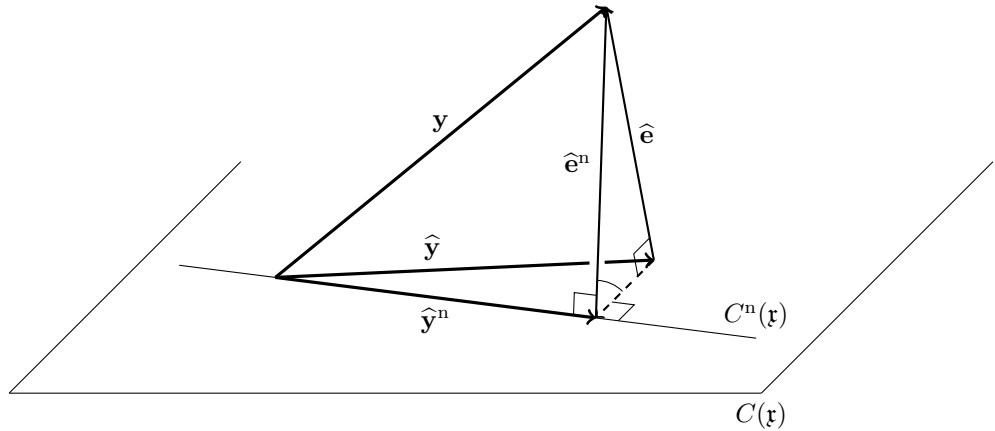
	If we view $SS_{\bf Y\bf Y}$ in (\ref{stat:test:lm}) and (\ref{dec:ss:2}) as the residual sum of squares (i.e.\! the norm squared residual) of the null model in (\ref{null:hyp:lm}) and proceed by analogy, it is not hard to check that the likelihood ratio statistics now is 
	\[
	h^{\rm n}({\bf Y})=
	\ln\left(\frac{SS_{\rm Res}^{\rm n}}{SS_{\rm Res}}\right)^n
	=
	\ln\left(1+\frac{p-q}{n-p-1}W^{\rm n}({\bf Y})\right)^n,
	\] 
	where
	\[
	W^{\rm n}({\bf Y})=\frac{(SS_{\rm Res}^{\rm n}-SS_{\rm Res})/(p-q)}{SS_{\rm Res}/(n-p-1)}  
	\]
	and  $SS_{\rm Res}^{\rm n}$ is the residual sum of squares of the null model in (\ref{null:hyp:lm:n}). 
	Since the usual counting of degrees of freedom shows that $W^{\rm n}({\bf Y})\sim\textsf F_{p-q,n-p-1}$ under $H_0^{\rm n}$, we find that the null hypothesis in (\ref{null:hyp:lm:n}) gets rejected if 
	\[
	W^{\rm n}({\bf y})\in\left[\textsf f_{p-q,n-p-1,1-\alpha},+\infty\right),
	\]
	the obvious extension of (\ref{much:info}). Put in another way, if $SS_{\rm Res}^{\rm n}$ and $SS_{\rm Res}$ are close to each other, which intuitively means that the null model fits  as well as the full model,  then $W^{\rm n}({\bf y})$ is small and hence $H_0^{\rm n}$ should {\em not} be rejected. 
	In any case, the geometry backing not only this latter assertion but also the whole argument above is fully discernible from Figure \ref{figg2}, 
	where $SS_{\rm Res}^{\rm n}=\|\widehat{\bf e}^{\rm n}\|^2$, $SS_{\rm Res}^{\rm n}-SS_{\rm Res}=\|\widehat{\bf e}^{\rm n}-\widehat{\bf e}\|^2$, the squared norm of the dashed vector, and so on.
	\qed 
\end{example}

\begin{remark}\label{p:value}($p$-value)
	As already observed, in all examples above  the likelihood ratio statistics $h(X)$ is $\Theta_0$-{ancillary} in the sense that its distribution does {not} depend on $\theta\in\Theta_0$. In those cases,
	an equivalent way of reporting the result of a likelihood ratio test is to look at the corresponding $p$-{\em value}
	\[
	\mathfrak p=\mathcal P_{\theta}(h(X)\geq h(x)), 
	\quad \theta\in\Theta_0,
	\] 
	where $h(x)$ is the observed value of $h(X)$.
	Thus, $\mathfrak p$ is the probability of finding,  under $H_0$, an observed value at least as extreme as the one actually observed. With this terminology, $H_0$ gets rejected if $\mathfrak p\leq \alpha$, which is just a rephrasing of the rejection condition $h(x)\geq r$. Although this seems to be the preferred way of summarizing the outcome of a test in Applied Statistics, it is argued that the common misinterpretation of regarding $\mathfrak p$ as the probability that $H_0$ is true, thus erroneously accepting the validity of the alternative hypothesis (with high probability) if $\mathfrak p$ is found to be sufficiently small, may be a source of confusion leading to ``$P$-hacking'', ``the replication crisis'', etc.; see \cite{hubbard2003confusion,wasserstein2016asa,ferreira2015does,gibson2021role} for more on this quite controversial issue. \qed
\end{remark}	

In all examples above where we have been able to directly carry out the corresponding computation, the {\em asymptotic} likelihood radio statistics turned out to be $\chi^2_l$-distributed, where $l=\dim\Theta-\dim\Theta_0$. In fact, this is a general phenomenon which substantially simplifies the implementation of the test for large samples.

\begin{theorem}\label{th:wilks}\cite{wilks1938large}
	Under the conditions above, and requiring suitable regularity assumptions on the underlying statistical model as usual, there holds $h(X)\stackrel{d}{\to}\chi^2_l$ as $n\to +\infty$ and under $H_0$. 
\end{theorem} 

\begin{proof}
	We only sketch the argument, which  relies on the (multi-dimensional version) of the proof of Theorem \ref{asym:cr} on the asymptotic normality of ML estimators (and the simplifying assumption that we may choose rectangular coordinates $(\theta_1,\cdots,\theta_p)$ on $\Theta$ so that $\Theta_0$ is singled out by $\theta_{k+1}=\cdots=\theta_p=0$). Now, under $H_0$ the true parameter value, say $\theta$, lies in $\Theta_0$. Moreover, if $n$ is large enough then both 
	\[
	\widehat\theta={\rm argmax}_{\theta\in\Theta}L({\bf x};\theta) \quad {\rm and }\quad
	\widehat\theta_0={\rm argmax}_{\theta\in\Theta_0}L({\bf x};\theta)
	\] 
	are close to $\theta$ and hence close to each other, so we may expand about $\widehat\theta$,
	\begin{eqnarray*}
		h(X)
		& = & 
		-2\left(\ln L({X};\widehat\theta)-\ln L({ X};\widehat\theta_0)\right)\\
		& \approx & 
		\langle \widehat\theta-\widehat\theta_0,(-\nabla^2_{\theta\theta})(\ln L({ X};\widehat\theta)(\widehat\theta-\widehat\theta_0)\rangle,
	\end{eqnarray*} 
	where we have discarded terms of order at least three and used that $\nabla_\theta \ln L({\bf x};\widehat\theta)=0$ by the definition of $\widehat\theta$; compare with (\ref{mle:cons}). 
	From (\ref{conv:den}) we also know that 
	\[
	(-\nabla^2_{\theta\theta})(\ln L({ X};\widehat\theta))\stackrel{p}{\to}\mathscr F(\theta),
	\]
	where  $\mathscr F(\theta)$ is the Fisher information matrix of a single observation. Moreover,  
	\[
	W:=\sqrt{n}\,\nabla_\theta\ln L({X};\theta)\stackrel{d}{\to}\mathcal N(\vec{0},\mathscr F(\theta))
	\]
	by (\ref{conv:num}),
	so that (\ref{asym:stand}) may be rewritten as
	\[
	\sqrt{n}(\widehat\theta-\theta)\stackrel{d}{\to}\mathscr F(\theta)^{-1}W. 
	\]
	Also, by examining the maximization problem restricted to $\Theta_0$ which yields $\widehat\theta_0$, it is not hard to check that 
	\[
	\sqrt{n}(\widehat\theta_0-\theta)\stackrel{d}{\to}\mathscr G(\theta)W, 
	\]
	where $\mathscr G(\theta)$ is a certain symmetric matrix with rank $k=\dim\Theta_0$ and satisfying 
	\[
	\mathscr G(\theta)\mathscr F(\theta)\mathscr G(\theta)=\mathscr G(\theta).
	\]
	Thus, eliminating $\theta$ in the convergences above and using the expansion we get
	\[
	h(X)\approx \langle W,(\mathscr F(\theta)^{-1}-\mathscr G(\theta))W\rangle
	\]
	Now, if $\mathscr F(\theta)=B^2$ then using Corollary \ref{ortho:norm} we have that $Z=B^{-1}W\approx \mathcal N(\vec{0},{\rm Id}_p)$, a standard normal vector, so that
	\begin{eqnarray*}
		h(X)
		& \approx &
		\langle BZ,(\mathscr F(\theta)^{-1}-\mathscr G(\theta))BZ\rangle\\
		& = & 
		\langle Z,B(\mathscr F(\theta)^{-1}-\mathscr G(\theta))BZ\rangle\\
		& = & 	\langle Z,({\rm Id}_p-A\mathscr G(\theta)B)Z\rangle,	
	\end{eqnarray*}
	where $B\mathscr G(\theta)B$ is easily seen to be idempotent with the same rank as $\mathscr G(\theta)$. Hence,  ${\rm Id}_p-B\mathscr G(\theta)B$ is idempotent as well with rank  $l=p-k=\dim\Theta-\dim\Theta_0$ and the result follows from Proposition \ref{u:v:gen}.
\end{proof}

We refer to \cite[Subsection 4.5.1 ]{amemiya1985advanced} and \cite[Theorem 6.5]{shao2008mathematical} for those interested in filling out the omitted details in the argument above. Also,   
it is worthwhile mentioning that Wilks' original proof in  \cite{wilks1938large} is equally elegant as it involves checking that  $\phi_{h(X)}$, the characteristic function of $h(X)$, asymptotically approaches $\phi_{\chi^2_l}$. 
In any case, we stress that it is not required in Theorem \ref{th:wilks} that $h(X)$ is $\Theta_0$-ancillary so in a sense the result guarantees that this property gets restored in the asymptotic regime. 
Finally, we note 
that the material above by no means exhausts the rich literature on hypothesis testing and extensive treatments may be 
found elsewhere 
\cite{amemiya1985advanced,dudewicz1988modern,	welsh1996aspects,lehmann2005testing,casella2008statistical,shao2008mathematical,hayashi2011econometrics,degroot2014probability}.

We now discuss a non-conventional hypothesis testing and its connection with a major result in Differential Geometry, namely, Weyl's formula for the volume of tubes \cite{weyl1939volume, gray2003tubes}.

\begin{example}\!\!$\bigstar$\label{hotelling}
	(Testing for an additional term in the linear model and Weyl's formula for the volume of tubes)
	Let us consider a (possibly non-linear) perturbation of the linear model with a normal error from Example \ref{mle:imp:lsm:n},
	\[
	{\bf Y}_j={\bf \mathfrak X}_j\beta+cf_j({\bf X}_j,\gamma)+{\bf e}, 
	\] 
	a setting first considered in a seminal paper by H. Hotelling \cite{hotelling1939tubes}. Here, $f$ is known  but $c\in\mathbb R$ and $\gamma\in\mathbb R^k$ are unknown parameters and our aim is to test
	\[
	H_0: c=0\quad {\rm vs}\quad H_{\rm a}: c\neq 0,
	\] 
	so that not being able to reject $H_0$ indicates statistical evidence 
	for ignoring $f$ in the model design.
	In the geometric language of Remark \ref{geom:mls}, the null hypothesis says that $\mathbb E({\bf Y}|_{{\mathfrak X}={\mathfrak x}})\in C({\mathfrak x})\equiv \mathbb R^{p+1}$, 
	whereas the alternative adds a multiple of $f_\gamma=f(\cdot,\gamma)$ to this vector. 
	Without loss of generality, we may assume 
	that $f_\gamma\in C({\mathfrak x})^{\perp}\equiv \mathbb R^{n-p-1}$
	and,
	in order to make the model identifiable, that  $f_\gamma$ is not a multiple of $f_{\gamma'}$ if $\gamma\neq\gamma'$, but notice that the model still gets 
	{\em non}-identifiable under $H_0$, so the usual dimensional counting that would allow us to determine the asymptotic distribution of the likelihood ratio statistics $h$ via Theorem \ref{th:wilks} does not apply. In particular, there is no point here in working with $h$ so we turn our attention to the corresponding likelihood ratio $\Lambda=e^{-h/2}$, thus rejecting $H_0$ if this ratio, when observed, is conveniently small.  Now, the full likelihood function is 
	\[
	L({\bf y};\beta,\theta_2,c,\gamma)=(2\pi \theta_2)^{-n/2}e^{-\frac{\|{\bf y}-{\mathfrak x}\beta-cf_\gamma\|^2}{2\theta_2}}, \quad \theta_2=\sigma^2,
	\] 
	which under $H_0$ reduces to the usual likelihood function of the linear model treated in Example \ref{ex:f:reg}:
	\[
	L({\bf y};\beta,\theta_2)=(2\pi \theta_2)^{-n/2}e^{-\frac{\|{\bf y}-{\mathfrak x}\beta\|^2}{2\theta_2}}.
	\] 
	Hence, 
	\[
	\sup_{\beta,\theta_2}	L({\bf y};\beta,\theta_2)=(2\pi\widehat\theta_{20})^{-n/2}e^{-n/2},
	\]
	where 
	\[
	\widehat\theta_{20}=\frac{1}{n}|{\bf Y}-\mathfrak x\widehat\beta|^2=\frac{1}{n}|\widehat{\bf e}|^2
	\]
	and $\widehat{\bf e}$ is the residual. 
	It follows that the likelihood ratio is
	\[
	\Lambda=\left(\frac{\widehat\theta_{20}}{\widehat\theta_2}\right)^{-n/2},
	\] 
	where 
	\[
	\widehat\theta_2=\frac{1}{n}\|{\bf Y}-\mathfrak r\widehat\beta-\widehat cf_{\widehat\gamma}\|^2
	\]
	and 
	\[
	(\widehat\beta,\widehat c,\widehat\gamma)=
	{\rm argmax}_{\beta,c,\gamma}L({\bf y};\beta,c,\gamma)
	=
	{\rm argmin}_{\beta,c,\gamma}
	\|{\bf y}-\mathfrak x\beta-cf_\gamma\|^2.
	\]
	Now, using that $f_\gamma^\top\mathfrak r\beta=0$ (in particular, $f_\gamma^t\mathfrak r\widehat\beta=0$) we compute
	\begin{eqnarray*}
		(\widehat\beta,\widehat c,\widehat\gamma)
		& = & 
		{\rm argmin}_{\beta,c,\gamma}\|{\bf y}-\mathfrak x\beta\|^2-2cf^\top_\gamma{\bf y}+c^2\|f_\gamma\|^2\\
		& = & 
		{\rm argmin}_{c,\gamma}\|{\bf y}-\mathfrak x\widehat\beta\|^2-2cf_\gamma^\top{\bf y}+c^2\|f_\gamma\|^2\\
		& = & 
		{\rm argmin}_{c,\gamma}\|{\bf y}-\mathfrak x\widehat\beta\|^2-2cf_\gamma^\top({\bf y}-\mathfrak x\widehat\beta)+c^2\|f_\gamma\|^2\\
		& = & 
		{\rm argmin}_{c,\gamma}\|{\widehat {\bf e}}-cf_\gamma\|^2,
	\end{eqnarray*}
	and since 
	\[
	\widehat c={\rm argmin}_c\|\widehat{\bf e}-cf_\gamma\|^2=\frac{f_\gamma^\top\widehat{\bf e}}{\|f_\gamma\|^2}, 
	\]
	we see that
	\begin{eqnarray*}
		\Lambda^{2/n}
		& = & 
		\inf_\gamma \frac{\|{\widehat {\bf e}}-\widehat cf_\gamma\|^2}{\|\widehat{\bf e}\|^2}\\
		& = & \inf_\gamma\left(1-\left(
		\frac{f_\gamma^\top{\widehat{\bf e}}}{\|f_\gamma\|\|\widehat{\bf e}\|}
		\right)^2\right)\\
		& = & 1-\sup_\gamma\,(\widetilde f^\top_\gamma {\bf U})^2.
	\end{eqnarray*}
	where, as $\gamma$ varies, $\widetilde f_\gamma=f_\gamma/\|f_\gamma\|$ traces a subset $M_\gamma\subset \mathbb S^{n-p-2}$, the unit sphere in $C({\mathfrak x})^\perp$, and ${\bf U}=\widehat{\bf e}/\|\widehat{\bf e}\|$ is a random vector also taking values in $\mathbb S^{n-p-2}$. Since 
	\[
	\widetilde f^\top_\gamma {\bf u}=
	\cos {\rm dist}(f_\gamma,{\bf u}),
	\]	
	where ${\rm dist}$ is the intrinsic distance in $\mathbb S^{n-p-2}$, 
	we may choose as rejection region the ``tubular neighborhood''
	\[
	B_{\rho}(M_\gamma)=\{\vartheta\in\mathbb S^{n-p-2}; {\rm dist}(\vartheta,M_\gamma)\leq \rho\}
	\]
	of radius $\rho>0$ around $M_\gamma$. Now, under $H_0$ we know from Remark \ref{geom:mls} that ${\bf U}$ is uniformly distributed in $\mathbb S^{n-p-2}$ and we conclude that
	the significance level $\alpha$ of the test satisfies 
	\[
	\alpha= P({\bf U}\in B_{\rho}(M_\gamma))={\rm vol}_{P_{\bf U}}(B_{\rho}(M_\gamma)),
	\] 
	where ${\rm vol}_{P_{\bf U}}$ is the (normalized) intrinsic volume (so that ${\rm vol}_{P_{\bf U}}(\mathbb S^{n-p-2})=1$). Thus, in order to determine the rejection ``tube'' associated to a given confidence level, an explicit formula for the volume of the tube is required, at least for $\rho$ small enough. This turns out to be a rather formidable geometric problem which has been completely solved by H. Weyl \cite{weyl1939volume} in case $M$ is a closed submanifold of a space form (a Riemannian manifold with constant sectional curvature)\footnote{In case $M\subset \mathbb R^l$ has dimension $m$, Weyl's formula says that
		\[
		{\rm vol}(B_{\rho}(M))=c_{m,l}\rho^{l-m}\sum_{q=0}^{[m/2]}\left(d_{m,l,q}\int_M\kappa_{2q}dM\right)\rho^{2q},
		\]
		where $c_{m,l}$ and $d_{m,l,q}$ are positive constants (only depending on the indicated natural parameters) and $\kappa_{2q}$ is a certain (universal) polynomial expression which is homogeneous of degree $q$ in the curvature tensor of $M$. In particular, ${\rm vol}(B_{\rho}(M))$ depends only on $\rho$ and the {\em intrinsic} geometry of $M$ and not on the specific way it is embedded in $\mathbb R^l$. 
		For a masterly account of this remarkable result and its many applications (including a proof, in full generality, of the Chern-Gauss-Bonnet formula in Riemannian Geometry) we refer to \cite{gray2003tubes}.
		Also, for a brief overview of the ubiquitous role the Gauss-Bonnet curvatures $\kappa_{2q}$ play in Riemannian Geometry and related areas, see \cite{labbi2007gauss} and the references therein. 
	}.  
	Unfortunately, this Weyl's formula does not directly apply to this problem (as $M_\gamma$ may be only piecewise smooth or carry a boundary, etc.) so adjustments, mainly based on suitable approximations, are required \cite{naiman1990volumes}. \qed   
\end{example}

\begin{example}\!\!$\bigstar$\label{scheffe:band}(Scheff\'e-type simultaneous band for the mean response in a normal linear model and the volume of tubes, again)
	One is often interested in obtaining 
	simultaneous confidence bands for the mean response ${\bm{\textsf x}}^\top\beta$ in a normal regression model (as in Example \ref{mle:imp:lsm:n}) with ${\bm{\textsf x}}$ varying in some subset set $\mathscr S\subset\bf{\textsf R}^p$,
	say $\mathscr S$ diffeomorphic to an interval or a rectangle and so on; here we retain the notation of Examples \ref{conf:ref:beta} and \ref{simult:band}. 
	Although the Scheff\'e-type band 
	in (\ref{scheffe}) may be applied to this end, it certainly provides a wider 
	band than required for a given confidence level, so a sensible  
	strategy here is to seek for $c>0$ satisfying
	\begin{equation}\label{scheffe:band:c}
		{\bm{\textsf x}}^\top\beta\in\left[
		{\bm{\textsf x}}^\top\widehat\beta\mp
		c\,\widehat\sigma\sqrt{{\bf x}^\top\mathfrak s{\bm{\textsf x}}}
		\right] \forall\, {\bm{\textsf x}}\in\mathscr S
		\,\,{\rm with}\,\,{\rm prob.}\,\,1-\delta.
	\end{equation}
	Now, in terms of 
	\begin{equation}\label{def:mathfrak:m}
		\mathfrak m=\sigma\mathfrak n=(\mathfrak p^\top)^{-1}(\widehat\beta-\beta)\sim\mathcal N(\vec{0},\sigma^2{\rm Id}_{p+1}), 
	\end{equation}
	$\varepsilon({\bm{\textsf x}})=\mathfrak p{\bm{\textsf x}}/|\mathfrak p{\bm{\textsf x}}|\in\mathbb S^{p}\subset\mathbb R^{p+1}$, ${\bm{\textsf x}}\in\mathscr S$,  
	and $\mathfrak u=\mathfrak m/\|\mathfrak m\|$, a uniformly distributed random vector in $\mathbb S^{p}$, we have  
	\[
	\frac{{\bm{\textsf x}}^\top(\widehat\beta-\beta)}{\sqrt{{\bm{\textsf x}}^t\mathfrak s{\bm{\textsf x}}}}
	= 
	\frac{(\mathfrak p^\top)^{-1}{\bm{\textsf x}}\mathfrak m}{\sqrt{(\mathfrak p{\bm{\textsf x}})^\top\mathfrak p{\bm{\textsf x}}}}
	= 
	\left(
	\varepsilon({\bm{\textsf x}})^\top\mathfrak u\right)\|\mathfrak m\|,
	\]
	so if $T=\widehat\sigma/\|\mathfrak m\|$ then (\ref{scheffe:band:c}) expresses the corresponding tail probability as 
	\begin{eqnarray*}
		\delta 
		& = &
		P\left(	
		\sup_{{\bm{\textsf x}}\in\mathscr S}
		\left|\frac{{\bm{\textsf x}}^\top(\widehat\beta-\beta)}{\sqrt{{\bm{\textsf x}}^\top\mathfrak s{\bm{\textsf x}}}}\right|\geq c\widehat\sigma
		\right)\\
		& = & 
		P\left(
		\sup_{{\bm{\textsf x}}\in\mathscr S}
		\left|
		\varepsilon({\bm{\textsf x}})^\top\mathfrak u
		\right|
		\geq cT
		\right).
	\end{eqnarray*}
	Hence, using a notation similar to the one of the previous example, we see that the coverage probability in (\ref{scheffe:band:c}) is 
	\[
	1-\delta={\rm vol}_{P_{\bf u}}\left(
	B_{cT}(M_\varepsilon\cup M_{-\varepsilon})
	\right),
	\] 
	so that being able to compute the (normalized) volume of certain tubes around
	$M_{\varepsilon}\cup M_{-\varepsilon}\subset\mathbb S^{p}$ intervenes in determining the critical value $c$. Since
	$T$ is independent of ${\bf u}$, this may be rewritten as 
	\[
	1-\delta=\int_0^{1/c}{\rm vol}_{P_{\bf u}}\left(
	B_{ct}(M_\varepsilon\cup M_{-\varepsilon}) 
	\right)\psi_T(t)dt,
	\] 
	where $\psi_T$ is the pdf of $T$. 
	It follows from (\ref{pivot:eq:3}), (\ref{def:mathfrak:m}) and the independence of ${\bf u}$ and $\widehat\sigma$ that 
	$(p+1)T^2\sim \textsf F_{n-p-1,p+1}$, so (\ref{regra:01}) applies to give
	\[
	\psi_T(t)=2(p+1)t\psi_{\textsf F_{n-p-1,p+1}}((p+1)t^2), \quad t\geq 0,
	\]
	and the substitution $t=\cos\theta/c$ in the previous integral leads to
	\[
	1-\delta=\int_0^{\pi/2}{\rm vol}_{P_{\bf u}}\left(
	B_{\cos\theta}(M_\varepsilon\cup M_{-\varepsilon}) 
	\right)\psi(\theta)
	d\theta,
	\]
	where 
	\[
	\psi(\theta)=	\frac{2(p+1)\sin \theta\cos\theta}{c^2}\psi_{\textsf F_{n-p-1,p+1}}\left(\frac{(p+1)\cos^2\theta}{c^2}\right).
	\]
	To see how this implies Scheff\'e's original contribution,
	note that if $\mathscr S=\mathcal R^p$ then the volume function within this integral clearly equals $1$ identically, so the substitution 
	$\tau=(p+1)\cos^2\theta/c^2$ gives 
	\[
	1-\delta=\int_0^{(p+1)/c^2}\psi_{\textsf F_{n-p-1,p+1}}(\tau)d\tau,
	\]
	and using Corollary \ref{f-dis:cons} with $\tau'=1/\tau$, 
	\[
	1-\delta=
	\int_0^{c^2/(p+1)}\psi_{\textsf F_{p+1,n-p-1}}(\tau')d\tau'.
	\]
	But this means that
	\[
	\frac{c^2}{p+1}=\textsf f_{p+1,n-p-1,1-\delta},
	\] 
	which recovers (\ref{scheffe}) as promised. In general, when $\mathscr S$ is a proper subset of $\mathcal R^p$, ``approximate'' versions of Weyl's formula are needed in order to establish simultaneous confidence bands based on the computations above \cite{sun1994simultaneous,liu2010simultaneous}. \qed
\end{example}

\section{From Fisher’s parametric paradigm to modern Statistical Learning: a brief overview}\label{br:over}

In light of the preceding developments, we revisit the main mathematical underpinnings of the classical approach to parametric estimation, as formulated in Fisher’s original program in his landmark paper \cite{fisher1922mathematical}, and reinterpret them from the standpoint of modern Statistical Learning Theory and Data Science.

\subsection{The classical legacy}\label{class:leg}
Consider, for instance, a uni-dimensional statistical model with random sample  
\begin{equation}\label{model:spec}
	X_1,\dots,X_n \sim \psi_\theta, \quad \theta \in \mathbb{R},
\end{equation}  
of size $n$. From this sample we form a statistic $h(X_1,\dots,X_n)$ with the aim of producing an estimator $\widehat\theta$ for $\theta$, as in (\ref{estim:def}). Importantly, $h$ should not depend on the unknown parameter $\theta$.  
A first step in evaluating the efficiency of $\widehat\theta$ is to compute (or at least reliably approximate) its mean squared error ${\rm mse}(\widehat\theta)$. As illustrated above in the case $\widehat\theta=\widehat\sigma_c^2$, this requires computing the associated variance, a generally demanding task that becomes substantially simpler if $\psi_\theta$ is assumed to be normal. Although restrictive, this assumption is sometimes heuristically justified by appealing to the Central Limit Theorem (Theorem \ref{clt}) together with the ``hypothesis of elementary errors'' \cite[Chapter 3]{fischer2011history}. If ${\rm mse}(\widehat\theta)$ is sufficiently small, guaranteeing good performance of $\widehat\theta$, one may still need further information about the sampling distribution of the estimator, depending on the inferential goal.  

For instance, as discussed in Subsection \ref{conf:int:sub} for the sample mean, if the aim is to provide ``small-sample'' confidence intervals for $\theta$, then for any $n$ one should determine the explicit distribution $\psi_\theta^{(n)}$ of $h(X_1,\dots,X_n)$ (or of a suitable pivotal quantity). This is often delicate, even under normality, especially when $h$ depends nonlinearly on the sample (as in the case of the sample variance). A landmark contribution here is Student’s determination of the distribution of his pivotal quantity $T_{n-1}$ in (\ref{stu:est}), which yields small-sample confidence intervals for the population mean $\mu$. This, however, was achieved only under the assumption that the original sample is normal. Student’s work profoundly influenced R. Fisher, who not only placed it on solid mathematical foundations through his ``geometric method'' (Remark \ref{fis:comp:new}), but also extended it to obtain the distribution of the correlation coefficient (Example \ref{corr:dist})\footnote{Because of its intuitive geometric appeal, Fisher’s approach is rarely reproduced in modern textbooks, where it is typically replaced by analytical methods involving Jacobian manipulations of the underlying coordinate transformations.}.  

A less formidable task is to seek asymptotic information. A preliminary step is verifying that ${\rm mse}(\widehat\theta_n)\to 0$ as $n\to\infty$, which ensures that $\widehat\theta_n$ is consistent (see Proposition \ref{mse:consist}, itself a version of the LLN in this broader setting). To complement this asymptotic point estimate with a dispersion analysis, one typically looks for a distribution $\psi$ such that a suitable standardization of ${\widehat\theta}_n$ converges in law to $\psi$. Ideally, this yields approximations for the distribution of $\sqrt{n}(\widehat\theta_n-\theta)$ that concentrate sharply around $\theta$, allowing the construction of reliable confidence intervals from the tail probabilities of $\psi$ in the large-sample regime. This approach, already illustrated in Subsection \ref{conf:int:sub} for $\widehat\theta_n=\overline X_n$, is formalized in Theorem \ref{asym:cr}, which establishes the asymptotic normality of a broad class of consistent ML estimators. These estimators are thereby shown to be asymptotically efficient: as $n\to\infty$, their dispersion (measured by standard deviation) achieves the Cramér--Rao lower bound (Theorem \ref{cr:rao:th}); see also Remark \ref{conf:int:f}\footnote{The same dichotomy between small and large sample regimes reappears in Section \ref{sec:hyp:test}, where hypothesis testing is treated.}. Full expositions of these large-sample methods are available in \cite{newey1994large,lehmann1999elements,lehmann2006theory,ferguson2017course,van2000asymptotic}.  

It is remarkable that this modern strategy essentially mirrors the program laid out by R. Fisher in his foundational paper \cite{fisher1922mathematical}, written a century ago. There Fisher declared that ``the object of statistical methods is the reduction of data,'' and identified as the first major challenge the ``Problems of Specification,'' namely, the choice of an appropriate statistical model (as in (\ref{model:spec}))\footnote{According to Fisher, ``these are entirely a matter for the practical statistician.''}. He then turned to the ``Problems of Estimation,'' concerned with selecting a statistic designed to estimate the parameters of the population. To judge the quality of an estimator, he proposed three criteria: the ``Criterion of Consistency'' (the estimator approaches the true parameter in the long run\footnote{Fisher’s notion of consistency differs from the modern one presented here.}, the ``Criterion of Efficiency'' (for large samples, the estimator with smallest dispersion, what we would now phrase as variance achieving the Cramér--Rao lower bound, where Fisher information plays a key role; cf. Remark \ref{decay:fluc}), and the ``Criterion of Sufficiency'' (the statistic should summarize all the relevant information in the sample). Finally, in what he called the ``Problems of Distribution,'' Fisher emphasized the importance of computing the exact sampling distributions of estimators, to fully elucidate their theoretical properties.  
As Fisher himself noted, estimators that satisfy both consistency and efficiency criteria may still differ in finite-sample performance, which justifies sufficiency as a decisive tiebreaker. While the asymptotic aspects of estimation can often be handled analytically, the Problems of Distribution involve small-sample results of great mathematical difficulty\footnote{This helps to explain Fisher’s deep admiration for Student’s work; see Remark \ref{student}. The same difficulty also reappears in hypothesis testing, as discussed in Remark \ref{size:test:3}.}. In this same paper Fisher also introduced the method of maximum likelihood, offering for the first time a systematic procedure for constructing estimators within a given model. This represented a decisive step toward establishing the conceptual framework on which the frequentist approach to Statistical Estimation still rests\footnote{As noted by the leading historian of statistics S. Stigler in \cite{stigler2005fisher}: ``The paper is an astonishing work: It announces and sketches out a new science of statistics, with new definitions, a new conceptual framework and enough hard mathematical analysis to confirm the potential and richness of this new structure.''}.  

\subsection{The modern outlook}\!\!$\bigstar$\label{mod:out}
It is tempting to ask how far this paradigm extends beyond the classical models for which it was originally conceived, given that Fisher’s 
framework  (likelihood as the central inferential object, curvature encoded by the Fisher information, and model comparison through likelihood-based criteria such as AIC) rests fundamentally on the structure of regular finite-dimensional parametric models. In this setting, maximum likelihood estimators are asymptotically normal, the Fisher information is nonsingular, and the effective complexity of a model is well approximated by its parameter dimension. Regarding this latter feature, we recall from Remark~\ref{kl:fisher:aic} that it is reflected in the penalty term of information criteria such as AIC, which arises naturally from a second-order expansion of the Kullback--Leibler risk and provides a principled bias correction for predictive performance.
In the contemporary landscape of highly overparameterized models, particularly deep neural networks, many of the regularity assumptions underlying this framework cease to hold: parameterizations are often non-identifiable, Fisher information may be singular or degenerate, and the number of parameters can vastly exceed the sample size ($p \gg n$). 
In such settings, the effective complexity of the model is governed more accurately by global complexity measures arising in Statistical Learning Theory, which in this modern perspective play the role that the parameter dimension $p$ once played in classical information criteria; see \cite{vapnik1998statistical,bousquet2003introduction,anthony2009neural,von2011statistical,bartlett2017spectrally,golowich2020size,bartlett2021deep} and Example \ref{super:learn}\footnote{
In many such ``singular'' models, interpretability is no longer intrinsic to the parameterization, and the resulting systems are often described as “black boxes” \cite{rudin2019stop}. This has led to the development of post hoc interpretability techniques and explainable artificial intelligence (XAI) methods \cite{guidotti2018survey,molnar2022interpretable}, which attempt to reintroduce explanatory structure at a layer external to the statistical model itself.}. 
Although this shift reflects a broader methodological transition, in which emphasis moves from parameter interpretation and inferential optimality toward predictive accuracy and generalization performance\footnote{See Subsection \ref{ridge}, where this trade-off already appears in the regression context.},
it  may be viewed as a natural extension, rather than a departure, from the program initiated by Fisher, and is perhaps best understood if linked to the distinction drawn between the so-called \emph{data modeling} and \emph{algorithmic modeling} cultures in an influential essay by L. Breiman \cite{breiman2001statistical}. 

To make this connection more precise, let us restrict attention to the Statistical Learning setting of Example \ref{super:learn}.
In the classical parametric framework, one assumes that the joint distribution of $(X,Y)$ belongs to a finite-dimensional family, and statistical analysis proceeds by estimating the associated parameter and studying its properties through likelihood-based methods. Within this paradigm, we have seen that Fisher’s notions of efficiency and sufficiency, together with the geometric role of the Fisher information, provide a coherent setup in which the effective complexity of the model is well approximated by the parameter dimension. 
By contrast, in modern high-dimensional or nonparametric settings where the model may be misspecified, non-identifiable, or too intricate to admit a meaningful low-dimensional parametrization, it is often preferable to make no assumption whatsoever on the underlying distribution $P_{(X,Y)}$ from which the data were drawn, so that the classical correspondence between parameter dimension and complexity breaks down. In such cases, it is more natural to consider a class of predictors $\mathcal F$, from which an estimator is selected by minimizing the empirical risk within $\mathcal F$.
In this formulation, the primary objective shifts from the estimation of a ``true'' parameter to the construction of predictors with good out-of-sample performance, thereby emphasizing generalization rather than inference. In particular, the corresponding population risk may be interpreted, for suitable choices of the loss function, as a Kullback--Leibler-type divergence between the underlying distribution and the predictive model, so 
the central problem boils down to controlling the discrepancy between empirical and population risks, now {\em uniformly} over $\mathcal F$.
A key insight of modern Statistical Learning is that such uniform control is possible even without parametric assumptions, provided the class $\mathcal F$ is suitably restricted. 
More precisely, one may reduce the analysis, typically via symmetrization arguments, to tractable empirical quantities depending only on the sample and on the richness of $\mathcal F$, leading to bounds expressed in terms of combinatorial or metric notions of capacity, such as the VC dimension or related complexity measures; see, for instance, the VC uniform deviation bound discussed in Example~\ref{super:learn}

From this perspective, the transition from data to algorithmic modeling, as advocated by Breiman, may be understood as a natural generalization of Fisher’s original program to settings in which no finite-dimensional parametrization is adequate. In particular, the role played by the parameter dimension in classical statistics is taken over by more flexible notions of capacity, which quantify the effective size of the model and govern the trade-off between approximation and estimation errors, as in Example \ref{bias:var:sl}. In this way, modern empirical process techniques extend Fisher’s framework beyond its original domain of regular parametric models, preserving its central concern with predictive performance, now understood in terms of out-of-sample risk, while adapting its mathematical foundations to the high-dimensional regime.

\section{A glimpse at the Bayesian pathway}\label{bay:way}

For comparison with the frequentist approach developed above, we now turn to the {\em Bayesian approach} to estimation. In this framework, the parameter $\theta$ governing the i.i.d.\ measurements $X_j\sim \psi(\cdot;\theta)$ is itself modeled as a random variable $\vartheta$ taking values in $\Theta$, endowed with a prior distribution $\psi_\vartheta d\theta$. In this way, probability is assigned not only to the observations but also to the parameter, and inference proceeds by updating the distribution of $\vartheta$ in light of the observed data.\footnote{For a conceptual motivation for this move, see de Finetti’s representation theorem discussed in Example \ref{de:finetti}.}. It follows from Theorem \ref{bayes} (Bayes rule) that
\[
\psi_{\vartheta|_{X={\bf x}}}(\theta)=\frac{\psi_{X|_{\vartheta=\theta}}({\bf x})\psi_\vartheta(\theta)}{\psi_X({\bf x})}, \quad \psi_X({\bf x})=\int_{\Theta} \psi_{X|_{\vartheta=\theta}}({\bf x})\psi_\vartheta(\theta)d\theta,
\]  
where $X=(X_1,\cdots,X_n)$ and ${\bf x}=(x_1,\cdots,x_n)$ is a realization of $X$.
In the Bayesian jargon, $\psi_\vartheta(\theta)$ is the {\em prior}, 
reflecting our knowledge of the underlying parameter $\theta$ previous to any measurement (and hence viewed as a {\em hypothesis}) and $\psi_{X|_{\vartheta=\theta}}({\bf x})$ is the {\em likelihood}, which indicates the compatibility of the {\em evidence} $X$ with the given hypothesis. Note that 
\[
\psi_{X|_{\vartheta=\theta}}({\bf x})=L({\bf x};\theta),
\]
the likelihood function in (\ref{like:exp:ind}), hence the terminology; here we are momentarily coming back to the ``frequentist'' setting of Section \ref{br:over} and thus regarding $\theta$ as deterministic (i.e. \negthinspace non-random). The prior and the likelihood combine to yield the {\em posterior} $\psi_{\vartheta|_{X={\bf x}}}(\theta)$ through the proportionality
\begin{equation}\label{prop:bayes}
	\psi_{\vartheta|_{X={\bf x}}}(\theta)\propto L({\bf x};\theta)\psi_\vartheta(\theta),
\end{equation}
which provides an update of the probability distribution of the hypothesis as more observed evidence becomes available. 

\begin{example}\label{pri:post:n}
	For a normal sample $X_j\sim\mathcal N(\mu,\sigma^2)$, with $\sigma$ known, we find that the likelihood is
	\[
	L({\bf x};\mu)=\frac{1}{(2\pi)^{n/2}\sigma^n}e^{-\frac{1}{2\sigma^2}\sum_j(x_j-\mu)^2}. 
	\]
	Now assume that the prior, which expresses our initial degree of belief on the unknown parameter $\mu$, follows the normal $\mathcal N(\mu_{\rm pr},\sigma_{\rm pr}^2)$, so that 
	\[
	\psi_\vartheta(\mu)=\frac{1}{\sqrt{2\pi}\sigma_{\rm pr}}e^{-\frac{(\mu-\mu_{\rm pr})^2}{2\sigma_{\rm pr}^2}}. 
	\]
	A direct computation using (\ref{prop:bayes}) confirms that the posterior also follows a normal, namely,  
	\[
	\psi_{\vartheta|_{X={\bf x}}}\sim \mathcal N(\mu_{\rm pos},\sigma^2_{\rm pos}),
	\]
	where 
	\begin{equation}\label{post:m:norm}
		\mu_{\rm pos}=(1-\lambda)\mu_{\rm pr}+\lambda\frac{\sum_jx_j}{n}, \quad \lambda =\frac{\sigma_{\rm pr}^2}{\sigma_{\rm pr}^2+\sigma^2/n},
	\end{equation}
	and 
	\[
	\sigma_{\rm pos}^2=\frac{\sigma_{\rm pr}^2\sigma^2/n}{\sigma_{\rm pr}^2+\sigma^2/n}. 
	\]
	Hence, the Bayesian recipe confines the posterior mean $\mu_{\rm pos}$ somewhere between  the prior mean $\mu_{\rm pr}$ and the realization $\sum_jx_j/n$ of the sample mean, with a higher degree of belief than before (since $\sigma_{\rm pos}<\min \{\sigma_{\rm pr},\sigma\}$). We thus see that the data-gathering provided by the sample has the net effect of fine-tuning our initial subjective knowledge regarding $\mu$.  \qed
\end{example}

\begin{example}\label{laplace}
	(Laplace's rule of succession)
	What chances are that the sun will rise tomorrow given that it has been so for the last $n$ days? To ponder on this, consider a Bernoulli sample $X_j\sim\mathsf{Ber}(p)$ assigning  probability $p$ to a  successful outcome corresponding to the event $\{1\}$. The question above is a special case (with $s=n$) of the general problem of  computing
	\[
	P(X_{n+1}=1|_{X^{(n)}=s})=P_{X_{n+1}|_{X^{(n)}=s}}(\{1\}), \quad X^{(n)}=X_1+\cdots+X_n,
	\]  
	the probability that success occurs at the $(n+1)^{\rm th}$ outcome  given that it has occurred  $s$ times previously; here we use the notation of (\ref{cond:prob:22}).  From Example \ref{mle:discrete} we know that the likelihood is 
	\[
	L({\bf x};p)=p^{s}(1-p)^{n-s},
	\]
	where  $s=x_1+\cdots+x_n$ is the realization of $X^{(n)}$.
	The simplest choice for the prior distribution of $\mathfrak p$, the random variable associated to the Bayesian parameter $p$, appeals to the ``Principle of Insufficient Reason'': we declare that 
	\[
	\psi_{\mathfrak p}(p)={\bf 1}_{[0,1]}(p),
	\]  
	the {\em uniform distribution} supported on the unit interval $[0,1]$. Using (\ref{prop:bayes}) we see  that the posterior is 
	\begin{eqnarray*}
		\psi_{\mathfrak p|_{X^{(n)}=s}}(p)
		& = & \frac{p^{s}(1-p)^{n-s}{\bf 1}_{[0,1]}(p)}{\int_0^1 p^{s}(1-p)^{s} dp}\\
		& = & \frac{(n+1)!}{s!(n-s)!}p^{s}(1-p)^{n-s}{\bf 1}_{[0,1]}(p), 
	\end{eqnarray*}
	the {Beta distribution} ${\mathsf{Beta}}(s+1,n-s+1)$; cf. Definition \ref{beta:def}. 
	It follows that 
	\begin{eqnarray*}
		P_{X_{n+1}|_{X^{(n)}=s}}(\{1\})
		& = & \mathbb E(\mathfrak p|_{X^{(n)}=s})\\
		& = & 	\int_0^1 p\psi_{\mathfrak p|_{X^{(n)}=s}}(p)dp,
	\end{eqnarray*} 
	and using the previous expression for the posterior we get
	\[
	P(X_{n+1}=1|_{X^{(n)}=s})=\frac{s+1}{n+2}, 
	\]
	which in particular gives $(n+1)/(n+2)$ as the solution for Laplace's sunrise problem\footnote{In the end of \cite[Chapter III]{laplace1998pierre}, Laplace takes $n$ corresponding to five thousand years and finds that ``it is a bet of 1820214 to one that it will rise again tomorrow''. But as Laplace himself recognizes in the sequel, this should be taken with a salt of grain, specially in regard to the choice of prior, as possibilities other than the uniform are certainly available; see Example \ref{bay:est:ex:beta}.}. \qed
\end{example}

\begin{example}\label{norm:like}
	If $\Theta\subset\mathbb R^q$ has a finite volume then the ``Principle of Insufficient Reason'' leads to
	\[
	\psi_\vartheta(\theta)=\frac{1}{{\rm vol}(\Theta)}{\bf 1}_{\Theta}(\theta), \quad \theta\in\Theta,
	\]
	as the choice for the prior, so the corresponding posterior is  
	\[
	\psi_{\vartheta|_{X={\bf x}}}(\theta)=\frac{L({\bf x};\theta){\bf 1}_{\Theta}(\theta)}{\int_{\Theta}L({\bf x};\theta)d\theta},
	\]
	a suitable normalization of the likelihood. This confirms that, viewed as a function of $\theta$, the likelihood in general does not  qualify as a pdf, which is consistent with the fact that the prescription for the MLE estimator in Definition \ref{mle:def:post} is insensitive to replacing $L$ by $cL$, $c>0$ a constant. In other words, any multiple of the likelihood function carries the same information as far as selecting the ML estimator is concerned. \qed
\end{example}

The examples above illustrate the Bayesian credo according to which 
probability is nothing but a measure of our degree of belief on the underlying parameter $\theta$, which thus should be random in nature.
In any case, we may proceed with the corresponding estimation theory as follows. Given  $\widetilde\theta$ define its {\em Bayes risk} as 
\[
\mathscr R(\widetilde\theta)=\mathbb E_{\psi_\vartheta}(\mathscr L(\widetilde\theta,\theta)),
\] 
where $\mathscr L$ is a (previously chosen) {\em loss function} (for instance, the quadratic loss $\mathscr L(\widetilde\theta,\theta)=|\widetilde\theta-\theta|^2$  gives rise to the Bayesian analogue of the mse in (\ref{mse:def}), but be aware of a crucial difference: here we average against the prior $\psi_\vartheta(\theta)d\theta$ in alignment with the Bayesian philosophy according to which $\theta$ is random, whereas there we integrate against $dP_\theta$ since $\theta$ is regarded as deterministic (i.e. \negthinspace non-random); see Remark \ref{underlying}.

\begin{definition}\label{bayes:estim}
	A {\em Bayes estimator} is any $\widehat\theta$ that minimizes the Bayes risk.
\end{definition}

Notice that this only depends on the prior distribution (and the given loss function) and hence involves no observation.  
In any case, given this setup we are now in a position to implement the 
Bayesian updating paradigm relying on the subsequent measurement $X=x$ via (\ref{prop:bayes}). This leads to  the following result, which provides a method for constructing Bayes estimators by solving a minimization problem formulated in terms of the posterior distribution $\psi_{\vartheta|X=x}$.  

\begin{theorem}\label{thm:casella}
	Assume that for almost all ${\bf x}$ there exists $\widehat\theta({\bf x})$ minimizing 
	\[
	\widetilde\theta\mapsto\mathbb E_{\psi_{\vartheta|_{X={\bf x}}}}(\mathscr L(\widetilde\theta({\bf x}),\theta)),
	\]
	where $\widetilde\theta$ runs over the set of estimators with a finite risk. Then $\widehat\theta=\widehat\theta(X)$ is a Bayes estimator. 
\end{theorem}

\begin{proof}
	By assumption we have, for almost all ${\bf x}$ and any $\widetilde\theta$, 
	\[
	\mathbb E_{\psi_{\vartheta|_{X={\bf x}}}}(\mathscr L(\widetilde\theta({\bf x}),\theta))\geq \mathbb E_{\psi_{\vartheta|_{X={\bf x}}}}(\mathscr L(\widehat\theta({\bf x}),\theta)).
	\]
	By Proposition \ref{int:prob}, this may be expressed in terms of conditional expectations as  
	\[
	\mathbb E_{\psi_\vartheta}(\mathscr L(\widetilde\theta(X),\vartheta)|\mathcal F_X)\geq \mathbb E_{\psi_\vartheta}(\mathscr L(\widehat\theta(X),\vartheta)|\mathcal F_X). 
	\]
	By applying $\mathbb E_{\psi_\vartheta}$ to both sides and using Proposition \ref{ceprop} (2) we conclude that $\mathscr R(\widetilde\theta(X))\geq \mathscr R(\widehat\theta(X))$.
\end{proof}

\begin{corollary}\label{bayes:est:l}
	The Bayes estimator associated to the weighted quadratic loss $\mathscr L(\widetilde\theta,\theta):=w(\theta)|\widetilde\theta-g(\theta)|^2$, $w>0$, is given by
	\[
	\widehat\theta({\bf x})=\frac{\mathbb E_{\psi_{\vartheta
				|_{X={\bf x}}}}(wg)}{\mathbb E_{\psi_{\vartheta
				|_{X={\bf x}}}}(w)}.
	\]
\end{corollary}

\begin{proof}
	The Cauchy-Schwartz inequality 
	\[
	E_{\psi_{\vartheta|_{X={\bf x}}}}(w(\theta)g(\theta))^2< \mathbb E_{\psi_{\vartheta|_{X={\bf x}}}}(w(\theta))\mathbb E_{\psi_{\vartheta|_{X={\bf x}}}}(w(\theta)g(\theta)^2)
	\]
	implies that
	\begin{eqnarray*}
		\mathbb E_{\psi_{\vartheta|_{X={\bf x}}}}(\mathscr L(\widetilde\theta({\bf x}),\theta))
		& = & \mathbb E_{\psi_{\vartheta|_{X={\bf x}}}}(w(\theta))\widetilde\theta({\bf x})^2\\
		&  & \quad -2\mathbb E_{\psi_{\vartheta|_{X={\bf x}}}}(w(\theta)g(\theta))\widetilde\theta({\bf x})+\mathbb E_{\psi_{\vartheta|_{X={\bf x}}}}(w(\theta)g(\theta)^2),
	\end{eqnarray*}
	viewed as a quadratic expression in $\widetilde\theta({\bf x})$, has a negative discriminant and hence is minimized at $\widetilde\theta({\bf x})=\widehat\theta({\bf x})$. 
\end{proof}

We refer to  \cite[Chapter 2]{robert2007bayesian} for an extensive discussion on this ``decision-theoretic'' approach to Bayes estimation.

\begin{example}\label{bay:est:ex:beta}
	A much more flexible choice 
	for the prior in the sunrise problem of Example \ref{laplace} is  
	\[
	\psi_{\mathfrak p}(p)=\frac{\Gamma(\alpha+\beta)}{\Gamma(a)\Gamma(b)}p^{\alpha-1}(1-p)^{\beta-1}{\bf 1}_{[0,1]},\quad \alpha,\beta>0,
	\]
	the $\mathsf{Beta}(\alpha,\beta)$ distribution (the case $\alpha=\beta=1$ corresponds to the uniform distribution).  We thus calculate that the posterior is 
	\[
	\psi_{\mathfrak p|_{X^{(n)}=s}}\sim \mathsf{Beta}(\alpha+s,\beta+n-s),
	\]
	so in this case we obtain
	\begin{equation}\label{exp:post}
		P(X_{n+1}=1|_{X^{(n)}=s})=\frac{s+\alpha}{n+\alpha+\beta}. 
	\end{equation}
	This illustrates how sensitive the Bayesian machinery is to the choice of the prior. Moreover, taking into account that the right-hand side of (\ref{exp:post}) equals the expected value of the posterior, if we choose the loss function to be $\mathscr L(\widetilde\theta,\theta)=|\widetilde\theta-\theta|^2$ in Corollary \ref{bayes:est:l} we find that the corresponding Bayes estimator is
	\[
	\widehat p_{(n)}(X)=\frac{X^{(n)}+\alpha}{n+\alpha+\beta}=(1-\gamma)
	\frac{\alpha}{\alpha+\beta}+\gamma\overline X_n, \quad \gamma=\frac{n}{n+\alpha+\beta}, 
	\]
	which interpolates between $\alpha/(\alpha+\beta)$, the expected value of $\psi_{\mathfrak p}$ (the natural estimator prior to any observation) and the sample mean $\overline X_n$ (the ``frequentist'' estimator that completely ignores the Bayesian paradigm incarnated in the prior). 
	In particular, for large samples the prior mean plays a negligible role as $\widehat p_{(n)}(X)$ becomes indistinguishable from the ML estimator $\overline X_n$.  
	Finally, if we apply this same recipe to the normal setting of Example \ref{pri:post:n}, it follows from
	(\ref{post:m:norm})
	that the corresponding Bayes estimator is 			
	\[
	\widehat\mu_{{\rm pos},n}(X)=(1-\lambda)\mu_{\rm pr}+\lambda\overline X_n, \quad \lambda =\frac{\sigma_{\rm pr}^2}{\sigma_{\rm pr}^2+\sigma^2/n}.
	\]
	Again, this interpolates between the prior mean and the sample mean with 
	\[
	\widehat\mu_{{\rm pos},n}(X)\approx_{n\to+\infty} \overline X_n,
	\] 
	so the asymptotic behavior completely disregards the prior mean. 
	\qed
\end{example}

\begin{remark}\label{asym:effi:bayes}(asymptotic efficiency of Bayes estimators) As illustrated in Example \ref{bay:est:ex:beta} above, Corollary \ref{bayes:est:l} shows that the determination of a Bayes estimator for $\theta$ under a quadratic loss boils down to computing the expectation of the posterior, a quite feasible task in some cases. Similarly to the course of action taken in the ``frequentist'' setting, with those estimators at hand we may then examine their asymptotic efficiency. In the cases treated above, this may be easily reduced to CLT. Indeed, in the Bernoulli case we compute that 
	\[
	\sqrt{n}\left(\widehat p_{(n)}(X)-p\right)=\sqrt{n}\left(\overline X_n-p\right)+\frac{\sqrt{n}}{\alpha+\beta+n}\left(\alpha-\left(\alpha+\beta\right)\overline X_n\right),
	\] 
	where $p$ is the true value of the unknown parameter, 
	so that Theorem \ref{slutsky} and CLT apply to conclude that 
	\[
	\sqrt{n}\left(\widehat p_{(n)}(X)-p\right)\stackrel{d}{\to}\mathcal N(0,p(1-p)). 
	\]
	Similarly, in the normal setting, 
	\[
	\sqrt{n}\left(\widehat\mu_{{\rm pos},n}(X)-\mu\right)=\sqrt{n}\lambda\left(\overline X_n-\mu\right)+\sqrt{n}\left(1-\lambda\right)\left(\mu_{\rm pr}-\mu\right),
	\]
	and since $\lambda\to 1$ and $\sqrt{n}(1-\lambda)=O(n^{-1/2})\to 0$
	we see that 
	\[
	\sqrt{n}\left(\widehat\mu_{{\rm pos},n}(X)-\mu\right)\stackrel{d}{\to}\mathcal N(0,\sigma^2).
	\]
	Thus, in each case the limiting distribution of the appropriate standardization of the Bayes estimator is normal with a dispersion independent of the parameters of the prior distribution. This turns out to be a quite general phenomenon. Indeed, results in \cite[Section 6.8]{lehmann2006theory} guarantee, under suitable regularity conditions and in the regime of large samples, that:
	\begin{itemize}
		\item  the posterior distribution becomes asymptotically normal, and  hence insensitive to the chosen prior, with a variance depending on the true value $\theta_0$ of the unknown parameter only through its Fisher information: 
		\begin{equation}\label{alt:1:1}
			\sqrt{n}\left(\psi_{\vartheta|_{X={\bf x}}}-\theta_0\right)\stackrel{d}{\to}\mathcal N(0,\mathscr F(\theta_0)^{-1}). 
		\end{equation}
		\item  as a consequence, the limiting distribution associated to the Bayes estimator $\widehat\theta_n$ (under quadratic loss) is normal as well  with the same asymptotic  variance: 
		\begin{equation}\label{alt:1:2}
			\sqrt{n}\left(\widehat\theta_n(X)-\theta_0\right)\stackrel{d}{\to}\mathcal N(0,\mathscr F(\theta_0)^{-1}). 
		\end{equation}
		In particular,  $\widehat\theta_n$ is asymptotically efficient.
	\end{itemize}
	We note that (\ref{alt:1:2}) follows from (\ref{alt:1:1}) and the fact that 
	\[
	\sqrt{n}\left(\widehat\theta_n(X)-\psi_{\vartheta|_{X={\bf x}}}\right)\stackrel{d}{\to} 0.
	\]
	Although in the long run these results eventually succeed in altogether eliminating the effect of the subjective choice of the prior, they remain a bit extraneous in regard to the Bayesian tenet according to which by its very nature the posterior is conditional on the sample, whose size has been fixed once and for all.  \qed
\end{remark}

\appendix

\section{Brownian motion, It\^o calculus, and some of their applications}\label{brmot}

In this rather long appendix, whose understanding 
requires only familiarity with the material above up to Subsection \ref{normaldist:sub},
we turn our attention to Brownian motion, an important example of a stochastic process, and its most basic properties. Although the main motivation here is to provide a proof of the Gaussian concentration inequality (\ref{gauss:conc:ineq}) with the sharp constant $C=1/2$, which is presented in Section \ref{gauss:conc:sec}, we also include a few other applications of the associated It\^o calculus, a cornerstone in the modern theory of stochastic processes.

\subsection{Brownian motion: its construction and basic regularity properties}\label{subs:def:bm}

Since Brownian motion is the prototypical example of a stochastic process, we start by recalling the definition of this fundamental concept.

\begin{definition}
	\label{stocproc} A (continuous-time) stochastic process on a probability space $(\Omega,\mathcal F,P)$ is a one-parameter family of random variables $X_t:\Omega\to\mathbb R^n$, $t\geq 0$. 
\end{definition}

\begin{remark}\label{stocproc:d}
	We may also consider situations in which the parameter $t$ indexing the process is replaced by a discrete one, say $n\in\mathbb N$. In this case, the family $\{X_n\}_{n\in\mathbb N}$ is called a discrete-time stochastic process.
	\qed
\end{remark}

The map $\omega\in\Omega\mapsto X_t(\omega)\in\mathbb R^n$ allows us to think of $\Omega$ as a subset of $(\mathbb R^n)^{[0,+\infty)}$. Thus, to each $\omega\in\Omega$ the process defines a path in $\mathbb R^n$. 
In general, the regularity of the process is expressed in terms of the regularity of these paths. For instance, we say that the process is continuous if $X_t(\omega)$ is continuous for almost any $\omega\in\Omega$. Here we only deal with processes which are at least continuous. In any case, this pathwise description of stochastic processes motivates the following definition.

\begin{definition}\label{probmeas}
	Given a stochastic process $X_t$, its {\em probability distributions} in $\mathbb R^{nk}$, $k=1,2,\cdots$, are given by 
	\[
	\mu^X_{t_1,\cdots,t_k}(F_1\times\cdots\times F_k)=P(X_{t_1}\in F_1,\cdots, X_ {t_k}\in F_k), 
	\]	
	where $t_i\geq 0$ and $F_i\in \mathcal B^n$, $i=1,\cdots,k$.
\end{definition}   

In other words, $	\mu^X_{t_1,\cdots,t_k}=P_{(X_{t_1},\cdots,X_{t_k})}$ is the joint distribution of $(X_{t_1},\cdots,X_{t_k})$; cf. Definition \ref{distr}. The next result shows that a stochastic process can be reconstructed from its probability distributions if a couple of compatibility conditions are satisfied.  

\begin{theorem}(Kolmogorov's extension)
	\label{extkolm} Assum that for any $t_1,\cdots, t_k\geq 0$ there exists a probability measure $\nu_{t_1,\cdots,t_k}$ in $\mathbb R^{nk}$ such that:
	\begin{itemize}
		\item $(K_1)$ $\nu_{t_{\tau(1)},\cdots,t_{\tau(k)}}(F_1\times\cdots\times F_k)=\nu_{t_1,\cdots,t_k}(F_{\tau^{-1}(1)},\cdots, F_{\tau^{-1}(k)})$, for any permutation $\tau$.
		\item $(K_2)$ $\nu_{t_1,\cdots,t_k}(F_1\times\cdots\times F_k)=\nu_{t_1,\cdots,t_k,t_{k+1},\cdots,t_{k+m}}(F_1\times\cdots\times F_k\times \mathbb R^n\times\cdots\times\mathbb R^n)$, for any $m\geq 1$. 
	\end{itemize}
	Then there exists a probability space $(\Omega,\mathcal F,P)$ and a stochastic process $X_t:\Omega\to\mathbb R^n$ such that 
	\[
	\nu_{t_1,\cdots,t_k}=\mu^X_{t_1,\cdots,t_k}, 
	\]
	for $(t_1,\cdots,t_k)$.
\end{theorem}

\begin{proof}
	See \cite[Section 2.4]{tao2011introduction}. 
\end{proof}

Now we can construct Brownian motion in $\mathbb R^n$ following an approach due to Kolmogorov; for other possibilities see \cite{schilling2014brownian}. If $0\leq t_1< \cdots< t_k$ and $y=(y_1,\cdots,y_k)\in\mathbb R^{nk}$ define, if $t_1>0$,
\[
\nu_{t_1,\cdots,t_k}(F_1\times\cdots\times F_k)=\frac{1}{(2\pi)^{nk/2}\sqrt{\det C}}
\int_{F_1\times\cdots\times F_k}e^{-\frac{1}{2}\langle C^{-1}y,y\rangle}dy,
\]
where each $F_i\in\mathcal B^n$ and $C$ is $nk\times nk$-matrix whose $ij$-block is $C_{ij}=t_i\wedge t_jI_n$\footnote{Here, $a\wedge b=\min\{a,b\}$. Also, note that the symmetric matrix $C$ is positive definite.}. If $t_1=0$ we use instead $\delta_{\vec{0}}\otimes \nu_{t_2,\cdots,t_k} $, where $\delta_{\vec{0}}$ is the Dirac measure centered at the origin.
This may be extended to all $(t_1,\cdots,t_k)$ so that $(K_1)$ is satisfied. Moreover, $(K_2)$ is satisfied as well because of Proposition \ref{welldef}. Thus, by means of Theorem \ref{extkolm} we establish the following foundational existence result. 

\begin{theorem}\label{brown:exist}
There exists a probability space $(\Omega,\mathcal F,P)$ and a stochastic process $b_t:\Omega\to\mathbb R^n$ so that
\[ 
P(b_{t_1}\in F_1,\cdots,b_{t_k}\in F_k)=\frac{1}{(2\pi)^{nk/2}\sqrt{\det C}}
\int_{F_1\times\cdots\times F_k}e^{-\frac{1}{2}\langle C^{-1}x,x\rangle}dx.
\]
This is called {\em Brownian motion} ({\em BM}) in $\mathbb R^n$ (starting at $\vec{0}$).
\end{theorem}

The next proposition lists the characterizing properties of $BM$.

\begin{proposition}
	\label{charbm} BM in $\mathbb R^n$ satisfies the following properties:
	\begin{enumerate}
		\item  $b_0=0 $ a.s.;
		\item it has stationary normal increments, i.e. for any $0\leq s< t$, $h\geq -s$, $b_{t+h}-b_{s+h}$ and $b_t-b_s$ are identically distributed with 
		\begin{equation}\label{br:dist}
			b_t-b_s\sim \mathcal N\left(0,(t-s){\rm Id}_n\right);
		\end{equation}
		\item it has independent increments, that is, for  any $0= t_0<t_1\cdots< t_k$, $\{b_{t_1}-b_{t_0},\cdots,b_{t_k}-b_{t_{k-1}}\}$ is an independent family of random vectors; 
		\item $t\mapsto b_t(\omega)$ is continuous for any $\omega$.
	\end{enumerate}
\end{proposition}

\begin{proof}
	From the construction, (1) follows immediately. To approach (2) and (3) we take $u=(u_1,\cdots,u_k)\in\mathbb R^{nk}$, $v_j=u_j+\cdots+u_k$, $j=1,\cdots,k$,  and $b=(b_{t_1},\cdots,b_{t_k})$, so if we use the language of characteristic functions in Definition \ref{funcchar} and the explicit computation of this object for normally distributed random vectors in Proposition \ref{funchnor} we have (recalling that $b_{t_0}=0$)
	\begin{eqnarray*}
		\phi_{(b_{t_1}-b_{t_{0}},\cdots,b_{t_k}-b_{t_{k-1}})}(v_1,\cdots,v_k)
		& = & 
		\mathbb E(e^{{\bf i}\sum_{j=1}^k\langle b_{t_j}-b_{t_{j-1}},v_j\rangle})\\
		& = & \mathbb E(e^{{\bf i}\langle b,u\rangle})\\
		& = & \phi_b(u)\\
		& = & e^{-\frac{1}{2}\langle Cu,u\rangle}.
	\end{eqnarray*}
	But
	\begin{eqnarray*}
		\langle Cu,u\rangle
		& = & 
		\sum_{j=1}^k\sum_{l=1}^k(t_j\wedge t_l)\langle u_j,u_l\rangle \\
		& = &  t_k\|u_k\|^2+\sum_ {j=1}^{k-1}t_j\langle u_j,u_j+2u_{j+1}+\cdots 2u_k\rangle\\
		&  = & t_k\|u_k\|^2+\sum_ {j=1}^{k-1}t_j(\|u_j+\cdots+u_k\|^2-\|u_{j+1}+\cdots+u_k\|^2)\\ 	
		& = & \sum_ {j=1}^kt_j\|u_j+\cdots+u_k\|^2 -\sum_ {j=1}^kt_{j-1}\|u_j+\cdots+u_k\|^2\\
		& = & \sum_ {j=1}^k(t_j-t_{j-1})\|v_j\|^2,	
	\end{eqnarray*}
	so that 
	\begin{equation}\label{stat:bm:f}
		\phi_{(b_{t_1}-b_{t_{0}},\cdots,b_{t_k}-b_{t_{k-1}})}(v_1,\cdots,v_k)
		=  \Pi_{j=1}^k e^{-\frac{1}{2}(t_j-t_ {j-1})\|v_j\|^2}.
	\end{equation}
	Notice that for $0\leq s<t$ this specializes to 
	\[
	\phi_{b_t-b_s}(v)=e^{-\frac{1}{2}\left(t-s\right)\|v\|^2}, \quad v\in\mathbb R^n,
	\]
	so that Corollary \ref{det:norm:dist} applies to ensure that (\ref{br:dist}) holds, which proves (2). As for (3), note that (\ref{stat:bm:f})
	may be rewritten as  
	\[
	\phi_{(b_{t_1}-b_{t_{0}},\cdots,b_{t_k}-b_{t_{k-1}})}(v_1,\cdots,v_k)
	=  \Pi_{j=1}^k\phi_{b_{t_j}-b_{t_{j-1}}}(v_j),
	\]
	so we may proceed as in the last step of the proof of Proposition \ref{unc:ind:n} and use the standard Fourier inversion formula to confirm that the joint distribution of the random vector of increments decomposes as 
	\[
	\psi_{(b_{t_1}-b_{t_{0}},\cdots,b_{t_k}-b_{t_{k-1}})}(x_1,\cdots,x_k)
	= \Pi_{j=1}^k\psi_{b_{t_j}-b_{t_{j-1}}}(x_j),
	\]
	which proves (3) by Proposition \ref{inddens}. 
	The proof of (4) is presented in the next section; see Proposition \ref{pass}.
\end{proof}

\begin{remark}\label{gauss:proc}
	(Brownian motion as a Gaussian process)
	Let \(T\subset\mathbb R_{\geq 0}\) be an index set and let \(\{X_t\}_{t\in T}\) be a stochastic process. We say that \(X_t\) is {\em Gaussian} if, for every \(t_1,\dots,t_n\in T\), the random vector
	\[
	(X_{t_1},\dots,X_{t_n})
	\]
	is normally distributed as in Definition \ref{normdistrv}, with its  mean and covariance functions being given by
	\[
	m(t):=\mathbb E(X_t), \qquad 
	K(s,t):=\mathbb C(X_s,X_t).
	\]
	 From its very construction in Theorem \ref{brown:exist}, Brownian motion is a Gaussian process (if $n=1$ we have $m(t)=0$ and $K(t,s)=t\wedge s$). 
	 Another relevant example in the sequel is the Brownian bridge constructed in Subsection \ref{fk:sec}; see Remark \ref{bb:0:b:sde}.
	It follows from Proposition \ref{det:norm:dist} that a 
	Gaussian process is completely determined (in law) by \(m\) and \(K\): if \(X\) and \(Y\) are Gaussian processes with the same mean and covariance functions, then
	\[
	(X_{t_1},\dots,X_{t_n})
	\,\,\textrm{and}\,\,
	(Y_{t_1},\dots,Y_{t_n})
	\]
	are identically distributed
	for every such finite collection \(t_1,\dots,t_n\in T\). 
	\qed
\end{remark}

\begin{proposition}\label{exp:dif:br}
	\label{furth} If $t\leq s$ then $\mathbb E(\|b_s-b_t\| ^2)=n(s-t)$. 
\end{proposition}

\begin{proof}
	We know that $\mathbb E(b_t)=0$ and 
	\[
	{\mathbb C}(b_s,b_t)=
	\left(
	\begin{array}{cc}
	s I_n & s\wedge t I_n\\
	s\wedge t I_n & t I_n	
		\end{array}
	\right).
	\]
	Also, if 
	$b_t=(b_t^{(1)},\cdots,b_t^{(n)})$ is the coordinate expression of $b_t$,  
	\[
		\mathbb E(\langle b_s,b_t\rangle)
		 =  \sum_i\mathbb E\left(b_s^{(i)}b_t^{(i)}\right)
		 =  \sum_i{\mathbb C}(b_s,b_t)_{ii}.
	\]
	Hence, if $t\leq s$,
	\begin{eqnarray*}
		\mathbb E(\|b_s-b_t\|^2) & = & \mathbb E(\|b_s\|^2)-2\mathbb E(\langle b_s,b_t\rangle) +\mathbb E(\|b_ t\|^2)\\
		& = & ns-2nt+tn,
	\end{eqnarray*} 
	as claimed.
\end{proof}

\begin{remark}\label{br:ind}
	It follows from (\ref{br:dist}) that ${\mathbb C}(b_t)_{ij}=t\delta_{ij}$,
	so that by Proposition \ref{unc:ind:n} we see that the coordinate components $\{b_t^{(i)}\}_{i=1}^n$of $b_t$ form an independent family of BMs in $\mathbb R$. Conversely, we may first construct BM $b_t$ in $\mathbb R$ by using the Kolmogorov's argument above (note that in this case the matrix $C$ has a much simpler structure) and then take $n$ {\em independent} copies of $b_t$, say $\{b_t^{(1)}, \cdots b_t^{(n)}\}$, in order to exhibit BM in $\mathbb R^n$ as $(b_t^{(1)},\cdots,b_t^{(n)})$. \qed
\end{remark}

We now turn to the basic regularity properties of Brownian motion and we
start by 
proving Proposition \ref{charbm}, (4). To simplify matters, we only consider the case $n=1$; cf.  Remark \ref{br:ind}. The proof is based on the following general regularity result for stochastic processes. 	We recall that saying that $X_t'$ is a {\em modification} of $X_t$ means that $P(X_t=X_t')=1$ for any $t$.

\begin{theorem}
	\label{kolchen} (Kolmogorov's continuity) If $X_t:\Omega\to \mathbb R$ is a stochastic process satisfying
	\[
	\mathbb E(|X_s-X_t|^\alpha)\leq C|s-t|^{\beta +1}, \quad s,t\geq 0,
	\]
	then there exists a modification $X'_t$ of $X_t$ whose paths are locally $\gamma$-H\"older continuous, where $0<\gamma<\beta/\alpha$. In particular, $X_t'\in C^0$.
\end{theorem}

\begin{proof}
	See \cite[Section 2.2]{le2013mouvement}.
\end{proof}

The regularity of BM now follows from the following fact.

\begin{proposition}
	\label{regfor} BM in $\mathbb R$ satisfies
	\[
	\mathbb E(|b_s-b_t|^{2k})= \frac{(2k)!}{2^kk!}|s-t|^k,\quad k\geq 1.
	\]
\end{proposition}

\begin{proof}
	Since $b_s-b_t\sim\mathcal N(0,s-t)$, this follows from the discussion in Example \ref{mom:normal}.
\end{proof}

\begin{proposition}
	\label{pass}
	Eventually passing to a modification, BM is locally $(\frac{1}{2}-\epsilon)$-H\"older continuous, for any $\epsilon>0$.
\end{proposition}

\begin{proof}
	Apply the results above with $\alpha=2k$ and $\beta=k-1$ and send $k\to+\infty$.
\end{proof}

This is in a sense the best regularity we can have. To check this we need a definition.

\begin{definition}
	\label{variat} If $X_t:\Omega\to \mathbb R$ and $p>0$, we define its $p^{\rm th}$ {\em variation} by
	\[
	\langle X\rangle_t^{(p)}(\omega)=\lim_{\Delta t_k\to 0}\sum_{t_k\leq t}|X_{t_{k+1}}(\omega)-X_{t_k}(\omega)|^p,
	\]
	where $\Delta t_k=t_{k+1}-t_k=t/k$ and the limit is taken in probability.
\end{definition}

It turns out that the quadratic variation of BM can be explicitly computed.

\begin{proposition}\label{qv:compt}
	\label{quadbm} BM $b_t$ in $\mathbb R$ satisfies 
	\[
	\langle b\rangle_t^{(2)}=t,
	\]
	with convergence in $L^2$-mean.
\end{proposition}

\begin{proof}
	Note that 
	\begin{eqnarray*}
		\mathbb E\left(\left(\sum_{t_k\leq t}(b_{t_{k+1}}-b_{t_k})^2-t\right)^2\right) & = & 
		\mathbb E\left(\left(\sum_{t_k\leq t}(b_{t_{k+1}}-b_{t_k})^2-(t_{k+1}-t_k)\right)^2\right)\\
		& = & I + II,
	\end{eqnarray*}
	where 
	\[
	I=\mathbb E\left(\sum_{t_k\leq t}\left((b_{t_{k+1}}-b_{t_k})^2-(t_{k+1}-t_k)\right)^2\right)
	\]
	and 
	\begin{eqnarray*}
	II
	& = & 2\sum_{t_j<t_k\leq t}{\mathbb E\left(\left((b_{t_{j+1}}-b_{t_j})^2-(t_{j+1}-t_j)\right)\left((b_{t_{k+1}}-b_{t_k})^2-(t_{k+1}-t_k)\right)\right)}\\
	& = & 2\sum_{t_j<t_k\leq t}{\mathbb C}\left((b_{t_{j+1}}-b_{t_j})^2,(b_{t_{k+1}}-b_{t_k})^2\right),
	\end{eqnarray*}
where we used  Proposition \ref{furth} with $n=1$ in the last step. 
	But by Proposition \ref{charbm}, (3), 
	\[
	b_{t_{j+1}}-b_{t_j}\perp\!\!\perp b_{t_{k+1}}-b_{t_k}  \Rightarrow (b_{t_{j+1}}-b_{t_j})^2\perp\!\!\perp (b_{t_{k+1}}-b_{t_k})^2,
	\]
	and hence $II=0$ by Corollary \ref{induncorr}.
	On the other hand, 
	\begin{eqnarray*}
		I & = & \sum_{t_k\leq t}\left(\mathbb E\left((b_{t_{k+1}}-b_{t_k})^4\right)-2(t_{k+1}-t_k)\mathbb E((b_{t_{k+1}}-b_{t_k})^2)+({t_{k+1}}-{t_k})^2\right)\\
		& \stackrel{{\rm Prop.} \ref{regfor}}{=} & \sum_{t_k\leq t}\left(3(t_{k+1}-{t_k})^2-2({t_{k+1}}-{t_k})^2+({t_{k+1}}-{t_k})^2\right)\\
		& = & 2\sum_{t_k\leq t} ({t_{k+1}}-{t_k})^2\\
		&\leq & 2\frac{t^2}{k}\to 0,
	\end{eqnarray*}
	as $k\to +\infty$.
\end{proof}

\begin{proposition}
	\label{firvarinf}One has $\langle b\rangle^{(1)}_t=+\infty$ a.s. In other words, the total variation of $b_t$ blows up in any interval. 
\end{proposition}

\begin{proof}
	For $\omega\in\Omega$ we have 
	\begin{eqnarray*}
		\sum_{t_k\leq t}(b_{t_{k+1}}(\omega)-b_{t_k}(\omega))^2
		& \leq & 
		\sum_{t_k\leq t}(b_{t_{k+1}}(\omega)-b_{t_k}(\omega))
		\sup_{t_k}\sum_{t_k\leq t}|b_{t_{k+1}}(\omega)-b_{t_k}(\omega)|\\
		& \leq & 
		\langle b\rangle_t^{(1)}(\omega) \sup_{t_k}\sum_{t_k\leq t}|b_{t_{k+1}}(\omega)-b_{t_k}(\omega)|.
	\end{eqnarray*}
	From Proposition \ref{qv:compt}, and possibly passing to a subequence along the given partitions of $[0,t]$, we may assume that the left-hand side converges to $t$ a.s.
	But the supremum goes to $0$ as $b_t(\omega)$ is uniformly continuous in $[0,t]$, which yields a contradiction if $\langle b\rangle_t^{(1)}(\omega)$ is finite.
\end{proof}

A similar argument yields the following result.

\begin{proposition}
	\label{nohol}
	The paths of $b_t$ are nowhere $\gamma$-H\"older continuous for $\gamma>1/2$.
\end{proposition}

\begin{proof}
	Assume $|b_{s'}-b_{t'}|\leq K|s'-t'|^\gamma$, $0\leq t'\leq s'\leq t$. It follows that  
	\[
	\sum_{t_k\leq t}(b_{t_{k+1}}(\omega)-b_{t_k}(\omega))^2\leq K^2t\sup_k|t_{k+1}-t_k|^{2\gamma-1}.
	\]
	As above, we may assume that the left-hand side converges to $t$ a.s. But the supremum goes to $0$ if $\gamma>1/2$, which gives a contradiction.
\end{proof}

\begin{remark}\!\!$\bigstar$\label{wiener:sp}
	(The Wiener space)
	From Proposition~\ref{charbm}~(4), we may identify the sample space $\Omega$
	underlying the construction of Brownian motion in Theorem~\ref{brown:exist}
	with $C_{\vec{0}}$, the space of continuous functions
	$\omega:[0,+\infty)\to\mathbb R^n$ such that $\omega(0)=\vec{0}$, via the rule
	$\omega(t)=b_t(\omega)$.
	It then follows from Proposition~\ref{nohol} that the support of the underlying
	probability measure $P$ contains no path $\omega$ that is too regular, for instance,
	$\gamma$-Hölder continuous with $\gamma>1/2$.
	From this perspective, we refer to $C_{\vec{0}}$, endowed with the induced
	probability measure (still denoted by $P$), as the \emph{Wiener space}
	(starting at $\vec{0}$), and to any $\omega$ in the support of $P$ (the
	\emph{Wiener measure}) as a \emph{Brownian path} (again starting at $\vec{0}$).
		In this sense, the Wiener space provides a canonical example of a Gaussian measure on an infinite-dimensional function space, a viewpoint that plays a central role in stochastic analysis and the theory of Gaussian processes
	\cite{bogachev1998gaussian}.
	However, in alignment with the extension dogma mentioned in Remark \ref{dogma}, the precise nature of $\Omega$ is ultimately irrelevant,
	as only the law of the process $(b_t)_{t\ge 0}$ matters.
\end{remark}

\subsection{Martingales}\label{condmart}

We now isolate another central notion in the theory.

\begin{definition}
	\label{marting} Let $b_t:\Omega\to\mathbb R^n$ be BM in $\mathbb R^n$ (starting at $x$) and let $\mathcal F_t$ be the $\sigma$-algebra generated by $\{b_{t'}\}_{t'\leq t}$. A {\em martingale} (rel. to $b_t$) is a stochastic process $M_t:\Omega\to\mathbb R^n$ such that 
	\begin{itemize}
		\item $M_t$ is $\mathcal F_t$-measurable for any $t>0$;
		\item $\mathbb E(\|M_t\|)<+\infty$;
		\item $\mathbb E(M_s|\mathcal F_t)=M_t$ whenever $t\leq s$.
	\end{itemize}
\end{definition}

\begin{proposition}
	\label{bismart}$b_t$ is a martingale.
\end{proposition}

\begin{proof}
	If $t\leq s$ write 
	\[
	\mathbb E(b_s|\mathcal F_t)=\mathbb E(b_s-b_t|\mathcal F_t)+\mathbb E(b_t|\mathcal F_t).
	\]
	Proposition \ref{charbm}, (3), implies that $b_s-b_t\perp\!\!\perp\mathcal F_t$. Hence, by Proposition \ref{ceprop}, (4), $\mathbb E(b_s-b_t|\mathcal F_t)=\mathbb E(b_s-b_t)=0$. On the other hand, since $b_t$ is (obviously) $\mathcal F_t$-measurable, Proposition \ref{ceprop}, (3), implies that $\mathbb E(b_t|\mathcal F_t)=b_t$.
\end{proof}

For our purposes, a basic property of a martingale is that its expectation is preserved in time. This confirms that martingales are ``pure fluctuation'' processes.

\begin{proposition}\label{martcons}
	If $M_t$ is a martingale then $\mathbb E(M_t)=\mathbb E(M_s)$, for any $s,t$.
\end{proposition}

\begin{proof}
	By Proposition \ref{ceprop}, (2), if $t\leq s$ we have 
	\[
	\mathbb E(M_t)=\mathbb E(\mathbb E(M_s|\mathcal F_t))=\mathbb E(M_s),
	\]
as claimed.	
\end{proof}

\subsection{It\^o calculus}\label{itoint}

Consider a partition $0=t_0<t_1\cdots <t_k=t$ and let $b_t$ be BM in $\mathbb R$ with $b_0=0$. By Proposition \ref{charbm}, (3), $b_{t_{j+1}}-b_{t_j}\perp\!\!\perp b_{t_j}$ and hence
\begin{equation}\label{ito:t:1}
	\mathbb E\left(\sum_jb_{t_j}(b_{t_{j+1}}-b_{t_j})\right)=0.
\end{equation}
On the other hand,
\begin{eqnarray*}
	\mathbb E\left(\sum_jb_{t_{j+1}}(b_{t_{j+1}}-b_{t_j})\right)
	& \stackrel{(\ref{ito:t:1})}{=} & \mathbb E\left(\sum_k(b_{t_{j+1}}-b_{t_j})^2\right)\\
	&  = & \sum_j(t_{j+1}-t_j)\\
	& = & t,
\end{eqnarray*}
where we used Proposition \ref{furth} in the second step.
This simple computation, which reflects the already known fact that $db_t$ can not be interpreted as a  classical Lebesgue-Stieltjes integrator since $b_t$ has infinite total variation by Proposition \ref{firvarinf}, illustrates the difficulty of making sense of stochastic integrals like $\int_0^tb_sdb_s$ by standard methods. Put in another way,  each choice of $\{\widehat t_j\}$ such that $t_{j}\leq \widehat t_j\leq t_{j+1}$ yields its own output for the  
``approximate'' stochastic integral 
$
\sum_jb_{\widehat t_{j}}(b_{t_{j+1}}-b_{t_j})
$.
Among the many possibilities available, It\^o integration corresponds to choosing the first option (\ref{ito:t:1}) above. The basics of this kind of stochastic integration may be found in many sources \cite{karatzas2012brownian,le2013mouvement,baudoin2014diffusion} and our presentation below follows \cite{oksendal2013stochastic} closely, a text very much oriented to applications (in particular, to Mathematical Finance) to which we refer for the detailed proofs of most of the results on It\^o calculus described in the sequel.  

Recall that a filtration $\mathcal F_t$ of a $\sigma$-algebra is a nested family of $\sigma$-subalgebras of $\mathcal F$. Here we consider the filtration $\mathcal F_t=\mathcal F_{\{b_{t'}\}_{t'\leq t}} $. 

\begin{definition}\label{adapproc} 
	We say that a process $f:\mathbb R_+\times\Omega\to\mathbb R$ is {\em adapted} if $\omega\mapsto f(t,\omega)$ is $\mathcal F_t$-measurable for any $t$.
\end{definition}

\begin{definition}
	\label{svclass}
	For $0\leq S<T$ we denote by $\mathcal A(S,T)$ the class of all processes $f:\mathbb R_+\times\Omega\to \mathbb R$ such that:
	\begin{enumerate}
		\item $f$ is $\mathcal B\times\mathcal F$-measurable;
		\item $f$ if {\em adapted} to $\mathcal F_t$, the filtration defined by $b_t$;
		\item $\mathbb E(\int_S^Tf(t,\omega)^2dt)<+\infty$.
	\end{enumerate}
\end{definition}

\begin{definition}
	\label{elemproc}We say that $f\in\mathcal A(S,T)$ is {\em elementary} if 
	\[
	f(t,\omega)=\sum_jf_j(\omega){\bf 1}_{[t_j,t_{j+1})}(t).
	\]
\end{definition}

Note that if $f$ is elementary then $f_j$ is $\mathcal F_{t_j}$-measurable.

\begin{definition}
	\label{itoelem}(It\^o integral for elementary processes) If $f\in\mathcal A(S,T)$ is elementary we define
	\[
	\int_S^Tf(t,\omega)db_t(\omega)=\sum_jf_j({\omega})(b_{t_{j+1}}(\omega)-b_{t_j}(\omega)).
	\]
\end{definition}	

Notice that this depends measurably on $\omega$ and hence defines a random variable.

\begin{proposition}
	\label{itoiso}(It\^o isometry for elementary processes) If $f\in\mathcal A(S,T)$ is elementary then
	\[
	\mathbb E\left(\left(\int_S^Tf(t,\omega)db_t(\omega)\right)^2\right)=\mathbb E\left(\int_S^Tf(t,\omega)^2dt \right).
	\]
\end{proposition}

\begin{proof}
	Let $f=\sum_jf_j{\bf 1}_{[t_j, t_{j+1})}$. Since $f_j$ is $\mathcal F_{t_j}$-measurable, Proposition \ref{charbm}, (3), implies that $f_j\perp\!\!\perp b_{t_{j+1}}-b_{t_j}$. Hence, $f_j^2\perp\!\!\perp (b_{t_{j+1}}-b_{t_j})^2$ and we have
	\[
	\mathbb E\left(f_j^2(b_{t_{j+1}}-b_{t_j})^2\right)=\mathbb E(f_j^2)\mathbb E((b_{t_{j+1}}-b_{t_j})^2)=\mathbb E(f_j^2)(t_{j+1}-t_j),
	\]
	where we used Proposition \ref{furth} in the last step. On the other hand, if $j<k$ we have $f_jf_k (b_{t_{j+1}}-b_{t_j})\perp\!\!\perp b_{t_{k+1}}-b_{t_k}$ and hence, 
	\[
	\mathbb E\left(f_jf_k (b_{t_{j+1}}-b_{t_j})(b_{t_{k+1}}-b_{t_k})\right)=\mathbb E\left(f_jf_k (b_{t_{j+1}}-b_{t_j})\right)\mathbb E(b_{t_{k+1}}-b_{t_k})=0.
	\]
	It follows that 
	\begin{eqnarray*}
		\mathbb E\left(\left(\int_S^Tf(t,\omega)db_t(\omega)\right)^2\right) & = & \sum_{jk}\mathbb E\left(f_jf_k (b_{t_{j+1}}-b_{t_j})(b_{t_{k+1}}-b_{t_k})\right)\\
		& = & \sum_j\mathbb E(f_j^2)({t_{j+1}}-{t_j})\\
		& = &  \mathbb E\left(\int_S^Tf(t,\omega)^2dt \right),
	\end{eqnarray*}
as claimed.	
\end{proof}

\begin{proposition}
	\label{aprox}(approximation) For any $f\in\mathcal A(S,T)$ there exists $\{f_i\}_{i=1}^{+\infty}\subset\mathcal A(S,T)$, $f_i$ elementary, so that 
	\begin{equation}\label{aproxl}
		\lim_{i\to+\infty}\mathbb E\left(\int_S^T|f-f_i|^2dt\right)=0.
	\end{equation}
\end{proposition}

\begin{proof}
	\cite[pg. 27-28]{oksendal2013stochastic}.
\end{proof}

\begin{definition}
	\label{itointvst}(It\^o integral in $\mathcal A(S,T)$) If $f\in\mathcal A(S,T)$ we define
	\[
	\int_S^Tf(t,\omega)db_t(\omega)\stackrel{L^2}{=}\lim_{i\to +\infty}\int_S^Tf_i(t,\omega)db_t(\omega),
	\]
	for some $\{f_i\}$ as in (\ref{aproxl}).
\end{definition}

Notice that, by Proposition \ref{itoiso}, the limit exists and does not depend on the sequence $\{f_i\}$ chosen to approximate $f$.

We now list the basic properties of It\^o integral. 

\begin{proposition}
	\label{itoprop}The It\^o integral satisfies the following properties:
	\begin{enumerate}
		\item $\int_S^Tfdb_t=\int_S^Ufdb_t+\int_U^Tfdb_t$;
		\item $\int_S^T(af+bg)db_t=a\int_S^Tfdb_t+ b\int_S^Tgdb_t$, $a,b\in\mathbb R$;
		\item $\mathbb E(\int_S^Tfdb_t)=0$;
		\item $\int_S^Tfdb_t$ is $\mathcal F_T$-measurable;
		\item (It\^o isometry) There holds 
		\[\mathbb E\left(\left(\int_S^Tf(t,\omega)db_t(\omega)\right)^2\right)=
		\mathbb E\left(\int_S^Tf(t,\omega)^2dt \right),
		\]
		and more generally, 
		\[
		\mathbb C\left(\left(\int_S^Tf(t,\omega)db_t(\omega)\right),
		\left(\int_S^Tg(t,\omega)db_t(\omega)\right)\right)=
		\mathbb E\left(\int_S^Tf(t,\omega)g(t,\omega)dt \right).
		\]
		\item If 
		\[
		\lim_{n\to +\infty}\mathbb E\left(\int_S^T\left(f_n(t,\omega)-f(t,\omega)\right)^2dt\right)=0
		\]
		then 
		\[
		\int_S^Tf_ndb_t \stackrel{L^2}{\to}\int_S^Tfdb_t.
		\]
		\item If $f=f(t)$ is non-random then
		\[
		\int_S^Tf(t)db_t\sim\mathcal N\left(0,\int_S^Tf(t)^2dt\right).
		\]
		\item 
		If $f=f(t)$ is non-random and  $T<T'$ then 
		\[
		\int_S^T fdb_t \perp\!\!\perp \int_T^{T'} fdb_t .
		\]
		\item Any It\^o integral has a continuous modification.
	\end{enumerate}
\end{proposition} 

\begin{proof}
	The proofs of (1)-(5) follow the same method, namely, we first check the property for elementary processes and then pass the limit, with the same argument also applying to (8). Also, (6) follows immediately from (5). As for (7), we may approximate $f$ by a sequence of elementary (and non-random) integrands $f_i$ for which the result holds (by Definition \ref{itoelem} and 
	Proposition \ref{norm:spce} (3)). It follows from Definition \ref{itointvst}  that 
	\[
	J_i:=\int_S^Tf_i(t)db_t\stackrel{L^2}{\longrightarrow} J:=\int_S^Tf(t)db_t,
	\] 
	with Proposition \ref{norm:spce} (1) implying that 
	\[
	\mathbb E(e^{{\bf i}uJ_i})=\exp\left(-\frac{u^2}{2}\int_S^Tf_i(t)^2db_t\right)\to \exp\left(-\frac{u^2}{2}\int_S^Tf(t)^2db_t\right).
	\]
	Since it may be easily checked that $\mathbb E(e^{{\bf i}uJ_i})\to \mathbb E(e^{{\bf i}uJ})$, we see that 
	\[
	\mathbb E(e^{{\bf i}uJ})=\exp\left(-\frac{u^2}{2}\int_S^Tf(t)^2db_t\right),
	\]
	and
	the result follows from Proposition \ref{norm:spce} (1) and Corollary \ref{det:norm:dist}. Finally, 
	the proof of (9) can be found in \cite[Theorem 3.2.5]{oksendal2013stochastic}. 
\end{proof}

\begin{example}
	\label{itolemesp}{\rm Let 
		\[
		f_n(s,\omega)=\sum_jb_{t_j}(\omega){\bf 1}_{[t_j,t_{j+1})}(s),
		\]
		where $\Delta t_j=t_{j+1}-t_j=t/n$. We have
		\begin{eqnarray*}
			\mathbb E\left(\int_0^t(f_n-b_s)^2ds\right) & = & \mathbb E\left(\sum_j\int_{t_j}^{t_{j+1}}(f_n-b_s)^2ds\right)\\& = & \mathbb E\left(\sum_j\int_{t_j}^{t_{j+1}}(b_{t_j}-b_s)^2ds\right),
	\end{eqnarray*}	
	and hence, by Proposition \ref{exp:dif:br},
	\begin{eqnarray*}	
		\mathbb E\left(\int_0^t(f_n-b_s)^2ds\right)
			& = & \sum_j\int_{t_j}^{t_{j+1}}(s-t_j)ds\\
			& = & \sum_j\frac{1}{2}(t_{j+1}-t_j)^2 \stackrel{n\to +\infty}{\to} 0.
		\end{eqnarray*}
		Thus, by Proposition \ref{itoprop}, (6),
		\[
		\int_0^tb_sdb_s = \lim_{n\to +\infty}\int_0^t f_ndb_s=\lim_{\Delta t_j\to 0}\sum_jb_{t_j}(b_{t_{j+1}}-b_{t_j}).
		\]
		But, since $b_0=0$, 
		\begin{eqnarray*}
			b_t^2 & = & \sum_j(b_{t_{j+1}}^2-b_{t_j}^2)\\& = & 
			\sum_j(b_{t_{j+1}}-b_{t_j})^2+2\sum_jb_{t_j}(b_{t_{j+1}}-b_{t_j}),
		\end{eqnarray*}
		that is, 
		\[
		\sum_jb_{t_j}(b_{t_{j+1}}-b_{t_j})=\frac{1}{2}b_t^2-\frac{1}{2} \sum_j(b_{t_{j+1}}-b_{t_j})^2.
		\]
		By passing the limit and using Proposition \ref{quadbm} to handle the last term in the right-hand side
		we conclude that 
		\[
		\int_0^tb_sdb_s=\frac{1}{2}b_t^2-\frac{t}{2}.
		\]
		At least formally, we can rewrite this as 
		\[
		db_t^2=2b_tdb_t+dt.
		\]
		Setting $f(x)=x^2$ we have 
		\[
		df(b_t)=f'(b_t)db_t+\frac{1}{2}f''(b_t)dt. 
		\]
		This rather special case of the famous It\^o formula illustrates the appearance of an extra term in the chain rule in the stochastic chain rule.
		In fact, if we interpret Proposition \ref{quadbm} as saying that $db_t^2=dt$, we have
		\[
		df(b_t)=f'(b_t)db_t+\frac{1}{2}f''(b_t)db_t^2, 
		\]
		which means that we must expand up to second order in $db_t$ to obtain the correct version of the chain rule.}\qed
\end{example}

We now prove that It\^o integrals are martingales.

\begin{proposition}\label{mart:ito}
	\label{itomart}If $f\in\mathcal A(0,t)$ consider the process 
	\[
	M_t=\int_0^tf(\rho,\omega)db_\rho(\omega).
	\]
	Then $M_t$ is a martingale. 
\end{proposition}

\begin{proof}
	If $t\leq s$ we have 
	\[
	\mathbb E\left(M_s|\mathbb F_t\right)=\mathbb E\left(M_t|\mathbb F_t\right)+\mathbb E\left(\int_t^sf(\rho,\omega)db_\rho(\omega)\right).
	\]
	Since $M_t$ is $\mathcal F_t$-measurable (Proposition \ref{itoprop}, (4)), we have $\mathbb E\left(M_t|\mathbb F_t\right)=M_t$. Moreover, $\int_t^sf(\rho,\omega)db_\rho(\omega)$ is `independent' of $\mathcal F_t$ in the sense that 
	\begin{equation}
		\label{indep}
		\mathbb E\left(\int_t^sf(\rho,\omega)db_\rho(\omega)\right)=0,
	\end{equation} 
	which completes the proof except for the checking of (\ref{indep}), which needs to be carried out only for $f$ of the type $f=\sum_jf_j{\bf 1}_{[t_j,t_{j+1})}$. In this case, 
	\begin{eqnarray*}
		\mathbb E\left(\int_t^sf(\rho,\omega)db_\rho(\omega)\right)& = & \mathbb E\left(\sum_jf_j(b_{t_{j+1}}-b_{t_j})|\mathcal F_t\right)\\
		& \stackrel{\mathcal F_t\subset\mathcal F_{t_j}+{\rm Prop.} \ref{ceprop}, (5)}{=}&
		\sum_j\mathbb E\left(\mathbb E\left(f_j(b_{t_{j+1}}-b_{t_j})|\mathcal F_{t_j}\right)|\mathcal F_t\right)\\
		& \stackrel{{\rm Prop.} \ref{ceprop}, (2)}{=}&   \sum_j\mathbb E\left(f_j\mathbb E\left((b_{t_{j+1}}-b_{t_j})|\mathcal F_{t_j}\right)|\mathcal F_t\right),
	\end{eqnarray*} 
	and this vanishes because $\mathbb E\left((b_{t_{j+1}}-b_{t_j})|\mathcal F_{t_j}\right)=0$.
\end{proof}

We now discuss a multi-dimensional version of It\^o integral which will suffice for our applications. We first recall from Remark \ref{br:ind} that if $b_t=(b_t^{(1)}, \cdots b^{(n)})$ is Brownian motion in $\mathbb R^n$ then $\{b_t^{(i)}\}_{i=1}^n$ is an independent family of BMs on $\mathbb R$ (and conversely). We will use this to define the integral
\begin{equation}\label{multi:ito}
	\int_S^Tv\,db_t=\int_S^T
	\left(
	\begin{array}{ccc}
		v_{11} & \cdots & v_{1n}\\
		\vdots & \ddots & \vdots \\
		v_{m1} & \cdots & v_{mn} 
	\end{array}
	\right)\left(
	\begin{array}{c}
		db_t^{(1)} \\
		\vdots  \\
		db_t^{(n)}   
	\end{array}
	\right)
\end{equation}
as a process in $\mathbb R^n$ for a suitable $v_{ij}=v_{ij}(t,\omega)$. 

\begin{definition}
	\label{extdefito}
	Let $\mathcal A_{\mathcal H}(S,T)$ be the collection of functions $f:[0,+\infty)\times \Omega\to \mathbb R$ such that
	\begin{enumerate}
		\item $f$ is $\mathcal B\times\mathcal F$-measurable;
		\item there exists a filtration $\mathcal H_t\subset \mathcal F$ such that:
		\begin{itemize}
			\item $b_t$ is a martingale with respect to $\mathcal H_t$ ($b_t$ is BM in $\mathbb R$);	
			\item $f(t,\cdot)$ is adapted to $\mathcal H_t$, $t>0$.
		\end{itemize}
		\item $\mathbb E(\int_S^Tf(t,\omega)^2dt)<+\infty$.
	\end{enumerate}
\end{definition}

Since $\mathcal F_t\subset\mathcal H_t$ and $\mathbb E(b_s-b_t|\mathcal H_t)=0$ we can proceed as before and define
\[
\int_S^Tfdb_t,\quad f\in \mathcal A_{\mathcal H}(S,T),
\] 
which is a martingale (see Proposition \ref{itomart}).  
Coming back to $b=(b^{(1)}, \cdots, b^{(n)})\in\mathbb R^n$, let 
$\mathcal F^{(n)}_t$ be the $\sigma$-algebra generated by $b^{(1)}_{s_1},\cdots, b^{(n)}_{s_n}$, where $s_k\leq t$, $k=1,\cdots,n$. Using the componentwise independence mentioned above, we see that $t<s$ implies that $b^{(k)}_s-b^{(k)}_t\perp\!\!\perp \mathcal F_t^{(n)}$, so that by a previous argument each $b_t^{(k)}$ is a martingale with respect to $\mathcal F^{(n)}_t$. This allows us to define integrals like
\[
\int_S^Tf(t,b_t^{(1)},\cdots,b_t^{(n)})db_t^{(k)}, \quad f\in \mathcal A_{\mathcal F^{(n)}}(S,T), \quad k=1,\cdots,n. 
\]  
Thus, if we set $\mathcal A_{\mathcal F^{(n)}}^{m,n}(S,T)$ to be the collection of all 
$\{v_{ij}\}_{i=1,\cdots,m;j=1,\cdots,n}$ such that $v_{ij}\in\mathcal A_{\mathcal F^{(n)}}(S,T)$
then the multi-dimensional It\^o integral in (\ref{multi:ito}) above is well-defined and has the expected properties (in particular, it is a martingale).

The It\^o integral considered above turns out to be the main ingredient in defining an important class of stochastic processes which, as we shall see, are quite amenable to formal manipulations resembling those available from the ordinary calculus. 

\begin{definition}
	Let $b_t$ be BM in $\mathbb R$ (with probability space $(\Omega,\mathcal F,P)$). Then an {\em It\^o process} (or {\em diffusion}) is a stochastic process in $(\Omega,\mathcal F,P)$ of the type
	\begin{equation}
		\label{itoproc1}
		X_t=X_0+\int_0^tu(s,\omega)ds+\int_0^tv(s,\omega)db_s,\quad v\in\mathcal A_{\mathcal H}(0,t),
	\end{equation}
	for some $\mathcal H$ as in Definition \ref{extdefito}. 
\end{definition} 

Formally, we can rewrite (\ref{itoproc1}) as 
\begin{equation}
	\label{itoproc2}dX_t=udt+vdb_t,
\end{equation}
where $u$ is the {\em drift coefficient} and $v$ is the {\em diffusion coefficient}. 
It\^o formula in (\ref{itoform2}) below shows that reasonable functions of It\^o processes are It\^o processes as well. It provides the correct change of variables formula in the setting of Stochastic Calculus.

\begin{proposition}
	\label{itolemma1}
	If $X_t$ is an It\^o process as in (\ref{itoproc2}) and $g=g(t,\omega)\in C^{2,1}([0,+\infty)\times\mathbb R)$ then $Y_t(t,\omega)=g(t,X_t(\omega))$ is an It\^o process as well. More precisely, 
	\begin{equation}
		\label{itoform}
		dY_t(t,X_t)=\frac{\partial g}{\partial t}(t,X_t)dt+\frac{\partial g}{\partial x}(t,X_t)dX_t+\frac{1}{2}\frac{\partial^2g}{\partial x^2}(t,X_t)dX_t^2,
	\end{equation}
	where in handling the quadratic term $dX_t^2$ we should use the multiplication table
	\begin{table}[H]
		\centering
		\begin{tabular}{c|c|c}
			& $dt$ & $db_t$\\
			\hline
			$dt$ & $0$ & $0$\\
			\hline
			$db_t$ & $0$ & $dt$\\
		\end{tabular}
	\end{table}
	\noindent
	As a consequence,  
	\begin{equation}\label{itoform2}
		dY_t=\left(\frac{\partial g}{\partial t}+u\frac{\partial g}{\partial x}+\frac{1}{2}v^2\frac{\partial^2g}{\partial x^2}\right)dt +v\frac{\partial g}{\partial x}db_t.
	\end{equation}
\end{proposition}

\begin{proof}
	\cite[Theorem 4.1.2]{oksendal2013stochastic}.
\end{proof}

We now discuss the multi-dimensional version of this result. Let $b_t=(b_t^{(1)},\cdots,b_t^{(n)})$ be BM in $\mathbb R^n$, $b_0=0$. We can consider a multi-dimensional It\^o process
\[
dX_t=udt+vdb_t,
\] 
where $X=(X_1,\cdots,X_n)^\top\in{\mathbb R^n}$, $u=(u_1,\cdots,u_n)^\top\in\mathbb R^n$, $u_i\in\mathcal A_{\mathcal F^{(n)}}$, and $v\in\mathcal A^{n,m}_{\mathcal F^{(n)}}$.  
Thus, 
\[
dX_{ti}=u_idt+\sum_{j=1}^mv_{ij}db_t^{(j)},\quad i=1,\cdots,n.
\]
\begin{proposition}
	\label{itomultidim}
	If $X$ is as above and $Y(t,\omega)=g(t,X_t(\omega))$, where $g:[0,+\infty)\times\mathbb R^n\to\mathbb R^p$, then  
	\[
	dY_k=\frac{\partial g_k}{\partial t}dt+\sum_{i=1}^n\frac{\partial g_k}{\partial x_i}dX_i+\frac{1}{2}\sum_{i,j=1}^n\frac{\partial^2g_k}{\partial x_i\partial x_j}dX_idX_j,
	\]
	where 
	in handling the quadratic terms $dX_idX_j$ we should use the multiplication table
	\begin{table}[H]
		\centering
		\begin{tabular}{c|c|c}
			& $dt$ & $db^{(i)}_t$\\
			\hline
			$dt$ & $0$ & $0$\\
			\hline
			$db^{(j)}_t$ & $0$ & $\delta_{ij}dt$\\
		\end{tabular}
	\end{table}
	\noindent
	As a consequence,  
	\begin{equation}\label{fkform2}
		dY_k=\left(\frac{\partial g_k}{\partial t}+\sum_{i=1}^n\frac{\partial g_k}{\partial x_i}u_i+\frac{1}{2}\sum_{i,j=1}^n\underbrace{\sum_{l=1}^mv_{il}v_{jl}}_{=(v^*v)_{ij}}\frac{\partial^2g_k}{\partial x_i\partial x_j}\right)dt+\sum_{i=1}^n\sum_{l=1}^m\frac{\partial g_k}{\partial x_i}v_{il}db^{(l)}.
	\end{equation}
\end{proposition}

\begin{proof}
	\cite[Theorem 4.2.1]{oksendal2013stochastic}.
\end{proof}

\begin{example}\label{geo:br}
(The geometric Brownian)
We assume that $Y_t(t,\omega)=g(t,b_t(\omega))$ in (\ref{itoform}), so that (\ref{itoform2}) becomes 
\begin{equation}\label{itoform:sp}
	dY_t=\left(\frac{\partial g}{\partial t}+\frac{1}{2}\frac{\partial^2g}{\partial x^2}\right)dt +\frac{\partial g}{\partial x}db_t.
\end{equation}
If $Y_t=e^{\sigma b_t}$, $\sigma\neq 0$, we then get 
\begin{equation}\label{exp:bt}
	dY_t=\frac{\sigma^2}{2}Y_tdt+\sigma Y_tdb_t,
\end{equation}
which shows that, in the stochastic setting, the exponential fails to satisfy the self-reproducing property under differentiation.  We may cancel out the annoying first term in the right-hand side above by considering the {\em geometric Brownian process}
\begin{equation}\label{geo:br:def}
	Z_t=e^{(\mu-\frac{\sigma^2}{2})t+\sigma b_t}, \quad \mu\in\mathbb R, 
\end{equation}
which satisfies 
\begin{equation}\label{geo:br:1}
	dZ_t=\mu Z_tdt+\sigma Z_tdb_t,
\end{equation}
or equivalently, 
\begin{equation}\label{geo:br:2}
	\frac{dZ_t}{Z_t}=\mu dt+\sigma db_t. 
\end{equation}
As we shall see in Subsection \ref{black:scholes},
this kind of process plays a central role in the Black-Scholes strategy in Finance.
\qed
\end{example}

\begin{example}\!\!$\bigstar$\label{hist}({Diffusion processes})
We discuss here one of  the most notable motivation behind It\^o's construction of his integral. Roughly, this accomplishment allowed the proper interpretation of solutions of a large class of stochastic differential equations, which in particular led to a pathwise approach to diffusion processes.
We start by recalling that 
a {\em transition function} is a map $\mathfrak R:[0,+\infty)\times\mathbb R^n\times \mathcal B_n\to[0,1]$ such that:
\begin{itemize}
	\item $x\mapsto \mathfrak R(t,x,B)$ is measurable;
	\item $B\mapsto \mathfrak R(t,x,B)$ is a probability measure on $\mathbb R^n$;
	\item $\mathfrak R(0,x,\cdot)=\delta_x$;
	\item The {\em Chapman-Kolmogorov equation} holds:
	\begin{equation}\label{ckol}
		\mathfrak R(t+s,x,B)=\int_{\mathbb R^n} \mathfrak R(t,x,dy)\mathfrak R(s,y,B).
	\end{equation}
\end{itemize}
An application of Theorem \ref{extkolm} guarantees the existence of a measurable space $(\Omega,\mathcal F)$ and, for each $x\in\mathbb R^n$, a probability measure $\mathbb P^x$ on $(\Omega,\mathcal F)$ and a stochastic process $X_t^x:(\Omega,\mathcal F,\mathbb P^x)\to\mathbb R^n$ such that $\mathfrak R(t,x,B)=\mathbb P^x(X_t^x\in B)$.
Equivalently, $\mathfrak R(t,x,\cdot)=\mathbb P^x_{X_t^x}$.
In particular, $\mathbb P^x_{X_0^x}=\delta_x$. Moreover, the following {\em Markov property} holds:
\begin{equation}\label{mark}
	\mathbb E^x(f(X^x_{t+s})|\mathcal F^X_s)=\mathbb E^x(f(X_t^x)),\quad  P^x\,\,{\rm a.s.},
\end{equation}
for any $x\in\mathbb R^n$ and any $f$ as above.
Intuitively, this means $X_t^x$ is memoryless.
Attached to any $\mathfrak R$ as above is the associated semigroup $t\mapsto \mathfrak B_t$ given by
\[
(\mathfrak P_tf)(x)=\int_{\mathbb R^n} f(y)\mathfrak R(t,x,dy),
\]
for $f$ a bounded, measurable function on $\mathbb R^n$. Notice that $(\mathfrak P_tf)(x)=\mathbb E^xf(X^x_t)$.
Now,
any {\em Markov process} $X_t^x$ as above has an {\em infinitesimal generator} $L$, which is a linear operator defined on the space  of all functions $f$ such that
\[
\lim_{t\to 0}\frac{\mathfrak P_tf-f}{t}
\]  
exists. We then define, for any such $f$,
\begin{equation}\label{semgen}
	(Lf)(x)=\lim_{t\to 0}\frac{(\mathfrak P_tf)(x)-f(x)}{t}.
\end{equation}
Under certain regularity assumptions, the generator is a {\em diffusion operator}, that is, 
\begin{equation}\label{gen:inf:mark}
	(Lf)(x)=\frac{1}{2}\sum_{ij}a_{ij}(x)\frac{\partial^2f}{\partial x_i\partial x_j}+\sum_iu_i(x)\frac{\partial f}{\partial x_i},\quad x\in\mathbb R^n,
\end{equation}
where $a$ is a symmetric, non-negative matrix. The  problem now is how to recover $X_t^x$ (in law) starting from $L$ (or, more precisely, from the coefficients $\{a_{ij}\}$ and $\{u_i\}$ defining it). It turns out that It\^o calculus may be used to solve this problem as follows. Write $a=vv^\top$ and form the stochastic differential equation
\begin{equation}\label{sde:dif}
dX_t=u(X_t)dt+v(X_t)db_t.
\end{equation}
Under mild conditions on the coefficients, it is shown that this equation has a unique solution $X_t^x$ with $X^x_0=x$. This means that 
\[
X_t^x=x+\int_0^tu(X^x_s)ds+\int_0^tv(X^x_s)db_s,
\]
so a notion of stochastic integral is required here in order to properly interpret the last term above. 
It is now immediate to check that $X_t^x$ solves the problem in the sense that (\ref{ckol}), (\ref{mark}) and (\ref{semgen}) are satisfied if we set $\mathfrak R(t,x,B)=P(X_t^x\in B)$, where $P$ is Wiener measure. Here we only check that (\ref{semgen}) holds. From It\^o formula (\ref{fkform2}), for any $f:\mathbb R^n\to\mathbb R$ we have
\[
f(X_t^x)=f(x)+\int^t_0(Lf)(X^x_s)ds+M_t^f,
\] 
where $M_t^f$ is a martingale. By taking expectation we see that
\[
(\mathfrak P_tf)(x)  =  \mathbb E^x(f(X^x_t))=
f(x)+\mathbb E^x\left(\int^t_0(Lf)(X^x_s)ds\right),
\]
and (\ref{semgen}) follows. 
In connection with this remarkable correspondence between a diffusion operator $L$ and a diffusion process $X_t^x$ induced by stochastic differential equations, we add the following comments:
\begin{enumerate}
	\item 
 The resulting process is completely characterized by the fact that for any $f$ the martingale
\[
M_t^{f,x}=f(X_t^x)-f(x)-\int_0^t(Lf)(X_s^x)ds
\]
has quadratic variation given by
\[
\langle M^{f,x}\rangle_t^{(2)}=\int_0^t\left(\sum_{ij}a_{ij}\frac{\partial f}{\partial x^i}\frac{\partial f}{\partial x^j}\right)(X_s^x)ds. 
\]
We then say that $X_t^x$ is the {\em diffusion process} driven by $L$. To see that we are in the right track, let us take $L=\frac{1}{2}\Delta$ and set $\mathsf X^i=M^{x_i,x}$ for simplicity. It follows that  
\[
	\langle \mathsf X^i, \mathsf X^j\rangle_t^{(2)}  := 
	\frac{1}{2}\left(	\langle \mathsf X^i+\mathsf X^j\rangle_t^{(2)}-	\langle \mathsf X^i\rangle_t^{(2)}-	\langle \mathsf X^j\rangle_t^{(2)}\right)
	 =  \delta_{ij}t,
\]
so a celebrated result due to P. L\'evy \cite[Theorem 3.16]{karatzas2012brownian} implies  that $X_t=b_t$. Thus, as expected, BM is the diffusion process driven by the Laplacian $\frac{1}{2}\Delta$. 
\item 
The diffusion process may in fact be defined directly in terms of the coefficients of the operator $L$. Indeed, taking $f(x)=x^i$ in (\ref{gen:inf:mark}), we obtain
\begin{equation}\label{interm}
	N_t^i:=X_t^i-x_i-\int_0^t u_i(X_s)\,ds,
\end{equation}
which is a martingale with quadratic covariation
\[
\langle N^i,N^j\rangle_t^{(2)}=\int_0^t a_{ij}(X_s)\,ds.
\]
In particular, if $L=\sum_i u_i\partial_i$ is a vector field, then $a_{ij}\equiv 0$, so that $\langle N^i,N^j\rangle_t=0$ and hence $N_t^i=0$ for all $i$. In this case, (\ref{interm}) reduces to
\[
X_t^i=x_i+\int_0^t u_i(X_s)\,ds,
\]
showing that $L$ integrates to a deterministic flow. In this way we recover the classical result on the integration of vector fields.
From this perspective, classical calculus may be viewed as providing a mechanism for integrating vector fields, leading to deterministic dynamical systems, while It\^o calculus extends this framework to diffusion operators, giving rise to random dynamical systems.
\end{enumerate}
For modern accounts on the theory of diffusion processes we refer to \cite{baudoin2014diffusion,bakry2014analysis}\qed
\end{example}

\begin{example}\!\!$\bigstar$\label{fp:lang}
	(Fokker-Planck equation and Langevin dynamics)
	If 
\[
X_t^x=X_0+\int_0^tu(X_s)ds+\int_0^tv(X_s)db_s
\]	
 is a  solution of (\ref{sde:dif}) then arguing as in Example 	
	\ref{hist}
we have, for $f:\mathbb R^n\to \mathbb R$, 
\[
 \mathbb E(f(X_t))=
f(X_0)+\mathbb E\left(\int^t_0(Lf)(X_s)ds\right),
\]	
where $L$ is given by (\ref{gen:inf:mark}) with $a=vv^\top$.
If $X_t$ admits a density $\psi_t$ in the sense that $P_{X_t}=\psi_t d{ x}$ then, upon derivation with respect to $t$, we obtain
\[
\int_{\mathbb R^n}f\frac{\partial \psi_t}{\partial t}d{x}=
\int_{\mathbb R^n}\left(
\frac{1}{2}\sum_{ij}a_{ij}\frac{\partial^2f}{\partial x_i\partial x_j}+\sum_iu_i\frac{\partial f}{\partial x_i}
\right)\psi_td{x}.
\]
If we further assume that $f$ is smooth and compactly supported then no boundary term survives after integrating the right-hand side by parts, so we get
\[
\int_{\mathbb R^n}f\frac{\partial \psi_t}{\partial t}d{x}=
\int_{\mathbb R^n}f\left(
\frac{1}{2}\sum_{ij}\frac{\partial^2(a_{ij}\psi_t)}{\partial x_i\partial x_j}-\sum_i\frac{\partial (u_i\psi_t)}{\partial x_i}
\right)d{x},
\]
and since $f$ is arbitrary we end up with the {\em Fokker-Planck equation}
\begin{equation}\label{eq:fp:0}
\frac{\partial \psi_t}{\partial t}=\frac{1}{2}\sum_{ij}\frac{\partial^2(a_{ij}\psi_t)}{\partial x_i\partial x_j}-\sum_i\frac{\partial (u_i\psi_t)}{\partial x_i},
\end{equation}
which governs how the density varies in time. 
Under these conditions, a solution $\psi_t$ of (\ref{eq:fp:0}) should be viewed as a ``macroscopic'' manifestation of a corresponding solution $X_t$ of (\ref{sde:dif}), which evolves at the level of sample paths and therefore is ``microscopic'' in nature. 
As an illustrative example, 
consider the {\em Langevin diffusion}
\begin{equation}\label{lang:diff}
dX_t=-\nabla V(X_t)dt +\sqrt{2}db_t,
\end{equation}
where $V:\mathbb R^n\to\mathbb R$ is a potential function, so the corresponding generator is the {\em drifted Laplacian}
\[
L=\Delta -\nabla V^\top\nabla,
\]
and the density obeys 
\begin{equation}\label{eq:fp:l}
\frac{\partial \psi_t}{\partial t}=\Delta\psi_t+{\rm div}\left(\psi_t\nabla V\right). 
\end{equation}
If we assume that
\[
Z:=\int_{\mathbb R^n}e^{-V(x)}dx<+\infty
\]
then a direct calculation shows that 
\[
\psi_{\rm G}(x):=Z^{-1}e^{-V(x)}
\]
is a stationary (i.e.\ time-independent) solution of (\ref{eq:fp:l}). In other words, if we solve (\ref{lang:diff}) with initial condition $X_0$ such that $P_{X_0}=\psi_{\rm G} dx$ then $P_{X_t}=\psi_{\rm G} dx$ for any $t\geq 0$, so
the Gibbs measure $\psi_{\rm G}\,dx$ is an invariant probability measure for this Langevin dynamics.
Since the set of all solutions of (\ref{lang:diff}) may be viewed as a random dynamical system, 
it is natural to ask for conditions on $V$ under which every such solution converges, in a suitable sense, to this Gibbs equilibrium as $t\to+\infty$. A standard sufficient condition is the uniform convexity assumption
\[
\nabla^2V(x)\ge \rho\, {\rm Id}_n
\qquad\text{for all }x\in\mathbb R^n,
\]
for some constant $\rho>0$, in the sense of quadratic forms. 
In the Bakry–\'Emery framework, as developed in \cite[Chapter 5]{bakry2014analysis}, this condition is interpreted as a curvature-dimension lower bound for the generator of the Langevin diffusion. This curvature bound implies a logarithmic Sobolev inequality for the associated Gibbs measure, which in turn yields exponential convergence to equilibrium.
This mechanism is particularly transparent in the Gaussian case corresponding to the Ornstein–Uhlenbeck process, obtained by taking $V(x)=\|x\|^2/2$, where all steps in the chain above can be verified explicitly; see \cite{gentil_2014} for a detailed exposition starting from this model case.
\qed
\end{example}

With the basics of  It\^o calculus at hand, we present in
the rest of this Appendix some of its most glamorous applications. 

\subsection{The Gaussian concentration inequality (again)}\label{gauss:conc:sec}
We start by providing here an elegant proof of the optimal version of the Gaussian concentration inequality (\ref{gauss:conc:ineq}) which is attributed to B. Maurey in \cite[Chapter 2]{pisier2006probabilistic} and relies on the full power of It\^o calculus developed in the previous section. 

For $0\leq t\leq 1$ we consider the ``reversed'' heat semigroup $P_t=e^{\frac{1}{2}(1-t)\Delta}$, so that for any (smooth and Lipschitz) $F:\mathbb R^k\to\mathbb R$ there holds
\[
\frac{\partial}{\partial t}(P_tF)+\frac{1}{2}\Delta (P_tF)=0.
\] 
We now use It\^o formula (\ref{fkform2}) with $Y(t,\cdot)=(P_tF)(b_t)$, where $b_t$ is a standard BM in $\mathbb R^k$ (recall that $b_t-b_{t'}\sim\mathcal N(\vec{0},(t-t'){\rm Id}_k)$, $t'<t$). Since $u=0$ and $v_{ij}=\delta_{ij}$ it simplifies to
\[
dY_t=\langle (\nabla P_tF)(b_t),db_t\rangle,
\]
and integrating this from $t=0$ to $t=1$,
\begin{eqnarray*}
	F(b_1)
	& = & (e^{\frac{1}{2}\Delta}F)(0)+\int_0^1\langle (\nabla P_tF)(b_t),db_t\rangle\\
	& = & \mathbb E(F(b_1))+\int_0^1\langle (\nabla P_tF)(b_t),db_t\rangle,
\end{eqnarray*}
where we used Propositions \ref{martcons} and \ref{itomart} in the last step. We may now  adapt the Cram\'er-Chernoff method in Section \ref{conc:ineq:appl} 
to this setting: for $\tau>0$ and $w\geq 0$, 
\begin{equation}\label{tail:est:p}
	P(|F(b_1)-\mathbb E(F(b_1))|>\tau)\leq 2 e^{-w\tau}{\mathbb E\left(e^{w\int_0^1\langle (\nabla P_tF)(b_t),db_t\rangle}\right)},
\end{equation}
so it remains to estimate the expectation in the right-hand side.

The key observation at this point is that ${\rm Lip}(P_tF)={\rm Lip}(F)$ and hence $|\nabla P_tF)|\leq {\rm Lip}(F)$ a.s. 
Now let $\pi=\{t_0=0<t_1<\cdots<t_n=1\}$ be a partition of the interval $[0,1]$ with $|\pi|=\max_l|t_l-t_{l-1}|$ its width. As $|\pi|\to 0$ the It\^o integral within the expectation, by its very definition, may be arbitrarily approximated (say, in probability) by  $S_n$, where 
\[
S_j=\sum_{l=1}^j\langle V_l,b_{t_l}-b_{t_{l-1}}\rangle, \quad 1\leq j\leq n,
\]
$V_l=(\nabla P_{t_{l-1}}F)(b_{t_{l-1}})$ is $\mathcal F_{t_{l-1}}$-measurable (where $\{\mathcal F_t\}$ is the filtration associated to $b_t$) and satisfies $|V_l|\leq {\rm Lip}(F)$. 
We have
\[
S_j=S_{j-1}+\langle V_j,b_{t_j}-b_{t_{j-1}}\rangle,
\]
a decomposition into independent factors, and 
since the inner product follows the normal $\mathcal N(0,|V_j|^2(t_j-t_{j-1}))$ by Proposition  (\ref{norm:space:4}) (3), assuming of course that $V_j\neq \vec{0}$, we obtain
\begin{eqnarray*}
	\mathbb E(e^{wS_j})
	& = & 
	\mathbb E\left(e^{wS_{j-1}}\right)
	\mathbb E\left(e^{w\left\langle {V_j},b_{t_j}-b_{t_{j-1}}\right\rangle}\right)\\
	&\stackrel{(\ref{mgf:normal:n})}{=} & 
	\mathbb E\left(e^{wS_{j-1}}\right)e^{\frac{1}{2}w^2|V_j|^2(t_j-t_{j-1})}\\
	& \leq &
	\mathbb E\left(e^{wS_{j-1}}\right)e^{\frac{1}{2}w^2{\rm Lip}(f)^2(t_j-t_{j-1})}.
\end{eqnarray*}	
Note that this obviously remains true if $V_j=\vec{0}$. In any case,
if we iterate this starting with $j=n$ we get
\[
\mathbb E(e^{wS_n})\leq e^{\frac{1}{2}w^2{\rm Lip}(f)^2},
\]
where the right-hand side, remarkably enough, does not depend on $\pi$.
Passing the limit as $|\pi|\to 0$ on the left-hand side we thus obtain
\[
\mathbb E\left(e^{w\int_0^1\langle (\nabla P_tF)(b_t),db_t\rangle}\right)\leq e^{\frac{1}{2}w^2{\rm Lip}(f)^2},
\]
which is the same as saying that 
\begin{equation}\label{sharp:gauss}
	\int_0^1\langle (\nabla P_tF)(b_t),db_t\rangle\in {\bf{\mathsf{SubG}}}({\rm Lip}(f)).
\end{equation}
Leading this to 
(\ref{tail:est:p}) we get
\[
P(|F(b_1)-\mathbb E(F(b_1))|>\tau)\leq 2e^{\frac{1}{2}w^2{{\rm Lip}(F)}^2-w\tau},
\]
so if we  minimize the right-hand side over $w\geq 0$ we find that
\[
P(|F(b_1)-\mathbb E(F(b_1))|>\tau)\leq 2 e^{-\frac{\tau^2}{2{{\rm Lip}(F)}^2}}, \quad \tau>0,
\]
which is equivalent to (\ref{gauss:conc:ineq}) with the optimal constant $C=1/2$ because the normal random vector appearing there has been chosen so that $X\sim \mathcal N(\vec{0},{\rm Id}_k)\sim b_1$.

\subsection{The Feynman-Kac formula and the path integral representation of the heat kernel}\label{fk:sec} 

Let us consider 
\[
g(t,X_t)=e^{-\int_0^tV(X_t)dt}w(T-t,X_t), \quad 0\leq t\leq T,
\]
where $V=V(x)$, $x\in\mathbb R^n$, is a (well-behaved) potential function and  $X_t$ is an It\^o diffusion as in (\ref{itoproc2}): 
\[
dX_t=u(X_t)dt+v(X_t)db_t,
\]
where $b_t$ is BM in $\mathbb R^n$.
We compute that 
\[
\frac{\partial g}{\partial t}=-e^{-\int_0^tV(X_t)dt}\left(
V(X_t)w(T-t,X_t)+\frac{\partial w}{\partial t}(T-t,X_t),
\right),
\] 
\[
\frac{\partial g}{\partial x_i}=e^{-\int_0^tV(X_t)dt}\frac{\partial w}{\partial x_i}(T-t, X_t),
\]
and 
\[
\frac{\partial^2g}{\partial x_i\partial x_j}=e^{-\int_0^tV(X_t)dt}\frac{\partial^2w}{\partial x_i\partial x_j}(T-t,X_t), 
\]
so that It\^o formula in (\ref{fkform2}) applies to give 
\begin{eqnarray*}
	dg & = & 
	e^{-\int_0^tV(X_t)dt}\left(-\frac{\partial w}{\partial t}(T-t,X_t)+\sum_iu_i\frac{\partial w}{\partial x_i}(T-t,X_t)-\right.\\
	& & \quad \left.-V(X_t)w(T-t,X_t)+\frac{1}{2}\mathcal L'w(T-t,X_t)
	\right)dt\\
	& & \quad\quad +e^{-\int_0^tV(X_t)dt}\sum_i\frac{\partial g}{\partial x_i}db_t^{(i)},
\end{eqnarray*}
where 
\[
\mathcal L'w=\sum_{ij}(v^*v)_{ij}\frac{\partial^2w}{\partial x_i\partial x_j}. 
\]
In this way we obtain a remarkable stochastic (or path integral) representation  of solutions of certain heat-type equations.

\begin{proposition} (Feynman-Kac formula I)
	\label{fkform3}
	If $w=w(t,x)$ satisfies the heat-type equation
	\[
	\left\{
	\begin{array}{rcl}
		\frac{\partial w}{\partial t} & = & \frac{1}{2}\mathcal L' w+\langle u,\nabla w\rangle -Vw\\
		w(0,x) & = & f(x)
	\end{array}
	\right.
	\]
	then the following holds:
	\begin{equation}
		\label{fkform4}
		w(t,x_0)=\mathbb E_{x_0}\left(e^{-\int_0^tV(X_t)dt}f(X_t)\right),
	\end{equation} 
where $\mathbb E_{x_0}$ refers to the law $P_{x_0}$ of BM in $\mathbb R^n$ starting at $x_0$.	
\end{proposition}

\begin{proof}
	From the computation above,
	\[
	dg=e^{-\int_0^tV(X_t)dt}\sum_i\frac{\partial g}{\partial x_i}db_t^{(i)},
	\]
	which shows that the process $g(t,X_t)$ is a martingale (with respect to $\mathcal F^{(n)}$); see Proposition \ref{itomart}. Since
	\[
	\mathbb E_{x_0}\left(g(t,X_t)\right)|_{t=0}=\mathbb E_{x_0}\left(w(T,X_0)\right)=w(T,x_0),
	\] 
	and 
	\[
	\mathbb E_{x_0}\left(g(t,X_t)\right)|_{t=T}=\mathbb E_{x_0}\left(e^{-\int_0^tV(X_s)ds}w(0,X_T)\right)=\mathbb E_{x_0}\left(e^{-\int_0^tV(X_s)ds}f(X_T)\right), 
	\]
	the result follows in view of Proposition \ref{martcons} and the fact that $T$ is arbitrary.
\end{proof}

\begin{corollary}
	\label{semidom}
	(Exponential control)
	Under the conditions above, if $|f|\leq M$ and $V\geq c$, $c\in\mathbb R$, then 
	\[
	|u(t,x_0)|\leq Me^{-ct}.
	\]
\end{corollary}

An important special case of Proposition \ref{fkform4} occurs when $u=0$ and $v_{ij}=\delta_{ij}$, so that 
\[
\mathcal L'=\frac{1}{2}\Delta,
\] 
where $\Delta$ is the Laplacian. We then see that any solution of 
\begin{equation}\label{init:val}
\left\{
\begin{array}{rcl}
	\frac{\partial w}{\partial t} & = & \frac{1}{2}\Delta w -Vw\\
	w(0,x) & = & f(x)
\end{array}
\right.
\end{equation}
satisfies
\begin{equation}
	\label{fkform5}
	w(t,x)=\mathbb E_{x}\left(e^{-\int_0^tV(b_\tau)d\tau}f(b_t)\right).
\end{equation} 
On the other hand, we know from Analysis \cite{pazy2012semigroups} that this can be rewritten as
\begin{equation}\label{exp:heat:w}
w(t,x)=(e^{t\mathcal L}f)(x)=\int_{\mathbb R^n}K_{\mathcal L}(t;x,y)f(y)dy,
\end{equation}
where $e^{t\mathcal L}$ is the heat semigroup generated by $\mathcal L=\frac{1}{2}\Delta-V$ and 
$K_{\mathcal L}$ is the associated {\em heat kernel}, i.e. $K_{\mathcal L}$ satisfies
\begin{equation}\label{heat:fund:sol}
	\left\{
	\begin{array}{rcl}
		\frac{\partial K_{\mathcal L}}{\partial t} & = & \frac{1}{2}\Delta K_{\mathcal L} - V K_{\mathcal L} \\
		K_{\mathcal L}(0;x,y) & = & \delta(x-y)
	\end{array}
	\right.
\end{equation}
This suggests the existence of a 
stochastic representation for  $K_{\mathcal L}$, thus pointing toward a version of the
Feynman-Kac formula working at the more fundamental level of heat kernels.

To find this representation we fix $t>0$ and consider the process $\{B_s\}_{0\leq s<t}$ satisfying 
\begin{equation}\label{bb:sde}
dB_s=db_s-\frac{B_s-y}{t-s}ds, \quad B_0=x,
\end{equation}
or equivalently, 
\begin{equation}\label{bb:sde:i}
B_s=x+b_s-\int_0^s\frac{B_\tau-y}{t-\tau}d\tau.
\end{equation}
We will now show that the law of this process can be expressed in terms of $K_{\frac{1}{2}\Delta}$,  the heat kernel of the Laplacian $\frac{1}{2}\Delta$,  and $P_x$, the law of BM starting at $x$. 

We first note that  the 
discussion above gives 
\begin{equation}\label{exp:heat:l}
(e^{\frac{1}{2}\tau\Delta}f)(x)=\int_{\mathbb R^n}K_{\frac{1}{2}\Delta}(\tau;x,y)f(y)dy=\mathbb E_{x}\left(f(b_\tau)\right), \quad \tau\geq 0.
\end{equation}
In particular, 
\begin{equation}\label{int:1:K}
\int_{\mathbb R^n}K_{\frac{1}{2}\Delta}(t;x,y)dy=1,
\end{equation}
which also follows from Proposition \ref{welldef} because, as is well-known,
\begin{equation}\label{heat:ker:euc}
K_{\frac{1}{2}\Delta}(t;x,y)=(2\pi t)^{-n/2}e^{-|x-y|^2/2t},
\end{equation}
From this we see that 
\begin{equation}\label{eq:der:hk}
\nabla_x\ln K_{\frac{1}{2}\Delta}(t;x,y)=-\frac{x-y}{t},
\end{equation}
and hence
\begin{equation}\label{sdebb}
	dB_s=db_s+\nabla_x\ln K_{\frac{1}{2}\Delta}(t-s;B_s,y)ds, \quad s<t.
\end{equation}
Thus, the Brownian bridge $B_s$ is just the Brownian motion $b_s$ with an added
drift involving the logarithmic derivative of $K_{\frac{1}{2}\Delta}$. We note however that the drift is
singular at $s = t$. Fortunately, careful first order estimates of  $K_{\frac{1}{2}\Delta}$ \cite[Section 5.5]{hsu2002stochastic} allow us to bypass
this difficulty and confirm not only that this is well defined for $s=t$ but also  that $B_s\to y$ as $s\to t$\footnote{See Remark \ref{bb:0:b:sde} for a more pedestrian approach to this key extension problem, which avoids the consideration of heat kernel estimates.}. Thus, we call $\{B_s\}_{0\leq s\leq t}$ the {\em Brownian bridge} connecting $x$ to $y$ with lifetime $t$. 

We should think of $B_s$ as a process on the {\em bridge space} $C_{t;x,y}\subset C_{x}$ of all Brownian paths starting at $x$ and conditioned to hit $y$ at time $t$; cf. Remark \ref{wiener:sp}. To find the Radon-Nikodym derivative of the law of $B_s$ with respect to $P_x$ we first note that 
\[
\frac{\partial}{\partial t}\ln K_{\frac{1}{2}\Delta}=\frac{1}{2}\Delta\ln K_{\frac{1}{2}\Delta}+\frac{1}{2}\|\nabla\ln K_{\frac{1}{2}\Delta}\|^2,
\] 
so if we apply It\^o formula to 
\[
E_s=\ln\frac{K_{\frac{1}{2}\Delta}(t-s;B_s,y)}{K_{\frac{1}{2}\Delta}(t;x,y)}
\]
we find that 
\[
dE_s=\langle F_s,dB_s\rangle-\frac{1}{2}\|F_s\|^2ds,
\]
where 
\begin{equation}\label{def:fs}
F_s:=\nabla_xK_{\frac{1}{2}\Delta}(t-s;B_s,y)\stackrel{(\ref{eq:der:hk})}{=}
-\frac{B_s-y}{t-s}.
\end{equation}
Now define a measure $Q$ in $C_{t:x,y}$ by 
\[
\frac{dQ}{dP_x}|_{\mathcal F_s}=\exp\left(\int_0^s\langle F_\tau,db_\tau\rangle-\frac{1}{2}\int_0^s\|F_\tau\|^2d\tau\right).
\]
By Girsanov's theorem \cite[Theorem 8.6.4]{oksendal2013stochastic}, under $Q$ the process
\[
B_s-\int_0^sF_\tau d\tau\stackrel{(\ref{def:fs})}{=}B_s+\int_0^s\frac{B_\tau-y}{t-\tau}d\tau
\]
is a BM. Thus, comparing with (\ref{bb:sde:i}) we see that $P_{t;x,y}:=Q$ is the law of the Brownian bridge $B_s$ and there holds
\begin{equation}\label{exp:law:brid}
\frac{dP_{t;x,y}}{dP_x}|_{\mathcal F_s}=\frac{K_{\frac{1}{2}\Delta}(t-s;b_s,y)}{K_{\frac{1}{2}\Delta}(t;x,y)}.
\end{equation}
With these preliminaries at hand, we will be able to provide a path integral representation for $K_{\mathcal L}$.

\begin{proposition}\label{fkac2}(Feynman-Kac formula II)
	One has 
	\begin{equation}\label{fkac2:rep:int}
	K_{\mathcal L}(t;x,y)=K_{\frac{1}{2}\Delta}(t;x,y)\mathbb E_{t;x,y}\left(e^{-\int_0^tV(B_\tau)d\tau}\right).
	\end{equation}
	In other words, if we define the {\em conditional Wiener measure} on $C_{t;x,y}$ by
	\begin{equation}\label{cond:wien}
	d\mu_{t;x,y}=K_{\frac{1}{2}\Delta}(t;x,y)P_{t;x,y}
	\end{equation}
	then
	\begin{equation}\label{fkac2:eq}
	K_{\mathcal L}(t;x,y)=\int_{C_{t;x,y}}e^{-\int_0^tV(B_\tau)d\tau} d\mu_{t;x,y}. 
	\end{equation}
\end{proposition}

\begin{proof}
	First we have from (\ref{exp:heat:l}) with $\tau=0$ that
	\[
	e^{-\int_0^tV(b_\tau)d\tau}f(b_t)=\int_{\mathbb R^n}K_{\frac{1}{2}\Delta}(0;b_t,y)e^{-\int_0^tV(b_\tau)d\tau}f(y)dy,
	\]
	and taking expectation we get
	\begin{eqnarray*}
		\mathbb E_x\left(e^{-\int_0^tV(b_\tau)d\tau}f(b_t)\right) 
		& = & \int_{\mathbb R^n}K_{\frac{1}{2}\Delta}(t;x,y)\mathbb E_x\left(\frac{K_{\frac{1}{2}\Delta}(0;b_t,y)}{K_{\frac{1}{2}\Delta}(t;x,y)}e^{-\int_0^tV(b_\tau)d\tau}\right)f(y)
		dy \\
		& \stackrel{(\ref{exp:law:brid})\,{\rm with}\, s=t}{=} & 
		\int_{\mathbb R^n}K_{\frac{1}{2}\Delta}(t;x,y)\mathbb E_{t;x,y}\left(e^{-\int_0^tV(B_\tau)d\tau}\right)f(y)dy.
	\end{eqnarray*}
	On the other hand, we know from (\ref{fkform5}) and (\ref{exp:heat:w}) that 
	\[
	\mathbb E_x\left(e^{-\int_0^tV(b_\tau)d\tau}f(b_t)\right)=\int_{\mathbb R^n}K_{\mathcal L}(t;x,y)f(y)dy. 
	\]
	Since $f$ is arbitrary, the result follows. 
\end{proof}

\begin{corollary}\label{pos:kern}(Positivity improving property of the heat flow)
	If the potential function is uniformly bounded from below, $V\geq c$, then $K_{\mathcal L}$ is strictly positive for all $t>0$. In particular, any solution of (\ref{init:val}) with $f\geq 0$ and $f \not\equiv 0$ remains strictly positive for all $t>0$. 
	\end{corollary}
	
\begin{proof}
Since
\[
0< e^{-\int_0^tV(B_\tau)d\tau}\leq e^{-ct}
\]
for every path for which the integral is finite, we see that the random variable inside the expectation in (\ref{fkac2:rep:int}) is strictly positive and uniformly bounded. Therefore,  its expectation is finite and strictly positive, so that the strict positivity of $K_{\mathcal L}$ follows from (\ref{fkac2:rep:int}) and (\ref{heat:ker:euc}). Finally, the last assertion is a consequence of (\ref{exp:heat:w}).	
	\end{proof}	
	
\begin{remark}\!\!$\bigstar$ \label{prop:ground:positive}(Uniqueness of the ground state) If, under the conditions of Corollary \ref{pos:kern}, we assume further that 
	the bottom of the spectrum of $\mathcal L$
is an eigenvalue, say $E_0$, then the corresponding ground state eigenspace is one-dimensional, and the associated eigenfunction may be chosen strictly positive.
Indeed, 
let $\psi$ be such a ground state, so that
	$
	\mathcal L\psi+E_0\psi=0
	$
	and hence
	\[
	e^{t\mathcal L}\psi=e^{-tE_0}\psi,
	\]
	which means  that $\psi$ is an eigenfunction of $e^{t\mathcal L}$ associated to its maximal eigenvalue $e^{-tE_0}$. In particular, by the spectral theorem, the operator norm of $e^{t\mathcal L}$ is
	\[
	\|e^{t\mathcal L}\|_{\rm op}=e^{-tE_0}. 
	\]
	Now, it follows from Corollary \ref{pos:kern} that 
	\begin{eqnarray*}
		\left|\langle e^{t\mathcal L}\psi,\psi\rangle_{L^2}\right|
		& = &
		\left|
		\int_{\mathbb R^n}\int_{\mathbb R^n}
		K_{\mathcal L}(t;x,y)\psi(y){\psi(x)}
		\,dy\,dx
		\right| \\
		& \leq &
		\int_{\mathbb R^n}\int_{\mathbb R^n}
		K_{\mathcal L}(t;x,y)|\psi(y)|\,|\psi(x)|
		\,dy\,dx,
	\end{eqnarray*}
	and hence, 
	\begin{equation}\label{eq:holds}
			\left|\langle e^{t\mathcal L}\psi,\psi\rangle_{L^2}\right|\leq 
	\langle e^{t\mathcal L}|\psi|,|\psi|\rangle_{L^2}. 		
		\end{equation}
	On the other hand,
	\[
	\langle e^{t\mathcal L}\psi,\psi\rangle_{L^2}
	=
	e^{-tE_0}\|\psi\|_{L^2}^2,
	\]
	while
	\[
	\langle e^{t\mathcal L}|\psi|,|\psi|\rangle_{L^2}
	\leq
	\|e^{t\mathcal L}\|_{\rm op}\,\|\psi\|_{L^2}^2
	=
	e^{-tE_0}\|\psi\|_{L^2}^2,
	\]
	so equality  holds in (\ref{eq:holds}) and therefore $|\psi|$ is also a ground state. It follows that $\psi^+:=(\psi+|\psi|)/2\geq 0$ is a ground state and  since 
		\[
	\psi^+(x)
	=
	e^{tE_0}
	\int_{\mathbb R^n}
	K_{\mathcal L}(t;x,y)\psi^+(y)\,dy,
	\]
	we may again use Corollary \ref{pos:kern} to see that $\psi^+$ is strictly positive, which allows us to conclude that {\em any} ground state vanishes nowhere. From this we easily deduce the uniqueness of $\psi$ (up to a scaling), for 
	if the ground state eigenspace had dimension greater than one, we could choose a nontrivial ground state
	 $L^2$-orthogonal to $\psi$ which would necessarily vanish somewhere, a contradiction. This beautiful argument is a particular instance of the theory of positivity improving semigroups; see \cite[Section XIII.12]{reed1978iv}. 
	 In contrast to the probabilistic approach adopted here, where Corollary \ref{pos:kern} is derived from the Feynman-Kac representation for $K_{\mathcal L}$, the treatment there is functional-analytic and based primarily on the Trotter product formula.
	\qed
\end{remark}

\begin{remark}\label{corfk2}\!\!$\bigstar$ (The Laplacian heat kernel as a transition probability density) The conditioned Wiener measure $d\mu_{t;x,y}$ in (\ref{cond:wien}) arises from a disintegration of the Wiener measure: there holds  
\[
\int_{\mathbb R^n}\left(\int_{C_{t;x,y}}F(\omega) d\mu_{t;x,y}(\omega)\right)dy=\int_{C_x}F(\omega)dP_x(\omega),
\]
where $F$ varies over the set all bounded measurable functions on  $(C_x,P_x)$, the Wiener space starting at $x$; cf.\!\! Remark \ref{wiener:sp}. By taking $F\equiv 1$ we thus see that 
\[
\int_{\mathbb R^n}\left(\int_{C_{t;x,y}} d\mu_{t;x,y}(\omega)\right)dy=1,
\]
which is just a restatement of (\ref{int:1:K}), as it follows either from (\ref{cond:wien}) or from (\ref{fkac2:eq}) with $V\equiv 1$ that 
\[
K_{\frac{1}{2}\Delta}(t;x,y)=\int_{C_{t;x,y}} d\mu_{t;x,y}(\omega),
\]	
the total measure of the Brownian bridge $C_{t;x,y}$ endowed with $d\mu_{t;x,y}$. It then follows that:
\begin{itemize}
	\item for each $t\geq 0$ and $x\in\mathbb R^n$ the function
	\[
	y\mapsto K_{\frac{1}{2}\Delta}(t;x,y)=\mu_{t;x,y}(C_{t;x,y})
	\]
defines a probability density in $\mathbb R^n$;
\item For each $U\in\mathcal B^n$ the quantity
\[
P_{t;x}(U):=\int_U K_{\frac{1}{2}\Delta}(t;x,y)dy
\]	
\end{itemize}
may be interpreted as 
the probability that a Brownian path passes through $U$ when $s=t$ given that it has started at $x$ when $s=0$. By shrinking $U$ to $\{y\}$ we thus conclude that 
$K_{\frac{1}{2}\Delta}(t;x,y)$ may be viewed as 
a {\em transition probability density} in the sense that $K_{\frac{1}{2}\Delta}(t;x,y)dy$ is the probability for a  Brownian path starting at $x$ to be found in the infinitesimal region $dy$ at time $t$.	
\end{remark}

\begin{remark}\!\!$\bigstar$\label{weyl}(Weyl's law and its modern incarnation in Index Theory)
	If $e^{t\mathcal L}$ is of trace class, its trace can be computed by integrating (\ref{fkac2:eq}) along the diagonal $x=y$ of $\mathbb{R}^n\times\mathbb{R}^n$:
	\[
	{\rm Tr}\,e^{t\mathcal L}=\int_{\mathbb R^n}\left(\int_{C_{t;x,x}}e^{-\int_0^tV(X_\tau)\,d\tau}\, d\mu_{t;x,x}\right)dx.
	\]
	Defining a measure $d\mu_t$ on the space $C_t=\cup_{x\in\mathbb R^n}C_{t;x,x}$ of all \emph{Brownian loops} in $\mathbb{R}^n$ with lifetime $t$ by setting $d\mu_t=d\mu_{t;x,x}\,dx$, we obtain
	\[
	{\rm Tr}\,e^{t\mathcal L}=\int_{C_t} e^{-\int_0^t V(X_\tau)\,d\tau}\, d\mu_t.
	\]
	This formula continues to hold when $\mathbb{R}^n$ is replaced by a compact Riemannian manifold $(M,g)$, with $\Delta_g$ denoting the Laplace--Beltrami operator. In this case,
	\[
	{\rm Tr}\,e^{\tfrac{1}{2}t\Delta_g}=\int_{C_t} d\mu_t,
	\]
	where the measure $\mu_t$ is defined locally in coordinate charts and then assembled in the usual way; see \cite{hsu2002stochastic} for a detailed account of Brownian motion on Riemannian manifolds and its fundamental properties.
		As $t\to 0$, a typical Brownian loop in $C_t$ contracts to its base point while remaining within a geodesic ball whose radius also vanishes with $t$ \cite[Lemma 7.7]{hsu2002stochastic}. Consequently, the path integral on the right-hand side becomes localized around $M\subset C_t$. Coupled with the ``principle of not feeling the curvature,'' which asserts that $K_{\tfrac{1}{2}\Delta_g}(t;x,x)\sim (2\pi t)^{-n/2}$ as $t\to 0$, we obtain
	\[
	{\rm Tr}\,e^{\tfrac{1}{2}t\Delta_g}\sim (2\pi t)^{-n/2}{\rm vol}(M,g).
	\]
	Since ${\rm Tr}\,e^{\tfrac{1}{2}t\Delta_g}=\sum_i e^{-\tfrac{1}{2}\lambda_it}$, where $\{\lambda_i\}$ are the positive eigenvalues of $\Delta_g$, this yields Weyl’s celebrated result: the volume of $(M,g)$ can be recovered from the asymptotic behavior of its spectrum.
	A more sophisticated version of this argument, involving the short-time asymptotics of the heat kernel associated with a supersymmetric Dirac operator on spinors, leads to a probabilistic proof of the Atiyah--Singer index theorem \cite{bismut1984atiyah,hsu2002stochastic}. Further developments along these lines, relying on refined Feynman--Kac representations of the heat kernel for certain Hodge Laplacians acting on sections of geometric vector bundles over Riemannian manifolds (possibly noncompact and with boundary), can be found in \cite{de2017feynman,de2017probabilistic,de2020heat} and references therein. \qed
\end{remark}

\begin{remark}\label{bb:0:b:sde}
	(The Brownian bridge as a Gaussian process)
	Starting from (\ref{bb:sde}) we may explicitly compute the law of the Brownian bridge, which will allow us to check that it is a Gaussian process; cf. Remark \ref{gauss:proc}. To simplify matters we take $n=1$, $x=0$, and $t=1$, so (\ref{bb:sde}) becomes 
	\begin{equation}\label{bb:sde:stan}
	dB_s=\frac{y-B_s}{1-s}\,ds+db_s,\qquad 0\le s<1,\qquad B_0=0,
	\end{equation}
	where $\{b_s\}_{s\geq 0}$ is the standard Brownian motion in $\mathbb R$.
Writing this as 
	\[
	dB_s+\frac{1}{1-s}B_s\,ds=\frac{y}{1-s}\,ds+db_s,
	\]
	and multiplying by $(1-s)^{-1}$, we obtain
	\[
	d\!\left(\frac{B_s}{1-s}\right)
	=\frac{y}{(1-s)^2}\,ds+\frac{1}{1-s}\,db_s,
	\]
	which may be easily integrated from $0$ to $s$ with $B_0=0$ to yield
		\[
	B_s=ys+(1-s)\int_0^s\frac{1}{1-\tau}\,db_\tau.
	\]
	By Proposition \ref{itoprop} (7), 
	\[
	\int_0^s\frac{1}{1-\tau}\,db_\tau\sim\mathcal N\left(0,\frac{s}{1-s}\right),
	\]
	so that Proposition \ref{norm:spce}  implies that the Brownian bridge is a Gaussian process with
	\[
	B_s \sim\mathcal N\left(ys,s(1-s)\right), \quad 0\leq s< 1.
	\]
	From this we get 
	\[
	\mathbb E(B_s)=ys,
	\]
	which says that the mean path is the line joining $0$ to $y$, and 
	\begin{equation}\label{bb:var}
		\mathbb{V}(B_s)=s(1-s),
	\end{equation}
	which gives 
	\[
	\mathbb{V}(B_0)=\lim_{s\uparrow 1}\mathbb{V}(B_s)=0,
	\]
	 thus confirming the intuition that $B_s$ is pinned at its endpoints $0$ and $y$.
	We may also determine the covariance function of $B_s$
	by first 
	centering the process,
	\[
	B_s-\mathbb E(B_s)=(1-s)\int_0^s\frac{1}{1-\tau}\,db_\tau,
	\]
	which gives, for $0\leq s<s'<1$,
	\begin{eqnarray*}
	\mathbb C(B_s,B_{s'})
	& = & 
	(1-s)(1-s')\mathbb C\left(
	\int_0^s\frac{1}{1-\tau}\,db_\tau,\int_0^{s}\frac{1}{1-\tau}\,db_\tau,
	\right)\\
	& & \quad +
	(1-s)(1-s')\mathbb C\left(
	\int_0^s\frac{1}{1-\tau}\,db_\tau,\int_s^{s'}\frac{1}{1-\tau}\,db_\tau
	\right).
	\end{eqnarray*}
	Since the last term in the right-hand side vanishes by Proposition \ref{itoprop}
	(8), we may use 
the polarized version of It\^o isometry in Proposition \ref{itoprop}
	(5) to find that 
	\[
	\mathbb{Cb}(B_s,B_{s'})
	=
	(1-s)(1-s')\int_0^s\frac{1}{(1-\tau)^2}\,d\tau
	=
	s(1-s'), 
	\]
	so by symmetry we end up with 
	\begin{equation}\label{bb:cov:0}
		\mathbb{C}(B_s,B_{s'})=s\wedge s'-ss', \quad s,s'\in [0,1).
	\end{equation}
	On the other hand, 
	by Theorem \ref{brown:exist} with $n=1$, $k=2$ and $s,s'\geq 0$, 
		\begin{equation}\label{bb:j:gauss}
	\left(
	\begin{array}{c}
	b_s\\
	b_{s'}	
		\end{array}
	\right)
	\sim
	\mathcal N
	\left(
	\left(
	\begin{array}{c}
		0\\
		0	
	\end{array}
	\right),
		\left(
	\begin{array}{cc}
		s & s\wedge s'\\
		s\wedge s' & s'	
	\end{array}
	\right)
	\right),
	\end{equation}
	so that Corollary \ref{ortho:norm} implies that
	\[
\widetilde B_s:=	b_s-s(b_1-y)=
	\left\langle
		\left(
	\begin{array}{c}
		1\\
		-s	
	\end{array}
	\right),
		\left(
	\begin{array}{c}
		b_s\\
		b_1-y	
	\end{array}
	\right)
	\right\rangle
	\sim
	\mathcal N(ys,s(1-s)), \quad 0\leq s \leq 1,
	\]
	is a Gaussian process having the same law as $B_s$ for each $s<1$.
Moreover, 
	for $0\leq s,s'\leq 1$,
	\begin{eqnarray*}
	\mathbb C(\widetilde B_s,\widetilde B_{s'})
	& = & 
	\mathbb C(b_s-sb_1,b_{s'}-s'b_1)\\
	& = & \mathbb C(b_s,b_{s'})-s'\mathbb C(b_s,b_1)-s\mathbb C(b_1,b_{s'})+ss'\mathbb C(b_1,b_1)\\
	& \stackrel{(\ref{bb:j:gauss})}{=}  & s\wedge s'-s's-ss'+ss',	
		\end{eqnarray*}
and we obtain
	\begin{equation}\label{bb:cov}
	\mathbb{C}(\widetilde B_s,\widetilde B_{s'})=s\wedge s'-ss', \quad s,s'\in [0,1].
	\end{equation}
	Upon comparison with (\ref{bb:cov:0}) and recalling from Remark \ref{gauss:proc} that a Gaussian process is completely determined by its mean and covariance functions, we conclude that the original process $\{B_s\}_{0\leq s<1}$ admits a natural extension to \(s=1\) with \(B_1=\widetilde B_1=y\), a property which is not immediately apparent from (\ref{bb:sde:stan}). 
	These calculations also show that the endpoint \(y\) affects only the mean $sy$, while the covariance structure is universal, coinciding with that of the standard Brownian loop $b_s-sb_1$ (corresponding to \(y=0\)). 
		\qed
\end{remark}

\begin{remark}\!\!$\bigstar$\label{gliv} (Glivenko-Cantelli and Donsker)
	If $X$ is a real-valued random variable with (unknown) distribution, a natural objective in Statistics is to determine its cdf $F$. However, this is in general out of the question, since knowing $F$ is equivalent to knowing the full distribution of $X$.
	Instead, we assume that we observe an i.i.d.\ sample $\{X_j\}_{j\in\mathbb N}$ drawn from the law of $X$, and seek to approximate $F$ from the data.	The canonical (unbiased) estimator of $F(s)$, $s\in\mathbb R$, to use here is the {\em empirical distribution function}, defined by
	\[
	\mathbb F_n(s)=\frac{1}{n}\sum_{j=1}^n{\bf 1}_{\{X_j\leq s\}}, 
	\quad n\in\mathbb N. 
	\]
	For each fixed $s$, $\mathbb F_n(s)$ is a random variable, and the collection $\{\mathbb F_n(s)\}_{n\in\mathbb N}$ may be viewed as a discrete-time stochastic process (as in Remark \ref{stocproc:d}) indexed by $s\in\mathbb R$. Since, 
	for each $s$, ${\bf 1}_{\{X_j\leq s\}}\sim \mathsf{Ber}(F(s))$, CLT (Theorem \ref{clt}) applies:
	\begin{equation}\label{gc:clt}
	\mathbb G_n(s):=\sqrt{n}\big(\mathbb F_n(s)-F(s)\big)\stackrel{d}{\longrightarrow}\mathcal N\big(0,F(s)(1-F(s))\big).
	\end{equation}
	In particular, 
	\begin{equation}\label{gc:lln}
\mathbb F_n(s)\stackrel{p}{\longrightarrow}F(s),
\end{equation}
which also follows diretly from LLN (Theorem \ref{lln}). Taken together, these results show that $\mathbb F_n(s)$ is a consistent and asymptotically normal estimator of $F(s)$.
	Now, a fundamental strengthening of (\ref{gc:lln}) asserts that the convergence is uniform in $s$, namely
	\[
	\|\mathbb F_n-F\|_{\infty}:=\sup_{x\in \mathbb R}|\mathbb F_n(s)-F(s)|\to 0 \quad \text{a.s.},
	\]
	a statement known as the {\em Glivenko--Cantelli theorem}. In order to examine a possible generalization of (\ref{gc:clt}), we first consider the case in which $F(s)=s{\bf 1}_{[0,1]}(s)$, corresponding to the uniform distribution in $[0,1]$, so the asymptotic variance in (\ref{gc:clt}) exactly matches (\ref{bb:var}). 
This is not a mere coincidence, but rather the first indication of a much deeper phenomenon.	
	Indeed, while (\ref{gc:clt}) describes the asymptotic behavior of $\mathbb G_n(s)$ for each fixed $s$, it is natural to ask whether the entire collection $\{\mathbb G_n(s)\}_{s\in[0,1]}$ converges, in a suitable sense, to a limiting stochastic process. The celebrated {\em Donsker theorem}, as applied to this instance, asserts that this is indeed the case: viewing $\mathbb G_n$ as a random element of an appropriate function space, one has
	\[
	\mathbb G_n \;\longrightarrow\; \{B_s\}_{s\in[0,1]},
	\]
	where $\{B_s\}$ is the standard Brownian bridge. 
	In particular, the covariance structure of the limiting process is given by (\ref{bb:cov}),
	which may be recognized as the continuum analogue of the covariance of the indicators $\mathbf{1}_{\{X_j\le s\}}$. Thus, the Brownian bridge emerges as the canonical Gaussian process describing the fluctuations of the empirical distribution function around its deterministic limit.
	In the general case of an arbitrary cdf $F$, the same phenomenon persists, with the limiting process given by $\{B_{F(s)}\}_{s\in\mathbb R}$, a reparametrization of the standard Brownian bridge whose covariance structure
	\[
	\mathbb C(B_{F(s)},B_{F(s')})=F(s\wedge s')-F(s)F(s')
	\]
	reflects the underlying distribution. In this way, Donsker's theorem may be viewed as a functional extension of the central limit theorem, complementing the first-order approximation provided by Glivenko--Cantelli with a precise description of second-order fluctuations. 
	For a full account of these results, 
	including their central role in empirical process theory and in distribution-free procedures for goodness-of-fit, such as the Kolmogorov--Smirnov test, we refer to
	\cite[Chapter 19]{van2000asymptotic}.
	\qed
\end{remark}

\begin{remark}\!\!$\bigstar$\label{emp:nonasymp}
	(Glivenko-Cantelli and  Donsker: the  non-asymptotic viewpoint)
		From a broader perspective, the results in Remark \ref{gliv}, describing the behavior of the empirical distribution function and its fluctuations as $n\to\infty$,
		mark the entry point into the theory of empirical processes, where one studies such fluctuations indexed by classes of functions rather than single points. 	In many applications, however, it is essential to provide non-asymptotic bounds quantifying how close these objects are to their limits for finite values of $n$.
	In the case of Glivenko-Cantelli, a fundamental example is the Dvoretzky--Kiefer--Wolfowitz-Massart inequality \cite{massart1990tight}, which asserts that
	\[
	P\left(\|\mathbb F_n-F\|_\infty\leq \sqrt{\frac{\ln(2/\delta)}{2n}}\right)
	\geq 1-\delta,\qquad 0<\delta<1,
	\]
	thus providing a high probability bound for the rate of convergence which aligns with the machinery of concentration inequalities in Section \ref{conc:ineq:appl}.
	On the other hand, Donsker's theorem may be complemented by strong approximation results, such as the Koml\'os-Major-Tusn\'ady theorem, which show that the empirical process $\mathbb G_n$ can be coupled with a Brownian bridge $\{B_s\}_{s\in[0,1]}$ in such a way that
	\[
	\sup_{s\in[0,1]}|\mathbb G_n(s)-B_s|
	=
	O\!\left(\frac{\log n}{\sqrt n}\right)
	\quad \text{a.s.}
	\]
More generally, when the empirical process is indexed by a class of functions $\mathcal F$ as in \ Remark \ref{rem:symmetrization}, one is led to study quantities of the form
	\[
	\|P_n-P\|_{\mathcal F}:=\sup_{f\in\mathcal F}|(P_n-P)f|,
	\]
	encompassing both Glivenko--Cantelli (consistency) and Donsker type (fluctuation) results.
	In this setting, nonasymptotic bounds are governed by complexity measures such as Rademacher averages, which control the size of the fluctuations uniformly over $\mathcal F$; cf.\ the general estimate in (\ref{non:as:rad}). From this standpoint, Rademacher complexity may be viewed as a finite-sample analogue of the Gaussian fluctuations described by Donsker's theorem.\qed
\end{remark}

\begin{example}\!\!$\bigstar$\label{feyn:int}(The Feynman integral)
The Feynman ``measure'' on the bridge space $C_{t;x,y}$ is given by
\[
d\nu_{t;x,y}=e^{\frac{1}{2}\int_0^t\left|\frac{dB_\tau}{d\tau}\right|^2d\tau}d\mu_{t;x,y},
\]
but notice that this lacks a mathematical meaning for at least one reason: the typical Brownian bridge path  $B_\tau$ is not differentiable! 
In any case, if we ignore this then  (\ref{fkac2:eq}) gives
\[
K_{\mathcal L}(t;x,y)=\int_{C_{t;x,y}}e^{-\int_0^t\left(\frac{1}{2}\left|\frac{dB_\tau}{d\tau}\right|^2+V(B_\tau)\right)d\tau}d\nu_{t;x,y},
\]
so if we pass from the imaginary (heat) time $t$ to the real (physical) time $\mathfrak t$ via the  Wick rotation $\mathfrak t=-{\bf i}\,t$ we get
\begin{equation}\label{hist:sum}
K_{\mathcal S}(\mathfrak t;x,y)=\int_{C_{\mathfrak t;x,y}}e^{{\bf i}\mathcal A\left(B_{\mathfrak t},\frac{dB_{\mathfrak t}}{d\mathfrak t}\right)}d\nu_{\mathfrak t;x,y},
\end{equation}
where 
\[
\mathcal A\left(B_{\mathfrak t},\frac{dB_{\mathfrak t}}{d\mathfrak t}\right)=\int_0^{\mathfrak t}\left(\frac{1}{2}\left|\frac{dB_{\mathfrak t'}}{d\mathfrak t'}\right|^2-V(B_{\mathfrak t'})\right)d\mathfrak t'
\]
is the classical Lagrangian action and
$K_{\mathcal S}$ is the kernel of the unitary group $e^{-{\bf i}\mathfrak t\mathcal S}$ associated to the Schr\"odinger operator 
\[
\mathcal S=-\frac{1}{2}\Delta +V,
\]
in the sense that it satisfies the formal analogue of (\ref{heat:fund:sol}):
\[
	\left\{
	\begin{array}{rcl}
	{\bf i}	\frac{\partial K_{\mathcal L}}{\partial {\mathfrak t}} & = & -\frac{1}{2}\Delta K_{\mathcal S} + V K_{\mathcal S} \\
		K_{\mathcal S}(0;x,y) & = & \delta(x-y)
	\end{array}
	\right.
\]
In this way, the Feynman integral in (\ref{hist:sum}) provides the celebrated ``sum over histories'' formula from quantum mechanics \cite{feynman2010quantum}. 
Here, the complex number $K_{\mathcal S}(\mathfrak t;x,y)$ should be viewed as a transition amplitude so that, according to Born rule, $|K_{\mathcal S}(\mathfrak t;x,y)|^2dy$ is the probability of finding the quantum particle at the infinitesimal region $dy$ at time $\mathfrak t$ given that it has been emitted at $x$ at time $\mathfrak t=0$. 
More generally, if as typically occurs in the quantum realm, at time $\mathfrak t=0$ the particle is not localized at $x$  but instead spreads out in space so that we only observe its initial wave function $\psi_0(\cdot)\in L^2(\mathbb R^n)$,
then
\[
\Psi(\mathfrak t,x)=(e^{-{\bf i}\mathfrak t\mathcal S}\psi_0)(x)=\int_{\mathbb R^n}K_{\mathcal S}(\mathfrak t;x,y)\psi_0(y)dy,\quad \mathfrak t\in\mathbb R,
\]
describes the corresponding dynamics in the sense that it solves the Schr\"odinger equation
\[
{\bf i}\frac{\partial \Psi}{\partial \mathfrak t}=\mathcal S\Psi, \quad \Psi(0,\cdot)=\psi_0.
\]
We note however that the above derivation is far from being mathematical rigorous since the putative Feynman measure $d\nu_{\mathfrak t;x,y}$ 
is not a genuine countably additive measure on bridge space. 
Despite its lack of a fully satisfactory mathematical foundation, this viewpoint has proved extraordinarily influential in theoretical physics: it provides the conceptual and computational backbone of perturbative quantum field theory, where correlation functions are formally expressed as functional integrals. 
From this perspective, the ``sum over histories'' above should be viewed as a particular instance of a functional integral, in which one integrates over trajectories depending on a single time parameter, whereas in quantum field theory this framework is extended to integrals over fields defined on spacetime.
For mathematically rigorous accounts of this approach, see \cite{simon2005functional,glimm2012quantum}.
\qed	
	\end{example}

\subsection{The Black-Scholes strategy in Finance}\label{black:scholes} Here we derive the celebrated Black-Scholes option pricing formula\footnote{As it is well-knwon, this has been worth a Nobel Prize in 1997.}. From the outset, this involves a risky asset $S_t$, a {\em stock},  evolving in time according to a geometric Brownian as in Example \ref{geo:br}:
\begin{equation}\label{stock}
	\frac{dS_t}{S_t}=\mu dt+\sigma db_t.
\end{equation}
Here, $\mu >0$ is the {\em mean rate of return} and $\sigma>0$ is the {\em volatility}. Recall that 
\[
S_t=S_0e^{(\mu-\frac{\sigma^2}{2})t+\sigma b_t}
\]
provides the explicit solution of (\ref{stock}). In particular, 
\[
\ln S_t=\ln S_0+\left(\mu-\frac{\sigma^2}{2}\right)t+\sigma b_t \sim\mathcal N
\left(\underbrace{\ln S_0+\left(\mu-\frac{\sigma^2}{2}\right)t}_{=:m},\underbrace{\sigma^2 t}_{=:\nu^2}\right),
\]
so that, by Example \ref{moment}, 
\begin{equation}\label{xt:ln}
	S_t=e^{\ln S_t}\sim \mathcal L\mathcal N\left(S_0e^{\mu t},S^2_0e^{2\mu t}(e^{\sigma^2t}-1)\right)=\Lambda(m,\nu^2).
\end{equation}

On the other hand, we have an investor's {\em portfolio} $(A_t,B_t)$ whose value is
\begin{equation}\label{porf} 
	V_t=A_tS_t+B_t\gamma_t,
\end{equation}
with the risk-less {\em bond} $\gamma_t$ satisfying $d\gamma_t=r\gamma_tdt$, where $r>0$ is the associated {\em interest rate}. The {\em option pricing problem} addressed by Black-Scholes consists in adjusting the trading strategy $(A_t,B_t)$ to the underlying  asset $S_t$ by (deterministically!) finding a function $u$ such that
\begin{equation}\label{adjust}
	V_t=u(t,S_t), \quad 0\leq t \leq T, 
\end{equation} 
where $T>0$ is the {\em expiration time} for the option. 
The key point here is that $u=u(t,x)$ should satisfy a certain PDE.
In order to find it, we start with (\ref{stock}) and apply It\^o formula to (\ref{adjust}) to check  that 
\begin{equation}\label{dvt:1}
	dV_t=\left(\frac{\partial u}{\partial t}+\mu S_t\frac{\partial u}{\partial x}+\frac{\sigma^2 S_t^2}{2}\frac{\partial^2u}{\partial x^2}\right)dt+\sigma S_t\frac{\partial u}{\partial x}db_t. 
\end{equation}
On the other hand, if we assume that our portfolio is {\em self-financing} in the sense that
\[
dV_t=A_tdS_t+B_td\gamma_t,
\]
we get 
\begin{equation}\label{vt:2d}
	dV_t=\left(\mu A_tS_t+rB_t\gamma_t\right)dt+\sigma A_tS_tdb_t. 
\end{equation}
By comparing the diffusion and drift coefficients in the expressions for $dV_t$ above we get
\begin{equation}\label{comp:1}
	A_t=\frac{\partial u}{\partial x}
\end{equation}
and hence
\begin{equation}\label{comp:2}
	\frac{\partial u}{\partial t}+\frac{\sigma^2 S_t^2}{2}\frac{\partial^2u}{\partial x^2}=rB_t\gamma_t. 
\end{equation}
Now note that from (\ref{porf}), (\ref{adjust}) and (\ref{comp:1}),
\begin{equation}\label{bt:f}
	B_t\gamma_t=V_t-A_tS_t=u-S_t\frac{\partial u}{\partial x},
\end{equation}
so if we replace this in the right-hand side of (\ref{comp:2}) and make $S_t=x$ we conclude that $u$ must satisfy the {\em Black-Scholes equation}
\begin{equation}\label{bscholes:eq}
	\frac{\partial u}{\partial t}+\frac{\sigma^2 x^2}{2}\frac{\partial^2u}{\partial x^2}+rx\frac{\partial u}{\partial x}-ru=0. 	
\end{equation}

Notice that the coefficients of this PDE depend on $\sigma$ and $r$ but not on $\mu$. Also, as written the PDE fails to be of heat type because the coefficients of $\partial u/\partial t$ and $\partial^2u/\partial x^2$ have the same sign. This  suggests that we should try to solve it by imposing
the ``terminal condition''
\begin{equation}\label{init:cond}
	u(T,S_T)=V_T.
\end{equation}
In fact, the choice 
\begin{equation}\label{init:cond:2}
	V_T=\max \{0,S_T-K\}, \quad K>0,
\end{equation}
corresponds to the investor holding at time $t=0$ the option (but not the obligation) of buying the stock by a fixed price $K$ at the expiration time $T$. Hence, in this  {\em European call}, if $S_T>K$ then the owner of the option will obtain the payoff $S_T-K$ whereas if $S_T\leq K$ the owner will not exercise his option,  thus obtaining a null payoff.  

In terms of the cumulative normal distribution
\[
\Phi(x)=\frac{1}{\sqrt{2\pi}}\int_{-\infty}^xe^{-y^2/2}dy, \quad x\in\mathbb R,
\]
the Black-Scholes equation (\ref{bscholes:eq}) with $u(T,S_T)=\max \{0,S_T-K\}$ may be explicitly solved as 
\begin{equation}\label{sol:bs}
	u(t,x)=x\Phi(g(t,x))-K e^{-r(T-t)}\Phi(h(t,x)), 
\end{equation}
where 
\[
g(t,x)=\frac{\ln (x/K)+\left(r+\frac{1}{2}\sigma^2\right)(T-t)}{\sigma\sqrt{T-t}}
\]
and 
\[
h(t,x)=g(t,x)-\sigma\sqrt{T-t}=\frac{\ln (x/K)+\left(r-\frac{1}{2}\sigma^2\right)(T-t)}{\sigma\sqrt{T-t}}.
\]
Notice that, as expected, the solution depends on $\sigma$ and $r$ but not on $\mu$.
We conclude that the {\em Black-Scholes option pricing formula}
\[
V_0=u(0,S_0)=S_0\Phi(g(0,S_0))-K e^{-rT}\Phi(h(0,S_0))
\]
provides the rational price  to hold at the initial time $t=0$ a European call option with price $K$. 
Also, from (\ref{comp:1}) and (\ref{bt:f}) we see that the corresponding self-financing strategy is 
\[
(A_t,B_t)=\left(\frac{\partial u}{\partial x}, \gamma_0^{-1}e^{-rt}\left(u-S_t\frac{\partial u}{\partial x}\right)\right).
\]

The explicit solution (\ref{sol:bs}) to (\ref{bscholes:eq})-(\ref{init:cond:2}) may be obtained in many ways. For instance, a ``deterministic'' approach may be pursued upon successive changes of variables so as to transform (\ref{bscholes:eq}) into the standard heat equation, which can then be explicitly solved by the usual methods \cite[Section 10.3]{kallianpur2012introduction}. Alternatively, we may appeal to the full power of It\^o calculus as follows. Let us set $\theta=(\mu-r)/\sigma$ and consider the process
\[
M_t=e^{-\theta b_t-\frac{1}{2}\theta^2t}, \quad 0\leq t\leq T.
\]  
From It\^o formula we obtain
\[
dM_t=-\theta M_tdb_t,
\]
so $M_t$ is a $b_t$-martingale. Hence, by Proposition \ref{martcons},
\[
\mathbb E^{P}(M_T)=\mathbb E^{P}(M_0)=\mathbb E^{P}(1)=1,
\]
where  $P$ is Wiener measure.
Thus, if we define a new measure $P^\bullet$ on Wiener space by requiring that $dP^\bullet=M_Td P$,
it is immediate that $P^\bullet$ is a probability measure. 
A version of Girsanov's theorem \cite[Theorem 8.6.4]{oksendal2013stochastic} applies here and we conclude that $b^\bullet_t:=b_t+\theta t$ is a standard Brownian with respect to $P^\bullet$ (so that $b_t^\bullet\sim_{P^\bullet}\mathcal N(0,t)$) and, moreover, 
\begin{equation}\label{bt:mod}
	dS_t=rS_tdt+\sigma S_tdb_t^\bullet.
\end{equation}
Thus, we have been able to modify the drift of the dynamics of the stock (from $\mu S_t$ to $rS_t$) at the cost of changing the underlying measure (from $P$ to $P^\bullet$) and the driving Brownian (from $b_t$ to $b_t^\bullet$). The reason for doing this is now obvious:
in terms of the infinitesimal generator of (\ref{bt:mod}), namely,
\[
L=\frac{\sigma^2x^2}{2}\frac{\partial^2}{\partial x^2}+rx\frac{\partial}{\partial x}, 
\]
(\ref{bscholes:eq}) may be rewritten as 
\[
\frac{\partial u}{\partial t}+Lu-ru=0, 
\]
whose solution may be obtained by the method leading to the Feynman-Kac formula discussed in Section \ref{fk:sec} (with $r$ playing the role of a constant potential). Indeed, if we apply It\^o formula to $v(t,S_t):=e^{r(T-t)}u(t,S_t)$, we easily see that 
\[
dv(t,S_t)=\sigma S_te^{r(T-t)}\frac{\partial u}{\partial x}(t,S_t)db^\bullet_t,
\]
which means that $v(t,S_t)$ is a $b_t^\bullet$-martingale. Thus, if we calculate the (identical!) expectations at the endpoints of the interval $[t,T]$ and use (\ref{init:cond:2}) we end up with  
\[
u(t,S_t)=e^{-r(T-t)}\mathbb E^{P^\bullet}(\max\{S_T-K,0\}).
\]

We now observe that, due to (\ref{bt:mod}) and similarly to (\ref{xt:ln}), we now have 
\begin{equation}\label{xt:ln:new}
	S_T\sim_{P^\bullet} \mathcal L\mathcal N\left(S_te^{r (T-t)},S^2_te^{2r (T-t)}(e^{\sigma^2(T-t)}-1)\right),
\end{equation}
or equivalently,
\begin{equation}\label{choice:par}
	\ln S_T\sim_{P^\bullet}\mathcal N
	\left(\underbrace{\ln S_t+\left(\mu-\frac{\sigma^2}{2}\right)(T-t)}_{=:m},\underbrace{\sigma^2 (T-t)}_{=:\nu^2}\right)=\Lambda(m,\nu^2).
\end{equation}
Now,  
\begin{eqnarray*}
	\mathbb E^{P^\bullet}(\max\{S_T-K,0\})
	& = & 
	\int_K^{+\infty}(S_T-K)dF_{S_T}\\
	& = & 
	\int_K^{+\infty}S_TdF_{S_T}-K\int_K^{+\infty}dF_{S_T},
\end{eqnarray*}
where $dF_{S_T}=\psi_{S_T}dx$ is the cdf of $S_T$, and these integrals may be computed in terms of $\Phi$ by means of the recipe in Example \ref{moment}. For instance, 
\begin{eqnarray*}
	\int_K^{+\infty}dF_{S_T} & = & 1- \int^K_{-\infty}dF_{S_T}\\
	& = & 1-F_{S_T}(K)\\
	& = & 1-\Phi\left(\frac{\ln K\overbrace{-\ln S_t-\left(r-\frac{\sigma^2}{2}\right)(T-t)}^{=-m}}{\underbrace{\sigma\sqrt{T-t}}_{=\nu}}\right)\\
	& = &  1-\Phi\left(-\frac{\ln S_t/K+\left(r-\frac{\sigma^2}{2}\right)(T-t)}{\sigma\sqrt{T-t}}\right)\\
	& = & \Phi\left(\frac{\ln S_t/K+\left(r-\frac{\sigma^2}{2}\right)(T-t)}{\sigma\sqrt{T-t}}\right)\\
	& = & \Phi(h(t,S_t)),	
\end{eqnarray*}
where we used (\ref{cum:fct}), our choices for $m$ and $\nu$ as in (\ref{choice:par}) and the fact that $\Phi(x)+\Phi(-x)=1$. 
Similarly, 
\begin{eqnarray*}
	\int_K^{+\infty}S_TdF_{S_T} & = & \mathbb E^{P^\bullet}(S_T)- \int^K_{-\infty}S_TdF_{S_T}\\
	& = & S_te^{r(T-t)}-\int_{-\infty}^Kx\psi_{S_T}(x)dx\\	
	& = & S_te^{r(T-t)}-\frac{1}{\sqrt{2\pi}\sigma\sqrt{T-t}}\int_{-\infty}^K e^{-\frac{1}{2}\left(
		\frac{\ln x-\ln S_t-\left(r-\frac{\sigma^2}{2}\right)(T-t)}{\sigma\sqrt{T-t}}\right)^2}dx,
\end{eqnarray*}
where we used (\ref{xt:ln:new}) and (\ref{dist:ln}). 
After an appropriate change of variables and using again that  $\Phi(x)+\Phi(-x)=1$, we  get
\begin{eqnarray*}
	\int_K^{+\infty}S_TdF_{S_T}
	& = & 
	e^{r(T-t)}S_t\Phi\left(\frac{-\ln K+\ln S_t+\left(r-\frac{\sigma^2}{2}\right)(T-t)+\sigma^2(T-t)}{\sigma\sqrt{T-t}}\right)\\
	& = & e^{r(T-t)}S_t\Phi\left(g(t,X_t)\right).
\end{eqnarray*}
Putting all the pieces of this computation together we find that 
\[
u(t,S_t)=S_t\Phi(g(t,S_t))-K e^{-r(T-t)}\Phi(h(t,S_t)),
\]
which matches (\ref{sol:bs}) if we make $S_t=x$.

\bibliographystyle{alpha}
\bibliography{prob_notes}

\end{document}